\documentclass[11pt,openright,twoside,a4paper,reqno]{amsbook}
\usepackage[a4paper,top=3cm,bottom=3cm,inner=4cm,outer=3cm]{geometry}
\usepackage{mathrsfs,amsmath,amsfonts,amsthm,amssymb,graphicx,pdfpages}
\usepackage{setspace,fancyhdr,calc,verbatim,bbm,enumerate,color,umoline}
\usepackage[all,2cell]{xy}\UseAllTwocells\SilentMatrices
\usepackage[linktoc=page]{hyperref}

\newtheorem*{thm*}{Theorem}
\theoremstyle{remark}
\newtheorem*{rmk*}{Remark}
\newtheorem*{lem*}{Lemma}
\theoremstyle{definition}
\newtheorem*{defi*}{Definition}
\newtheorem*{cor*}{Corollary}
\theoremstyle{definition}
\newtheorem*{examples*}{Examples}
\newtheorem{prop*}{Proposition}

\theoremstyle{plain}
\newtheorem{thm}{Theorem}[section]
\theoremstyle{plain}
\newtheorem{prop}[thm]{Proposition}
\theoremstyle{remark}
\newtheorem{rmk}[thm]{Remark}
\theoremstyle{plain}
\newtheorem{lem}[thm]{Lemma}
\theoremstyle{plain}
\newtheorem{cor}[thm]{Corollary}
\theoremstyle{definition}
\newtheorem{defi}[thm]{Definition}
\theoremstyle{definition}
\newtheorem{examples}[thm]{Examples}

\newcommand{\ca}{\mathcal}
\newcommand{\caa}{\mathbb} 
\newcommand{\Hom}{\ensuremath{\mathrm{Hom}}}
\newcommand{\HOM}{\textrm{\scshape{Hom}}}
\newcommand{\Comod}{\ensuremath{\mathbf{Comod}}}
\newcommand{\Mod}{\ensuremath{\mathbf{Mod}}}
\newcommand{\Vect}{\ensuremath{\mathbf{Vect}}}
\newcommand{\Alg}{\ensuremath{\mathbf{Alg}}}
\newcommand{\Coalg}{\ensuremath{\mathbf{Coalg}}}
\newcommand{\Mon}{\ensuremath{\mathbf{Mon}}}
\newcommand{\Comon}{\ensuremath{\mathbf{Comon}}}

\newcommand{\ob}{\ensuremath{\mathrm{ob}}}
\newcommand{\Cart}{\ensuremath{\mathrm{Cart}}}
\newcommand{\Cocart}{\ensuremath{\mathrm{Cocart}}}
\newcommand{\ps}{\mathscr}
\newcommand{\B}{\mathbf}
\newcommand{\Gr}{\mathfrak}
\newcommand{\Mat}{\mathbf{Mat}}
\newcommand{\pullbackcorner}[1][dr]{\save*!/#1+1.2pc/#1:(1,-1)@^{|-}\restore}
\DeclareMathOperator*\colim{colim}
\newcommand{\op}{\mathrm{op}}
\newcommand{\lcong}{\rotatebox[origin=c]{90}{$\cong$}}
\newcommand{\MMat}{\mathbf{\caa{M}at}}
\newcommand{\id}{\mathrm{id}}

\numberwithin{equation}{chapter}
\numberwithin{section}{chapter}

\makeatletter
\@addtoreset{thm}{chapter}
\makeatother

\onehalfspace

\makeatletter
\renewcommand{\l@chapter}{\@tocline{0}{12pt}{0pt}{}{\bfseries}}
\renewcommand{\l@subsection}{\@tocline{2}{0pt}{2pc}{5pc}{}}
\makeatother

\mainmatter

\begin{document}

\frontmatter

\thispagestyle{empty}

\begin{center}
\mbox{}
{\Large{\bf GENERALIZATION OF ALGEBRAIC OPERATIONS VIA ENRICHMENT}}
\vskip 1in
{\Large{Christina Vasilakopoulou}}
\vskip 0.7in
{\small Trinity College \\
and \\
Department of Pure Mathematics and Mathematical Statistics \\
University of Cambridge}
\vskip 1.8in
This dissertation is submitted for the degree of \\
\emph{Doctor of Philosophy}
\vskip 0.8in
March 2014
\end{center}

\pagestyle{fancy}
\fancyhf{} 
\renewcommand{\headrulewidth}{0pt}
\pagenumbering{roman}

\newpage
\mbox{}
\vskip 2in
\begin{quote}
This dissertation is the result of my own work and includes nothing that is the outcome of work done in collaboration except 
where specifically indicated in the text.

This dissertation is not substantially the same as any that I have
submitted for a degree or diploma or any other qualification at any
other university.
\vskip 1in

\hfill Christina Vasilakopoulou

\hfill 27 March 2014

\end{quote}

\newpage
\begin{center}
{\large{\bfseries Generalization of Algebraic Operations via Enrichment\\}}
\vskip 0.1in
Christina Vasilakopoulou
\vskip 0.3in
{\sc Summary}
\vskip 0.1in
\end{center}
\begin{quote}
\begin{singlespace}

In this dissertation we examine enrichment relations between 
categories of dual structure and we sketch
an abstract framework where the theory
of fibrations and enriched category theory are appropriately united.

We initially work in the context of a monoidal category, where we study
an enrichment of the category of monoids in the category of 
comonoids under certain assumptions. This is induced by 
the existence of the \emph{universal measuring comonoid},
a notion originally defined by Sweedler in \cite{Sweedler}
in vector spaces. We then consider 
the fibred category of modules over arbitrary monoids, and 
we establish its enrichment in the opfibred category of 
comodules over arbitrary comonoids. This is now exhibited via 
the existence of the \emph{universal measuring comodule},
introduced by Batchelor in \cite{Batchelor}.

We then generalize these results to their `many-object'
version. In the setting of the bicategory of $\ca{V}$-enriched 
matrices (see \cite{KellyLack}), we investigate an enrichment 
of $\ca{V}$-categories in $\ca{V}$-\emph{co\-ca\-te\-go\-ries} as 
well as of $\ca{V}$-modules in $\ca{V}$-\emph{comodules}. 
This part constitutes the core 
of this treatment, and the theory 
of fibrations and adjunctions between them plays a 
central role in the development. 
The newly constructed categories are described in detail,
and they appropriately fit in a picture of duality,
enrichment and fibrations as in the previous case.

Finally, we introduce the concept of an \emph{enriched fibration},
ai\-med to provide a formal description for the above examples.
Related work in this direction, though from a different perspective
and with dissimilar outcomes, has been realized by 
Shulman in \cite{Enrichedindexedcats}.
We also discuss an abstraction of this picture 
in the environment of double categories, concerning 
categories of \emph{monoids} and \emph{modules} therein.
Relevant ideas can be found in \cite{Monadsindoublecats}.

\end{singlespace}
\end{quote}

\newpage
\mbox{}
\vskip 1in
\begin{center}
{\sc Acknowledgements}
\vskip 0.1in
\end{center}
I would like to begin by thanking my supervisor Professor
Martin Hyland for 
his patient guidance and encouragement throughout 
my studies. His continuous support and inspiration 
have been extremery valuable to me.
I would also like to thank Dr. Ignacio Lopez Franco,
for all the stimulating discussions and constructive
suggestions, several of which were decisive
in the formation of this thesis.

I would like to offer my special thanks to Professor Panagis
Karazeris for all his help and motivation during and after 
my undergraduate studies at the University of Patras, 
as well as Professor Georgios Dassios
for his support and advice.

This thesis would not have been possible
without the mathematical and non-mathematical assistance of the 
Category Theory `gang': Achilleas, Guilherme, Ta\-ma\-ra 
and the rest of my academic family. 
Also, to all my fellow PhD students
and friends, Gabriele, Giulio, John, Liyan, Micha\l, Beth, Anastasia, Richard,
Robert, Jerome, Rachel, Eleni and particularly Stefanos, special thanks.
You all made Cambridge such a wonderful place to be,
full of happy moments.

I am grateful to Trinity College, EPSRC and DPMMS for
financially supporting my studies.   
I would also like to thank the Propondis
Foundation and the Leventis Foundation
for providing me with PhD scholarships 
that gave me the opportunity to realize 
my research project.

Special thanks must be attributed to my dear friends
Elpida, Ioanna, Artemis, Irida, Aggeliki, Evi and Andreas,
who have proved companions for life. 
I would not have made it that far 
if it weren't for you.

Finally, I owe my deepest gratitude to my beloved family: 
my parents Giorgos and Magda and my brother Paris 
for their unconditional loyalty and support,
as well as Dimitris for his love and dedication.
Thank you for always being there for me.
\tableofcontents

\pagestyle{headings}
\pagenumbering{arabic}

\chapter{Introduction} \label{introduction}
Algebras and their modules, as well as coalgebras 
and their comodules, are amongst the simplest 
and most fundamental structures in abstract mathematics. 
Formally, algebras are dual to coalgebras and modules 
are dual to comodules, 
but in practice that point of view is very limited. 
The initial motivation for the material included in
the present thesis was a more 
striking relation between these notions: in 
natural circumstances, the mere category of 
algebras is enriched in the category of coalgebras,
and that of modules in comodules. 
These enrichments encapsulate some very rich algebraic structure,
that of the so-called measuring coalgebras and comodules.

More specifically, the notion of the \emph{universal measuring coalgebra} 
$P(A,B)$ was first 
introduced by Sweedler in \cite{Sweedler}, and has been employed
as a way of giving sense to an idea of generalized maps between algebras.
Examples of this point of view and applications are 
given by Marjorie Batchelor in \cite{MeasuringCoalgebras}
and \cite{Differenceoperators}.
It was
Gavin Wraith in the 1970's, who first suggested that 
this coalgebra gives an enrichment
of the category of algebras in the category of coalgebras,
however for a long time there was no explicit treatment of Wraith's idea in 
the literature. Furthermore, this idea can be appropriately extended
to give an enrichment of a global category of modules in 
a global category of comodules, via the \emph{universal measuring comodule}
$Q(M,N)$ introduced by Batchelor in \cite{Batchelor}. These objects
have also found applications on their own, analytically presented
in the provided references.

Independently of questions of enrichment, there is a well-known
fibration of the global category of modules over algebras 
in addition to an opfibration of the comodules over coalgebras. 
This extra structure seems to point torwards a picture 
that integrates the two classical notions, enrichment and fibration, 
which generally do not go well together. One of the basic objectives
of this thesis is to successfully describe what could 
be called an \emph{enriched fibration}. 

Inspired by the above, we are led to consider the `many-object' 
generalization of the previous situation. Since an algebra 
is evidently a (linear) category with one object,  
the categories of interest on this next step 
are naturally those of enriched categories and 
enriched modules, on the one hand. For the analogues of coalgebras
and comodules, we proceed to the definitions of 
an \emph{enriched cocategory} and \emph{enriched comodule}. After setting 
up the theory of these new categories and exploring
some of their more pertinent properties, we establish an 
enrichment of $\ca{V}$-categories in $\ca{V}$-cocategories,
and of $\ca{V}$-modules in $\ca{V}$-comodules. The similarities with 
the base case of (co)algebras and (co)modules are expressed primarily 
by the methodology and the series of arguments followed. However,
this generalization reveals more advanced ideas and certain patterns
of expected behaviour of the categories involved.
This newly acquired perspective urges us to 
develop a theoretic frame in which 
a general machinery, certain aspects of which were
described in detail for the two particular cases, would
always result in the speculated enriched fibration picture.  

Thus, another central aim of this dissertation is
to identify
this abstract framework which leads to instances of the 
enriched fibration notion, with starting point a monoidal bicategory
or even more closely related, a monoidal \emph{pseudo double category}.
In fact, the longer term goal of such a development was its possible
application to different contexts, 
and in particular to the theory of operads. In more detail,
if we replace the bicategory of $\ca{V}$-matrices
(which is the starting point for the duality and enrichment relations
for $\ca{V}$-categories and $\ca{V}$-modules) 
with the bicategory of $\ca{V}$-\emph{symmetries}
(see also \cite{GambinoJoyal}), 
there is strong evidence that we can 
establish an analogous enriched fibration which merges 
symmetric $\ca{V}$-operads
and operad modules and their duals. Moreover, both
coloured and non-coloured versions 
can be included in this plan. 
This indicates a fruitful area for future work.

The thesis is divided in two parts: the material
in Part I is mostly well-known, serving as the background
for the development that follows, while 
the material in Part II 
is mostly new.
We assume familiarity with the basic theory of categories,
as in the standard textbook \cite{MacLane} by MacLane.

In Chapter \ref{bicategories}, we review the basic definitions and 
features of the theory of bicategories and 2-categories,
with particular emphasis on the concepts of monads/comonads
and their modules/comodules in this abstract setting. 
Classic references on the main notions
are \cite{Benabou,GrayFormalCategoryTheory,
FibrationsinBicats,Handbook1,Review}. Coherence for bicategories,
very briefly mentioned here, is discussed
in \cite{CoherenceTricats,MacLane-Pare,Ageneralcoherenceresult},
and of course MacLane's coherence theorem for monoidal
categories preceded it (\cite{natass&comm,OnMacLanesconditions,
BraidedTensorCats}). 
Monads in a 2-category have been widely studied, 
with basic reference Ross Street's \cite{FormalTheoryMonadsI}.
Categories of modules, more commonly referred to as algebras especially 
in the 2-category $\B{Cat}$, are formed as 
categories of Eilenberg-Moore algebras on the hom-categories
$\ca{K}(A,B)$ of a bicategory $\ca{K}$.
\nocite{Carmody,2-catcompanion}
\nocite{Handbook2,2-dimmonadtheory}

Chapter \ref{monoidalcategories} summarizes basic concepts 
related to monoidal categories, following some 
of the many standard 
references such as 
\cite{MacLane,BraidedTensorCats,Quantum}.
Categories of monoids and modules
will play a very important role for the development of 
this dissertation, hence extra attention has been given 
to the presentation of their properties. In particular,
questions regarding the existence of the free monoid and 
the cofree comonoid constructions have been of primary 
interest. Certain papers by Hans Porst \cite{MonComonBimon,
FundConstrCoalgCorComod,AdjAlgCoalg}
have addressed this issue from a particular point of view, in the context
of locally presentable categories (see \cite{LocallyPresentable}). 
Specific methods, 
especially the ones related to local presentability 
of the categories of dual objects, are 
carefully exhibited here and in some cases generalized a bit further.
\nocite{Closedcategories,Doctrinal,Sweedler}

The main definitions and elementary features of the 
theory of enriched categories are summarized in Chapter \ref{enrichment}, 
with standard references
\cite{Kelly,Closedcategories}. Since enriched modules
are essential for the generalization of the monoids and modules
correlation to a $\ca{V}$-categories and $\ca{V}$-modules one,
we devote a section to some of their aspects needed for 
our purposes, see \cite{Distributeurs,Lawvereclosedcats}.
In the last part, we recall parts of the theory of actions
of monoidal categories on ordinary categories, which lead to 
a particular enrichment,
as described also in Janelidze and Kelly's \cite{AnoteonActions}.
In fact, this constitutes a special case of a more general result
discussed in \cite{enrthrvar}, namely that 
there is an equivalence
between the 2-category of tensored $\ca{W}$-categories and 
the 2-category of closed $\ca{W}$-representations, for $\ca{W}$
a right-closed bicategory.
\nocite{Carmody,GarnerShulman,Monoidalbicats&hopfalgebroids}

In Chapter \ref{fibrations}, the key material about fibred category theory is 
reviewed. Central notions and results are presented, including 
the correspondence between cloven fibrations and indexed categories
due to Grothendieck. The notion of a fibration was first introduced
in \cite{Grothendieckcategoriesfibrees},
and suitable references on the subject are 
\cite{Grayfibredandcofibred,Jacobs,Elephant2} and Hermida's work 
as can be found in, for example, \cite{hermidaphd,FibredAdjunctions}.
Finally, we move to the topic of fibred adunctions and 
fibrewise limits, where the main constructions and ideas
can be found in \cite{FibredAdjunctions} and \cite{Handbook2}.
Presently, we develop the issue a bit further: we examine conditions
not only for adjunctions between fibrations over the same basis,
but also for general fibred adjunctions, \emph{i.e.}
between fibrations over arbitrary bases. This slightly 
generalizes results which exist in the literature currently. 
This was not done 
aimlessly: Theorem \ref{totaladjointthm}
constitutes an extremely valuable tool for the 
establishment of the pursued enrichments later in the thesis.

Chapter \ref{enrichmentofmonsandmods} describes in 
detail the enrichment of monoids and modules, 
which is the motivating case for what
follows. In fact, the results of this chapter in a 
somewhat more restricted version previously appeared
in \cite{mine}, and have already been of 
use to a certain extent, see for example \cite{AnelJoyal}.
Explicitly, we identify the more general categorical
ideas underlying the existence of Sweedler's 
measuring coalgebra $P(A,B)$ of \cite{Sweedler,MeasuringCoalgebras} and 
prove its existence in a much broader context. Its defining
equation is in particular also provided in \cite{AdjAlgCoalg}
and observed in \cite{BarrCoalgebras}. 
Combined with the theory 
of actions of monoidal categories, we show 
how these $P(A,B)$ for any two monoids $A$ and $B$ 
induce an enrichment of the category of monoids $\Mon(\ca{V})$
in the category of comonoids $\Comon(\ca{V})$,
under specific assumptions on $\ca{V}$.
Subsequently, the `global' categories of modules
and comodules $\Mod$ and $\Comod$ are defined,
fibred and opfibred respectively over monoids
and comonoids. These categories have nice properties, and 
in particular, as hinted by Wischnewsky at the end 
of \cite{LinearReps}, $\Comod$ is comonadic over 
$\ca{V}\times\Comon(\ca{V})$, a fact which clarifies its structure.
Via the existence of an adjoint of a functor between the 
global categories, the universal measuring 
comodule $Q(M,N)$ is constructed,
as a variation of the notion in
\cite{Batchelor} in our general setting. Again
through a specific action functor, we obtain an enrichment
of $\Mod$ in $\Comod$, induced by these $Q(M,N)$ for 
any two modules $M$ and $N$ as the enriched hom-objects. Parts
of this work were accomplished in collaboration with
Prof. Martin Hyland and Dr. Ignacio Lopez Franco.
The diagram which roughly depicts the above is the following:
\begin{displaymath}
 \xymatrix @C=.6in @R=.6in
{\Mod\ar@{-->}[r]^-{\mathrm{enriched}}
\ar[d]_-{\mathrm{fibred}} & \Comod
\ar[d]^-{\mathrm{opfibred}} \\
\Mon(\ca{V})\ar@{-->}[r]_-{\mathrm{enriched}} &
\Comon(\ca{V}).}
\end{displaymath}
\nocite{Quantum,HopfAlg,Species}

Chapter \ref{VCatsVCocats} moves up a level, aiming to estabish 
essentially the same results as in the previous chapter
but for the `many-object' case of (co)monoids
and (co)modules as explained earlier.
The bicategory of $\ca{V}$-matrices is the base on which
the categories of 
enriched (co)categories and (co)modules are formed,
following until a certain point the development of 
\cite{VarThrEnr} and \cite{KellyLack}. The former
in fact examines categories enriched 
in bicategories via matrices enriched in bicategories,
but for our purposes we restrict to the one-object
case, that of monoidal categories. This approach
of employing matrices presents certain advantages:
it leads to more unified results 
such as existence of limits and colimits, monadicity relations,
local presentability for the categories of $\ca{V}$-graphs,
$\ca{V}$-categories and $\ca{V}$-modules, avoiding
explicit formulas if they are not desired. Regarding this, 
Wolff's much earlier
\cite{Wolff} contains many important explicit constructions
for $\ca{V}$-$\B{Grph}$ and $\ca{V}$-$\B{Cat}$, for a 
symmetric monoidal closed category $\ca{V}$.
In the same underlying framework of $\ca{V}$-matrices,
the category $\ca{V}$-$\B{Cocat}$ of enriched cocategories is 
described (Definition \ref{cocategory}). 
Except from generalizing the concept of comonoids
for a monoidal category, $\ca{V}$-cocategories appear to
have important applications in their own right.
In papers of Lyubashenko, Keller and others (e.g.
\cite{Ainfinitycategories,Ainfinityalgebras,
Equalizerscocompletecocategories}) $A_\infty$-categories,
which are natural generalizations of
$A_\infty$-algebras arising in connection with Floer homology 
and related to mirror symmetry,
are defined as a special kind of differential graded cocategories.
The category of $\ca{V}$-comodules is also accordingly defined, 
and the diagram which summarizes the main results of 
the chapter is
\begin{displaymath}
 \xymatrix @C=.6in @R=.6in
{\ca{V}\text{-}\Mod\ar@{-->}[r]^-{\mathrm{enriched}}
\ar[d]_-{\mathrm{fibred}} & \ca{V}\text{-}\Comod
\ar[d]^-{\mathrm{opfibred}} \\
\ca{V}\text{-}\B{Cat}\ar@{-->}[r]_-{\mathrm{enriched}} &
\ca{V}\text{-}\B{Cocat}.}
\end{displaymath}
Notice that both enrichments are established
via adjoint functors to actions, making use of the fibrational
and opfibrational structure of the categories involved
(though for the bottom one, the hom-functor can be obtained directly 
via an adjoint functor theorem). The same holds for the simpler
case of the previous chapter, for the global category of modules
and comodules.
This is precisely why general fibred adjunctions in 
Part I prove to be essential for the study of the 
particular examples analyzed in this thesis.   

Finally, in Chapter \ref{abstractframework} we utilize 
the results and theoretical patterns 
of the previous two chapters in
order to move `from special to general': we
formulate a definition of an enriched fibration 
and sketch how it is possible to obtain such a formation
in the context of a bicategory or double category.
The structures of importance here are the categories
of \emph{monoids} and \emph{comonoids}, 
\emph{modules} and \emph{comodules} of a 
(pseudo) double category.
We note that the enriched
fibration concept, originally mentioned in \cite{GouzouGrunig}, 
has been studied from an admittedly different point of view by 
Mike Shulman in \cite{Enrichedindexedcats} and also independently
in \cite{Bunge}. However, the main definitions and constructions
diverge from the ones presented here. Other particular references for notions 
employed, such as monoidal bicategories (or monoidal 2-categories)
and pseudomonoids therein, are for example 
\cite{Carmody,CoherenceTricats,Gurski} and
\cite{Monoidalbicats&hopfalgebroids,Marmolejopseudomonads}.
The fundamental definition of a monoidal fibration
was first introduced in \cite{Framedbicats}.
Appropriate references for the theory of pseudo double categories 
for our purposes are 
\cite{Limitsindoublecats,Adjointfordoublecats,ConstrSymMonBicats,Monadsindoublecats}, 
and the original concept of a double category, \emph{i.e.}
a category (weakly) internal in $\B{Cat}$, is traced 
back to \cite{Ehresmanndouble}.
This last part of the dissertation is not as 
detailed as it could be, due to limitations of
the current treatment. 
In the double categories section, most definitions and proofs
are only outlined, 
whereas enrichment in the setting of fibrations 
could be the starting point of an entire
enriched fibred category theory.
The principal function of this final chapter is to 
clarify the occurrence of the main results of 
this work in an abstract environment,
and serve as a guide for future applications.

\part*{PART I}
\chapter{Bicategories}\label{bicategories}
The purpose of this chapter is to provide 
the reader with the necessary
background material 
regarding the theory of bicategories. In this way, 
the related constructions and results used later in the thesis 
can be readily referred to herein.

The original 
definition of a bicategory and a lax functor (`morphism') between
bicategories can be found in B{\'e}nabou's 
\cite{Benabou}. Other references, including the definitions
of transformations and modifications are 
\cite{Categoricalstructures,Handbook1}. 
2-categorical notions 
are here presented as `strictified' versions of the bicategorical ones,
whereas in later chapters the $\B{Cat}$-enriched view
is also addressed. 
Due to coherence for bicategories, we are often able to use
2-categorical machinery and operations such as pasting and mates 
correspondence, directly in the weaker context.
Categories of (co)monads and (co)modules
in bicategories are carefully presented in this chapter,
in order to later be employed as the appropriate
formalization for specific categories of interest.
Regarding 2-category theory, 
see the indicative \cite{Review,2-catcompanion},
whereas \cite{FormalTheoryMonadsI} presents the formal theory of 
monads in 2-categories.

With respect to the notation followed in 
this chapter, note that the multiplication for monads
is denoted by the letter ``$m$'' rather than the usual 
``$\mu$'', since the latter is employed 
for the monad action on their modules. Similarly,
we use ``$\Delta$'' for comultiplication of comonads
and ``$\delta$'' for the coaction on comodules. We also 
prefer the term `(co)module' from the more common 
`(co)algebra' for a (co)monad.

\section{Basic definitions}\label{bicatbasicdefinitions}
\begin{defi}\label{bicategory}
A \emph{bicategory} $\ca{K}$ is specified
by the following data:

$\bullet$ A collection of objects $A,B,C,...$, also called \emph{0-cells}.

$\bullet$ For each pair of objects $A,B$, a category $\ca{K}(A,B)$ whose
objects are called morphisms or \emph{1-cells} and whose arrows
are called \emph{2-cells}. The composition
is called \emph{vertical composition} of $2$-cells
and is denoted by
\begin{displaymath}
\xymatrix @C=.8in
{A\ar @/^4ex/[r]^-f 
\ar[r]|-g
\ar @/_4ex/[r]_-h
\rtwocell<\omit>{<-2>\alpha}
\rtwocell<\omit>{<2>\;\alpha'} &
B}=
\xymatrix @!=.5in {A\rtwocell^f_h{\;\;\;\alpha'\cdot\alpha} & B.}
\end{displaymath}
The identity 2-cell for this composition is
\begin{displaymath}
\xymatrix @!=.5in {A\rtwocell^f_f{\;1_f} & B.}
\end{displaymath}

$\bullet$ For each triple of objects $A,B,C$, a functor
\begin{displaymath}
\circ:\ca{K}(B,C)\times\ca{K}(A,B)\longrightarrow
\ca{K}(A,C)
\end{displaymath}
called \emph{horizontal composition}. It maps a pair
of 1-cells $(g,f)$ to $g\circ f=gf$ and a pair
of 2-cells $(\beta, \alpha)$ to 
$\beta*\alpha$, depicted by
\begin{displaymath}
\xymatrix @!=.5in {A\rtwocell^f_u{\alpha} & 
B\rtwocell^g_v{\beta} & C}=
\xymatrix @!=.5in {A\rtwocell<\omit>{\;\;\;\beta*\alpha}
\ar @/^3ex/[r]^-{gf}
\ar @/_3ex/[r]_{vu} & C.}
\end{displaymath}

$\bullet$ For each object $A\in\ca{K}$, a
$1$-cell $1_A:A\to A$ called the \emph{identity
1-cell} of $A$.

$\bullet$ Associativity constraint:
for each quadruple of objects $A,B,C,D$ of $\ca{K}$, 
a natural isomorphism
\begin{displaymath}
\xymatrix @C=.8in @R=.5in
{\ca{K}(C,D)\times\ca{K}(B,C)\times\ca{K}(A,B)
\ar[r]^-{1\times\circ}\ar[d]_-{\circ\times1} &
\ca{K}(C,D)\times\ca{K}(A,C)\ar[d]^-{\circ}\\
\ca{K}(B,D)\times\ca{K}(A,B)\ar[r]_-{\circ} &
\ca{K}(A,D)\ultwocell<\omit>{\alpha}}
\end{displaymath}
called the \emph{associator}, with components invertible 2-cells
\begin{displaymath}
\alpha_{h,g,f}:(h\circ g)\circ f\stackrel{\sim}
{\longrightarrow}
h\circ(g\circ f).
\end{displaymath}

$\bullet$ Identity constraints: for each pair of objects $A,B$ in $\ca{K}$,
natural isomorphisms
\begin{displaymath}
\xymatrix @R=.7in @C=.1in
{& \B{1}\times\ca{K}(A,B)\cong\ca{K}(A,B)\times\B{1}
\dtwocell<\omit>{<-8>\lambda}
\ar[d]|-{\stackrel{\;}{\cong}} \ar @/_4ex/[dl]_-{I_A \times1}
\ar @/^4ex/[dr]^-{1\times I_B}
 & \\
\ca{K}(A,A)\times\ca{K}(A,B)\ar[r]_-{\circ} &
\ca{K}(A,B)\utwocell<\omit>{<-8>\rho}
& \ca{K}(A,B)\times\ca{K}(B,B)
\ar[l]^-{\circ}}
\end{displaymath}
called the \emph{unitors}, with components
invertible 2-cells
\begin{displaymath}
\lambda_f:1_B\circ f\stackrel{\sim}
{\longrightarrow}f,\quad
\rho_f:f\circ 1_A\stackrel{\sim}{\longrightarrow}f.
\end{displaymath}
Notice that the functor $I_A:\B{1}\to\ca{K}(A,A)$ is 
given by $1_A$ on objects and $1_{1_A}$ on arrows.

The above are subject to the coherence conditions expressed
by the following axioms: for 1-cells
$A\xrightarrow{f}B\xrightarrow{g}C\xrightarrow{h}D\xrightarrow{k}E$, the diagrams
\begin{equation}\label{bicataxiom1}
\xymatrix @R=.4in
{((k\circ  h)\circ g)\circ f\ar[d]_-{\alpha_{kh,g,f}}
\ar[rr]^-{\alpha_{k,h,g}*1_f} && (k\circ(h\circ g))\circ f\ar[d]^-{\alpha_{k,hg,f}}
 \\
(k\circ h)\circ(g\circ f)\ar[dr]_-{\alpha_{k,h,gf}} && k\circ((h\circ g)\circ f)
 \ar[dl]^-{1_k*\alpha_{h,g,f}}\\
& k\circ(h\circ(g\circ f)) &}
\end{equation}
\begin{equation}\label{bicataxiom2}
\xymatrix
{(g\circ 1_B)\circ f\ar[rr]^-{\alpha_{g,1_B,f}}
\ar[dr]_-{\rho_g*1_f} && g\circ(1_B\circ f)
\ar[dl]^-{1_g*\lambda_f} \\
& g\circ f &}
\end{equation}
commute.
\end{defi}
It follows from the functoriality
of the horizontal composition
that for any composable
$1$-cells $f$ and $g$ we have the equality
\begin{displaymath}
 \xymatrix @C=.5in {A\rtwocell^f_f{1_f} & 
B\rtwocell^g_g{1_g} & C}=
\xymatrix @C=.6in {A\rtwocell^{gf}_{vu}{\;\;1_{gf}}
 & C}
\end{displaymath}
and for any $2$-cells $\alpha,\alpha',\beta,\beta'$
as below we have the equality
\begin{displaymath}
 \xymatrix @C=.5in
{{}\ruppertwocell{\alpha}\rlowertwocell{\alpha'}
\ar[r] & \ruppertwocell{\beta}\rlowertwocell{\beta'}
\ar[r] & }=
\xymatrix @C=.5in @R=.5in
{{}\rruppertwocell{\;\;\;\beta*\alpha}\rrlowertwocell{\;\quad\beta'*\alpha'}
\ar[rr] && }
\end{displaymath}
also known as the \emph{interchange law}. 
The above equalities
can also be written
\begin{align*}
 1_g\circ 1_f&=1_{g\circ f}, \\
(\beta'\cdot\beta)*(\alpha'\cdot\alpha)&=
(\beta'*\alpha')\cdot(\beta*\alpha).
\end{align*}

Given a bicategory $\ca{K}$, we may reverse the $1$-cells but not
the $2$-cells and form the bicategory $\ca{K}^\mathrm{op}$,
with $\ca{K}^\mathrm{op}(A,B)=\ca{K}(B,A)$. We may
also reverse only the $2$-cells and 
form the bicategory $\ca{K}^\mathrm{co}$
with $\ca{K}^\mathrm{co}(A,B)=\ca{K}(A,B)^\mathrm{op}$.
Reversing both $1$-cells and $2$-cells yields 
a bicategory $(\ca{K}^\mathrm{co})^\mathrm{op}=
(\ca{K}^\mathrm{op})^\mathrm{co}$.
\begin{examples}\hfill\label{examplesbicat}
\begin{enumerate}
 \item For any category $\ca{C}$ with chosen pullbacks, 
there is the bicategory of spans
$\B{Span}(\ca{C})$. This has the same objects
as $\ca{C}$ and hom-categories
$\B{Span}(X,Y)$ with objects spans
$X\leftarrow A\rightarrow Y$
and arrows $\alpha:A\Rightarrow B$ commutative diagrams
\begin{displaymath}
\xymatrix @R=.05in @C=.4in
{& A\ar[dl] \ar[dr] \ar[dd]_-\alpha & \\
X & & Y \\
& B\ar[ur] \ar[ul] &}
\end{displaymath}
with obvious (vertical) composition.
The horizontal composition is given by pullbacks,
and their universal property defines the constraints
$\alpha,\rho,\lambda$. 
\item Suppose $\ca{C}$ is a regular category,
\emph{i.e.} any morphism factorizes as a strong epimorphism
followed by a monomorphism, and strong
epimorphisms are closed under pullbacks.
The bicategory of relations
$\B{Rel}(\ca{C})$ is defined as $\B{Span}(\ca{C})$,
but its 1-cells are spans $X\leftarrow R\rightarrow Y$
with jointly monic legs, or equivalently relations 
$R\rightarrowtail X\times Y.$ The factorization system is 
required in order to define composition $X\to Y\to Z$, since 
the resulting map from the pullback to $X\times Z$ is 
not necessarily monic.
\item In the bicategory of bimodules $\B{BMod}$
objects are rings, 1-cells from $R$ to $S$ are 
$(R,S)$-bimodules (\emph{i.e.} abelian groups which 
have a left $R$-action and a right
$S$-action that commute with each other), and
2-cells are bimodule maps. The horizontal
composition$\SelectTips{eu}{10}\xymatrix@C=.2in
{R\ar[r]|-{\object@{|}} & S\ar[r]|-{\object@{|}} & T}$is given by 
tensoring over $S$, constructed as in Section 
\ref{Categoriesofmodulesandcomodules}. 
This generalizes to the
bicategory $\ca{V}$-$\B{BMod}$ of 
$\ca{V}$-categories and $\ca{V}$-bimodules,
described in Section \ref{Vbimodulesandmodules}.
\item The bicategory of matrices $\B{Mat}$ has
sets as objects, $X\times Y$-indexed
families of sets as 1-cells from $X$ to $Y$ 
and families of functions as 2-cells. Composition
is given by `matrix multiplication': if $A=(A_{xy}):X\to Y$
and $B=(B_{yz}):Y\to Z$, their composite is given 
by the family of sets $(AB)_{xy}=
\sum_{y}{\big(A_{xy}\times B_{yz}\big)}$.
The enriched version of this
bicategory, $\ca{V}$-$\B{Mat}$, is going
to be extensively employed for the needs of this thesis.
\end{enumerate}
\end{examples}
\begin{defi}\label{laxfunctor}
Given bicategories $\ca{K}$ and $\ca{L}$, a 
\emph{lax functor} $\ps{F}:\ca{K}\to\ca{L}$
consists of the following data:
 
$\bullet$ For any object $A\in\ca{K}$, an object $\ps{F}A\in\ca{L}$.

$\bullet$ For every pair of objects $A,B\in\ca{K}$, a functor
$\ps{F}_{A,B}:\ca{K}(A,B)\to\ca{L}(\ps{F}A,\ps{F}B).$

$\bullet$ For every triple of objects $A,B,C\in\ca{K}$, 
a natural transformation 
\begin{equation}\label{delta}
\xymatrix @C=.8in @R=.57in
{\ca{K}(B,C)\times\ca{K}(A,B)\ar[r]^-\circ
\ar[d]_-{\ps{F}_{B,C}\times\ps{F}_{A,B}} & \ca{K}(A,C)
\ar[d]^-{\ps{F}_{A,C}}\\
\ca{L}(\ps{F}B,\ps{F}C)\times\ca{L}(\ps{F}A,\ps{F}B)
\ar[r]_-\circ & \ca{L}(\ps{F}A,\ps{F}C)
\ultwocell<\omit>{\delta}}
\end{equation} 
with components $\delta_{g,f}:(\ps{F}g)\circ(\ps{F}f)\to
\ps{F}(g\circ f)$, for 1-cells $g:B\to C$ and $f:A\to B.$

$\bullet$ For every object $A\in\ca{K}$, a natural transformation
\begin{equation}\label{gamma}
\xymatrix @C=.7in
{1\ar[r]^-{I_A} \ar @/_2ex/[dr]_-{I_{\ps{F}A}} &
\ca{K}(A,A)\ar[d]^-{\ps{F}_{A,A}} \\
& \ca{L}(\ps{F}A,\ps{F}A)
\ultwocell<\omit>{<1>\gamma}}
\end{equation}
with components $\gamma_A:1_{\ps{F}A}\to\ps{F}(1_A).$

The natural transformations $\gamma$ and $\delta$ have to satisfy
the following coherence axioms: for 1-cells
$A\xrightarrow{f}B\xrightarrow{g}C\xrightarrow{h}D$, 
the diagrams
\begin{equation}\label{laxcond1}
\xymatrix @C=.6in @R=.4in
{(\ps{F}h\circ\ps{F}g)\circ\ps{F}f\ar[r]^-{\delta_{h,g}*1}
\ar[d]_-{\alpha} & \ps{F}(h\circ g)\circ\ps{F}f\ar[d]^-{\delta_{hg,f}} \\
\ps{F}h\circ(\ps{F}g\circ\ps{F}f) \ar[d]_-{1*\delta_{g,f}} & 
\ps{F}((h\circ g)\circ f)\ar[d]^-{\ps{F}\alpha} \\
\ps{F}h\circ\ps{F}(g\circ f)\ar[r]_-{\delta_{h,gf}} &
\ps{F}(h\circ(g\circ f)),}
\end{equation}

\begin{equation}\label{laxcond2}
\xymatrix @C=.5in @R=.4in
{1_{\ps{F}B}\circ\ps{F}f\ar[r]^-{\gamma_B*1}
\ar[d]_-\lambda & \ps{F}(1_B)\circ\ps{F}f
\ar[d]^-{\delta_{1_B,f}} \\
\ps{F}f & 
\ps{F}(1_B\circ f)\ar[l]^-{\ps{F}\lambda}},\quad
\xymatrix @C=.5in @R=.4in
{\ps{F}f\circ 1_{\ps{F}A}\ar[r]^-{1*\gamma_A}
\ar[d]_-\rho & \ps{F}f\circ\ps{F}(1_A)
\ar[d]^-{\delta_{f,1_A}} \\
\ps{F}f & 
\ps{F}(f\circ 1_A)\ar[l]^-{\ps{F}\rho}}
\end{equation}
commute.
\end{defi}
If $\gamma$ and $\delta$ are natural isomorphisms (respectively
identities), then $\ps{F}$ is called a \emph{pseudofunctor}
or \emph{homomorphism} (respectively \emph{strict functor})
of bicategories. Similarly, we can define a \emph{colax 
functor} of bicategories by reversing the direction of $\gamma$ and $\delta$,
sometimes also called \emph{oplax}.
All these kinds of functors between bicategories
can be composed, and this composition obeys 
strict associativity and identity laws. Thus
we obtain categories $\B{Bicat}_l$, $\B{Bicat}_c$,
$\B{Bicat}_{ps}$, $\B{Bicat}_s$ with the same objects 
and arrows lax, colax, pseudo and strict functors respectively.
\begin{defi}\label{laxnattrans}
Consider two lax functors
$\ps{F},\ps{G}:\ca{K}\to\ca{L}$ between
bicategories. A \emph{lax natural transformation}
$\tau:\ps{F}\Rightarrow\ps{G}$ consists of the 
following data:

$\bullet$ For each object 
$A\in\ca{K}$, a morphism $\tau_A:\ps{F}A\to\ps{G}A$ in $\ca{L}$.

$\bullet$ For any pair of objects $A,B\in\ca{K}$,
a natural transformation
\begin{equation}\label{nattranslax}
\xymatrix @C=.4in @R=.4in
{\ca{K}(A,B)\ar[r]^-{\ps{F}_{A,B}}\ar[d]_-{\ps{G}_{A,B}} &
\ca{L}(\ps{F}A,\ps{F}B)\ar[d]^-{\ca{L}(1,\tau_B)} \\
\ca{L}(\ps{G}A,\ps{G}B)\ar[r]_-{\ca{L}(\tau_A,1)} &
\ca{L}(\ps{F}A,\ps{G}B)\ultwocell<\omit>{\tau}}
\end{equation}
with components, for any $f:A\to B$, 2-cells
\begin{equation}\label{tau_f}
\xymatrix @C=.4in @R=.4in
{\ps{F}A\ar[r]^-{\ps{F}f}
\ar[d]_-{\tau_A} & \ps{F}B
\ar[d]^-{\tau_B} \\
\ps{G}A\ar[r]_-{\ps{G}f} & 
\ps{G}B.\ultwocell<\omit>{\tau_f}}
\end{equation}
These data are subject to following axioms:
given any pair of arrows $A\xrightarrow{f}B\xrightarrow{g}C$
in $\ca{K}$, the component $\tau_{g\circ f}$ 
relates to the 2-cells
$\tau_f$, $\tau_g$ by the equality
\begin{equation}\label{compatcomp1}
\xymatrix @C=.4in
{\hole \\
\ps{F}A\ar[r]^-{\ps{F}f}\ar[d]_-{\tau_A}
\ar @/^8ex/[rr]^{\ps{F}(g\circ f)} &
\ps{F}B\ar[r]^-{\ps{F}g}\ar[d]_-{\tau_B} &
\ps{F}C\ar[d]^-{\tau_C}
\lltwocell<\omit>{<4>\delta_{g,f}\;\;} \\
\ps{G}A\ar[r]_-{\ps{G}f} & \ps{G}B\ar[r]_-{\ps{G}g}
\ultwocell<\omit>{\tau_f} & \ps{G}C
\ultwocell<\omit>{\tau_g}}\;
\xymatrix @!=.2in
{\hole \\ =}\;
\xymatrix @C=.4in
{\ps{F}A\ar[rr]^-{\ps{F}(g\circ f)}\ar[d]_-{\tau_A} &&
\ps{F}C\ar[d]^-{\tau_C} \\
\ps{G}A\ar[rr]_-{\ps{G}(g\circ f)}\ar @/_2ex/[dr]_-{\ps{G}f}
&& \ps{G}C\ulltwocell<\omit>{\tau_{g\circ f}\quad}
\lltwocell<\omit>{<-5>\delta'_{g,f}\;\;} \\
& \ps{G}B \ar @/_2ex/[ur]_-{\ps{G}g} & \\
\hole}
\end{equation}
expressing the compatibility of $\tau$
with composition. Also, for any object $A\in\ca{K}$
we have the equality
\begin{equation}\label{compatunit1}
\xymatrix @R=.4in @C=.7in
{\ps{F}A\ar[r]^-{\ps{F}1_A}\ar[d]_-{\tau_A}&
\ps{F}A\ar[d]^-{\tau_A} \\
\ps{G}A\ar[r]^-{\ps{G}1_A}
\ar @/_5ex/[r]_-{1_{\ps{G}A}} & \ps{G}A
\ltwocell<\omit>{<-2.5>\gamma'_A} \ultwocell
<\omit>{\tau_{1_A}\quad}}
\;
\xymatrix @R=.4in
{\hole \\
=}
\;
\xymatrix  @R=.4in @C=.7in
{\ps{F}A\ar @/^5ex/[r]^-{\ps{F}1_A}
\ar[r]_-{1_{\ps{F}A}}
\ar[dr]_-{\tau_A}\ar[d]_-{\tau_A}
\ar@<+4ex>@{}[dr]_(.6)\cong
\ar@<-2ex>@{}[dr]_(.35)\cong
& \ps{F}A\ar[d]^-{\tau_A}
\ltwocell<\omit>{<2.5>\gamma_A\;} \\
\ps{G}A\ar[r]_-{1_{\ps{G}A}} & \ps{G}A}
\end{equation}
expressing the compatibility of $\tau$ with units.
\end{defi}
\begin{rmk*}\hfill

(1) The naturality for the transformation (\ref{nattranslax}) 
can be expressed by the equality
\begin{displaymath}
\xymatrix @C=.6in @R=.5in
{\ps{F}A\ruppertwocell<3>^{\ps{F}g}{\omit}\rlowertwocell<-3>_{\ps{F}f}{\omit}
\rtwocell<\omit>{'\qquad \ps{F}\alpha}\ar[d]_-{\tau_A}
& \ps{F}B \ltwocell<\omit>\ar[d]^-{\tau_B}\\
\ps{G}A \ar[r]_-{\ps{G}f} & \ps{G}B\ultwocell<\omit>{<-1>\tau_f}}\; 
\xymatrix @R=.25in {\hole \\
=}\;
\xymatrix @C=.6in @R=.5in
{\ps{F}A\ar[r]^-{\ps{F}g}\ar[d]_-{\tau_A} &
\ps{F}B\ar[d]^-{\tau_B}\\
\ps{G}A\ruppertwocell<3>^{\ps{G}g}{\omit}\rlowertwocell<-3>_{\ps{G}f}{\omit}
\rtwocell<\omit>{'\qquad \ps{G}\alpha} & \ps{G}B\ltwocell<\omit>
\ultwocell<\omit>{<1>\tau_g}}
\end{displaymath}
for any 2-cell $\alpha:f\Rightarrow f$.

(2) Using pasting operations properties
(see Section \ref{2cats}),
the equality (\ref{compatcomp1})
can be expressed by the commutativity of
\begin{displaymath}
\xymatrix @R=.25in @C=.5in
{\ps{G}g\circ(\ps{G}f\circ \tau_A) \ar[r]^-{\ps{G}g*\tau_f}
\ar[d]_-{\alpha^{-1}} & \ps{G}g\circ(\tau_B\circ \ps{F}f) \ar[r]^-{\alpha^{-1}} &
(\ps{G}g\circ\tau_B)\circ \ps{F}f \ar[d]^-{\tau_g* \ps{F}f} \\
(\ps{G}g\circ \ps{G}f)\circ\tau_A\ar[d]_-{\delta'_{g,f}*\tau_A}
&& (\tau_C\circ \ps{F}g)\circ \ps{F}f\ar[d]^-\alpha \\
\ps{G}(g\circ f)\circ\tau_A\ar[dr]_-{\tau_{g\circ f}} &&
\tau_C\circ(\ps{F}g\circ \ps{F}f)\ar[dl]^-{\tau_C*\delta_{g,f}} \\
& \tau_C\circ \ps{F}(g\circ f) &}
\end{displaymath}
inside the hom-category $\ca{L}(\ps{F}A,\ps{G}C)$.

(3) Similarly, the equality (\ref{compatunit1}) can
be expressed by the commutativity of
\begin{displaymath}
\xymatrix @C=.6in @R=.15in
{1_{\ps{G}A}\circ\tau_A\ar[r]^-{\gamma'_A*\tau_A}
\ar[d]_-{\lambda} & \ps{G}(1_A)\circ\tau_A\ar[dd]^-{\tau_{1_A}} \\
\tau_A\ar[d]_-{\rho^{-1}} & \\
\tau_A\circ 1_{\ps{F}A}\ar[r]_-{\tau_A*\gamma_A} 
& \tau_A\circ\ps{F}(1_A).}
\end{displaymath}
\end{rmk*}
A lax natural transformation $\tau$ is a 
\emph{pseudonatural} transformation (respectively
\emph{strict}) when all the 2-cells $\tau_f$ 
as in (\ref{tau_f}) are isomorphisms (respectively identities).
Also, a \emph{colax} (or \emph{oplax}) natural transformation
is equipped
with a natural transformation in the opposite direction
of (\ref{nattranslax}). Note that 
between either lax or colax functors 
$\ps{F},\ps{G}:\ca{K}\to\ca{L}$
of bicategories, we can consider both lax
and colax natural transformations.
\begin{defi}\label{modification}
Consider lax functors $\ps{F},\ps{G}:\ca{K}\to\ca{L}$
between bicategories, and 
$\tau,\sigma:\ps{F}\Rightarrow\ps{G}$ two lax natural transformations.
A \emph{modification} $m:\tau\Rrightarrow\sigma$ is 
a family of 2-cells
\begin{displaymath}
 \xymatrix{\ps{F}A\rrtwocell^{\tau_A}_{\sigma_A}{\quad m_A} && 
\ps{G}A}
\end{displaymath}
for every object $A$ of $\ca{K}$, such that
\begin{equation}\label{modificcond}
\xymatrix @C=.6in @R=.6in
{\ps{F}A\ar[r]^-{\ps{F}f}
\dtwocell^{\;\;\sigma_A}_{\tau_A}{\omit}
& \ps{F}B\ar[d]^-{\sigma_B} \\
\ps{G}A\ar[r]_-{\ps{G}f}
\utwocell<\omit>{m_A} & \ps{G}B
\ultwocell<\omit>{\sigma_f}}\;
\xymatrix @!=.2in {\hole \\
=}\;
\xymatrix @C=.6in @R=.6in
{\ps{F}A\ar[r]^-{\ps{F}f}\ar[d]_-{\tau_A} &
\ps{F}B\dtwocell^{\;\;\sigma_B}_{\tau_B}{\omit} \\
\ps{G}A\ar[r]_-{\ps{G}f} & \ps{G}B.\ultwocell
<\omit>{\tau_f}\utwocell<\omit>{m_B}}
\end{equation}
\end{defi}

It is not hard to define composition
of natural transformations and modifications,
and respective identities. Therefore,
for any two bicategories $\ca{K},\ca{L}$ 
there is a functor bicategory 
$\B{Lax}(\ca{K},\ca{L})=\B{Bicat}_l(\ca{K},\ca{L})$
of lax functors, lax natural transformations and modifications,
and it has a sub-bicategory $\B{Hom}(\ca{K},\ca{L})=
\B{Bicat}_{ps}(\ca{K},\ca{L})$ of pseudofunctors,
pseudonatural transformations and modifications.
In fact, the \emph{tricategory} $\B{Hom}$ is a very 
important 3-dimensional category of bicategories 
(see \cite{CoherenceTricats,Gurskitricats}).
Notice that $\B{Hom}(\ca{K},\ca{L})$ is a strict bicategory,
\emph{i.e.} 2-category when $\ca{L}$
is a 2-category.

\section{Monads and modules in bicategories}\label{monadsinbicats}
\begin{defi}\label{monadbicat}

A \emph{monad} in a bicategory $\ca{K}$
consists of an object $B$ together with an endomorphism
$t:B\to B$ and 2-cells $\eta:1_B\Rightarrow t$,
$m:t\circ t\Rightarrow t$ called the 
\emph{unit} and \emph{multiplication} respectively,
such that the diagrams
\begin{displaymath}
\xymatrix @R=.25in
{(t\circ t)\circ t\ar[rr]^-{\alpha_{t,t,t}}
\ar[d]_-{m\circ1} && t\circ(t\circ t)\ar[d]^-{1\circ m} \\
t\circ t\ar[dr]_-m && t\circ t\ar[ld]^-m \\
& t &}\qquad\mathrm{and}\qquad
\xymatrix @R=.65in @C=.5in
{1_B\circ t\ar[r]^-{\eta\circ 1}
\ar[dr]_-{\lambda_t} & t\circ t
\ar[d]^-{m} & t\circ1_B \ar[l]_-{1\circ\eta}
\ar[ld]^-{\rho_t} \\
& t &}
\end{displaymath}
commute.
\end{defi}
Equivalently,
a monad in a bicategory $\ca{K}$ is a 
lax functor $\ps{F}:\B{1}\to\ca{K}$, where
$\B{1}$ is the terminal bicategory with 
a unique 0-cell $\star$ (one 1-cell
and one 2-cell). This amounts to
an object $\ps{F}(\star)=B\in\ca{K}$
and a functor 
\begin{displaymath}
\ps{F}_{\star,\star}:\B{1}(\star,\star)
\to\ca{K}(B,B)
\end{displaymath}
which picks up an endoarrow $t:B\to B$
in $\ca{K}$. The natural transformations 
$\delta$ and $\gamma$
of the lax functor give
the multiplication and the unit of $t$
\begin{displaymath}
m\equiv\delta_{1_\star,1_\star}:t\circ t\to t
\quad\textrm{and}\quad
\eta\equiv\gamma_{\star}:1_B\to t
\end{displaymath}
and the axioms for $\ps{F}$ 
give the monad axioms for $(t,m,\eta)$.
\begin{rmk}\label{laxfunctorspreservemonads}
As mentioned earlier, lax functors
between bicategories compose. Therefore
if $\ps{G}:\ca{K}\to\ca{L}$ is a lax functor 
between bicategories, the composite
\begin{displaymath}
\B{1}\xrightarrow{\;\ps{F}\;}\ca{K}\xrightarrow{\;\ps{G}\;}\ca{L}
\end{displaymath}
is itself a lax functor from $\B{1}$
to $\ca{L}$, hence
defines a monad. In other words,
if $t:B\to B$ is a monad in the bicategory
$\ca{K}$, then $\ps{G}t:\ps{G}B\to\ps{G}B$
is a monad in the bicategory $\ca{L}$, \emph{i.e.}
lax functors preserve monads.
\end{rmk}
For an object $B$ in the bicategory $\ca{K}$ and
a monad $t:B\to B$, there is an induced ordinary
monad (\emph{i.e.} in $\B{Cat}$) on the 
hom-categories, namely `post-composition with $t$'. 
Explicitly, for any 0-cell $A$ we have an endofunctor 
\begin{displaymath}
 \ca{K}(A,t):\ca{K}(A,B)\longrightarrow\ca{K}(A,B)
\end{displaymath}
which is the mapping
\begin{displaymath}
 \xymatrix
 {A\rtwocell^f_g{\alpha} & B
 \ar @{|->}[r] & A\rtwocell^f_g{\alpha}
 & B\ar[r]^-t & B}
\end{displaymath}
for objects and morphisms in $\ca{K}(A,B)$.
The multiplication and unit of the monad
$\bar{m}$ and $\bar{\eta}$,
are natural transformations with components,
for each $f:A\to B$ in $\ca{K}(A,B)$,
\begin{displaymath}
 \xymatrix @R=.02in
 {\hole \\ 
 \bar{m}_f\quad\textrm{=}} 
 \xymatrix @R=.02in
 {&& B \ar @/^/[dr]^-t & \\
 A\ar[r]^-f & B \rrtwocell<\omit>{m}
 \ar @/^/[ur]^-t \ar@/_4ex/[rr]_-t && B,}\quad
 \xymatrix @R=.02in
 {\hole \\
 \bar{\eta}_f\quad\textrm{=}}
 \xymatrix @R=.03in
 {\hole \\
 A \rtwocell<\omit>{\quad\;\rho_f^{-1}}
 \ar @/^4ex/[rr]^-f
 \ar @/_/[dr]_-f && B. \\
 & B \ar @/_2ex/[ur]_-t \ar @/^2ex/[ur]^-{1_B}
 \urtwocell<\omit>{\eta} &}
\end{displaymath}
Now, consider the Eilenberg-Moore category
$\ca{K}(A,B)^{\ca{K}(A,t)}$ of
$\ca{K}(A,t)$-algebras. It has as objects
1-cells $f:A\to B$ equipped with an \emph{action}
$\mu:\ca{K}(A,t)(f)\Rightarrow f$, 
\emph{i.e.} a 2-cell
\begin{equation}\label{actionbicat}
 \xymatrix @R=.02in
 {& B\ar@/^/[dr]^-t & \\
 A\ar@/^/[ur]^-f\ar@/_3ex/[rr]_-f
 \rrtwocell<\omit>{\mu} && B}
\end{equation}
compatible with the 
multiplication and unit
of the monad $\ca{K}(A,t)$:
\begin{equation}\label{axiomslefttmodule}
 \xymatrix @R=.03in
 {&B\ar@/^/[r]^-t \rrtwocell<\omit>{<2>m}
 \ar@/_3ex/[drr]_-t &
 B \ar @/^/[dr]^-t & \\
 A\ar@/^/[ur]^-{f}\rrtwocell<\omit>{<1>\mu}
 \ar@/_4ex/[rrr]_-t &&& B}\quad
\xymatrix @R=.03in {\hole \\
\textrm{=}}\quad
 \xymatrix @R=.03in
 {&B\ar@/^/[r]^-t &
 B \ar @/^/[dr]^-t & \\
 A \ar@/_3ex/[urr]_-f
 \rrtwocell<\omit>{<-1>\mu}
 \ar@/^/[ur]^-{f}
 \ar@/_4ex/[rrr]_-t & \rrtwocell<\omit>{<1>\mu}
 && B,}
\end{equation}
\begin{displaymath}
 \xymatrix @R=.03in @C=.45in
 { & B\ar@/^1.5ex/[dr]^-{1_B}
 \ar@/_1.5ex/[dr]_-t
 \drtwocell<\omit>{\eta} & \\
 A\ar@/^/[ur]^-f
 \ar@/_3ex/[rr]_-f
 \rrtwocell<\omit>{\mu} &&
 B\quad\textrm{=}}\quad
  \xymatrix @R=.03in
 { & B\ar@/^/[dr]^-{1_B} & \\
 A\ar@/^/[ur]^-f
 \ar@/_3ex/[rr]_-f
 \rrtwocell<\omit>{\;\;\rho_f} &&
 B.}\qquad
\end{displaymath}
Such an 1-cell $f$ together with an action $\mu$
is called a \emph{$t$-module} or \emph{$t$-algebra}.
An arrow $(f,\mu)\xrightarrow{\tau}(g,\mu')$
is a 2-cell $\tau:f\Rightarrow g$ in $\ca{K}$
compatible with the actions, \emph{i.e.} such that
\begin{equation}\label{axiomlefttmodulemorphism}
\xymatrix @R=.03in
{& B\ar@/^/[dr]^-t & \\
A\urtwocell^f_g{\tau}
\ar@/_6ex/[rr]_g 
\rrtwocell<\omit>{<2>\mu'}&& B\quad\textrm{=}}
\quad
\xymatrix @R=.03in
{& B\ar@/^/[dr]^-t & \\
A\ar@/^/[ur]^-f 
\ar@/_0.5ex/[rr]_-f 
\rrtwocell<\omit>{<-1>\mu}
\ar@/_6ex/[rr]_-g
\rrtwocell<\omit>{<4>{\tau}} && B,}
\end{equation}
called a \emph{morphism of $t$-modules}.
\begin{defi}\label{lefttmodules}
 The category of Eilenberg-Moore algebras
 $\ca{K}(A,B)^{\ca{K}(A,t)}$ for
 $t:B\to B$ a monad in the bicategory $\ca{K}$
 is the \emph{category of left $t$-modules} with 
 domain $A$, denoted by ${^A_t}\Mod$.
\end{defi}
We may similarly 
define the category $\Mod{^B_s}$ of \emph{right $s$-modules}
with codomain $B$. It is the category of Eilenberg-Moore algebras
$\ca{K}(A,B)^{\ca{K}(s,B)}$ for $s:A\to A$ a monad
in the bicategory $\ca{K}$, where $\ca{K}(s,B)$ is the 
monad `pre-composition with $s$'. Moreover,
the above endofunctors combined define the monad
\begin{displaymath}
\ca{K}(s,t):
\xymatrix@R=.02in
{\ca{K}(A,B)\ar[r] & \ca{K}(A,B) \\
(A\xrightarrow{f}B)\ar@{|->}[r]
& (A\xrightarrow{s}A\xrightarrow{f}B\xrightarrow{t}B)}
\end{displaymath}
on $\ca{K}(A,B)$, and the category of algebras
$\ca{K}(A,B)^{\ca{K}(s,t)}$ is now called the 
category of \emph{right $s$/left $t$-bimodules},
${_t}\Mod{_s}$.
\begin{rmk}\label{modswithdom1}
In the classical case where $\ca{K}$=$\B{Cat}$, 
the term left (respectively right) 
`$t$-algebra' is more commonly restricted
to those with domain (respectively codomain)
the unit category $\B{1}$.
A left $t$-module with domain $\B{1}$, \emph{i.e.} 
a functor $f:\B{1}\to B$, is then
identified with the corresponding object $X$ in
the category $B$, and the actions $\mu:t(X)\to X$
and maps $\tau:X\to Y$ are morphisms in $B$.
The category $\ca{K}(\B{1},B)^{\ca{K}(\B{1},t)}$
is then denoted by $B^t$.

Notice that in the above presentation,
there is a certain circularity in the definition
of modules for a monad in an arbitrary bicategory $\ca{K}$.
More precisely, the Eilenberg-Moore category of algebras 
which is used in the very
definition of the category of modules in this 
abstract setting (Definition \ref{lefttmodules}),
is in reality a particular 
example of a category of modules for 
a monad in $\ca{K}=\B{Cat}$.
However, this could be easily avoided: in 
Kelly-Street's \cite{Review}, an action
of a monad $t$ in a 2-category 
is defined to be a 2-cell as in (\ref{actionbicat})
satisfying the specified axioms, and maps
are defined accordingly. Hence, 
the fact that we now identify from the beginning 
this structure with the Eilenberg-Moore category
for an ordinary monad does not affect the 
level of generality.  
\end{rmk}
Dually to the above, we have the following
definitions.
\begin{defi}\label{comonadbicat}
A \emph{comonad} in a bicategory $\ca{K}$
consists of an object $A$ together
with an endoarrow $u:A\to A$ and 2-cells
$\Delta:u\Rightarrow u\circ u$, $\varepsilon:u\Rightarrow 1_A$
called the \emph{comultiplication} and \emph{counit} respectively,
such that the diagrams 
\begin{displaymath}
\xymatrix @R=.25in
{& u\ar[dr]^-\Delta
\ar[dl]_-\Delta & \\
u\circ u \ar[d]_-{\Delta\circ1} &&
u\circ u\ar[d]^-{1\circ\Delta} \\
(u\circ u)\circ u\ar[rr]_-{\alpha_{u,u,u}} &&
u\circ(u\circ u),}\qquad
\xymatrix @R=.65in @C=.5in
{1_A\circ u \ar[dr]_-{\lambda_u} & u\circ u
\ar[l]_-{\varepsilon\circ 1}
\ar[r]^-{1\circ\varepsilon} 
& u\circ1_A 
\ar[ld]^-{\rho_u} \\
& u\ar[u]_-{\Delta} &}
\end{displaymath}
commute. 
\end{defi}
Notice that a comonad in the bicategory $\ca{K}$
is precisely a monad in the bicategory $\ca{K}^\mathrm{co}$.
Similarly to before, for an object A and a comonad
$u:A\to A$ in a bicategory $\ca{K}$, there is an induced
comonad in $\B{Cat}$ between hom-categories
\begin{displaymath}
 \ca{K}(u,B):\ca{K}(A,B)\longrightarrow\ca{K}(A,B)
\end{displaymath}
which precomposes objects and arrows in $\ca{K}(A,B)$
with the 1-cell $u:A\to A$. The axioms for a comonad
follow again from those of $u$, 
hence we can form the category
of coalgebras $\ca{K}(A,B)^{\ca{K}(u,B)}$. Its objects
are 1-cells $h:A\to B$ equipped with a \emph{coaction}
$\delta:h\Rightarrow\ca{K}(u,B)(h)$, \emph{i.e.} a 2-cell
\begin{displaymath}
 \xymatrix @R=.03in
 {A \ar@/^3ex/[rr]^-h
 \ar@/_/[dr]_-u
 \rrtwocell<\omit>{\delta} && B \\
 & A \ar@/_/[ur]_-h &}
\end{displaymath}
compatible with the comultiplication
and counit of $\ca{K}(u,B)$, and arrows
$\sigma:(h,\delta)\to(k,\delta')$
are 2-cells $\sigma:h\Rightarrow k$
compatible with the coactions $\delta$ and $\delta'$.
\begin{defi}\label{rightucomodules}
 The category of Eilenberg-Moore coalgebras
 $\ca{K}(A,B)^{\ca{K}(u,B)}$ for 
 a comonad $u:A\to A$ in the bicategory $\ca{K}$
 is the \emph{category of right $u$-comodules}
or \emph{coalgebras} with codomain $B$, denoted by $\Comod{^B_u}$.
\end{defi}
Similarly, for a comonad $v:B\to B$ we can define the category
${^A_v}\Comod$ of \emph{left $v$-comodules} with domain $A$
as the category
$\ca{K}(A,B)^{\ca{K}(A,v)}$ as well as the category 
of \emph{right $u$/left $v$-bicomodules}
${_v}\Comod{_u}$
as the category of coalgebras of the comonad
`pre-composition with $u$ and post-composition with $v$',
$\ca{K}(u,v)$, on $\ca{K}(A,B)$.
\begin{rmk*}
As mentioned in Remark \ref{modswithdom1},
for the classical case $\ca{K}=\B{Cat}$
the term `$v$-coalgebra' is more
commonly restricted to the case that
the domain of a left $v$-comodule
(or respectively the codomain of 
a right $u$-comodule) is the unit
category $\B{1}$. The coalgebra $h:\B{1}\to B$
is then identified with the object $Z$
of the category which is picked out by 
the functor $h$, and we denote
$\ca{K}(\B{1},B)^{\ca{K}(\B{1},v)}$ for a comonad $v$
as $B^v$.
\end{rmk*}
\begin{defi}\label{monadfunctor}
A \emph{(lax) monad functor}
between two monads $t:B\to B$
and $s:C\to C$ in a bicategory
consists of an 1-cell
$f:B\to C$ between the 0-cells of the monads
together with a 2-cell
\begin{displaymath}
 \xymatrix
{B\ar[r]^-f \ar[d]_-t 
\drtwocell<\omit>{\psi}
& C\ar[d]^-s \\
B\ar[r]_-f & C}
\end{displaymath}
satisfying compatibility conditions with 
multiplications and units:
\begin{displaymath}
 \xymatrix @R=.02in @C=.5in
{& C\ar@/^/[dr]^-s && \\
B\ar@/^/[ur]^-f \ar@/_/[dr]^-t \rrtwocell<\omit>{\psi}
\ar@/_5ex/[ddrr]_-t \ddrrtwocell<\omit>{<2.8>m} &&
C\ar[r]^-s & C\\
& B\ar@/_/[ur]^-f \ar@/_/[dr]^-t \rrtwocell<\omit>{\psi} && \\
&& B \ar@/_/[uur]_-f &}
\;\xymatrix @R=.02in @C=.5in
{\hole \\ =}\;
\xymatrix @R=.02in @C=.5in
{&& C\ar@/^/[dr]^-s & \\
B \rrtwocell<\omit>{<3.5>{\psi}}
\ar[r]^-f \ar@/_/[ddr]_-t & C\ar@/^/[ur]^-s 
\ar@/_3ex/[rr]_-s  \rrtwocell<\omit>{\;m'} && C \\
\hole \\
& B\ar@/_3ex/[uurr]_-f &&}
\end{displaymath}
\begin{displaymath}
 \xymatrix @R=.03in @C=.5in
{ & C\ar@/^/[dr]^-{1_C} \ar@/_2ex/[dr]_-s
\drtwocell<\omit>{<.5>\eta'} & \\
B\ar@/^/[ur]^-f \ar@/_/[dr]_-t \rrtwocell<\omit>{\psi} && C \\
& B\ar@/_/[ur]_f &}
\xymatrix @R=.02in @C=.5in
{\hole \\ =}
\xymatrix @R=.02in @C=.5in
{\hole \\
B\rtwocell^{1_B}_t{\eta} & B\ar[r]^f & C}
\end{displaymath}
\end{defi}
If the 2-cell $\psi$ is in the opposite
direction, and the diagrams are accordingly
modified, we have a \emph{colax} monad functor
(or monad \emph{opfunctor}) between two monads.
There are also appropriate 
notions of monad
natural transformations for monads
in bicategories, not essential for
the purposes of this thesis,
which can be found in detail in \cite{FormalTheoryMonadsI}.
Because of the correspondence between
monads and lax functors from the terminal
bicategory, we obtain a bicategory 
$\B{Mnd}(\ca{K})\equiv[\B{1},\ca{K}]_l$.

In search of an induced functor 
between categories of modules,
we will need some well-known results
related to maps of monads on ordinary categories. The 
following definition is just a special case 
of the above definition for $\ca{K}=\B{Cat}$.
\begin{defi}\label{mapsofmonads}
 Let $T=(T,m,\eta)$ be a monad
on a category $\ca{C}$ and $T'=(T',m',\eta')$
a monad on a category $\ca{C}'$. A \emph{lax map
of monads} $(\ca{C},T)\to(\ca{C}',T')$
is a functor $F:\ca{C}\to\ca{C}'$ together with 
a natural transformation
\begin{displaymath}
 \xymatrix @C=.4in @R=.4in
{\ca{C}\ar[r]^-T\ar[d]_-F & \ca{C}\ar[d]^-F \\
\ca{C}'\ar[r]_-{T'} & \ca{C}' \ultwocell<\omit>{\psi} }
\end{displaymath}
making the diagrams
\begin{displaymath}
\xymatrix
{T'T'F\ar[r]^-{T'\psi}\ar[d]_-{m'F} &
T'FT\ar[r]^-{\psi T} & FTT\ar[d]^-{Fm} \\
T'F\ar[rr]_-\psi && FT,}\quad\quad
\xymatrix @C=.4in
{F\ar[rr]^-{F\eta}\ar[dr]_-{\eta'F} && FT \\
& T'F\ar[ur]_-\psi &}
\end{displaymath}
commute. A \emph{strong} or \emph{pseudo} (respectively \emph{strict})
map of monads is a lax
map $(F,\psi)$ in which $\psi$ is an isomorphism
(respectively the identity). 
\end{defi}
A very important property of lax maps
of monads is that they give rise to maps
between categories of algebras: 
a lax map
$(F,\psi):(\ca{C},T)\to(\ca{C}',T')$
induces a functor
\begin{displaymath}
 F_*:\xymatrix @R=.02in
{\ca{C}^T\ar[r] & \ca{C}'^{T'} \\
(X,a) \ar@{|->}[r] & (FX,Fa\circ \psi_X)}
\end{displaymath}
which means that if $TX\xrightarrow{a}X$ is the action
of the $T$-algebra $X$, then 
$T'FX\xrightarrow{\psi_X}FTX\xrightarrow{Fa}FX$ is
the action which makes $FX$ into a $T'$-algebra. 
In fact, there is a bijection
between the two structures.
\begin{lem}\label{mapsbetweenalgs}
 Let $T$ and $T'$ be monads on categories
$\ca{C}$ and $\ca{C}'$. There is a one-to-one correspondence
betweeen lax maps of monads $(\ca{C},T)\to(\ca{C}',T')$ and 
pairs of functors $(K,F)$ such that the square
\begin{displaymath}
 \xymatrix @R=.4in @C=.4in
{\ca{C}^T\ar[r]^-K\ar[d]_-U & \ca{C}'^{T'}\ar[d]^-{U'} \\
\ca{C}\ar[r]_-F & \ca{C}'}
\end{displaymath}
commutes, where $U,U'$ are the forgetful functors.
Explicitly, a lax map $(F,\psi)$ corresponds bijectively
to the pair $(F_*,F)$.
\end{lem}
We can apply this lemma to obtain functors
between the categories of modules for a monad in 
a bicategory as described above. More specifically,
by Remark \ref{laxfunctorspreservemonads}
lax functors between bicategories preserve monads, and
this in a
sense carries over to the categories of their modules.
\begin{prop}\label{laxfunctorbetweenmodules}
 If $\ps{F}:\ca{K}\to\ca{L}$ is a lax functor between
two bicategories and $t:B\to B$ a monad in $\ca{K}$, 
there is an induced functor
\begin{displaymath}
 \Mod(\ps{F}_{A,B}):\ca{K}(A,B)^{\ca{K}(A,t)}\longrightarrow
\ca{L}(\ps{F}A,\ps{F}B)^{\ca{L}(\ps{F}A,\ps{F}t)}
\end{displaymath}
between the category of left $t$-modules
in $\ca{K}$ and the category of 
left $\ps{F}t$-modules in $\ca{L}$, which maps
a $t$-module $f:A\to B$ to the $\ps{F}t$-module 
$\ps{F}f:\ps{F}A\to\ps{F}B$.
Moreover,
the following diagram commutes:
\begin{equation}\label{diagthis}
 \xymatrix @C=.9in @R=.4in
 {\ca{K}(A,B)^{\ca{K}(A,t)}\ar[r]^-{\Mod(\ps{F}_{A,B})}
 \ar[d]_-{U} & \ca{L}(\ps{F}A,\ps{F}B)^{\ca{L}(\ps{F}A,\ps{F}t)}
 \ar[d]^-U \\
 \ca{K}(A,B)\ar[r]_-{\ps{F}_{A,B}} & \ca{L}(\ps{F}A,\ps{F}B) }
\end{equation}
\end{prop}
\begin{proof}
The endofunctor
$\ca{K}(A,t)$ is an ordinary monad on the hom-category
$\ca{K}(A,B)$, and since $\ps{F}t$ is a monad
in $\ca{L}$, the endofunctor
$\ca{L}(\ps{F}A,\ps{F}t)$
is also a monad on the hom-category
$\ca{L}(\ps{F}A,\ps{F}B)$. 
 
In order to apply Lemma \ref{mapsbetweenalgs},
we need to exhibit a map of monads
as in Definition \ref{mapsofmonads}. In fact,
we have a functor 
\begin{displaymath}
 \ps{F}_{A,B}:\ca{K}(A,B)\to\ca{L}(\ps{F}A,\ps{F}B)
\end{displaymath}
and also 
a natural transformation
\begin{displaymath}
 \xymatrix@R=.6in @C=.7in
{\ca{K}(A,B)\ar[r]^-{\ca{K}(A,t)}\ar[d]_-{\ps{F}_{A,B}}
& \ca{K}(A,B)\ar[d]^-{\ps{F}_{A,B}} \\
\ca{L}(\ps{F}A,\ps{F}B)\ar[r]_-{\ca{L}(\ps{F}A,\ps{F}t)} 
& \ca{L}(\ps{F}A,\ps{F}B) \ultwocell<\omit>{\psi} }
\end{displaymath}
with components, for any 1-cell $f:A\to B$,
\begin{displaymath}
 \xymatrix @R=.05in
 {& \ps{F}B \ar@/^/[dr]^-{\ps{F}t} & \\
 \ps{F}A\ar@/^/[ur]^-{\ps{F}f} 
 \ar@/_3ex/[rr]_-{\ps{F}(t\circ f)}
 \rrtwocell<\omit>{\quad\delta_{t,f}} 
 && FB}
\end{displaymath}
where $\ps{F}_{A,B}$ and $\delta$ come
from the definition of a lax functor.
Hence, we do have a map of monads
\begin{displaymath}
 (\ps{F}_{A,B},\psi):(\ca{K}(A,B),\ca{K}(A,t))\longrightarrow
 (\ca{L}(\ps{F}A,\ps{F}B),\ca{L}(\ps{F}A,\ps{F}t))
\end{displaymath}
which induces a functor
between the categories of algebras
\begin{displaymath}
 (\ps{F}_{A,B})_*\equiv \B{Mod}(\ps{F}_{A,B})
\end{displaymath}
such that the diagram (\ref{diagthis}) commutes.
\end{proof}
In a completely dual way, we can verify that colax functors
between bicategories preserve comonads, and that 
they also induce
functors between the corrresponding categories 
of comodules. 

\section{2-categories}\label{2cats}

A (strict) \emph{2-category} is a bicategory
in which all constraints are identities, \emph{i.e.} $\alpha,\rho,\lambda=1$.
In this case, the horizontal composition
is strictly associative and unitary and the
axioms (\ref{bicataxiom1}), (\ref{bicataxiom2})
hold automatically. Consequently, 
the collection of 0-cells and 1-cells
form a category on its own. 
\begin{examples*}\hfill
\begin{enumerate}
 \item The collection of all (small) categories, functors and 
natural transformations forms the 2-category $\B{Cat}$, 
which is a leading example in category theory.
\item Monoidal categories, 
(strong) monoidal functors and monoidal natural tra\-nsfo\-rma\-tions
form the 2-category $\Mon\B{Cat}$ (see Chapter \ref{monoidalcategories}).
\item If $\ca{V}$ is a monoidal category, $\ca{V}$-enriched categories, 
$\ca{V}$-functors and $\ca{V}$-na\-tu\-ral tra\-nsfo\-rma\-tions
form the 2-category $\ca{V}$-$\B{Cat}$ (see Chapter \ref{enrichment}).
\item Fibrations and opfibrations over $\caa{X}$, 
(op)fibred functors and (op)fibred natural transformations form the 
2-categories $\B{Fib}(\caa{X})$ and $\B{OpFib}(\caa{X})$ 
(see Chapter \ref{fibrations}).
\item Suppose $\caa{E}$ is a category with finite limits. There is a 
2-category $\B{Cat}(\caa{E})$ with objects categories internal to $\caa{E}$,
which have an $\caa{E}$-object of objects and an $\caa{E}$-object of morphisms.
Instances of this
are ordinary categories ($\caa{E}=\B{Set}$), double categories
($\caa{E}=\B{Cat}$) and crossed modules ($\caa{E}=\B{Grp}$).
\end{enumerate}
\end{examples*}
A (strict) \emph{2-functor} is a strict functor between
2-categories, whereas a (strict) \emph{2-natural transformation}
is a strict natural transformation between 2-functors.

Since a 2-category is a special case
of a bicategory, all kinds of functors (and natural
transformations) described in Section \ref{bicatbasicdefinitions}
can be defined in this context. They now give
rise to categories $\B{2\text{-}Cat},
\B{2\text{-}Cat}_{ps},\B{2\text{-}Cat}_l,
\B{2\text{-}Cat}_c$. Moreover, for $\ca{K},\ca{L}$ 2-categories,
there are various kinds of functor 2-categories:
$[\ca{K},\ca{L}]$ with 2-functors, 2-natural transformations and modifications,
$\B{Lax}(\ca{K},\ca{L})_s$ 
with lax functors, strict 2-natural transformations and modifications,
$[\ca{K},\ca{L}]_\mathrm{ps}$ with pseudofunctors,
pseudonatural transformations and modifications etc.
Evidently, this implies 
that all flavours of categories with objects 2-categories
are in reality 2-categories themselves, and moreover
$\B{2\textrm{-}Cat}$ is a paradigmatic example
of a \emph{3-category}.
\begin{rmk}\label{icons}
We saw earlier how bicategories and lax/colax/pseudo functors
form ordinary categories, and also how structures
like $\B{Lax}(\ca{K},\ca{L})$ or $\B{Hom}(\ca{K},\ca{L})$ 
of appropriate functors, natural transformations
and modifications are in fact bicategories
themselves (or functor 2-categories in the strict case like above). 
However bicategories, lax functors
and (co)lax natural transformations fail to form a 2-category.
Even restricting from bicategories to 2-categories and 
from lax functors to 2-functors does not suffice in order
to form a 2-dimensional structure with a weaker notion
of natural transformation. This is due to problems arising
regarding the vertical and horizontal composition of 2-cells.

The above is thoroughly discussed in Lack's \cite{Icons}, where 
\emph{icons} are employed so that 
bicategories and lax functors can be the objects 
and 1-cells of a 2-category
$\B{Bicat}_2$.
More precisely, the 2-cells $\tau:\ps{F}\Rightarrow\ps{G}$
are colax natural transformations (see Definition \ref{laxnattrans})
whose components $\tau_A:\ps{F}A\to\ps{G}A$
are identities, hence the name \emph{I}denitity \emph{C}omponent 
\emph{O}plax \emph{N}atural transformation. That reduces the 
natural transformation in the opposite direction of 
(\ref{nattranslax}) to the simpler
\begin{displaymath}
 \xymatrix @C=1in
{\ca{K}(A,B)\rtwocell^{\ps{F}_{A,B}}_{\ps{G}_{A,B}}{\tau} &
\ca{L}(\ps{F}A,\ps{F}B)}
\end{displaymath}
which satisfies accordingly simplified axioms.
Icons were firstly 
introduced in \cite{2nervesbicats} and they allow the study 
of bicategories in a plain 2-dimensional setting, with
applications in various contexts.
\end{rmk}
In many cases, various concepts used in ordinary category
theory are special instances of abstract notions
defined in an arbitrary 2-category or bicategory.
For example, the usual notion of equivalence of categories is 
just a special case of the following notion
of (internal) equivalence in any bicategory, applied to $\B{Cat}$. 
\begin{defi*}
 A 1-cell $f:A\to B$
in a bicategory $\ca{K}$ is an \emph{equivalence}
when there exist another 1-cell $g:B\to A$ and 
invertible 2-cells
\begin{displaymath}
 \xymatrix @R=.05in
{& B\ar@/^/[dr]^-g & \\
A\ar@/^/[ur]^-f \ar@/_3ex/[rr]_{1_A}
\rrtwocell<\omit>{\cong} && A,}
\qquad
\xymatrix @R=.05in
{& A\ar@/^/[dr]^-f & \\
B\ar@/^/[ur]^-g \ar@/_3ex/[rr]_{1_B}
\rrtwocell<\omit>{\cong} && B}
\end{displaymath}
\emph{i.e.} isomorphisms
$gf\cong 1_A$ and $fg\cong1_B$ in $\ca{K}$. We write $f\simeq g$.
\end{defi*}
Just as the notion of equivalence of categories 
can be internalized in any 2-category,
the notion of equivalence for 2-categories
can be internalized in any 3-category in an appropriate
way, hence we obtain the following definition for $\B{2\textrm{-}Cat}$.
\begin{defi*}
A 2-functor $T:\ca{K}\to\ca{L}$ between two 2-categories 
$\ca{K}$ and $\ca{L}$ is a \emph{(strict) 2-equivalence}
if there is some 2-functor $S:\ca{L}\to\ca{K}$ and 
isomorphisms $\B{1}\cong TS$, 
$ST\cong\B{1}$.
We write $\ca{K}\backsimeq\ca{L}$.
\end{defi*}
There is a well-known proposition
which gives conditions for a 2-functor to 
be a 2-equivalence.
\begin{prop}\label{knownequivalenceprop}
The 2-functor $T:\ca{K}\to\ca{L}$ is an equivalence if 
and only if $T$ is fully faithful, \emph{i.e.}
$T_{A,A'}:\ca{A}(A,A')\to\ca{B}(TA,TA')$ is 
an isomorphism of categories for every $A,A'\in\ca{A}$,
and essentially
surjective on objects, \emph{i.e.} 
every object $B\in\ca{L}$ is isomorphic 
to $TA$ for some $A\in\ca{A}$.
\end{prop}
The appropriate weaker version 
for the notion of equivalence in the context of 
bicategories is the following.
\begin{defi*}
 A \emph{biequivalence} between bicategories $\ca{K}$ and $\ca{L}$ 
consists of two pseudofunctors $\ps{F}:\ca{K}\to\ca{L}$
and $\ps{G}:\ca{L}\to\ca{K}$ and pseudonatural
transformations $\ps{G}\ps{F}\to1_\ca{K}$,
$1_\ca{L}\to\ps{F}\ps{G}$ which are invertible
up to isomorphism. Equivalently,
$\ps{F}:\ca{K}\to\ca{L}$ is a biequivalence if and only if
it is locally an equivalence, \emph{i.e.} each $\ps{F}_{A,B}:\ca{K}(A,B)\to
\ca{L}(\ps{F}A,\ps{F}B)$ is an equivalence of categories,
and every $B\in\ca{L}$ is equivalent to $\ps{F}A$ for some $A$.
\end{defi*}
Notice that the second statement in fact is equivalent 
to the first, only if the axiom of choice is assumed.
This has to do with the fact that in general, there exist
notions of \emph{strong} and \emph{weak} equivalence 
between categories, and every weak equivalence being 
a strong one is equivalent to the axiom of choice.

The coherence theorems for bicategories 
and their homomorphisms are of 
great importance, and have been fundamental for
the development of higher category theory. 
In particular, it is asserted that certain
diagrams involving the constraint isomorphisms
of bicategories will always commute. Coherence allows
us to replace any bicategory with an
appropriate strict 2-category, so that various
situations are greatly simplified. This ensures
for example that the pasting diagrams, commonly used when
working with 2-categories, can also be used for bicategories.
\begin{thm}\label{coherenceforbicats}
 Every bicategory is biequivalent to a 2-category.
\end{thm}
The proof is based on a 
bicategorical generalization of the Yoneda
Lemma (see Street's \cite{FibrationsinBicats}), 
which states that the embedding
\begin{displaymath}
 \xymatrix @R=.05in @C=.5in
{\ca{K}\ar[r] & \B{Hom}(\ca{K}^\op,\B{Cat}) \\
A\ar@{|->}[r] & \ca{K}^\op(A,-)}
\end{displaymath}
is locally an equivalence, hence any bicategory $\ca{K}$
is biequivalent to a full sub-2-category of
$\B{Hom}(\ca{K}^\op,\B{Cat})$.

Using the notion of \emph{category enriched graph},
which is a particular case of a $\ca{V}$-graph studied in 
detail in Section \ref{Vgraphs} and originates from 
Wolff's \cite{Wolff}, we can actually construct a 
strict functor of bicategories between $\ca{K}$
and a 2-category, which is a biequivalence.
Hence the coherence theorem can be stated in the 
following more
conventional way. 
\begin{thm}[Coherence for Bicategories]\label{alldiagscommute}
 In a bicategory, every 2-cell diagram made up
of expanded instances of $\alpha,\lambda,\rho$ and 
their inverses must commute.
\end{thm}
A more detailed description of coherence for 
bicategories and homomorphisms and further references
can be found in
\cite{MacLane-Pare,CoherenceTricats,Gurskitricats}.
Also, the approach of Joyal-Street in \cite{BraidedTensorCats}
for monoidal categories can be modified to show
the above result.

We now turn to composition of 2-cells in a general 2-category.
Additionally to the usual vertical and horizontal composition,
we consider a special case of horizontal composition
which acts 
on a 1-cell and a 2-cell 
and produces a 2-cell. Explicitly, if we
identify any morphism $f:A\to B$
with its identity 2-cell $1_f$, we can
form the composite 2-cell
\begin{displaymath}
 \xymatrix
{A\ar[r]^-f & B\rtwocell<\omit>{\alpha}\ar@/^3ex/[r]^-g \ar@/_3ex/[r]_-h 
& C\ar[r]^-k & D}
\equiv
\xymatrix @C=.5in
{A\rtwocell<\omit>{\;1_f}\ar@/^3ex/[r]^-f \ar@/_3ex/[r]_-f& 
B\rtwocell<\omit>{\alpha}\ar@/^3ex/[r]^-g \ar@/_3ex/[r]_-h &
C\rtwocell<\omit>{\;1_k}\ar@/^3ex/[r]^-k \ar@/_3ex/[r]_-k & D}
\end{displaymath}
called \emph{whiskering} $\alpha$ by $f$ and $k$.
It is denoted by $k\alpha f:kgf\Rightarrow khf$
and really is the horizontal composite
$1_k*\alpha*1_f$.

The various kinds of composition 
can be combined to give a
more general operation of \emph{pasting} 
(see \cite{Benabou,Review,Quantum}). 
The two basic situations are
\begin{displaymath}
\xymatrix @R=.4in @C=.5in
{A\ar[rr]^-f\ar[dr]_-h
\rrtwocell<\omit>{<3>\alpha} && \ar[dr]^-g & \\
& \ar[rr]_-k\ar[ur]_-l
\rrtwocell<\omit>{<-3>\beta} && B}
\quad\mathrm{and}\quad
\xymatrix @R=.4in @C=.5in
{ & \ar[rr]^-r\ar[dr]^-u
\rrtwocell<\omit>{<3>\delta} && D \\
C\ar[rr]_-s\ar[ur]^-p
\rrtwocell<\omit>{<-3>\gamma} &&\ar[ur]_-t &}
\end{displaymath}
For the first, we can first whisker $\alpha$
by $g$ and also $\beta$ by $h$,
\begin{displaymath}
\xymatrix @C=.5in
{A\rtwocell<\omit>{\alpha}\ar@/^3ex/[r]^-f \ar@/_3ex/[r]_-{lh} 
&\ar[r]^g & B}
\quad\mathrm{and}\quad
\xymatrix @C=.5in
{A\ar[r]^-h &\rtwocell<\omit>{\beta}\ar@/^3ex/[r]^-{gl} \ar@/_3ex/[r]_-k 
& B}
\end{displaymath}
in order to obtain two vertically composed
2-cells
\begin{displaymath}
\xymatrix @C=1in
{A \ar @/^6ex/[r]^-{gf}
\ar[r]|-{glh}
\ar @/_6ex/[r]_-{kh}
\rtwocell<\omit>{<-3>\;g\alpha}
\rtwocell<\omit>{<3>\;\beta h}
& B}=\xymatrix @C=.7in
{A \ar @/^4ex/[r]^-{gf}
\ar @/_4ex/[r]_-{kh}
\rtwocell<\omit>{\quad\;\;\beta h\cdot g\alpha}
& B}
\end{displaymath}
which is called the \emph{pasted composite} of the 
original diagram. Following a similar procedure,
we can deduce that the pasted composite of the 
second diagram is the 2-cell
\begin{displaymath}
\xymatrix @C=.7in
{C \ar @/^3ex/[r]^-{rp}
\ar @/_3ex/[r]_-{ts}
\rtwocell<\omit>{\quad\;t\gamma\cdot\delta p}
& D.}
\end{displaymath}
One can generalize the pasting operation
further, in order to compute multiple composites like
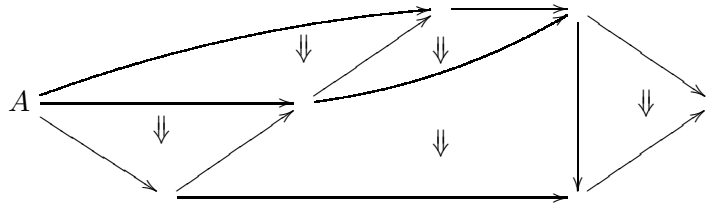
\begin{displaymath}
\xymatrix @C=.6in
{&\rrtwocell<\omit>{<3>} &&\ar[r] &\ar[dr]\ar[dd] &\\
A\ar @/^/[urrr] \ar[rr]
\ar[dr]\rrtwocell<\omit>{<2>} &&\ar[ur]
\ar @/_/[urr] \rrtwocell<\omit>{<-4>}\rrtwocell<\omit>
{<3>}&&\rtwocell<\omit>{}& \\
&\ar[rrr]\ar[ur]&&&\ar[ur] &}
\end{displaymath}
It is a general fact that the result of
pasting is independent of the choice of 
the order in which the composites are taken, \emph{i.e.}
of the way it is broken down into basic
pasting operations. This is clear in simple cases,
and can be proved inductively in the general case,
after an appropriate formalization in terms of
polygonal decompositions of the disk.
A formal 2-categorical pasting theorem, showing 
that the operation is 
well-defined using Graph Theory, can be found in Power's 
\cite{Pastingtheorem}. 

We finish this section with some classical
notions in 2-categories and their properties,
which are going to be of use later in the thesis.
\begin{defi}\label{adjunction2cat}
An \emph{adjunction} in a 2-category $\ca{K}$ 
consists of 0-cells $A$ and $B$, 
1-cells $f:A\to B$ and $g:B\to A$
and 2-cells $\eta:1_A\Rightarrow g\circ f$ and 
$\varepsilon:f\circ g\Rightarrow 1_B$ subject
to the usual triangle equations:
\begin{gather*}
 \xymatrix @R=.08in @C=.45in
{& A \ar@/^/[dr]^-{1_A}
\ar@/^/[dd]_-f & \\
B \ar@/^/[ur]^-g \ar@/_/[dr]_{1_B}
\rtwocell<\omit>{\varepsilon} &
\rtwocell<\omit>{\eta} & A \\
& B\ar@/_/[ur]_-g &}
\xymatrix @R=.08in @C=.45in
{\hole \\
= \\
\hole}
\xymatrix @R=.08in @C=.6in
{\hole \\
B\rtwocell^g_g{\;1_g} & A,\\
\hole} \\
\xymatrix @R=.08in @C=.45in
{& B \ar@/^/[dr]^-{f} & \\
A \ar@/^/[ur]^-{1_A} \ar@/_/[dr]_{f}
\rtwocell<\omit>{\eta} &
\rtwocell<\omit>{\varepsilon} & B \\
& B\ar@/_/[ur]_-{1_B}
\ar@/_/[uu]^-g &}
\xymatrix @R=.08in @C=.45in
{\hole \\
= \\
\hole}
\xymatrix @R=.08in @C=.6in
{\hole \\
A\rtwocell^f_f{\;1_f} & B\\
\hole}
\end{gather*}
which can be written as
$(g\varepsilon)\cdot(\eta g)=1_g$ and 
$(\varepsilon f)\cdot(f\eta)=1_f$.
\end{defi}
The standard notation for an adjunction
is $f\dashv g:B\to A$. The same definition 
applies in case $\ca{K}$
is a bicategory, with the associativity and identity
constraints suppressed because of coherence.
\begin{rmk*}
 Suppose that $f\dashv g$ is an adjunction
in a 2-category (or bicategory) $\ca{K}$ and
$\ps{F}:\ca{K}\to\ca{L}$ is a pseudofunctor.
Then $\ps{F}f\dashv\ps{F}g$ in $\ca{L}$,
with unit
\begin{displaymath}
 1_{\ps{F}A}\cong\ps{F}(1_A)\xrightarrow{\ps{F}\eta}
\ps{F}(gf)\cong\ps{F}g\circ\ps{F}f
\end{displaymath}
and counit
\begin{displaymath}
 \ps{F}f\circ\ps{F}g\cong
\ps{F}(fg)\xrightarrow{\ps{F}\varepsilon}
\ps{F}(1_B)\cong1_{\ps{F}B}
\end{displaymath}
where the isomorphisms are components
of the constraints $\gamma$ and $\delta$ of the 
pseudofunctor $\ps{F}$.
In other words, pseudofunctors
preserve adjunctions.

In particular, we can apply the representable 2-functor
$\ca{K}(X,-):\ca{K}\to\B{Cat}$ for any $0$-cell $X$ and 
obtain an adjunction in $\B{Cat}$
\begin{displaymath}
 \xymatrix @C=.5in
 {\ca{K}(X,A) \ar @<+.8ex>[r]^-{f\circ\text{-}}
 \ar@{}[r]|-\bot
 & \ca{K}(X,B) \ar @<+.8ex>[l]^-{g\circ\text{-}}}
 \end{displaymath} 
with bijections $\phi_{h,k}:\ca{K}(X,B)(f\circ h,k)
\cong\ca{K}(X,A)(h,g\circ k)$ natural in $h$ and $k$.
We can also apply the contravariant 
representable 2-functor $\ca{K}(-,X):\ca{K}^\op\to\B{Cat}$
which produces an (ordinary) adjunction 
$(\text{-}\circ g)\dashv(\text{-}\circ f)$.
This is sometimes called the \emph{local
approach} to adjunctions, and of course by usual Yoneda lemma 
arguments we can reobtain the \emph{global approach}
of Definition \ref{adjunction2cat}.
\end{rmk*}
\begin{defi}\label{mapofadunctions}
 Suppose that $f\dashv g:B\to A$ and 
$f'\dashv g':B'\to A'$ are two adjunctions
in a 2-category $\ca{K}$. 
A \emph{map of adjunctions} from $(f\dashv g)$
to $(f'\dashv g')$ consists of a pair
of 1-cells $(h:A\to A',k:B\to B')$
such that both squares
\begin{displaymath}
 \xymatrix @C=.4in @R=.4in
{A\ar[r]^-f\ar[d]_-h & B\ar[d]^-k\ar[r]^-g & A\ar[d]^-h \\
A'\ar[r]_-{f'} & B'\ar[r]_-{g'} & A'}
\end{displaymath}
commute, and $h\eta=\eta' h$ or equivalently
$k\varepsilon=\varepsilon' k$ for the units and counits of the adjunctions.
\end{defi}
The equivalence of the two conditions becomes
evident as a particular case of the mate correspondence
described below.
\begin{prop}\label{mates}
Let 
$f\dashv g:A\to B$ and 
$f'\dashv g':B'\to A'$ be two adjunctions in a 2-category
(or bicategory) $\ca{K}$, 
and $h:A\to A'$, $k:B\to B'$ 1-cells. There is a natural 
bijection between 2-cells
\begin{displaymath}
\xymatrix @C=.4in @R=.4in
{A\ar[r]^-{h}\ar[d]_-{f}
\drtwocell<\omit>{m}
& A'\ar[d]^-{f'}\\
B\ar[r]_-{k} & B'}\quad \textrm{and}\quad
\xymatrix @C=.4in @R=.4in
{B\drtwocell<\omit>{\nu}
\ar[d]_-{k}\ar[r]^-g & A\ar[d]^-{h}\\
B'\ar[r]_-{g'} & A'}
\end{displaymath}
where $\nu$ is given by the composite 
\begin{displaymath}
\xymatrix @C=.5in
{B\ar[r]^-g \ar@/_2ex/[dr]_-{1_B}
\drtwocell<\omit>{<-1>\varepsilon} & A\ar[d]^-f
\ar[r]^-h \drtwocell<\omit>{m} & A'\ar[d]^-{f'}
\ar@/^2ex/[dr]^-{1_{A'}}
\drtwocell<\omit>{<1>{\eta'}} & \\
& B\ar[r]_-k & B'\ar[r]_-{g'} & A'}
\end{displaymath}
and $m$ is given by the composite
\begin{displaymath}
\xymatrix @C=.5in
{A\ar[r]^-{1_A} \ar@/_2ex/[dr]_-f & A\ar[r]^-h & 
A' \ar@/^2ex/[dr]^-{f'} & \\
\urtwocell<\omit>{\eta} & B
\urtwocell<\omit>{\nu}
\ar[r]_-k \ar[u]_-g & 
B'\urtwocell<\omit>{\varepsilon'}
\ar[r]_-{1_{B'}} \ar[u]_-{g'}
 & B'.}
\end{displaymath}
\end{prop}
We call the 2-cells
\emph{mates} under the adjunctions 
$f\dashv g$ and $f'\dashv g'$. 
In particular, for $h=k=1$, there is 
a bijection between 2-cells
$\mu:f\Rightarrow f'$ and $\nu:g\Rightarrow g'$.

Using pasting operation, we can deduce
that the 2-cells above are explicitly
given by the composites 
\begin{equation}\label{mate1}
\xymatrix 
{\nu:hg\ar@{=>}[r]^-{\eta'hg} &
g'f'hg\ar@{=>}[r]^-{g'\mu g} &
g'kfg\ar@{=>}[r]^-{g'k\varepsilon} & g'k,}
\end{equation}
\begin{equation}\label{mate2}
\xymatrix  
{\mu:f'h\ar@{=>}[r]^-{f'h\eta} &
f'hgf\ar@{=>}[r]^-{f'\nu f} &
f'g'kf\ar@{=>}[r]^-{\varepsilon'kf} & kf.}
\end{equation}

In Section \ref{monadsinbicats}
we studied monads and modules in bicategories.
In the special case when $\ca{K}$ is the 2-category  
$\B{2\textrm{-}Cat}$, the monad $t$ is usually called
a \emph{doctrine} (or 2-\emph{monad}) and consists of a
2-functor $D:\ca{B}\to\ca{B}$ with 2-natural transformations
$\eta:1_\ca{B}\to D$, $m:D^2\to D$ satisfying the usual axioms.
A $D$-algebra is considered in the strict sense, 
although most often
the 2-functor has domain $\B{1}$ so it is identified
with an object $A$ in $\ca{B}$, 
as explained in Remark \ref{modswithdom1}.
For morphisms of $D$-algebras, however, the lax
ones are the more usual to appear in nature.

Explicitly, for $D$-algebras $(A,\mu)$ and $(A',\mu')$, 
a \emph{lax morphism} (or \emph{lax $D$-functor})
is a pair $(f,\bar{f})$ where $f:A\to A'$ is a morphism
in $\ca{B}$ and $\bar{f}$ is a 2-cell
\begin{displaymath}
 \xymatrix
{DA\ar[r]^-\mu \ar[d]_-{Df} & A\ar[d]^-f \\
DA'\ar[r]_-{\mu'} & A'\ultwocell<\omit>{\bar{f}}}
\end{displaymath}
satisfying compatibility axioms with the multiplication
and unit of $D$. If $\bar{f}$ is an isomorphism,
then this is a \emph{strong} morphism of $D$-algebras,
whereas if $\bar{f}$ is the identity then we have 
\emph{strict} morphism which coincides with the `$D$-modules
morphism' as defined in the previous section.
If we reverse the direction of $\bar{f}$ and accordingly
in the axioms, we have a \emph{colax} 
morphism. Clearly a strong morphism of $D$-algebras
is both lax and colax.

With appropriate notions of $D$-natural transformations,
we can form 2-categories $D$-$\B{Alg}_l$ with lax, 
$D$-$\B{Alg}_c$ with colax,
$D$-$\B{Alg}_s$ with strong and $D$-$\B{Alg}\equiv\ca{B}^D$
with strict morphisms. All the above can be found in detail
in \cite{Review,2-dimmonadtheory}, and the main results come
from the so-called \emph{doctrinal adjunction}.
\begin{thm}\label{laxoplaxadjoints}
 Let $f\dashv g$ be an adjunction in a 2-category $\ca{C}$ 
and let $D$ be a 2-monad on $\ca{C}$. There is a bijective
correspondance between 2-cells $\bar{g}$ which make $(g,\bar{g})$
into a lax $D$-morphism and 2-cells $\bar{f}$ which make
$(f,\bar{f})$ into a colax $D$-morphism.
\end{thm}
\begin{prop}\label{proplaxoplaxadjoints}
 There is an adjunction $(f,\bar{f})\dashv(g,\bar{g})$ in the 2-category
$D$-$\B{Alg}_l$ if and only if $\;f\dashv g\;$ in the 2-category $\ca{C}$
and $\bar{f}$ is invertible.
\end{prop}
The inverse of $\bar{f}$ is in fact the mate of $\bar{g}$,
and both proofs rely solely on the properties of the mates correspondence.
More precisely, 2-cells of the form
\begin{displaymath}
 \xymatrix
 {DB\ar[r]^-\beta \ar[d]_-{Dg} & B
 \ar[d]^-g\\
 DA\ar[r]_-\alpha & A\ultwocell<\omit>{\bar{g}}}
 \quad\mathrm{and}\quad
 \xymatrix
 {DA\ar[r]^-\alpha \ar[d]_-{Df}
 \drtwocell<\omit>{\bar{f}} & A\ar[d]^-f \\
 DB\ar[r]_-\beta & B}
\end{displaymath}
which are mates under the adjunctions
$Df\dashv Dg$ and $f\dashv g$ are considered, and
all details can be found in \cite{Doctrinal}.

An application of these facts is going to be exhibited
in the next chapter, for the 2-monad $D$ on $\B{Cat}$ which gives rise to 
monoidal categories.

\chapter{Monoidal Categories}\label{monoidalcategories}
This chapter presents the basic theory of 
monoidal categories, with particular emphasis on
the categories of monoids/comonoids and modules/comodules.
These structures are of central importance for our purposes,
since ultimately they form a first example 
of the enriched fibration notion (see Chapter \ref{enrichmentofmonsandmods}).
Key references are \cite{BraidedTensorCats,Quantum,MonComonBimon},
and the monoidal category 
$\ca{V}=\Mod_R$ of $R$-modules and 
$R$-linear maps for a commutative ring $R$
serves as a motivating illustration of our results. 

A recurrent process in this treatment is the 
establishment of the existence of certain adjoints 
for various purposes, such as monoidal closed structures,
free monoid and cofree comonoid constructions, enriched hom-functors
etc. This also justifies the significance of locally 
presentable categories (see \cite{LocallyPresentable})
in our context, since their properties allow
the application of adjoint functor theorems
in a straightforward way.
Below we quote some relevant, well-known 
results which will be employed throughout the thesis, 
so that we do not interrupt the main progress.

The following simple adjoint functor theorem
which can be found in Max Kelly's 
\cite[5.33]{Kelly}
ensures that any cocontinuous functor
with domain a locally presentable category has 
a right adjoint.
\begin{thm}\label{Kelly}
If the cocomplete $\ca{C}$ has a small dense subcategory, every
cocontinuous $S:\ca{C}\to\ca{B}$ has a right adjoint.
\end{thm}
The standard way of determining adjunctions via 
representing objects is connected
with the following `Adjunctions with
a parameter' theorem (see \cite[Theorem IV.7.3]{MacLane}),
which defines the important notion of a \emph{parametrized
adjunction}.
\begin{thm}\label{parametrizedadjunctions}
Suppose that, for a functor of two variables $F:\ca{A}\times\ca{B}\to\ca{C}$,
there exists an adjunction
\begin{equation}\label{fixedparameter}
 \xymatrix @C=.6in
{\ca{A}\ar@<+.8ex>[r]^-{F(-,B)}\ar@{}[r]|-\bot &
\ca{C}\ar@<+.8ex>[l]^-{G(B,-)}}
\end{equation}
for each object $B\in\ca{B}$, with an isomorphism
 $\ca{C}(F(A,B),C)\cong\ca{A}(A,G(B,C)),$
natural in $A$ and $C$. Then, there is 
a unique way to assign an arrow 
\begin{displaymath}
 G(h,1):G(B',C)\longrightarrow G(B,C)
\end{displaymath}
for each $h:B\to B'$ in $\ca{B}$ and $C\in\ca{C}$, 
so that $G$ becomes a functor of two variables
$\ca{B}^\mathrm{op}\times\ca{C}\to\ca{A}$
for which the above bijection 
is natural in all three variables $A$, $B$, $C$.
\end{thm}
The unique choice of $G(h,-)$ to realize the above, 
coming from the fact that it is a \emph{conjugate}
natural transformation to 
$F(-,h):F(-,B)\Rightarrow F(-,B')$, is given for example
by the commutative 
\begin{equation}\label{parameterdefinining}
 \xymatrix @C=1.1in
{G(B',-)\ar@{-->}[r]^-{G(h,-)}\ar[d]_-{\eta} & 
G(B,-) \\
G(B,F(G(B',-),B))\ar[r]^-{G(1,F(G(1,-),h))} &
G(B,F(G(B',-),B'))\ar[u]_-{G(1,\varepsilon')}}
\end{equation}
where $\eta$ is the unit of $F(-,B)\dashv G(B,-)$ and 
$\varepsilon'$ the counit of $F(-,B')\dashv G(B',-)$.

The first instance of a parametrized adjoint in this chapter
is the internal hom in a monoidal category, which will 
play a decisive role. In \cite{Multivariableadjunctions},
more advanced ideas on multivariable adjunctions are presented.  

\section{Basic definitions}\label{Basicdefinitions}

\begin{defi*}
 A \emph{monoidal
category} $(\ca{V},\otimes,I,a,l,r)$ is a category $\ca{V}$
equipped
with a functor $\otimes:\ca{V}\times\ca{V}\to\ca{V}$ called 
the \emph{tensor product}, an object $I$ of $\ca{V}$ called the 
\emph{unit object}, and natural isomorphisms 
with components
\begin{displaymath}
a_{A,B,C}:(A\otimes B)\otimes C
\xrightarrow{\sim}
A\otimes(B\otimes C), 
\end{displaymath}
\begin{displaymath}
r_A:A\otimes I\xrightarrow{\sim}A,
\quad l_A:I\otimes A\xrightarrow{\sim}A 
\end{displaymath}
called the \emph{associativity} constraint, the \emph{right unit}
constraint and the \emph{left unit} co\-nstraint
respectively, 
subject to two coherence axioms: the following diagrams
\begin{displaymath}
\xy  0;/r.22pc/: 
(0,20)*{(A \otimes B) \otimes(C \otimes D)}="1"; 
(30,0)*{ A \otimes(B \otimes(C \otimes D))}="2"; 
(25,-25)*{\quad A \otimes((B \otimes C) \otimes D),}="3"; 
(-25,-25)*{(A \otimes(B \otimes C)) \otimes D}="4"; 
(-30,0)*{((A \otimes B) \otimes C) \otimes D}="5"; 
{\ar^-{a_{A,B,C\otimes D}} "1";"2"} 
{\ar_-{1 \otimes a_{B,C,D}} "3";"2"} 
{\ar^-{a_{A,B\otimes C,D}} "4";"3"} 
{\ar_-{a_{A,B,C}\otimes 1} "5";"4"} 
{\ar^-{a_{A\otimes B,C,D}} "5";"1"} 
\endxy 
\end{displaymath}
\begin{displaymath}
\xymatrix@C=.35in
{(A\otimes I)\otimes B\ar[rr]^-{a_{A,I,B}}
\ar[dr]_-{r_A\otimes 1} && A\otimes(I\otimes B)\ar[dl]^-{1\otimes l_B}\\
& A\otimes B &}
\end{displaymath}
commute.
\end{defi*}
Given a monoidal category $\ca{V}$, we can define
a bicategory $\ca{K}$ with one object $\star$
by setting $\ca{K}(\star,\star)=\ca{V}$, $\circ_{\star,\star,\star}=\otimes$
and $\alpha$, $\lambda$, $\rho$ given by the constraints
of the monoidal category. 
Conversely, any such
one-object bicategory yields a monoidal category. In fact,
for any object $A$ in a bicategory $\ca{K}$,
the hom-category $\ca{K}(A,A)$ is equipped with a 
monoidal structure induced by the 
horizontal composition of the bicategory:
\begin{equation}\label{tensorcirc}
 \otimes:
\xymatrix @R=.05in
{\ca{K}(A,A)\times\ca{K}(A,A)
\ar[r] & \ca{K}(A,A)\qquad\quad \\
\quad\quad(A\xrightarrow{g}A,\;A\xrightarrow{f}A)
\ar @{|->}[r] & A\xrightarrow{g\otimes f:=g\circ f}A}
\end{equation}
The unit object is the identity 1-cell $I=1_A$
and the associativity and left/right unit constraints
come from the associator and the left/right unitors
of the bicategory $\ca{K}$. The coherence axioms follow
in a straightforward way from those of a bicategory.

Due to this correspondence, various results of 
the previous chapter are of relevance
to the theory of monoidal categories. In particular,
coherence for bicategories 
(Theorems \ref{coherenceforbicats} and 
\ref{alldiagscommute}) ensures that 
monoidal categories are also `coherent'. 
The coherence theorem for monoidal
categories first appeared in Mac Lane's 
\cite{natass&comm}.
A formulation of it states that 
every diagram which consists of arrows obtained by 
repeated applications of the functor $\otimes$ to 
instances of $a, r, l$ and their inverses (the so-called 
`expanded instances') and 1 commutes. This essentially
allows one to work as if $a, r, l$ are all identities.
This is derived from the fact 
that any monoidal category 
is monoidally equivalent (via a strict monoidal functor) 
to a \emph{strict}
monoidal category, where $a, r, l$ are identities.
\nocite{Ignacionotes}

Notice that if $\ca{V}$ is a monoidal category,
then its opposite category $\ca{V}^\mathrm{op}$
is also monoidal with the same tensor product
$\otimes^\mathrm{op}$. Some authors
call `opposite monoidal category' the
\emph{reverse} category $\ca{V}^\mathrm{rev}$,
which is $\ca{V}$ with 
$A\otimes^\mathrm{rev}B=B\otimes A$, 
$a^\mathrm{rev}=a^{-1}$,
$l^\mathrm{rev}=l$ and $r^\mathrm{rev}=r$.

A \emph{braiding} $c$ for a monoidal category
$\ca{V}$ is a natural isomorphism
\begin{displaymath}
 \xymatrix @C=.6in
 {\ca{V}\times\ca{V}\ar[r]^-{\otimes}
 \ar[d]_-{\mathrm{sw}}
 \drtwocell<\omit>{c} & \ca{V} \\
 \ca{V}\times\ca{V} \ar@/_3ex/[ur]_-{\otimes} &}
\end{displaymath}
with components invertible arrows
$c_{A,B}:A\otimes B\xrightarrow{\sim}B\otimes A$
for all $A,B\in\ca{V}$, where $sw$ 
switches the entries of the pair.
These isomorphisms
satisfy the coherence axioms expressed by the commutativity of
\begin{displaymath}
\xymatrix @C=.03in @R=.22in
{& A\otimes(B\otimes C)\ar[rrrr]^-{c_{A,B\otimes C}} &&&& 
(B\otimes C)\otimes A\ar[dr]^-{a_{B,C,A}} &\\
(A\otimes B)\otimes C\ar[ur]^-{a_{A,B,C}}\ar[dr]_-{c_{A,B}\otimes 1} &&&&&& 
B\otimes(C\otimes A)\\
& (B\otimes A)\otimes C\ar[rrrr]_-{a_{B,A,C}} 
&&&& B\otimes(A\otimes C),\ar[ur]_-{1\otimes c_{A,C}} &}
\end{displaymath}
\begin{displaymath}
 \xymatrix @C=.03in @R=.22in
{& (A\otimes B)\otimes C\ar[rrrr]^-{c_{A\otimes B,C}} &&&& 
C\otimes(A\otimes B)\ar[dr]^-{a^{-1}_{C,A,B}} &\\
A\otimes(B\otimes C)\ar[ur]^-{a^{-1}_{A,B,C}}\ar[dr]_-{1\otimes c_{B,C}} &&&&&& 
(C\otimes A)\otimes B\\
& A\otimes(C\otimes B)\ar[rrrr]_-{a^{-1}_{A,C,B}} 
&&&& (A\otimes C)\otimes B.\ar[ur]_-{c_{A,C}\otimes1} &}
\end{displaymath}
A \emph{braided monoidal category}
is a monoidal category with a chosen braiding.
A \emph{symmetry} $s$ for a monoidal category $\ca{V}$ 
is a braiding $s$ with components 
\begin{displaymath}
 s_{A,B}:A\otimes B\xrightarrow{\sim}B\otimes A
\end{displaymath}
which also satisfies the commutativity of 
\begin{displaymath}
\xymatrix @R=.4in
{A\otimes B \ar[rr]^-{=}\ar[dr]_-{s_{A.B}} && A\otimes B \\
& B\otimes A,\ar[ur]_-{s_{B,A}} &}
\end{displaymath}
which expresses that $s^{-1}_{A,B}=s_{B,A}$. Because of
this, only the one hexagon from the definition
of the braiding is needed to define a symmetry.

A monoidal category 
with a chosen symmetry is called \emph{symmetric}. 
Coherence theorems for braided and symmetric monoidal categories 
again state that any (braided) symmetric monoidal 
category is (braided) symmetric monoidally equivalent 
to a \emph{strict} (braided) symmetric monoidal category,
see \cite{BraidedTensorCats}.
\begin{examples*}
 
 $(1)$ A special collection of examples
 called \emph{cartesian monoidal categories}
 is given by considering any category
 with finite products, taking $\otimes=\times$
 and $I=1$ the terminal object. The constraints
 $a,l,r$ are the canonical isomorphisms
 induced by the universal property of products.
 Important particular cases of this are the categories
 $\B{Set}$ of (small) sets, $\B{Cat}$ of categories,
 $\B{Gpd}$ of groupoids, $\B{Top}$ of topological 
 spaces etc. All these examples are in fact
 symmetric monoidal categories. 
 
 $(2)$ The category $\B{Ab}$ of abelian groups 
 and group homomorphisms
 is a symmetric monoidal category with the usual
 tensor product $\otimes$ of abelian groups and 
 the additive group of integers $\mathbb{Z}$ as unit
 object. The associativity and unit constraints
 come from the respective canonical isomorphisms
 for the tensor of abelian groups. Notice that 
 there is also a different symmetric monoidal structure 
 on the cocomplete $\B{Ab}$, namely
 $(\B{Ab},\oplus,0)$ where $\oplus$ is the direct 
 product.

 $(3)$ The category $\Mod_R$ of modules
 over a commutative ring $R$ and $R$-module
 homomorphisms is a symmetric monoidal category 
 with tensor the usual tensor product $\otimes_R$
 of $R$-modules. The unit object is the ring $R$
 and the associativity and unit constraints
 are the canonical ones. The symmetry $s$
 has components the canonical isomorphisms
 $A\otimes_R B\cong B\otimes_R A$. Clearly
 the category of $k$-vector spaces
 and $k$-linear maps $\B{Vect}_k$ for a field $k$ 
 is again a symmetric monoidal category.
 
 $(4)$ For any bicategory $\ca{K}$, the 
 hom-categories $(\ca{K}(A,A),\circ,1_A)$ for any 0-cell A
 are monoidal categories as explained earlier,
 but not necessarily symmetric. 
 As a special case for $\ca{K}=\B{Cat}$,
 the category $\mathrm{End}(\ca{C})$ of endofunctors
 on a category $\ca{C}$ is a monoidal category
 with composition as the tensor product and $1_\ca{C}$
 as the unit.
\end{examples*}
\begin{defi*}
If $\ca{V}$ and $\ca{W}$ are monoidal categories, 
a \emph{lax monoidal functor} between them
consists of a 
functor $F:\ca{V}\to\ca{W}$
together with natural transformations
\begin{equation}\label{laxstructure}
\xymatrix @C=.6in @R=.5in
{\ca{V}\times\ca{V}\ar[r]^-{F\times F}\ar[d]
_-{\otimes}\drtwocell<\omit>{\phi}
& \ca{W}\times\ca{W}\ar[d]^-{\otimes}\\
\ca{V}\ar[r]_-{F} & \ca{W}}\qquad\mathrm{and}\qquad
\xymatrix @C=.7in @R=.5in
{\B{1}\drtwocell<\omit>{<-1>\phi_0}
\ar[r]^-{I_{\ca{V}}}\ar @/_2ex/[dr]_-{I_\ca{W}} 
& \ca{V}\ar[d]^-F\\
& \ca{W}}
\end{equation}
with components
\begin{align*}
\phi_{A,B}:FA\otimes FB&\to F(A\otimes B) \\
\phi_0:I&\to FI
\end{align*}
satisfying the associativity and unitality axioms: the diagrams
\begin{equation}\label{assoc&unitality}
\xymatrix @C=.7in
{FA\otimes FB\otimes FC\ar[r]^-{\phi_{A,B}\otimes 1}
\ar[d]_-{1\otimes\phi_{B,C}} & 
F(A\otimes B)\otimes FC\ar[d]^-{\phi_{A\otimes B,C}}\\
FA\otimes F(B\otimes C)\ar[r]_-{\phi_{A,B\otimes C}} & 
F(A\otimes B\otimes C),}
\end{equation}
\begin{displaymath}
\xymatrix @C=.5in {FA\ar[r]^-{1\otimes \phi_0}
\ar[d]_-{\phi_0\otimes 1}\ar @{-->}[dr]^-{1}
& FA\otimes FI\ar[d]^-{\phi_{A,I}}\\
FI\otimes FA\ar[r]_-{\phi_{I,A}} & FA}
\end{displaymath}
commute, where the constraints $\alpha, l, r$ have been suppressed.
\end{defi*}
In the case where $\phi_{A,B},\phi_0$
are isomorphisms, the functor $F$ is called
\emph{(strong) monoidal}, whereas if they are
identities $F$ is called \emph{strict monoidal}.
Dually, $F$ is a \emph{colax monoidal functor}
when it is equipped with with natural families
in the opposite direction, 
$\psi_{A,B}:F(A\otimes B)\to FA\otimes FB$ and
$\psi_0:FI\to I$.
Notice how these definitions follow from 
Definition \ref{laxfunctor}
for the one-object bicategory case.

A functor $F:\ca{V}\to\ca{W}$ between 
braided monoidal 
categories $\ca{V}$ and $\ca{W}$ is \emph{braided monoidal} if 
it is monoidal and also makes the diagram
\begin{displaymath}
\xymatrix @C=.6in 
{FA\otimes FB \ar[r]^-{c_{FA,FB}} \ar[d]_-{\phi_{A,B}} &
FB\otimes FA\ar[d]^-{\phi_{B,A}}\\
F(A\otimes B)\ar[r]_-{F(c_{A,B})} & F(B\otimes A)}
\end{displaymath}
commute, for all $A,B\in\ca{V}$. If $\ca{V}$
and $\ca{W}$ are symmetric, then $F$ is a 
\emph{symmetric monoidal functor} 
with no extra conditions.
\begin{defi*}
If $F,G:\ca{V}\to\ca{W}$ are lax monoidal functors, 
a \emph{monoidal natural transformation} 
$\tau:F\Rightarrow G$
is an (ordinary) natural transformation such that 
the following two diagrams commute:
\begin{equation}\label{nattransf}
\xymatrix @C=.7in @R=.4in
{FA\otimes FB\ar[r]^-{\phi_{A,B}}\ar[d]_-{\tau_A\otimes\tau_B} 
& F(A\otimes B)
\ar[d]^-{\tau_{A\otimes B}}\\ 
GA\otimes GB\ar[r]_-{\phi'_{A,B}} & G(A\otimes B),} \qquad
\xymatrix @C=.4in @R=.4in
{I\ar[r]^-{\phi_0}\ar[dr]_-{\phi'_0} & FI\ar[d]^-{\sigma_I}\\
& GI.}
\end{equation}
\end{defi*}
A \emph{braided} or \emph{symmetric} monoidal 
natural transformation is just a 
monoidal natural transformation between braided or 
symmetric monoidal functors.

It is not hard to verify that the different
kinds of monoidal functors compose.
Depending on the monoidal
structure that the functors are equipped with, 
we have the 2-categories 
$\B{MonCat}_s$,
$\B{MonCat}$, $\B{MonCat}_l$ and
$\B{MonCat}_c$ of monoidal categories, 
strict/strong/lax/colax monoidal functors and 
monoidal natural transformations. If
the functors are moreover
braided or symmetric, we have different versions
of 2-categories $\B{BrMonCat}$ and
$\B{SymmMonCat}$.

\begin{rmk}\label{monoidalproductofmoncats}
 The category $\B{MonCat}$ is 
 itself a cartesian monoidal category.
 For $\ca{V}$, $\ca{W}$ two monoidal categories,
 their product $\ca{V}\times\ca{W}$ has
 the structure of a monoidal category
 with tensor product the composite
 \begin{displaymath}
\xymatrix @C=.8in @R=.6in
{\ca{V}\times\ca{W}\times\ca{V}\times\ca{W}
\ar @{-->}[r]^-{\otimes_{(\ca{V}\times\ca{W})}} 
\ar[d]_-{1\times\mathrm{sw}\times1}
& \ca{V}\times\ca{W}\\
\ca{V}\times\ca{V}\times\ca{W}\times\ca{W}
\ar@/_/[ur]_-{\otimes_\ca{V}\times\otimes_\ca{W}} &}
\end{displaymath}
and unit the pair $(I_\ca{V},I_\ca{W})$. 
On objects, the above operation
explicitly gives
\begin{displaymath}
 ((A,B),(A',B'))\mapsto(A\otimes A',B\otimes B').
\end{displaymath}
Similarly
$F\times G$ is a monoidal functor when $F$ and $G$ 
are. The terminal
category $\B{1}$ is the unit monoidal category,
hence $(\B{MonCat},\times,\B{1})$ is 
in fact a monoidal category.
\end{rmk}
\begin{defi*}
The monoidal category $\ca{V}$ is said to be
\emph{(left) closed} when, for each 
$A\in\ca{V}$, the functor $-\otimes A:\ca{V}\to\ca{V}$ 
has a right adjoint $[A,-]:\ca{V}\to\ca{V}$ 
with a bijection
\begin{equation}\label{internalhomadj}
\ca{V}(C\otimes A,B)\cong \ca{V}(C,[A,B]).
\end{equation}
natural in $C$ and $B$. We call $[A,B]$ the 
\emph{(left) internal hom} of $A$ and $B$.
\end{defi*}
If also every $A\otimes -$ has a right adjoint $[A,-]'$,
we say that the monoidal category $\ca{V}$ is \emph{right
closed}. When $\ca{V}$ is a braided monoidal category,
each left internal hom gives a right internal hom
$[A,B]=[A,B]'$. A monoidal category is called
\emph{closed} (or biclosed)
when it is left and right closed. 

For example, the symmetric monoidal category $\Mod_R$
is a monoidal closed category, by the well-known 
adjunction 
\begin{displaymath}
\xymatrix @C=.6in
{\Mod_R\ar @<+.8ex>[r]^-{-\otimes_R M}
\ar@{}[r]|-\bot
& \Mod_R \ar @<+.8ex>[l]^-{\Hom_R(M,-)}}
\end{displaymath}
where $\Hom_R$ is the linear hom functor.

By `adjunctions with a parameter'
theorem \ref{parametrizedadjunctions},
the definition of the internal hom
for a monoidal closed category $\ca{V}$ implies
that there is a unique way of making it into a
functor of two variables
\begin{displaymath}
[-,-]:\ca{V}^\mathrm{op}\times\ca{V}\longrightarrow\ca{V}
\end{displaymath}
such that the bijection (\ref{internalhomadj})
is natural in all three variables.
Explicitly, if 
$f:C\to A$ and $g:B\to D$ are arrows of $\ca{V}$,
there is a unique arrow 
$[f,g]:[A,B]\to[C,D]$ such that the diagram 
\begin{displaymath}
\xymatrix{[A,B]\otimes C\ar[rr]^-{[f,g]\otimes1}\ar[d]_-{1\otimes f} &&
[C,D]\otimes C\ar[d]^-{\mathrm{ev}^C_D}\\
[A,B]\otimes A\ar[r]_-{\mathrm{ev}^A_B} & B\ar[r]_-{g} & D}
\end{displaymath}
commutes, where $\mathrm{ev}^A$ 
is the counit of the adjunction $-\otimes A\dashv[A,-]$
usually called the \emph{evaluation}.
In other words,
the internal hom bifunctor $[-,-]$ is the 
\emph{parametrized adjoint}
of the tensor bifunctor $(-\otimes -)$.

Notice that in any parametrized
adjunction as in (\ref{fixedparameter})
with natural isomorphisms 
$\ca{C}((F(A,B),C)\cong\ca{A}(A,G(B,C)),$
the counit is a collection of components
\begin{displaymath}
 \varepsilon_A^B:F(G(B,A),B)\longrightarrow A
\end{displaymath}
which is natural in $A$ and also
\emph{dinatural} or \emph{extranatural} in $B$. 
This is expressed by the 
commutativity of 
\begin{equation}\label{dinaturality}
 \xymatrix @R=.4in @C=.4in
{F(G(B',A),B)\ar[r]^-{F(1,f)}
\ar[d]_-{F(G(f,1),1)} & F(G(B',A),B')
\ar[d]^-{\varepsilon_A^{B'}} \\
F(G(B,A),B)\ar[r]_-{\varepsilon_A^B} & A}
\end{equation}
for any arrow $f:B\to B'$. Dinaturality
is discussed in detail in \cite[IX.4]{MacLane}.

Finally, in any symmetric monoidal closed category 
$\ca{V}$ we also have an adjunction
\begin{equation}\label{adjunctionopinthom}
 \xymatrix @=.6in
{\ca{V}\ar@<+.8ex>[r]^-{[-,A]^\op}
\ar@{}[r]|-\bot & \ca{V}^\op
\ar@<+.8ex>[l]^-{[-,A]} }
\end{equation}
with a natural isomorphism $\ca{V}^\op([V,A],W)\cong\ca{V}(V,[W,A])$,
explicitly given by the following bijective correspondences:
\begin{displaymath}
 \xymatrix @R=.02in
{\qquad\qquad W\ar[rr] && [V,A] 
\qquad\qquad& \mathrm{in}\;\ca{V} \\ 
\ar@{-}[rr] &&& \\  
\qquad\qquad\quad W\otimes V\ar@{}[d]|-{\quad\qquad\qquad\lcong} \ar[rr] 
&& A\qquad\qquad 
& \mathrm{in}\;\ca{V} \\
\qquad\qquad\quad V\otimes W\ar[urr] &&&& \\
\ar@{-}[rr] &&& \\
\qquad\qquad V\ar[rr] && [W,A]\qquad\qquad
 & \mathrm{in}\;\ca{V}.} 
\end{displaymath}

\section{Doctrinal adjunction for monoidal categories}

As mentioned briefly 
at the end of Section \ref{2cats},
monoidal categories are (strict) algebras
for a specific
2-monad $D$ on $\B{Cat}$, which arise
from \emph{clubs}. 
Details 
of these facts and structures
can be found
in \cite{AbstrApproachCoherence,
Doctrinal,ClubsDoctrines,WeberGenericMorphisms}.
In this context, lax morphisms of $D$-algebras turn out
to be lax monoidal functors and
$D$-natural transformations are monoidal
natural transformations. Therefore,
by doctrinal adjunction we can see how
lax and colax monoidal structures
on adjoint functors between monoidal
categories relate to each other.

Depending on which 2-category 
of monoidal categories we are 
working in, Definition \ref{adjunction2cat}
gives us different notions of 
\emph{monoidal adjunctions}. 
For example, an adjunction in 
the 2-category $\B{MonCat}_l$ is an adjunction 
between monoidal categories 
\begin{displaymath}
\xymatrix @C=.5in
{\ca{C}\ar @<+.8ex>[r]^-{F}\ar@{}[r]|-\bot 
& \ca{D}\ar @<+.8ex>[l]^-{G}}
\end{displaymath}
where $F$ and $G$ 
are lax monoidal functors 
and the unit and the counit are monoidal 
natural transformations.

Now, suppose that $F\dashv G$ is an ordinary adjunction
between two monoidal categories 
$\ca{C}$ and $\ca{D}$, where the left adjoint $F$ 
has the structure of a colax monoidal functor, 
\emph{i.e.} it is equipped with 2-cells 
$\psi,\psi_0$ in the opposite direction 
of (\ref{laxstructure}). 
Consider the diagram
\begin{displaymath}
\xymatrix @C=.6in 
{\ca{C}\times\ca{C}\ar @<+.8ex>[r] ^-{F\times F}
\ar@{}[r]|-\bot \ar[d]_-\otimes
& \ca{D}\times\ca{D}\ar @<+.8ex>[l]^-{G\times G}\ar[d]^-{\otimes}\\
\ca{C}\ar @<+.8ex>[r]^-{F}\ar@{}[r]|-\bot & 
\ca{D}\ar @<+.8ex>[l]^-{G}}
\end{displaymath}
which illustrates two adjunctions and two functors
between the categories involved.
Then, by Proposition \ref{mates}
which gives the mate correspondance, 
the 2-cell $\psi$ corresponds
uniquely to a 2-cell $\phi$ via
\begin{equation}\label{lax1}
\xymatrix @R=2pc { & \ca{D}\times\ca{D}\ar @/_7pt/[dl]
_-{G\times G}\ar[d]^-{1}\dtwocell<\omit>{^<4>\varepsilon\times\varepsilon}\\
\ca{C}\times\ca{C}\ar[r]^-{F\times F}\ar[d]_-{\otimes}  
& \ca{D}\times\ca{D}\ar[d]^-{\otimes}\dtwocell
<\omit>{^<6>\psi}\\
C\ar[r]^-{F}\ar[d]_-{1} & 
\ca{D}\ar @/^7pt/[dl]^-{G}\\
\ca{C}\utwocell<\omit>{<4>\eta}&}
\quad
\xymatrix{{\hole}\\{\hole}\\ {\hole}\\ \hole \ar@{}[uuu]-|{\displaystyle{=}}}
\quad
\xymatrix{{\hole}&\\
\ca{D}\times\ca{D}\ar[r]^-{\otimes}\ar[d]_-{G\times G}
& \ca{D}\ar[d]^-{G}\\
\ca{C}\times\ca{C}\ar[r]^-{\otimes} & \ultwocell<\omit>{<0>\phi}\ca{C}.}
\end{equation}
In terms of components via pasting, $\phi_{A,B}$ is expressed
as the composite	
\begin{displaymath}
GA\otimes GB\xrightarrow{\eta_{GA\otimes GB}}
GF(GA\otimes GB)\xrightarrow{G\psi_{GA,GB}}
G(FGA\otimes FGB)\xrightarrow{G(\varepsilon_A\otimes\varepsilon_B)}
G(A\otimes B).
\end{displaymath} 
Similarly, the 2-cell $\psi_0$ 
corresponds uniquely to a 2-cell $\phi_0$ via
\begin{equation}\label{lax2}
\xymatrix @C=.4in 
{\B{1}\ar[r]^-{1}\ar[d]_-{I_\ca{C}}  & \B{1}\ar[d]^-{I_\ca{D}}
\dtwocell<\omit>{^<5>\psi_0}\\
\ca{C}\ar[r]^-{F}\ar[d]_-{1} &\ca{D}\ar @/^2ex/[dl]^-{G}\\
\ca{C}\utwocell<\omit>{<4>\eta}&}
\xymatrix
{\hole \\
=}
\xymatrix @C=.4in
{\B{1}\ar[r]^-{I_\ca{D}}\ar[dd]_-{I_\ca{C}} & 
\ca{D}\ar @/^2ex/[ddl]^-{G}\\
\hole \\
\ca{C} \uutwocell<\omit>{<3>\phi_0}&}
\end{equation}
and in terms of components, $\phi_0$ is the composite
\begin{displaymath}
I\xrightarrow{\eta_I}GFI
\xrightarrow{G\psi_0}GI.
\end{displaymath}
Moreover, the arrows $\phi_{A,B}$ and $\phi_0$ turn out to satisfy the 
axioms (\ref{assoc&unitality}) thus they constitute
a \emph{lax monoidal structure} for the right adjoint $G$.

On the other hand, if we start with a lax monoidal structure 
$(\phi,\phi_0)$ on $G$, 
again due to the bijective correspondance of mates 
we end up with 
a colax structure $(\psi,\psi_0)$
on the left adjoint $F$, given
by the composites
\begin{equation}\label{colax}
\xymatrix @C=.3in
{F(A\otimes B)\ar@{-->}@/_/[drr]_-{\scriptscriptstyle{\psi_{A,B}}}
\ar[r]^-{\scriptscriptstyle{F(\eta\otimes\eta)}} &
F(GFA\otimes GFB)\ar[r]^-{\scriptscriptstyle{F\phi}} &
FG(FA\otimes FB)\ar[d]^-{\scriptscriptstyle{\varepsilon}} \\
&& FA\otimes FB,}\;\;
\xymatrix @C=.3in
{FI\ar[r]^-{\scriptscriptstyle{F\phi_0}}
\ar@{-->}@/_/[dr]_-{\scriptscriptstyle{\psi_0}}
& FGI\ar[d]^-{\scriptscriptstyle{\varepsilon}} \\
& I.}
\end{equation}
The above establish the following result.
\begin{prop} 
Suppose we have two (ordinary) adjoint functors $F\dashv G$ 
between monoidal categories. Then, colax monoidal structures on 
the left adjoint $F$ correspond bijectively, via mates, 
to lax monoidal structures on 
the right adjoint $G$.
\end{prop}
Of course this is a special case of Theorem
\ref{laxoplaxadjoints} for $\ca{K}=\B{Cat}$
and $D$ the 2-monad whose algebras
are monoidal categories. Proposition
\ref{proplaxoplaxadjoints} also applies.
\begin{prop}
A functor $F$ equipped with a lax 
monoidal structure has a right adjoint 
in $\B{MonCat}_l$ if and only if 
$F$ has a right adjoint in $\B{Cat}$ and 
its lax monoidal structure is a strong monoidal structure.
\end{prop}
\begin{proof}
`$\Rightarrow$' Suppose $F\dashv G$ is an adjunction 
in $\B{MonCat}_l$ and
$(\phi,\phi_0)$, $(\phi',\phi'_0)$ are the lax structure 
maps of $F$ and $G$. By the above corollary,
the lax monoidal structure of the right adjoint 
$G$ it induces a colax structure $(\psi,\psi_0)$
on the left 
adjoint $F$, given by the composites
(\ref{colax}).

In order for $F$ to be a strong monoidal functor, 
it is enough to show that this colax structure induced 
from $G$ is the two-sided
inverse to the lax structure of $F$.

$\bullet \qquad \psi_{A,B}\circ\phi_{A,B}= 1_{FA\otimes FB}$:
\begin{displaymath}
\xymatrix @C=3pc 
{FA\otimes FB\ar[r]^-{\phi_{A,B}}
\ar @{.>}[dr]
_-{F\eta_A\otimes F\eta_B}\ar @/_10ex/[ddrr]_-{1_{FA\otimes FB}} &
F(A\otimes B)\ar[r]^-{F(\eta_A\otimes\eta_B)} 
\dtwocell<\omit>{'{(i)}}&
F(GFA\otimes GFB)\ar[d]^-{F\phi'_{FA,FB}}\\
& FGFA\otimes FGFB\ar @{.>}[ru]_-{\phi_{GFA,GFB}}
\ar @{.>}[dr]_-{\varepsilon_{FA}\otimes\varepsilon_{FB}} 
\rtwocell<\omit>{'{(ii)}}\dtwocell<\omit>{'{(iii)}}& 
FG(FA\otimes FB)\ar[d]^-{\varepsilon_{FA\otimes FB}}\\
& & FA\otimes FB}
\end{displaymath}
where $(i)$ commutes by naturality of $\phi$, $(ii)$ by the 
fact that $\varepsilon:FG\Rightarrow 1_{\ca{D}}$ 
is a monoidal natural transformation between
lax monoidal functors, 
and $(iii)$ by one of the triangular identities. 

$\bullet\qquad\psi_0\circ\phi_0=1_I$:
\begin{displaymath}
\xymatrix{I\ar[r]^-{\phi_0}\ar[drr]_-{1_I} &
FI\ar[r]^-{F\phi'_0} &
FGI\ar[d]^-{\varepsilon_I}\\
&& I}
\end{displaymath}
which commutes by the axioms (\ref{nattransf}) 
for the monoidal counit $\varepsilon$ of 
the adjunction.

By forming similar diagrams
we can see how 
$\phi_{A,B}\circ\psi_{A,B}=1_{F(A\otimes B)}$ and 
$\psi_0\circ\phi_0=id_I$, hence $F$ is equipped 
with a strong monoidal
structure.

`$\Leftarrow$' Suppose that $F$ has the structure 
of a strong monoidal functor 
($\phi,\phi_0$) and it has an ordinary right adjoint
$G$. Clearly $F$ has a lax monoidal structure and 
a colax monoidal structure ($\phi^{-1}$, $\phi_0^{-1}$). 
Therefore it induces a lax monoidal structure 
on the right adjoint $G$ given
by the composites (\ref{lax1}), (\ref{lax2}).

What is left to show is that the unit $\eta$ 
and the counit $\varepsilon$ 
of the adjunction are monoidal natural transformations,
\emph{i.e.} they satisfy the commutativity of the diagrams 
(\ref{nattransf}).
For example, the first diagram for 
$\eta:1_{\ca{C}}\Rightarrow GF$ 
becomes
\begin{displaymath} 
\xymatrix @C=.6in @R=.4in
{A\otimes B\ar[r]^-{\eta_{A\otimes B}}
\ar[d]_-{\eta_A\otimes\eta_B} 
\drtwocell<\omit>{'{(i)}} &
GF(A\otimes B)\ar @{.>}[ddl]^-{GF(\eta_A\otimes\eta_B)} 
\ddtwocell<\omit>{'{(ii)}} & \\
GFA\otimes GFB\ar[d]_-{\eta_{GFA\otimes GFB}} && 
G(FA\otimes FB)\ar[ul]_-{G\phi_{A,B}}
\ar @/_2ex/@{.>}[ld]_-{G(F\eta_A\otimes F\eta_B)} \\
GF(GFA\otimes GFB)\ar [r]_-{G{\phi^{-1}_{GFA,GFB}}} &
G(FGFA\otimes FGFB)\ar[ur]_-{G(\varepsilon_{FA}\otimes\varepsilon_{FB})}
\ar @/_2ex/@{.>}[l]_{G\phi_{GFA,GFB}} 
& }
\end{displaymath}
where $(i)$ commutes by naturality of $\eta$, and $(ii)$ 
by naturality of $\phi$ and one of the triangular identities. Notice
that the lower composite from $GFA\otimes GFB$ to $GF(A\otimes B)$
is the lax structure map $\phi''_{A,B}$
of the composite lax monoidal functor $GF$. 

The second diagram commutes trivially, and in a very
similar way we can show that $\varepsilon$ 
is also a monoidal natural transformation.
Hence, the adjunction can be lifted in $\B{MonCat}_l$.
\end{proof}
The above propositions generalize
to the case of parametrized adjoints.
For example, if the functor
$F:\ca{A}\times\ca{B}\to\ca{C}$ between monoidal
categories has a colax structure
\begin{gather*}
 \psi_{(A,B),(A',B')}:F(A\otimes A',B\otimes B')\to F(A,B)\otimes F(A',B') \\
 \psi_0:F(I_\ca{A},I_\ca{B})\to I_\ca{C},
\end{gather*}
then its parametrized adjoint 
$G:\ca{B}^\op\times\ca{C}\to\ca{A}$
obtains a lax structure via the composites
\begin{displaymath}
\xymatrix @C=.8in @R=.4in
{G(B,C)\otimes G(B',C')\ar[r]^-{\eta^{B\otimes B'}_{G(B,C)\otimes G(B',C')}}
\ar@{-->}@/_4ex/[ddr]_-{\phi_{(B,C),(B',C')}} & 
G(B\otimes B',F(G(B,C)\otimes G(B',C'),
B\otimes B'))\ar[d]^-{G(1,\psi_{(G(B,C),B),(G(B',C'),B')})} \\
& G(B\otimes B',F(G(B,C),B)\otimes F(G(B',C'),B'))
\ar[d]^-{G(1,\varepsilon^B_C\otimes\varepsilon^{B'}_{C'})} \\
& G(B\otimes B',C\otimes C'),}
\end{displaymath}
 \begin{displaymath}
  \xymatrix @C=.7in @R=.4in
{I_\ca{A}\ar[r]^-{\eta^{I_\ca{B}}_{I_\ca{A}}}
\ar@{-->}@/_2ex/[dr]_-{\phi_0} & 
G(I_\ca{B},F(I_\ca{A},I_\ca{B}))
\ar[d]^-{G(1,\psi_0)} \\
& G(I_\ca{B},I_\ca{C}).}
 \end{displaymath}
The respective axioms are satisfied by naturality
and dinaturality of the unit and counit
$\eta$, $\varepsilon$ of the parametrized adjunction
and the axioms for $(\psi,\psi_0)$ of $F$.
\begin{prop}\label{parametdoctrinal}
 Suppose $F:\ca{A}\times\ca{B}\to\ca{C}$
 and $G:\ca{B}^\mathrm{op}\times\ca{C}\to\ca{A}$
 are parametrized adjoints between
 monoidal categories, \emph{i.e.} 
 $F(-,B)\dashv G(B,-)$
 for all $B\in\ca{B}$. Then, colax monoidal structures on 
 $F$ correspond bijectively to lax monoidal
 structures on $G$.
\end{prop}
As an application, consider the case
of a symmetric monoidal closed 
category $\ca{V}$, with symmetry $s$. 
The tensor product functor
$\otimes:\ca{V}\times\ca{V}\to\ca{V}$ from the 
monoidal $\ca{V}\times\ca{V}$ 
(see Remark \ref{monoidalproductofmoncats})
is equipped
with a strong monoidal structure,
namely 
\begin{align*} 
\phi_{(A,B),(A',B')}:A\otimes B\otimes
A'\otimes B'&\xrightarrow{\;1\otimes s_{B,A'}\otimes1\;}
A\otimes A'\otimes
B\otimes B', \\
\phi_0:I&\xrightarrow{\;r_I^{-1}\;}I\otimes I.
\end{align*}
Therefore Proposition \ref{parametdoctrinal} applies
and its parametrized adjoint 
obtains the structure of a lax monoidal functor.
\begin{prop}\label{laxinthom}
In a symmetric monoidal closed category $\ca{V}$, the 
internal hom functor 
$[-,-]:\ca{V}^\mathrm{op}\otimes\ca{V}\to\ca{V}$
has the structure of a lax monoidal functor, with 
structure maps
\begin{displaymath}
\chi_{(A,B),(A',B')}:[A,B]\otimes[A',B']\to[A\otimes A',B\otimes B'],
\end{displaymath}
\begin{displaymath} 
\chi_0:I\to[I,I]
\end{displaymath}
which correspond, under the adjunction $-\otimes A\dashv[A,-]$, 
to the morphisms 
\begin{gather*}
[A,B]\otimes[A',B']\otimes A\otimes A'\xrightarrow{1\otimes s\otimes1}
[A,B]\otimes A\otimes[A',B']\otimes A'
\xrightarrow{\mathrm{ev}\otimes\mathrm{ev}}B\otimes B', \\
I\otimes I\xrightarrow{l_I=r_I}I.
\end{gather*}
\end{prop}

\section{Categories of monoids and comonoids}\label{Categoriesofmonoidsandcomonoids}

A \emph{monoid} in a monoidal category $\ca{V}$
is an object $A$ equipped with arrows
\begin{displaymath}
m:A\otimes A\to A\qquad\mathrm{and}\qquad
\eta:I\to A
\end{displaymath}
called the \emph{multiplication}
and the \emph{unit}, satisfying the
associativity and identity conditions: the diagrams
\begin{equation}\label{monoidaxioms}
\xymatrix @R=.4in @C=.4in
{A\otimes A\otimes A\ar[r]^-{1\otimes m}
\ar[d]_-{m\otimes1} & A\otimes A\ar[d]^-m \\
A\otimes A\ar[r]_-m & A}\qquad\mathrm{and}\qquad
\xymatrix @R=.4in @C=.4in
{I\otimes A\ar[dr]_-{l_A}\ar[r]^-{\eta\otimes1}
& A\otimes A\ar[d]_-m & A\otimes I\ar[l]_-{1\otimes\eta}
\ar[dl]^-{r_A}\\ 
& A &}
\end{equation}
commute, where the associativity constraint is 
suppressed from the first diagram. A \emph{monoid morphism} between 
two monoids $(A,m,\eta)$ and $(A',m',\eta')$
is an arrow $f:A\to A'$ in $\ca{V}$ such that
the diagrams
\begin{equation}\label{monoidarrowaxioms}
\xymatrix @R=.4in @C=.4in
{A\otimes A\ar[r]^-m \ar[d]_-{f\otimes f} &
A\ar[d]^-f \\
A'\otimes A'\ar[r]_-{m'} & A'}\qquad\mathrm{and}\qquad
\xymatrix @R=.4in @C=.4in
{I\ar[r]^-\eta \ar[dr]_-{\eta'} & A\ar[d]^-f \\
& A'}
\end{equation}
commute. We obtain a category $\Mon(\ca{V})$
of monoids and monoid morphisms.
Furthermore, a 2-cell $\alpha:f\Rightarrow g$
is defined to be an arrow $\alpha:I\to B$
such that 
\begin{equation}\label{monoid2cells}
\xymatrix{A\ar[r]^-{\alpha\otimes f}
\ar[d]_-{g\otimes\alpha} 
& B\otimes B\ar[d]^-m\\
B\otimes B\ar[r]_-m & B}
\end{equation}
commutes, thus $\Mon(\ca{V})$
is a 2-category.

Dually, there is a 2-category
of \emph{comonoids}
$\Comon(\ca{V})$ with
objects triples $(C,\Delta,\epsilon)$ 
where $C$ is an object in $\ca{V}$, 
$\Delta:C\to C\otimes C$ is the \emph{comultiplication} 
and $\epsilon:C\to I$ is the \emph{counit}, such that 
dual diagrams to (\ref{monoidaxioms}) commute.
\emph{Comonoid morphisms}
$(C,\Delta,\epsilon)\to(C',\Delta',\epsilon')$ 
are arrows $g:C\to C'$ in 
$\ca{V}$ such that the dual of (\ref{monoidarrowaxioms})
commutes, and 2-cells $\beta:f\Rightarrow g$ 
are arrows $\beta:C\to I$ satisfying
dual diagrams to (\ref{monoid2cells}). 

For the purposes of this dissertation, 
the 2-dimensional structure of the categories
of monoids and comonoids (and modules and comodules later) 
will not be employed. 
Notice that as categories, 
$\Comon(\ca{V})=\Mon(\ca{V}^\mathrm{op})^\mathrm{op}$.
\begin{rmk}\label{monadsaremonoids}
We saw in Section \ref{Basicdefinitions}
how, for any object $B$ in a bicategory $\ca{K}$,
the hom-category $\ca{K}(B,B)$ obtains the structure
of a monoidal category, with tensor product
the horizontal composition and unit the identity 1-cell.
From this viewpoint, the data that define the notion
of a monad $t:B\to B$ in a bicategory (Definition \ref{monadbicat})
equivalently define a monoid in the monoidal
category $(\ca{K}(B,B),\circ,1_B)$. Dually, a comonad
$u:A\to A$ in a bicategory $\ca{K}$ as in
Definition \ref{comonadbicat} is precisely
a comonoid in the monoidal $\ca{K}(A,A)$. 
\end{rmk}
If the monoidal category $\ca{V}$ is 
braided, we can define a monoid
structure on the tensor product $A\otimes B$
of two monoids $A$, $B$ via
\begin{align*}
A\otimes B\otimes A\otimes B\xrightarrow{1\otimes c\otimes1}&
A\otimes A\otimes B\otimes B\xrightarrow{m\otimes m}A\otimes B \\
 I\xrightarrow{r_I^{\textrm{-}1}}&
I\otimes I\xrightarrow{\eta\otimes\eta}A\otimes B \notag
\end{align*}
where the constraints are again suppressed.
This induces a monoidal structure on the
category $\Mon(\ca{V})$, such that
the forgetful functor to $\ca{V}$
is a strict monoidal functor.
The braiding/symmetry of $\ca{V}$ lifts
to its category of monoids,
so $\Mon(\ca{V})$ is a braided/symmetric
monoidal category when $\ca{V}$ is. This happens
because $\Mon(\ca{V})\to\ca{V}$ always reflects isomorphisms.
Dually, $\Comon(\ca{V})$ also inherits the
monoidal structure from $\ca{V}$, via
\begin{displaymath}
 C\otimes D\xrightarrow{\delta\otimes\delta}
C\otimes C\otimes D\otimes D\cong
C\otimes D\otimes C\otimes D, \quad C\otimes D
\xrightarrow{\epsilon\otimes\epsilon}I\otimes I\cong I.
\end{displaymath}
The monoidal unit in both cases is $I$, with trivial 
monoid and comonoid structure via $r_I$.

For example, the category of monoids 
in the symmetric monoidal category 
$(\B{Ab},\otimes,\mathbb{Z})$ is the category of rings
$\B{Rng}$, and in the symmetric
cartesian monoidal category $(\B{Cat},\times,\B{1})$
it is the category of 
strict monoidal categories $\B{MonCat}_{st}$.
Also, the category of monoids
in the symmetric monoidal category $\Mod_R$ for
a commutative ring $R$ is the category 
of $R$-algebras $\Alg_R$
and the category of comonoids is the category
of $R$-coalgebras $\Coalg_R$.

An important property of lax monoidal 
functors is that they map
monoids to monoids. More precisely, if 
$F:\ca{V}\to\ca{W}$ is a lax monoidal functor 
between monoidal categories $\ca{V}$ and $\ca{W}$, 
there is an induced functor
\begin{equation}\label{MonF} 
\Mon(F):
\xymatrix @R=.05in
{\Mon(\ca{V})\ar[r] &
\Mon(\ca{W})\\
(A,m,\eta)\ar @{|->}[r] &
(FA,m',\eta')}
\end{equation}
which gives $FA$ the structure of a monoid in $\ca{W}$,
with multiplication and unit
\begin{gather*}
m':FA\otimes FA\xrightarrow{\phi_{A,A}}F(A\otimes A)
\xrightarrow{Fm}FA \\
\eta':I\xrightarrow{\phi_0}FI
\xrightarrow{F\eta}FA
\end{gather*}
where $\phi_{A,A}$ and $\phi_0$ are the structure
maps of $F$. The associativity and identity conditions
are satisfied because of naturality of
$\phi$, $\phi_0$ and the fact that $A$ is a monoid. 
Dually, if $G:\ca{V}\to\ca{W}$ is colax monoidal functor,
it maps comonoids to comonoids via an induced functor
\begin{displaymath}
 \Comon(F):
\xymatrix @R=.05in
{\Comon(\ca{V})\ar[r] & \Comon(\ca{W}) \\
(C,\delta,\epsilon)\ar@{|->}[r] & 
(GC,\psi\circ G\delta,\psi\circ G\epsilon).}
\end{displaymath}

For example, in a symmetric monoidal
closed category $\ca{V}$,
the internal hom functor 
$[-,-]:\ca{V}^\mathrm{op}\times\ca{V}\to\ca{V}$
is lax monoidal by Proposition \ref{laxinthom}.
The category of monoids of the monoidal
category $\ca{V}^\mathrm{op}\times\ca{V}$ is
\begin{displaymath}
\Mon(\ca{V}^\mathrm{op}\times\ca{V})\cong\Mon
(\ca{V}^\mathrm{op})\times\Mon(\ca{V})
\cong\Comon(\ca{V})^\mathrm{op}\times\Mon(\ca{V}),
\end{displaymath}
so there is an induced functor betweem the categories of monoids
\begin{equation}\label{defMon[]} 
\Mon[-,-]:
\xymatrix @R=.05in
{\Comon(\ca{V})^\mathrm{op}
\times\Mon(\ca{V})\ar[r]
&\Mon(\ca{V})\\
\qquad\;(\;C\;,\;A\;)\;\ar @{|->}[r] & [C,A].}
\end{equation}
The concrete content of this observation is that whenever $C$ 
is a comonoid and $A$ a monoid, the object $[C,A]$ obtains the 
structure of a monoid, with unit $I\to [C,A]$ which is the 
transpose under the adjunction $-\otimes C\dashv [C,-]$ of 
\begin{displaymath}
C\xrightarrow{\;\epsilon\;}I\xrightarrow{\;\eta\;}A
\end{displaymath}
and with multiplication $[C,A]\otimes[C,A]\to[C,A]$ the 
transpose of the composite
\begin{displaymath}
\xymatrix @R=.3in @C=.35in
{[C,A]\otimes[C,A]\otimes C
\ar[r]^-{1\otimes\Delta}
\ar @{-->}[ddrr]
& [C,A]\otimes[C,A]\otimes C\otimes C\ar[r]^-{1\otimes s\otimes 1}
& [C,A]\otimes C\otimes [C,A]\otimes C\ar[d]
^-{\mathrm{ev}\otimes\mathrm{ev}}\\
&& A\otimes A\ar[d]^-{m}\\
&& A.}
\end{displaymath}
\begin{rmk}\label{rmkconvolution}
For the symmetric monoidal closed category
$\Mod_R$, the internal hom
\begin{displaymath}
 [-,-]=\Hom_R(-,-):\Mod_R^\mathrm{op}
 \times\Mod_R\longrightarrow\Mod_R
\end{displaymath}
has the structure of a lax monoidal functor
by Proposition \ref{laxinthom}. Therefore 
it induces a functor
\begin{displaymath} 
\Mon(\Hom_R):
\xymatrix @R=.05in
{\Coalg_R^{\mathrm{op}}\times\Alg_R\ar[r]
 & \Alg_R \\ 
(\;C\;,\;A\;)\;\ar @{|->}[r] & \Hom_R(C,A)}
\end{displaymath}
between the categories of coalgebras and algebras. 
This implies the well-known fact that
for $C$ an $R$-coalgebra and $A$ an $R$-algebra,
the set $\Hom_R(C,A)$ of the linear maps between them 
obtains the structure of an $R$-algebra
under the \emph{convolution} structure
\begin{displaymath}
(f*g)(c)=\sum\limits_{(c)}{f(c_1)g(c_2)}
\quad\mathrm{and}\quad
1=\eta\circ\epsilon
\end{displaymath}
where $*$ is expressed using the `sigma notation' 
for the coalgebra comultiplication $\Delta(c)=\sum_{i}{c_{1i}\otimes
c_{i2}}:=\sum_{(c)}{c_{(1)}\otimes c_{(2)}}$ introduced in \cite{Sweedler}.
\end{rmk}
Another example
of a functor induced between categories
of monoids is the following.
\begin{lem}\label{lemmonlaxfun}
 If $\ps{F}:\ca{K}\to\ca{L}$ is a lax functor between
two bicategories, there is an induced functor
\begin{equation}\label{monlaxfun}
 \Mon\ps{F}_{A,A}:\Mon\ca{K}(A,A)\longrightarrow
\Mon\ca{L}(\ps{F}A,\ps{F}A)
\end{equation}
for each object $A$ in $\ca{K}$, which
is the functor $\ps{F}_{A,A}$ restricted to 
the category of monoids of the monoidal
category $(\ca{K}(A,A),\circ,1_A)$.
\end{lem}
\begin{proof}
Since $\ps{F}$
is a lax functor between bicategories,
we have a functor $\ps{F}_{A,B}:\ca{K}(A,B)\to\ca{L}(\ps{F}A,\ps{F}B)$
between the hom-categories for all $A,B\in\ca{K}$.
In particular, there is a functor 
\begin{displaymath}
\ps{F}_{A,A}:\ca{K}(A,A)\to\ca{L}(\ps{F}A,\ps{F}A)
\end{displaymath}
which maps the 1-cell $f:A\to A$ to $\ps{F}f:\ps{F}A\to\ps{F}A$
and a 2-cell $\alpha:f\Rightarrow g$ to 
$\ps{F}\alpha:\ps{F}f\Rightarrow\ps{F}g$. If we regard 
$\ca{K}(A,A)$ and $\ca{L}(\ps{F}A,\ps{F}A)$ as monoidal 
categories with respect
to the horizontal composition as in (\ref{tensorcirc}), 
$\ps{F}_{A,A}$ has the structure of a lax monoidal functor. 
Indeed, it is equipped with natural transformations with components,
for each $f,g\in\ca{K}(A,A)$,
\begin{displaymath}
 \phi_{f,g}:\ps{F}f\otimes\ps{F}g\to\ps{F}(f\otimes g)
\quad\textrm{and}\quad
\phi_0:I_{\ca{L}(FA,FA)}\to\ps{F}I_{\ca{K}(A,A)}
\end{displaymath}
which are precisely the components 
$\delta_{f,g}$ and $\gamma_A$ of the natural
transformations (\ref{delta}, \ref{gamma}) that the lax
functor $\ps{F}$ is equipped with, since $\otimes\equiv\circ$
and $I_{\ca{K}(A,A)}\equiv1_A$. The axioms
follow from those of $\delta$ and $\gamma$.
Hence a functor (\ref{monlaxfun})
between the categories of monoids is induced.  
\end{proof}
In Remark \ref{monadsaremonoids}
we saw how a monad $t:A\to A$ 
in a bicategory $\ca{K}$ is actually
a monoid in $\ca{K}(A,A)$. The above lemma 
states that if $\ps{F}$
is a lax functor, then $\ps{F}t:\ps{F}A\to\ps{F}A$
is a monoid in $\ca{L}(\ps{F}A,\ps{F}A)$, \emph{i.e.}
$\ps{F}t$ is a monad in the bicategory $\ca{L}$.
Therefore we re-discover the fact that lax functors between 
bicategories preserve monads,
from a different point of view than 
Remark \ref{laxfunctorspreservemonads}, where
a monad was identified with a lax functor from the terminal 
bicategory to $\ca{K}$.

For any monoidal category $\ca{V}$, there are forgetful
functors
\begin{displaymath}
 S:\Mon(\ca{V})\longrightarrow\ca{V}
\quad\textrm{and}\quad
U:\Comon(\ca{V})\longrightarrow\ca{V}
\end{displaymath}
which just discard the (co)multiplication
and the (co)unit. A crucial issue
for our needs is the assumptions
under which these functors have a left or right adjoint
accordingly. In other words,
we are interested in the conditions on $\ca{V}$ that allow
the \emph{free monoid} and the 
\emph{cofree comonoid} construction.

The existence of a free monoid functor is quite frequent, 
since the monoidal structures
that arise in practice may well be closed,
so that the tensor product preserves colimits
in both arguments. In particular, the following is true.
\begin{prop}\label{freemonoidprop}
Suppose that $\ca{V}$ is a monoidal
category with countable coproducts which are
preserved by $\otimes$ on either side.
The forgetful $\Mon(\ca{V})\to\ca{V}$ has a left
adjoint $L$, and the free monoid on an object $X$ 
is given by the `geometric series'
\begin{displaymath}
LX=\coprod_{n\in\mathbb{N}}{X^{\otimes n}}.
\end{displaymath}
\end{prop}
There are various sets of conditions,
stronger or weaker, that 
guarantee the existence of free monoids and 
are connected with the different kinds of 
settings where they apply, such as free monads, free algebras,
free operads etc. There are many classical references
on these constructions, for example by Kelly, Dubuc, Barr
and others, and most are outlined in Lack's \cite{FreeMonoids}.

On the other hand,
the existence of a cofree comonoid
functor is more problematic. In
Sweedler's \cite{Sweedler}, the cofree coalgebra
on a vector space $V$ is constructed as a 
certain subcoalgebra of $T(V^*)^o$, 
where $T$ gives the tensor algebra
of the linear dual of $V$, and $(-)^o$ is the 
dual algebra functor as described later 
in Remark \ref{dualalgebra}. In 
\cite{BlockLerouxcofreecoalgebras}, a new description
of the cofree coalgebra is given, still in $\B{Vect}_k$
for a field $k$. In Barr's \cite{BarrCoalgebras}, it is shown that 
the forgetful $\Coalg_R\to\Mod_R$ for a commutative ring $R$
has a right adjoint, and in Fox's \cite{Foxcofreecoalgebras}
two constructions on the cofree coalgebra on an $R$-module
are presented. Finally in \cite{Hazewinkelcofreecoalgebras},
connections of cofree coalgebras in $\Mod_R$
with the notion of multivariable recursiveness are 
examined.

We are here interested in the generalization from 
$\B{Vect}_k$ and $\Mod_R$ to the existence 
of such cofree objects (comonoids) 
in an arbitrary monoidal category $\ca{V}$.
Hans Porst in a series of papers \cite{CoringsComod,
AdjAlgCoalg,FundConstrCoalgCorComod,MonComonBimon}
studied the categories of monoids and comonoids (also 
the categories of modules and comodules for them) and their various
categorical properties, with emphasis on the local presentability
structure inherited from the initial monoidal category. 
We are going to employ many of those strategies for 
our purposes, so at this point we briefly describe
the most basic parts of this theory. A standard reference
for locally presentable categories is 
Adamek-Rosicky's \cite{LocallyPresentable}. 

Recall that a small full subcategory $\ca{A}$ of a category $\ca{C}$ 
is called \emph{dense} provided that every object of $\ca{C}$
is a canonical colimit of objects of $\ca{A}$, \emph{i.e.}
the colimit of the forgetful 
$(\ca{A}\downarrow C)\to\ca{C}$. Also, an object 
in a category $\ca{C}$ is called 
\emph{$\lambda$-presentable} for $\lambda$ a regular cardinal, 
provided that its hom-functor 
$\ca{C}(C,-)$ preserves $\lambda$-filtered limits. For 
$\lambda=\aleph_0$, we have the notion of a 
\emph{finitely presentable object}.
\begin{defi*} $(1)$ A \emph{locally $\lambda$-presentable} 
category $\ca{C}$ is a cocomplete category
which has a set $\ca{A}$ of $\lambda$-presentable objects, such that
every object is a $\lambda$-filtered colimit of objects from $\ca{A}$.
A category is called \emph{locally presentable} when it is locally
$\lambda$-presentable for some regular cardinal $\lambda$, and 
\emph{locally finitely presentable} for $\lambda=\aleph_0$.
 
$(2)$ A \emph{$\lambda$-accessible} category is a category
with $\lambda$-filtered colimits and a set of 
$\lambda$-presentable objects, such that
every object is a $\lambda$-filtered colimit of those. A
category is called \emph{accessible} if it is 
$\lambda$-accessible for some regular cardinal $\lambda$.
\end{defi*}
Notice that in a locally $\lambda$-presentable category $\ca{C}$, all
$\lambda$-presentable objects have a set of representatives
(with respect to isomorphism). Any such set is denoted by
$\B{Pres}_\lambda\ca{C}$ and is a small dense full subcategory
of $\ca{K}$, hence also a strong generator.
Recall that a \emph{generator} is a family of objects $\ca{G}$ such 
that for pairs $\xymatrix{A\ar@<+.5ex>[r]|-f\ar@<-.5ex>[r]|-g & B}$ with 
$f\neq g$, there exists $G\in\ca{G}$ and $h:G\to A$ with $fh\neq gh$. It is 
\emph{strong} if for any $A$ and a proper subobject, there exists
$G\in\ca{G}$ and $G\to A$ which doesn't factorize through the subobject.

Other useful properties of locally presentable categories are
completeness, well-poweredness and co-wellpoweredness. 
Obviously, an accessible category with all colimits
is locally presentable, but so is an accessible
category with all limits (see \cite[2.47]{LocallyPresentable}).
A functor $F$ between $\lambda$-accessible categories is \emph{accessible}
if it preserves $\lambda$-filtered colimits, whereas a \emph{finitary}
functor in general preserves all filtered colimits.

In \cite{MonComonBimon} the class of \emph{admissible
monoidal categories} is introduced. These are locally 
presentable symmetric monoidal categories $\ca{V}$, 
such that for each object $A$
the functor $A\otimes -$ preserves filtered colimits.
Examples are the category $\Mod_R$ for a 
commutative ring $R$,
every locally presentable category with 
respect to binary products, 
and every monoidal closed category which is 
locally presentable. However, the results exhibited
below also hold for small
variations from the above conditions. For example,
the symmetry can be replaced with $\otimes$
preserving filtered colimits on both entries. 

The notion of
\emph{functor algebras} and
\emph{functor coalgebras} for an endofunctor
are of importance in the proofs below.
Given an endofunctor on any category 
$F:\ca{C}\to\ca{C}$, the category $\Alg F$
of $F$-algebras has objects pairs 
$(A,\alpha:FA\to A)$
and morphisms $(A,\alpha)\to(A',\alpha')$
are arrows $f:A\to A'$ making the diagram
\begin{displaymath}
\xymatrix
{FA\ar[r]^-\alpha \ar[d]_-{Ff} & A\ar[d]^-f \\
FA'\ar[r]^-{\alpha'} & A'}
\end{displaymath}  
commute. The category 
$\Coalg F=(\Alg F^\op)^\op$
is defined dually, with objects pairs
$(C,\beta:C\to FC)$ and arrows
$g:C\to C'$ making the diagram
\begin{displaymath}
\quad\xymatrix{C\ar[r]^-\beta \ar[d]_-{g} & FC\ar[d]^-{Fg} \\
C'\ar[r]^-{\beta'} & FC'}
\end{displaymath}
commute. More about these categories
and their properties can be found in 
\cite{LocallyPresentable,VarietiesCovarieties}.
The most useful facts are the following:
\begin{enumerate}[(i)]
\item The forgetful functor
$\Alg F\to\ca{C}$ creates all limits and and those
colimits which are preserved by $F$.
\item The forgetful functor $\Coalg F\to\ca{C}$
creates all colimits and those
limits which are preserved by $F$.
\item If $\ca{C}$ is locally presentable 
and $F$ preserves filtered colimits,
the categories $\Alg F$ and $\Coalg F$
are locally presentable.
\end{enumerate}
Notably, these categories can be expressed
as specific \emph{inserters} $\Alg F=\mathbf{Ins}(F,\mathrm{id}_\ca{C})$
and $\Coalg F=\mathbf{Ins}(\mathrm{id}_\ca{C},F)$.
Fact $(iii)$ thus follows from the more general `Weighted 
Limit Theorem' by Makkai and Par{\'e} \cite[5.1.6]{MakkaiPare}, 
which in particular
asserts that the above inserters are accessible categories
when $\ca{C}$ and $F$ are accessible. For details about these
constructions, see \cite[Theorem 2.72]{LocallyPresentable}.

In the applications where $\Alg F$ and $\Coalg F$
for specific endofunctors are studied,
they usually turn out to be monadic 
and comonadic respectively over $\ca{C}$.
Since coequalizers of split pairs are absolute
colimits, \emph{i.e.} preserved by any functor,
monadicity and comonadicity are established as soon
as the forgetful functor has a left or right
adjoint respectively.
\begin{prop}~\cite[2.6-2.7]{MonComonBimon}\label{moncomonadm}
Suppose $\ca{V}$ is an admissible category.

$(1)$ $\Mon(\ca{V})$ is finitary monadic over 
$\ca{V}$ and locally presentable.

$(2)$ $\Comon(\ca{V})$ is a locally 
presentable category and comonadic
over $\ca{V}$.
\end{prop}
\begin{proof}
(Sketch) The idea is to view both categories
of monoids and comonoids as subcategories
of the functor algebras and
functor coalgebras categories,
for specific endofunctors on $\ca{V}$.

Consider the functors
$T_+$ and $T_\times$ on our admissible
category $\ca{V}$ given by
\begin{displaymath}
 T_+(C)=(C\otimes C)+I, \qquad
 T_\times(C)=(C\otimes C)\times I.
\end{displaymath}
These are finitary functors, because
the `$n$-th tensor power' functor 
$T_n=(-)^{\otimes n}$ preserves filtered colimits,
and $(-\times I)$ preserves filtered colimits
for any locally presentable category (where
finite limits commute with filtered colimits).

We deduce that $\Alg T_+$ is finitary
monadic over $\ca{V}$, locally presentable
and contains $\Mon(\ca{V})$
as a full subcategory, and also
$\Coalg T_\times$ is comonadic over $\ca{V}$,
locally presentable and
contains $\Comon(\ca{V})$ as a full subcategory. 
Moreover, the categories
of monoids and comonoids are closed under limits and 
colimits respectively.

The first part of the proposition 
regarding $\Mon(\ca{V})$ follows from general
arguments for monadicity and local presentability
of categories of algebras for a finitary monad (see
\cite[Satz 10.3]{GabrielUlmer}). On the
other hand, these arguments cannot be dualized 
for $\Comon(\ca{V})$. For example,
the dual of a locally presentable category
is not locally presentable (unless it is 
a small complete lattice).

Therefore a different approach is followed,
using the notion of an equifier of a family
of natural transformations. The decisive fact then is that 
if all functors involved are accessible, 
then the equifier is an accessible category 
(see \cite[2.76]{LocallyPresentable}).
\begin{defi*}
 Let $F_1^i,F_2^i:\ca{A}\to\ca{B}_i$ be a family of 
functors, and for each $i\in I$, 
$(\phi^i,\psi^i):F_1^i\to F_2^i$
be a pair of natural transformations. Then, the full
subcategory of $\ca{A}$ spanned by those object $A$ 
which satisfy $\phi^i_A=\psi^i_A$ for all $i$ 
is called the \emph{equifier} 
of the above family of natural transformations, denoted
by
\begin{displaymath}
 \B{Eq}(\phi^i,\psi^i)_{\{i\in I\}}.
\end{displaymath}
\end{defi*}
More explicitly, three pairs $(\phi^i,\psi^i)$
of natural transformations between composites
of the forgetful $\Coalg T_\times\to\ca{V}$ and
the `tensor power functor' $\otimes^n$
are defined, the equality of which give 
precisely the coassociativity
and coidentity conditions of the definition of a comonoid.
Hence $\Comon(\ca{V})=\B{Eq}((\phi^i,\psi^i)_{i=1,2,3}),$
and for $\ca{V}$ admissible this implies that $\Comon(\ca{V})$ 
is locally presentable. 

Now comonadicity of $\Comon(\ca{V})$ over $\ca{V}$
follows: in the commutative triangle
\begin{displaymath}
 \xymatrix @C=.5in @R=.2in
{\Comon(\ca{V})\ar@{-->}[dr]_U\ar@{^(->}[r] & \Coalg F \ar[d] \\
& \ca{V}}
\end{displaymath}
where all categories are locally presentable, both forgetful
functors to $\ca{V}$ have a right adjoint 
by Theorem \ref{Kelly}, since they are cocontinuous.
Moreover, the right leg is comonadic by basic facts for 
functor coalgebras, and the inclusion preserves and reflects
all limits from the complete full subcategory $\Comon(\ca{V})$
to the complete $\Coalg F$. Therefore
it creates equalizers of split pairs and 
so does $U$, which then satisfies the conditions
of Precise Monadicity Theorem. In particular, the 
existence of the \emph{cofree comonoid functor}
$R:\ca{V}\to\Comon(\ca{V})$ is established.
\end{proof}
Another property which $\Comon(\ca{V})$ inherits from 
the monoidal category $\ca{V}$ is monoidal closedness.
\begin{prop}~\cite[3.2]{MonComonBimon}\label{Comonclosed}
If $\ca{V}$ is a symmetric monoidal closed 
category which is locally presentable,
then the category of comonoids 
$\Comon(\ca{V})$ 
is a locally presentable symmetric
monoidal closed category as well.
\end{prop}
\begin{proof}
The symmetric monoidal structure of $\Comon(\ca{V})$ 
was described earlier. In order 
to prove the existence of a right adjoint to
\begin{equation}\label{tensorcomon}
 -\otimes C:\Comon(\ca{V})\to\Comon(\ca{V})
\end{equation}
for any comonoid $C$ in $\ca{V}$,
we can use the adjoint functor theorem \ref{Kelly}.
The category $\Comon(\ca{V})$ is cocomplete and has a small
dense subcategory, since it is locally presentable
by Proposition \ref{moncomonadm}. Moreover, 
the functor (\ref{tensorcomon}) preserves all 
colimits by the commutativity of 
\begin{displaymath} 
\xymatrix @C=.5in
{\Comon(\ca{V})\ar[r]^-{-\otimes C}
\ar[d]_-U & \Comon(\ca{V})\ar[d]^-U\\
\ca{V}\ar[r]_-{-\otimes UC} & \ca{V}}
\end{displaymath}
where the comonadic forgetful $U$
creates all colimits and $-\otimes UC$
preserves them since $\ca{V}$ is monoidal closed.
Hence we have an adjunction
\begin{displaymath}
\xymatrix @C=.65in
 {\Comon(\ca{V})\ar@<+.8ex>[r]^-{(-\otimes C)}
 \ar@{}[r]|-\bot & \Comon(\ca{V})
 \ar@<+.8ex>[l]^-{\HOM(C,-)}}
\end{displaymath}
where $\HOM$ denotes the internal 
hom of $\Comon(\ca{V})$.
\end{proof}
\begin{cor*}
 For a commutative ring $R$, the category of $R$-algebras
 $\Alg_R$ is monadic over $\Mod_R$ and locally presentable,
 and the category of $R$-coalgebras $\Coalg_R$ is comonadic
 over $\Mod_R$, locally presentable and monoidal closed.
\end{cor*}
The fact that $\Coalg_R$ is locally presentable
in fact generalizes the \emph{Fundamental Theorem
of Coalgebras}, 
which states that every $k$-coalgebra 
for a field $k$ is a filtered 
colimit of finite dimensional coalgebras, \emph{i.e.}
whose underlying vector space is finite
dimensional (see \cite{Sweedler,HopfAlg}). 
These are precisely the 
finitely presentable objects in $\Coalg_k$, hence 
we obtain an analogous statement for $\Coalg_R$
for a commutative ring $R$.

\section{Categories of modules and comodules}\label{Categoriesofmodulesandcomodules}

If $(A,m,\eta)$ is a monoid in a monoidal
category $\ca{V}$, a \emph{(left) $A$-module}
is an object $M$ of $\ca{V}$ equipped with an
arrow $\mu:A\otimes M\to M$ called
\emph{action}, such that the diagrams
\begin{equation}\label{defmod} 
\xymatrix @C=.45in
{A\otimes A\otimes M\ar[r]^-{m\otimes1}\ar[d]_-{1\otimes\mu} 
& A\otimes M\ar[d]^-{\mu}\\ 
A\otimes M\ar[r]_-{\mu} & M}
\qquad\mathrm{and}\qquad
\xymatrix 
{& A\otimes M\ar[rd]^-{\mu} & \\ 
I\otimes M\ar[rr]_-{l_M}\ar[ur]^-{\eta\otimes1} && M}
\end{equation}
commute, where $a$ is suppressed. 
An \emph{$A$-module morphism} $(M,\mu)\to(M',\mu')$
is an arrow $f:M\to M'$ in $\ca{V}$ such that
the diagram
\begin{equation}\label{defmod2}
 \xymatrix 
{A\otimes M\ar[r]^-{1\otimes
f}\ar[d]_-{\mu} & A\otimes M'\ar[d]^-{\mu'}\\ 
M\ar[r]_-f & M'}
\end{equation}
commutes. Thus for any monoid $A$ in $\ca{V}$,
there is a category $\Mod_\ca{V}(A)$
of left $A$-modules
and $A$-module morphisms.

Dually, a \emph{(right) $C$-comodule}
for $(C,\Delta,\epsilon)$ a comonoid in $\ca{V}$
is an object $X$ in $\ca{V}$ 
together with the \emph{coaction}
$\delta:X\to X\otimes C$, satisfying compatibility 
conditions with the comultiplication and counit.
A \emph{$C$-comodule morphism} $(X,\delta)\to(X',\delta')$
is a arrow $g:X\to X'$ in $\ca{V}$ which respects the
coactions. There is a category of right $C$-comodules
$\Comod_\ca{V}(C)$ for every comonoid $C$ in a 
monoidal category $\ca{V}$. 

In a very similar way, we can define categories of 
\emph{right $A$-modules} and \emph{left $C$-comodules}.
If $\ca{V}$ is a symmetric monoidal category, there
is an obvious 
isomorphism between categories of left and right
$A$-modules and left and right $C$-comodules, so 
usually there
is no distinction in the notation
between left and right
modules and comodules.

For example, in the monoidal category
of abelian groups $\B{Ab}$, the category
of modules for a ring $R\in\Mon(\B{Ab})$
is precisely the category of $R$-modules
$\Mod_R$. Moreover,
for $\ca{V}=\Mod_R$ itself, we denote
by $\Mod_A$ the category of those $R$-modules
which are equipped with the structure of an 
$A$-module for an $R$-algebra $A\in\Mon(\Mod_R)$. 
Similarly, $\Comod_C$ is the category of $C$-comodules
for an $R$-coalgebra $C\in\Comon(\Mod_R)$.

Recall how, when a monoidal category $\ca{V}$ is viewed 
as the hom-category
$\ca{K}(\star,\star)$ of a bicategory $\ca{K}$ with one object $\star$,
a monoid $A$ in $\ca{V}$ is precisely a monad in $\ca{K}$
(Remark \ref{monadsaremonoids}). This 
analogy carries over to modules for a monoid in $\ca{V}$.
In Definition \ref{lefttmodules}, the category of left
$t$-modules for a monad $t$ in the bicategory $\ca{K}$
was defined to be the category of Eilenberg-Moore algebras
for the monad `post-composition with $t$'. For the
one-object case, since the tensor product
of $\ca{K}(\star,\star)$ is just horizontal composition,
the following well-known fact is immediately implied.
\begin{prop}\label{modulesmonadic}
For any monoid $A$ and any comonoid $C$
in a monoidal category $\ca{V}$, the categories
of $A$-modules $\Mod_\ca{V}(A)$ and $C$-comodules
$\Comod_\ca{V}(C)$ are respectively monadic and comonadic
over $\ca{V}$.
\end{prop} 
Explicitly, the category
of (left) modules for a monoid $(A,m,\eta)$
is the category of algebras
for the monad $(A\otimes -,\eta\otimes -,m\otimes -)$
on $\ca{V}$,
and the category of (right) comodules for
a comonoid $(C,\Delta,\epsilon)$
is the category of coalgebras for the comonad
$(-\otimes C,-\otimes\epsilon,-\otimes\Delta)$
on $\ca{V}$.

In the previous section, it was demonstrated 
how a lax monoidal functor between monoidal
categories $F:\ca{V}\to\ca{W}$ induces a functor
$\Mon F$ between their categories of monoids, 
as in (\ref{MonF}). Furthermore, for any 
monoid $A$ in $\ca{V}$, there is an induced
functor between the categories of modules
\begin{equation}\label{ModF}
\Mod F:\xymatrix @R=.02in
{\Mod_\ca{V}(A)\ar[r]
& \Mod_\ca{W}(FA)\\ 
(M,\mu)\ar @{|->}[r] & (FM,\mu')}
\end{equation}
where the object $FM$ in $\ca{W}$ obtains
the structure of a $FA$-module via the action
\begin{displaymath} 
\mu':FA\otimes FM\xrightarrow{\phi_{A,M}}
F(A\otimes M)\xrightarrow{F\mu}FM
\end{displaymath}
with $\phi_{A.M}$ the lax structure map of $F$. 

As an application, consider the internal hom
functor $[-,-]:\ca{V}^\mathrm{op}\times\ca{V}\to\ca{V}$
in a symmetric monoidal closed category $\ca{V}$.
By Proposition \ref{laxinthom} it is lax monoidal, as the parametrized adjoint of 
the strong monoidal $(-\otimes-)$, and it
induces the functor $\Mon[-,-]$ as in (\ref{defMon[]}).
Now, a monoid in $\ca{V}^\mathrm{op}\times\ca{V}$ is a pair
$(C,A)$ where $C$ is a comonoid and $A$ a monoid, and also
\begin{displaymath}
 \Mod_{\ca{V}^\mathrm{op}\times\ca{V}}((C,A))\cong
\Mod_{\ca{V}^\mathrm{op}}(C)\times\Mod_\ca{V}(A)
 \cong
\Comod_\ca{V}(C)^\mathrm{op}\times\Mod_\ca{V}(A).
\end{displaymath}
Hence the induced functor (\ref{ModF}) in this case is
\begin{equation} \label{defMod[]}
\Mod[-,-]:
\xymatrix @R=.05in
{\Comod_\ca{V}(C)^\mathrm{op}\times\Mod_\ca{V}(A)\ar[r]
& \Mod_\ca{V}([C,A])\\ 
\qquad(\;(X,\delta)\;,\;(M,\mu)\;)\;\ar
@{|->}[r] & ([X,M],\mu').}
\end{equation} 
This concretely means that whenever $X$ is a $C$-comodule
and $M$ is an $A$-module, the object $[X,M]$ obtains the structure 
of a $[C,A]$-module, with action 
\begin{displaymath}
 \mu':[C,A]\otimes[X,M]\to [X,M]
\end{displaymath}
which is the transpose under $-\otimes X\dashv [X,-]$ of 
the composite
\begin{equation}\label{lala} 
\xymatrix @R=.2in @C=.25in
{[C,A]\otimes
[X,M]\otimes X\ar[r]^-{1\otimes\delta}\ar @{-->}[ddrr] &
[C,A]\otimes [X,M]\otimes X\otimes
C\ar[r]^-{1\otimes s} &
[C,A]\otimes C\otimes
[X,M]\otimes X\ar[d]^-{\mathrm{ev}\otimes
\mathrm{ev}}\\ 
&& A\otimes
M\ar[d]^-{\mu}\\ 
&& M.}
\end{equation}
\begin{cor*}
For $A$ an $R$-algebra and $C$ an $R$-coalgebra
for a commutative ring $R$, there is an induced map
\begin{displaymath}
\Mod(\Hom_R):
\xymatrix @R=.05in
{\Comod_C^\mathrm{op}\times\Mod_A \ar[r] &
\Mod_{\Hom_R(C,A)} \\
\quad(\;X\;,\;M\;)\;\ar @{|->}[r] & \Hom_R(X,M)}
\end{displaymath}
which endows the $R$-module of linear
maps between $X$ and $M$ with the structure of a
$\Hom_R(C,A)$-module.
\end{cor*}
In the previous section, it turned out
that for the class of admissible monoidal
categories, the categories of monoids
and comonoids had very useful properties
(see Proposition \ref{moncomonadm}).
As far as the categories of modules and comodules
are concerned, $\Comod_\ca{V}(C)$ is again more
particular than $\Mod_\ca{V}(A)$ and
similar techniques as for $\Comon(\ca{V})$
can be used. The following generalizes the results
for comodules over a coalgebra in $\ca{V}=\Mod_R$
of \cite{CoringsComod}.
\begin{prop}\label{comodlocpresent}
Suppose $\ca{V}$ is a locally presentable
monoidal category, such that $\otimes$ preserves
filtered colimits in both variables. Then
\begin{enumerate}
\item $\Mod_\ca{V}(A)$ for a monoid A is 
finitary monadic over $\ca{V}$ and so locally presentable.
\item $\Comod_\ca{V}(C)$ for a comonoid C is a locally
presentable category.
\end{enumerate}
\end{prop} 
\begin{proof}
By Proposition \ref{modulesmonadic}, the endofunctor on 
$\ca{V}$ which induces the monad for which the algebras
are (left) $A$-modules is $(A\otimes -)$, which is finitary
by assumptions. 

Similarly, the endofunctor which gives rise to the comonadic 
$\Comod_\ca{V}(C)$ over $\ca{V}$
is $F_C=-\otimes C$, which is
also finitary. Imitating 
the proof of Proposition \ref{moncomonadm}, consider the
category of functor $F_C$-coalgebras which
contains $\Comod_\ca{V}(C)$ as its full subcategory,
closed under formation of colimits. Then $\Coalg F_C$
is comonadic over $\ca{V}$ and locally presentable itself.
Now define pairs of natural transformations
\begin{displaymath}
 \phi^1,\psi^1:
\xymatrix@C=.5in
{\Coalg F_C\rrtwocell^{U}_{F_CF_CU} && \ca{V}},\quad
\phi^2,\psi^2:
\xymatrix@C=.5in
{\Coalg F_C\rrtwocell^{U}_{(-\otimes I)U} && \ca{V}}
\end{displaymath}
with components
\begin{displaymath}
 \phi^1_X:X\xrightarrow{\beta}X\otimes C\xrightarrow{\beta\otimes1}
X\otimes C\otimes C
\qquad\mathrm{and}\qquad
\phi^2_X:X\xrightarrow{\beta}X\otimes C\xrightarrow{1\otimes\epsilon}
X\otimes I
\end{displaymath}
\begin{displaymath}
 \psi^1_X:X\xrightarrow{\beta}X\otimes C\xrightarrow{1\otimes\Delta}
X\otimes C\otimes C
\qquad\qquad\qquad 
\psi^2_X:X\xrightarrow{r^{-1}}X\otimes I\qquad\qquad
\end{displaymath}
where $\beta:X\to X\otimes C$ is the structure
map of the functor $F_C$-coalgebra $X$, and $\Delta,\epsilon$
are the comultiplication and counit
of the comonoid $C$. Since all categories and 
functors involved are accessible, the equifier of this
family of natural transformations is accessible as well.
It is not hard to see that
\begin{displaymath}
 \B{Eq}((\phi^i,\psi^i)_{i=1,2})=\Comod_\ca{V}(C)
\end{displaymath}
so the category of comodules is accessible
and moreover cocomplete, thus locally presentable.
\end{proof}
The above proposition indicates the structure 
that finitary monadic and finitary 
comonadic categories over locally presentable categories
inherit. We note that Gabriel and Ulmer's result in \cite{GabrielUlmer}
for algebras of finitary monads
does not seem to dualize, but by following a similar approach 
to Ad{\'a}mek and Rosick{\'y}'s `Locally presentable and accessible categories',
we obtain the following result.
\begin{thm}\label{Monadiccomonadicpresentability}
Suppose that $\ca{C}$ is a locally presentable category.
\begin{itemize}
\item If $(T,m,\eta)$ is a finitary monad on $\ca{C}$, the category 
of algebras $\ca{C}^T$ is locally presentable.
\item If $(S,\Delta,\epsilon)$ is a finitary comonad on $\ca{C}$, the category
of coalgebras $\ca{C}^S$ is locally presentable.
\end{itemize}
\end{thm}
\begin{proof}
The category of Eilenberg-Moore algebras $\ca{C}^T$ is always a 
full subcategory of the locally presentable
category of endofunctor algebras $\Alg T$ (see previous section).
More precisely, it is expressed as an equifier of natural transformations
between accessible functors $\Alg T\to\ca{C}$
hence is accessible as in \cite[2.78]{LocallyPresentable}, 
and by default is also complete.

On the other hand, the category of coalgebras $\ca{C}^S$ 
is a full subcategory of the locally presentable 
category of endofunctor coalgebras $\Coalg T$, expressed 
as the equifier $\ca{C}^S=\mathbf{Eq}\big((\phi^t,\psi^t)_{t=1,2}\big)$ for
\begin{displaymath}
\xymatrix@C=.8in
{\Coalg S\rtwocell<\omit>{\;\quad\scriptscriptstyle{\phi^1,\psi^1}}
\ar@/^3ex/[r]^-U \ar@/_3ex/[r]_-{SSU} & \ca{C}}\qquad\mathrm{with}\quad
\xymatrix @R=.03in
{\phi^1_{(C,\beta)}:C\ar[r]^-\beta & SC\ar[r]^-{S\beta} & SSC \\
\psi^1_{(C,\beta)}:C\ar[r]^-\beta & SC\ar[r]^-{\Delta_C} & SSC}
\end{displaymath}
\begin{displaymath}
\xymatrix@C=.8in
{\Coalg S\rtwocell<\omit>{\;\quad\scriptscriptstyle{\phi^2,\psi^2}}
\ar@/^3ex/[r]^-U \ar@/_3ex/[r]_-U & \ca{C}}\qquad\mathrm{with}\quad
\xymatrix @R=.03in
{\phi^2_{(C,\beta)}:C\ar[r]^-\beta & SC\ar[r]^-{\epsilon_C} & C \\
\psi^2_{(C,\beta)}:C\ar[rr]^-{1_C} && C.}
\end{displaymath}
All categories and functors involved are accessible, hence $\ca{C}^S$ 
is an accessible category, with all colimits created from $\ca{C}$.
\end{proof}

Proposition \ref{comodlocpresent} could directly
be established from the above. Notice that the assumptions on $\ca{V}$
could of course be changed to `locally presentable,
symmetric monoidal category, such that $B\otimes -$
preserves filtered colimits', \emph{i.e.} admissible monoidal
category. As mentioned earlier, symmetry allows us to 
identify in a sense the categories of left and right modules and comodules,
without distinguishing cases in the respective proofs.
Even in the non-symmetric case though, the results hold
for all four cases separately.
\begin{cor*}
If $A$ is an $R$-algebra and $C$ an $R$-coalgebra
for a commutative ring $R$,
the categories $\Mod_A$ and $\Comod_C$
are locally presentable. 
\end{cor*}
Notably, many useful properties and 
constructions for $\Comod_C$ in the category 
$\Mod_R$ are included in 
Wischnewsky's \cite{LinearReps}.

So far we have studied categories of modules
and comodules for fixed monoids and comonoids
in a monoidal category $\ca{V}$. Since a (co)module
is just an object in $\ca{V}$ with extra structure,
relative to some (co)monoid, it could be expected
that the same object is possible to be endowed
with (co)module 
structures relating it with different (co)monoids.

Suppose that $A,B$ are two monoids
in the monoidal category $\ca{V}$.
Each monoid morphism $f:A\to B$ between them
determines a functor 
\begin{equation}\label{defres}
f^*:\Mod_\ca{V}(B)\longrightarrow\Mod_\ca{V}(A)
\end{equation}
which makes every $B$-module $(N,\mu)$ into an $A$-module
$f^* N$ via the action
\begin{displaymath} 
A\otimes N\xrightarrow{f\otimes1}
B\otimes N\xrightarrow{\mu}N.
\end{displaymath}
This functor is 
sometimes called \emph{restriction of scalars} along $f$.
Also, each $B$-module arrow becomes
an $A$-module arrow (\emph{i.e.} commutes with the $A$-actions), 
and so we have a commutative triangle of categories
and functors
\begin{equation}\label{tr1}
\xymatrix @R=.22in
{\Mod_\ca{V}(B)\ar[rr]^-{f^*}\ar[dr] &&
\Mod_\ca{V}(A)\ar[dl] \\
 & \ca{V}. &}
\end{equation}
On the other hand, if $C$ and $D$ are two
comonoids in $\ca{V}$,
each comonoid arrow $g:C\to D$
induces a functor
\begin{equation}\label{defcores}
g_!:\Comod_\ca{V}(C)\longrightarrow\Comod_\ca{V}(D)
\end{equation}
which makes every $C$-comodule $(X,\delta)$ into a $D$-comodule
$g_!X$ via the coaction
\begin{displaymath}
X\xrightarrow{\delta}X\otimes C\xrightarrow{1\otimes g}X\otimes D,
\end{displaymath}
called \emph{corestriction of scalars}
along $g$. The respective commutative triangle is
\begin{equation}\label{tr2}
\xymatrix @R=.22in
{\Comod_\ca{V}(C)\ar[rr]^-{g_!}\ar[dr] &&
\Comod_\ca{V}(D)\ar[dl]\\
 & \ca{V}. &}
\end{equation}
Notice that by the above triangles, where the legs are monadic 
and comonadic respectively, 
$f^*$ is a continuous functor and 
$g_!$ is a cocontinuous functor when $\ca{V}$ is (co)complete.

It is often of interest to deduce the existence
of adjoints of the functors $f^*$ and $g_!$.
This is why the last part of this section
is a digression, devoted to the identification of 
certain assumptions 
on the monoidal category
$\ca{V}$ which permit the explicit construction
of such adjoints. Most of the constructions
are well-known in particular categories, like 
$\ca{V}$=$\B{Ab}$ for the categories 
of modules for rings, which is also 
our motivating example.

If $A,B$ are two monoids in $\ca{V}$,
define a \emph{left $A$/right $B$-bimodule}
$M$ to be an object in $\ca{V}$ with 
a left $A$-action $A\otimes M\xrightarrow{\lambda}M$
and a right $B$-action 
$M\otimes B\xrightarrow{\rho}B$ such that the actions commute, 
and denote it
by $_AM_B$. In a dual way,
we can define a \emph{left $C$/right $D$-bicomodule}
$_CX_D$.
\begin{enumerate}[i)]
 \item In an arbitrary monoidal
category $\ca{V}$, the \emph{tensor product} of
the bimodules $_AM_B$, $_BN_{A'}$ over $B$
is the coequalizer
\begin{equation}\label{tensor}
\xymatrix @C=.5in
{M\otimes B\otimes N\ar @<+.8ex>[r]^-{1\otimes\lambda_N}
\ar @<-.8ex>[r]_-{\rho_M\otimes 1} &
M\otimes N\ar @{->>}[r] &
M\otimes_B N}
\end{equation}
where $\rho_M$ is the right $B$-action on $M$
and $\lambda_N$ is the left $B$-action on $N$.
Dually, the \emph{cotensor product}
for bicomodules $_CX_D$, $_DY_{C'}$ over $D$
is the equalizer
\begin{displaymath}
\xymatrix
{X\square_DY\, \ar @{>->}[r] &
X\otimes Y \ar @<+.8ex>[r]^-{r_X\otimes 1}
\ar @<-.8ex>[r]_-{1\otimes l_Y} &
X\otimes D\otimes Y}
\end{displaymath}
where $r_X$ is the right $D$-coaction on $X$ and 
$l_Y$ is the left $D$-coaction on $Y$.
\item In a symmetric monoidal closed category $\ca{V}$,
we can form $\Hom_A(M,N)$ for two $A$-modules
$M,N$ as the equalizer
\begin{displaymath}
\xymatrix
{\Hom_A(M,N)\,\, \ar @{>->}[r] &
[M,N]\ar @<+.8ex>[r]^-{t} \ar @<-.8ex>[r]_-{k} &
[A,[M,N]]}
\end{displaymath}
where $t$ corresponds under $-\otimes X\dashv[X,-]$ to
\begin{displaymath}
[M,N]\otimes A\otimes M\xrightarrow{1\otimes s}
[M,N]\otimes M\otimes A\xrightarrow{\mathrm{ev}\otimes1}
N\otimes A\xrightarrow{\rho_N}N
\end{displaymath}
and $k$ corresponds to
\begin{displaymath}
[M,N]\otimes A\otimes M\xrightarrow{1\otimes\lambda_M}
[M,N]\otimes M\xrightarrow{\mathrm{ev}}N.
\end{displaymath}
\end{enumerate}
\begin{prop}\label{f*leftadjoint}
Suppose that the monoidal category 
$\ca{V}$ has coequalizers and the functor
$B\otimes -$ preserves them for any monoid $B$.
Then the functor $f^*$ has a left adjoint,
for any monoid morphism $f$. Dually, 
if $\ca{V}$ has equalizers and the functor
$-\otimes C$ preserves them for any
comonoid $C$, then $g_!$ has a right adjoint
for any comonoid morphism $g$.
\end{prop}
\begin{proof}
Firstly notice that any monoid $A$
can be considered as a left and right
$A$-module via multiplication, and any comonoid
$C$ is a left and right $C$-comodule
via comultiplication.

When $B$ is viewed as a left $B$/right $A$-bimodule
via restriction of scalars along $f:A\to B$, there
exists a natural bijection
\begin{displaymath}
\Mod_\ca{V}(B)(B\otimes_AM,N)\cong\Mod_\ca{V}(A)(M,f^* N)
\end{displaymath}
for any left $A$-module M and left $B$-module N, which establishes
an adjunction
\begin{displaymath}
\xymatrix @C=.6in
{\Mod_\ca{V}(A)\ar @<+.8ex>[r]^-{B\otimes_A-}\ar@{}[r]|-\bot
& \Mod_\ca{V}(B).\ar @<+.8ex>[l]^-{f^*}}
\end{displaymath}
Notice that the left $B$-action on
$B\otimes_AM$ is induced by universality
of the top coequalizer, since $B\otimes-$ preserves them:
\begin{displaymath}
\xymatrix @R=.07in @C=.1in
{B\otimes B\otimes A\otimes M \ar @/^/[rr]^-{1\otimes1\otimes\lambda_M}
\ar @/_/[dr]_-{1\otimes1\otimes f\otimes1} \ar[ddd]_-{m\otimes1\otimes1} &&
B\otimes B\otimes M\ar @{->>}[rr] 
\ar[ddd]^-{m\otimes1} &&
B\otimes B\otimes_AM \ar @{-->}[ddd]^-{\exists!\lambda_{B\otimes_A M}}\\
& B\otimes B\otimes B\otimes M \ar @/_/[ru]_-{1\otimes m\otimes1} &&& \\
&&&&& \\
B\otimes A\otimes M \ar @/^/[rr]^-{1\otimes\mu}
\ar @/_/[dr]_-{1\otimes f\otimes1} &&
B\otimes M\ar @{->>}[rr] &&
B\otimes_AM.\\
& B\otimes B\otimes M \ar @/_/[ru]_-{m\otimes1} &&&}
\end{displaymath}
Dually, for a comonoid arrow $g:C\to D$ we have
the adjunction
\begin{displaymath}
\xymatrix @C=.6in
{\Comod_\ca{V}(C)\ar @<+.8ex>[r]^-{g_!}\ar@{}[r]|-\bot 
& \Mod_\ca{V}(D)\ar @<+.8ex>[l]^-{-\square_DC}}
\end{displaymath}
when $C$ is viewed as a left $D$-comodule
via corestriction along $g$.
\end{proof}
\begin{rmk*}
By the adjoint lifting theorem (see for example \cite[1.1.3]{Elephant1}),
we can deduce the sheer existence of a left adjoint for $f^*$ 
and a right adjoint for $g_!$ if $\Mod_\ca{V}(B)$ and 
$\Comod_\ca{V}(C)$ have (co)equalizers (of (co)reflexive 
pairs) accordingly. This happens because 
the legs of the triangles (\ref{tr1}, \ref{tr2})
are respectively monadic and comonadic. Of course, this agrees
with the assumptions of the above proposition, since
$B\otimes\textrm{-}$ and $\textrm{-}\otimes C$ are the monad
and comonad which give rise to the (co)monadic categories
of modules and comodules.
\end{rmk*}
\begin{prop}\label{f*rightadjoint}
If $\ca{V}$ is a symmetric monoidal
closed category with equalizers, then 
$f^*$ has a right adjoint for any monoid
arrow $f$. Dually, if $\ca{V}$ has coequalizers
and $\ca{V}^\mathrm{op}$ is monoidal
closed, then $g_!$ has a left adjoint for
any comonoid arrow $g$.
\end{prop}
\begin{proof}
 There is a natural bijection
\begin{displaymath}
 \Mod_\ca{V}(A)(f^* M,N)\cong\Mod_\ca{V}(B)(M,\Hom_A(B,N))
\end{displaymath}
for any $B$-module $M$, $A$-module $N$ and $f:A\to B$ monoid morphism.
Thus we have an adjunction
\begin{displaymath}
\xymatrix @C=.6in
{\Mod_\ca{V}(B)\ar @<+.8ex>[r]^-{f^*}
\ar@{}[r]|-\bot
& \Mod_\ca{V}(A)\ar @<+.8ex>[l]^-{\Hom_A(B,-)}.}
\end{displaymath}
The $B$-action on $\Hom_A(B,N)$ is the 
unique map induced 
by universality of the bottom equalizer
\begin{displaymath}
\xymatrix
{B\otimes\Hom_A(B,M)\,\, \ar @{>->}[r] 
\ar @{-->}[d]_-{\exists!\lambda_{\Hom_A(B,M)}} &
B\otimes[B,M]\ar @<+.8ex>[r]^-{1\otimes t} 
\ar @<-.8ex>[r]_-{1\otimes k} \ar[d]_-u
&
B\otimes[A,[B,M]]\ar[d]^-v\\
\Hom_A(B,M)\,\, \ar @{>->}[r] &
[B,M]\ar @<+.8ex>[r]^-t \ar @<-.8ex>[r]_-k &
[A,[B,M]],}
\end{displaymath}
where $u$ and $v$ are adjuncts to composites of multiplication
of $B$ and evaluation. The left adjoint of $g_!$ is constructed dually.
\end{proof}
Obviously, the above sufficient conditions for the existence of adjoints for the 
corestriction of scalars are much less common to appear than the ones 
for the restriction. After all, for most interesting monoidal categories $\ca{V}$, 
their opposite $\ca{V}^\mathrm{op}$ is not monoidal closed.

In particular, for $\ca{V}=\Mod_R$ where $R$ is a commutative ring,
the situation is as follows.
\begin{prop}
The functor $f^*$ for any $R$-algebra morphism $f:A\to B$ has
a pair of adjoints
\begin{displaymath}
\xymatrix @C=.5in
{\Mod_B \ar[rr]|-{f^*} &&
\Mod_A. \ar @/_4ex/[ll]_-{B\otimes_A -}
^-{\bot}
\ar @/^4ex/[ll]^-{\Hom_A(B,-)}
_-{\bot}}
\end{displaymath}
On the other hand, the functor $g_!$ has a right adjoint 
for any $R$-coalgebra morphism $g:C\to D$, and also a left adjoint
in certain cases, e.g. if $g$ is between two finitely 
presentable projective $R$-coalgebras.
\end{prop}
\begin{proof}
The symmetric monoidal closed category $\Mod_R$ has 
all limits and colimits, therefore Propositions 
\ref{f*leftadjoint} and \ref{f*rightadjoint} for the
restriction of scalars apply 
and the respective adjoints are constructed as above.
 
Regarding the corestriction of scalars, 
the functor $C\otimes-$ does not in general preserve equalizers
in $\Mod_R$ for any $R$-coalgebra $C$ (except for 
flat coalgebras). Also $\Mod_R^\mathrm{op}$ is not 
a monoidal closed category, thus the above propositions do not apply
in this case. 
However, since $g_!$ is cocontinuous
and $\Comod_C$ is a locally presentable category,
Theorem \ref{Kelly} can be applied instead, to give the existence
of a right adjoint for any $g_!$. In particular, when $C$ is 
a flat coalgebra, we can construct this adjoint as above:
\begin{displaymath}
\xymatrix
{\Comod_C \ar @/^2ex/[rr]^-{g_!}
\ar @{}[rr]|-\bot &&
\Comod_D. \ar @/^2ex/[ll]^-{-\square_D C}}
\end{displaymath}
Moreover, $\Comod_C$ is complete, well-powered and has a cogenerator as
shown in \cite{LinearReps}. We can then apply the special adjoint functor theorem
to obtain a right adjoint only when $g_!$ preserves all limits. For example,
if the coalgebras $C$ and $D$ have duals in $\Mod_R$, the functors 
$-\otimes C$ and $-\otimes D$ preserve limits.
Hence in the commutative triangle (\ref{tr2}), the comonadic legs create
all limits that the comonads preserve, hence $g_!$ is continuous.
\end{proof}

\chapter{Enrichment}\label{enrichment}
This chapter begins by presenting the most basic definitions
and structures related to enriched category theory, largely 
following the standard book on the subject by 
Kelly \cite{Kelly}. 

Then, a brief introduction 
to enriched bimodules is given, intended to clarify
certain essential concepts of Chapter \ref{VCatsVCocats}. The theory 
of bimodules (or distributors or profunctors) has 
been widely studied, and the notion of a 
distributor was first introduced by Lawvere.
Here we restrict to the parts relevant to
what follows, hence more emphasis is given on one-sided
modules. Appropriate references are 
\cite{Distributeurs,Handbook1,Monoidalbicats&hopfalgebroids},
and also \cite{GarnerShulman} where a theory 
of modules not between enriched categories 
but between \emph{enriched bicategories} is developed.

In the last section, we give 
the definition of an action of a monoidal category
on an ordinary category and we demonstrate in detail
how a $\ca{V}$-representation may give rise to a
$\ca{V}$-enriched category. This forms
one direction of a correspondence between
categories with an action from $\ca{V}$ with a certain adjoint
and tensored $\ca{V}$-categories, for $\ca{V}$ a right closed
monoidal category. In fact, the adjoint gives the hom-objects
and the action gives the tensor of the enriched category. 
The main references are \cite{enrthrvar,AnoteonActions},
and for example in \cite{Mccruddencoalgebroidsreps} the 
structure of the 2-category of $\ca{V}$-\emph{actegories}
(\emph{i.e.} $\ca{V}$-representations)
$\ca{V}$-$\B{Act}$ is explored.

\section{Basic definitions}\label{basicdefienrichment}
Suppose that $(\ca{V},\otimes,I,a,l,r)$ is a monoidal category. A
$\ca{V}$-\emph{enriched category} $\ca{A}$ 
consists of a set ob$\ca{A}$ of objects, 
a hom-object $\ca{A}(A,B)\in\ca{V}$ 
for each pair of objects of $\ca{A}$, a composition law
\begin{equation}\label{compositionlawplain}
 M:\ca{A}(B,C)\otimes\ca{A}(A,B)\to\ca{A}(A,C)
\end{equation}
for each triple of objects, and 
an identity element $j_A:I\to\ca{A}(A,A)$ for each object, subject 
to the associativity and unit axioms expressed by the commutativity of
\begin{displaymath}
\xymatrix @C=.01in @R=.5in
{(\ca{A}(C,D)\otimes \ca{A}(B,C))\otimes 
\ca{A}(A,B)\ar[rr]^-{a}\ar[d]_-{M\otimes
1} &&
\ca{A}(C,D)\otimes(\ca{A}(B,C)\otimes \ca{A}(A,B))\ar[d]^-{1\otimes
M}\\ \ca{A}(B,D)\otimes \ca{A}(A,B)\ar[dr]_M &&
\ca{A}(C,D)\otimes \ca{A}(A,C)\ar[dl]^M\\ & \ca{A}(A,D),}
\end{displaymath}
\begin{displaymath}
\xymatrix @C=.6in
{\ca{A}(B,B)\otimes \ca{A}(A,B) \ar[r]^-{M} &
\ca{A}(A,B) & \ca{A}(A,B)\otimes \ca{A}(A,A) \ar[l]_-{M}\\ I\otimes \ca{A}(A,B)
\ar[u]^-{j_B\otimes 1} \ar[ur]_-{l} && \ca{A}(A,B)\otimes I.\ar[ul]^-r
\ar[u]_-{1\otimes j_A}}
\end{displaymath}
For example, $\B{Set}$-enriched categories
are ordinary small categories,
$\B{Ab}$-categories are additive categories,
$\B{Vect}_k$-categories are $k$-linear categories 
and $\B{Cat}$-enriched categories are 2-categories.
The latter gives a different perspective of 2-category 
theory from the one presented in Chapter \ref{bicategories}. 
Thus, in order to deal with 2-categories
we can either employ the theory of bicategories 
or the theory of enriched categories.

Notice how in all the examples above,
the \emph{base} $\ca{V}$ of the enrichment
is in fact enriched over itself: $\B{Set}$
is an ordinary category, $\B{Ab}$ is an additive category,
$\B{Vect}_k$ is a $k$-linear category and $\B{Cat}$ 
is a 2-category. This is due to the fact that
the base monoidal categories are closed, and 
the internal hom functor in any monoidal
closed category $\ca{V}$ 
\begin{displaymath}
 [-,-]:\ca{V}^\mathrm{op}\times\ca{V}\longrightarrow\ca{V}
\end{displaymath}
induces an enrichment
of the category over itself: the hom-object
for $A,B\in\ca{V}$ is $[A,B]$, the composition
law $M:[B,C]\otimes[A,B]\to[A,C]$
corresponds under the adjunction $-\otimes X\dashv[X,-]$ to
the composite
\begin{displaymath}
 [B,C]\otimes[A,B]\otimes A\xrightarrow{1\otimes\mathrm{ev}}[B,C]\otimes B
\xrightarrow{\mathrm{ev}}C
\end{displaymath}
and the identity $I\to[A,A]$ corresponds to
$ I\otimes A\xrightarrow{l_A}A.$
It is a straightforward
verification that these data indeed exhibit $\ca{V}$
as a $\ca{V}$-category.

If $\ca{A}$ is a $\ca{V}$-category for a 
symmetric monoidal category $\ca{V}$, then 
$\ca{A}^\mathrm{op}$ 
is also a $\ca{V}$-category called
the \emph{opposite $\ca{V}$-category}, with the same objects 
ob$\ca{A}^\mathrm{op}$~=~ob$\ca{A}$, 
and hom-objects $\ca{A}^\mathrm{op}(A,B):=\ca{A}(B,A)$. 
The composition law 
$\ca{A}^\mathrm{op}(B,C)\otimes\ca{A}^\mathrm{op}(A,B)
\to\ca{A}^\mathrm{op}(A,C)$ is
\begin{displaymath}
\ca{A}(B,C)\otimes\ca{A}(B,A)\xrightarrow{s}\ca{A}(B,A)\otimes\ca{A}(C,B)
\xrightarrow{M}\ca{A}(C,A)
\end{displaymath} 
and the identity elements
$I\to\ca{A}^\mathrm{op}(A,A)$ are 
the same as in $\ca{A}$.

For $\ca{V}$-categories $\ca{A}$ and $\ca{B}$, a 
$\ca{V}$-\emph{functor}
$F:\ca{A}\to\ca{B}$ between them consists of a 
function $F:\mathrm{ob}\ca{A}\to
\mathrm{ob}\ca{B}$ together with a map 
\begin{equation}\label{classicVfunctor}
F_{AB}:\ca{A}(A,B)\to
\ca{B}(FA,FB)
\end{equation}
for each pair of 
objects in $\ca{A}$,
subject to the commutativity of
\begin{equation}\label{Venrichedfunctordiagrams}
\xymatrix @R=.5in @C=.5in
{\ca{A}(B,C)\otimes\ca{A}(A,B)\ar[r]^-{M}
\ar[d]_-{F_{BC}\otimes F_{AB}}
& \ca{A}(A,C)\ar[d]^-{F_{AC}}\\
\ca{B}(FB,FC)\otimes\ca{B}(FA,FB)\ar[r]_-{M} & \ca{B}(FA,FC),}\qquad
\xymatrix @R=.5in @C=.5in
{I\ar[r]^-{j_A}\ar[dr]_-{j_{FA}} & 
\ca{A}(A,A)\ar[d]^-{F_{AA}}\\
&\ca{B}(FA,FA)}
\end{equation}
expressing the compatibility of $F$ with
composition and identities.

The notion of an enriched functor in the context
of the examples above becomes respectively
an ordinary functor, an additive functor,
a $k$-linear functor and a 2-functor. 
Clearly the composite of two composable
$\ca{V}$-functors is again a
$\ca{V}$-functor, and the composition
is associative and unital with $\B{1}_\ca{A}$
the identity $\ca{V}$-functor. 

If $\ca{V}$ is symmetric monoidal closed, then
for any $\ca{V}$-category $\ca{A}$ the assignment 
$(A,B)\mapsto\ca{A}(A,B)$ is in fact the 
object function of a $\ca{V}$-functor
of two variables
\begin{equation}\label{2enrichedhomfunctor}
\Hom_{\ca{A}}:\ca{A}^\mathrm{op}\otimes\ca{A}\longrightarrow\ca{V}
\end{equation}
where $\ca{V}$ is regarded as a $\ca{V}$-category via
the internal hom.
Its partial functors are the 
covariant and the contravariant \emph{representable}
$\ca{V}$-functors $\Hom_{\ca{A}}(A,-):\ca{A}\to\ca{V}$,
$\Hom_{\ca{A}}(-,B):\ca{A}^\mathrm{op}\to\ca{V}$.
For example, the former sends $B\in\ob\ca{A}$
to $\ca{A}(A,B)\in\ca{V}$, and on hom-objects it
consists of arrows
\begin{displaymath}
 \Hom_\ca{A}(A,-)_{BC}:\ca{A}(B,C)\to[\ca{A}(A,B),\ca{A}(A,C)]
\end{displaymath}
which correspond to the composition arrows under 
$(-\otimes A)\dashv[A,-]$.

For $\ca{V}$-functors $F,G:\ca{A}\to\ca{B}$, a 
$\ca{V}$-\emph{natural 
transformation} $\tau:F\Rightarrow G$ consists
of an ob$\ca{A}$-indexed family of components
$\tau_A:I\to\ca{B}(FA,GA)$ satisfying
the $\ca{V}$-naturality condition 
expressed by the commutativity of
\begin{displaymath}
\xymatrix @C=.15in
{& I\otimes\ca{A}(A,B)\ar[rrr]^-{\tau_B\otimes F_{AB}} 
&&& \ca{B}(FB,GB)\otimes
\ca{B}(FA,FB)\ar[dr]^-{M} &\\
\ca{A}(A,B)\ar[ur]^-{l^{-1}}\ar[dr]_-{r^{-1}} &&&&& \ca{B}(FA,GB).\\
& \ca{A}(A,B)\otimes I\ar[rrr]_-{G_{AB}\otimes\tau_A} &&& \ca{B}(GA,GB)\otimes
\ca{B}(FA,GA)\ar[ur]_-{M} &}
\end{displaymath}
It is not hard to see that $\ca{V}$-natural transformations
compose both vertically and horizontally, in a very
similar way to ordinary natural transformations. 
Thus (small) $\ca{V}$-categories, $\ca{V}$-functors and 
a $\ca{V}$-natural transformations constitute a 
2-category, which is denoted by $\ca{V}$-$\B{Cat}$.

When $\ca{V}$ is a symmetric monoidal category, 
we can define a tensor product in $\ca{V}$-$\B{Cat}$: 
$\ca{A}\otimes\ca{B}$ has 
objects ob$\ca{A}$ $\times$ ob$\ca{B}$, 
hom-objects 
\begin{displaymath}
(\ca{A}\otimes\ca{B})
((A,B),(A',B')):=\ca{A}(A,A')\otimes\ca{B}(B,B'),
\end{displaymath}
the composition law is the composite
\begin{displaymath}
 \xymatrix @C=.5in @R=.5in
{\ca{A}(A',A'')\otimes\ca{B}(B',B'')\otimes\ca{A}(A,A')
\otimes\ca{B}(B,B'')\ar[d]_-{1\otimes s\otimes1}
\ar@{-->}[r] & \ca{A}(A,A'')\otimes\ca{B}(B,B'') \\
\ca{A}(A',A'')\otimes\ca{A}(A,A')\otimes\ca{B}(B',B'')
\otimes\ca{B}(B,B'')
\ar[ur]_-{M\otimes M} & }
\end{displaymath}
and the identity element is
\begin{displaymath}
 I\xrightarrow{\sim}I\otimes I\xrightarrow{j_A\otimes j_B}
\ca{A}(A,A)\otimes\ca{B}(B,B).
\end{displaymath}
The axioms are satisfied so $\ca{A}\otimes\ca{B}$
is a $\ca{V}$-category.
The tensor product of $\ca{V}$-functors
and $\ca{V}$-natural transformations can also be defined accordingly, 
so that we obtain a 2-functor 
$\otimes:\ca{V}\text{-}\B{Cat}\times\ca{V}\text{-}\B{Cat}
\to\ca{V}\text{-}\B{Cat}$.
The unit $\ca{I}$ is the \emph{unit $\ca{V}$-category}
with one object $*$ and $\ca{I}(*,*)=I$.
Hence with the appropriate constraints,
$\ca{V}$-$\B{Cat}$ is a \emph{monoidal 2-category}
(for a formal definition, see for example \cite{BaezNeuchl}).
Also, it has a symmetry 
$s_{\ca{A},\ca{B}}:\ca{A}\otimes\ca{B}\cong\ca{B}\otimes\ca{A}$
which renders it a symmetric monoidal 2-category.

There is the so-called `underlying category functor'
\begin{displaymath}
(-)_0:\ca{V}\textrm{-}\B{Cat}\to\B{Cat}
\end{displaymath}
which maps the $\ca{V}$-category 
$\ca{A}$ to the ordinary category 
$\ca{A}_0=\ca{V}$-$\B{Cat}(\ca{I},\ca{A})$,
the \emph{underlying category} of the enriched $\ca{A}$.
Explicitly, $\ca{A}_0$ has the same objects 
as the $\ca{V}$-enriched $\ca{A}$, 
while a map $f:A\to B$ in $\ca{A}_0$ is an \emph{element}
$f:I\to\ca{A}(A,B)$ of $\ca{A}(A,B)$, \emph{i.e.}
$\ca{A}_0(A,B)=\ca{V}(I,\ca{A}(A,B))$
as sets. There are appropriate
definitions for the underlying
$\ca{V}$-functor and underlying
$\ca{V}$-natural transformation.
The amount of information lost in the
passage from enriched categories
to their underlying categories depends 
very much on the base $\ca{V}$. In particular,
how much information about $\ca{A}$ 
is retained by the underlying $\ca{A}_0$ depends on 
`how faithful' the functor $\ca{V}(I,-)$ is.

We saw earlier that if
$\ca{A}$ is enriched in
a symmetric monoidal
closed category $\ca{V}$, there is a 
$\ca{V}$-functor $\Hom_\ca{A}$ as in (\ref{2enrichedhomfunctor})
which gives the hom-objects of the enrichment.
There is also an ordinary functor between the
underlying categories
\begin{equation}\label{enrichedhomfunctor}
 \ca{A}(-,-):
 \xymatrix @R=.02in
 {\ca{A}_0^\mathrm{op}\times\ca{A}_0\ar[r] & \ca{V} \\
 (A,B)\ar@{|->}[r] & \ca{A}(A,B)}
\end{equation}
sometimes called the \emph{enriched hom-functor},
which maps a pair of arrows 
$(A'\xrightarrow{f}A,B\xrightarrow{g}B')$
in $\ca{A}_0^\mathrm{op}\times\ca{A}_0$ to the 
top arrow
\begin{displaymath}
 \xymatrix
 {\ca{A}(A,B)\ar@{-->}[rr]^-{\ca{A}(f,g)} 
 \ar[d]_-{r^{-1}} && \ca{A}(A',B'). \\
 \ca{A}(A,B)\otimes I\ar[d]_-{1\otimes f} && 
 \ca{A}(B,B')\otimes\ca{A}(A',B)\ar[u]_-M \\\
 \ca{A}(A,B)\otimes\ca{A}(A',A)\ar[r]_-M &
 \ca{A}(A',B)\ar[r]_-{l^{-1}} & 
 I\otimes\ca{A}(A',B)\ar[u]_-{g\otimes 1}}
\end{displaymath}
This functor is evidently the composite
\begin{displaymath}
 \ca{A}_0^\mathrm{op}\times\ca{A}_0\longrightarrow
 (\ca{A}^\mathrm{op}\otimes\ca{A})_0
 \xrightarrow{(\Hom_\ca{A})_0}\ca{V}_{(0)}
\end{displaymath}
where the first arrow is a canonical functor
relating the two underlying categories.
Notice how this functor
$\ca{A}(-,-)$, unlike $\Hom_\ca{A}$, 
can be defined for categories enriched in any 
monoidal category $\ca{V}$, without
further conditions on it.

Speaking loosely, we say that an ordinary category $\ca{C}$
is enriched in a monoidal category $\ca{V}$ when we have a 
$\ca{V}$-enriched category $\ca{A}$ and an isomorphism
$\ca{A}_0\cong\ca{C}$. Consequently, \emph{to be enriched 
in} $\ca{V}$ is not a property, 
but additional structure. Of course, 
a given ordinary category may be enriched 
in more than one 
monoidal categories: that is evident in view of the 
change of base described below. But also, a category 
$\ca{C}$ may be enriched in $\ca{V}$ in more than one way, so that 
there may be many different 
$\ca{V}$-categories with the same underlying
ordinary category.
\begin{prop}\label{changeofbase}
Suppose $F:\ca{V}\to\ca{W}$ is a lax
monoidal functor between two monoidal categories. There 
is an induced 2-functor 
\begin{displaymath}
\tilde{F}:\ca{V}\textrm{-}\B{Cat}\longrightarrow\ca{W}\textrm{-}\B{Cat}
\end{displaymath}
between the 2-categories of $\ca{V}$ and $\ca{W}$-enriched categories, 
which maps any $\ca{V}$-category $\ca{A}$ to a 
$\ca{W}$-category with the same objects as $\ca{A}$ and hom-objects 
$F\ca{A}(A,B)$.
\end{prop}
\begin{proof}
Given a $\ca{V}$-category $\ca{A}$, 
the $\ca{W}$-category $\tilde{F}\ca{A}$
has objects ob$(\tilde{F}\ca{A})=$ob$\ca{A}$ and hom-objects
$(\tilde{F}\ca{A})(A,B)=F(\ca{A}(A,B))\in\ca{W}$. 
The composition and the identities are given by
\begin{displaymath}
\xymatrix @C=.7in @R=.5in
{F\ca{A}(B,C)\otimes F\ca{A}(A,B)\ar[d]_-{\phi_{\ca{A}(B,C),\ca{A}(A,B)}}
\ar @{-->}[r] & F\ca{A}(A,C)\\
F(\ca{A}(B,C)\otimes\ca{A}(A,B))\ar[ru]_-{FM}}\qquad
\xymatrix @C=.7in @R=.5in
{I_\ca{W}\ar[d]_-{\phi_0}\ar @{-->}[r] & F\ca{A}(A,A)\\
FI_\ca{V}\ar[ur]_-{Fj_A} &}
\end{displaymath}
where $\phi$, $\phi_0$ are the structure maps 
of the lax monoidal functor $F$. 
It can be checked that the diagrams of 
associativity and identities commute,
therefore $\tilde{F}\ca{A}$ is a $\ca{W}$-category.

If $K:\ca{A}\to\ca{B}$ is a $\ca{V}$-functor with
maps $K_{AB}:\ca{A}(A,B)\to\ca{B}(KA,KB)$ in 
$\ca{V}$ for every pair of objects in $\ca{A}$, we
can form a $\ca{W}$-functor 
\begin{displaymath}
\tilde{F}K:\tilde{F}\ca{A}\to\tilde{F}\ca{B}
\end{displaymath}
with the same function on objects, and for every 
pair of objects in $\tilde{F}\ca{A}$ a map
\begin{displaymath}
(\tilde{F}K)_{AB}:F\ca{A}(A,B)\xrightarrow{F(K_{AB})}F\ca{B}
(KA,KB)
\end{displaymath}
in $\ca{W}$, such that the axioms 
of a $\ca{W}$-functor are satisfied.

If $\tau:K\Rightarrow L$ is 
a $\ca{V}$-natural transformation
between $\ca{V}$-functors $K,L:\ca{A}\to\ca{B}$, 
its image $\tilde{F}\tau:\tilde{F}K\Rightarrow
\tilde{F}L$ consists of an 
ob($\tilde{F}\ca{A}$)-indexed family of components
\begin{displaymath}
\xymatrix @C=.7in @R=.5in
{I_\ca{W}\ar @{-->}[r]^-{\tilde{F}\tau_A}
\ar[d]_-{\phi_0} &
F\ca{B}(KA,LA)\\
FI_\ca{V}\ar[ur]_-{F\tau_A} &}
\end{displaymath}
in $\ca{W}$, which satisfy the $\ca{W}$-naturality condition
in a straightforward way.
\end{proof}
Another standard example of enrichment
is the usual functor category between two
$\ca{V}$-categories, which 
under specific assumptions on 
$\ca{V}$ obtains a $\ca{V}$-enriched
structure itself. Explicitly, when $\ca{V}$ is a symmetric
monoidal closed category with all limits, 
we can define the \emph{enriched functor category}
$[\ca{A},\ca{B}]$
with objects $\ca{V}$-functors 
$\ca{A}\to\ca{B}$,
and hom-object $[\ca{A},\ca{B}](F,G)$ for any two
$\ca{V}$-functors $F$, $G$ the following end
\begin{displaymath}
\xymatrix @C=.3in
 {\int_{A\in\ca{A}}\ca{B}(FA,GA)\;\;\ar@{>->}[r]
& \prod_{\scriptstyle{A\in\ca{A}}}{\ca{B}(FA,GA)}
\ar@<+.8ex>[r] \ar@<-.7ex>[r] & \prod_{\scriptstyle{A,A'\in\ca{A}}}
{[\ca{A}(A,A'),\ca{B}(FA,GA')]}}
\end{displaymath}
constructed in detail in \cite[2.1]{Kelly}.
It is not hard to define a composition law
and identity elements for the functor category, 
and the axioms which make $[\ca{A},\ca{B}]$ into 
$\ca{V}$-category follow from the corresponding
axioms for $\ca{B}$.

It can be deduced that,
when $\ca{V}$ has the above mentioned
properties, the functor
\begin{displaymath}
-\otimes\ca{A}:\ca{V}\text{-}\B{Cat}\to\ca{V}\text{-}\B{Cat}
\end{displaymath}
in the monoidal category $\ca{V}$-$\B{Cat}$
has $[\ca{A},-]$ as its right adjoint, 
with a (2-)natural isomorphism
\begin{displaymath}
 \ca{V}\text{-}\B{Cat}(\ca{A}\otimes\ca{B},\ca{C})
\cong\ca{V}\text{-}\B{Cat}(\ca{A},[\ca{B},\ca{C}])
\end{displaymath}
for any $\ca{V}$-categories $\ca{A}$, $\ca{B}$ and $\ca{C}$.
Therefore the symmetric monoidal 2-category $\ca{V}$-$\B{Cat}$
has also a closed structure. 

\section{$\ca{V}$-enriched bimodules and modules}\label{Vbimodulesandmodules}

As mentioned in Examples \ref{examplesbicat} of bicategories, 
there is a 
bicategory of bimodules $\B{BMod}$ 
where 1-cells are abelian groups $M$
which are left $R$-modules and right $S$-modules
for rings $R,S$, such that the two actions
are compatible. More explicitly,
these actions yield group homomorphisms
\begin{displaymath}
R\otimes M\xrightarrow{\cdot}M, \quad
M\otimes S\xrightarrow{\cdot}M
\end{displaymath}
such that $(r\cdot m)\cdot s=r\cdot(m\cdot s)$
for all $r\in R$, $s\in S$ and $m\in M$.
Morphisms between them are group homomorphisms
$f:M\to N$ which respect the $R$ and $S$-actions. 
These data define a category of $(R,S)$-bimodules
and bimodule maps between them, which is 
furthermore an additive category, \emph{i.e.}
each $_R\Hom_S(M,N)$ is an abelian group. 

There is another equivalent formulation of these 
definitions, which make it easier to obtain a
generalization to $\ca{V}$-enriched modules.
Recall that a ring is essentially the same 
as an $\B{Ab}$-category with only one object,
in the sense that the underlying additive group
of the ring is the single hom-object 
and the multiplication
of the ring is composition law.
Then an $(R,S)$-bimodule can be regarded as
an additive functor
\begin{displaymath}
 S^\op\otimes R\to\B{Ab}
\end{displaymath}
where the opposite ring $S^\op$ has 
reversed multiplication.
Equivalently,
this amounts to an additive functor
\begin{displaymath}
 R\to[S^\op,\B{Ab}].
\end{displaymath}
In these terms, a bimodule map is an 
additive natural transformation between the 
respective additive functors. 

The next step, since $\B{Rng}\text{=}\Mon(\B{Ab})$,
would be to consider bimodules for monoids
in an arbitrary monoidal category $\ca{V}$.
The definitions that arise are precisely
the ones which were employed
in Section \ref{Categoriesofmodulesandcomodules}
in order to study the existence of 
adjoints for the restriction of scalars. By analogy
with ring bimodules,
an $(A,B)$-bimodule $M$ for
monoids $A$ and $B$ in a symmetric monoidal closed
$\ca{V}$ can 
also be expressed as a $\ca{V}$-functor
\begin{displaymath}
 M:B^\op\otimes A\to\ca{V}
\end{displaymath}
where the monoids $A$ and $B$ are 
viewed as one-object $\ca{V}$-categories,
in the same way as rings were viewed as one-object
additive categories.

Even more generally, we 
can consider left $\ca{A}$-/right $\ca{B}$-bimodules
for general $\ca{V}$-categories 
$\ca{A}$, $\ca{B}$. Hence, define
a \emph{$\ca{V}$-bimodule} $M$ to be a $\ca{V}$-functor
\begin{equation}\label{Vbimodulefunctor}
 \ca{B}^\op\otimes\ca{A}\to\ca{V}
\end{equation}
for $\ca{V}$ a symmetric monoidal closed category,
where the opposite category $\ca{B}^\op$ is
$\ca{V}$-enriched by symmetry of $\ca{V}$, the 
tensor product in $\ca{V}$-$\B{Cat}$ was defined 
in the previous section and $\ca{V}$ is enriched
in itself via the internal hom.
Equivalently, if $\ca{V}$ is moreover
complete, a $(\ca{A},\ca{B})$-bimodule
is a $\ca{V}$-functor $\ca{A}\to[\ca{B}^\op,\ca{V}]$
using the monoidal closed structure of $\ca{V}\text{-}\B{Cat}$. 
Bimodules (enriched in $\B{Set}$) are also called
\emph{profunctors} or \emph{distributors}.
Maps of enriched bimodules are evidently 
defined as $\ca{V}$-natural transformations
and are called \emph{$\ca{V}$-bimodule maps}.

There is an alternative, more intuitive formulation
of the definition of a $\ca{V}$-module (see \cite{Lawvereclosedcats}) 
which is closer
to the initial notion of ordinary bimodules. 
More specifically, an $(\ca{A},\ca{B})$-bimodule
$M$ is given by a family of objects $M(B,A)\in\ca{V}$
for all $(B,A)\in\ob\ca{A}\times\ob\ca{B}$,
together with arrows
\begin{gather*}
\ca{A}(A,A')\otimes M(B,A)\to M(A',B) \\
M(B,A)\otimes\ca{B}(B',B) \to M(A,B')
\end{gather*}
in $\ca{V}$, which satisfy usual axioms
and are compatible
with each other.
A detailed presentation of the diagrams involved 
can be found for example in \cite{Carmody} and 
\cite{GarnerShulman},
and the equivalence between these two definitions 
of $\ca{V}$-bimodules is easily verified.

Regarding the maps between them, a $\ca{V}$-bimodule
map $\alpha:M\to M'$ between two left 
$\ca{A}$-/right $\ca{B}$ bimodules
consists of a family of arrows
\begin{displaymath}
 \alpha_{A,B}:M(B,A)\to M'(B,A)
\end{displaymath}
in $\ca{V}$ for all $A\in\ca{A}$, $B\in\ca{B}$,
which respect the $\ca{A}$
and $\ca{B}$-actions. These can be 
composed in an evident way,
thus we have a category of $(\ca{A},\ca{B})$-bimodules
for any $\ca{V}$-categories $\ca{A}$ and $\ca{B}$,
denoted by $\ca{V}$-$\B{BMod}(\ca{A},\ca{B})$ or
$\ca{V}$-${_\ca{A}\Mod_\ca{B}}$.
Notice that for the second characterization of 
$\ca{V}$-bimodules, we do not need any extra 
assumptions on the monoidal category $\ca{V}$.

Back to the example of 
ordinary bimodules, an important 
feature is the fact that there is a
`composition' operation, by taking the tensor product
of bimodules over a ring.
More precisely, given an $(R,S)$-bimodule $M$ and 
a $(S,T)$-bimodule $N$, their tensor product
$M\otimes_SN$ obtains a structure of a 
$(R,T)$-bimodule. 
Having in mind that bimodules constitute
the 1-cells in the bicategory $\B{BMod}$, 
if we denote them as$\SelectTips{eu}{10}\xymatrix@C=.2in
{M:R\ar[r]|-{\object@{|}} & S}$and
$\SelectTips{eu}{10}\xymatrix@C=.2in
{N:S\ar[r]|-{\object@{|}} & T}$so that they are 
also distinguished from ring homomorphisms, 
this 
process can be written
\begin{displaymath}
NM=M\otimes_SN:\SelectTips{eu}{10}\xymatrix
{R\ar[r]^-M|-{\object@{|}} & S
\ar[r]^-N|-{\object@{|}} & T},
\end{displaymath}
and by the canonical isomorphisms
\begin{displaymath}
(M\otimes_SN)\otimes_TL\cong 
M\otimes_S(N\otimes_TL),\quad M\otimes_SS\cong M
\end{displaymath}
for a $(T,V)$-bimodule $L$,
it is clear that
this tensor product satisfies associativity
and identity laws up to isomorphism.

In order to generalize the composition 
of ordinary bimodules to the enriched case, 
we will use the expression of the tensor 
product of modules over
a ring as a coequalizer.
For bimodules in a general monoidal
category $\ca{V}$, this is expressed precisely
by the construction (\ref{tensor})
of the tensor product of a right $B$-bimodule
and a left $B$-bimodule over a monoid $B$.

For this operation to be accordingly 
defined in the level of enriched bimodules,
and for the collection of $\ca{V}$-bimodules and
bimodule morphisms between two $\ca{V}$-categories to
obtain the structure of a $\ca{V}$-enriched category
itself,
the base category $\ca{V}$ is requested to be 
a complete and cocomplete symmetric
monoidal closed category.
Based on the above idea,
for two $\ca{V}$-bimodules$\SelectTips{eu}{10}\xymatrix@C=.2in
{M:\ca{A}\ar[r]|-{\object@{|}} & \ca{B}}$and$\SelectTips{eu}{10}\xymatrix@C=.2in
{N:\ca{B}\ar[r]|-{\object@{|}} & \ca{C}}$we 
define their composite$\SelectTips{eu}{10}\xymatrix@C=.2in
{N\circ M:\ca{A}\ar[r]|-{\object@{|}} & \ca{C}}$by specifying
its components by the following coequalizer:
\begin{displaymath}
\xymatrix @C=.4in
{\scriptstyle{\sum_{B,B'\in\ca{B}}{M(B',A)\otimes\ca{B}(B,B')
\otimes N(C,B)}}\ar@<+.5ex>[r] \ar@<-.5ex>[r] &
\scriptstyle{\sum_{B\in\ca{B}}{M(B,A)\otimes N(C,B)}}
\ar@{->>}[r] & 
\scriptstyle{(N\circ M)(C,A)}.}
\end{displaymath}
The parallel arrows come from the $\ca{B}$-actions
on $M$ and $N$.
This definition in fact exhibits the composite as the coend
\begin{displaymath}
N\circ M=\int_{B\in\ca{B}}M(B,-)\otimes N(-,B),
\end{displaymath}
which inherits a
left $\ca{A}$-action and a right $\ca{C}$-action,
and we also write $(N\circ M)(C,A)=M(B,A)\otimes_\ca{B}N(C,B)$.
This operation can be verified to be associative and 
unitary up to isomorphism by the
associativity, left and right unit
constraints of the monoidal category $\ca{V}$.
So the above data
indeed define a bicategory $\ca{V}$-$\B{BMod}$
(or $\ca{V}$-$\B{Dist}$ or $\ca{V}$-$\B{Prof}$)
with objects $\ca{V}$-categories, 1-cells $\ca{V}$-bimodules
and 2-cells $\ca{V}$-bimodule maps.

Bimodules can also be thought of as generalized
$\ca{V}$-functors between $\ca{V}$-categories, 
considered as `$\ca{V}$-valued relations' between them
(as in Lawvere's \cite{Lawvereclosedcats}). 
In particular, every $\ca{V}$-functor $F:\ca{A}\to\ca{B}$
gives rise to bimodules$\SelectTips{eu}{10}\xymatrix@C=.2in
{F_*:\ca{A}\ar[r]|-{\object@{|}} & \ca{B}}$and
$\SelectTips{eu}{10}
\xymatrix@C=.2in
{F^*:\ca{B}\ar[r]|-{\object@{|}} & \ca{A}}$
defined by
\begin{equation}\label{Fstars}
 F_*(B,A)=\ca{B}(B,FA),\quad F^*(A,B)=\ca{B}(FA,B).
\end{equation}
This structure implies that $\ca{V}$-$\B{BMod}$ fits in the 
context of the final Section \ref{doublecatssetting}.

For the purposes of this thesis, we are more interested
in the categories of one-sided modules, \emph{i.e.}
left $\ca{A}$-modules or right $\ca{B}$-modules
for $\ca{V}$-categories $\ca{A}$ or $\ca{B}$. We follow
the second formulation of the definition 
of $\ca{V}$-modules, which
does not require extra conditions on $\ca{V}$.
\begin{defi}\label{leftAmodule}
 A \emph{left $\ca{A}$-module} $\Psi$, also
written as$\SelectTips{eu}{10}\xymatrix@C=.2in
{\Psi:\ca{A}\ar[r]|-{\object@{|}} & \ca{I}}$,
is given by objects $\Psi A$
in $\ca{V}$ for each $A\in\ca{A}$ and
arrows
\begin{displaymath}
 \mu:\ca{A}(A,A')\otimes\Psi A\to\Psi A'
\end{displaymath}
in $\ca{V}$ for each $A,A'\in\ca{A}$,
subject to the commutativity of
\begin{displaymath}
\xymatrix @C=.35in @R=.4in
{\ca{A}(A',A'')\otimes\ca{A}(A,A')\otimes\Psi A
\ar[r]^-{1\otimes\mu} \ar[d]_-{M\otimes1} &
\ca{A}(A',A'')\otimes \Psi A'\ar[d]^-{\mu} \\
\ca{A}(A,A'')\otimes\Psi A \ar[r]_-{\mu} &
\Psi A'',}
\xymatrix @C=.1in @R=.4in
{& \ca{A}(A,A)\otimes\Psi A\ar[dr]^-{\mu} & \\
\Psi A\ar[ur]^-{j_A} \ar[rr]_-{1_{\Psi A}} &&
\Psi A.}
\end{displaymath}
The arrows $M$ and $j_A$ are the composition and identity element 
in $\ca{V}$, and the associativity and identity
constraints are suppressed.
If$\SelectTips{eu}{10}\xymatrix@C=.2in
{\Xi:\ca{A}\ar[r]|-{\object@{|}} & \ca{I}}$
is another left $\ca{A}$-module, then a 
\emph{left module morphism} $\alpha:\Psi\to\Xi$
is given by an $\ob\ca{A}$-indexed family
\begin{displaymath}
 \alpha_A:\Psi A\to\Xi A
\end{displaymath}
of arrows in $\ca{V}$, satisfying the commutativity
of
\begin{displaymath}
\xymatrix @C=.6in @R=.4in
 {\ca{A}(A,A')\otimes\Psi A \ar[r]^-{1\otimes\alpha} 
\ar[d]_-\mu & \ca{A}(A,A')\otimes\Xi A
\ar[d]^-{\mu} \\
\Psi A'\ar[r]_-\alpha & \Xi A'}
\end{displaymath}
for all $A,A'\in\ca{A}$.
\end{defi}
The category of left $\ca{A}$-modules
is denoted by $\ca{V}$-$\Mod(A)$ or
$\ca{V}$-$_\ca{A}\Mod$. Dually, we can 
define the category of \emph{right
$\ca{B}$-modules} $\ca{V}$-$\Mod_\ca{B}$,
with objects $\SelectTips{eu}{10}\xymatrix@C=.2in
{\ca{I}\ar[r]|-{\object@{|}} & \ca{B}}$.

We could define a left
$\ca{A}$-module $\Psi$ to be a $\ca{V}$-functor
$\Psi:\ca{A}\to\ca{V}$ and a right
$\ca{B}$-module $\Xi$ to be a $\ca{V}$-functor
$\Xi:\ca{B}^\op\to\ca{V}$. This agrees with 
(\ref{Vbimodulefunctor}),
since of course $\ca{A}\otimes\ca{I}\cong\ca{A}$
and $\ca{B}^\op\otimes\ca{I}\cong\ca{B}^\op$
for the unit $\ca{V}$-category $\ca{I}$,
but that would require extra structure on $\ca{V}$
as clarified earlier. We would then be able to 
identify the categories
$\ca{V}$-$_\ca{A}\Mod$ and $\ca{V}$-$\Mod_\ca{B}$
with the presheaf categories
$[\ca{A},\ca{V}]$ and $[\ca{B}^\op,\ca{V}]$
of $\ca{V}$-functors and $\ca{V}$-natural transformations.
Via this presentation, many useful properties are 
inherited from $\ca{V}$, such as completeness,
cocompleteness (obtained pointwise)
and local presentability: for any locally 
$\lambda$-presentable category $\ca{C}$
and small category $\ca{A}$, the functor
category $\ca{C}^\ca{A}=[\ca{A},\ca{C}]$ 
is locally $\lambda$-presentable
itself by \cite[1.54]{LocallyPresentable}.

Notice that the above concepts are
evidently associated with the general notion of a module
(or bimodule) for a monad in a bicategory,
as described in Section \ref{monadsinbicats}. This relation
will be illustrated at the last sections
of Chapter \ref{VCatsVCocats}, in the formal 
context of the bicategory
of $\ca{V}$-matrices $\ca{V}$-$\Mat$.

\section{Actions of monoidal categories and enrichment}\label{actions}

We now recall some parts of the general theory of actions of
monoidal categories, leading to
specific enrichments. We largely follow \cite{AnoteonActions}
by Janelidze and Kelly, adding some details.

An \emph{action} of a monoidal category
$\ca{V}={(\ca{V},\otimes,I,a,l,r)}$ on an ordinary
category $\ca{D}$ is given
by a functor 
\begin{displaymath}
*:\xymatrix @R=.02in
{\ca{V}\times\ca{D}\ar[r] & \ca{D} \\
(X,D)\ar@{|->}[r] & X*D}
\end{displaymath}
with a natural isomorphism with components
${\chi}_{XYD}:(X\otimes Y)*D\stackrel{\sim}{\longrightarrow} X*(Y*D)$
and a natural isomorphism with components
${\nu}_D :I*D\stackrel{\sim}{\longrightarrow} D$, 
satisfying the commutativity of the diagrams
\begin{equation}\label{actiondiag} 
\xymatrix{((X\otimes Y)\otimes
Z)*D\ar[r]^-{\chi}\ar[d]_-{a*1} & (X\otimes Y)*(Z*D)\ar[r]^-{\chi}
& X*(Y*(Z*D))\\ (X\otimes(Y\otimes Z))*D\ar[rr]_-{\chi} &&
X*((Y\otimes Z)*D),\ar[u]_-{1*\chi}}
\end{equation}
\begin{displaymath}
\xymatrix@R=.3in{(I\otimes X)*D\ar[rr]^-{\chi}\ar[dr]_-{l*1} 
&& I*(X*D)\ar[dl]^-{\nu}\\ &
X*D, &}
\end{displaymath}
\begin{displaymath} 
\xymatrix@R=.3in{(X\otimes I)*D\ar[rr]^-{\chi}\ar[dr]_-{r*1} 
&& X*(I*D)\ar[dl]^-{1*{\nu}}\\
& X*D. &}
\end{displaymath}
Notice that if $*$ is an action, then the opposite
functor $*^\op:\ca{V}^\mathrm{op}\times\ca{D}^\mathrm{op}
\longrightarrow\ca{D}^\mathrm{op}$
is still an action: the corresponding
natural isomorphisms
are $\chi^{-1}$ and $\nu^{-1}$
and the action axioms
follow from those for $*$.

For example, any monoidal category $\ca{V}$ has a 
canonical action on itself, by taking 
\begin{displaymath}
*=\otimes:\ca{V}\times\ca{V}\to\ca{V}
\end{displaymath}
and $\chi=a$, $\nu=l$ the monoidal constraints. 
This is sometimes called the 
\emph{regular representation} of $\ca{V}$.
\begin{rmk}\label{rmkpseudoaction}

$(i)$ In B{\'e}nabou's \cite{Benabou},
a very inspiring characterization of actions is provided,
connecting the notion with a bicategorical construction.
More specifically, it is asserted that a (left) action 
of a monoidal category $\ca{V}$ (\emph{multiplicative} category 
in the terminology therein) on any category $\ca{A}$
can be identified with a bicategory $\ca{K}$
with only two objects $\{0,1\}$ and hom-categories
\begin{displaymath}
 \ca{K}(0,0)=\B{1},\;\ca{K}(1,0)=\emptyset,\;\ca{K}(0,1)=\ca{A},\;\ca{K}(1,1)=\ca{V}.
\end{displaymath}
The horizontal composition for the possible combinations
of the objects $0,1$ gives the tensor product 
$\otimes=\circ_{1,1,1}$ of $\ca{V}$
and the action $*=\circ_{0,1,1}$ on $\ca{A}$,
the associativity and identity constraints give the 
monoidal constraints and the action structure transformations,
whereas the coherence conditions correspond to the appropriate axioms.

In particular, the canonical action of any monoidal 
category $\ca{V}$ on itself gives rise to a bicategory 
$\ca{M}\ca{V}$ with two objects as above, and hom-categories
$\ca{M}\ca{V}(0,0)=\B{1}$, $\ca{M}\ca{V}(1,0)=\emptyset$, 
$\ca{M}\ca{V}(0,1)=\ca{M}\ca{V}(1,1)=\ca{V}$.
This bicategory will be used for a certain description
of global
categories of modules and comodules in Chapter \ref{enrichmentofmonsandmods}.

$(ii)$ A \emph{pseudomonoid} in a monoidal category
is an object with multiplication and unit
as defined in Section 
\ref{Categoriesofmonoidsandcomonoids},
for which the diagrams (\ref{monoidaxioms})
commute up to coherent isomorphism.
For example, a pseudomonoid in the 
cartesian monoidal
category $(\B{Cat},\times,\B{1})$
is precisely a monoidal category $\ca{V}$,
with multiplication being the tensor product
and unit picking the unit object.

Furthermore, a \emph{pseudomodule} for a 
pseudomonoid in a monoidal
category is again defined as in Section 
\ref{Categoriesofmodulesandcomodules},
where the diagrams (\ref{defmod}) commute
up to isomorphism, satisfying coherence axioms.
From this point of view, the action
of a monoidal category described above is
a \emph{pseudoaction} of a 
pseudomonoid on an object of the monoidal
$\B{Cat}$, \emph{i.e.} $\ca{D}$ is a $\ca{V}$-pseudomodule.
More on this viewpoint will be discussed in the final chapter.
\end{rmk}
Another example of an action which will be used repeatedly
is the following.
\begin{lem}\label{inthomaction}
Suppose $\ca{V}$ is a symmetric monoidal closed category.
The internal hom 
\begin{displaymath}
 [-,-]:\ca{V}^\mathrm{op}\times\ca{V}\longrightarrow\ca{V}
\end{displaymath}
constitutes an action of the monoidal category
$\ca{V}^\mathrm{op}$ on the category $\ca{V}$,
via the standard natural isomorphisms
\begin{align*}
\chi_{XYZ}:[X\otimes Y,D]
&\xrightarrow{\sim}[X,[Y,Z]] \\
\nu_D:[I,D]&\xrightarrow{\sim}D.
\end{align*}
Moreover, the induced functor
\begin{displaymath}
 \Mon[-,-]:\Comon(\ca{V})^\mathrm{op}\times
\Mon(\ca{V})\longrightarrow\Mon(\ca{V})
\end{displaymath}
is an action of the monoidal
category $\Comon(\ca{V})^\mathrm{op}$
on the category $\Mon(\ca{V})$.
\end{lem}
\begin{proof}
The isomorphisms $\chi_{XYZ},\nu_D$ can be verified 
to satisfy the axioms 
(\ref{actiondiag}) using the transpose diagrams 
under the adjunction $(-\otimes Y)\dashv[Y,-]$.
Moreover, the functor $\Mon[-,-]$ 
induced between the categories
of comonoids and monoids 
as in (\ref{defMon[]})
is just a restriction of $[-,-]$. Hence the
natural isomorphisms in $\Mon(\ca{V})$, 
reflected by the conservative forgetful
functor $S:\Mon(\ca{V})\to\ca{V}$, 
render $\Mon[-,-]$ into an action too.
Recall that 
$\Comon(\ca{V})$ and its opposite are monoidal since $\ca{V}$
is symmetric. 
\end{proof}
For our purposes, it is a very important fact 
that given a category $\ca{D}$ along with an action
of a monoidal category $\ca{V}$ with a 
parametrized adjoint, we obtain a $\ca{V}$-enriched category. 
In fact, this follows 
from a much stronger result of \cite{enrthrvar}
regarding categories enriched in bicategories, 
as mentioned in the introduction.
\begin{thm}\label{actionenrich}
Suppose that $\ca{V}$ is a monoidal category which acts on 
a category $\ca{D}$ via a functor 
$*:\ca{V}\times\ca{D}\to\ca{D}$ such that $-*D$ has
a right adjoint $F(D,-)$ for every $D\in\ca{D}$, with a
natural isomorphism
\begin{equation}\label{adjact}
\ca{D}(X*D,E)\cong\ca{V}(X,F(D,E)).
\end{equation}
Then we can enrich $\ca{D}$ in $\ca{V}$, in the sense that 
there is a $\ca{V}$-category
with the same objects as $\ca{D}$, hom-objects
$F(A,B)$ for $A,B\in\ob\ca{D}$
and underlying category $\ca{D}$.
\end{thm}
\begin{proof}
Suppose we have an adjunction
\begin{equation}\label{adjact2}
\xymatrix @C=.6in
{\ca{V}\ar@<+.8ex>[r]^-{-*D}\ar@{}[r]|-\bot &
\ca{D}\ar@<+.8ex>[l]^-{F(D,-)}} 
\end{equation}
for every object $D$ in $\ca{D}$, where
$*$ is the action of $\ca{V}$ on $\ca{D}$. This implies 
that there is a unique way
of defining a functor of two variables
\begin{displaymath}
 F:\ca{D}^\mathrm{op}\times\ca{D}\longrightarrow\ca{V}
\end{displaymath}
such that the isomorphism (\ref{adjact})
is natural in all three variables, \emph{i.e.}
$F$ is the parametrized adjoint of $(-*-)$.
We are going to show in detail
how these data induce an enrichment
of $\ca{D}$ in $\ca{V}$, with 
enriched hom the functor $F$.

The composition law is the arrow
$M:F(B,C)\otimes F(A,B)\to F(A,C)$
which corresponds uniquely under the adjunction
$-*D\dashv F(D,-)$ to the composite
\begin{equation}\label{compositionunderadjunction} 
\xymatrix @C=.6in
{(F(B,C)\otimes F(A,B))*A\ar[d]_-{\chi_{F(B,C),F(A,B),A}} 
\ar @{-->}[r]
& C \\
F(B,C)*(F(A,B)*A)\ar[r]_-{1*\varepsilon^A_B}
& F(B,C)*B\ar[u]_-{\varepsilon^B_C}}
\end{equation}
where $\varepsilon$ is the counit of the adjunction 
(\ref{adjact2}).
The identity element is the morphism 
$j_A:I\to F(A,A)$
which corresponds uniquely to the isomorphism
\begin{equation}\label{identityunderadjunction}
 I*A\xrightarrow{\nu_A}A.
\end{equation}
The associativity axiom
diagram translates under the adjunction
to the following diagram in $\ca{D}$
\begin{displaymath}
 \xymatrix @C=.01in @R=.27in
{((F(C,D)\otimes F(B,C))\otimes
F(A,B))*A\ar[rr]^-{a*1}\ar[d]_-\chi &&
(F(C,D)\otimes(F(B,C)\otimes F(A,B)))*A\ar[d]^-\chi\\
(F(C,D)\otimes
F(B,C))*(F(A,B)*A)\ar[d]_-{1*\varepsilon^A_B}\ar @{.>}[drr]^-\chi &&
F(C,D)*((F(B,C)\otimes F(A,B))*A)\ar[d]^-{1*\chi}\\
(F(C,D)\otimes F(B,C))*B\ar[d]_-\chi &&
F(C,D)*(F(B,C)*(F(A,B)*A))\ar[d]^-{1*(1*\varepsilon^A_B)}\\
F(C,D)*(F(B,C)*B)\ar @{.>}[rr]^-{=}\ar[d]_-{1*\varepsilon^B_C} &&
F(C,D)*(F(B,C)*B)\ar[d]^-{1*\varepsilon^B_C}\\
F(C,D)*C\ar[dr]_-{\varepsilon^C_D} && F(C,D)*C\ar[dl]^-{\varepsilon^C_D} \\
& D &}
\end{displaymath}
which commutes due to naturality 
of $\chi$ and $\varepsilon$.
The identity axioms correspond to 
the diagrams
\begin{displaymath}
\xymatrix @C=.01in @R=.35in
{ & (F(B,B)\otimes F(A,B))*A\ar[rr]^-\chi &&
F(B,B)*(F(A,B)*A)\ar[d]^-{1*\varepsilon^A_B} \\
(I\otimes F(A,B))*A\ar@/^/[ur]^-{(j_B\otimes1)*1}
\ar@/_/[dr]_-{l*1} \ar[r]^-\chi &
I*(F(A,B)*A) \ar[d]^-\nu \ar[urr]_-{j_B*1}
\ar[r]_-{1*\varepsilon^A_B} &
I*B\ar[dr]_-\nu \ar[r]_-{j_B*1} &
F(B,B)*B\ar[d]^-{\varepsilon^B_B} \\
& F(A,B)*A \ar[rr]_-{\varepsilon^A_B}
&& B,}
\end{displaymath}
\begin{displaymath}
\xymatrix @C=.15in @R=.35in
{& (F(A,B)\otimes F(A,A))*A\ar[r]_-\chi & 
F(A,B)*(F(A,A)*A)\ar[d]^-{1*\varepsilon^A_A} \\
(F(A,B)\otimes I)*A\ar@/^/[ur]^-{(1\otimes j_A)*1}
\ar[r]^-\chi \ar@/_/[dr]_-{r*1} &
F(A,B)*(I*A)\ar[ur]_-{1*(j_A*1)}
\ar[r]_-{1*\nu} \ar[d]_-{1*\nu} &
F(A,B)*A\ar[d]^-{\varepsilon^A_B}\\ 
& F(A,B)*A\ar[r]_-{\varepsilon^A_B} & B}
\end{displaymath}
in $\ca{D}$, which commute again
for evident reasons. 

Therefore we obtain a $\ca{V}$-enriched category $\ca{A}$,
with objects $\ob\ca{A}=\ob\ca{D}$
and hom-objects $\ca{A}(A,B)=F(A,B)$. 
The underlying category of $\ca{A}$ is
precisely $\ca{D}$:
\begin{displaymath}
\ca{A}_o(A,B)=\ca{V}(I,\ca{A}(A,B))=\ca{V}(I,F(A,B))\cong
\ca{D}(I*A,B)\cong\ca{D}(A,B)
\end{displaymath}
\begin{displaymath} \Rightarrow \ca{A}_o\cong \ca{D}
\end{displaymath}
using the isomorphisms (\ref{adjact}) and $\nu_A$.
\end{proof}
As a straightforward application, 
we recover the well-known fact that 
the internal hom in a 
monoidal closed category $\ca{V}$ induces
an enrichment of $\ca{V}$ in itself with hom-objects 
$[A,B]$, as 
mentioned in Section \ref{basicdefienrichment}.
This is the case, since the canonical action $*=\otimes$
has as parametrized adjoint the functor $[-,-]$.

As shown in detail in \cite{AnoteonActions},
when $\ca{V}$ is a monoidal closed category
the natural isomorphism (\ref{adjact})
lifts to a $\ca{V}$-natural isomorphism
\begin{displaymath}
\ca{A}(X*D,E)\cong[X,\ca{A}(D,E)]. 
\end{displaymath}
The existence of a $\ca{V}$-enriched representation
of $[X,\ca{A}(D,-)]$ is expressed by
saying that the $\ca{V}$-category $\ca{A}$
is \emph{tensored}, with $X*D$ as 
the tensor product of $X$ and $D$.
Furthermore, 
the case when not only the functor
$(-*D)$ but also $(X*-)$ has 
a right adjoint $G(X,-)$ is addressed 
when $\ca{V}$ is moreover symmetric.
Together with the natural isomorphism (\ref{adjact})
we get
\begin{displaymath}
 \ca{D}(D,G(X,E))\cong\ca{D}(X*D,E)\cong\ca{V}(X,F(D,E)).
\end{displaymath}
The bottomline is that the above assumptions result in the existence
of a tensored and \emph{cotensored} $\ca{V}$-enriched category,
with underlying category $\ca{D}$, tensor product $X*D$ 
and cotensor product $G(X,E)$.

As a special case of the above theorem,
suppose that there is an action of
a monoidal category $\ca{V}$ on the opposite
of a category $\ca{D}=\ca{B}^\mathrm{op}$
via a bifunctor
\begin{displaymath}
*:\ca{V}\times\ca{B}^\mathrm{op}\to\ca{B}^\mathrm{op}. 
\end{displaymath}
If we have an adjunction as in (\ref{adjact2}),
the parametrized adjoint 
$F:\ca{B}\times\ca{B}^\mathrm{op}\to\ca{V}$
of $*$ can be denoted as
\begin{displaymath}
 P:\ca{B}^\mathrm{op}\times\ca{B}\longrightarrow\ca{V}
\end{displaymath}
by switching the entries of 
a pair in the cartesian product.
The natural isomorphism (\ref{adjact}) then becomes
\begin{displaymath}
\ca{B}(A,X*B)\cong\ca{V}(X,P(A,B))
\end{displaymath}
for $X\in\ca{V}$ and $A,B\in\ca{B}$.
\begin{cor}\label{importcor1}
If $\;*:\ca{V}\times\ca{B}^\mathrm{op}
\to\ca{B}^\mathrm{op}$ is an action 
of the monoidal closed $\ca{V}$ 
on the category $\ca{B}^\mathrm{op}$
along with an adjunction $(-*B) 
\dashv P(-,B)$ 
for each $B\in\ca{B}$, then $\ca{B}^\mathrm{op}$  
is tensored $\ca{V}$-enriched with hom-objects
$\ca{B}^\mathrm{op}(A,B)=P(B,A)$. 
\end{cor}
Moreover, if $\ca{V}$ is symmetric then 
the opposite of a $\ca{V}$-enriched category is still
enriched in $\ca{V}$. Hence
in the situation above, 
there is an induced enrichment of 
$\ca{B}=(\ca{B}^\mathrm{op})^\mathrm{op}$
in $\ca{V}$
with the same objects and 
hom-objects 
$\ca{B}(A,B)=\ca{B}^\mathrm{op}(B,A)$.
\begin{cor}\label{importcor2}
Let $\ca{V}$ be a symmetric monoidal closed category acting
on $\ca{B}^\mathrm{op}$.
If for each object $B$, the action functor
$-*B:\ca{V}\to\ca{B}^{\mathrm{op}}$ has a right
adjoint $P(-,B):\ca{B}^\mathrm{op}\to\ca{V}$, then
the parametrized adjoint
\begin{displaymath}
P(-,-):\ca{B}^\mathrm{op}\times\ca{B}\longrightarrow\ca{V}
\end{displaymath}
of the action provides the hom-objects of a 
cotensored $\ca{V}$-enriched category 
with underlying ordinary category $\ca{B}$
and $X*B$ the cotensor product of $X$ and $B$.
If furthermore $X*-:\ca{B}^\op\to\ca{B}^\op$
has a right adjoint, the $\ca{V}$-category is also tensored. 
\end{cor}
We are interested in these variations of Theorem
\ref{actionenrich}, because our examples in the following 
chapters fall into these precise formulations.

\chapter{Fibrations and Opfibrations}\label{fibrations}
This chapter begins with a detailed 
review of the basic concepts of 
the theory of fibrations and opfibrations, 
which plays a central role in the
development of this thesis. Our presentation follows mainly 
\cite{Handbook2,Jacobs,Elephant2}.
The notion of a fibred category, which arose from 
the work of Grothendieck in algebraic geometry, successfully
captures the concept of a category varying over 
(or indexed by) a different category.
There has also been a connection of fibrations with
foundations for category theory, investigated in 
\cite{BenabouFibered}.

Inside the total category of a fibration,
the cartesian morphisms which are universally characterized
incorporate a coherent structure: that of an indexed 
category, \emph{i.e.} a certain pseudofunctor. This is 
best understood via the Grothendieck construction
(see Theorem \ref{maintheoremfibr}) originally in 
\cite{Grothendieckcategoriesfibrees}, 
which demonstrates the essential equivalence between 
these two concepts.  
In fact, the coherent structure maps 
of an indexed category, whose existence is only implicit
in fibrations, show that an indexed category has a structure
whereas a fibration has a property (which determines
such structure when a cleavage is chosen).
Despite their correspondence, fibrations 
are technically often more convenient than indexed categories.

In the last section, we turn our attention to 
the fibrewise limits and adjunctions between 
fibred categories. Following the 
terminology and results of \cite{FibredAdjunctions,Jacobs},
we slightly extend the existing theory by examining 
under which assumptions a fibred 1-cell between fibrations
over different bases has a (fibred) adjoint. Hermida in 
his thesis \cite{hermidaphd} had already established
the factorization of general fibred adjunctions in 
terms of cartesian fibred adjunctions and fibred adjunctions,
suggesting an important characterization of fibred completeness.
However, we follow a different approach to related problems.\vspace{-.1in}

\section{Basic definitions}\label{fibrationsbasicdefinitions}

Consider a functor $P:\ca{A}\to\caa{X}$. 
A morphism $\phi:A\to B$ in $\ca{A}$ over 
a morphism $f=P(\phi):X\to Y$ in $\caa{X}$
is called ($P$-)\emph{cartesian} if and only if, for all 
$g:X'\to X$ in $\caa{X}$ and $\theta:A'\to B$ in 
$\ca{A}$ with 
$P\theta=f\circ g$, there exists a unique arrow $\psi:A'\to A$ 
such that $P\psi=g$ and $\theta=\phi\circ\psi$:
\begin{displaymath}
\xymatrix @R=.1in @C=.6in
{A'\ar [drr]^-{\theta}\ar @{-->}[dr]_-{\exists!\psi} 
\ar @{.>}@/_/[dd] &&& \\
& A\ar[r]_-{\phi} \ar @{.>}@/_/[dd] & 
B \ar @{.>}@/_/[dd] & \textrm{in }\ca{A}\\
X'\ar [drr]^-{f\circ g=P\theta}\ar[dr]_-g &&&\\
& X\ar[r]_-{f=P\phi} & Y & \textrm{in }\caa{X}}
\end{displaymath}
For $X\in\ob\caa{X}$, the \emph{fibre
of $P$ over $X$} written $\ca{A}_X$, 
is the subcategory of $\ca{A}$ 
which consists of objects $A$
such that $P(A)=X$ and morphisms $\phi$ with 
$P(\phi)=1_X$, called
($P-$)\emph{vertical} morphisms. 

The functor $P:\ca{A}\to\caa{X}$ 
is called a \emph{fibration} if and only if, for all $f:X\to Y$ in
$\caa{X}$ and $B\in\ca{A}_Y$, there is a cartesian morphism
$\phi$ with codomain $B$ above $f$, \emph{i.e.} $\phi:A\to B$ with $P(\phi)=f$.
We call such an $\phi$ a \emph{cartesian lifting} of $B$ along
$f$. The category $\caa{X}$ is then called the \emph{base} of the fibration,
and $\ca{A}$ its \emph{total} category.

Dually to the above, suppose we have 
a functor $U:\ca{C}\to\caa{X}$.
A morphism $\beta:C\to D$ is \emph{cocartesian}
over a morphism $U\beta=f:X\to Y$ in $\caa{X}$ if, for
all $g:Y\to Y'$ in $\caa{X}$ and all
$\gamma:C\to D'$ with $U\gamma=g\circ f$, 
there exists a unique morphism $\delta:D\to D'$ such that
$U\delta=g$ and $\gamma=\delta\circ\beta$:
\begin{displaymath}
\xymatrix @R=.1in @C=.6in
{&& D'\ar @{.>}@/_/[dd] &&\\
C\ar[r]_-{\beta} \ar @{.>}@/_/[dd]
\ar[urr]^-{\gamma} & 
D \ar @{.>}@/_/[dd] \ar @{-->}[ur]_-{\exists! \delta}
&& \textrm{in }\ca{C}\\
&& Y' &&\\
X\ar[r]_-{f=U\beta} \ar[urr]^-{g\circ f=U\gamma}
 & Y \ar[ur]_-g && \textrm{in }\caa{X}}
\end{displaymath}
The functor $U:\ca{C}\to\caa{X}$ is an \emph{opfibration} 
if $U^\mathrm{op}$ is a fibration, \emph{i.e.} for every $C\in\ca{C}_X$
and $f:X\to Y$ in $\caa{X}$, there is a cocartesian morphism
with domain $C$ above $f$, called the \emph{cocartesian lifting}
of $C$ along $f$. 
If $U$ is both a fibration
and an opfibration, it is called a \emph{bifibration}.
\begin{rmk*}
 What was above called `cartesian' is sometimes
called `hypercartesian' instead. In that case, a 
cartesian morphism $\phi:A\to B$ would satisfy the property
that for any $\theta:A'\to B$ with 
$P(\phi)=P(\theta)$, there is a unique vertical arrow
$\psi:A'\to A$ with $\phi\circ\psi=\theta$:
\begin{displaymath}
 \xymatrix @C=.4in
{A'\ar[drr]^-\theta\ar@{-->}[d]_-{\exists!\psi} && \\
A\ar[rr]_-\phi && B & \textrm{in }\ca{A}.}
\end{displaymath}
If we were to use this definition, we would have 
to add the requirement that cartesian arrows are 
closed under composition, in order to define a fibration. 
\end{rmk*}
\begin{examples*}
$(1)$ Every category $\ca{C}$ gives
rise to the \emph{family fibration}
\begin{displaymath}
 f(\ca{C}):Fam(\ca{C})\longrightarrow\B{Set}
\end{displaymath}
over the category of sets. The category
$Fam(\ca{C})$ has objects indexed families
of objects in $\ca{C}$, $\{ X_i \}_{i\in I}$
for a set $I$, and morphisms 
\begin{displaymath}
 (\{f_i\}_{i\in I},u):\{ X_i \}_{i\in I}\longrightarrow
\{ Y_j \}_{i\in J}
\end{displaymath}
are pairs which consist of a function 
$u:I\to J$ and a family of morphisms
$f_i:X_i\to Y_{u(i)}$ in $\ca{C}$ for all $i$'s. 
The functor $f(\ca{C})$ just takes a 
family of objects to its indexing set and 
a morphism to its function part.
The cartesian arrows are pairs for which
every $f_i$ is an isomorphism,
so a cartesian lifting of $\{Y_j\}_{j\in J}$
above $u:I\to J$ is
\begin{displaymath}
\xymatrix @C=.4in @R=.4in
{\{Y_{u(i)}\}_{i\in I} \ar[rr]^-{(1,u)}
\ar @{.>}[d]_-{f(\ca{C})} && 
\{Y_j\}_{j\in J}\ar @{.>}[d]^-{f(\ca{C})} & 
\textrm{in }Fam(\ca{C}) \\
I\ar[rr]_-u && J & \textrm{in }\B{Set}.}
\end{displaymath}
$(2)$ Consider the `codomain' functor
for any category $\ca{A}$
\begin{equation}\label{cod}
 cod:\ca{A}^\B{2}\longrightarrow\ca{A}
\end{equation}
where 
$\ca{A}^\B{2}=[\B{2},\ca{A}]$
is the \emph{category of arrows} of
$\ca{A}$, \emph{i.e.} the functor category from
$\B{2}={(0\to1)}$ with two objects
and one non-identity arrow to $\ca{A}$.
This functor takes a morphism $f:A\to B$
in $\ca{A}$ to its codomain $B$, and 
a commutative square
\begin{displaymath}
 \xymatrix
 {A\ar[r]^-f\ar[d]_-h & B\ar[d]^-k \\
 C\ar[r]_-g & D}
\end{displaymath}
which expresses an arrow from $f$ to $g$,
to $k:B\to D$. Now, a $cod$-cartesian arrow
of $C\xrightarrow{g}D$ above
$k:B\to D$ is the pullback square
\begin{displaymath}
\xymatrix
{\bullet \pullbackcorner[ul] \ar[r] \ar[d] & B\ar[d]^-k \\
C\ar[r]_-g & D}
\end{displaymath}
therefore if $\ca{A}$ has pullbacks, $cod$
is a fibration. Since this allows 
one to consider $\ca{A}$ as fibred
over itself and this is central
for developing category theory over
$\ca{A}$, we call $cod$
the \emph{fundamental fibration} 
of $\ca{A}$. The fibre over an 
object $A$ is simply the
slice category $\ca{A}/A$.
Dually we have the `domain opfibration'
with pushouts as cocartesian morphisms.
\end{examples*}
As an immediate consequence of the definition
of cartesian and cocartesian morphisms, we have 
that if $g$ and $f$ are composable (co)cartesian
arrows, their composite $g\circ f$ 
is again a (co)cartesian arrow. Also if 
$g\circ f$ and $g$ are (co)cartesian arrows, then 
so is $f$. For example,
for the fundamental fibration this is
the standard result that if $A$ and 
$B$ as in 
\begin{displaymath}
 \xymatrix @C=.4in @R=.4in
{\bullet\ar[r]\ar[d]\drtwocell<\omit>{'A} & 
\bullet\ar[r]\ar[d]\drtwocell<\omit>{'B} &
\bullet\ar[d] \\
\bullet\ar[r] & \bullet\ar[r] & \bullet}
\end{displaymath}
are pullbacks, then the pasted square is a pullback.
Moreover, if the outer square and $B$ are 
pullbacks, then so is $A$.

If $P:\ca{A}\to\caa{X}$ is a fibration, assuming the 
axiom of choice we may select a 
cartesian arrow over each $f:X\to Y$ in $\caa{X}$
and $B\in\ca{A}_Y$, denoted by 
\begin{displaymath}
\Cart(f,B):f^*(B)\longrightarrow B. 
\end{displaymath}
Such a choice of cartesian liftings is called a 
\emph{cleavage} for $P$, which is then
called a \emph{cloven} fibration. Any fibration can 
be turned into a cloven one, using the axiom of 
choice to obtain a cleavage.
Thus, henceforth we can assume that the 
fibrations we deal with are
cloven.
Dually, if $U$ is an opfibration, for any $C\in\ca{C}_X$
and $f:X\to Y$ in $\caa{X}$ we can choose a cocartesian
lifting of $C$ along $f$
\begin{displaymath}
 \Cocart(f,C):C\longrightarrow f_!(C).
\end{displaymath}

As a result of the above definitions,
any arrow $\theta$ in the total category 
above $f$ in a cloven 
fibration $P:\ca{A}\to\caa{X}$
factorizes uniquely into a
vertical morphism
followed by a cartesian one. Dually, 
any arrow $\gamma$ in the total
category above $f$ in the base category
of a cloven opfibration $U:\ca{C}\to\caa{X}$ 
factorizes uniquely into
a cocartesian arrow followed by a vertical one:
\begin{displaymath}
\xymatrix @C=.4in @R=.4in
{A\ar[rr]^\theta\ar @{-->}[d]_-{\psi} && B 
\ar @{.>}[dd]^P &&\\
f^*B\ar[urr]_-{\;\Cart(f,B)}
\ar @{.>}[d]_P &&& \textrm{in }\ca{A} \\
X\ar[rr]_-{f=P\theta} && Y & \textrm{in }\caa{X},}\quad
\xymatrix @C=.4in @R=.4in
{C \ar @{.>}[dd]_U
\ar[rr]^\gamma \ar[drr]_-{\Cocart(f,C)} && D &\\
&& f_!C \ar @{-->}[u]_-{\delta}
\ar @{.>}[d]^U & \textrm{in }\ca{C} \\
X\ar[rr]_-{f=U\gamma} && Y & \textrm{in }\caa{X}.}
\end{displaymath}
\begin{rmk*}
Cartesian liftings
of $B\in\ca{A}_Y$ along $f:X\to Y$ in 
$\caa{X}$ are unique up to 
vertical isomorphism:
\begin{displaymath}
\xymatrix @R=.35in @C=.6in
{A'\ar [drr]^-{\psi=\Cart(f,B)}\ar @/_/@{-->}[d]_-{\alpha} 
&&& \\
A\ar[rr]_-{\phi=\Cart(f,B)}  
\ar @/_/@{-->}[u]_-{\beta}
\ar @{.>}[d]_-P &&
B \ar @{.>}[dd]^-P & \textrm{in }\ca{A}\\
X\ar [drr]^-f \ar[d]_-{1_X} && \\
X\ar[rr]_-{f=P(\phi)} && Y & \textrm{in }\caa{X}}
\end{displaymath}
If $\phi$ and $\psi$ are both cartesian morphisms,
there exist unique $\alpha:A'\to A$ and 
unique $\beta:A\to A'$ vertical arrows 
such that $\phi\circ\alpha
=\psi$ and $\psi\circ\beta=\phi$ respectively.
It follows that $\alpha\circ\beta=1_A$
and $\beta\circ\alpha=1_{A'}$.
Dually, cocartesian arrows are unique
up to vertical isomorphism.
\end{rmk*}
A cleavage for a fibration $P:\ca{A}\to\caa{X}$ induces,
for every morphism $f:X\to Y$ in $\caa{X}$, 
a so-called \emph{reindexing functor} 
between the fibre categories
\begin{equation}\label{reindexing}
 f^*:\ca{A}_Y\longrightarrow\ca{A}_X.
\end{equation}
This maps each $B\in\ca{A}_Y$ to 
$f^*(B)$, the domain of the
cartesian lifting along $f$ given by the cleavage, 
and each $\phi:B\to B'$ in the fibre $\ca{A}_Y$ to
$f^*(\phi):f^*(B)\to f^*(B')$, the unique
vertical arrow making the diagram
\begin{equation}\label{deff*phi}
\xymatrix @C=.7in @R=.4in
{f^*(B)\ar[r]^-{\Cart(f,B)}
\ar @{-->} [d]_-{f^*(\phi)} & B\ar[d]^-{\phi}\\
f^*(B') \ar[r]_-{\Cart(f,B')} & B'}
\end{equation}
commute. Explicitly, since the composite $\phi\circ\Cart(f,B)$ has
codomain $B'$ in the total category $\ca{A}$,
it uniquely factorizes through the chosen cartesian 
lifting of $B'$ along $f$ by universal property
of cartesian arrows.

The uniqueness of the factorization
through a chosen cartesian arrow implies
immediately that $f^*(\psi)\circ
f^*(\phi)=f^*(\psi\circ\phi)$, \emph{i.e.}
that $f^*$ is a functor:
\begin{displaymath}
\xymatrix @C=.4in @R=.4in
{f^*(B)\ar[rr]^-{\Cart(f,B)}
\ar @{-->} [d]^-{f^*(\phi)}\ar @/_7ex/@{-->}[dd]_-{f^*(\psi
\circ\phi)} && B\ar[d]^-{\phi} &\\
f^*(B') \ar @{-->}[d]^-{f^*(\psi)} \ar[rr]_-{\Cart(f,B')} && 
B'\ar[d]^-{\psi} &\\
f^*(B'')\ar @{.>}[d]_-P
\ar[rr]_-{\Cart(f,B'')} && B''\ar @{.>}[d]^-P &\textrm{in }\ca{A}\\
X\ar[rr]_-{f} && Y & \textrm{in }\caa{X}}
\end{displaymath}
Dually, if $U:\ca{C}\to\caa{X}$
is a cloven opfibration, for every $f:X\to Y$
in $\caa{X}$ we get a reindexing functor between the fibres
\begin{displaymath}
 f_!:\ca{C}_X\longrightarrow\ca{C}_Y
\end{displaymath}
mapping each object
$C\in\ca{C}_X$ to the codomain $f_!(C)$ of the chosen
cocartesian lifting along $f$, and vertical morphisms 
$\gamma:C\to C'$ to
the unique $f_!(\gamma)$ defined 
dually to (\ref{deff*phi}).

Notice that the opfibration $P^\op$
for a fibration $P:\ca{A}\to\caa{X}$ 
has cocartesian liftings
\begin{displaymath}
\xymatrix @C=.15in @R=.4in
{A\ar[rrr]^-{\Cocart(f,A)}
\ar @{.>}[d]_-{P^\op} &&& 
f_!A\ar @{.>}[d]^-{P^\op} & 
\textrm{in }\ca{A}^\op \\
X\ar[rrr]_-f &&& X' & \textrm{in }\caa{X}^\op}
\;\xymatrix @R=.1in
{\hole \\
\equiv}\;\;
\xymatrix @C=.15in @R=.4in
{f^*A\ar[rrr]^-{\Cart(f,A)}
\ar @{.>}[d]_-{P} &&& 
A\ar @{.>}[d]^-{P} & 
\textrm{in }\ca{A} \\
X'\ar[rrr]_-f &&&  X & \textrm{in }\caa{X}}
\end{displaymath}
and reindexing functors 
$f_!\equiv(f^*)^\op:\ca{A}^\op_X\longrightarrow\ca{A}^\op_{X'}$.
\begin{rmk}\label{rmkforadjointintexingbifr}
Due to the unique factorization of an 
arrow in a fibration and an opfibration through
cartesian and cocartesian liftings respectively, 
we can deduce that a fibration $P:\ca{A}\to\caa{X}$
is also an opfibration (consequently a bifibration)
if and only if, for every $f:X\to Y$
the reindexing $f^*:\ca{A}_Y\to\ca{A}_X$ has
a left adjoint, namely $f_!:\ca{A}_X\to\ca{A}_Y$ 
(e.g. \cite[Proposition 1.2.7]{hermidaphd}). 
\end{rmk}
In general, for composable
maps $f:X\to Y$ and $g:Y\to Z$ in the base category $\caa{X}$
of a fibration, it
is not true that $f^*\circ g^*=(g\circ f)^*$. 
However, these functors are canonically
isomorphic, as demonstrated by the following diagram:
\begin{displaymath}
\xymatrix @C=.6in @R=.1in
{f^*g^*A \ar[dr]^-{\Cart(f,g^*A)}
\ar @{-->} [dd]_-{\delta_A} &&&\\
& g^*A \ar@{.}[ddd]
\ar[dr]^-{\Cart(g,A)} && \\
(g\circ f)^*A\ar @{.}[d]
\ar[rr]_-{\Cart(g\circ f,A)}
&& A \ar @{.}[ddd] & \textrm{in }\ca{A}\\
X\ar[dr]^-f \ar[dd]_-{1_X} &&&\\
& Y\ar[dr]^-g &&\\
X\ar[rr]_-{g\circ f} && Z & 
\textrm{in }\caa{X}.}
\end{displaymath}
Since the composition of two
cartesian arrows is again cartesian,
$\delta_A$ is the unique vertical
isomorphism which makes the above diagram commute.
Thus we obtain a natural isomorphism
\begin{equation}\label{deltaiso}
\delta^{f,g}:f^*\circ g^*\xrightarrow{\;\sim\;}
(g\circ f)^*
\end{equation}
with components vertical isomorphisms 
$\delta^{f,g}_A:f^*g^*A\cong(g\circ f)^*A$ for any
$A\in\ca{A}$.

Moreover, the identity morphism $1_A:A\to A$ 
for an object $A$ above $X$ is cartesian over 
$1_X:X\to X$, and so there
exists a unique vertical
isomorphism $\gamma^X_A:A\cong (1_X)^*A$
making the top diagram
\begin{displaymath}
\xymatrix @C=.9in @R=.25in
{A\ar[dr]^-{1_A} \ar @{-->}[d]_-{\gamma_A} && \\
(1_X)^*A\ar@{.}[d] \ar[r]_-{\Cart(1_X,A)} & A\ar@{.}[d]
 &\textrm{in }\ca{A}\\
X\ar[r]_-{1_X} & X &\textrm{in }\caa{X}}
\end{displaymath}
commute. These
morphisms are the components of a natural isomorphism
\begin{equation}\label{gammaiso}
\gamma^X:1_{\ca{A}_X}\xrightarrow{\;\sim\;}
(1_X)^*
\end{equation}
where $1_{\ca{A}_X}$ is the identity functor
on the fibre $\ca{A}_X$. The natural
transformations $\delta$ and $\gamma$ will 
play an important part for the equivalence between
fibrations and indexed categories 
described in the next section.

In a completely analogous way, for an opfibration 
$U:\ca{C}\to\caa{X}$
there is a natural isomorphism
\begin{equation}\label{qiso}
q^{f,g}:(g\circ f)_!\xrightarrow{\;\sim\;}
g_!\circ f_!
\end{equation}
between the reindexing functors for composable
arrows $f$ and $g$,
with components vertical isomorphisms
$q^{f,g}_C:(g\circ f)_!C\cong
g_!f_!C$ induced by universality of
cocartesian arrows, and also
a natural isomorpism
\begin{displaymath}
p^X:(1_X)_!\xrightarrow{\;\sim\;}
1_{\ca{C}_X}
\end{displaymath} 
with components vertical isomorphisms
$p^X_C:(1_X)_!C\cong C$.

Nevertheless, a functorial choice of
cartesian liftings is possible
in specific situations. 
We usually assume, without loss
of generality, that the cleavage is 
\emph{normalized}
in the sense that $\Cart(1_X,A)=1_A$, in which case
the isomorphisms $\gamma_A$ are equalities. 
If also $\Cart(g\circ f,A)
=\Cart(f,A)\circ \Cart(g,f^*(A))$, and so 
$\delta_A$ are equalities, 
the cleavage of the fibration is called a 
\emph{splitting}, and a fibration
 endowed with a split cleavage
is called a \emph{split} fibration. Dually,
we have the notion of a \emph{split opfibration}.

We now turn to the appropriate notions 
of 1-cells and 2-cells for fibrations.
A morphism of fibrations
$(S,F):P\to Q$ between $P:\ca{A}\to\caa{X}$
and
$Q:\ca{B}\to\caa{Y}$
is given by a commutative square of functors
and categories
\begin{equation}\label{commutativefibredcell}
\xymatrix @C=.4in @R=.4in
{\ca{A}\ar[r]^-S \ar[d]_-P &
\ca{B}\ar[d]^-Q \\
\caa{X}\ar[r]_-F &
\caa{Y}}
\end{equation}
where $S$ preserves cartesian arrows, meaning
that if $\phi$ is $P$-cartesian, then
$S\phi$ is $Q$-cartesian. The pair $(S,F)$ is called
a \emph{fibred 1-cell}.
In particular, when $P$ and $Q$ are fibrations over
the same base category $\caa{X}$, we may
consider fibred 1-cells of the form 
$(S,1_{\caa{X}})$ displayed by
commutative triangles
\begin{displaymath}
\xymatrix @C=.2in
{\ca{A}\ar[rr]^-S \ar[dr]_-P
 && \ca{B}\ar[dl]^-Q\\
 & \caa{X} &}
\end{displaymath}
which are just cartesian functors $S$
such that $Q\circ S=P$.
Then $S$ is called
a \emph{fibred functor}.

Dually, we have the notion of an \emph{opfibred 1-cell}
$(K,F)$ and \emph{opfibred functor} $(K,1_\caa{X})$ 
between opfibrations over arbitrary bases or the same 
base respectively, where $K$ preserves cocartesian arrows.
\begin{rmk}\label{functorsbetweenfibres}
Any fibred or opfibred 1-cell determines
a collection of functors
$\{ S_X:\ca{A}_X\to\ca{B}_{FX} \}$
between the fibre
categories 
for all $X\in\mathrm{ob}\caa{X}$:
\begin{equation}\label{inducedfunfibres}
S_X:
\xymatrix @R=.05in @C=.6in
{\ca{A}_X\ar[r]^-{S|_X} &
\ca{B}_{FX} \\
A\ar @{|.>}[r] \ar[dd]_-f & SA
\ar[dd]^-{Sf}\\
\hole \\
A' \ar @{|.>}[r] & SA'}
\end{equation}
This functor is well-defined, since $Q(SA)
=F(PA)=FX$ by commmutativity of 
(\ref{commutativefibredcell}),
so $SA,SA'$ are in the fibre $\ca{B}_{FX}$.
Also $Q(Sf)=F(Pf)=F(1_X)=1_{FX}$ since
$F$ is a functor, so $Sf$ is an arrow in $\ca{B}_{FX}$.
\end{rmk}
The following well-known proposition gives
a way, given a fibration and a different functor to its 
base, to construct a new fibration over the 
domain of the given functor. A non-elementary proof
(not as the one below) can be found in 
\cite{Grayfibredandcofibred}.
\begin{prop}[Change of Base]\label{changeofbasefibr}
Given a fibration $Q:\ca{B}\to\caa{Y}$ and an arbitrary 
functor $F:\caa{X}\to\caa{Y}$, the pullback diagram
\begin{displaymath}
\xymatrix @R=.5in @C=.5in
{F^*(\ca{B}) \pullbackcorner[ul] \ar[r]^-{\pi} 
\ar[d]_-{F^*Q} & \ca{B}\ar[d]^-Q\\
\caa{X}\ar[r]_-F & \caa{Y}}
\end{displaymath}
exhibits $F^*Q:F^*(\ca{B})\to\caa{X}$ as a fibration and $(\pi,F)$ as a 
fibred 1-cell, \emph{i.e.} $\pi$ preserves cartesian arrows.
\end{prop}
\begin{proof}
By construction of pullbacks in the complete
category $\B{Cat}$, 
objects in $F^*(\ca{B})$ are pairs
$(B,X)\in\ob\ca{B}\times\ob\caa{X}$ 
such that $QB=FX$, and morphisms are 
$(h,k):(B,X)\to(B',X')$ with $h:B\to B'$ in
$\ca{B}$, $k:X\to X'$ in $\caa{X}$ and $Qh=Fk$ in $\caa{Y}$. 
The functors $\pi$ and $F^*Q$ 
are the respective projections.

It can be easily verified, since
$Q$ is a fibration, that cartesian
morphisms in $F^*(\ca{B})$ exist and are of the form
\begin{displaymath}
\xymatrix @C=.5in @R=.5in
{((Ff)^*B,Z)\ar[rr]^-{(\Cart(Ff,B),f)}
\ar @{.>}[d]_-{F^*Q} && (B,X)\ar @{.>}[d]^-{F^*Q} & 
\textrm{in }\caa{F^*(\ca{B})}\\
Z\ar[rr]^-f && X & \textrm{in }\caa{X}}
\end{displaymath}
where $\Cart(Ff,B)$ is the $Q$-cartesian
lifting of $B$ along $Ff$.
The projection $\pi$ obviously 
preserves cartesian arrows
by the choice of cleavage.
\end{proof} 
The same construction applies
to opfibrations. 
We say that the fibration
$P=F^*Q$ is obtained from
$Q$ \emph{by change of base along
$F$}.
Notice also that for every object 
$X\in\ob\caa{X}$, we have an isomorphism 
${F^*(\ca{B})}_X\cong\ca{B}_{FX}$ of
the fibre categories which is given by 
$S_X$, the induced functor between the fibres
\begin{equation}\label{Sxiso}
\xymatrix @C=0.5in @R=.02in
{F^*(\ca{B})_X\ar[r]^-{S_X}
& \ca{B}_{FX}\\
(B,X)\ar @{|.>}[r] \ar[dd]_-{(f,1_X)}& B
\ar[dd]^-{f}\\
\hole \\
(B',X) \ar @{|.>}[r] & B'.}
\end{equation}

Going back to properties
of fibred 1-cells, if we unravel the 
definition of a cartesian functor
it is easy to deduce the following well-known result.
\begin{lem}\label{cartfunctcommute}
Suppose we have two fibrations
$P:\ca{A}\to\caa{X}$, $Q:\ca{B}\to\caa{Y}$ and 
a fibred 1-cell $(S,F)$ between them
\begin{displaymath}
\xymatrix @C=.4in @R=.4in
{\ca{A}\ar[r]^-S \ar[d]_-P &
\ca{B}\ar[d]^-Q \\
\caa{X}\ar[r]_-F &
\caa{Y}.}
\end{displaymath}
Then the reindexing functors commute up to
isomorphism with the induced functors
between the fibres. In other words, there
is a natural isomorphism 
\begin{equation}\label{reindexcommute}
\xymatrix @C=.5in @R=.5in
{\ca{A}_Y \ar[d]_-{f^*} \ar[r]^-{S_Y}
\drtwocell<\omit>{'\stackrel{\tau^f}{\cong}} &
\ca{B}_{FY} \ar[d]^-{(Ff)^*} \\
\ca{A}_X \ar[r]_-{S_X}
& \ca{B}_{FX}}
\end{equation}
for every $f:X\to Y$ in $\caa{X}$.
\end{lem}
\begin{proof}
Consider a $P$-cartesian lifting $\Cart(f,A):f^*A\to A$
of $A\in\ca{A}_Y$ along $f:X\to Y$ in $\caa{X}$.
The functor $S$ maps cartesian arrows to
cartesian arrows, so
the morphism 
\begin{displaymath}
\xymatrix @C=.3in
{S(f^*A)\ar[rr]^-{S\Cart(f,A)}
\ar @{.>}[d] && SA \ar @{.>}[d] &\textrm{in }\ca{B} \\
FX\ar[rr]_-{Ff} && FY & \textrm{in }\caa{Y}}
\end{displaymath}
is $Q$-cartesian above 
$Ff$ with codomain $SA$.
On the other hand, the canonical
choice of a $Q$-cartesian lifting of $SA$ along $Ff$
is $\Cart(Ff,SA):(Ff)^*(SA)\to SA.$

Since cartesian arrows are unique up
to vertical isomorphism, 
there exists a unique vertical
isomorphism 
$\tau^f_A:(Ff)^*(SA)\xrightarrow{\;\sim\;}S(f^*A)$
in the fibre $\ca{B}_{FX}$, such that the diagram
\begin{displaymath}
\xymatrix @C=.7in @R=.4in
{(Ff)^*(SA) \ar[d]_-{\tau^f_A}
\ar[dr]^-{\Cart(Ff,SA)} &\\
S(f^*A)\ar[r]_-{S\Cart(f,A)} & SA}
\end{displaymath}
commutes in the total category $\ca{B}$.
The family of invertible arrows $\tau^f_A$
in fact determines a natural isomorphism $\tau^f$
as in (\ref{reindexcommute}).
To establish naturality, for an arrow
$m:A\to A'$ in the fibre
$\ca{A}_Y$ we can form the following diagram:
\begin{displaymath}
 \xymatrix @R=.4in @C=.5in
 {(Ff)^*(SA) \ar @/_10ex/[ddd]_-{(Ff)^*(Sm)}
 \ar[d]_-{\cong}^-{\tau^f_A}
 \ar[drr]^-{\Cart(Ff,SA)} &&& \\
 S(f^*A)\ar[rr]_-{S\Cart(f,A)} 
 \ar[d]_-{(**)\quad}^-{S(f^*m)} && SA\ar[d]^-{Sm} & \\
 S(f^*A')\ar[rr]^-{S\Cart(f,A')}
 \ar[d]_-{\cong}^-{{\tau^f_{A'}}^{-1}} &&
 SA'\ar @{.}[dd] & \textrm{in }\ca{B} \\
 (Ff)^*(SA')\ar[urr]_-{\Cart(Ff,SA')} \ar @{.}[d]
 &&& \\
 FX\ar[rr]^-{Ff} && FY &\textrm{in }\caa{X}}
\end{displaymath}
The outer diagram commutes by definition
of the mapping of $(Ff)^*$
on the arrow $Sm$, the right three inner diagrams
commute for obvious reasons, hence the part (**)
commutes as well, establishing naturality of $\tau^f$.
\end{proof}
In particular, when $S$ is a fibred functor 
between fibrations over the
same base category $\caa{X}$, the isomorphism
(\ref{reindexcommute}) is written
\begin{equation}\label{reindexcommute2}
 \xymatrix @R=.5in @C=.5in
{\ca{A}_Y \ar[d]_-{f^*} \ar[r]^-{S_Y}
\drtwocell<\omit>{'\stackrel{\tau^f}{\cong}} &
\ca{B}_Y \ar[d]^-{f^*} \\
\ca{A}_X \ar[r]_-{S_X}
& \ca{B}_X.}
\end{equation}
This lemma is relevant to the 
correspondence between fibrations and indexed categories,
on the level of structure-preserving functors
appropriate for these concepts. This will become 
clearer in the next section. 

Now given two fibred 1-cells $(S,F),(T,G):P\rightrightarrows Q$
between fibrations $P:\ca{A}\to\caa{X}$ and 
$Q:\ca{B}\to\caa{Y}$, 
a \emph{fibred 2-cell} from $(S,F)$ to
$(T,G)$ is a pair of natural transformations 
($\alpha:S\Rightarrow T,\beta:F\Rightarrow G$)
with $\alpha$ above $\beta$, \emph{i.e.} $Q(\alpha_A)
=\beta_{PA}$ for all $A\in\ca{A}$. We can 
display a fibred 2-cell $(\alpha,\beta)$ 
between two fibred 1-cells as
\begin{equation}\label{fibred2cell}
\xymatrix @C=.8in @R=.5in
{\ca{A}\rtwocell^S_T{\alpha}\ar[d]_-P
& \ca{B}\ar[d]^-Q \\
\caa{X}\rtwocell^F_G{\beta} & \caa{Y}.}
\end{equation}
In particular, when $P$ and $Q$ are 
fibrations over the same base category
$\caa{X}$, we may consider 
fibred 2-cells of the form
$(\alpha,1_{1_{\caa{X}}}):(S,1_{\caa{X}})
\Rightarrow(T,1_\caa{X})$ between
the fibred functors $S$ and $T$,
displayed as
\begin{displaymath}
\xymatrix @C=.3in @R=.5in
{\ca{A}\rrtwocell^S_T{\alpha}\ar[dr]_-P
 && \ca{B}\ar[dl]^-Q\\
 & \caa{X} &}
\end{displaymath}
which
are in fact just natural transformations
$\alpha:S\Rightarrow T$
such that $Q(\alpha_A)=1_{PA}$, \emph{i.e.} whose components
are vertical arrows.
A 2-cell like this is called a \emph{fibred 
natural transformation}.
Dually, we have the notion of an \emph{opfibred 2-cell}
and \emph{opfibred natural transformation}
between opfibred 1-cells and functors respectively.

In this way, we obtain a
2-category $\B{Fib}$ of
fibrations over arbitrary base categories,
fibred 1-cells and fibred 2-cells, with 
the evident compositions coming from $\B{Cat}$. 
In particular, there is a 2-category 
$\B{Fib}(\caa{X})$
of fibrations over a fixed base category $\caa{X}$,
fibred functors and fibred natural transformations.
We also have the 2-categories $\B{Fib}_{\mathrm{sp}}$
and $\B{Fib}(\caa{X})_{\mathrm{sp}}$ of 
split fibrations and morphisms which preserve 
the splitting on the nose (\emph{i.e.}
up to equality and not 
only up to isomorphism).
Dually, we have the 2-categories
$\B{OpFib}$ and $\B{OpFib(\caa{X})}$ of opfibrations
over arbitrary base categories and over a fixed base
category $\caa{X}$ accordingly, as 
well as $\B{OpFib}_{\mathrm{sp}}$ and 
$\B{OpFib}(\caa{X})_{\mathrm{sp}}$ for the split cases.

As a matter of fact, $\B{Fib}$ and $\B{OpFib}$
are (non-full) sub-2-categories of the 2-category 
$\B{Cat}^{\B{2}}=[\B{2},\B{Cat}]$
of `arrows in  $\B{Cat}$',
with objects plain functors
between categories, morphisms commutative squares
of categories and functors as in 
(\ref{commutativefibredcell}) and 2-cells 
pairs of natural transformations as in 
(\ref{fibred2cell}). 
Also
$\B{Fib}(\caa{X})$ and $\B{OpFib}(\caa{X})$
are sub-2-categories of the slice 2-category
$\B{Cat}/\caa{X}$.

These 2-categories form part of fibrations
themselves: we already know that the functor
$cod:\B{Cat}^{\B{2}}\to\B{Cat}$
is the fundamental fibration (\ref{cod}), 
so consider its restriction to 
$\B{Fib}$. Proposition \ref{changeofbasefibr}
implies that this functor
\begin{displaymath}
 cod|_{\B{Fib}}:\B{Fib}\longrightarrow\B{Cat}
\end{displaymath}
which sends a fibration to its base
is still a fibration, with cartesian morphisms
pullback squares and fibre categories
$\B{Fib}(\caa{X})$ for a category 
$\caa{X}\in\B{Cat}$. Also,
the restricted functor
\begin{displaymath}
 cod|_\B{OpFib}:\B{OpFib}\longrightarrow\B{Cat}
\end{displaymath}
is again a fibration, with fibres 
$\B{OpFib}(\caa{X})$ for each category $\caa{X}$.

\section{Indexed categories and the Grothendieck construction}\label{indexedcats}

Given a category
$\caa{X}$, a \emph{$\caa{X}$-indexed category}
is a pseudofunctor 
$$\ps{M}:\caa{X}^\mathrm{op}\to\B{Cat}$$ 
which, by Definition
\ref{laxfunctor}, amounts to the following
data: a category $\ps{M}X$ for every object 
$X\in\ob\caa{X}$ and a functor 
$\ps{M}f:\ps{M}Y\to\ps{M}X$ for each arrow 
$f:X\to Y$, 
together with natural isomorphisms
$\delta_{f,g}:\ps{M}f\circ\ps{M}g
\cong\ps{M}(g\circ f)$ for each composable
pair of arrows and
$\gamma_X:1_{\ps{M}X}\cong\ps{M}(1_X)$ for
each object in $\caa{X}$, satisfying
associativity and identity laws
(\ref{laxcond1}, \ref{laxcond2}). 
The categories $\ps{M}X$ are usually
called \emph{fibres} and the functors $\ps{M}f$
are called \emph{reindexing} and are sometimes
denoted by $f^*$. The terminology already 
indicates the relation with fibrations.

If $\ps{M}$ and $\ps{H}$ are $\caa{X}$-indexed
categories, a $\caa{X}$-\emph{indexed functor}
$\tau:\ps{M}\to\ps{H}$
is a pseudonatural
transformation  
\begin{displaymath}
 \xymatrix @C=.6in
{\caa{X}^\mathrm{op}
\rtwocell^{\ps{M}}_{\ps{H}}{\tau} & 
\B{Cat}}. 
\end{displaymath}
By Definition \ref{laxnattrans},
this means that for each object $X$ of $\caa{X}$
there is a functor $\tau_X:\ps{M}X\to\ps{H}X$
and for each arrow $f:X\to Y$ there is a natural
isomorphism
\begin{displaymath}
\xymatrix
{\ps{M}X\ar[r]^-{\ps{M}f}
\ar[d]_-{\tau_X} & \ps{M}Y
\ar[d]^-{\tau_Y} \\
\ps{H}X\ar[r]_-{\ps{H}f} & 
\ps{H}Y\ultwocell<\omit>{'{\stackrel{\tau_f}{\cong}}}}
\end{displaymath}
subject to the compatibility
conditions with the $\delta_{f,g}$ and $\gamma_X$
expressed by (\ref{compatcomp1}, \ref{compatunit1}).

If $\tau,\sigma:\ps{M}\to\ps{H}$
are $\caa{X}$-indexed functors, a 
$\caa{X}$-\emph{indexed natural transformation}
$m:\tau\to\sigma$ is a modification,
which by
Definition \ref{modification} consists of a 
family $m_X:\tau_X\Rightarrow\sigma_X$ of natural transformations
for every object $X\in\ob\caa{X}$ subject
to compatibility conditions with the coherence
isomorphisms $\tau_f$ and $\sigma_f$
expressed by (\ref{modificcond}).

Notice that in the above definitions,
the ordinary category $\caa{X}$ is regarded
as a 2-category with no non-trivial 2-cells. 
As discussed in Section \ref{bicatbasicdefinitions},
the above data form a 2-category 
$[\caa{X}^\mathrm{op},\B{Cat}]_{\mathrm{ps}}$
of $\caa{X}$-indexed categories,
$\caa{X}$-indexed functors and $\caa{X}$-indexed natural 
transformations, also denoted as $\B{ICat}(\caa{X})$.

The following establishes a correspondence  
between cloven fibration and
indexed categories, due to Grothendieck, which
amounts to an equivalence between the 2-categories
$\B{Fib}(\caa{X})$ and $\B{ICat}(\caa{X})$
for a category $\caa{X}$.
\begin{thm}\label{maintheoremfibr} \hfill
\begin{enumerate}[(i)]
\item Every cloven fibration $P:\ca{A}\to\caa{X}$
gives rise to a $\caa{X}$-indexed category
$\ps{M}_P:\caa{X}^\mathrm{op}\to\B{Cat}$.
\item \emph{[Grothendieck construction]}
Every indexed category $\ps{M}:\caa{X}^\mathrm{op}
\to\B{Cat}$ gives rise to a cloven
fibration $P_\ps{M}:\Gr{G}\ps{M}\to
\caa{X}$.
\item
The above correspondences yield an equivalence of
2-categories
\begin{equation}\label{equivalencegroth}
\B{ICat}(\caa{X})\simeq\B{Fib}(\caa{X})
\end{equation}
so that $\ps{M}_{P_\ps{M}}\cong\ps{M}$ and $P_{\ps{M}_P}\cong
P$.
\end{enumerate}
\end{thm}
\begin{proof}
(i) Let $P:\ca{A}\to\caa{X}$ be 
a cloven fibration. We can define a pseudofunctor 
$\ps{M}_P:\caa{X}^\mathrm{op}\to\B{Cat}$ as follows:
\begin{itemize}
\renewcommand{\labelitemi}{$\cdot$}
\item Each object $X\in\caa{X}$ is mapped
to the fibre category over this object, \emph{i.e.} 
$\ps{M}_P(X)=\ca{A}_X$.
\item Each morphism $f:X\to Y$ in $\caa{X}$ is mapped to
the reindexing functor
$\ps{M}_P(f)=f^*:\ca{A}_Y\to\ca{A}_X$ as in (\ref{reindexing}).
\item Given $g:Y\to Z$ and $A\in\ca{A}_Z$,
there is a natural isomorphism
$\delta^{f,g}:\ps{M}_P(f)\circ\ps{M}_P(g)\xrightarrow{\sim}
\ps{M}_P(g\circ f)$, explicitly described above (\ref{deltaiso}).
\item For any object $A\in\ca{A}_X$, 
there is a natural isomorphism
$\gamma^X:1_{\ps{M}_P(X)}\xrightarrow{\sim}\ps{M}_P(1_A)$ 
described in detail above (\ref{gammaiso}).
\end{itemize}
It is straightforward to check that these natural isomorphisms
$\delta$ and $\gamma$ satisfy the coherence conditions
for a pseudofunctor as described in the Definition \ref{laxfunctor}.

(ii) Let $\caa{X}$ be a category and $\ps{M}:\caa{X}^\mathrm{op}
\to\B{Cat}$ an indexed category over $\caa{X}$. The
\emph{Grothendieck category}  $\Gr{G}\ps{M}$
of $\ps{M}$ is defined as follows: 
objects are pairs $(A,X)$ where $X\in\ob\caa{X}$
and $A\in\ob(\ps{M}X)$, and morphisms 
$(A,X)\to (B,Y)$ are pairs $(\phi,f)$ where
$f:X\to Y$ is an arrow in $\caa{X}$ and 
$\phi:A\to(\ps{M}f)B$
is an arrow in $\ps{M}X$.
This can also be written as 
\begin{equation}\label{grmorphism}
\begin{cases} 
A\xrightarrow{\phi}(\ps{M}f)B &\text{in }\ps{M}X\\
X\xrightarrow{f}Y &\text{in }\caa{X}.
\end{cases}
\end{equation}
The composite of two arrows in this category 
$(A,X)\xrightarrow{(\phi,f)}
(B,Y)\xrightarrow{(\psi,g)}(C,Z)$ is
$(\theta,g\circ f):(A,X)\to(C,Z)$,
where $\theta$ is the composite
\begin{displaymath}
A\xrightarrow{\phi}(\ps{M}f)B\xrightarrow{(\ps{M}f)\psi}
(\ps{M}f\circ\ps{M}g)C\xrightarrow{(\delta_{f,g})_C}
[\ps{M}(g\circ f)]C
\end{displaymath}
for $\delta$ is the natural isomorphism as in 
(\ref{delta}) of the pseudofunctor $\ps{M}$.
The coherence axiom (\ref{laxcond1}) 
for $\delta_{g,f}$ ensures
the associativity of this composition.
Notice how we employ the components of
the 2-cell $\delta_{f,g}$, since in this case it 
is actually a natural transformation 
(the codomain of the 
pseudofunctor is $\B{Cat}$).

Moreover, the identity arrow for each
$(A,X)\in\Gr{G}\ps{M}$ is 
$(i,1_X):(A,X)\to(A,I)$, 
where $i$ is the composite
\begin{displaymath}
 A\xrightarrow{1_A}(1_{\ps{M}X})A\xrightarrow{(\gamma_X)_A}
\ps{M}(1_X)A
\end{displaymath}
where $\gamma$ is the natural isomorphism
as in (\ref{gamma}). Again, the identity laws follow from
the coherence conditions (\ref{laxcond2}) of the
pseudofunctor $\ps{M}$, and so $\Gr{G}\ps{M}$
is a category.

In fact, the projection functor 
\begin{displaymath}
 P_\ps{M}:\Gr{G}\ps{M}\longrightarrow\caa{X}
\end{displaymath}
which maps each object $(A,X)$ to $X$ and each
morphism $(\phi,f)$ to $f$ is a cloven 
fibration: for each arrow
$f:X\to Y$ of the base category $\caa{X}$ and an object 
$(B,Y)$ over $Y$, we can choose the following top arrow
\begin{displaymath}
\xymatrix @R=.4in @C=.8in
{((\ps{M}f)B,X)\ar[r]^-{(1_{(\ps{M}f)B},f)} 
\ar @{.>}[d]_-{P_\ps{M}} & 
(B,Y) \ar @{.>}[d]^-{P_\ps{M}} & \textrm{in }\Gr{G}\ps{M}\\
X\ar[r]_-{f} & Y & \textrm{in }\caa{X}}
\end{displaymath}
to be the cartesian lifting $\Cart(f,(B,Y))$.
Notice that the fibres $(\Gr{G}\ps{M})_X$ of the
fibration $P_\ps{M}$
over $X\in\ob(\caa{X})$ are isomorphic 
to the categories $\ps{M}X$,
due to the isomorphism 
$A\equiv(1_{\ps{M}X})A\cong\ps{M}(1_X)A$.

(iii) By Proposition \ref{knownequivalenceprop},
in order to exhibit an equivalence
between two 2-categories, it is enough 
to construct a fully faithful 
and essentially surjective on objects
2-functor between them.
Hence, we will demonstrate how 
the `Grothendieck construction' mapping 
on objects 
$\ps{M}\mapsto(P_\ps{M}:\Gr{G}\ps{M}\to\caa{X})$
extends to a 2-functor
\begin{displaymath}
\Gr{P}:[\caa{X}^\mathrm{op},\B{Cat}]_{\mathrm{ps}}
\longrightarrow\B{Fib}(\caa{X})
\end{displaymath}
with the following two properties:

$\bullet$ If $\ps{M},\ps{H}:\caa{X}^\mathrm{op}\to
\B{Cat}$ are two $\caa{X}$-indexed categories, 
there is an isomorphism
\begin{equation}\label{isohomcats}
\Gr{P}_{\ps{M},\ps{H}}:
[\caa{X}^\mathrm{op},\B{Cat}]_{\mathrm{ps}}(\ps{M},\ps{H})
\cong\B{Fib}(\caa{X})(P_\ps{M},P_\ps{H})
\end{equation}
between the category of pseudonatural transformations
and modifications and the category of fibred functors
and fibred natural transformations accordingly.

$\bullet$ Every fibration $P:\ca{A}\to\caa{X}$ is 
isomorphic to a fibration $P_\ps{M}:\Gr{G}\ps{M}\to\caa{X}$
arising from a pseudofunctor $\ps{M}:\caa{X}^\mathrm{op}\to\B{Cat}$.

Consider a pseudonatural transformation
$\tau:\ps{M}\Rightarrow\ps{H}$, consisting of
functors $\tau_X:\ps{M}X\to\ps{H}X$ for all $X\in\ob\caa{X}$
and natural isomorphisms 
$\tau^f:\tau_X\circ\ps{M}f\xrightarrow{\sim}\ps{H}f\circ\tau_Y$ 
for all arrows $f:X\to Y$.
Then, define  
\begin{displaymath}
 \Gr{P}_{\ps{M},\ps{H}}(\tau):\Gr{G}\ps{M}
\longrightarrow\Gr{G}\ps{H}
\end{displaymath}
to be the functor
which maps an object $(A,X)\in\Gr{G}\ps{M}$ to
$(\tau_XA,X)\in\Gr{G}\ps{H}$ and an arrow
$(\phi,f):(A,X)\to(B,Y)$ in $\Gr{G}\ps{M}$ like
(\ref{grmorphism}) to
\begin{displaymath}
 \begin{cases} 
\tau_XA\xrightarrow{\psi}(\ps{H}f)(\tau_YB) &\text{in }\ps{H}X\\
X\xrightarrow{f}Y &\text{in }\caa{X}
\end{cases}
\end{displaymath}
in $\Gr{G}\ps{H}$,
where $\psi$ is the composite
\begin{displaymath}
 \tau_XA\xrightarrow{\tau_X(\phi)}(\tau_X\circ\ps{M}f)B
 \xrightarrow{\tau^f_B}(\ps{H}f\circ\tau_Y)B.
\end{displaymath}
The fact that $\Gr{P}_{\ps{M},\ps{H}}(\tau)$ is a functor follows
from the axioms of the pseudonatural transformation $\tau$,
and it can be easily shown that it preserves
cartesian liftings, via the isomorphisms $\tau^f_B$
for all $B$.
The triangle 
\begin{displaymath}
 \xymatrix @C=.5in @R=.5in
 {\Gr{G}\ps{M}\ar[rr]^-{\Gr{P}_{\ps{M},\ps{L}}(\tau)} 
 \ar[dr]_{P_\ps{M}} && \Gr{G}\ps{H}\ar[dl]^-{P_\ps{H}} \\
 & \caa{X} &}
\end{displaymath}
commutes trivially, since
the object of $\caa{X}$ which is projected 
by the fibrations remains unchanged, therefore 
 $\Gr{P}_{\ps{M},\ps{H}}(\tau)$ is a fibred functor.

Now consider a modification 
$m:\tau\Rrightarrow\sigma$
between pseudonatural 
transformations $\tau,\sigma:\ps{M}\Rightarrow\ps{H}$,
given by a family of 
natural transformations $m_X:\tau_X\Rightarrow\sigma_X$.
We can then define a natural transformation
\begin{displaymath}
 \Gr{P}_{\ps{M},\ps{H}}(m):\Gr{P}_{\ps{M},\ps{H}}(\tau)
\Rightarrow \Gr{P}_{\ps{M},\ps{H}}(\sigma)
\end{displaymath}
by setting its components, for each $(A,X)$
in $\Gr{G}\ps{M}$, to be
$((m_X)_A,1_X):(\tau_XA,X)\to(\sigma_XA,X)$.
The conditions which
make 
\begin{displaymath}
 \xymatrix @C=.7in @R=.7in
{\Gr{G}\ps{M}\rrtwocell<\omit>
{\qquad\quad\;\Gr{P}_{\ps{M},\ps{H}}(m)}
\ar @/^2.5ex/[rr]^-{\Gr{P}_{\ps{M},\ps{H}}(\tau)}
\ar @/_2.5ex/[rr]_-{\Gr{P}_{\ps{M},\ps{H}}(\sigma)}
 \ar[dr]_-{P_\ps{M}} && 
\Gr{G}\ps{H}\ar[dl]^-{P_\ps{H}} \\
 & \caa{X} &}
\end{displaymath}
into a fibred natural
transformation are satisfied by the coherence axioms for
the modification $m$.

The above data define a 2-functor in a 
straightforward way, and moreover
the functor $\Gr{P}_{\ps{M},\ps{H}}$
is an isomorphism of categories, 
since the mappings above are bijective.
Therefore an isomorphism (\ref{isohomcats})
is established.

For the second property, 
the goal is to show that every
fibration $P:\ca{A}\to\caa{X}$ is
specifically isomorphic to $P_{\ps{M}_P}$
in $\B{Fib}(\caa{X})$. The latter fibration arises 
by applying
the Grothendieck construction to the
induced pseudofunctor $\ps{M}_P$
as constructed
at part $(i)$ of the proof.
Indeed, there exists an invertible
fibred functor
\begin{displaymath}
 \xymatrix @C=.3in @R=.4in
 {\ca{A}\ar[rr]^-F\ar[dr]_-P && \Gr{G}\ps{M}_P
 \ar[dl]^-{P_{\ps{M}_P}} \\
 & \caa{X} &}
\end{displaymath}
which maps an object $A$ in $\ca{A}$
to the pair $(A,PA)$ in the Grothendieck
category $\Gr{G}\ps{M}_P$, and a morphism 
$\phi:A\to B$ to
$(\theta,P\phi):(A,PA)\to (B,PB)$, where
$\theta$ is the unique vertical arrow
of the $P$-factorization of $\phi$:
\begin{displaymath}
 \xymatrix @C=.5in
 {A\ar[r]^-\phi \ar@{-->}[d]_-\theta &
 B \\
 (P\phi)^*B\ar[ur]_-{\Cart(P\phi,B)} &}
\end{displaymath}
Functoriality follows from uniqueness of 
cartesian liftings, and $F$ is evidently
bijective on 
objects and on arrows. Also,
it preserves cartesian arrows and
commutes with the fibrations $P$,
$P_{\ps{M}_P}$ hence it is 
an isomorphism of fibrations
over $\caa{X}$.
\end{proof}
Notice that in the Grothendieck construction above,
we write the pairs in the opposite 
way from the standard notation.
The same will apply for the form 
of objects
and morphisms of all fibred categories
studied later on.

The equivalence (\ref{equivalencegroth})
clearly restricts to one between 
split fibrations over $\caa{X}$ and `strict'
$\caa{X}$-indexed categories, \emph{i.e.} functors
from $\caa{X}^\mathrm{op}$ to $\B{Cat}$:
\begin{displaymath}
[\caa{X}^\mathrm{op},\B{Cat}]
\simeq\B{Fib}(\caa{X})_\mathrm{sp}
\end{displaymath}
Dually, we have an analogous result relating opfibrations
$U:\ca{C}\to\caa{X}$ and `covariant indexed categories',
\emph{i.e.} pseudofunctors $\ps{F}:\caa{X}\to\B{Cat}$.
\begin{thm}\label{maintheoremopfibr}
There is an equivalence of 2-categories
\begin{displaymath}
[\caa{X},\mathbf{Cat}]_{\mathrm{ps}}\simeq\B{OpFib}(\caa{X}).
\end{displaymath}
In particular, every opfibration $U:\ca{C}\to\caa{X}$
is isomorphic to $U_\ps{F}:\Gr{G}\ps{F}\to\caa{X}$ 
arising from a pseudofunctor
$\ps{F}:\caa{X}\to\B{Cat}$, and there is an isomorphism
of categories
\begin{displaymath}
[\caa{X},\B{Cat}](\ps{F},\ps{G})
\cong\B{OpFib}(\caa{X})(U_\ps{F},U_\ps{G})
\end{displaymath}
for any two pseudofunctors
$\ps{F},\ps{G}:\caa{X}\to\B{Cat}$.
\end{thm}
The above theorems show how 
$\caa{X}$-indexed
categories are `essentially the same as' cloven
fibrations over $\caa{X}$, and covariant
indexed categories as opfibrations, 
hence we are able
to freely pass from the one structure to the other
depending on our needs. Via this process, we can also 
transfer properties and state them in the fibrational
or indexed categories language at will.

As an example of how indexed and covariant
indexed categories can be convenient means of
studying fibrations and opfibrations, 
consider the following situation. 
If $K:\ca{C}\to\ca{D}$ is an opfibred functor
between $U:\ca{C}\to\caa{X}$ and $V:\ca{D}\to\caa{X}$,
by the dual of Lemma 
\ref{cartfunctcommute} 
there is a natural isomorphism 
\begin{equation}\label{commutewithreindex}
\xymatrix
{\ca{C}_X\ar[r]^-{f_!}
\ar[d]_-{K_X} \drtwocell<\omit>{'\cong}
& \ca{C}_{Y}\ar[d]^-{K_Y}
\\
\ca{D}_X\ar[r]_-{f_!} & \ca{D}_{Y}}
\end{equation}
for any arrow $f:X\to Y$ in $\caa{X}$.
We can also deduce this as follows.
By Theorem \ref{maintheoremopfibr},
the opfibrations $U$, $V$ correspond to
pseudofunctors $\ps{F},\ps{G}:\caa{X}\to\B{Cat}$,
in the sense that 
$U$ is isomorphic to
$U_\ps{F}:\Gr{G}\ps{F}\to\caa{X}$
and $V$ is isomorphic to
$U_\ps{G}:\Gr{G}\ps{G}\to\caa{Y}$.
In particular $\ca{C}_X\cong\ps{F}X$,
$\ca{D}_X\cong\ps{G}X$ and 
the reindexing functors are
$\ps{F}f$ and $\ps{G}f$ respectively. Now, 
the opfibred functor $K$ 
corresponds uniquely to an 
$\caa{X}$-indexed functor $\tau:\ps{F}\Rightarrow\ps{G}$, which is 
a pseudonatural transformation
equipped with 
an (ordinary) natural isomorphism with components
\begin{displaymath} 
\xymatrix
{\ps{F}X\ar[r]^-{\ps{F}f}
\ar[d]_-{\tau_X} \drtwocell<\omit>{'\stackrel{\tau_f}{\cong}}
 & \ps{F}Y\ar[d]^-{\tau_{Y}}\\
\ps{G}X\ar[r]_-{\ps{G}f} & \ps{G}Y}
\end{displaymath}
for every $f:X\to Y$ in $\caa{X}$. This diagram
corresponds uniquely to an isomorphism
exactly like (\ref{commutewithreindex}).
This is evident after the realization
that the functors $K_X$ induced between the fibres
as in Remark \ref{functorsbetweenfibres}
are precisely $\tau_X$.

As another example, suppose that 
$Q:\ca{B}\to\caa{Y}$ is a fibration
which corresponds uniquely to 
the pseudofunctor 
$\ps{H}:\caa{Y}^\mathrm{op}\to\B{Cat}$.
Then, if $F:\caa{X}\to\caa{Y}$ is a functor,
the fibration $F^*Q$
obtained from $Q$ by change of base
along $F$
\begin{displaymath}
\xymatrix
{F^*(\ca{B}) \pullbackcorner[ul] 
\ar[r]^-K \ar[d]_-{F^*Q} & \ca{B}\ar[d]^-Q\\
\caa{X}\ar[r]_-F & \caa{Y}}
\end{displaymath}
as in Proposition \ref{changeofbasefibr},
corresponds to the 
composite pseudofunctor
\begin{displaymath}
\caa{X}^\mathrm{op}\xrightarrow{\;F^\mathrm{op}\;}
\caa{Y}^\mathrm{op}\xrightarrow{\;\ps{H}\;}\B{Cat}.
\end{displaymath}
This is evident by part $(i)$ of the
proof of the above theorem, since 
its mapping on objects is
\begin{displaymath}
\ps{H}(FX)\cong\ca{B}_{FX}\cong F^*(\ca{B})_X
\end{displaymath}
by $\ca{B}\cong\Gr{G}\ps{H}$ and the isomorphism
(\ref{Sxiso}). On arrows, for $f:Z\to X$ in $\caa{X}$
and $B$ above $FX$ we have
\begin{displaymath}
(\ps{H}(Ff)B,X)\cong((Ff)^*B,X)=f^*(B,X)
\end{displaymath}
for $f^*$ the reindexing functor
of the fibration $F^*Q$.

The 2-categories of the form
$\B{ICat}(\caa{X})$ for each $\caa{X}$, sometimes also
denoted as $\B{Cat}_\caa{X}$, turn out to also be 
fibres of a fibration, like their equivalent $\B{Fib(\caa{X})}$.
Explicitly, there is a 2-category $\B{ICat}$
with objects indexed categories 
$\ps{M}:\caa{X}^\mathrm{op}\to\B{Cat}$ for
arbitrary categories $\caa{X}$. A morphism
from $\ps{M}$ to $\ps{H}:\caa{Y}^\mathrm{op}\to\B{Cat}$
is given by a functor $F:\caa{X}\to\caa{Y}$
and an $\caa{X}$-indexed functor
\begin{displaymath}
\xymatrix @C=.6in @R=.25in
{\caa{X}^\mathrm{op} \ar[r]^-{\ps{M}} \ar[d]_-{F^\mathrm{op}}
\rtwocell<\omit>{<2>\tau} & \B{Cat} \\
\caa{Y}^\mathrm{op} \ar @/_2ex/[ur]_-{\ps{H}} & }
\end{displaymath}
and we write $(F,\tau):\ps{M}\to\ps{H}$.
Notice the direct relation with the indexed expression
of pullbacks described above.
A 2-cell $(F,\tau)\to(G,\sigma)$ is given
by a natural transformation
$\beta:F\Rightarrow G$ and a modification
\begin{displaymath}
 \xymatrix @C=.5in
{\caa{X}^\mathrm{op} \ar[r]^-{F^\mathrm{op}}
\ar @/^5ex/[rr]^-{\ps{M}} \rrtwocell<\omit>{<-3>\tau}
\ar @/_4ex/[r]_-{G^\mathrm{op}}
\rtwocell<\omit>{<2>\quad\beta^\mathrm{op}} & 
\caa{Y}^\mathrm{op} \ar[r]^-{\ps{H}} & 
\B{Cat}}\stackrel{m}{\Rrightarrow}
\xymatrix @C=.3in
{\caa{X}^\mathrm{op} 
\ar @/^3ex/[rr]^-{\ps{M}} \rrtwocell<\omit>{\sigma}
\ar @/_3ex/[rr]_-{\ps{H}\circ G^\mathrm{op}} && \B{Cat}.}
\end{displaymath}
Compositions and identities are defined
using those in $\B{Cat}$ and $\B{ICat}(-)$.
Hence, there is a (2-)functor
\begin{displaymath}
 base:\B{ICat}\longrightarrow\B{Cat}
\end{displaymath}
which maps an indexed category to its domain
and a morphism to its first component. This is 
a split fibration, with fibres $\B{ICat}(\caa{X})$ 
above $\caa{X}$
and reindexing functors
precomposition with $F^\mathrm{op}$
for each $F:\caa{X}\to\caa{Y}$ in $\B{Cat}$.
For more details,
we refer the reader to \cite{hermidaphd} or 
\cite{Jacobs}.
\begin{thm}
 There is a (2-)equivalence in the 
2-category $\B{Fib}(\B{Cat})$
\begin{displaymath}
 \xymatrix
{\B{ICat}\ar[rr]^-{\simeq} \ar[dr]_-{base} &&
\B{Fib}\ar[dl]^-{cod} \\
& \B{Cat}. &}
\end{displaymath}
\end{thm}

\section{Fibred adjunctions and fibrewise limits}\label{fibredadjunctions}

The notions of fibred and opfibred adjunction
come from Definition \ref{adjunction2cat}
applied to the 2-categories $\B{Fib}$ and 
$\B{OpFib}$.
\begin{defi}\label{generalfibredadjunction}
Given fibrations $P:\ca{A}\to\caa{X}$ and 
$Q:\ca{B}\to\caa{Y}$,
a \emph{general fibred adjunction} is given by
a pair of fibred 1-cells $(L,F):P\to Q$ and 
$(R,G):Q\to P$ together with
fibred 2-cells $(\zeta,\eta):(1_\ca{A},1_\caa{X})\Rightarrow
(RL,GF)$ and 
$(\xi,\varepsilon):(LR,FG)\Rightarrow
(1_\ca{B},1_\caa{Y})$
such that $L\dashv R$ via $\zeta,\xi$
and $F\dashv G$ via $\eta,\varepsilon$. This
is displayed as
\begin{displaymath}
\xymatrix @C=.7in @R=.4in
{\ca{A} \ar[d]_-P 
\ar @<+.8ex>[r]^-L\ar@{}[r]|-\bot
& \ca{B} \ar @<+.8ex>[l]^-{R} \ar[d]^-Q \\
\caa{X} \ar @<+.8ex>[r]^-F\ar@{}[r]|-\bot
& \caa{Y} \ar @<+.8ex>[l]^-G}
\end{displaymath} 
and we write $(L,F)\dashv(R,G):Q\to P$. 
In particular, a \emph{fibred adjunction}
is an adjunction in the 2-category $\B{Fib}(\caa{X})$,
displayed as
\begin{equation}\label{fibredadjunction}
\xymatrix @C=.3in @R=.4in
{\ca{A} \ar[dr]_-P 
\ar @<+.8ex>[rr]^-L\ar@{}[rr]|-\bot
&& \ca{B} \ar @<+.8ex>[ll]^-R \ar[dl]^-Q \\
& \caa{X}. & }
\end{equation}
\end{defi}
Notice that since $(\zeta,\eta)$
and $(\xi,\varepsilon)$ are fibred 2-cells, by definition
$\zeta$ is above $\eta$ and $\xi$ is above $\varepsilon$,
which makes $(P,Q)$ into a map of adjunctions (see Definition
\ref{mapofadunctions}).

Dually, we have the notions of \emph{general opfibred 
adjunction} and \emph{opfibred adjunction} for adjunctions
in the 2-categories $\B{OpFib}$ and $\B{OpFib}(\caa{X})$
respectively.
Moreover, for the 2-categories
$\B{Fib}_\mathrm{sp}$, $\B{Fib}(\caa{X})_\mathrm{sp}$,
$\B{OpFib}_\mathrm{sp}$ and $\B{OpFib}(\caa{X})_\mathrm{sp}$
they are called \emph{general split (op)fibred 
adjunction} and \emph{split (op)fibred adjunction}.
Then, the functors $L$ and $R$ are required to preserve
the cleavages of the split (op)fibrations on the nose. 

Since a basic aim in this section is to identify 
conditions under which (op)fibred functors 
and (op)fibred 1-cells have left or right adjoints, 
we recall the following well-known 
important fact (e.g. see 
\cite[4.5]{Winskel}).
\begin{lem}\label{Winskellemma}
 Right adjoints in the 2-category
$\B{Cat}/\caa{X}$ preserve cartesian arrows and dually
left adjoints in $\B{Cat}/\caa{X}$ preserve cocartesian arrows. The
same holds for adjoints in the 2-category $\B{Cat}^\B{2}$.
\end{lem}
This will prove very useful, since for example if a fibred
functor has an ordinary right adjoint between the total categories
which commutes with the fibrations,
then the adjoint is necessarily fibred too. 

It is clear that a fibred adjunction as in (\ref{fibredadjunction})
induces fibrewise adjunctions 
\begin{displaymath}
\xymatrix @C=.6in
 {\ca{A}_X \ar @<+.8ex>[r]^-{L_X}\ar@{}[r]|-\bot
  & \ca{B}_X \ar @<+.8ex>[l]^-{R_X}}
\end{displaymath}
between the fibre categories for each $X$ in $\caa{X}$.
In the converse direction,
we have the following result, see for example
\cite[8.4.2]{Handbook2} or \cite[1.8.9]{Jacobs}.
\begin{prop}\label{Borceuxfibrewise}
 Suppose $S:Q\to P$ is a fibred functor between
fibrations $Q:\ca{B}\to\caa{X}$ and $P:\ca{A}\to\caa{X}$.
Then $S$ has a fibred
left adjoint $L$ if and only if for each $X\in\caa{X}$ 
we have $L_X\dashv S_X$, and the adjunct arrows
\begin{equation}\label{chi}
\chi_A:(L_X\circ f^*)A\longrightarrow(f^*\circ L_Y)A
\end{equation}
described below are isomorphisms for all $A\in\ca{A}_Y$
and $f:X\to Y$. Similarly, $S$ has a fibred right adjoint $R$
iff $S_X\dashv R_X$ and $(f^*\circ R_Y)B\cong(R_X\circ f^*)B$.
\end{prop}
\begin{rmk*}
An equivalent formulation of the above, coming from the
correspondent notion of indexed adjunctions (\emph{i.e.}
adjunction in the 2-category $\B{ICat(\caa{X})}$),
appears in \cite{hermidaphd}: a fibred adjunction $L\dashv R$
amounts to a family of adjunctions 
$\{L_X\dashv R_X:\ca{B}_X\to\ca{A}_X\}_{X\in\caa{X}}$
such that for every $f:Y\to X$, $({f^*}^P,{f^*}^Q)$
is a pseudo-map of adjunctions.
\end{rmk*}
\begin{proof}
Since $S$ is cartesian,
the image of a cartesian lifting 
\begin{displaymath}
 S_X(f^*L_YA)\xrightarrow{S\Cart(f,L_YA)}S_Y(L_YA)
\end{displaymath}
is again a cartesian arrow above $f$ in the total
category $\ca{A}$, for any $A\in\ca{A}_Y$. 
Therefore the composite
top arrow below factorizes uniquely through it
via an isomorphism:
\begin{equation}\label{defpiA}
 \xymatrix @C=.4in @R=.4in
{f^*A\ar[rr]^-{\Cart(f,A)}
\ar @{-->} [d]_-{\exists!\pi_A} && 
A\ar[d]^-{\eta^Y_A} & \\
S_X(f^*L_YA)\ar@{.>}[d] \ar[rr]_-{S\Cart(f,L_YA)} && 
S_YL_YA\ar@{.>}[d] & \textrm{in }\ca{A}\\
X\ar[rr]_-{f} && Y & \textrm{in }\caa{X}}
\end{equation}
The arrow $\chi_A$ in (\ref{chi})
which we require to be
an isomorphism is the one that 
corresponds under $L_X\dashv S_X$
to $\pi_A$:
\begin{displaymath}
 \xymatrix @R=.02in
{\qquad\qquad\pi_A: f^*A \ar[rr] && S_Xf^*L_YA 
\qquad\qquad& \mathrm{in}\;\ca{A}_X\\ 
\ar@{-}[rr] &&& \\  
\qquad\qquad\chi_A:L_Xf^*A\ar[rr] && f^*L_YA\qquad\qquad 
& \mathrm{in}\;\ca{B}_X} 
\end{displaymath}
Then, these $L_X$ assemble into
a fibred left adjoint $L:\ca{A}\to\ca{B}$:
on objects we define $LA:=L_YA$ for $A\in\ca{A}_Y$,
and on arrows we define
$L(\phi)$ for 
\begin{displaymath}
 \xymatrix @C=.4in @R=.2in
 {C\ar[rr]^\phi\ar[d]_-{\theta} && A 
 \ar @/^/@{.>}[dd] &&\\
 f^*A\ar[urr]_-{\;\Cart(f,A)}
 \ar @{.>}[d] &&& \textrm{in }\ca{A} \\
 X\ar[rr]_-{f} && Y & \textrm{in }\caa{X}}
\end{displaymath}
to be the composite
\begin{equation}\label{Lonarrows}
 \xymatrix @C=.4in @R=.25in
 {L_XC\ar @{-->}[rr]^-{L\phi} \ar[d]_-{L_X\theta} && 
 L_YA \ar @/^/@{.>}[ddd] &\\
 L_Xf^*A\ar[d]_-{\chi_A} &&& \textrm{in }\ca{B}\\
 f^*L_YA \ar[uurr]_-{\;\Cart(f,L_YA)}
 \ar @{.>}[d] &&&  \\
 X\ar[rr]_-{f} && Y & \textrm{in }\caa{X}.}
\end{equation}
Functoriality of $L$ follows, and also
we can directly verify that it is a cartesian functor. 
Using the fibrewise adjunctions we can also
show that $\eta$ and $\varepsilon$ are natural
with respect to all morphisms
and not just those in the fibres.
\end{proof}
\begin{rmk}\label{Kock}
There is an equivalent and perhaps more intuitive
way of phrasing the condition
that the transpose $\chi_A$
of $\pi_A$ defined in (\ref{defpiA})
is an isomorphism, as in \cite{Jacobs} or \cite{KockKock}.
We require that the \emph{Beck-Chevalley condition} holds, 
\emph{i.e.}
the mate 
\begin{displaymath}
\xymatrix @R=.4in @C=.4in
{\ca{A}_Y \ar[d]_-{f^*}
\ar[r]^-{L_Y} &
\ca{B}_Y \ar[d]^-{f^*}  \\
\ca{A}_X\ar[r]_-{L_X} & \ca{B}_X
\ultwocell<\omit>{\chi}}
\end{displaymath}
of the canonical invertible 2-cell 
\begin{displaymath}
\xymatrix @R=.4in @C=.4in
{\ca{B}_Y \ar[d]_-{f^*}
\ar[r]^-{S_Y} &
\ca{A}_Y \ar[d]^-{f^*}  \\
\ca{B}_X\ar[r]_-{S_X} & \ca{A}_X
\ultwocell<\omit>{'\cong}}
\end{displaymath}
as in (\ref{reindexcommute2})
which comes with the cartesian functor $S:\ca{B}\to\ca{A}$,
is invertible as well.
Using the mates correspondence 
of Proposition \ref{mates},
we can explicitly compute the
component $\chi_A$
as the composite
\begin{displaymath}
 \xymatrix @C=.55in
{L_Xf^*A\ar[r]^-{L_Xf^*\eta_A}
\ar@{-->}[drr]_-{\chi_A} & 
L_Xf^*S_YL_YA\ar[r]^-{L_X\tau_{L_YA}} &
L_XS_Xf^*L_YA\ar[d]^-{\varepsilon_{f^*L_YA}} \\
&& f^*L_YA}
\end{displaymath}
by applying (\ref{mate2})
for the adjunctions $L_Y\dashv S_Y$
and $L_X\dashv S_X$.

Similarly for the existence of a right fibred adjoint,
the mate
\begin{displaymath}
\xymatrix @R=.4in @C=.4in
{\ca{A}_Y \drtwocell<\omit>{\omega}
\ar[d]_-{f^*}
\ar[r]^-{R_Y} &
\ca{B}_Y \ar[d]^-{f^*}  \\
\ca{A}_X\ar[r]_-{R_X} & \ca{B}_X}
\end{displaymath}
under the fibrewise adjunctions $S_{(-)}\dashv R_{(-)}$
is requested to be an isomorphism.
\end{rmk}
Notice that in order to just define an ordinary 
left adjoint $L:\ca{A}\to\ca{B}$ of the 
fibred functor $S$ between the 
total categories, the adjunction between 
the fibres and the components of the mate
$\chi$ are sufficient, as can be seen from the defining
diagram (\ref{Lonarrows}).
The supplementary fact that $\chi$
should be an isomorphism ensures that this adjoint
is also cartesian, therefore constitutes a fibred
adjoint of $K$. On the other hand,
for the existence of a right adjoint of $S$, 
the natural transformation $\omega$ being an isomorphism 
is required for the very construction of $R$,
since the components $\omega_A$ initially go to 
the opposite direction than the one needed.

Similarly, there is a dual
result concerning fibrewise adjunctions
between opfibrations over a fixed base.
\begin{prop}\label{Borceuxfibrewise2}
 Suppose that $K:U\to V$ is an opfibred
functor between opfibrations $U:\ca{C}\to\caa{X}$
and $V:\ca{D}\to\caa{X}$.
It has a right opfibred adjoint 
$R:\ca{D}\to\ca{C}$ (respectively left opfibred
adjoint $L$) if and only if
it has a fibrewise adjoint $K_X\dashv R_X$
(respectively $L_X\dashv K_X$)
and the mate of the 
isomorphism $\sigma:K_Y\circ f_!\cong f_!\circ K_X$,
given by the components
\begin{equation}\label{omegaD}
 \xymatrix @C=.55in @R=.3in
{f_!R_XD\ar[r]^-{\eta_{f_!R_XD}}
\ar@{-->}[drr] & 
R_YK_Yf_!R_XD\ar[r]^-{R_Y\sigma_{R_XD}} &
R_Yf_!K_XR_XD\ar[d]^-{R_Yf_!\varepsilon_D} \\
&& R_Yf_!D}
\end{equation}
(respectively the mate $L_Yf_!\Rightarrow
f_!L_Y$) for any $D\in\ca{D}_X$ is also invertible.
\end{prop}
These results give rise to questions concerning
adjunctions between fibrations over two different bases 
rather than the same as above.
In this direction, Theorem \ref{totaladjointthm} below is a 
generalization whose special case coincides
with the above proposition. In what follows, 
we emphasize more on the existence of a total adjoint
(induced only by its mapping on objects) 
and then we proceed to its full description.
We initially look at 
opfibrations because of the nature of the examples that arise
later.
\begin{lem}\label{totaladjointlem}
Suppose $(K,F):U\to V$ is an opfibred 1-cell
given by the commutative square
\begin{displaymath}
\xymatrix @C=.4in
{\ca{C}\ar[r]^-K \ar[d]_-U
& \ca{D}\ar[d]^-V\\
\caa{X}\ar[r]_-F & \caa{Y}}
\end{displaymath} 
and there is an 
adjunction between the base categories
\begin{equation}\label{baseadjunction}
 \xymatrix @C=.4in @R=.3in
 {\caa{X} \ar @<+.8ex>[r]^-F
\ar@{}[r]|-\bot
& \caa{Y}.\ar @<+.8ex>[l]^-G}
\end{equation}
with counit $\varepsilon$. If, for each $Y\in\caa{Y}$, the composite
functor
\begin{equation}\label{specialfunctor}
\ca{C}_{GY}\xrightarrow{K_{GY}}\ca{D}_{FGY}
\xrightarrow{(\varepsilon_Y)_!}\ca{D}_Y
\end{equation}
has a right adjoint $R_Y$, 
then $K:\ca{C}\to\ca{D}$ between 
the total categories 
has a right adjoint, with $R_{(-)}$ 
its mapping on objects.
\end{lem}
\begin{proof}
The adjunction
$(\varepsilon_Y)_!K_{GY}\dashv R_Y$
comes with a natural isomorphism
\begin{equation}\label{specialcaseadjun}
\ca{D}_Y([(\varepsilon_Y)_!\circ K_{GY}](Z), D)
\cong\ca{C}_{GY}(Z,R_Y(D))
\end{equation}
for any $Z\in\ca{C}_{GY}$, $D\in\ca{D}_Y$.
We claim that this induces
a bijective correspondence
\begin{equation}\label{naturalbijcor}
\ca{D}(KC,D)\cong\ca{C}(C,R_YD)
\end{equation}
for any $C\in\ca{C}_X$ and 
$D\in\ca{D}_Y$, which is natural in $C$. 
In other words, there is a representation
of the functor $\ca{D}(K-,D)$ with representing
object $R_YD$.
Then, by adjunctions via representations,
there is a unique way to define a 
functor 
\begin{displaymath}
 R:\ca{D}\longrightarrow\ca{C}
\end{displaymath}
with object functions
$R_{(-)}$ depending on the fibre of
the objects,
such that (\ref{naturalbijcor})
is natural also in $D$ thus gives
an adjunction $K\dashv R$.

An element of the 
left hand side of (\ref{naturalbijcor}) 
is an arrow $m:KC\to D$ in the total category $\ca{D}$,
which can be encoded by 
\begin{equation}\label{leftside}
\begin{cases}
f_!(KC)\xrightarrow{k}D &\text{in }\ca{D}_Y\\
FX\xrightarrow{f}Y &\text{in }\caa{Y}
\end{cases}
\end{equation}
where $k$ is the unique vertical arrow 
of the factorization
$m=k\circ\Cocart(f,KC)$.

An element of the right hand side of (\ref{naturalbijcor})
is an arrow $n:C\to R_YD$
in the total category $\ca{C}$, \emph{i.e.}
\begin{displaymath}
\begin{cases}
g_!C\xrightarrow{l}R_YD &\text{in }\ca{C}_{GY}\\
X\xrightarrow{g}GY &\text{in }\caa{X}
\end{cases}
\end{displaymath}
where $n=l\circ\Cocart(g,C)$.
By the natural isomorphism (\ref{specialcaseadjun}) and
the adjunction (\ref{baseadjunction}), 
this pair of arrows corresponds bijectively 
to a pair
\begin{displaymath}
\begin{cases}
[(\varepsilon_Y)_!\circ K _{GY}]
(g_!C)\xrightarrow{\hat{l}}D &\text{in }\ca{D}_Y\\
\qquad\qquad\quad\; FX\xrightarrow{\tilde{g}}Y &\text{in }\caa{Y}
\end{cases}
\end{displaymath}
where $\hat{l}$ is the adjunct of $l$
under (${\varepsilon_Y}_!K_{GY}\dashv R_Y$)
and $\tilde{g}$ is the adjunct of $g$ 
under $F\dashv G$, hence it satisfies
$\tilde{g}=Fg\circ\varepsilon_Y$.

In order for this pair to 
actually constitute an arrow
$KC\to D$ in $\ca{D}$ as in (\ref{leftside}),
it is enough to show that 
\begin{displaymath}
 [(\varepsilon_Y)_!K_{GY}](g_!C)\cong\tilde{g}_!(KC)
\end{displaymath}
in the fibre $\ca{D}_Y$.
For that, observe that the diagram
\begin{equation}\label{diagimportantproof}
\xymatrix @C=.7in @R=.5in
{\ca{C}_X\ar[r]^-{g_!}\ar[d]_-{K_X} 
&\ca{C}_{GY}\ar[r]^-{K_{GY}} &
\ca{D}_{FGY}\ar[d]^-{(\varepsilon_Y)_!} \\
\ca{D}_{FX}\ar[rr]_-{\tilde{g}_!}
\ar @{-->}[urr]_-{(Fg)_!} &&
\ca{D}_Y}
\end{equation}
commutes up to isomorphism: the
left part is an
isomorphism for any cocartesian 
functor $K$, dual to 
(\ref{reindexcommute}),
and the right part is
the isomorphism 
\begin{displaymath}
q^{Fg,\varepsilon_Y}:
\tilde{g}_!\xrightarrow{\sim}(Fg)_!\circ
(\varepsilon_Y)_!
\end{displaymath}
from the uniqueness of cartesian
liftings, as in (\ref{qiso}).

In other words, the bijective correspondence
(\ref{naturalbijcor}) 
is formally induced by a mapping
$\ca{C}(C,R_YD)\to\ca{D}(KC,D)$
explicitly given by
\begin{displaymath}
\begin{cases}
g_!C\xrightarrow{l}RD &\text{in }\ca{C}_{GY}\\
X\xrightarrow{g}GY &\text{in }\caa{X}
\end{cases}
\mapsto
\begin{cases}
(\varepsilon Fg)!KC\stackrel{q}{\cong}
\varepsilon_!(Fg)_!KC\stackrel{\varepsilon_!\sigma^g}{\cong}
\varepsilon_!Kg_!C\stackrel{\theta(l)}{\to}D
&\text{in }\ca{D}_Y\\
FX\xrightarrow{\tilde{g}}Y &\text{in }\caa{Y}
\end{cases}
\end{displaymath}
where $\theta$ is the natural bijection 
(\ref{specialcaseadjun}) and $\varepsilon$ is short for $\varepsilon_Y$. 
Naturality
in $C$ can be checked, so a right
adjoint $R$ of $K$ between the 
total categories can be defined.
\end{proof}
Since this result in essence generalizes 
Proposition \ref{Borceuxfibrewise2}, it is 
reasonable to explore
the appropriate conditions 
in order for this right adjoint $R$ to 
be cocartesian and thus to establish a general
opfibred adjunction. Initially we are interested
in adjusting Remark \ref{Kock} on this case.

If we call $\sigma^f$ the isomorphism induced by
cocartesianness of the functor $K$
employed in the above proof,
for some $h:Y\to W$ in $\caa{Y}$ in particular 
we have a natural isomorphism
\begin{displaymath}
 \xymatrix @C=.5in @R=.5in
{\ca{C}_{GY}\ar[d]_-{(Gh)_!}\ar[r]^-{K_{GY}}
\drtwocell<\omit>{'\stackrel{\sigma^{Gh}}{\cong}} &
\ca{D}_{FGY}\ar[d]^-{(FGh)_!} \\
\ca{C}_{GW}\ar[r]_-{K_{GW}} & \ca{D}_{FGW}.} 
\end{displaymath}
Also, by sheer naturality of $\varepsilon$,
we have an isomorphism
\begin{displaymath}
 \nu:(\varepsilon_W)_!(FGh)_!\stackrel{q}{\cong}
(\varepsilon_W\circ FGh)_!=(h\circ \varepsilon_Y)_!
\stackrel{q}{\cong}h_!(\varepsilon_Y)_!.
\end{displaymath}
We can now form an invertible composite
2-cell
\begin{equation}\label{KGh*nu}
 \xymatrix @C=.9in @R=.75in
{\ca{C}_{GY}\ar[d]_-{(Gh)_!}\ar[r]^-{K_{GY}}
\drtwocell<\omit>{'\stackrel{\sigma^{Gh}}{\cong}} &
\ca{D}_{FGY}\ar[d]^-{(FGh)_!}
\ar[r]^-{(\varepsilon_Y)_!} & 
\ca{D}_Y \ar[d]^-{h_!} \\
\ca{C}_{GW}\ar[r]_-{K_{GW}} &
\ca{D}_{FGW}\ar[r]_-{(\varepsilon_W)_!} &
\ca{D}_W. \ultwocell<\omit>{'\stackrel{\nu}{\cong}} }
\end{equation}
Its mate $\omega$ under the adjunctions
$(\varepsilon_Y)_!K_{GY}\dashv R_Y$ and
$(\varepsilon_W)_!K_{GW}\dashv R_W$
has components, by (\ref{mate1}),
\begin{equation}\label{omegaD2}
 \xymatrix @C=.5in @R=.6in
{(Gh)_!R_YD\ar[r]^-{\bar{\eta}^W}\ar@{-->}[drr]_-{\omega_D} &
\big(R_W((\varepsilon_W)_!K_{GW})\big)(Gh)_!R_YD \ar[r]^-{R_W(\sigma^{Gh}*\nu)}
& R_W\big(h_!(\varepsilon_E)_!K_{GY}\big)R_YD\ar[d]^-{R_Wh_!\bar{\varepsilon}^Y} \\
&& (R_Wh_!)D}
\end{equation}
where $\bar{\eta}$ and $\bar{\varepsilon}$ are the unit and counit
of the adjunctions ${\varepsilon_{(-)}}_!K_{G(-)}\dashv R_{(-)}$. 
These arrows $\omega_D$ which generalize
the composites (\ref{omegaD}), are essential for 
the explicit construction of $R$.

In a dual way to Proposition 
\ref{Borceuxfibrewise}, $R$
maps an arrow
\begin{displaymath}
\xymatrix @C=.6in @R=.3in
{D \ar @{.}[dd]
\ar[rr]^k \ar[drr]_-
{\Cocart(h,D)} && E 
&\\
&& h_!D \ar[u]_-{\psi}
\ar @{.}[d] & \textrm{in }\ca{D} \\
Y\ar[rr]_-h && W & \textrm{in }\caa{Y}}
\end{displaymath}
to the composite

\begin{displaymath}
\xymatrix @C=.6in @R=.12in
{R_YD \ar @{.}[ddd]
\ar @{-->}[rr]^{Rk} \ar[ddrr]_-{\Cocart(Gh,R_WD)\;\;\;}
&& R_WE &\\
&& R_W(h_!D) \ar[u]_-{R_W\psi} & \textrm{in }\ca{C}\\
&& (Gh)_!R_YD \ar[u]_-{\omega_D} \ar@{.}[d]
&  \\
GY\ar[rr]_-{Gh} && GW & \textrm{in }\caa{X}}
\end{displaymath}
where $\omega_D$ are the arrows (\ref{omegaD2}).
It is now not hard to see that
by construction 
of $R$, the square of categories 
and functors
\begin{displaymath}
\xymatrix @R=.3in @C=.4in
{\ca{C}\ar[d]_-U & \ca{D}\ar[l]_-R
\ar[d]^-V \\
\caa{X} & \caa{Y}\ar[l]^-G}
\end{displaymath}
commutes. Moreover, if 
$(\zeta,\xi)$ is the unit and counit
of $K\dashv R$, the pairs
$(\zeta,\eta)$ and $(\xi,\varepsilon)$
are above each other. Consequently
$(K,F)\dashv(R,G)$ is already
an adjunction in $\B{Cat}^\B{2}$.
Finally, if we request that the $\omega_D$'s
are isomorphisms, putting
$k=\Cocart(g,D)$ in the 
mapping above exhibits the cocartesianness
of $R$.
\begin{thm}\label{totaladjointthm}
Suppose $(K,F):U\to V$ is an opfibred 1-cell
and $F\dashv G$ is an adjunction between 
the bases of the fibrations, 
as in
\begin{displaymath}
\xymatrix @C=.6in
{\ca{C}\ar[r]^-K\ar[d]_-U & \ca{D}\ar[d]^-V \\
\caa{X}\ar @<+.8ex>[r]^-F
\ar@{}[r]|-\bot
& \caa{Y}. \ar @<+.8ex>[l]^-G}
\end{displaymath}
If the composite (\ref{specialfunctor}) has 
a right adjoint for each $Y\in\caa{Y}$, then $K$
has a right adjoint $R$ between the total categories,
with $(K,F)\dashv(R,G)$ in $\B{Cat^2}$. 
If the mate
\begin{displaymath}
\quad\xymatrix @R=.3in @C=.3in
{\ca{D}_Y\ar[r]^-{R_Y}\ar[d]_-{h_!}
\drtwocell<\omit>{\omega} & \ca{C}_{GY}\ar[d]^-{(Gh)_!} \\
\ca{D}_W\ar[r]_-{R_W} & \ca{C}_{GW}}
\end{displaymath}
of the composite
invertible 2-cell (\ref{KGh*nu})
is moreover an isomorphism for any
$h:Z\to W$ in $\caa{Y}$,
then $R$ is cocartesian and so
\begin{displaymath}
\xymatrix @C=.6in
{\ca{C} \ar[d]_-U \ar @<+.8ex>[r]^-K
\ar@{}[r]|-\bot
& \ca{D}\ar @<+.8ex>[l]^-R \ar[d]^-V \\
\caa{X}\ar @<+.8ex>[r]^-F
\ar@{}[r]|-\bot
& \caa{Y}\ar @<+.8ex>[l]^-G}
\end{displaymath}
is a general opfibred adjunction. Conversely,
if $(K,F)\dashv (R,G)$ in $\B{OpFib}$, then evidently
$F\dashv G$, $K\dashv R$, $R$ is cocartesian, 
and moreover for every $Y\in\caa{Y}$ there
is an adjunction $(\varepsilon_Y)_!K_{GY}\dashv R_Y$
between the fibres.
\end{thm}
\begin{proof}
The first part is just Lemma \ref{totaladjointlem} and
the process that follows. For the converse,
start with some $f:C\to R_YD$ in $\ca{C}_{GY}$.
There is a bijective correspondence
\begin{displaymath}
 \xymatrix @R=.02in @C=.2in
{& (C,GY)\ar[r]^-{(f,1_{GY})} & (R_YD,GY)\equiv R(D,Y) && \mathrm{in}\;\ca{C}\\ 
\ar@{-}[rrr] &&&& \\  
& K(C,GY)\equiv (K_{GY}C,FGY) \ar[r]^-{(\bar{f},\varepsilon_Y)} 
& (D,Y) && \mathrm{in}\;\ca{D}}
\end{displaymath}
since $K\dashv R$, but the latter morphism is uniquely determined
by the vertical arrow 
$\bar{\bar{f}}:(\varepsilon_Y)_!K_{GY}C\to D$
in $\ca{D}_Y$ because of the factorization of any arrow
through the cocartesian lifting.
Hence the required fibrewise adjunction is established.
\end{proof}
Dually, we get the following version
about adjunctions between fibrations.
\begin{thm}\label{totaladjointthm2}
 Suppose $(S,G):Q\to P$ is a fibred 1-cell
between two fibrations
and $F\dashv G$ is an adjunction between
the bases, as shown in the diagram
\begin{displaymath}
 \xymatrix @C=.6in
{\ca{A}\ar[d]_-P & \ca{B}\ar[l]_-S \ar[d]^-Q \\
\caa{X}\ar @<+.8ex>[r]^-F
\ar@{}[r]|-\bot
& \caa{Y}. \ar @<+.8ex>[l]^-G}
\end{displaymath}
If, for each $X\in\caa{X}$, the composite functor
\begin{displaymath}
 \ca{B}_{FX}\xrightarrow{S_{FX}}\ca{A}_{GFX}\xrightarrow{\eta_X^*}
\ca{A}_X
\end{displaymath}
has a left adjoint $L_X$, then $S$ 
has a left adjoint $L$ between the total categories, 
with $(L,F)\dashv(S,G)$ in $\B{Cat}^\B{2}$.
Furthermore, if the mate
\begin{displaymath}
 \quad\xymatrix @C=.3in @R=.3in
{\ca{A}_Z\ar[r]^-{L_Z}\ar[d]_-{f^*}
 & \ca{B}_{FZ}\ar[d]^-{(Ff)^*} \\
\ca{A}_X\ar[r]_-{L_X} & \ca{B}_{FX}
\ultwocell<\omit>{ }}
\end{displaymath}
of the composite isomorphism
\begin{displaymath}
 \xymatrix @C=.65in @R=.45in
{\ca{B}_{FZ}\ar[d]_-{(Ff)^*}\ar[r]^-{S_{FZ}}
\drtwocell<\omit>{'\stackrel{\tau^{Ff}}{\cong}} &
\ca{A}_{GFZ}\ar[d]^-{(GFf)^*}
\ar[r]^-{(\eta_Z)^*} & 
\ca{A}_Z \ar[d]^-{f^*} \\
\ca{B}_{FX}\ar[r]_-{S_{FX}} &
\ca{A}_{GFX}\ar[r]_-{(\eta_X)^*} &
\ca{A}_X \ultwocell<\omit>{'\stackrel{\kappa}{\cong}}}
\end{displaymath}
is invertible for any $f:X\to Z$ in $\caa{X}$,
then
\begin{displaymath}
 \xymatrix @C=.6in
{\ca{A} \ar[d]_-P \ar @<+.8ex>[r]^-L
\ar@{}[r]|-\bot
& \ca{B}\ar @<+.8ex>[l]^-S \ar[d]^-Q \\
\caa{X}\ar @<+.8ex>[r]^-F
\ar@{}[r]|-\bot
& \caa{Y}\ar @<+.8ex>[l]^-G}
\end{displaymath}
is a general fibred adjunction.
Conversely, if $(L,F)\dashv(S,G)$ is 
an adjunction in $\B{Fib}$, we have adjunctions
$L_X\dashv\eta_X^* S_{FX}$ for all $X\in\caa{X}$.
\end{thm}
In the above composite 2-cell, 
the 2-isomorphism $\tau^{Ff}$ comes from the cartesian functor
$S$ as in (\ref{reindexcommute}) 
and $\kappa$ from naturality of $\eta$,
the unit of the base adjunction.

We finish this section with some 
general results concerning fibrewise completeness
and cocompleteness. In fact, Hermida's work on
fibred adjunctions was mainly motivated by its
applications on the existence of 
fibred limits and colimits. For us though,
the establishment of general (op)fibred adjunctions
serves different purposes.

For any small category $\ca{J}$,
we say that a fibration $P:\ca{A}\to\caa{X}$ 
has \emph{fibred $\ca{J}$-limits} (respectively
\emph{colimits})
if and only if the fibred functor
$\hat{\Delta}_\ca{J}:\ca{A}\to\Delta^*([\ca{J},\ca{A}])$
uniquely determined by the diagram below has 
a fibred right (respectively left) adjoint:
\begin{equation}\label{fibredlimits}
 \xymatrix @R=.15in @C=.3in
{\ca{A}\ar @/^3ex/[drrr]^-{\tilde{\Delta}_\ca{J}} 
\ar @/_3ex/[dddr]_-P
\ar @{-->} [dr] ^-{\hat{\Delta}_\ca{J}} &&&\\
& \Delta_\ca{J}^*[\ca{J},\ca{A}]\pullbackcorner[ul]
\ar[rr]^-\pi \ar[dd]^-{\Delta_\ca{J}^*[\ca{J},P]}_-{\star\quad} && 
[\ca{J},\ca{A}]\ar[dd]^-{[\ca{J},P]}\\
\hole \\
&\caa{X}\ar[rr]^-{\Delta_\ca{J}} && [\ca{J},\caa{X}]}
\end{equation}
where $\Delta_\ca{J}$ and $\tilde{\Delta}_\ca{J}$
are the constant diagram functors. Notice that $[\ca{J},P]$
is a fibration when $P$ is, where cartesian morphisms
are formed componentwise. 
We write 
$\big(\hat{\mathrm{lim}_\ca{J}}\dashv\hat{\Delta}_\ca{J}
\dashv\hat{\mathrm{colim}_\ca{J}}\big)$
when the fibration $P$ has fibred limits and colimits.
Dually we can define \emph{opfibred $\ca{J}$-colimits}
and \emph{limits} for an opfibration $U$.
\begin{prop}\label{reindexcont}
 A fibration $P:\ca{A}\to\caa{X}$ has all fibred
$\ca{J}$-limits (colimits)
if and only if every fibre has $\ca{J}$-limits
(colimits) and the reindexing functors $f^*$
preserve them, for any arrow $f$.
\end{prop}
\begin{proof}
 By Proposition \ref{Borceuxfibrewise}, the fibred functor
$\hat{\Delta}_\ca{J}$ has a fibred right adjoint $R$ if 
and only if there is an adjunction between the fibres
$(\hat{\Delta}_\ca{J})_X\dashv R_X$ and
we have isomorphisms 
$(R_X f^*)C\cong(f^*R_Y)C$ for any $f:X\to Y$ and $C\in\ca{C}_Y$.
The first condition is equivalent to each fibre $\ca{A}_X$
being $\ca{J}$-complete, since 
\begin{displaymath}
(\Delta_\ca{J}^*[\ca{J},\ca{A}])_X\cong[\ca{J},\ca{A}]_{\Delta_{\ca{J}}X}=
[\ca{J},\ca{A}_X]
\end{displaymath}
by construction of the pullback fibration, and 
$(\hat{\Delta}_\ca{J})_X:\ca{A}_X\to[\ca{J},\ca{A}_X]$ 
is the constant diagram functor.
If we call this fibrewise adjoint
$R_X=\mathrm{lim}_X$, the second condition becomes
\begin{displaymath}
(\mathrm{lim}_X\circ[\ca{J},f^*])F\cong(f^*\circ\mathrm{lim}_Y)F
\end{displaymath}
for any functor $F:\ca{J}\to\ca{A}_Y$, which means precisely
that any $f^*$
preserves limits between the fibre categories. Dual
arguments apply for the existence of colimits.
\end{proof}
There is an equivalent definition of a
fibred $\ca{J}$-complete
fibration $P:\ca{A}\to\caa{X}$. 
In \cite[8.5.1]{Handbook2}, it is
stated that $P$ has all $\ca{J}$-limits 
when the (outer) fibred 1-cell 
$(\tilde{\Delta}_\ca{J},\Delta_\ca{J})$
given by (\ref{fibredlimits})
has a fibred right adjoint. The difference
relates to whether we require an adjunction between fibrations
over the same bases or not, since 
the factorization through the pullback 
is a tool which permits the restriction
of the problem
from $\B{Fib}$ to $\B{Fib}(\caa{X})$. The following result
illustrates the latter.\vspace{.05in}
\begin{thm}~\cite[3.2.3]{hermidaphd}\label{hermidathm}
Given $P:\ca{A}\to\caa{X}$, $Q:\ca{B}\to\caa{Y}$,
$F\dashv G:\caa{Y}\to\caa{X}$ via $\eta$, $\varepsilon$
and a fibred 1-cell $(S,F):P\to Q$ as shown
in the following diagram
\begin{displaymath}
 \xymatrix @C=.6in
{\ca{A}\ar[d]_-P \ar[r]^-S
& \ca{B}\ar[d]^-Q \\
\caa{X}\ar @<+.8ex>[r]^-F
\ar@{}[r]|-\bot
& \caa{Y},\ar @<+.8ex>[l]^-G}
\end{displaymath}
let $\hat{S}:P\to F^*Q$ in $\B{Fib}(\caa{X})$
be the unique mediating functor in
\begin{displaymath}
 \xymatrix @R=.1in @C=.35in
{\ca{A}\ar @/^4ex/[drrr]^-S 
\ar @/_4ex/[dddr]_-P
\ar @{-->} [dr] ^-{\hat{S}} &&&\\
& F^*\ca{B}\pullbackcorner[ul]
\ar[rr]^-{\pi} \ar[dd]_-{F^*Q} && 
\ca{B}\ar[dd]^-Q\\
\hole \\
&\caa{X}\ar[rr]_-F && \caa{Y}.}
\end{displaymath}
Then, the following statements are
equivalent:

$i)$ There exists $R:\ca{B}\to\ca{A}$ such that
$S\dashv R$ in $\B{Cat}$ and $(S,F)\dashv(R,G)$
in $\B{Fib}$.

$ii)$ There exists $\hat{R}:F^*Q\to P$ such that 
$\hat{S}\dashv\hat{R}$ in $\B{Fib}(\caa{X})$.
\end{thm}
This theorem uses the fact that change of base along $F$ 
as in Proposition \ref{changeofbasefibr} yields
a so-called \emph{cartesian fibred adjunction}
when $F$ has a right adjoint, meaning $(\pi,F)$ has
an adjoint in $\B{Fib}$. Therefore,
by performing change of base along a left
adjoint functor, we can factorize a general
fibred adjunction into a cartesian and 
`vertical' fibred adjunction, hence `reduce' 
a general fibred adjunction to a fibred adjunction.
Dually, this can be done for a general opfibred adjunction
accordingly.

Using the above theorem, we can deduce 
fibrewise completeness conditions from the total
category of the fibration and vice versa.
\begin{cor}~\cite[3.3.6]{hermidaphd}\label{AhasPpreserves}
Let $\ca{J}$ be a small category and 
$P:\ca{A}\to\caa{X}$ be a fibration such that
the base category $\caa{X}$ has all $\ca{J}$-limits.
Then the fibration $P$ has all 
fibred $\ca{J}$-limits if and only if $\ca{A}$
has and $P$ strictly preserves (chosen)
$\ca{J}$-limits.
\end{cor}
The proof relies on Lemma \ref{Winskellemma} 
and essentially
constructs a general fibred adjunction
$(\tilde{\Delta}_\ca{J},\Delta_\ca{J})
\dashv(\mathrm{\tilde{lim}}_\ca{J},\mathrm{lim}_\ca{J})$
for the outer diagram (\ref{fibredlimits}).
Dually, we obtain fibred colimits 
for an opfibration with a cocomplete base, from colimits
in the total category which are strictly preserved
by the opfibration.
\begin{rmk*}
In essence, Theorems \ref{totaladjointthm} and 
\ref{totaladjointthm2} relate to very similar questions
as Theorem \ref{hermidathm}, namely
the assumptions under which we obtain general 
fibred and opfibred adjunctions (starting with an (op)fibred
1-cell). However, they actually respond to the exact
opposite problems: Theorem \ref{totaladjointthm2} provides
with a \emph{left} adjoint between the total functors,
whereas Theorem \ref{hermidathm} reduces the existence 
of a \emph{right} fibred 1-cell adjoint to a right
fibred adjoint. This connection
should perhaps be further explored. For example,
we could use the new results to study 
fibred cocompleteness of fibrations and fibred 
completeness of opfibrations.
\end{rmk*}

\part*{PART II}
\chapter{Enrichment of Monoids and Modules}\label{enrichmentofmonsandmods}
\section{Universal measuring comonoid and enrichment}\label{Universalmeasuringcomonoid}

The notion of the universal measuring coalgebra
was first introduced by Sweedler \cite{Sweedler}
in the context of vector spaces over a field $k$.
The question that motivated the definition 
of measuring coalgebras is under which 
conditions, for $A,B$ $k$-algebras and $C$ a 
$k$-coalgebra, the linear map 
$\rho\in\Hom_k(A,\Hom_k(C,B))$ corresponding 
under the usual tensor-hom adjunction to 
$\sigma\in\Hom_k(C\otimes_k A,B)$ in $\Vect_k$
is actually an algebra map.

More explicitly, the natural bijective
correspondence defining the adjunction
$(-\otimes_k C)\dashv\Hom_k(C,-)$
is given by the invertible mapping
\begin{displaymath} 
\xymatrix
@R=.02in
{\Vect_k(A,\Hom_k(C,B))\ar[r] & 
\Vect_k(A\otimes C,B)\\ 
A\xrightarrow{\rho}\Hom_k(C,B)\ar @{|->}[r] &
A\otimes C\xrightarrow{\bar{\rho}}B\qquad \\ 
& \quad a\otimes c\mapsto[\rho(a)](c)}
\end{displaymath} 
where of course $\Vect_k(-,-)=\Hom_k$.
If $C$ is a $k$-coalgebra and 
$B$ a $k$-algebra, it is well-known
that $\Hom_k(C,B)$ obtains the structure of 
a $k$-algebra via convolution, 
also by Remark \ref{rmkconvolution}.
Hence if $A$ is also a $k$-algebra,
we may ask under which conditions on 
$\bar{\rho}$, the corresponding linear 
map $\rho$ is a $k$-algebra homomorphism. 
This resulted in the following definition.
\begin{defi*}
 If $A,B$ are $k$-algebras, $C$ a $k$-coalgebra and
$\sigma:C\otimes_k A\to B$ a linear map, 
we say that $(\sigma,C)$ \emph{measures} $A$ to $B$ when 
$\sigma$ satisfies:
\begin{align*}
\sigma(c\otimes aa')&=
\sum\limits_{(c)}{\sigma(c_{(1)}\otimes a)
\sigma(c_{(2)}\otimes a')} \\
\sigma(c\otimes 1)&=\epsilon (c)1
\end{align*}
where the sum comes from the sigma notation for
the comultiplication of $C$, and $\epsilon$
is the counit.
\end{defi*}
There is a category of \emph{measuring coalgebras}
and it has a terminal object $P(A,B)$, 
equivalently 
defined by the following one-to-one correspondences
\begin{align}\label{existencePABinitial}
\Alg_k(A,\Hom_k(C,B))&\cong
\{\sigma\in\Hom_k(C\otimes A,B)|\sigma 
\;\mathrm{measures}\} \\
&\cong\Coalg_k(C,P(A,B))\notag
\end{align}
where the first isomorphism comes
from the definition of measuring,
and the second expresses the 
universal property of $P(A,B)$.
This object is called the \emph{universal measuring
coalgebra}, and in \cite[Theorem 7.0.4]{Sweedler}
is constructed as the sum of
certain subcoalgebras of the cofree coalgebra
on the vector space $\Hom_k(A,B)$.

As illustrated in Section
\ref{Categoriesofmonoidsandcomonoids}, 
$\Alg_k$ and $\Coalg_k$
are the categories of monoids and comonoids and 
$\Hom_k$ is the internal hom
in the symmetric monoidal
closed category $\Vect_k$ of $k$-vector spaces
and $k$-linear maps.
The aim is to obtain a generalization of 
$P(A,B)$ in a broader setting,
by identifying the appropriate assumptions on 
a monoidal category $\ca{V}$ in place of 
$\Vect_k$ which allow its existence.

Consider a symmetric monoidal closed category
$\ca{V}$. The lax monoidal internal hom 
functor induces a functor 
between the categories of comonoids and monoids
as in (\ref{defMon[]}),
\begin{displaymath}
 \Mon[-,-]:
\xymatrix @R=.05in
{\Comon(\ca{V})^\op
\times\Mon(\ca{V})\ar[r]
&\Mon(\ca{V})\\
\qquad\;(\;C\;,\;A\;)\;\ar @{|->}[r] & [C,A]}
\end{displaymath}
which is in fact just the restriction of 
the internal hom on $\Comon(\ca{V})^\op
\times\Mon(\ca{V})$.
If we call this functor of two variables $H$,
in order to generalize the isomorphism 
(\ref{existencePABinitial})
it is enough to prove that the functor
\begin{displaymath}
H(-,B)^\op:\Comon(\ca{V})
\longrightarrow\Mon(\ca{V})^\op
\end{displaymath} 
for a fixed monoid $B$ has a right adjoint.
Because of the useful properties of the 
categories of monoids and comonoids in 
admissible categories discussed in Section 
\ref{Categoriesofmonoidsandcomonoids},
and since $\Vect_k$ is itself an example of
such a category, we continue in this direction.
\begin{prop}\label{measuringcomonoidprop}
Suppose that $\ca{V}$ is a locally presentable
symmetric monoidal closed category. 
There is an adjunction 
$H(-,B)^\op\dashv P(-,B)$
with a natural isomorphism
\begin{equation}\label{meascomon}
\Mon(\ca{V})(A,[C,B])\cong
\Comon(\ca{V})(C,P(A,B))
\end{equation}
for any monoids $A$, $B$ and comonoid $C$.
\end{prop}
\begin{proof}
A monoidal category $\ca{V}$ with these properties
belongs to the class of admissible categories, therefore
Proposition \ref{moncomonadm} applies. 
As a result, the category of comonoids
$\Comon(\ca{V})$ is a locally presentable category,
and in particular cocomplete with a small
dense subcategory.
Moreover, there is a commutative diagram
\begin{displaymath}
\xymatrix @C=.7in @R=.5in
 {\Comon(\ca{V})^\op
\ar[r]^-{H(-,B)}\ar[d]_-{U^\op} &
\Mon(\ca{V})\ar[d]^-S\\
\ca{V}^\op\ar[r]_-{[-,SB]} &\ca{V}}
\end{displaymath} 
where the forgetful functors $U$, $S$ are respectively
comonadic and monadic. The bottom functor $[-,SB]$ is continuous 
as the right adjoint
of $[-,SB]^\op$ as in (\ref{adjunctionopinthom}), 
thus the diagram
exhibits $H(-,B)$ as a  
continuous functor.
Hence by Theorem \ref{Kelly}, the cocontinuous $H(-,B)^\op$
has a right adjoint $P(-,B)$
with an isomorphism as in (\ref{meascomon}).
Since this is natural in $A$ and $C$,
there is a unique way to define a functor of two variables
\begin{equation}\label{defP}
P(-,-):\Mon(\ca{V})^\op\times\Mon(\ca{V})
\longrightarrow\Comon(\ca{V})
\end{equation}
which is the parametrized adjoint of $H^\op(-,-)$
by `adjunctions with 
a parameter' Theorem \ref{parametrizedadjunctions}.
\end{proof}
The object $P(A,B)$ for monoids $A$, $B$ is called
the \emph{universal measuring comonoid},
and the functor $P$ is called the universal measuring
comonoid functor or \emph{Sweedler hom} in \cite{AnelJoyal}.
Notice that in fact, a parametrized adjoint for $H^\op$
should have domain $\Mon(\ca{V})\times\Mon(\ca{V})^\op$,
but it is more natural to work with an essentially
identical functor which is contravariant on the first entry,
just by switching the cartesian product in our notation.  

In particular, for the admissible monoidal
closed $\Mod_R$ 
for a commutative ring $R$,
there is a natural isomorphism
\begin{equation}\label{meascoal}
\Coalg_R(C,P(A,B))\cong\Alg_R(A,\Hom_R(C,B))
\end{equation}
defining the \emph{universal measuring coalgebra}
 $P(A,B)$. This is also given by
\cite[Proposition 4]{AdjAlgCoalg}.
\begin{rmk}\label{dualalgebra}
It is a well-known fact that the
\emph{dual} $C^*=\Hom_k(C,k)$
of a $k$-coalgebra, where $k$ is viewed as an algebra
over itself, has a natural structure of 
an algebra.
On the other hand, if $A$ is a $k$-algebra,
its dual 
$A^*=\Hom_k(A,k)$ in general fails to
be a coalgebra, unless
for example it is finite dimensional
as a $k$-vector space. This is 
due to the failure of the canonical
linear map
\begin{displaymath}
 V^*\otimes_k W^*\to(V\otimes_k W)^*
\end{displaymath}
which gives the lax monoidal structure on $\Hom_k$,
to always be invertible.
However, we can define the subspace 
\begin{displaymath}
A^0=\{g\in A^*|\exists\; \textrm{ideal }I\subset\textrm{ker}g\;
\textrm{s.t.}\;(\textrm{ker}g/I)\;\textrm{f.d.}\}
\end{displaymath}
of $A^*$ which turns out to have the structure of a coalgebra. 
Then, the \emph{dual algebra functor}  
$\Hom_k(-,k)=( - )^*$ is adjoint to $( - )^0$
via the classical isomorphism
\begin{displaymath}
\Coalg_k(C,A^0)\cong\Alg_k(A,C^*).
\end{displaymath} 
This is a special case of
(\ref{meascoal}) for $R=k$, hence Proposition
\ref{measuringcomonoidprop} in fact 
generalizes the dual algebra functor adjunction
to $\Mod_R$, but also in a sense to a 
more general monoidal category $\ca{V}$,
with $(-)^*\cong[-,I]$ and $(-)^0\cong P(-,I)$.
\nocite{Quantum,HopfAlg}
\end{rmk}
We now proceed to the statement and proof of a 
lemma which connects the
adjunction (\ref{meascomon}) with the usual 
$(-\otimes C)\dashv[C,-]$ defining the internal hom.
\begin{lem}\label{lemma} 
Suppose we have a monoid arrow $f:A\to
[C,B]$ for $A$, $B$ monoids, $C$ a comonoid
in a locally presentable symmetric monoidal closed category
$\ca{V}$. 
If this arrow corresponds to 
$\bar{f}:A\otimes C\to B$ in $\ca{V}$ 
under $(-\otimes C)\dashv
[C,-]$ and to $\hat{f}:C\to P(A,B)$ 
in $\Comon(\ca{V})$ under
$H(-,B)^\op\dashv P(-,B)$, then
the two transposes are connected via
\begin{equation}\label{relationlemma}
 \bar{f}=(\varepsilon\otimes\hat{f})\circ\mathrm{ev}
\end{equation}
where $\mathrm{ev}$ is the
evaluation and $\varepsilon$ the counit of the 
universal measuring comonoid adjunction.
\end{lem}
\begin{proof}
Consider the following diagram
\begin{displaymath}
\xymatrix @C=.4in @R=.5in {& [P(A,B),B]\otimes
C\ar[r]^-{1\otimes\hat{f}}\ar[ddr]^-{[\hat{f},1]\otimes1} &
[P(A,B),B]\otimes P(A,B)\ar[dr]^-{\mathrm{ev}_B} &\\ A\otimes
C\ar[ur]^-{\varepsilon_A\otimes1}\ar[drr]_-{f\otimes1}\ar @{-->}
@/_3ex/[rrr]_{\bar{f}} && & B\\ && [C,B]\otimes
C\ar[ur]_-{\mathrm{ev}_B} &}
\end{displaymath}
where the bottom composite defines $\bar{f}$.
Notice that the counit $\varepsilon$ in reality 
has components $H(P(A,B),B)^\op\to A$ in $\Mon(\ca{V})^\op$.

The left part of the diagram gives $f$ from its transpose map $\hat{f}$ 
under $H(-,B)^\op\dashv P(-,B)$. The right part commutes by
dinaturality as in (\ref{dinaturality}) of the counit
$\mathrm{ev}_D^E:[D,E]\otimes D\to E$ of the parametrized adjunction
$(-\otimes -)\dashv [-,-]$. 
Therefore the diagram commutes
and the relation (\ref{relationlemma}) holds.
\end{proof}
We can now combine the existence of the universal
measuring comonoid $P(A,B)$ with the theory 
of actions of monoidal categories
in Section \ref{actions}, in order to
establish an enrichment of $\Mon(\ca{V})$ 
in the symmetric monoidal closed $\Comon(\ca{V})$.
Recall that for any symmetric monoidal
closed category $\ca{V}$, the internal hom 
\begin{displaymath}
 [-,-]:\ca{V}^\op\times\ca{V}\longrightarrow\ca{V}
\end{displaymath}
is an action of the monoidal category
$\ca{V}^\op$ on the category $\ca{V}$,
as explained in Lemma \ref{inthomaction}. Furthermore,
the restricted functor on the categories
of comonoids and monoids
$H=\Mon[-,-]$ is an action too, by the same lemma.
Finally, the opposite
functor of an action is still an action.
Therefore, for the action
\begin{equation}\label{Hopaction}
 H^\op:\Comon(\ca{V})\times\Mon(\ca{V})^\op
\longrightarrow\Mon(\ca{V})^\op
\end{equation}
of the symmetric monoidal closed 
category $\Comon(\ca{V})$ (see Proposition \ref{Comonclosed})
on the ordinary category $\Mon(\ca{V})^\op$, Corollaries
\ref{importcor1} and \ref{importcor2} apply. 
\begin{thm}\label{MonVenrichedinComonV}
 Let $\ca{V}$ be a locally presentable symmetric
monoidal closed category and $P$ the 
Sweedler hom functor.
\begin{enumerate}
 \item The opposite category of monoids 
$\Mon(\ca{V})^\op$ is enriched in 
the category of comonoids $\Comon(\ca{V})$, with hom-objects
\begin{displaymath}
\Mon(\ca{V})^\op(A,B)=P(B,A)
\end{displaymath}
where the $\Comon(\ca{V})$-enriched category is denoted 
by the same name.
\item The category of monoids $\Mon(\ca{V})$
is a tensored and cotensored $\Comon(\ca{V})$-enriched 
category, with hom-objects
\begin{displaymath}
 \Mon(\ca{V})(A,B)=P(A,B)
\end{displaymath}
and cotensor products $[C,B]$ for any comonoid $C$ and monoid $B$.
\end{enumerate}
\end{thm} 
\begin{proof}
By Proposition \ref{measuringcomonoidprop}, there is an 
adjunction 
\begin{displaymath}
 \xymatrix @C=.6in
{\Comon(\ca{V})\ar@<+.8ex>[r]^-{H(-,B)^\op} 
\ar@{}[r]|-{\bot} &
\Mon(\ca{V})^\op\ar@<+.8ex>[l]^-{P(-,B)}} 
\end{displaymath}
which defines the bifunctor $P$ (\ref{defP})
as the parametrized adjoint of the bifunctor
$H^\op$. The latter is an action,
thus an enrichment of the category acted on is induced,
as well as of its opposite category $\Mon(\ca{V})$
since the monoidal category $\Comon(\ca{V})$
is symmetric. 

In particular, since $\Comon(\ca{V})$ is closed, 
the action $[-,-]$ which induces the enrichment
of $\Mon(\ca{V})^\op$ renders it a tensored 
$\Comon(\ca{V})$-category, hence 
its opposite enriched category is cotensored.
On the other hand, $\Mon(\ca{V})$ is also a tensored $\Comon(\ca{V})$-category
because the functor
\begin{displaymath}
 H(C,-)^\op:\Mon(\ca{V})^\op\longrightarrow\Mon(\ca{V})^\op
\end{displaymath}
has a right adjoint for every comonoid $C$. This follows from
the Adjoint Triangle Theorem (see \cite{AdjointTriangles}) 
applied to the commutative diagram
\begin{displaymath}
 \xymatrix @C=.5in
{\Mon(\ca{V})\ar[r]^-{H(C,-)} \ar[d]_-S & \Mon(\ca{V})\ar[d]^-S \\
 \ca{V}\ar[r]_-{[C,-]} & \ca{V}.}
\end{displaymath}
The forgetful $S$ is monadic, the locally presentable
$\Mon(\ca{V})$ has coequalizers and
$[C,-]$ has a left adjoint $(-\otimes C)$. 
Therefore $H(C,-)$ has a left adjoint $C\triangleright -$
for all $C$'s and so there is 
a unique way to define a bifunctor
\begin{equation}\label{deftriangle}
 \triangleright:\Comon(\ca{V})\times\Mon(\ca{V})\longrightarrow\Mon(\ca{V}).
\end{equation}
In \cite{AnelJoyal}, this functor is called the \emph{Sweedler product}.
\end{proof}

\section{Global categories of modules and comodules}\label{globalcats}

In Section \ref{Categoriesofmodulesandcomodules}, 
the categories $\Mod_\ca{V}(A)$ and $\Comod_\ca{V}(C)$
of $A$-modules and $C$-comodules 
for a monoid $A$ and a comonoid $C$
in a monoidal category $\ca{V}$ were defined. The idea here 
is that there exist global categories of modules and comodules, 
which contain all these `fixed (co)monoids' categories, 
with appropriate arrows between
modules and comodules of actions and coactions
from different sources. These global categories 
are central for the development
of this thesis, and their construction
is interrelated with the
theory of fibrations and opfibrations.
\begin{defi}\label{defComod}
The \emph{global category of comodules} $\Comod$
is the category of all $C$-comodules $X$ for
any comonoid $C$, denoted by $X_C$.
A morphism $k_g:X_C\to Y_D$
for $X$ a $C$-comodule and $Y$ a $D$-comodule
consists of a comonoid morphism $g:C\to D$
and an arrow $k:X\to Y$ in $\ca{V}$ which makes
the diagram
\begin{displaymath}
 \xymatrix @C=.45in @R=.5in
{X\ar[r]^-\delta \ar[d]_-k & 
X\otimes C\ar[r]^-{1\otimes g} &
X\otimes D\ar[d]^-{k\otimes1} \\
Y\ar[rr]_-\delta && Y\otimes D}
\end{displaymath}
commute. Dually, the \emph{global category
of modules} $\Mod$ has as objects all
$A$-modules $M$ for any monoid $A$, and
morphisms are $p_f:M_A\to N_B$
where $f:A\to B$ is a monoid morphism
and $p:M\to N$ makes the dual diagram
\begin{displaymath}
 \xymatrix @C=.45in @R=.5in
{A\otimes M\ar[rr]^-\mu\ar[d]_-{1\otimes p} &&
M\ar[d]^-p \\
A\otimes N\ar[r]_-{f\otimes1} & B\otimes N\ar[r]_-\mu &
N}
\end{displaymath}
commute. Conventially, unless otherwise stated
the modules considered will be left and the comodules
considered will be right. 
\end{defi}
There are obvious forgetful functors
\begin{equation}\label{forgetGV}
 G:\Mod\longrightarrow\Mon(\ca{V})\quad\textrm{and}\quad
V:\Comod\longrightarrow\Comon(\ca{V})
\end{equation}
which simply map any module $M_A$/comodule $X_C$ to its
monoid $A$/comonoid $C$ and the morphisms to
their monoid/comonoid part respectively. In fact,
$G$ is a split fibration and $V$ is a
split opfibration:
the descriptions of the global categories 
agree with the Grothendieck categories 
for specific (strict)
functors
\begin{displaymath}
\xymatrix @R=.05in @C=.4in
{\Mon(\ca{V})^\op\ar[r]^-{\Mod_\ca{V}}
& \B{Cat} \\
A\ar@{|.>}[r]\ar[dd]_-f & \Mod_\ca{V}(A) \\
\hole \\
B\ar@{|.>}[r] & \Mod_\ca{V}(B)
\ar[uu]_-{f^*}}
\qquad
\xymatrix @R=.05in @C=.4in
{\Comon(\ca{V})\ar[r]^-{\Comod_\ca{V}} & \B{Cat} \\
C\ar@{|.>}[r]\ar[dd]_-g & \Comod_\ca{V}(C)
\ar[dd]^-{g_!} \\
\hole \\
D\ar@{|.>}[r] & \Comod_\ca{V}(D)}
\end{displaymath}
where $f^*$ and $g_!$ are the restriction
and corestriction of scalars as in 
(\ref{defres}) and (\ref{defcores}).
\begin{rmk*}
Under the assumptions of Proposition \ref{f*leftadjoint}, 
the functor $f^*$ has
a left adjoint and the functor $g_!$ has a right adjoint. 
Thus by Remark \ref{rmkforadjointintexingbifr},
when $\ca{V}$ has and $A\otimes-$ preserves 
coequalizers for any monoid $A$, the fibration
$G$ is a bifibration. 
Dually, when $\ca{V}$ has and $-\otimes C$ preserves 
equalizers for any comonoid $C$, the opfibration
$V$ is a bifibration.
\end{rmk*}
If we unravel the Grothendieck 
construction of Theorem \ref{maintheoremfibr}, 
we have the following equivalent
characterization of, for example, $\Comod$:

$\cdot$ Objects are pairs $(X,C)$ with $C\in\Comon(\ca{V})$
and $X\in\Comod_\ca{V}(C)$.

$\cdot$ Morphisms are pairs $(k,g):(X,C)\to(Y,D)$ with
\begin{displaymath}
\begin{cases}
g_!X\xrightarrow{k}Y & \text{in }\Comod_\ca{V}(D)\\
\quad C\xrightarrow{g}D & \text{in }\Comon(\ca{V}).
\end{cases}
\end{displaymath}

$\cdot$ Composition $X_C\xrightarrow{(k,g)}Y_D\xrightarrow{(l,h)}Z_E$
is given by 
\begin{displaymath}
\begin{cases}
(hg)_!X\xrightarrow{\theta}Z & \text{in }\Comod_\ca{V}(E)\\
\qquad C\xrightarrow{hg}E & \text{in } \Comon(\ca{V})
\end{cases}
\end{displaymath}
where $\theta$ is the composite
$(hg)_!X=h_!g_!X\xrightarrow{h_!l}h_!Y\xrightarrow{k}Z$.

$\cdot$ The identity morphism is 
\begin{displaymath}
\begin{cases}
X\xrightarrow{1_X}X &\text{in }\Comod_\ca{V}(C)\\
C\xrightarrow{1_C}C &\text{in }\Comon(\ca{V})
\end{cases}
\end{displaymath}
since $(1_C)_!X=X$.

By comparing this with Definition \ref{defComod},
we deduce that $\Comod=\Gr{G}(\Comod_\ca{V})$
in a straightforward way. 
Dually $\Mod=\Gr{G}(\Mod_\ca{V})$,
so objects $M_A$ can be seen as pairs
$(M,A)$ with $A\in\Mon(\ca{V})$ and 
$M\in\Mod_\ca{V}(A)$, and morphisms $p_f$
as 
\begin{displaymath}
\begin{cases}
M\xrightarrow{p}f^*N & \text{in }\Mod_\ca{V}(A)\\
A\xrightarrow{f}B & \text{in }\Mon(\ca{V}).
\end{cases}
\end{displaymath}
Since these presentations of 
the global categories are essentially
the same, we can freely use 
the notation which
is more convenient depending on the case.
The fibre categories for $V=U_{\Comod_\ca{V}}$
and $G=P_{\Mod_\ca{V}}$ are respectively $\Comod_\ca{V}(C)$
and $\Mod_\ca{V}(A)$
and the canonical chosen cartesian and cocartesian
liftings are
\begin{align}\label{canoncart}
\Cart(f,N)&:f^*N\xrightarrow{(1_{f^*N},f)}N\textrm{ in }\Mod, \\
\Cocart(g,X)&:X\xrightarrow{(1_{g_!X},g)}g_!X\textrm{ in }\Comod. \notag
\end{align}
\begin{rmk}
 There is another way of viewing the global category of modules $\Mod$
for a monoidal category $\ca{V}$.
It is based on the observation 
that to give a lax functor of bicategories
$\ca{M}\ca{I}\to\ca{M}\ca{V}$
which is identity on objects is to give an object in $\Mod$.
I thank Steve Lack for explaining this point of 
view to me.

The bicategories are constructed
as in the Remark \ref{rmkpseudoaction}(i), arising from
the canonical actions of the monoidal categories $\ca{I},\;\ca{V}$
on themselves via tensor product. For the unit monoidal category,
we of course have that $\ca{M}\ca{I}(0,0)=\ca{M}\ca{I}(0,1)=
\ca{M}\ca{I}(1,1)=\B{1}$ and $\ca{M}\ca{I}(1,0)=\emptyset$.
Such an identity-on-objects lax functor $\ps{F}$ would in particular 
consist of functors
\begin{displaymath}
 \ps{F}_{0,1},\ps{F}_{1,1}:1\rightrightarrows\ca{V}
\end{displaymath}
which pick up two objects $M$ and $A$ in $\ca{V}$. The components of 
the natural transformations $\delta$ as in (\ref{delta}) give arrows 
$\mu:A\otimes M\to A$ and $m:A\otimes A\to A$ in $\ca{V}$,
the components of $\gamma$ as in (\ref{gamma}) give $\eta:I\to A$
and the axioms ensure that $(A,m,\eta)$ is a monoid in $\ca{V}$
and $(M,\mu)$ is an $A$-module.

Then, morphisms in $\Mod$ are icons, as described in 
Remark \ref{icons}: if $M_A$, $N_B$ are two identity-on-objects
lax functors between $\ca{M}\ca{I}$ and $\ca{M}\ca{V}$,
an icon between them consists in particular of natural
transformations 
\begin{displaymath}
 \xymatrix@C=.5in{\B{1}\rtwocell<\omit>{f}
\ar@/^3ex/[r]^-A\ar@/_3ex/[r]_-B & \ca{V}}\;\;\mathrm{and}\;\;
\xymatrix@C=.5in{\B{1}\rtwocell<\omit>{p}
\ar@/^3ex/[r]^-M\ar@/_3ex/[r]_-N & \ca{V}}
\end{displaymath}
which are two arrows $f:A\to B$ and 
$p:M\to N$ in $\ca{V}$, subject to conditions which
coincide with those of Definition \ref{defComod}.

Dually, colax natural transformations $\ca{M}\ca{I}\to\ca{M}\ca{V}$
correspond to comodules over comonoids, and icons then turn out 
to be comodule
morphisms. Therefore we have 
\begin{align*}
\Mod&=\B{Bicat}_2(\ca{M}\ca{I},\ca{M}\ca{V})_l \\
\Comod&=\B{Bicat}_2(\ca{M}\ca{I},\ca{M}\ca{V})_c
\end{align*}
where $\B{Bicat}_2$ is the 2-category of bicategories, lax/colax
functors and icons (see \cite{Icons}).
\end{rmk}
We now explore some of the main properties
of the global categories. First of all,
if $\ca{V}$ is a symmetric monoidal category, 
$\Comod$ and $\Mod$ are symmetric monoidal categories
as well. It is easy to verify that 
if $s$ is the symmetry in $\ca{V}$,
the object $X_C\otimes Y_D$ 
in $\ca{V}$ for $X_C,Y_D\in\Comod$ is a comodule over 
the comonoid $C\otimes D$ via the coaction
\begin{equation}\label{coactioncomod}
X\otimes Y\xrightarrow{\delta_X\otimes\delta_Y}
X\otimes C\otimes Y\otimes D
\xrightarrow{1\otimes s\otimes1} X\otimes Y\otimes C\otimes D. 
\end{equation}
The fact that $\Comon(\ca{V})$ is monoidal 
itself is evidently required, which holds again due to 
symmetry of $\ca{V}$.
Notice that there is no appropriate way
of endowing the fibre categories
$\Comod_\ca{V}(C)$ with a monoidal structure
in general, 
since for example, the tensor product
of two $C$-comodules would end up
as a $C\otimes C$-module by the above.
Similarly, for $M_A,N_B\in\Mod$, the object $M_A\otimes N_B$
is a $A\otimes B$-module via the action
\begin{displaymath}
 A\otimes B\otimes M\otimes N\xrightarrow{1\otimes s\otimes1}
A\otimes M\otimes B\otimes N\xrightarrow{\mu_M\otimes\mu_N}
M\otimes N.
\end{displaymath}
The symmetry of $\Mod$ and $\Comod$ is inherited from $\ca{V}$.
Moreover, in this case the functors 
$V$ and $G$ of (\ref{forgetGV})
have the structure of a strict symmetric monoidal
functor:
\begin{align}\label{strictmonoidalVG}
V(X_C\otimes Y_D)&=C\otimes D=VX_C\otimes VY_D \\
G(M_A\otimes N_B)&=A\otimes B=GM_A\otimes GN_B.\notag
\end{align}
The monoidal unit in both cases is $I$,
with a trivial $I$-action and coaction via $r_I$.

The following result, also mentioned
at the end of \cite{LinearReps} for
$\ca{V}=\Mod_R$, illustrates the structure
of the global categories.
\begin{prop}\label{Comodcomonadic}
The functor $F:\Comod\to\ca{V}
\times\Comon(\ca{V})$ which maps an object
$X_C$ to the pair $(X,C)$ is comonadic.
\end{prop}
\begin{proof}
First notice that $F$ `consists
of' the forgetful functor which discards the comodule
structure from the object of $\ca{V}$, and the
forgetful $V$ which keeps the comonoid. 
Therefore this result
is, in a sense, a generalization
of Proposition \ref{modulesmonadic}.

Define a functor
\begin{displaymath}
R:\xymatrix @R=.02in @C=.4in
{\ca{V}\times
\Comon(\ca{V})\ar[r] & \Comod \\
(A,D)\ar@{|.>}[r]\ar[dd]_-{(l,g)} & 
(A\otimes D)_D \ar[dd]^-{(l\otimes g)_g} \\
\hole \\
(B,E)\ar@{|.>}[r] & (B\otimes E)_E}
\end{displaymath}
where the $D$-action
on the object $A\otimes D$ is
given by
$A\otimes D\xrightarrow{1\otimes\Delta}A\otimes D\otimes D$
with $\Delta$ the comultiplication
of the comonoid $D$. It is not hard
to establish a 
natural bijection
\begin{align*}
(\ca{V}\times\Comon(\ca{V}))((X,C),(A,D))
&\cong\ca{V}(X,A)\times\Comon(\ca{V})(C,D) \\
&\cong\Comod(X_C,(A\otimes D)_D)
\end{align*}
where $(X,C)=F(X_C)$ and $(A\otimes D)_D=R(A,D)$,
so we obtain an adjunction
\begin{displaymath}
\xymatrix @C=.5in
{\Comod\ar @<+.8ex>[r]^-F
\ar@{}[r]|-\bot & \ca{V}\times\Comon(\ca{V}).
\ar @<+.8ex>[l]^-R}
\end{displaymath}
This induces a comonad on $\ca{V}\times\Comon(\ca{V})$,
namely $(FR,F\eta_R,\varepsilon)$, where the comultiplication
and counit have components
\begin{align*}
F\eta_{K(A,D)}&:(A\otimes D,D)\xrightarrow{(1\otimes\Delta,1)}
(A\otimes D\otimes D,D) \\
\varepsilon_{(A,D)}&:(A\otimes D,D)\xrightarrow{(1\otimes\epsilon,1)}
(A,D)
\end{align*}
for the comonoid $(D,\Delta,\epsilon)$.
The category of coalgebras for this comonad
is precisely $\Comod$.
\end{proof} 
This in particular implies that if 
$\ca{V}$ and $\Comon(\ca{V})$ are cocomplete
categories, then $\Comod$ is also cocomplete.
In fact, using results from Section \ref{fibredadjunctions}
concerning fibrewise colimits, we can recover 
this as follows.
\begin{cor}\label{Comodcocomplete}
If $\ca{V}$ and $\Comon(\ca{V})$ have all colimits,
then $\Comod$ has all colimits and 
$V:\Comod\to\Comon(\ca{V})$
strictly preserves them. 
\end{cor}
\begin{proof}
Since every fibre $\Comod_\ca{V}(C)$
of the opfibration $V$ is comonadic over $\ca{V}$, 
it has all colimits for any comonoid $C$. 
Moreover, the reindexing functors
$\Comod_\ca{V}(g)=g_!$ preserve all colimits
by the commutative diagram (\ref{tr2}) for any comonoid
arrow $g$. By Proposition \ref{reindexcont}, 
the opfibration $V:\Comod\to\Comon(\ca{V})$
has all opfibred colimits. Then, by the dual of Corollary
\ref{AhasPpreserves}, this is equivalent to
the total category $\Comod$ being cocomplete and 
$V$ being strictly cocontinuous. 
\end{proof}
Colimits in $\Comod$ are therefore constructed as follows.
If we consider a diagram $D:\ca{J}\to\Comod$, 
the composite functor
\begin{displaymath}
 \ca{J}\xrightarrow{D}\Comod\xrightarrow{V}\Comon(\ca{V})
\end{displaymath}
has a colimiting cocone
$(\tau_j:VD_j\to\colim(VD)\,|\,j\in\ca{J})$
since $\Comon(\ca{V})$ is cocomplete.
Define a new diagram
\begin{displaymath}
\xymatrix @R=.02in @C=.4in
{\ca{J}\ar[r]^-H & \Comod_\ca{V}(\colim VD) \\
 j\ar@{|.>}[r]\ar[dd]_-{\kappa} & 
(\tau_j)_!Dj=(\tau_{j'})_!(VD\kappa)_!Dj
\ar[dd]^-{(\tau_{j'})_!D\kappa} \\
\hole \\
j'\ar@{|.>}[r] & (\tau_{j'})_!Dj'}
\end{displaymath}
which, since the category 
$\Comod_\ca{V}(\colim VD)$ is cocomplete,
also has a colimiting cocone
$(\sigma_j:(\tau_j)_!Dj\to\colim H\,|\,j\in\ca{J})$.
It turns out that
\begin{displaymath}
\big(D_j\xrightarrow{\;(\sigma_j,\tau_j)\;}\colim H\,|\,j\in\ca{J}\big)
\end{displaymath}
is the colimiting cocone
of $D$ in $\Comod$, and of course 
$V\colim D=\colim(VD)$.

Dually to the above results, we obtain the following.
\begin{prop}\label{Modmonadic}
 The global category of modules $\Mod$ is monadic over
the category $\ca{V}\times\Mon(\ca{V})$, and so if
$\ca{V}$ and $\Mon(\ca{V})$ are complete, $\Mod$ has all limits
and $G:\Mod\to\Mon(\ca{V})$ strictly preserves them.
\end{prop}
Now suppose that
$\ca{V}$ is a symmetric monoidal closed category.
In Section \ref{Categoriesofmodulesandcomodules}
it was explained how the internal hom bifunctor
induces a functor
\begin{displaymath}
 \Mod_{CA}[-,-]:\Comod_\ca{V}(C)^\op\times\Mod_\ca{V}(A)
\longrightarrow\Mod_\ca{V}([C,A])
\end{displaymath}
as in (\ref{defMod[]}),
which is again the restriction of the internal hom
on the cartesian product of the categories of 
$C$-comodules and $A$-modules.
There is a way to lift this functor on the level
of the global categories of comodules and modules,
in the sense that there is a functor 
between the total categories
\begin{equation}\label{defbarH}
\bar{H}:
\xymatrix @R=0.02in
{\Comod^\op\times\Mod\ar[r] & \Mod\qquad \\
\qquad(\;X_C\;,\;M_A\;)\ar @{|->}[r] & 
[X,M]_{[C,A]}}
\end{equation}
such that $\Mod_{CA}[-,-]$ are the functors induced 
between the fibres (see Remark \ref{functorsbetweenfibres}).
If $(k_g,l_f):(X_C,M_A)\to(Y_D,N_B)$ is a morphism
in the cartesian product, the fact that
$k$ and $l$ commute with the corestricted and
restricted actions accordingly forces the arrow 
$[k,l]:[X,M]\to[Y,N]$ in $\ca{V}$ to satisfy the appropriate
property. Hence
\begin{displaymath}
\begin{cases}
[X,M]\xrightarrow{[k,l]}[g,f]^*[Y,N] & \text{in }\Mod_\ca{V}[C,A]\\
[C,A]\xrightarrow{[g,f]}[D,B] & \text{in }\Mon(\ca{V})
\end{cases}
\end{displaymath}
defines an arrow
$\bar{H}(k,l)_{[g,f]}:[X,M]_{[C,A]}\to[Y,N]_{[D,B]}$
in $\Mod$. In fact,
the pair $(\bar{H},H)$ is a 
fibred 1-cell depicted by the
square
\begin{equation}\label{barHHfibred1cell}
 \xymatrix @C=.6in
{\Comod^\op\times\Mod\ar[r]^-{\bar{H}(-,-)}
\ar[d]_-{V^\op\times G} & \Mod\ar[d]^-G \\
\Comon(\ca{V})^\op\times\Mon(\ca{V})\ar[r]_-{H(-,-)} 
&\Mon(\ca{V}),}
\end{equation}
where of course the cartesian product 
$V^\op\times G$ is treated as a fibration.
Commutativity is clear from the above construction,
which ensures that
\begin{displaymath}
 G([X,N]_{[C,B]})=[VX_C,GN_B]=[C,B].
\end{displaymath}
Moreover $\bar{H}$ is a cartesian functor:
it maps a cartesian arrow of the domain, 
which is a pair of a cocartesian lifting in $\Comod$ 
and a cartesian lifting in $\Mod$,
to the arrow
\begin{displaymath}
\xymatrix @C=.3in @R=.4in
{[g_!Y,f^*N]\ar[rrrr]^-{\bar{H}(\Cocart(g,Y),\Cart(f,N))}
\ar @{.>}[d] &&&& [Y,N] \ar @{.>}[d] &\textrm{in }\Mod \\
[C,A]\ar[rrrr]_-{[g,f]} &&&& [D,B] & \textrm{in }\Mon(\ca{V}).}
\end{displaymath}
By the canonical liftings (\ref{canoncart})
from the Grothendieck construction, that module
arrow is specifically
\begin{displaymath}
\bar{H}((1_{g_!Y},g),(1_{f^*N},f))=([1_{g_!Y},1_{f^*N}],[g,f])=(1_{[g_!Y,f^*N]},[g,f])
\end{displaymath}
by the definition of $\bar{H}$ and 
functoriality of $[-,-]$. On the other hand, 
the canonical cartesian lifting of $[Y,N]$ along $[g,f]$ is
\begin{displaymath}
\xymatrix @C=.5in @R=.4in
{[g,f]^*[Y,N]\ar[rr]^-{(1_{[g,f]^*[Y,N]},[g,f])}
\ar @{.>}[d] && [Y,N] \ar @{.>}[d] &\textrm{in }\Mod \\
[C,A]\ar[rr]_-{[g,f]} && [D,B] & \textrm{in }\Mon(\ca{V}).}
\end{displaymath} 
The above two arrows
in $\Mod$ are essentially identical, both being
$1_{[Y,N]}:[Y,N]\to[Y,N]$ as
morphisms in $\ca{V}$ between the modules,
and the $[C,A]$-actions on $[g_!Y,f^*N]$ and
$[g,f]^*[Y,N]$ can be computed to be the same. 
Therefore $(\bar{H},H)$ is actually a split 
fibred 1-cell.

Finally, suppose $\ca{V}$ is monoidal such that $\otimes$
preserves (filtered) colimits on both sides,
and moreover locally presentable. It is not hard to see that the comonad 
on $\ca{V}\times\Comon(\ca{V})$ whose category of coalgebras 
is $\Comod$ (see Proposition \ref{Comodcomonadic}) is finitary: if 
$(\lambda_j,\tau_j):(X_j,C_j)\to(X,C)$ is a filtered colimiting
cocone, then 
\begin{displaymath}
(\lambda_j\otimes\tau_j,\tau_j):(X_j\otimes C_j,C_j)\longrightarrow
(X\otimes C,C)
\end{displaymath}
is too, since $\otimes$ preserves
colimits on both variables and $\Comon(\ca{V})$ is 
comonadic over $\ca{V}$. Dually, $\Mod$ is finitary monadic
over $\ca{V}\times\Mon(\ca{V})$, since  
$(\lambda_j\otimes\tau_j,\tau_j):(A_j\otimes M_j,A_j)\to
(A\otimes M,A)$ is a filtered colimit when $\lambda_j$ is a colimiting
cocone in $\ca{V}$ and $\tau_j$ in $\Mon(\ca{V})$. This happens 
because the monadic $\Mon(\ca{V}\to\ca{V}$ creates all 
colimits that the finitary monad preserves
(see Proposition \ref{moncomonadm}(1)). 
Since $\ca{V}$, $\Mon(\ca{V})$ and $\Comon(\ca{V})$
are all locally presentable categories under the above assumptions,
we can apply Theorem \ref{Monadiccomonadicpresentability}
for the global categories.
\begin{thm}\label{ModComodlp}
If $\ca{V}$ is a locally presentable monoidal category such that
$(-\otimes-)$ is finitary on both entries,
$\Mod$ and $\Comod$ are locally presentable.
\end{thm}

\section{Universal measuring comodule and enrichment}\label{Universalmeasuringcomodule}

The notion of a universal measuring comodule
in the category of vector spaces $\Vect_k$
was first introduced by Batchelor in \cite{Batchelor},
where emphasis was given to its applications.
Very similarly to the context of measuring coalgebras, a $k$-linear
map $\psi:X\to\Hom_k(M,N)$ is said to \emph{measure}
if it satisfies
\begin{displaymath}
\psi(x)(am)=\sum_{(x)}{\phi x_{(1)}(a)\psi x_{(0)}(m)}
\end{displaymath}
again using sigma notation. Here
$X$ is a $C$-comodule,
$M$ an $A$-module and $N$ a $B$-module, for $(C,\phi)$ a 
measuring coalgebra
and $A$,$B$ algebras. The pair $(X,\psi)$ is called
\emph{measuring comodule}. The question that gave rise
to this definition is whether the transpose arrow
$\bar{\psi}:M\to\Hom_k(X,N)$ is a map of $A$-modules,
using the symmetry in $\Vect_k$ and the module
structure on $\Hom_k(X,N)$.

There is a category of measuring comodules
for a fixed measuring coalgebra $C$, and it 
has a terminal object $Q(M,N)$ satisfying the property
that there is a correspondence
\begin{equation}\label{meascomodMarj}
\big\{ C\textrm{-comodule maps}\;X\to Q(M,N) \big\}\leftrightarrow
\big\{ A\textrm{-module maps}\;M\to\Hom_k(X,N) \big\}.
\end{equation}
The object $Q(M,N)$ is called \emph{universal measuring comodule}.
Initially, the goal is to extend the 
existence of the universal measuring comodule
in a more general context than $\Vect_k$.

Consider a symmetric monoidal closed category
$\ca{V}$. In the end of the previous section,
we defined a functor of two variables
$\bar{H}:\Comod^\op\times\Mod\to\Mod$
which maps a comodule and a module to 
their internal hom in $\ca{V}$.
Since the aim is a generalization of the 
correspondence (\ref{meascomodMarj}) in order to define the 
universal measuring comodule,
in fact we need a natural isomorphism 
\begin{displaymath}
\Comod(X,Q(M,N))\cong\Mod(M,\bar{H}(X,N))
\end{displaymath}
where $X=X_C$, $M=M_A$, $N=N_B$ and $\bar{H}(X,N)=[X,N]_{[C,B]}$.
Thus it is enough to show that
the functor $\bar{H}(-,N_B)^\op:\Comod\longrightarrow\Mod^\op$
for a fixed $B$-module $N$ has a right adjoint. 

Moreover, we intend to show that $Q(M,N)$
is a comodule over the universal measuring coalgebra, 
hence the assumptions on $\ca{V}$ have to also cover
the existence of $P(A,B)$.
The following result is an application 
of Theorem \ref{totaladjointthm} in the abstract setting
of (op)fibrations. A direct proof can
be found at the end of this chapter.
\begin{prop}\label{propmeasuringcomodule}
 Let $\ca{V}$ be a locally presentable
symmetric monoidal closed category. There is an adjunction
\begin{displaymath}
\xymatrix @C=.6in
{\Comod \ar @<+.8ex>[r]^-{\bar{H}(-,N_B)^\op} 
 \ar@{}[r]|-{\bot} &
\Mod^\op\ar @<+.8ex>[l]^-{Q(-,N_B)}}
\end{displaymath}
between the global categories of modules and comodules,
with a natural isomorphism
\begin{equation}\label{meascomod} 
\Comod(X_C,Q(M,N)_{P(A,B)})\cong\Mod(M_A,[X,N]_{[C,B]}).
\end{equation}
\end{prop}
\begin{proof}
The pair of bifunctors
$(\bar{H},H)$ depicted as (\ref{barHHfibred1cell})
constitutes a fibred 1-cell between the fibrations
$V^\op\times G$ and $G$, as shown earlier. This implies
that the pair of functors 
$(\bar{H}(-,N_B),H(-,B))$ for a fixed monoid
$B$ and a $B$-module $N$ is again a fibred 1-cell
between $V^\op$ and $G$, and hence the opposite
square
\begin{displaymath}
 \xymatrix @C=.8in @R=.6in
{\Comod\ar[r]^-{\bar{H}(-,N_B)^\op} \ar[d]_-V &
\Mod^\op\ar[d]^-{G^\op} \\
\Comon(\ca{V})\ar[r]_-{H(-,B)^\op} & \Mon(\ca{V})^\op}
\end{displaymath}
is an opfibred 1-cell between the opfibrations
$V$ and $G^\op$. Also, by Proposition 
\ref{measuringcomonoidprop} there is 
an adjunction between the base categories
\begin{displaymath}
 \xymatrix @C=.6in
{\Comon(\ca{V}) \ar @<+.8ex>[r]^-{H(-,B)^\op} 
 \ar@{}[r]|-{\bot} &
\Mon(\ca{V})^\op\ar @<+.8ex>[l]^-{P(-,B)}}
\end{displaymath}
where $P$ is the Sweedler hom functor. 

In order for Lemma \ref{totaladjointlem}
to apply, we need the existence of a right adjoint
of the composite functor
\begin{equation}\label{specialcomposite}
 \Comod_\ca{V}(P(A,B))\xrightarrow{\bar{H}(-,N_B)^\op_{P(A,B)}}
\Mod_\ca{V}^\op([P(A,B),B])\xrightarrow{(\varepsilon_A)_!}
\Mod_\ca{V}^\op(A)
\end{equation}
where 
\begin{displaymath}
\varepsilon_A^B:H(P(A,B),B)\to A \quad\textrm{ in }\quad\Mon(\ca{V})^\op
\end{displaymath}
are the components
of the counit of the parametrized
adjunction $H^\op\dashv P$.
We already know that $\Comod_\ca{V}(C)$ is a locally
presentable category by Proposition
\ref{comodlocpresent}, so cocomplete
with a small dense subcategory, namely
the presentable objects.
Moreover, the reindexing functors are always
cocontinuous as seen in Section 
\ref{Categoriesofmodulesandcomodules},
hence so is $(\varepsilon_A)_!$ of the opfibration
$V^\op$. Finally,
the following commutative diagram
\begin{equation}\label{cocontbarHN}
\xymatrix @C=1in @R=.5in
{\Comod\ar[r]^-
{\bar{H}(-,N_B)^\op}
\ar[d] & \Mod^\op\ar[d]\\
\ca{V}\times\Comon(\ca{V})
\ar[r]_-{[-,N]^\op\times
H(-,B)^\op} & \ca{V}^\op\times\Mon(\ca{V})^\op}
\end{equation} 
implies that $\bar{H}(-,N_B)^\op$ 
preserves all colimits: both functors at the bottom
have right adjoints, and the vertical functors 
create all colimits
by Propositions \ref{Comodcomonadic} and
\ref{Modmonadic}.
Since the fibres of the total categories
$\Comod$ and $\Mod^\op$ are closed under colimits,
the restricted fibrewise functor $\bar{H}(-,N_B)^\op_{P(A,B)}$ 
is cocontinuous too.

Consequently, by Theorem \ref{Kelly}
the composite (\ref{specialcomposite})
has a `fibrewise' right adjoint
\begin{displaymath}
 Q_A(-,N_B):\Mod_\ca{V}(A)^\op\longrightarrow
\Comod_\ca{V}(P(A,B))
\end{displaymath}
and Theorem \ref{totaladjointthm} implies
that this lifts to a functor between 
the total categories $Q(-,N_B):\Mod^\op\longrightarrow\Comod$
such that 
\begin{displaymath}
\xymatrix @R=.6in @C=.8in 
{\Comod \ar @<+.8ex>[r]^-{\bar{H}(-,N_B)^\op}
\ar@{}[r]|-\bot
\ar[d]_-V &
\Mod^\op\ar @<+.8ex>[l]^-{Q(-,N_B)}
\ar[d]^-{G^\op} \\ 
\Comon(\ca{V}) \ar @<+.8ex>[r]^-{H(-,B)^\op} 
\ar@{}[r]|-\bot &
\Mon(\ca{V})^\op \ar @<+.8ex>[l]^-{P(-,B)}}
\end{displaymath}
is an adjunction in $\B{Cat}^\B{2}$.
The isomorphism (\ref{meascomod})
for the adjunction between the total categories, natural in $X_C$ and $M_A$,
makes this adjoint uniquely into a functor of two 
variables
\begin{displaymath}
 Q(-,-):\Mod^\op\times\Mod\longrightarrow\Comod
\end{displaymath}
such that the isomorphism is natural
in all three variables. In other words, 
$Q$ is the parametrized adjoint of 
the bifunctor $\bar{H}^\op$.
\end{proof}
The bifunctor $Q$ is called
the \emph{universal measuring comodule functor}.
By construction of $Q$, the object
$Q(M_A,N_B)$ has the structure of a $P(A,B)$-comodule.

Similarly, we can show that the symmetric
monoidal category $\Comod$ has
a monoidal closed structure.
\begin{prop}\label{Comodclosed}
 The global category of comodules $\Comod$
 for a locally presentable symmetric monoidal
 closed category $\ca{V}$ is a symmetric
 monoidal closed category.
\end{prop}
\begin{proof}
By the definition of the symmetric monoidal tensor product
\begin{displaymath}
\otimes:
\xymatrix @R=.05in
{\Comod\times\Comod\ar[r] & \Comod\qquad \\
(\;Y_D\;,\;X_C\;)\ar @{|->}[r] & (Y\otimes X)_{D\otimes C}}
\end{displaymath}
in $\Comod$ as in (\ref{coactioncomod}),
we have a commutative square
\begin{equation}\label{otimesopfibred1cell}
\xymatrix @C=.5in @R=.5in
{\Comod\ar[r]^-{(-\otimes X_C)} \ar[d]_-V &
\Comod\ar[d]^-V \\
\Comon(\ca{V})\ar[r]_-{(-\otimes C)} & \Comon(\ca{V}).}
\end{equation}
Actually, this is an 
opfibred 1-cell: the functor $(-\otimes X_C)$
for a fixed $C$-comodule $X$ maps 
a cocartesian lifting to the right top arrow 
\begin{displaymath}
\xymatrix @C=.2in @R=.15in
{Y\ar[rr]^-{\Cocart(f,Y)}
\ar@{.}[dd] && f_!Y\ar@{.}[dd] &&
Y\otimes X\ar[rr]^-{\Cocart(f,Y)\otimes1}
\ar[drr]_-{\Cocart}\ar@{.}[dd] && f_!Y\otimes X 
& \textrm{in }\Comod \\
&&& {\mapsto} &&& (f\otimes1)_!(Y\otimes X)\ar@{-->}[u]_-{\exists!}
\ar@{.}[d]\\
D\ar[rr]_-f && E && D\otimes C\ar[rr]_-{f\otimes1} && E\otimes C &
\textrm{in }\Comon(\ca{V}).}
\end{displaymath}
The two $(E\otimes C)$-comodules $f_!Y\otimes X$ 
and $(f\otimes1)_!(Y\otimes X)$
are both $Y\otimes X$ as objects in $\ca{V}$, 
and the coactions
induced in both cases are equal. Hence, 
by the canonical choice
of cocartesian liftings
for the opfibration $V:\Comod\to\Comon(\ca{V})$ and
functoriality of the tensor product,
$1_{(f\otimes1)_!(Y\otimes X)}=1_{f_!Y}\otimes1_X$
and so $(-\otimes X_C)$ is a cocartesian functor.

By Proposition \ref{Comonclosed}, the category of comonoids
for such a monoidal category $\ca{V}$ is monoidal
closed with internal hom functor $\HOM$ via an 
adjunction 
\begin{displaymath}
 \xymatrix @C=.65in
 {\Comon(\ca{V})\ar@<+.8ex>[r]^-{(-\otimes C)}
 \ar@{}[r]|-\bot & \Comon(\ca{V})
 \ar@<+.8ex>[l]^-{\HOM(C,-)}}
\end{displaymath}
between the bases of (\ref{otimesopfibred1cell}).
Finally, if $\varepsilon$ is the counit for 
this adjunction, the composite functor
\begin{displaymath}
 \Comod_\ca{V}(\HOM(C,D))\xrightarrow{(-\otimes X_C)}
 \Comod_\ca{V}(\HOM(C,D)\otimes C)\xrightarrow{(\varepsilon_D)_!}
 \Comod_\ca{V}(D)
\end{displaymath} 
has a right adjoint $\overline{\HOM}_D(X_C,-)$.
This follows from the adjoint functor Theorem 
\ref{Kelly}, since $\Comod_\ca{V}(\HOM(C,D))$
is locally presentable and the composite functor
preserves all colimits. This is the case 
because reindexing functors are always cocontinuous, and the 
commutative diagram
\begin{equation}\label{cocontotimesX}
 \xymatrix @C=.85in
{\Comod\ar[r]^-{(-\otimes X_C)}
\ar[d]_-F & \Comod\ar[d]^-F\\
\ca{V}\times\Comon(\ca{V})
\ar[r]^-{(-\otimes X)\times(-\otimes C)} & 
\ca{V}\times\Comon(\ca{V})}
\end{equation}
implies that $(-\otimes X_C)$
preserves all colimits, since the bottom arrow 
does by monoidal closedness of $\ca{V}$ and $\Comon(\ca{V})$,
and $F$ is comonadic.

By Theorem \ref{totaladjointthm}, the
functors $\overline{\HOM}_D(X_C,-)$ between the fibres
assemble into a total adjoint
$\overline{\HOM}(X_C,-):\Comod\to\Comod$
such that
\begin{displaymath}
\xymatrix @R=.5in @C=.8in 
{\Comod \ar @<+.8ex>[r]^-{-\otimes X_C}
\ar@{}[r]|-\bot \ar[d]_-V &
\Comod\ar @<+.8ex>[l]^-{\overline{\HOM}(X_C,-)}
\ar[d]^-V \\ 
\Comon(\ca{V}) \ar @<+.8ex>[r]^-{-\otimes C} 
\ar@{}[r]|-\bot &
\Comon(\ca{V})\ar @<+.8ex>[l]^-{\HOM(C,-)}}
\end{displaymath}
is an adjunction in $\B{Cat}^\B{2}$.
Thus the uniquely defined parametrized adjoint
\begin{equation}\label{defbarHOM}
 \overline{\HOM}:\Comod^\op\times\Comod\longrightarrow\Comod
\end{equation}
of $(-\otimes-)$ is the internal hom
of the global category of comodules $\Comod$.
\end{proof}
\begin{rmk}
An alternative approach for the existence
of the functors $Q$ and $\overline{\HOM}$ 
would be to show that
\begin{align*}
 \bar{H}^\op(-,N_B)&:\Comod\longrightarrow\Mod^\op \\
-\otimes X_C&:\Comod\longrightarrow\Comod
\end{align*}
have right adjoints via an adjoint functor theorem. 
Both functors are cocontinuous 
by diagrams (\ref{cocontbarHN}) and (\ref{cocontotimesX})
respectively, and the domain $\Comod$ is locally presentable
by Theorem \ref{ModComodlp}. Hence Theorem \ref{Kelly}
directly establishes the existence of right adjoints. However, we prefer 
the method which employs the fibrational structure
of the global categories, because it provides with a better 
understanding of the situation. 
For example, the above proposition ensures that 
$\overline{\Hom}(X_C,Y_D)$ is specifically a $\HOM(C,D)$-comodule. 
\end{rmk}
We can now once more combine
the existence of the universal measuring comodule
with the theory of actions of monoidal
categories, in order to show how the functor $Q$ induces
an enrichment of the global category
of modules in the global category of comodules.

For any symmetric monoidal closed category $\ca{V}$,
the functor of two variables
$\bar{H}(-,-):\Comod^\op\times\Mod\longrightarrow\Mod$
defined as in (\ref{defbarH}) is in fact an action of the
symmetric monoidal category $\Comod^\op$ on the 
ordinary category $\Mod$. It is easy to see that there 
exist natural isomorphisms 
\begin{align*}
[X\otimes Y,M]_{[C\otimes D,A]}&\stackrel{\sim}
{\longrightarrow}[X,[Y,M]]_{[C,[D,A]]} \\
[I,M]_{[I,A]}&\stackrel{\sim}{\longrightarrow}M_A\qquad
\end{align*}
for any coalgebras $C$, $D$, algebras $A$, comodules $X_C$, $Y_D$
and modules $M_A$ that satisfy the 
axioms of an action. This follows from the facts
that $[-,-]$ and $H(-,-)$ are actions and 
the monadic functor $\Mod\to\ca{V}\times\Mon(\ca{V})$
reflects isomorphisms.
Therefore the opposite functor
\begin{equation}\label{barHopaction}
 \bar{H}^\op:\Comod\times\Mod^\op\longrightarrow\Mod^\op
\end{equation}
is an action of the symmetric monoidal
$\Comod$ on $\Mod^\op$.

Since we have an adjunction
$\bar{H}(-,N_B)^\op\dashv Q(-,N_B)$
for any module $N_B$ by Proposition \ref{propmeasuringcomodule},
Corollaries \ref{importcor1} and 
\ref{importcor2} apply
and give the following result.
\begin{thm}\label{ModenrichedinComod}
Let $\ca{V}$ be a locally  presentable
symmetric monoidal closed category and $Q$
the universal measuring comodule functor. 
\begin{enumerate}
\item The opposite of the global
category of comodules $\Mod^\op$ is enriched
in the global category of comodules $\Comod$,
with hom-objects 
\begin{displaymath}
\Mod^\op(M_A,N_B)=Q(N,M)_{P(B,A)}
\end{displaymath}
where the $\Comod$-enriched category is denoted
with the same name.
\item The global category of modules
$\Mod$ is a tensored and cotensored $\Comod$-enriched
category,
with hom-objects 
\begin{displaymath}
\Mod(M_A,N_B)=Q(M,N)_{P(A,B)}
\end{displaymath}
and cotensor products $[X,N]_{[C,B]}$ for any
$C$-comodule $X$ and $B$-module $N$.
\end{enumerate}
\end{thm}
\begin{proof}
The only part left to show is that the functor
\begin{displaymath}
\bar{H}(X_C,-)^\op:\Mon(\ca{V})^\op\longrightarrow\Mon(\ca{V})^\op 
\end{displaymath}
has a right adjoint for every comodule $X_C$. Consider 
the commutative square
\begin{displaymath}
 \xymatrix @C=.8in
{\Mod\ar[r]^-{\bar{H}(X_C,-)}\ar[d] & \Mod\ar[d] \\
\ca{V}\times\Mon(\ca{V})\ar[r]_-{[X,-]\times H(C,-)} & \ca{V}\times\Mon(\ca{V})}
\end{displaymath}
where the vertical functors are monadic, $\Mod$ is locally 
presentable by Theorem \ref{ModComodlp}, $[X,-]\vdash(-\otimes X)$ in 
$\ca{V}$ and  
$H(C,-)$ has a left adjoint as in (\ref{deftriangle}).
By Dubuc's Adjoint Triangle Theorem, the top functor
has a left adjoint $X_C\,\overline{\triangleright}-$ for all 
$X_C$'s, inducing a bifunctor
\begin{displaymath}
\overline{\triangleright}:\Comod\times\Mod\longrightarrow\Mod
\end{displaymath}
which gives the tensor products of the $\Comod$-enriched
category $\Mod$.
\end{proof}
We finish this chapter by giving
a direct proof of 
Proposition \ref{propmeasuringcomodule}, which 
can also be found in \cite{mine}. We should note 
here that the proof of the more general Theorem 
\ref{totaladjointthm} in the context of opfibrations
actually relied heavily on this special
case of modules and comodules. These objects' nature
and the effect of the well-behaved
reindexing functors on them illustrate the 
correspondences between the 
hom-sets clearly and give insight 
for the generalized result.
\begin{proof}[Proof 2]
Suppose that $\ca{V}$ is a locally
presentable monoidal closed category,
$P$ is the Sweedler hom as in (\ref{defP})
and $\bar{H}$ is the restricted internal hom
between the global categories as in (\ref{defbarH}).
We are going to explicitly establish a 
bijective correspondence
\begin{equation}\label{comodold} 
\Comod(X,Q_A(M,N))\cong\Mod(M,\bar{H}(X,N))
\end{equation} 
for any $C$-comodule $X$, $A$-module $M$ and 
$B$-module $N$.
The object $Q_{(-)}(M,N)$ arises once more 
from the existence of a `special case adjunction'
\begin{displaymath}
\xymatrix @C=.6in @R=.05in
{& \Mod_\ca{V}^\op([P(A,B),B])\ar@/^/[dr]^-{{\varepsilon_A}_!} & \\
\Comod_\ca{V}(P(A,B))\ar@/^/[ur]^-{\bar{H}(-,N_B)^\op}
\ar@{}[rr]|-{\bot} &&
\Mod_\ca{V}^\op(A)\phantom{ABCD}\ar@/^5ex/[ll]^-{Q_A(-,N_B)}}
\end{displaymath}
with a natural isomorphism for $Z$ a $P(A,B)$-comodule
\begin{equation}\label{specialadj}
\big(\Comod_\ca{V}(P(A,B))\big)(Z,Q_A(M,N))\cong
\big(\Mod_\ca{V}(A)\big)(M,(\varepsilon_A)^*[Z,N]).
\end{equation}
An arbitrary element of
$\Comod(X_C,Q_A(M,N)_{P(A,B)})$
\begin{displaymath}
\begin{cases}
h_!X\xrightarrow{k}Q_A(M,N) &\text{in }\Comod_\ca{V}(P(A,B))\\
\quad C\xrightarrow{h}P(A,B) &\text{in }\Comon(\ca{V})
\end{cases} 
\end{displaymath}
corresponds uniquely to a pair of arrows
\begin{equation}\label{thispair}
\begin{cases}
M\xrightarrow{t}(\varepsilon_A)^*[h_!X,N] &\text{in }\Mod_\ca{V}(A)\\
 A\xrightarrow{\tilde{h}}[C,B] &\text{in }\Mon(\ca{V})
\end{cases} 
\end{equation}
as follows:
the top one is obtained via the 
special case adjunction (\ref{specialadj}) 
since $h_!X$ is a $P(A,B)$-comodule, and the bottom one via 
the adjunction (\ref{meascomon}).
Here, $(\varepsilon_A)^*[h_!X,N]$ is an $A$-module
via the induced $A$-action on $[X,N]$
\begin{displaymath} 
\xymatrix @R=.55in @C=.5in 
{A\otimes
[X,N]\ar[r]^-{\varepsilon_A\otimes1}\ar @{-->}@/_3ex/[drr] &
[P(A,B),B]\otimes [X,N]\ar[r]^-{[h,1]\otimes 1}
& [C,B]\otimes [X,N]\ar[d]^-{\mu} \\ 
&& [X,N]}
\end{displaymath} 
where $\mu$ is the canonical $[C,B]$-action on
$[X_C,N_B]$ given by (\ref{lala}).
By definition of the global category $\Mod$,
$t$ is a morphism $M\to[X,N]$ in $\ca{V}$ 
which is compatible with the respective $A$-actions.
Thus the diagram (\ref{defmod2}) which it has to satisfy
corresponds under the adjunction $(-\otimes X)\dashv[X,-]$ to 
\begin{equation}\label{impdiag1} 
\xymatrix @C=0.05in @R=0.2in 
{A\otimes M\otimes X \ar[dd]_-{\mu\otimes1}
\ar[rr]^-{\scriptscriptstyle{1\otimes t\otimes
\delta}} && A\otimes [X,N]\otimes X\otimes C
\ar[r]^-{\scriptscriptstyle{\varepsilon\otimes1\otimes1}}
\ar @{-->} @/_9ex/[dddr]^(.3){(*)} & 
[P(A,B),B]\otimes [X,N]\otimes X\otimes
C\ar[d]^-{\scriptscriptstyle{1\otimes h}}\\ 
&&& [P(A,B),B]\otimes
[X,N]\otimes X\otimes P(A,B)\qquad\ar[d]^-
{\scriptscriptstyle{1\otimes s}}\\ 
M\otimes X\ar[dddrr]_-{\bar{t}} &&&
[P(A,B),B]\otimes P(A,B)\otimes [X,N]\otimes
X\qquad\ar[d]^-{\scriptscriptstyle{\mathrm{ev}\otimes1}}\\ 
&&& B\otimes [X,N]\otimes
X\ar[d]^-{\scriptscriptstyle{1\otimes\mathrm{ev}}}\\ 
 &&& B\otimes
N\ar[dl]^-{\mu}\\ 
&& N &}
\end{equation}
where $\bar{t}:M\otimes X\to N$ is the adjunct of $t$
in $\ca{V}$.

The goal is to show that the pair (\ref{thispair}) is
actually an element of the set 
$\Mod(M,\bar{H}(X,N))$, which is of the general form
\begin{displaymath}
\begin{cases}
M\to f^*[X,N] &\text{in }\Mod_\ca{V}(A)\\
A\xrightarrow{f}[C,B] &\text{in }\Mon(\ca{V})
\end{cases}
\end{displaymath}
for some $f:A\to[C,B]$,
so that a bijective correspondence
(\ref{comodold}) will be established. For that, it 
is enough to prove
that $t$ coincides with an $A$-module map 
$M\to\tilde{h}^*[X,N]$, since there is already
a monoid arrow $\tilde{h}:A\to[C,B]$.
So the question would be whether $t$ satisfies
the commutativity of a diagram
\begin{displaymath} 
\xymatrix @R=.2in @C=.7in 
{A\otimes M\ar[r]^-{1\otimes t}
\ar[dd]_-{\mu} & A\otimes [X,N]\ar
@{.>}[dd]\ar[rd]^-{\tilde{h}\otimes 1} &\\ 
& & [C,B]\otimes
[X,N]\ar[dl]^-{\mu}\\ 
M\ar[r]_-t & [X,N] &}
\end{displaymath}
which again under the adjunction 
$(-\otimes X)\dashv[X,-]$ translates, by rearranging
the terms appropriately, to the diagram
\begin{equation}\label{impdiag2} 
\xymatrix @C=0.4in @R=0.2in
{A\otimes M\otimes X\ar[r]^-{1\otimes
t\otimes\delta} \ar[ddd]_-{\mu\otimes1}& 
\drtwocell<\omit>{'(**)} A\otimes [X,N]\otimes
X\otimes C\ar[r]^-{\tilde{h}\otimes1}\ar @{-->}
@/_2pc/[ddr] & [C,B]\otimes [X,N]\otimes X\otimes
C\ar[d]^-{\scriptscriptstyle{1\otimes s\otimes1}}\\ 
&& [C,B]\otimes C\otimes [X,N]\otimes
X\ar[d]^-{\scriptscriptstyle{\mathrm{ev}\otimes1\otimes1}}\\ 
&& B\otimes[X,N]\otimes
X\ar[d]^-{\scriptscriptstyle{1\otimes\mathrm{ev}}}\\
M\otimes X\ar[dr]_-{\bar{t}} && B\otimes
N\ar[dl]^-{\mu}\\ 
& N. &}
\end{equation} 
By inspection of
the commutative diagram (\ref{impdiag1}) and
this one (\ref{impdiag2}), it suffices to show that the 
parts $(*)$ and $(**)$ are the same
for the latter to commute as well. 
Since the term $[X,N]$ remains unchanged,
this comes down to the
commutativity of 
\begin{displaymath}
\xymatrix @C=.4in @R=.2in
{& [P(A,B),B]\otimes
C\ar[r]^-{1\otimes h} & [P(A,B),B]\otimes
P(A,B)\ar[dr]^-{\mathrm{ev}} & \\ 
A\otimes C\ar[ur]^-{\varepsilon\otimes1}
\ar[drr]_-{\tilde{h}\otimes1} &&& B.\\ 
&& [C,B]\otimes
C\ar[ur]_-{\mathrm{ev}} &}
\end{displaymath}
This is satisfied by Lemma \ref{lemma},
since $h=\hat{\tilde{h}}$.

Thus a bijection (\ref{comodold}) is established,
and by standard arguments of adjunctions via representing 
objects and Theorem \ref{parametrizedadjunctions}, this
results once again to the existence of a parametrized adjoint 
$Q(-,-)$ of $H^\op(-,-)$.
\end{proof}
In essence, the above proof establishes 
that the $A$-modules 
$(\varepsilon_A)^*[h_!X,N]$ and 
$(\tilde{h})^*[X,N]$ are essentially
the same. 
As objects they are both $[X,N]$, and
their $A$-actions can be verified to coincide, 
when we conveniently translate them under 
the usual tensor-hom adjunction. 
If we compare this with 
the proof of Lemma \ref{totaladjointlem},
the above fact follows from the final
diagram (\ref{diagimportantproof}), where the 
part on the right is actually equality since
we are now dealing with split fibrations 
and opfibrations, and the part on the left follows
from cocartesianess (on the nose) of the functor 
$\bar{H}$ as shown at the end of Section 
\ref{globalcats}. However,
since in the direct proof neither
cocartesianess nor splitness is explicitly used
or mentioned,
Lemma \ref{lemma} incorporates the necessary
information for the proof to be completed.
\nocite{Species,HopfAlg}

\chapter{Enrichment of $\ca{V}$-Categories and $\ca{V}$-Modules}\label{VCatsVCocats}
\section{The bicategory of $\ca{V}$-matrices}\label{bicatVMat}

The bicategory of $\ca{V}$-matrices was
mentioned in Examples \ref{examplesbicat} 
for $\ca{V}=\B{Set}$. We now give a 
detailed description of enriched
matrices and the structure of the
bicategory they form, unravelling
Definition \ref{bicategory} in this specific
case. The main references
here are \cite{VarThrEnr} and
\cite{KellyLack}. In the former, 
the more general bicategory $\ca{W}$-$\Mat$ of
matrices enriched in a bicategory $\ca{W}$
was studied, leading to the theory of
bicategory enriched categories. For the one-object case, 
\emph{i.e.} monoidal categories, the main results
are in works of B{\'e}nabou \cite{Distributeurs} and 
Wolff \cite{Wolff}.

Suppose that $\ca{V}$ is a cocomplete
monoidal category,
such that the functors $A\otimes-$
and $-\otimes A$ preserve colimits,
as is certainly the case if
$\ca{V}$ is monoidal closed.
For sets $X$ and $Y$,
a $\ca{V}$\emph{-matrix}$\SelectTips{eu}{10}
\xymatrix @C=.2in
{S:X\ar[r]|-{\object@{|}} & Y}$from $X$ to $Y$
is a functor $S:Y\times X\to\ca{V}$
given by a family
\begin{displaymath}
 \{S(y,x)\}_{(x,y)\in X\times Y}
\end{displaymath}
of objects in $\ca{V}$, where the set 
$Y\times X$ is viewed as a discrete
category.

The bicategory $\ca{V}$-$\Mat$ consists of 
(small) sets $X,Y$ as objects,
$\ca{V}$-matrices $\SelectTips{eu}{10}
\xymatrix @C=.2in
{S:X\ar[r]|-{\object@{|}} & Y}$as 1-cells
and natural transformations
\begin{displaymath}
 \xymatrix
{Y\times X\rrtwocell^{S}_{S'}{\;\sigma} && \ca{V}}=:
\xymatrix
{X\ar@/^2ex/[rr]|-{\object@{|}}^-S \ar@/_2ex/[rr]|-{\object@{|}}_-{S'}
\rrtwocell<\omit>{\sigma} && Y}
\end{displaymath}
as 2-cells between $\ca{V}$-matrices $S$ and $S'$.
These are given by families of arrows 
\begin{displaymath}
\sigma_{y,x}:S(y,x)\to
S'(y,x)
\end{displaymath}
in $\ca{V}$,
for every $(x,y)\in X\times Y$.
Hence the hom-category for two objects $X$ and $Y$
is the category 
\begin{displaymath}
\ca{V}\textrm{-}\B{Mat}(X,Y)=\ca{V}^{Y\times X}
\end{displaymath}
with (vertical) composition of 2-cells 
being `componentwise' in $\ca{V}$ and
the identity 2-cell $1_S:S\Rightarrow S$
consisting of identity morphisms 
$(1_S)_{x',x}=1_{S(x',x)}$ in $\ca{V}$.
The horizontal composition
\begin{displaymath}
 \circ:\ca{V}\textrm{-}\Mat(Y,Z)\times\ca{V}\textrm{-}\Mat(X,Y)
\to\ca{V}\textrm{-}\Mat(X,Z)
\end{displaymath}
maps two composable $\ca{V}$-matrices$\SelectTips{eu}{10}
\xymatrix @C=.2in
{T:Y\ar[r]|-{\object@{|}} & Z}$and$\SelectTips{eu}{10}
\xymatrix @C=.2in
{S:X\ar[r]|-{\object@{|}} & Y}$to their
composite 1-cell$\SelectTips{eu}{10}
\xymatrix @C=.2in
{T\circ S:X\ar[r]|-{\object@{|}} & Z,}$
given by the family of objects in $\ca{V}$
\begin{equation}\label{horizontalcompositionVmatrices}
(T\circ S)(z,x)=\sum_{y\in Y} T(z,y)\otimes S(y,x)
\end{equation}
for all $z\in Z$ and $x\in X$.
A pair of 2-cells 
$(\tau:T\Rightarrow T',\sigma:S\Rightarrow S')$
is mapped to the 2-cell $\tau*\sigma:T\circ S\Rightarrow T'\circ S'$
with components arrows
\begin{equation}\label{horizontalcompositionVmatricearrows}
 (\tau*\sigma)_{z,x}:\sum_{y\in Y} T(z,y)\otimes S(y,x)
\xrightarrow{\;\sum{\tau_{z,y}\otimes\sigma_{y,x}}\;}
\sum_{y\in Y} T'(z,y)\otimes S'(y,x)
\end{equation}
in $\ca{V}$.
For each set $X$, the 
identity 1-cell is$\SelectTips{eu}{10}
\xymatrix @C=.2in
{1_X:X\ar[r]|-{\object@{|}} & X,}$
which is given by
\begin{displaymath}
1_X(x',x)=\begin{cases}
I,\quad \mathrm{if  }\;x=x'\\
0,\quad \mathrm{ otherwise}
\end{cases}
\end{displaymath}
where $I$ is the unit object in $\ca{V}$
and $0$ is the initial object.

For composable $\ca{V}$-matrices$\SelectTips{eu}{10}
\xymatrix @C=.2in
{X\ar[r]|-{\object@{|}}^-S & 
Y\ar[r]|-{\object@{|}}^-T & Z
\ar[r]|-{\object@{|}}^-R & W,}$
the associator $\alpha$ has components
invertible 2-cells
\begin{displaymath}
\alpha^{R,T,S}:
(R\circ T)\circ S\stackrel{\sim}{\longrightarrow}
R\circ(T\circ S)
\end{displaymath}
in $\ca{V}$-$\B{Mat}$,
given by the family $\{\alpha_{w,x}\}_{w,x}$ of composite
isomorphisms
\begin{displaymath}
\xymatrix @C=.4in @R=.53in
{\sum\limits_{y\in Y}{\big(\sum\limits_{z\in Z}{R(w,z)\otimes T(z,y)}\big)\otimes
S(y,x)}\ar@{-->}[r] \ar[d]_-{\cong} &
\sum\limits_{z\in Z}{R(w,z)\otimes
\big(\sum\limits_{y\in Y}{T(z,y)\otimes S(y,x)}\big)} \\
\sum\limits_{\stackrel{y\in Y}{\scriptscriptstyle{z\in Z}}}
\big((R(w,z)\otimes T(z,y))\otimes S(y,x)\big)
\ar[r]_-{\sum{a}} & 
\sum\limits_{\stackrel{y\in Y}{\scriptscriptstyle{z\in Z}}}
\big(R(w,z)\otimes(T(z,y)\otimes S(y,x))\big)
\ar[u]_-{\cong}}
\end{displaymath}
in $\ca{V}$. The isomorphism $a$ is the associativity constraint of $\ca{V}$
and the vertical invertible arrows express the fact that 
$\otimes$ commutes with colimits. This definition 
clearly makes the horizontal composition 
associative up to isomorphism.
Finally, for each $\ca{V}$-matrix$\SelectTips{eu}{10}
\xymatrix @C=.2in
{S:X\ar[r]|-{\object@{|}} & 
Y,}$ 
the unitors $\lambda, \rho$ have components
invertible 2-cells
\begin{displaymath}
\lambda^S:1_Y\circ S\xrightarrow{\;\sim\;}S,\quad
\rho^S:S\circ 1_X\xrightarrow{\;\sim\;}S
\end{displaymath}
given by families of isomorphisms
\begin{align*}
\lambda^S_{y,x}&:\sum_{y'\in Y}{1_Y(y,y')\otimes S(y',x)}\equiv
I\otimes S(y,x)\xrightarrow{\;l_{S(y,x)}\;} S(y,x) \\
\rho^S_{y,x}&:\sum_{x'\in X}
{S(y,x')\otimes 1_X(x',x)}\equiv
S(y,x)\otimes I\xrightarrow{\;r_{S(y,x)}\;} S(y,x)
\end{align*}
where $l$ and $r$ are the right and left unit constraints
of $\ca{V}$. The respective coherence condition is 
satisfied, thus these data indeed define a bicategory.
Notice that only the existence of coproducts in $\ca{V}$ is enough
for the formation of $\ca{V}$-$\B{Mat}$.

The hom-categories $\ca{V}$-$\B{Mat}(X,X)$ of this 
bicategory for a fixed
set $X$ will play an important role in this chapter.
The following proposition underlines 
some of the properties that these categories
possess, and more specifically the ones
that imply certain results with regard to
categories of monoids and comonoids as seen in Chapter
\ref{monoidalcategories}.
\begin{prop}\label{propVMat}
Let $\ca{V}$ be a cocomplete monoidal category such 
that $\otimes$ preserves colimits on both entries.
The category 
$\ca{V}$-$\B{Mat}(X,X)$ for any set $X$
\begin{enumerate}[(i)]
\item is cocomplete, and has all limits 
that exist in $\ca{V}$;
\item is a monoidal
category, and $\otimes=\circ$
preserves colimits on both entries;
\item
is locally presentable when $\ca{V}$ is; 
\item
is monoidal closed when $\ca{V}$ is monoidal closed
with products.
\end{enumerate}
\end{prop} 
\begin{proof}
$(i)$
Since $\ca{V}$-$\B{Mat}(X,X)=[X\times X,\ca{V}]$,
all limits and colimits can be formed pointwise
from those in $\ca{V}$.

$(ii)$
The hom-categories $\ca{K}(X,X)$ 
for any bicategory $\ca{K}$ obtain a monoidal structure
via the horizontal composition, as in (\ref{tensorcirc}).
The unit object is the identity $\ca{V}$-matrix $1_X$,
so $(\ca{V}\textrm{-}\Mat(X,X),\circ,1_X)$ is
a monoidal category.

Horizontal composition of $\ca{V}$-matrices
preserves colimits on both entries:
if $(G_j\to G\,|\,j\in\ca{J})$ is a colimiting cocone
for a diagram of shape $\ca{J}$ in 
$\ca{V}$-$\B{Mat}(X,X)$, this means that
for any $x,y\in X$, the arrows $G_j(x,y)\to G(x,y)$
form colimiting cocones in $\ca{V}$. If we apply the functor
\begin{displaymath}
-\circ S:\ca{V}\textrm{-}\B{Mat}(X,X)
\to\ca{V}\textrm{-}\B{Mat}(X,X)
\end{displaymath}
for any $\ca{V}$-matrix$\SelectTips{eu}{10}\xymatrix @C=.2in
{S:X\ar[r]|-{\object@{|}} & X,}$
we obtain a collection of 2-cells
$(G_j\circ S\to G\circ S\,|\,j\in\ca{J})$
in $\ca{V}$-$\Mat$.
For this to be a colimit,
for any $x,z\in X$ the arrows
\begin{displaymath}
\sum_{y\in X}{G_j(x,y)\otimes S(y,z)}\longrightarrow
\sum_{y\in X}{{\mathrm{colim}_j}G_j(x,y)\otimes S(y,z)}
\end{displaymath}
must also form colimiting cocones in $\ca{V}$. Since by
assumptions 
$(-\otimes A)$ preserves colimits for any $A\in\ca{V}$,
we have isomorphisms
\begin{align*}
\sum_{y\in X}{({\mathrm{colim}_j}G_j(x,y))\otimes S(y,z)}
&\cong\sum_{y\in X}{{\mathrm{colim}_j}(G_j(x,y)\otimes S(y,z))} \\
&\cong{\mathrm{colim}_j}(\sum_{y\in X}{G_j(x,y)\otimes S(y,z)}),
\end{align*}
thus $-\circ S$ is cocontinuous. Similarly, $S\circ -$
preserves colimits for any $\ca{V}$-matrix,
since $(A\otimes -)$ does in $\ca{V}$.

$(iii)$
For each locally 
$\lambda$-presentable category $\ca{C}$, it is known
that the functor
category $\ca{C}^\ca{A}=[\ca{A},\ca{C}]$ 
for any small
category $\ca{A}$ is locally $\lambda$-presentable
itself, see \cite[1.54]{LocallyPresentable}.
Hence, for the discrete small category $X\times X$, 
the functor category $\ca{V}^{X\times X}$ is a locally 
presentable category.

$(iv)$
We need to demonstrate a bijective correspondence
between morphisms
\begin{equation}\label{thisthis}
\xymatrix @R=.02in
{\qquad\quad S\circ T \ar[rr] && R\phantom{ABCDE} 
&\mathrm{in}\;\ca{V}\textrm{-}\B{Mat}(X,X)\\ 
\ar@{-}[rr] &&& \\  
\qquad\qquad S \ar[rr] && G(T,R)\phantom{ABC} 
& \mathrm{in}\;\ca{V}\textrm{-}\B{Mat}(X,X).} 
\end{equation} 
We define the $\ca{V}$-matrix $G(T,R)$ from $X$ to $X$ 
to be given by the family of objects in $\ca{V}$
\begin{displaymath}
G(T,R)(x,y):=\prod_{z\in X}{[T(y,z),R(x,z)]}
\end{displaymath}
where $[-,-]$ is the internal hom in $\ca{V}$.
Then, an arrow $\sigma:S\to G(T,R)$
in $\ca{V}\textrm{-}\Mat(X,X)$ is given by a family of arrows
\begin{displaymath}
\sigma_{x,y}:S(x,y)\to\prod_{z\in X}{[T(y,z),R(x,z)]}
\end{displaymath}
in $\ca{V}$, for each $x,y\in X$. Since $\ca{V}$
is monoidal closed, for any fixed
$z$ the arrow $S(x,y)\to[T(y,z),R(x,z)]$ 
corresponds uniquely to
$S(x,y)\otimes T(y,z)\to R(x,z)$, which in turn
gives a unique arrow in $\ca{V}$ from the sum over
all $y$'s in $X$
\begin{displaymath}
\rho_{x,z}:\sum_{y\in X}{S(x,y)\otimes T(y,z)}\to R(x,z).
\end{displaymath} 
These arrows form a family which
defines a 2-cell $\rho:S\circ T\to R$ in 
$\ca{V}\textrm{-}\B{Mat}(X,X)$, thus
the correspondence (\ref{thisthis})
is now established. 

Notice that this actually shows that
$\ca{V}$-$\Mat(X,X)$ is left closed,
but we can repeat the above argument
using the (right) internal hom of the monoidal
closed $\ca{V}$ appropriately, and show that
$\ca{V}$-$\Mat(X,X)$ is (bi)closed.
\end{proof}
Recall that Proposition \ref{moncomonadm} 
presented some very
useful properties for the categories of monoids
and comonoids of admissible categories,
\emph{i.e.} locally presentable symmetric monoidal categories,
such that tensoring on one side preserves all filtered
colimits. However, as was also noted then,
the results are still valid
if we drop the symmetry condition
and ask instead that both $A\otimes -$ 
and $-\otimes A$ preserve
(filtered) colimits.
\begin{cor}\label{cofreecomonVMat}
 If $\ca{V}$ is a locally presentable monoidal
category, where $\otimes$ preserves colimits
in both entries, the forgetful functors
\begin{gather*}
S:\Mon(\ca{V}\textrm{-}\Mat(X,X))\to\ca{V}\textrm{-}\Mat(X,X) \\
U:\Comon(\ca{V}\textrm{-}\Mat(X,X))\to\ca{V}\textrm{-}\Mat(X,X)
\end{gather*}
are monadic and comonadic respectively,
and all categories are locally presentable.
\end{cor}
The existence of the free monoid and cofree comonoid
functors will be of use in Section \ref{VcatsandVcocats}.
As mentioned again in Chapter \ref{monoidalcategories}, in reality
the free monoid construction requires less assumptions than 
the ones above, \emph{i.e.} existence
of coproducts which are preserved by the tensor product.
Notice that the current setting only differs from the general one of 
Section \ref{Categoriesofmonoidsandcomonoids}, 
in that the categories
of monoids and comonoids of the non-symmetric
$(\ca{V}\text{-}\Mat(X,X),\circ,1_X)$
cannot inherit its monoidal structure.

The bicategory $\ca{V}$-$\B{Mat}$ is in fact 
a \emph{monoidal bicategory} (see \cite{Carmody})
via a pseudofunctor
\begin{displaymath}
\otimes:\ca{V}\textrm{-}\Mat
\times\ca{V}\textrm{-}\Mat\longrightarrow
\ca{V}\textrm{-}\Mat.
\end{displaymath}
This maps any two sets $X$ and $Y$ to their cartesian 
product $X\times Y$, any two matrices 
$\{S(y,x)\}_{y,x}$ and $\{T(z,w)\}_{z,w}$ 
to the $\ca{V}$-matrix
with components
\begin{equation}\label{monoidalVMat}
 (S\otimes T)\big((y,z),(x,w)\big)=S(y,x)\otimes T(z,w)
\end{equation}
and any 2-cells to their pointwise tensor product in $\ca{V}$.
The monoidal unit is the unit 
$\ca{V}$-matrix$\SelectTips{eu}{10}
\xymatrix@C=.2in{\ca{I}:1\ar[r]|-{\object@{|}} & 1}$where 
$1=\{*\}$ is the singleton set, with $\ca{I}(*,*)=I$.
This monoidal structure will be discussed in detail 
in the next chapter (see Proposition \ref{bicatVMatmonoidal}).

We now proceed to the definition of a 
specific lax functor which 
will later give rise to certain very important
mappings for particular enrichment relations we want to 
establish. Intuitively, there is an analogy
with the internal hom functor of our monoidal closed $\ca{V}$ 
in the previous chapter, which induced the mappings 
$H$ and $\bar{H}$ between the categories
of monoids/comonoids and modules/comodules.

Suppose that
$\ca{V}$ is a cocomplete symmetric monoidal closed
category with products. If $\ca{V}$-$\Mat^\textrm{co}$
is the bicategory of $\ca{V}$-matrices with reversed 
2-cells,
define a lax functor of bicategories
\begin{equation}\label{defimportHom}
 \Hom:(\ca{V}\textrm{-}\Mat)^{\textrm{co}}
\times\ca{V}\textrm{-}\Mat\longrightarrow
\ca{V}\textrm{-}\Mat
\end{equation}
as follows: 

$\cdot$
each pair of sets $(X,Y)$ is mapped to
the set $\Hom(X,Y):=Y^X$ of functions
from $X$ to $Y$;

$\cdot$
for all pairs $(X,Y),(Z,W)$
there is a functor
\begin{equation}\label{Hom_}
\xymatrix @R=.02in @C=.9in
{\ca{V}\textrm{-}\Mat(X,Z)^\op\times
\ca{V}\textrm{-}\Mat(Y,W)\ar[r]^-{\Hom_{(X,Y),(Z,W)}} &
\ca{V}\textrm{-}\Mat(Y^X,W^Z) \\
(\;S\;,\;T\;)\ar@{|.>}[r]
\ar[dd]_-{(\sigma,\tau)} & \Hom(S,T)\ar[dd]^-{\Hom(\sigma,\tau)} \\
\hole \\
(\;S'\;,\;T'\;)\ar@{|.>}[r] & \Hom(S',T')}
\end{equation}
where the $\ca{V}$-matrix$\SelectTips{eu}{10}
\xymatrix @C=.2in
{\Hom(S,T):Y^X\ar[r]|-{\object@{|}} & W^Z}$is given by the family 
\begin{equation}\label{Homobjects}
 \Hom(S,T)(q,k):=\prod_{\scriptscriptstyle{\stackrel{z\in Z}{x\in X}}}
{[S(z,x),T(qz,kx)]}
\end{equation}
of objects in $\ca{V}$,
for all $q\in W^Z$ and $k\in Y^X$, where $[-,-]$
is the internal hom in $\ca{V}$.
For $\sigma:S'\Rightarrow S$ and $\tau:T\Rightarrow T'$,
the 2-cell
\begin{equation}\label{Hom2cells}
\xymatrix @C=1.2in
{Y^X\ar @/^3ex/[r]|-{\object@{|}}^-{\Hom(S,T)}
\ar@/_3ex/[r]|-{\object@{|}}_-{\Hom(S',T')}
\rtwocell<\omit>{\qquad\;\;\;\Hom(\sigma,\tau)} & W^Z}
\end{equation}
has components, for every $(q,k)\in  W^Z\times Y^X$,
arrows in $\ca{V}$ 
\begin{displaymath}
\Hom(\sigma,\tau)_{q,k}:
\prod_{(z,x)}
{[S(z,x),T(qz,kx)]}\longrightarrow
\prod_{(z,x)}{[S'(z,x),T'(qz,kx)].}
\end{displaymath}
For fixed
$z,x$, these correspond under the usual tensor-hom adjunction
in $\ca{V}$
to
\begin{displaymath}
\xymatrix @C=.8in
{[S(z,x),T(qz,kx)]\otimes S'(z,x)\ar@{-->}[r]
\ar[d]_-{1\otimes\sigma_{z,x}} &T'(qz,kx) \\
[S(z,x),T(qz,kx)]\otimes S(z,x)\ar[r]_-{\mathrm{ev}_{T(qz,kx)}} &
T(qz,kx)\ar[u]_-{\tau_{qz,kx}}&}
\end{displaymath}
where $\mathrm{ev}$ is the evaluation;

$\cdot$ for all $(X,Y),(Z,W),(U,V)$, there is a natural transformation
$\delta$ with components, for$\SelectTips{eu}{10}
\xymatrix @C=.2in
{(R:Z\ar[r]|-{\object@{|}} & U,}
\SelectTips{eu}{10}
\xymatrix @C=.2in
{O:W\ar[r]|-{\object@{|}} & V)}$and$\SelectTips{eu}{10}
\xymatrix @C=.2in
{(S:X\ar[r]|-{\object@{|}} & Z,}
\SelectTips{eu}{10}
\xymatrix @C=.2in
{T:Y\ar[r]|-{\object@{|}} & W),}$2-cells in $\ca{V}$-$\Mat$
\begin{equation}\label{Homlaxfunctor1}
 \xymatrix @R=.02in @C=.8in
{& W^Z \ar @/^/[dr]|-{\object@{|}}^-{\Hom(R,O)} & \\
Y^X \ar @/^/[ur]|-{\object@{|}}^-{\Hom(S,T)}
\rrtwocell<\omit>{\qquad\qquad\delta_{(S,T),(R,O)}}
\ar @/_3ex/[rr]|-{\object@{|}}_-{\Hom(R\circ S,O\circ T)} && V^U}
\end{equation}
which are given by families of arrows in $\ca{V}$
\begin{displaymath}
\sum_{q\in W^Z}{\Hom(R,O)(t,q)\otimes\Hom(S,T)(q,k)}\xrightarrow{\delta_{t,k}}
\prod_{(u,x)}{[(R\circ S)(u,x),(O\circ T)(tu,kx)]}
\end{displaymath}
for all $(t,k)\in V^U\times Y^X$. These again
can be understood via their transposes
under the tensor-hom adjunction, \emph{i.e.}  
composites of projections,
inclusions, symmetries and evaluations,
using the fact that the tensor product
preserves sums;

$\cdot$
for all pairs of sets $(X,Y)$, there is a natural
transformation $\gamma$ with components
\begin{equation}\label{Homlaxfunctor2}
\xymatrix @C=1in
{Y^X\ar @/^3ex/[r]|-{\object@{|}}^-{1_{Y^X}}
\ar@/_3ex/[r]|-{\object@{|}}_-{\Hom(1_X,1_Y)}
\rtwocell<\omit>{\qquad\;\;\gamma_{(X,Y)}} & Y^X}
\end{equation}
which for $q=k\in Y^X$ and
$x'=x\in X$ consist of the isomorphisms
\begin{displaymath}
 (\gamma_{(X,Y)})_{q,q}:I\longrightarrow[1_X(x,x),1_Y(kx,kx)]=[I,I].
\end{displaymath}
The coherence axioms of Definition \ref{laxfunctor}
are satisfied, therefore $\Hom$
is a lax functor of bicategories.

We now turn to some more technical points
of the bicategory $\ca{V}$-$\Mat$.
Any function $f:X\to Y$ between two sets $X$, $Y$
determines two $\ca{V}$-matrices,$\SelectTips{eu}{10}
\xymatrix @C=.2in
{f_*:X\ar[r]|-{\object@{|}} & Y}$
and$\SelectTips{eu}{10}
\xymatrix @C=.2in
{f^*:Y\ar[r]|-{\object@{|}} & X,}$given by
\begin{equation}\label{f*}
f_*(y,x)=f^*(x,y)=\begin{cases}
I,\quad \mathrm{if  }\;f(x)=y\\
0,\quad \mathrm{ otherwise}
\end{cases}
\end{equation}
for any $x\in X$, $y\in Y$.
It can be easily
verified that there is a 
natural bijection between 2-cells
$f_*\circ S\Rightarrow T$
and $S\Rightarrow f^*\circ T$ for
any $\ca{V}$-matrices$\SelectTips{eu}{10}
\xymatrix @C=.2in
{S:Z\ar[r]|-{\object@{|}} & X}$
and$\SelectTips{eu}{10}
\xymatrix @C=.2in
{T:Z\ar[r]|-{\SelectTips{eu}{}\object@{|}} & W,}$
thus they
form an adjunction
$f_*\dashv f^*$ in 
the bicategory $\ca{V}$-$\B{Mat}$.
The unit and counit of this adjunction are the 2-cells
\begin{displaymath}
\xymatrix @C=.6in
{X \ar @/^2ex/[r]|-{\object@{|}}^-{1_X}
\ar@/_2ex/[r]|-{\object@{|}}_-{f^*\circ f_*}
\rtwocell<\omit>{\;\check{\eta}} & X}
\qquad\mathrm{and}\qquad
\xymatrix @=.6in
{Y \ar @/^2ex/[r]|-{\object@{|}}^-{f_*\circ f^*}
\ar@/_2ex/[r]|-{\object@{|}}_-{1_Y}
\rtwocell<\omit>{\;\check{\varepsilon}} & Y}
\end{displaymath}
with components arrows in $\ca{V}$
\begin{displaymath}
\check{\varepsilon}_{y',y}:(f_*\circ f^*)(y',y)\to 1_Y(y',y)\equiv
\begin{cases}
\sum\limits_{x\in f^{-1}(y)}{I\otimes I}\xrightarrow{r_I}I, & \text{if }y=y' \\
\phantom{\sum\limits_{x\in f^{-1}(y)}}0\xrightarrow{!}0, & \text{if }y\neq y'
\end{cases}
\end{displaymath}
and
\begin{displaymath}
\check{\eta}_{x',x}:1_X(x',x)\to
(f^*\circ f_*)(x',x)\equiv
\begin{cases}
I\xrightarrow{(r_I)^{-1}}I\otimes I, & \text{if }x'=x \\
0\xrightarrow{!}{\begin{cases} I\otimes I, & fx=fx'\\
0, & \text{else}
\end{cases}} & \text{if }x'\neq x
\end{cases}
\end{displaymath}
where $!$
is the unique
arrow from the initial to any object. 
Notice that $\check{\eta}$ and $\check{\varepsilon}$
are isomorphisms if and only if the function
$f$ is a bijection.

These $\ca{V}$-matrices induced by 
functions between sets are of central importance
to constructions in later sections. Below
we show some useful properties.
\begin{lem}\label{isosofstars}
Let $f:X\to Y$ and $g:Y\to Z$ be functions. There
exist isomorphisms
\begin{gather*}
\zeta^{g,f}:g_*\circ f_*\cong(gf)_*:
\SelectTips{eu}{10}\xymatrix
{X\ar[r]|-{\object@{|}} & Z} \\
\xi^{g,f}:f^*\circ g^*\cong(gf)^*:
\SelectTips{eu}{10}\xymatrix
{Z\ar[r]|-{\object@{|}} & X}
\end{gather*}
which are families of invertible arrows
\begin{equation}\label{zeta}
\zeta^{g,f}_{z,x}=\xi^{g,f}_{x,z}:\begin{cases}
I\otimes I\xrightarrow{r_I=l_I}I, &\textrm{if }g(f(x))=z \\
\quad\quad 0\xrightarrow{\;!\;}0, & \textrm{otherwise}
\end{cases}
\end{equation}
for each pair of elements $(x,z)\in X\times Z$.
\end{lem}
\begin{proof}
In general,
for any $\ca{V}$-matrix$\SelectTips{eu}{10}
\xymatrix @C=.2in
{S:Y\ar[r]|-{\object@{|}} & 
Z,}$the composite 1-cell
$S\circ f_*$
is computed to be the family
\begin{displaymath}
(S\circ f_*)(z,x)=\sum_{y\in Y}{S(z,y)\otimes
f_*(y,x)}=\sum_{y=fx}S(z,y)\otimes I=S(z,fx)\otimes I
\stackrel{r}{\cong}S(z,fx)
\end{displaymath}
of objects in $\ca{V}$, for any $(z,x)\in Z\times X$.
Similarly, for a $\ca{V}$-matrix $\SelectTips{eu}{10}
\xymatrix @C=.2in
{T:Z\ar[r]|-{\object@{|}} & Y,}$the composite
$\ca{V}$-matrix $f^*\circ T$ is the family 
\begin{displaymath}
(f^*\circ T)(x,z)=\sum_{y\in Y}{f^*(x,y)\otimes T(y,z)}
=I\otimes T(fx,z)\stackrel{l}{\cong}T(fx,z)
\end{displaymath}
of objects in $\ca{V}$, for all $(x,z)\in X\times Z$.

Using the above technique, we can explicitly 
write the families of objects in $\ca{V}$ which
define the $\ca{V}$-matrices
$g_*\circ f_*$ and $f^*\circ g^*$
\begin{displaymath}
(g_*\circ f_*)(z,x)=(f^*\circ g^*)(x,z)=\begin{cases}
I\otimes I, & \textrm{if }g(f(x))=z\\
0, & \textrm{otherwise}
\end{cases}
\end{displaymath}
for any pairs of elements $(x,z)\in X\times Z$.
We can now provide isomorphisms
\begin{displaymath}
\xymatrix @R=.1in @C=.5in
{& Y \ar @/^/[dr]|-{\object@{|}}^-{g_*} & \\
X \ar @/^/[ur]|-{\object@{|}}^-{f_*}
\rrtwocell<\omit>{\quad\zeta^{g,f}}
\ar @/_4ex/[rr]|-{\object@{|}}_-{(gf)_*} && Z}
\quad
\xymatrix @R=.05in{\hole \\ \mathrm{and}}
\quad
\xymatrix @R=.1in @C=.5in
{& Y \ar @/^/[dr]|-{\object@{|}}^-{f^*} & \\
Z \ar @/^/[ur]|-{\object@{|}}^-{g^*}
\rrtwocell<\omit>{\quad\xi^{g,f}}
\ar @/_4ex/[rr]|-{\object@{|}}_-{(gf)^*} && X}
\end{displaymath}
which consist of families of invertible arrows
in $\ca{V}$ exactly the (\ref{zeta}).
\end{proof}
Based on the above formulas, 
it is straightforward
to show that $\zeta$ and $\xi$ satisfy the
following relations, which clarify
how the composition of
three such matrices works.
\begin{lem}\label{corisos}
Consider three composable functions
$X\xrightarrow{f}Y\xrightarrow{g}Z\xrightarrow{h}W$. Then
\begin{displaymath}
\xymatrix @R=.2in
{& Y\ar @/^/[r]|-{\object@{|}}^-{g_*} & 
Z\ar @/^/[dr]|-{\object@{|}}^-{h_*} & \\
X \urrtwocell<\omit>{\quad\zeta^{g,f}}
\ar @/^/[ur]|-{\object@{|}}^-{f_*} 
\ar @/_3ex/[urr]|-{\object@{|}}_-{(gf)_*}
\ar @/_4ex/[rrr]|-{\object@{|}}_-{(hgf)_*}
& \rrtwocell<\omit>{\quad\zeta^{gf,h}} && W}
\;\xymatrix @R=.2in {\hole \\ \textrm{=}}\;
\xymatrix @R=.2in
{& Y \drrtwocell<\omit>{\quad\zeta^{g,h}}
\ar @/_3ex/[drr]|-{\object@{|}}_-{(hg)_*}
\ar @/^/[r]|-{\object@{|}}^-{g_*} & 
Z\ar @/^/[dr]|-{\object@{|}}^-{h_*} & \\
X \rrtwocell<\omit>{\quad\zeta^{f,hg}}
\ar @/^/[ur]|-{\object@{|}}^-{f_*} 
\ar @/_4ex/[rrr]|-{\object@{|}}_-{(hgf)_*}
&&& W}
\end{displaymath}
and 
\begin{displaymath}
\xymatrix @R=.2in
{& Z\ar @/^/[r]|-{\object@{|}}^-{g^*} & 
Y\ar @/^/[dr]|-{\object@{|}}^-{f^*} & \\
W \urrtwocell<\omit>{\quad\xi^{g,h}}
\ar @/^/[ur]|-{\object@{|}}^-{h^*} 
\ar @/_3ex/[urr]|-{\object@{|}}_-{(hg)^*}
\ar @/_4ex/[rrr]|-{\object@{|}}_-{(hgf)^*}
& \rrtwocell<\omit>{\quad\xi^{f,hg}} && X}
\;\xymatrix @R=.2in {\hole \\ \textrm{=}}\;
\xymatrix @R=.2in
{& Z \drrtwocell<\omit>{\quad\xi^{g,f}}
\ar @/_3ex/[drr]|-{\object@{|}}_-{(gf)^*}
\ar @/^/[r]|-{\object@{|}}^-{g^*} & 
Y\ar @/^/[dr]|-{\object@{|}}^-{f^*} & \\
W \rrtwocell<\omit>{\quad\xi^{gf,h}}
\ar @/^/[ur]|-{\object@{|}}^-{h^*} 
\ar @/_4ex/[rrr]|-{\object@{|}}_-{(hgf)^*}
&&& X.}
\end{displaymath}
\end{lem}

\section{The category of $\ca{V}$-graphs}\label{Vgraphs}

Graphs, with variations on their exact meaning
depending on the mathematical context they arise,
have been of use for a very long time.
For the needs of this thesis,
we study the case of graphs enriched in 
a monoidal category, 
in order to better understand $\ca{V}$-categories.
In this setting, enriched categories are enriched
graphs with extra structure, and 
$\ca{V}$-cocategories will also 
naturally fit in later.

As a primary example, in \cite[11.7]{MacLane}
the notion of a small (directed) graph
consisting of a set of objects and a set of arrows
was employed to describe the free category 
construction, in analogy with the free monoid
construction on a set. 
Moreover, the idea of $O$-graphs with a 
fixed set of objects $O$ inspires the fibrational
view of these categories, which is going
to be explicitly described in the following sections.
For the main results regarding $\ca{V}$-$\B{Grph}$
and $\ca{V}$-$\B{Cat}$ from a more traditional point 
of view, Wolff's \cite{Wolff} is a classic reference
for $\ca{V}$ a symmetric monoidal closed category,
whereas for the description of $\ca{V}$-graphs
in terms of $\ca{V}$-matrices, we again closely follow
\cite{VarThrEnr,KellyLack}.

A (small) $\ca{V}$-\emph{graph} $\ca{G}$
consists of a set of objects $\ob\ca{G}$,
and for every pair of objects $A,B\in\ob\ca{G}$
an object $\ca{G}(A,B)\in\ca{V}$. 
If $\ca{G}$ and $\ca{H}$ are $\ca{V}$-graphs,
a $\ca{V}$-\emph{graph morphism}
$F:\ca{G}\to\ca{H}$ consists of a function 
$f:\ob\ca{G}\to\ob\ca{H}$ between
their sets of objects, together with
arrows in $\ca{V}$
\begin{equation}\label{defvgrapharrow}
F_{A,B}:\ca{G}(A,B)\to\ca{H}(fA,fB)
\end{equation}
for each pair of objects $A,B$ in $\ca{G}$.
These data, with appropriate compositions
and identities, form a category 
$\ca{V}$-$\B{Grph}$.

Notably, the above definition does not require 
any assumptions on the monoidal category $\ca{V}$.
However, the context of the bicategory 
$\ca{V}$-$\Mat$ is
very convenient for connecting relations
between the above mentioned categories
to be exhibited. For this reason,
we proceed to the presentation of
equivalent characterizations in the language
of $\ca{V}$-matrices. Inevitably, 
we have to impose 
appropriate conditions
on $\ca{V}$ as in the previous section, namely cocompleteness
and $\otimes$ preserving colimits on both variables.
 
One can easily deduce that a $\ca{V}$-graph 
$\ca{G}$ as described above is an endoarrow in 
the bicategory $\ca{V}$-$\B{Mat}$, \emph{i.e.} 
a set $X=\ob\ca{G}$ 
together with a $\ca{V}$-matrix$\SelectTips{eu}{10}
\xymatrix @C=.2in
{G:X\ar[r]|-{\object@{|}} & X}$given by a family of objects
$G(x',x)$ in $\ca{V}$, for all $x',x\in X$.
Such a $\ca{V}$-graph will be denoted as
$(G,X)$ or $\ca{G}_X$.
Furthermore, a morphism of $\ca{V}$-graphs between
$(G,X)$ and $(H,Y)$ can be viewed as a function
$f:X\to Y$ between their sets of objects,
equipped with a 2-cell
\begin{displaymath}
\xymatrix @C=.6in
{X\ar @/^2ex/[r]|-{\object@{|}}^-{G}
\ar@/_2ex/[r]|-{\object@{|}}_-{f^*\circ H\circ f_*}
\rtwocell<\omit>{\;\phi} & X}
\end{displaymath}
in $\ca{V}$-$\Mat$, where $f_*$ and $f^*$ are as in 
(\ref{f*}). This is the case,
because the composite $\ca{V}$-matrix
\begin{displaymath}
\SelectTips{eu}{10}
\xymatrix @C=.3in
{X\ar[r]|-{\object@{|}}^-{f_*} & Y
\ar[r]|-{\object@{|}}^-{H} & Y
\ar[r] |-{\object@{|}}^-{f^*} & X}
\end{displaymath}
is given by the family of objects, for all $x',x\in X$,
\begin{align*}
(f^*Hf_*)(x',x) &= \sum_{y\in Y}{f^*(x',y)\otimes(Hf_*)(y,x)}
= I\otimes(Hf_*)(f(x'),x) \\
&= I\otimes\sum_{y\in Y}{H(f(x'),y)\otimes f_*(y,x)}
= I\otimes H(f(x'),f(x))\otimes I \\
&\cong H(f(x'),f(x)).
\end{align*}
Hence the 2-cell $\phi$ has components
arrows in $\ca{V}$ 
\begin{displaymath}
\phi_{x',x}:G(x',x)\longrightarrow I\otimes H(fx',fx)\otimes I\cong H(fx',fx)
\end{displaymath}
for $x',x\in X$.
This is essentially (\ref{defvgrapharrow}), in the sense
that the arrows $F_{x',x}$ and 
$\phi_{x'x,}$ are in bijective correspondence.
We write $F=(\phi,f)$ for this way of viewing 
$\ca{V}$-graph morphisms.
 
In fact, because of the adjunction
$f_*\dashv f^*$, the `mates correspondence'
of Proposition \ref{mates}
gives a bijection
between 2-cells
\begin{equation}\label{mates2cells}
\xymatrix {X\ar[r]|-{\SelectTips{eu}{}\object@{|}}
^-{G}\ar[d]|-{\SelectTips{eu}{}\object@{|}}
_-{f_*} & X\\
Y\ar[r]|-{\SelectTips{eu}{}\object@{|}}_-{H} 
\rtwocell<\omit>{<-4>\phi} & Y,
\ar[u]|-{\SelectTips{eu}{}\object@{|}}
_-{f^*}}
\qquad\textrm{and}\qquad
\xymatrix {X\ar[r]|-{\SelectTips{eu}{}\object@{|}}
^-{G} & X\ar[d]|-{\SelectTips{eu}{}\object@{|}}
^-{f_*}\\
Y\ar[u]|-{\SelectTips{eu}{}\object@{|}}
^-{f^*}\ar[r]|-{\SelectTips{eu}{}\object@{|}}_-{H}
\rtwocell<\omit>{<-4>\psi}
& Y}
\end{equation}
in the bicategory $\ca{V}$-$\Mat$.
By computing as before,
the composite $\ca{V}$-matrix
\begin{displaymath}
\SelectTips{eu}{10}
\xymatrix @C=.3in
{Y\ar[r]|-{\object@{|}}^-{f^*} & X
\ar[r]|-{\object@{|}}^-{G} & X
\ar[r] |-{\object@{|}}^-{f_*} & Y}
\end{displaymath}
is the family of objects in $\ca{V}$, for each $y,y'\in Y$,
\begin{displaymath}
(f_*Gf^*)(y',y)=
\sum_{\scriptscriptstyle{\stackrel{fx'=y'}{fx=y}}}{I\otimes G(x',x)\otimes I}
\cong\sum_{\scriptscriptstyle{\stackrel{fx'=y'}{fx=y}}}{G(x',x).}
\end{displaymath}
So the components of $\psi$ are the arrows in $\ca{V}$
\begin{displaymath}
 \psi_{y',y}:\sum_{\scriptscriptstyle{\stackrel{fx'=y'}{fx=y}}}
{I\otimes G(x',x)\otimes I}
\longrightarrow H(y',y)
\end{displaymath}
which, for fixed $x\in f^{-1}(y)$ and $x'\in f^{-1}(y')$, 
correspond uniquely to the components $\phi_{x',x}$.
Hence, a $\ca{V}$-graph arrow can equivalently
be given as a pair $(\psi,f):(G,X)\to(H,Y)$
where $f:X\to Y$ is a function and $\psi:f_*Gf^*\Rightarrow H$
a 2-cell in $\ca{V}$-$\Mat$.

In the established terminology,
the composition of two $\ca{V}$-graph
morphisms
\begin{displaymath}
\ca{G}_X\xrightarrow{\;F=(\phi,f)\;}\ca{H}_Y
\xrightarrow{\;K=(\chi,k)\;}\ca{J}_Z 
\end{displaymath}
is given by the function
$kf:X\to Y\to Z$ and the 
composite 2-cell
\begin{displaymath}
\xymatrix @C=.7in @R=.4in
{X\ar[r]|-{\object@{|}}^-G
\ar[d]|-{\object@{|}}_-{f_*}
\ar @/_7ex/[dd]|-{\object@{|}}_-{(kf)_*}
\drtwocell<\omit>{\phi} & X \\
Y\ar[r]|-{\object@{|}}_-H
^(1.2)\cong ^(-0.2)\cong
\ar[d]|-{\object@{|}}_-{k_*}
\drtwocell<\omit>{\chi}
& Y\ar[u]|-{\object@{|}}_-{f^*} \\
Z\ar[r]|-{\object@{|}}_-J &
Z\ar[u]|-{\object@{|}}_-{k^*}
\ar @/_7ex/[uu]|-{\object@{|}}
_-{(kf)^*}}
\end{displaymath}
where the isomorphisms are $\xi^{f,k}$ 
and $\zeta^{f,k}$ from Lemma \ref{isosofstars}.
The identity arrow on $(G,X)$
is given by the identity function
$\mathrm{id}_X:X\to X$ on the set, and the 
2-cell
\begin{displaymath}
\xymatrix @C=.35in @R=.35in
{X\ar[r]|-{\object@{|}}^-G
\ar[d]|-{\object@{|}}
_-{(\mathrm{id}_X)_*} & X\\
X\ar[r]|-{\object@{|}}_-G 
\rtwocell<\omit>{<-4>\;i_G} & X
\ar[u]|-{\object@{|}}
_-{(\mathrm{id}_X)^*}}
\end{displaymath}
with components arrows in $\ca{V}$
\begin{displaymath}
 (i_G)_{x',x}:G(x',x)\xrightarrow{\;l^{-1}r^{-1}\;}I\otimes G(x',x)\otimes I
\cong G(x',x),
\end{displaymath}
evidently isomorphic to 
the identity arrows $1_{G(x',x)}:G(x',x)\to G(x',x)$.
We write $1_{(G,X)}=(i_G,\mathrm{id}_X)$.
Notice that in fact, the $\ca{V}$-matrices
$(\mathrm{id}_X)_*,(\mathrm{id}_X)^*$ are the same
as the identity 1-cell$\SelectTips{eu}{10}
\xymatrix @C=.2in
{1_X:X\ar[r]|-{\object@{|}} & X}$on $X$:
\begin{equation}\label{idX=1X}
 (\mathrm{id}_X)_*(x',x)=
(\mathrm{id}_X)^*(x',x)=\begin{cases}
I,\quad \mathrm{if  }\;x=x'\\
0,\quad \mathrm{ otherwise}.
\end{cases}
\end{equation} 

We can encode the above data in the 
following isomorphic
characterization
of the category of $\ca{V}$-graphs and 
$\ca{V}$-graph morphisms.
\begin{defi}\label{charactVGrph}
The category of small $\ca{V}$-graphs
$\ca{V}$-$\B{Grph}$ has objects pairs 
$(G,X)\in\ca{V}\textrm{-}\B{Mat}(X,X)\times\B{Set}$
and arrows (in bijection with) pairs 
$(\phi,f):(G,X)\to(H,Y)$
where 
\begin{displaymath}
\begin{cases}
\phi:G\to f^*Hf_* &\textrm{in }\ca{V}\textrm{-}\Mat(X,X)\\
f:X\to Y & \textrm{in }\B{Set}
\end{cases}
\end{displaymath}
or equivalently pairs $(\psi,f)$
where 
\begin{displaymath}
\begin{cases}
\psi:f_*Gf^*\to H &\textrm{in }\ca{V}\textrm{-}\Mat(Y,Y)\\
f:X\to Y & \textrm{in }\B{Set}.
\end{cases}
\end{displaymath}
\end{defi}
From now on, we will use either description
of $\ca{V}$-graph morphisms according to 
our needs, and the choice will be evident by the context
and notation. In particular, we will usually
denote a $\ca{V}$-graph morphism in the classic sense as 
$F_f:\ca{G}_X\to\ca{H}_Y$ and 
$(\phi,f):(G,X)\to(H,Y)$ in the $\ca{V}$-matrices
view.
There is an evident forgetful functor
$Q:\ca{V}\textrm{-}\B{Grph}\to\B{Set}$
which sends each graph $(G,X)$
to its set of objects $X$, and each arrow $(\phi,f)$ to the
function between the objects $f$.

We now continue with the basic
properties of $\ca{V}$-$\B{Grph}$.
First of all, when $\ca{V}$ is complete, 
it is straightforward to 
construct limits inside $\ca{V}$-$\B{Grph}$. 
Indeed, a diagram of shape $\ca{J}$ in $\ca{V}$-$\B{Grph}$
\begin{displaymath}
D:\xymatrix @R=.01in @C=.6in
{\ca{J}\ar[r] & \ca{V}\text{-}\B{Grph} \\
j\ar@{|.>}[r] \ar[dd]_-\theta & 
(\ca{G}_j)_{X_j}\ar[dd]^-{(F_\theta)_{f_\theta}} \\
\hole \\
k\ar@{|.>}[r] & (\ca{G}_k)_{X_k}}
\end{displaymath}
has as limit the graph $\ca{G}_X$ constructed
as follows. The set of objects is 
the limit $X$ of the composite diagram
\begin{displaymath}
 \ca{J}\xrightarrow{\;D\;}\ca{V}\text{-}\B{Grph}
\xrightarrow{\;Q\;}\B{Set},
\end{displaymath}
thus if $\pi_j$ are the projections from $X$,
we have $\pi_k=f_\theta\pi_j$ in $\B{Set}$
for every $\theta$.
Then, for any $x,x'\in X$ the hom-object $G(x',x)$
is the following limit in $\ca{V}$:
\begin{displaymath}
 \xymatrix @C=.8in @R=.3in
{\ca{G}_j(\pi_jx',\pi_jx)
\ar[d]_-{(F_\theta)_{\pi_jx',\pi_jx}} & 
G(x',x)\ar@{.>}[l]_-{(\Pi_j)_{x',x}} 
\ar@{.>}[dl]^-{(\Pi_k)_{x',x}} \\
G_k(f_\theta\pi_jx',f_\theta\pi_jx). & }
\end{displaymath}
The cocone
$\big(\ca{G}_X\xrightarrow{(\Pi_j)_{\pi_j}}
(\ca{G}_j)_{X_j}\,|\,j\in\ca{J}\big)$
now satisfies the required universal property.

On the other hand, when $\ca{V}$ is cocomplete,
the category $\ca{V}$-$\Mat(X,X)$ for any set $X$ is cocomplete
as well, which leads to the following
construction of colimits in $\ca{V}$-$\B{Grph}$.
\begin{prop}[\cite{KellyLack}]\label{Vgrphcocomplete}
The category $\ca{V}$-$\B{Grph}$ is cocomplete
when $\ca{V}$ is.
\end{prop}
\begin{proof}
Suppose $\ca{J}$ is a small category and $F$
is a diagram of shape $\ca{J}$ in $\ca{V}$-$\B{Grph}$
given by
\begin{equation}\label{diagraminVgraph}
 F:\xymatrix @R=.01in @C=.5in
{\ca{J}\ar[r] & \ca{V}\textrm{-}\B{Grph} \\
j\ar@{.>}[r]
\ar[dd]_-\theta & 
(G_j,X_j)\ar[dd]^-{(\psi_\theta,f_\theta)} \\
\hole \\
k\ar@{.>}[r] & (G_k,X_k).}
\end{equation}
By Definition \ref{charactVGrph}, $f_\theta$ is a function
between the sets of objects
and $(f_\theta)_*G_j(f_\theta)^*\stackrel{\psi_\theta}{\Rightarrow}$ 
$G_k$ 
is a 2-cell in $\ca{V}$-$\Mat$.
Again, the composite 
\begin{displaymath}
\ca{J}\xrightarrow{\;F\;}
\ca{V}\text{-}\B{Grph}\xrightarrow{\;Q\;}\B{Set}
\end{displaymath}
has a colimiting
cocone $(\tau_j:X_j\to X\,|\,j\in\ca{J})$
in the cocomplete $\B{Set}$. Notice that, since 
$\tau_j=f_\theta\tau_k$ 
for any $f_\theta:X_j\to X_k$, we have 
isomorphisms of $\ca{V}$-matrices 
\begin{displaymath}
 \xymatrix
{X_j\ar[rr]|-{\object@{|}}^-{(\tau_j)_*}
_-{\stackrel{\stackrel{\phantom{A}}{\zeta}}{\cong}}
\ar[dr]|-{\object@{|}}_-{(f_\theta)_*} && X, \\
& X_k\ar[ur]|-{\object@{|}}_-{(\tau_k)_*} &}\qquad
\xymatrix
{X\ar[rr]|-{\object@{|}}^-{(\tau_j)^*}
_-{\stackrel{\stackrel{\phantom{A}}{\xi}}{\cong}}
\ar[dr]|-{\object@{|}}_-{(\tau_k)^*} && X_j \\
& X_k\ar[ur]|-{\object@{|}}_-{(f_\theta)^*} &}
\end{displaymath}
where $\zeta$ and $\xi$ are defined as in Lemma \ref{zeta}.
Now consider the functor 
\begin{equation}\label{defKdiagram}
K:\xymatrix @C=.6in @R=.02in
{\ca{J}\ar[r]
& \ca{V}\textrm{-}
\B{Mat}(X,X)\qquad\qquad\qquad\qquad\quad \\
j\ar @{|.>}[r] \ar[dd]_-{\theta}& 
(\tau_j)_*G_j(\tau_j)^*
{\scriptstyle{\cong}}(\tau_k)_*(f_\theta)_*G_j
(f_\theta)^*(\tau_k)^*
\ar@<-14ex>
[dd]^-{(\tau_k)_*\psi_\theta(\tau_k)^*}\\
\hole \\
k\ar@{.>}[r] & 
(\tau_k)_*G_k(\tau_k)^*
\phantom{\;\cong(\tau_k)_*(f_\theta)_*G_j
(\tau_k)^*(f_\theta)^*}}
\end{equation}
which explicitly maps an arrow $\theta:j\to k$ in $\ca{J}$
to the composite 2-cell
\begin{equation}\label{Konarrows}
\xymatrix @R=.4in @C=.8in
{X\ar@/_3.5ex/[dr]|-{\object@{|}}_-{(\tau_k)^*}
\drtwocell<\omit>{'\stackrel{\xi}{\cong}}
\ar[r]|-{\object@{|}}^-{(\tau_j)^*} &
 X_j\ar[r]|-{\object@{|}}
^-{G_j} & X_j \ar[d]|-{\object@{|}}
^-{(f_\theta)_*}\ar[r]|-{\object@{|}}^-
{(\tau_j)_*} & X.
\dltwocell<\omit>{'\stackrel{\zeta}{\cong}}\\
 & X_k\ar[u]|-{\object@{|}}
^-{(f_\theta)^*}\ar[r]|-{\object@{|}}_-{G_k}
\rtwocell<\omit>{<-4>\;\psi_\theta} & 
X_k\ar@/_3.5ex/[ur]|-{\object@{|}}_-{(\tau_k)_*} &}
\end{equation}
The colimit of $K$ is formed pointwise in 
$\ca{V}$-$\B{Mat}(X,X)=
[X\times X,\ca{V}]$,
so there is a colimiting cocone 
$(\lambda_j:(\tau_j)_*G_j(\tau_j)^*\to
G\,|\,j\in\ca{J})$.
These data allow us to form a new cocone 
\begin{displaymath}
\big((G_j,X_j)\xrightarrow{(\lambda_j,\tau_j)}
(G,X)\,|\,j\in\ca{J}\big)
\end{displaymath}
for the initial diagram $F$ in $\ca{V}$-$\B{Grph}$, since
$(G,X)$ is an endoarrow in $\ca{V}$-$\Mat$ by construction, 
and also the pairs $(\lambda_j,\tau_j)$ 
commute accordingly with the $(\psi_\theta,f_\theta)$'s. 
This cocone which can be checked
to be colimiting, since
$\tau_j$ and $\lambda_j$ are.
Therefore $(G,X)$
satisfies the universal property of a 
colimit of $F$ in 
$\ca{V}$-$\B{Grph}$.
\end{proof}
The above construction is presented in \cite{VarThrEnr},
again in the more general case of enrichment in a bicategory.
The existence of all colimits in $\ca{V}$-$\B{Grph}$
was also shown in \cite{Wolff} via 
the explicit construction of coproducts and coqualizers.

The category $\ca{V}$-$\B{Grph}$ 
has a monoidal structure inherited from $\ca{V}$:
given two $\ca{V}$-graphs $\ca{G}_X$ and $\ca{H}_Y$,
their tensor product $\ca{G}\otimes\ca{H}$
is defined to be the $\ca{V}$-graph 
with set of objects $X\times Y$
and hom-objects
\begin{displaymath}
(\ca{G}\otimes\ca{H})((z,w),(x,y)):=G(z,x)\otimes H(w,y).
\end{displaymath}
Of course, this comes from the monoidal structure 
of the bicategory $\ca{V}$-$\Mat$ as in 
(\ref{monoidalVMat}).
Similarly, we can define the tensor product
of two $\ca{V}$-graph morphisms: given $\ca{V}$-graph
arrows $F_f:\ca{G}_X\to\ca{H}_Y$, 
$D_d:\ca{G}'_{X'}\to\ca{H}'_{Y'}$,
their tensor product
\begin{displaymath}
 F\otimes D:\ca{G}\otimes\ca{H}\to\ca{G}'\otimes\ca{H}'
\end{displaymath}
is given by the function $f\times d:X\times Y\to X'\times Y'$
between their sets of objects, and for every $x,z\in X$,
$y,w\in Y$, arrows
\begin{displaymath}
(F\otimes D)_{(z,w),(x,y)}:G(z,x)\otimes H(w,y)
\xrightarrow{\;F_{z,x}\otimes D_{w,y}\;}
G(fz,fx)\otimes H(dw,dy)
\end{displaymath}
in $\ca{V}$. The monoidal unit is the unit 
$\ca{V}$-graph $\ca{I}$ with one object, 
and hom-object $\ca{I}(*,*)=I$.
Also, symmetry is also evidently inherited
from $\ca{V}$.

Furthermore, the category of $\ca{V}$-graphs
is a monoidal closed category if we assume
certain extra conditions on $\ca{V}$.
\begin{prop}\label{VGrphclosed}
Suppose $\ca{V}$ is a monoidal closed category with 
small products. The functor
\begin{displaymath}
^g\Hom:\ca{V}\textrm{-}\B{Grph}^\mathrm{op}\times
\ca{V}\textrm{-}\B{Grph}\to\ca{V}\textrm{-}\B{Grph}
\end{displaymath}
which maps a pair $(\ca{G}_X$, $\ca{H}_Y)$ to 
the $\ca{V}$-graph 
${^g}\Hom(\ca{G},\ca{H})_{Y^X}$ with
\begin{displaymath}
^g\Hom(\ca{G},\ca{H})(k,s):=
\prod_{\scriptscriptstyle{\stackrel{x'\in X}{x\in X}}}
{[\ca{G}(x',x),\ca{H}(kx',sx)]}
\end{displaymath}
for $k,s\in Y^X$ is the internal hom of $\ca{V}$-$\B{Grph}$. 
\end{prop}
\begin{proof}
In order to establish an adjunction
$(-\otimes\ca{H}_Y)\dashv{^g\Hom(\ca{H}_Y,-)}$
for any $\ca{V}$-graph $\ca{H}_Y$,
take a $\ca{V}$-graph morphism
$F_f:\ca{G}_X\to{^g\Hom(\ca{H}_Y,\ca{J}_Z)}$.
This consists of a function $f:X\to Z^Y$
between the sets of objects, and arrows 
\begin{displaymath}
F_{x',x}:\ca{G}(x',x)\to\prod_{y,y'\in Y}{[(\ca{H}(y',y),
\ca{J}(f_{x'}y',f_xy))}
\end{displaymath}
in $\ca{V}$ between the hom-objects, where
$f_x=f(x):Y\to Z$, for all $x,x'\in X$. 
These arrows correspond
bijectively, under the tensor-hom adjunction in $\ca{V}$ 
for a fixed pair
of elements $(y,y')\in Y$, to
\begin{displaymath}
\ca{G}(x',x)\otimes\ca{H}(y',y)\to\ca{J}(f_{x'}y',f_xy)
\end{displaymath}
since $\ca{V}$ is monoidal closed.
The category $\B{Set}$ is cartesian closed,
thus the function $f$ corresponds
uniquely to a function $\bar{f}:X\times Y\to Z$.
This function together with the arrows above written as
\begin{displaymath}
\bar{F}_{(x',y'),(x,y)}:
\ca{G}(x',x)\otimes\ca{H}(y',y)\to\ca{J}\big(\bar{f}(x',y'),\bar{f}(x,y)\big)
\end{displaymath}
determines a $\ca{V}$-graph morphism 
$\bar{F}_{\bar{f}}:\ca{G}_X\otimes\ca{H}_Y\to\ca{J}_Z$
which establishes a bijective correspondence
\begin{displaymath}
\ca{V}\textrm{-}\B{Grph}(\ca{G}_X\otimes
\ca{H}_Y,\ca{J}_Z)\cong
\ca{V}\textrm{-}\B{Grph}(\ca{G}_X,
{^g\Hom(\ca{H}_Y,\ca{J}_Z))}.
\end{displaymath}
Moreover, this bijection is natural in $\ca{G}_X$, 
hence $^g\Hom(\ca{H},\ca{J})$ is the object
function of a right adjoint functor 
$^g\Hom(\ca{H},-)$ of $(-\otimes\ca{H})$. 
Hence the induced functor of two variables $^g\Hom$
is the parametrized adjoint of $\otimes$. 

Explicitly, $^g\Hom$ on a pair of $\ca{V}$-arrows
$(F_f:\ca{J}_Z\to\ca{G}_X,D_d:\ca{H}_Y\to\ca{M}_W)$
gives a $\ca{V}$-graph morphism
\begin{equation}\label{defHom(F,D)}
^g\Hom(F,D):{^g\Hom}(\ca{G},\ca{H})_{Y^X}
\longrightarrow{^g\Hom}(\ca{J},\ca{M})_{W^Z}.
\end{equation}
This consists of the function 
`pre-composing with $f$ and post-composing
with $d$' $d^f:Y^X\to W^Z$
between the sets of objects,
and for each pair 
$(k,s)\in Y^X$ an arrow
\begin{align*}
{^g\Hom}(F,D)_{k,s}:&\;{^g\Hom}(\ca{G},\ca{H})(k,s)\longrightarrow
{^g\Hom}(\ca{J},\ca{M})(d^f(k),d^f(s))\equiv\\
&\prod_{x,x'\in X}
{[G(x',x),H(kx',sx)]}\to\prod_{z,z'\in Z}
{[J(z',z),M(dkfz',dsfz)]}.
\end{align*}
For fixed $z,z'\in Z$, the latter corresponds uniquely
under the usual tensor-hom adjunction
to the composite
\begin{displaymath}
\xymatrix @C=.8in @R=.4in
{\prod\limits_{x,x'\in X}{[\ca{G}(x',x),\ca{H}(kx',sx)]}
\otimes\ca{J}(z',z)\ar @{-->}[r]
\ar[d]_-{1\otimes F_{z,z'}} & \ca{M}(dkfz',dsfz). \\
\prod\limits_{x,x'\in X}{[\ca{G}(x',x),\ca{H}(kx',sx)]}
\otimes\ca{G}(fz',fz)\ar[d]_-{\pi_{fz',fz}\otimes 1} & \\
[\ca{G}(fz',fz),\ca{H}(kfz',sfz)]\otimes\ca{G}(fz',fz) \ar[r]^-{\textrm{ev}}
& \ca{H}(kfz',sfz)\ar[uu]_-{D_{kfz',sfz}}}
\end{displaymath}
\end{proof}
In the above proof, there was no need to 
move to the world of $\ca{V}$-matrices. If we did,
however, it would be clear that the mapping of 
the functor $^g\Hom$ on two objects
$(G,X)$ and $(H,Y)$ is in fact the mapping 
of the functor
$\Hom_{(X,Y),(X,Y)}$ (\ref{Hom_}) provided 
by the lax functor of bicategories
$\Hom:(\ca{V}\textrm{-}\Mat)^{\textrm{co}}
\times\ca{V}\textrm{-}\Mat\to
\ca{V}\textrm{-}\Mat$ defined
explicitly in the previous section.
For the mapping on morphisms though,
the definition of $\Hom(\sigma,\tau)$ as in 
(\ref{Hom2cells}) is not sufficient, because 
the morphisms in $\ca{V}$-$\B{Grph}$
are not just between endoarrows
in $\ca{V}$-$\Mat$ with the same set of objects.
Hence, in terms of $\ca{V}$-matrices,
for $F=(\phi,f)$ and $D=(\chi,d)$ as in Definition 
\ref{charactVGrph}, the $\ca{V}$-graph
arrow $^g\Hom((\phi,f),(\chi,d))$ is the pair
$([[\phi,\chi]],d^f)$ where
\begin{equation}\label{def[[]]}
\xymatrix @C=.8in @R=.6in 
{Y^X\ar[r]|-{\object@{|}}
^-{\Hom(G,H)}\ar[d]|-{\object@{|}}
_-{(d^f)_*}\rtwocell<\omit>{<6>\qquad[[\phi,\chi]]} & Y^X\\
W^Z\ar[r]|-{\object@{|}}_-{\Hom(J,M)} & W^Z
\ar[u]|-{\object@{|}}_-{(d^f)^*}}
\end{equation}
has components isomorphic to $^g\Hom(F_f,D_d)_{k,s}$
up to tensoring with $I$'s on both sides 
of the codomain product.

Another important property of $\ca{V}$-$\B{Grph}$
is the fact that it inherits local presentability
from $\ca{V}$. The detailed
arguments and constructions for this result
can be found in \cite{KellyLack}.
\begin{prop}~\cite[4.4]{KellyLack}\label{VGrphlocpresent}
The category $\ca{V}$-$\B{Grph}$ is locally
$\lambda$-presentable when $\ca{V}$ is so.
\end{prop} 
\begin{proof}(Sketch) Suppose $\ca{V}$ 
is a locally $\lambda$-presentable category.
Then, if the set $\ps{G}$ of objects constitutes
a strong generator of $\ca{V}$, it can be shown that
the set
\begin{displaymath}
 \{(\bar{G},2)\;/\;G\in\ps{G}\;\textrm{or}\;G=0\}
\end{displaymath}
constitutes a strong generator of 
$\ca{V}$-$\B{Grph}$, where
the graph $(\bar{G},2)$ has as set of objects
$2=\{0,1\}$ and consists of the objects
\begin{displaymath}
 \bar{G}(0,0)=G,\;\; \bar{G}(0,1)=
\bar{G}(1,0)=\bar{G}(1,1)=0 
\end{displaymath}
in $\ca{V}$. Also, this set is $\lambda$-presentable,
in the sense that the hom-functors
\begin{displaymath}
 \ca{V}\textrm{-}\B{Grph}((\bar{G},2),-):\ca{V}\textrm{-}\B{Grph}
\to\B{Set}
\end{displaymath}
preserve $\lambda$-filtered colimits.
\end{proof}

\section{$\ca{V}$-categories and $\ca{V}$-cocategories}\label{VcatsandVcocats}

In Chapter \ref{enrichment},
we recalled what it means for a category
$\ca{A}$ to be $\ca{V}$-enriched
for a monoidal category $\ca{V}$. 
In this section, we are going to re-define $\ca{V}$-categories from
a slightly different perspective, 
in the context of $\ca{V}$-matrices.
This is of importance because it allows us,
just by dualizing certain arguments,
to later construct the category of $\ca{V}$-cocategories
in a natural way. Evidently,
the motivation for this
is that enriched categories and cocategories 
generalize monoids and comonoids in a monoidal category,
since for example it is well-known that
a one-object $\ca{V}$-category is precisely
an object in $\Mon(\ca{V})$.

Notice that strictly speaking, composition in the 
bicategory $\ca{V}$-$\Mat$ (\ref{horizontalcompositionVmatrices})
results in the opposite convention (\ref{compositionlaw2}) to 
that preferred by Kelly (\ref{compositionlawplain})
for the composition law in an enriched category.
Similar issues arise regarding $\ca{V}$-modules later.
There seems to be no consistent practice in these matters.

Following once again the approach of \cite{VarThrEnr},
a $\ca{V}$-category is defined to be a monad in
the bicategory $\ca{V}$-$\B{Mat}$. Unravelling 
Definition \ref{monadbicat}, it consists 
of a set $X$ together with an endoarrow$\SelectTips{eu}{10}
\xymatrix @C=.2in
{A:X\ar[r]|-{\object@{|}} & X,}$\emph{i.e.} it 
is a $\ca{V}$-graph with
set of objects $\ob\ca{A}=X$, equipped
with two 2-cells,
the multiplication and the unit
\begin{displaymath}
\xymatrix @R=.1in @C=.4in
{& X \ar[dr]|-{\object@{|}}
^-{A} &\\
X\ar[ru]|-{\object@{|}}^-A
\ar @/_/[rr]|-{\object@{|}}_-A
\rrtwocell<\omit>{<-1.3>\;M} && X}
\quad
\xymatrix @C=.3in @R=.1in
{\hole \\ \textrm{and} }
\quad
\xymatrix @C=.3in @R=.1in
{\hole \\
X \rrtwocell<\omit>{\eta}
\ar @/^2.2ex/ [rr]|-{\object@{|}}^-{1_X}
\ar @/_2.2ex/ [rr]|-{\object@{|}}_-A && X}
\end{displaymath}
satisfying the following axioms:
\begin{displaymath}
\xymatrix @R=.2in
{& X\ar[r]|-{\object@{|}}^-A &
X\ar @/^/[dr]|-{\object@{|}}^-A & \\
X\ar @/^/[ur]|-{\object@{|}}^-A 
\ar @/_2ex/[urr]|-{\object@{|}}_-A
\ar @/_3ex/[rrr]|-{\object@{|}}_-A
\urrtwocell<\omit>{<-.3>\;M}
&\urrtwocell<\omit>{<1.3>\;M} && X}
\quad\xymatrix @R=.2in
{\hole\\
=}\quad
\xymatrix @R=.2in
{\drrtwocell<\omit>{<1.3>\;M} & 
X\ar[r]|-{\object@{|}}^-A 
\ar @/_2ex/[drr]|-{\object@{|}}_-A 
\drrtwocell<\omit>{<-.3>\;M} &
X\ar @/^/[dr]|-{\object@{|}}^-A & \\
X\ar @/^/[ur]|-{\object@{|}}^-A 
\ar @/_3ex/[rrr]|-{\object@{|}}_-A
&&& X,}
\end{displaymath}
\begin{displaymath}
\xymatrix @C=.5in @R=.2in
{& X \ar @/^/[dr]|-{\object@{|}}^-A &\\
X\urrtwocell<\omit>{<1.3>\;M}
\urtwocell<\omit>{\eta}
\ar @/^2ex/[ur]|-{\object@{|}}^-{1_X}
\ar @/_2ex/[ur]|-{\object@{|}}_-A 
\ar @/_2ex/[rr]|-{\object@{|}}_-A
&& X}
\xymatrix @C=.5in @R=.2in
{\hole \\
=}
\xymatrix @C=.5in @R=.2in
{\hole\\
X\rtwocell<\omit>{\;1_A}
\ar @/^2.3ex/[r]|-{\object@{|}}^-A
\ar @/_2.3ex/[r]|-{\object@{|}}_A 
& X}
\xymatrix @C=.5in @R=.2in
{\hole \\
=}
\xymatrix @C=.4in @R=.2in
{\drrtwocell<\omit>{<1.3>\;M}
& X \ar @/^2ex/[dr]|-{\object@{|}}^-{1_X} 
\ar @/_2ex/[dr]|-{\object@{|}}_-A 
\drtwocell<\omit>{\eta} &\\
X \ar @/^/[ur]|-{\object@{|}}^-A
\ar @/_2ex/[rr]|-{\object@{|}}_-A
&& X.}
\end{displaymath}
Notice that in the above diagrams,
the associator and the unitors of the bicategory 
$\ca{V}$-$\Mat$ which
are essential for the domains and codomains
of the equal 2-cells to coincide, are
suppressed. In terms of components, they are given by
\begin{align}\label{compositionlaw2}
M_{z,y,x}&:\sum_{y\in X}{A(z,y)\otimes
A(y,x)}\longrightarrow A(z,y) \\
\eta_x&:I\longrightarrow A(x,x) \notag
\end{align}
which are the usual composition law and identity
elements.
If we also express the above relations that $M$ and $\eta$
have to satisfy in terms of 
components of the 2-cells involved, we re-obtain
the associativity and unit axioms of an enriched category.
Also by Remark \ref{monadsaremonoids}, a monad
in a bicategory is the same as a monoid in 
the appropriate endoarrow hom-category, \emph{i.e.} a $\ca{V}$-category
$\ca{A}$ with set of objects $X$ is a monoid
in the monoidal category
($\ca{V}$-$\Mat(X,X)$,$\circ$,$1_X$). 
Denote a $\ca{V}$-category as a pair $(A,X)$ or $\ca{A}_X$.

A $\ca{V}$-functor $F:\ca{A}\to\ca{B}$ between two
$\ca{V}$-categories $\ca{A}_X$ and $\ca{B}_Y$ 
was again defined in Section \ref{basicdefienrichment},
and in fact is a $\ca{V}$-graph morphism $F_f:\ca{A}_X\to\ca{B}_Y$
(in the classic sense) which respects the composition law and the identities.
In the current context of $\ca{V}$-matrices,
a $\ca{V}$-functor can be defined
to be a morphism of $\ca{V}$-graphs 
$(\phi,f):(A,X)\to(B,Y)$ as in
Definition \ref{charactVGrph},
which satisfies
\begin{align}\label{functaxioms}
\xymatrix @C=.5in @R=.5in
{X\ar[r]|-{\object@{|}}^-A
\ar[d]|-{\object@{|}}_-{f_*}
\drtwocell<\omit>{\hat{\phi}}
& X\ar[r]|-{\object@{|}}^-A
\ar[d]|-{\object@{|}}^-{f_*}
\drtwocell<\omit>{\hat{\phi}}
& X \ar[d]|-{\object@{|}}^-{f_*}\\
Y \ar[r]|-{\object@{|}}_-B
\ar @/_6ex/[rr]|-{\object@{|}}_-B
\rrtwocell<\omit>{<3>\;M}
& Y \ar[r]|-{\object@{|}}_-B
& Y} 
&\xymatrix @R=.2in{\hole \\ = \\ \hole}
\xymatrix @C=.5in @R=.3in
{& X \ar @/^/[dr]|-{\object@{|}}
^-{A} &\\
X\ar[d]|-{\object@{|}}_-{f_*}
\ar @/^/[ru]|-{\object@{|}}^-A
\ar[rr]|-{\object@{|}}_-A
\rrtwocell<\omit>{<-3>\;M}
\drrtwocell<\omit>{\hat{\phi}} && X
\ar[d]|-{\object@{|}}^-{f_*}\\
Y \ar[rr]|-{\object@{|}}_-B && Y,} \\
\xymatrix @C=1.2in @R=.5in
{X\ar @/^5ex/[r]|-{\object@{|}}
^-{1_X} \rtwocell<\omit>{<-2>\eta}
\ar[r]|-{\object@{|}}
_-{A}\ar[d]|-{\object@{|}}
_-{f_*} & X \ar[d]|-{\object@{|}}
^-{f_*} \\
Y\ar[r]|-{\object@{|}}_-{B} 
\rtwocell<\omit>{<-4>\hat{\phi}} & Y}
&\xymatrix{ = \\ \hole}
\xymatrix @C=1.2in  @R=.5in
{X\ar[r]|-{\object@{|}}
^-{1_X}
\ar[d]|-{\object@{|}}_-{f_*}
^{\phantom{ab}\cong}
\ar @{.>}[dr]|-{\object@{|}}^-{f_*}
& X
\ar[d]|-{\object@{|}}^-{f_*}
_{\cong\phantom{ab}} \\
Y\ar[r]|-{\object@{|}}^-{1_Y}
\ar @/_5ex/ [r]|-{\object@{|}}_-B
\rtwocell<\omit>{<2>\eta} & Y.}\notag
\end{align}
Here, the 2-cell $\hat{\phi}:f_*A\Rightarrow Bf_*$ 
corresponds bijectively to $\phi$
via mates correspondence `on the one side',
\emph{i.e.} by pasting the counit $\check{\varepsilon}$
of $f_*\dashv f^*$ on the right.
This description agrees with the standard
$\ca{V}$-functor definition
up to isomorphism again: the 2-cell $\bar{\phi}$
has components
$$\bar{\phi}_{y,x}:I\otimes A(x',x)\to B(fx',fx)\otimes I$$
for $x'\in f^{\text{-1}}y$,
and the equality of the above pasted diagrams
agrees with the commutative diagrams
(\ref{Venrichedfunctordiagrams}) 
up to tensoring the objects
with $I$'s and composing the arrows
with the left and right unit constraints of $\ca{V}$.
\begin{rmk}\label{Vfunct=monadopfunct}
The pair $(f_*,\hat{\phi})$
is a special case of `colax monad functor'
between the monads $(A,X)$ and $(B,Y)$ in 
the bicategory $\ca{V}$-$\B{Mat}$,
as in Definition \ref{monadfunctor}.
However, it is not true 
that any colax monad functor
given by the data
\begin{displaymath}
\xymatrix @C=.45in @R=.45in
{X\ar[r]|-{\object@{|}}^-A
\ar[d]|-{\object@{|}}_-S
\rtwocell<\omit>{<5>\chi}
& X\ar[d]|-{\object@{|}}^-S\\
Y\ar[r]|-{\object@{|}}_-B &
Y}
\end{displaymath}
for some $\ca{V}$-matrix $S$ can be seen 
as a $\ca{V}$-functor, since it is obviously
not true that any$\SelectTips{eu}{10}\xymatrix @C=.2in
{S:X\ar[r]|-{\object@{|}} & Y}$is of the form $f_*$ 
for some function $f:X\to Y$.
This explains why the category $\ca{V}$-$\B{Cat}$
cannot be characterized as $\B{Mnd}(\ca{V}$-$\B{Mat})$,
even if they have the same objects. Similar issues 
were discussed in a bigger depth in \cite{GarnerShulman}, 
employing the theory of \emph{proarrow equipments}.
\end{rmk}
There is a 2-dimensional
aspect for all the basic categories 
we study in this chapter, 
including $\ca{V}$-$\B{Cat}$. However, we choose to omit its description
in this treatment, because it is not of central
importance for our main results. More specifically,
for the enrichment relations and the fibrational structures
we explore, the 2-categorical structure of those categories
is unnecessary.

Since a $\ca{V}$-category with set of 
objects $X$ can be seen as a monoid
in the monoidal category $\ca{V}$-$\Mat(X,X)$,
a similar characterization for $\ca{V}$-functors
could be attempted, in order to obtain
a result analogous to Definition \ref{charactVGrph}
for $\ca{V}$-$\B{Grph}$. The following
is indicative of how to proceed.
\begin{lem}\label{B*monoid}
Let $(B,Y)$ be a $\ca{V}$-category. If 
$f:X\to Y$ is any function, the composite $\ca{V}$-matrix
\begin{displaymath}
\SelectTips{eu}{10}\xymatrix @C=.3in
{X\ar[r]|-{\object@{|}}^-{f_*} &
Y\ar[r]|-{\object@{|}}^-B &
Y\ar[r]|-{\object@{|}}^-{f^*} &
X}
\end{displaymath}
is a monoid in $\ca{V}$-$\B{Mat}(X,X)$, \emph{i.e.}
the pair $(f^*Bf_*,X)$ constitutes a $\ca{V}$-category.
\end{lem}
\begin{proof}
The multiplication $M':f^*Bf_*f^*Bf_*\to f^*Bf_*$
is given by the composite 2-cell
\begin{displaymath}
\xymatrix @C=.4in @R=.05in
{&&& X\ar @/^/[dr]|-{\object@{|}}^-{f_*} &&&\\
&& Y\ar @/^/[ur]|-{\object@{|}}^-{f^*} 
\ar @/_3ex/[rr]|-{\object@{|}}_-{1_Y}
\rrtwocell<\omit>{\check{\varepsilon}} 
&& Y\ar @/^/[dr]|-{\object@{|}}^-B && \\
X\ar[r]|-{\object@{|}}^-{f_*}
&Y\ar @/^/[ur]^-B
\ar @/_6ex/[rrrr]|-{\object@{|}}_-B
& \rrtwocell<\omit>{<3>\;M} &&&
Y\ar[r]|-{\object@{|}}^-{f^*} & X}
\end{displaymath}
and the unit $\eta':1_X\to f^*Bf_*$ is given by 
the composite 2-cell
\begin{displaymath}
\xymatrix @R=.1in
{X\ar @/^3ex/[rrr]|-{\object@{|}}^-{1_X} 
\ar[dr]|-{\object@{|}}_-{f_*} &&& X \\
& Y\ar @/^2ex/[r]|-{\object@{|}}^-{1_Y}
\ar @/_2ex/[r]|-{\object@{|}}_-B 
\rtwocell<\omit>{\eta} 
\rtwocell<\omit>{<-5>\check{\eta}} &
Y,\ar[ur]|-{\object@{|}}_-{f^*} &}
\end{displaymath}
where $\check{\varepsilon}$ and $\check{\eta}$
are the counit and unit of the adjunction
$f_*\dashv f^*$ in $\ca{V}$-$\B{Mat}$, and $M$ and
$\eta$ the structure maps of the monoid $B$. 
Using pasting operations, 
the new multiplication and unit
can be expressed as 
\begin{gather*}
M'=f^*\big(M\cdot(B\check{\varepsilon}B)\big)f_*, \\
\eta'=(f^*\eta f_*)\cdot\check{\eta}.
\end{gather*}
The associativity
and unit axioms follow from the ones for
the multiplication and unit of the 
monoid$\SelectTips{eu}{10}\xymatrix @C=.2in
{B:Y\ar[r]|-{\object@{|}} &Y}$and the 
triangular identities for $\check{\eta}$ and 
$\check{\varepsilon}$.
\end{proof}
It is not hard to see that the diagrams
(\ref{functaxioms})
which a $\ca{V}$-functor $F=(\phi,f):(A,X)\to(B,Y)$ 
has to satisfy,
coincide with the ones that 
an arrow in $\Mon(\ca{V}$-$\B{Mat}(X,X))$
between the monoids $A$ and $f^*Bf_*$
has to satisfy. 
For example, associativity can be written,
using mates correspondence, as
\begin{displaymath}
\xymatrix @C=.4in @R=.4in
{X\ar[r]|-{\object@{|}}^-A
\ar[d]|-{\object@{|}}_-{f_*}
\rtwocell<\omit>{<5>\phi}
& X \ar@/^/[dr]|-{\object@{|}}^-{f_*}
\ar[rr]|-{\object@{|}}^-A
& \rtwocell<\omit>{<5>\phi}
& X \ar@/^/[dr]|-{\object@{|}}^-{f_*} &\\
Y \ar[r]|-{\object@{|}}_-B
\ar @/_6ex/[rrr]|-{\object@{|}}_-B
& Y \ar[u]|-{\object@{|}}_-{f^*}
\rtwocell<\omit>{<3.5>\;M}
\ar[r]|-{\object@{|}}_-{1_Y}
\rtwocell<\omit>{<-3>\check{\varepsilon}}
& Y \ar[r]|-{\object@{|}}_-B
& Y \ar[u]|-{\object@{|}}_-{f^*}
\ar[r]|-{\object@{|}}_-{1_Y}
\rtwocell<\omit>{<-3>\check{\varepsilon}} &Y}
\xymatrix @R=.2in{\hole \\ = \\ \hole}
\xymatrix @C=.4in @R=.3in
{& X \ar @/^/[dr]|-{\object@{|}}
^-{A} &&\\
X\ar[d]|-{\object@{|}}_-{f_*}
\ar @/^/[ru]|-{\object@{|}}^-A
\ar[rr]|-{\object@{|}}_-A
\rrtwocell<\omit>{<-3>\;M}
\rrtwocell<\omit>{<4>\phi} && X
\ar@/^/[dr]|-{\object@{|}}^-{f_*} &\\
Y \ar[rr]|-{\object@{|}}_-B && Y
\ar[u]|-{\object@{|}}_-{f^*}
\ar[r]|-{\object@{|}}_-{1_Y}
\rtwocell<\omit>{<-2>\check{\varepsilon}}
& Y}
\end{displaymath}
which implies the commutativity of the first diagram
in (\ref{monoidarrowaxioms}) for a monoid
morphism, taking into account the form of
multiplication $M'$ of $f^*Bf_*$.
Therefore, the following
characterization of the category of $\ca{V}$-categories
is established.
\begin{lem}\label{charactVCat}
The objects of $\ca{V}$-$\B{Cat}$ 
are pairs 
\begin{displaymath}
(A,X)\in
\Mon(\ca{V}\textrm{-}\B{Mat}(X,X))\times\B{Set} 
\end{displaymath}
and morphisms are pairs
$(\phi,f):(A,X)\to(B,Y)$ where
\begin{displaymath}
\begin{cases}
\phi:A\to f^*Bf_* &\textrm{in }\Mon(\ca{V}\textrm{-}\B{Mat}(X,X))\\
f:X\to Y & \textrm{in }\B{Set}.
\end{cases}
\end{displaymath}
\end{lem}
As in the case of $\ca{V}$-$\B{Grph}$
in the previous section, the category $\ca{V}$-$\B{Cat}$
as presented in Chapter \ref{enrichment}
is in fact isomorphic with the category
described above, in the sense
that there is a bijection between objects (\emph{i.e.}
the identity)
and a bijection between arrows of these categories.

We already saw how $\ca{V}$-$\B{Cat}$ inherits
a (symmetric) monoidal structure from $\ca{V}$.
The tensor product of the 
$\ca{V}$-categories $\ca{A}_X$ and $\ca{B}_Y$
is defined to be the $\ca{V}$-graph 
$(\ca{A}\otimes\ca{B})_{X\times Y}$,
given by the family of objects 
$\{\ca{A}(z,x)\otimes\ca{B}(w,y)\}$
in $\ca{V}$ for all $x,z\in X$ and $y,w\in Y$, with
composition law and identities as given in 
Section \ref{basicdefienrichment}.

Similarly to the free monoid construction
on an object in a monoidal category $\ca{V}$,
briefly discussed in Section 
\ref{Categoriesofmonoidsandcomonoids},
we now proceed to the description of an endofunctor
on $\ca{V}$-$\B{Grph}$ inducing the 
`free $\ca{V}$-category' monad. The following proof 
can also be found in \cite{VarThrEnr,KellyLack}.
\begin{prop}\label{freeVcatfunctor}
Let $\ca{V}$ be a monoidal category with coproducts, 
such that $\otimes$ preserves them
on both sides.
The functor 
\begin{displaymath}
\tilde{S}:\ca{V}\textrm{-}\B{Cat}\to\ca{V}\textrm{-}\B{Grph}
\end{displaymath}
which forgets composition and identities
has a left adjoint $\tilde{L}$, which maps a 
$\ca{V}$-graph$\SelectTips{eu}{10}\xymatrix @C=.2in
{G:X\ar[r]|-{\object@{|}} & X}$to the geometric series
\begin{displaymath}
\SelectTips{eu}{10}\xymatrix @C=.3in
{\sum\limits_{n\in\caa{N}}{G^{\otimes n}}:X\ar[r]
|-{\object@{|}} & X.}
\end{displaymath}
\end{prop}
\begin{proof}
Recall that by Proposition \ref{propVMat},
$\ca{V}$-$\B{Mat}(X,X)$ admits the same 
class of colimits as $\ca{V}$, 
and also $\otimes=\circ$
preserves colimits on both sides.
Hence, the forgetful functor $S$
from its category of monoids has a left adjoint,
namely the `free monoid' functor, as in
Proposition \ref{freemonoidprop}:
\begin{displaymath}
L:\xymatrix @R=.02in
{\ca{V}\textrm{-}\Mat(X,X)\ar[r] & 
\Mon(\ca{V}\textrm{-}\Mat(X,X))\\
G\ar@{|->}[r] &
\sum_{n\in\caa{N}}{G^n}.}
\end{displaymath}
By Lemma \ref{charactVCat},
we deduce that this geometric series
is in fact a $\ca{V}$-category
with set of objects $X$. 
We now claim that the mapping
\begin{equation}\label{deftildeL}
\tilde{L}:\xymatrix @R=.02in @C=.4in
{\ca{V}\textrm{-}\B{Grph}\ar[r] & 
\ca{V}\textrm{-}\B{Cat}\\
(G,X)\ar@{|->}[r] &
(LG,X)}
\end{equation}
induces a left adjoint of the forgetful functor 
$\tilde{S}$.
For that, it is enough to show that the $\ca{V}$-graph
morphism
$\tilde{\eta}:(G,X)\to\tilde{S}\tilde{L}(G,X)$
which is the identity function on objects and 
the injection 2-cell of the summand $G$ into the series,
has the following universal property:
if $(B,Y)$ is a $\ca{V}$-category and
$F$ is a $\ca{V}$-graph arrow from $(G,X)$
to its underlying 
$\ca{V}$-graph $\tilde{S}(B,Y)$, then there exists a unique
$\ca{V}$-functor $H:(LG,X)\to(B,Y)$ such that the diagram
\begin{equation}\label{univprop}
\xymatrix @R=.2in
{(G,X)\ar[rr]^{\tilde{\eta}} \ar[dr]_-F && 
\tilde{S}(\sum\limits_{n\in\caa{N}}{G^n},X)
\ar @{.>}[dl]^-{\tilde{S}H}\\
& \tilde{S}(B,Y) &}
\end{equation}
commutes. 

By Definition \ref{charactVGrph},
a $\ca{V}$-graph functor
$F$ can be seen as a pair $(\phi,f)$ where
$\phi:G\to f^*Bf_*$ is an arrow in
$\ca{V}$-$\B{Mat}(X,X)$,
and furthermore
Lemma \ref{B*monoid} ensures that
$f^*Bf_*$ obtains a monoid structure. 
Since $LG$ is the free
monoid on the object $G$ of $\ca{V}$-$\B{Mat}(X,X)$,
$\phi$ extends uniquely to a monoid morphism
$\chi:LG\to f^*Bf_*$ such that
the diagram
\begin{displaymath}
\xymatrix @R=.15in
{G\ar[rr]^{\eta} \ar[dr]_-{\phi} && 
\sum\limits_{n\in\caa{N}}{G^n}\ar @{.>}[dl]^-{S\chi}\\
& f^*Bf_* &}
\end{displaymath}
commutes in the category $\ca{V}$-$\B{Mat}(X,X)$,
where $\eta$ and $S$ are respectively
the unit and forgetful functor of the `free monoid'
adjunction $L\dashv S$.

By Lemma \ref{charactVCat}, this 2-cell 
$\chi:\sum_{n\in\caa{N}}{G^n}\Rightarrow f^*Bf_*$
in $\ca{V}$-$\Mat$, together with the function $f$,
determine
a $\ca{V}$-functor $H=(\chi,f):(LG,X)\to
(B,Y)$ satisfying
the universal property (\ref{univprop}). These data are 
sufficient to define an adjoint functor $\tilde{L}$
with object function (\ref{deftildeL}), thus
the `free $\ca{V}$-category' adjunction
\begin{displaymath}
\xymatrix @C=.5in
{\ca{V}\textrm{-}\B{Grph}
\ar@<+.8ex>[r]^-{\tilde{L}} 
 _-*-<5pt>{\text{\scriptsize{$\bot$}}} &
\ca{V}\textrm{-}\B{Cat}\ar@<+.8ex>[l]^-{\tilde{S}}}
\end{displaymath}
is established.
\end{proof}
The above result was also given earlier in
\cite[Proposition 2.2]{Wolff} but constructively,
in the sense that the explicit
description of the free $\ca{V}$-category along with
its composition and identities is provided, 
and the universal property is shown
without the use of $\ca{V}$-matrices.
As a result, in that plain context, just the
existence of coproducts in $\ca{V}$ suffices to establish
the free $\ca{V}$-category adjunction, without
requiring $\otimes$ to preserve them.
Also, as proved in detail in \cite{Wolff}
and later generalized in \cite{VarThrEnr}
for categories enriched in bicategories,
$\ca{V}$-$\B{Cat}$ has and the forgetful functor 
$\tilde{S}$ reflects split coequalizers when $\ca{V}$ is 
cocomplete. By Beck's
monadicity theorem, since $\tilde{S}$ also
reflects isomorphisms, we have the following well-known
result. 
\begin{prop}\label{VCatmonadic}
If $\ca{V}$ is a cocomplete monoidal
category (such that $\otimes$ preserves colimits on both
variables), the forgetful
$\tilde{S}:\ca{V}$-$\B{Cat}\to\ca{V}$-$\B{Grph}$
is monadic.
\end{prop}
Consequently, the category $\ca{V}$-$\B{Cat}$
is isomorphic 
to the category of $\tilde{S}\tilde{L}$-algebras on 
$\ca{V}$-$\B{Grph}$. As mentioned earlier, $\ca{V}$-$\B{Grph}$ is
complete when $\ca{V}$ is, thus
\begin{cor}\label{VCatcomplete}
The category $\ca{V}$-$\B{Cat}$ is complete when $\ca{V}$ is.
\end{cor}
The fact that $\ca{V}$-$\B{Cat}$ also has all
colimits follows from a result 
by Linton in \cite{Linton}, which states that
if the category of algebras for a monad
has coequalizers of reflexive pairs and $\ca{A}$
has all small coproducts, then $\ca{A}^T$ has all small 
colimits. By Proposition \ref{Vgrphcocomplete}
$\ca{V}$-$\B{Grph}$ admits all colimits if $\ca{V}$ does, 
hence the following is true.
\begin{cor}\label{VCatcocomplete}
The category $\ca{V}$-$\B{Cat}$ is cocomplete when $\ca{V}$ is.
\end{cor}
Finally, $\ca{V}$-$\B{Cat}$ also inherits
local presentability from $\ca{V}$-$\B{Grph}$.
As shown in \cite{KellyLack}, the monad 
$\tilde{S}\tilde{L}$ is finitary. Thus
by a result of Gabriel and Ulmer \cite[Satz 10.3]{GabrielUlmer}
which states that if $\ca{A}$ is locally
presentable, then $\ca{A}^T$ for a finitary monad
is locally presentable, we obtain the following result.
\begin{thm*}~\cite[4.5]{KellyLack}
If $\ca{V}$ is a monoidal closed category whose
underlying ordinary category is locally 
$\lambda$-presentable,
then $\ca{V}$-$\B{Cat}$ is also
$\lambda$-presentable.
\end{thm*} 
We can now turn to the `dualization' of the concept of
a $\ca{V}$-category in the context of the 
bicategory $\ca{V}$-$\Mat$. Henceforth $\ca{V}$ is a monoidal category
with coproducts, such that the tensor product $\otimes$
preserves them on both entries. 
The definition below follows Definition \ref{comonadbicat}.
\begin{defi}\label{cocategory}
A (small) $\ca{V}$-\emph{cocategory} 
$\ca{C}$ is a comonad
in the bicategory $\ca{V}$-$\B{Mat}$. Thus it
consists of a set $X$ with 
an endoarrow$\SelectTips{eu}{10}
\xymatrix @C=.2in
{C:X\ar[r]|-{\object@{|}} & X,}$\emph{i.e.}
a $\ca{V}$-graph with set of objects
$\ob\ca{C}=X$,
equipped with two 2-cells,
the comultiplication and the counit
\begin{displaymath}
\xymatrix @R=.1in @C=.4in
{X \ar[dr]|-{\object@{|}}
_-{C} \ar @/^3ex/[rr]|-{\object@{|}}^-C
\rrtwocell<\omit>{<.5>\Delta} && X \\
 & X\ar[ur]|-{\object@{|}}_-C  &}
\quad
\textrm{and}
\quad
\xymatrix @C=.3in
{X \rrtwocell<\omit>{<0>\epsilon}
\ar @/^2.2ex/ [rr]|-{\object@{|}}^-{C}
\ar @/_2.2ex/ [rr]|-{\object@{|}}_-{1_X} && X}
\end{displaymath}
satisfying the following axioms:
\begin{gather*}
\xymatrix @R=.2in
{X\ar @/_/[dr]|-{\object@{|}}_-C
\ar @/^3ex/[drr]|-{\object@{|}}^-C
\ar @/^4ex/[rrr]|-{\object@{|}}^-C
\drrtwocell<\omit>{<+.3>\Delta}
&\drrtwocell<\omit>{<-1.3>\Delta} && X \\
& X\ar[r]|-{\object@{|}}_-C &
X\ar @/_/[ur]|-{\object@{|}}_-C &} 
\quad = \quad
\xymatrix @R=.2in
{X\ar @/_/[dr]|-{\object@{|}}_-C
\ar @/^4ex/[rrr]|-{\object@{|}}^-C
&&& X \\
\urrtwocell<\omit>{<-1.3>\Delta} & 
X\ar[r]|-{\object@{|}}_-C 
\ar @/^3ex/[urr]|-{\object@{|}}^-C 
\urrtwocell<\omit>{<+.3>\Delta} &
X,\ar @/_/[ur]|-{\object@{|}}_-C &} \\
\xymatrix @C=.5in @R=.2in
{X\drrtwocell<\omit>{<-2.3>\Delta}
\drtwocell<\omit>{<-0.4>\epsilon}
\ar @/_2ex/[dr]|-{\object@{|}}_-{1_X}
\ar @/^3ex/[dr]|-{\object@{|}}^-C 
\ar @/^4ex/[rr]|-{\object@{|}}^-C
&& X \\
& X \ar @/_/[ur]|-{\object@{|}}_-C &}
\xymatrix @C=.5in @R=.2in
{=\\
\hole}
\xymatrix @C=.5in @R=.2in
{X\rtwocell<\omit>{\;1_C}
\ar @/_2.3ex/[r]|-{\object@{|}}_-C
\ar @/^2.3ex/[r]|-{\object@{|}}^-C 
& X \\ \hole}
\xymatrix @C=.5in @R=.2in
{= \\
\hole}
\xymatrix @C=.5in @R=.2in
{X \ar @/_/[dr]|-{\object@{|}}_-C
\ar @/^4ex/[rr]|-{\object@{|}}^-C
&& X \\
\urrtwocell<\omit>{<-2.3>\Delta}
& X. \ar @/_2ex/[ur]|-{\object@{|}}_-{1_X} 
\ar @/^3ex/[ur]|-{\object@{|}}^-C 
\urtwocell<\omit>{<-0.4>\epsilon} &}
\end{gather*}
\end{defi}
In terms of components,
the \emph{cocomposition} of a $\ca{V}$-cocategory 
$\ca{C}$ is given by
\begin{displaymath}
\Delta_{x,z}:\ca{C}(x,z)\to\sum_{y\in X}
{\ca{C}(x,y)\otimes\ca{C}(y,x)}
\end{displaymath}
for any two objects $x,z\in X$,
and the \emph{coidentity elements} are given by
\begin{displaymath}
 \epsilon_{x,y}:\ca{C}(x,y)\to 1_X(x,y)\equiv
\begin{cases}
\ca{C}(x,x)\xrightarrow{\epsilon_{x,x}}I,\quad \mathrm{if }\;x=y\\
\ca{C}(x,y)\xrightarrow{\epsilon_{x,y}}0,\quad \mathrm{if }\;x\neq y
\end{cases}
\end{displaymath}
for all objects $x\in X$. The commutative diagrams expressing the 
coassociativity and counit axioms are
\begin{displaymath}
\xymatrix @C=.01in @R=.45in
{& \ca{C}(x,w)\ar[dl]_-{\Delta}
\ar[dr]^-{\Delta} &\\
\sum\limits_{z}{\ca{C}(x,z)\otimes\ca{C}(z,w)}
\ar[d]_-{\sum\limits_{z}{\Delta\otimes 1}} &&
\sum\limits_{y}{\ca{C}(x,y)\otimes\ca{C}(y,w)}
\ar[d]^-{\sum\limits_{y}{1\otimes\Delta}} \\
\sum\limits_{z}(\sum\limits_{y}{\ca{C}(x,y)\otimes\ca{C}(y,z)})
\otimes\ca{C}(z,w)\ar[rr]_-{\alpha}^-\cong &&
\sum\limits_{y}{\ca{C}(x,y)\otimes(\sum\limits_{z}{\ca{C}(y,z)\otimes
\ca{C}(y,w)})}}
\end{displaymath}
\begin{displaymath}
\xymatrix @C=.9in @R=.4in
{\sum\limits_{z}{\ca{C}(x,z)\otimes\ca{C}(z,y)}
\ar[d]_-{\sum\limits_{z}{\epsilon\otimes 1}} &
\ca{C}(x,y)\ar[dr]_-{\rho^{\text{-}1}}
\ar[dl]^-{\lambda^{\text{-}1}}
\ar[l]_-{\Delta}\ar[r]^-{\Delta} &
\sum\limits_{z}{\ca{C}(x,z)\otimes\ca{C}(z,y)}\ar[d]^-
{\sum\limits_{z}{1\otimes\epsilon}} \\
I\otimes\ca{C}(x,y) &&
\ca{C}(x,y)\otimes I}
\end{displaymath}
where $\alpha$ is the associator 
and $\lambda$, $\rho$ are the unitors of $\ca{V}$-$\Mat$.
The vertical arrows of the latter diagram  
are explicitely the unique
morphisms making the left and right
parts of the diagram commute:
\begin{displaymath}
\xymatrix @C=.25in @R=.2in
{& \sum\limits_{z}{C(x,z)\otimes\ca{C}(z,y)}
\ar @<-4ex>@{.>}[dd]_-{\sum\limits_{z}{\epsilon_{x,z}\otimes 1}}
\ar @<+4ex>@{.>}[dd]^-{\sum\limits_{z}{1\otimes\epsilon_{z,y}}}
&\\
\ca{C}(x,x)\otimes\ca{C}(x,y) \ar @{^{(}->}[ur]^-{i}
\ar[dr]_-{\epsilon_{x,x}\otimes1} &&
\ca{C}(x,y)\otimes\ca{C}(y,y) \ar @{_{(}->}[ul]_-{i}
\ar[dl]^-{1\otimes \epsilon_{y,y}} \\
& \qquad I\otimes\ca{C}(x,y)\qquad\ca{C}(x,y)\otimes I.\qquad &}
\end{displaymath}

As for comonads in any bicategory, a $\ca{V}$-cocategory
$\ca{C}$ with $\ob\ca{C}=X$ is the same as a comonoid
in the monoidal category $(\ca{V}$-$\B{Mat}(X,X),\circ,1_X)$.
Thus a one-object $\ca{V}$-cocategory
is the same as a comonoid in the monoidal category $\ca{V}$.
We denote such a $\ca{V}$-cocategory as $\ca{C}_X$ or $(C,X)$.
Analogously to $\ca{V}$-graphs and $\ca{V}$-categories,
the notation $(C,X)$ is preferred for the $\ca{V}$-matrices 
context, whereas $\ca{C}_X$ for the dual to the `classic
presentation' which basically corresponds to the componentwise 
version. The latter can evidently be expressed without the explicit
use of $\ca{V}$-matrices.

The next step is to define the appropriate morphisms
between $\ca{V}$-cocategories. For $\ca{V}$-graph 
arrows and $\ca{V}$-functors,
morphisms $F$ were initially defined in the standard way, \emph{i.e.}
consisting of certain arrows in $\ca{V}$ as in 
(\ref{defvgrapharrow}) and (\ref{classicVfunctor}).
Then, using the formulation in terms of 
$\ca{V}$-matrices, $F$ was expressed as a 
pair $(\phi,f)$, where $\phi$ is a 2-cell
in $\ca{V}$-$\Mat$ with components
\emph{isomorphic} arrows to the previous ones.
This led to the characterization of Definition 
\ref{charactVGrph}
for $\ca{V}$-$\B{Grph}$,
and allowed the $\ca{V}$-functor
axioms to be written in a colax monad functor 
style which resulted in characterization of 
Lemma \ref{charactVCat} for $\ca{V}$-$\B{Cat}$.
We similarly proceed for arrows for $\ca{V}$-cocategories.
\begin{defi}\label{cofunctor}
A $\ca{V}$-\emph{cofunctor} $F_f:\ca{C}_X\to\ca{D}_Y$
between two $\ca{V}$-cocategories 
is a morphism of $\ca{V}$-graphs,
consisting of a function $f:X\to Y$ between their sets
of objects and arrows in $\ca{V}$
\begin{equation}\label{Vcofunctarrows}
 F_{x,z}:\ca{C}(x,z)\to\ca{D}(fx,fz)
\end{equation}
for any two objects $x,z\in\ob\ca{C}$,
which satisfy the commutativity of
\begin{equation}\label{cofuncts}
\xymatrix @C=.03in @R=.1in
{\ca{C}(x,z)\ar[rrr]^-{\Delta^C_{x,z}}
\ar[dd]_-{F_{x,z}} &&&
\sum\limits_{y\in X}{\ca{C}(x,y)\otimes\ca{C}(y,z)}
\ar @{.>}[dd] \ar @/^2ex/[dr]^-{\sum\limits_{y}{F_{x,y}\otimes
F_{y,z}}} & \\
&&&& \sum\limits_{fy\in Y}{\ca{D}(fx,fy)\otimes\ca{D}(fy,fz)}
\ar @{>->}@/^3ex/[dl]^-{\iota} \\
\ca{D}(fx,fz)\ar[rrr]_-{\Delta^D_{fx,fz}} &&& 
\sum\limits_{w\in Y}{\ca{D}(fx,w)\otimes\ca{D}(w,fz)} &}
\end{equation}

\begin{displaymath}
\mathrm{and}\qquad\xymatrix @C=.7in @R=.3in
{\ca{C}(x,x)\ar[r]^-{\epsilon^C_{x,x}} 
\ar[d]_-{F_{x,x}} & I \\
\ca{D}(fx,fx).\ar[ur]_-{\epsilon^D_{fx,fx}} &}
\end{displaymath}
\end{defi}
The above commutative diagrams express
the compatibility with cocomposition and
coidentities.
Equivalently, we can view a $\ca{V}$-functor
as a pair 
$(\phi,f):(C,X)\to(D,Y)$ between two comonads
in $\ca{V}$-$\Mat$,
with $f:X\to Y$ a function and
$\phi$ a 2-cell $C\Rightarrow f^*Df_*$
which satisfies the equalities
\begin{align}\label{cofunctaxioms}
\xymatrix @C=.5in @R=.5in
{X \ar[r]|-{\object@{|}}^-C
\ar[d]|-{\object@{|}}_-{f_*}
\rrtwocell<\omit>{<-3>\Delta}
\drtwocell<\omit>{\hat{\phi}}
\ar @/^6ex/[rr]|-{\object@{|}}^-C
& X\ar[r]|-{\object@{|}}^-C
\ar[d]|-{\object@{|}}^-{f_*}
\drtwocell<\omit>{\hat{\phi}}
& X \ar[d]|-{\object@{|}}^-{f_*}\\
Y \ar[r]|-{\object@{|}}_-D
& Y \ar[r]|-{\object@{|}}_-D & Y} 
&\xymatrix @R=.2in{\hole \\ = \\ \hole}
\xymatrix @C=.5in @R=.3in
{X\ar[d]|-{\object@{|}}_-{f_*}
\ar[rr]|-{\object@{|}}^-C
\drrtwocell<\omit>{\hat{\phi}} && X
\ar[d]|-{\object@{|}}^-{f_*}\\
Y \rrtwocell<\omit>{<4>\Delta}
\ar @/_/[dr]|-{\object@{|}}_-D
\ar[rr]|-{\object@{|}}_-D && Y\\
& Y \ar @/_/[ur]|-{\object@{|}}_-D &} \\
\xymatrix @C=1in  @R=.4in
{X\rtwocell<\omit>{<-2>\epsilon}
\ar @/^5ex/[r]|-{\object@{|}}^-C
\ar[r]|-{\object@{|}}
_-{1_X}
\ar[d]|-{\object@{|}}_-{f_*}
^{\phantom{ab}\cong}
\ar @{.>}[dr]|-{\object@{|}}_-{f_*}
& X
\ar[d]|-{\object@{|}}^-{f_*}
_{\cong\phantom{ab}} \\
Y\ar[r]|-{\object@{|}}_-{1_Y}
 & Y}
&\xymatrix{ = \\ \hole}
\xymatrix @C=1in @R=.4in
{X\ar[r]|-{\object@{|}}
^-C \ar[d]|-{\object@{|}}
_-{f_*} & X \ar[d]|-{\object@{|}}
^-{f_*} \\
Y\ar @/_5ex/[r]|-{\object@{|}}
_-{1_Y} \rtwocell<\omit>{<2.5>\epsilon}
\ar[r]|-{\object@{|}}^-D 
\rtwocell<\omit>{<-4>\hat{\phi}} & Y}
\notag
\end{align}
for $\hat{\phi}:f_*C\Rightarrow Df_*$ the mate
of $\phi$ `on the one side'.
These two ways of defining a 
$\ca{V}$-cofunctor are equivalent in the sense
that there is a bijection between them. 
The components of $\hat{\phi}$ are given by
\begin{displaymath}
 \sum_{x'\in f^{\text{-}1}y}I\otimes C(x',x)
\to D(fx',fx)\otimes I
\end{displaymath}
which for fixed $x'$ are in bijection to 
(\ref{Vcofunctarrows}). The equalities 
(\ref{cofunctaxioms}) written in terms of components then
agree with the commutativity of 
(\ref{cofuncts}) up to appropriate tensoring with $I$.

It is not hard to see that $\ca{V}$-cofunctors compose,
also by viewing them as specific types of 
lax comonad functors dually to Remark \ref{Vfunct=monadopfunct}.
Therefore we obtain a category $\ca{V}$-$\B{Cocat}$ of 
$\ca{V}$-cocategories and $\ca{V}$-cofunctors.

Dually to Lemma \ref{B*monoid}, we have the following.
\begin{lem}\label{C*comonoid}
Let $(C,X)$ be a $\ca{V}$-cocategory.
If $f:X\to Y$ is a function, then the composite 
$\ca{V}$-matrix
\begin{displaymath}
\SelectTips{eu}{10}\xymatrix
{Y\ar[r]|-{\object@{|}}^-{f^*} &
X\ar[r]|-{\object@{|}}^-C &
X\ar[r]|-{\object@{|}}^-{f_*} &
Y}
\end{displaymath}
is a comonoid in $\ca{V}$-$\B{Mat}(Y,Y)$, which implies
that $(f_*Cf^*,Y)$ is also a $\ca{V}$-cocategory.
\end{lem}
\begin{proof}
The comultiplication $\Delta':f_*Cf^*\to f_*Cf^*f_*Cf^*$
and the counit $\epsilon':f_*Cf^*\to 1_X$
are given by the composites 
\begin{gather*}
\xymatrix @C=.4in @R=.05in
{\hole \\
Y \ar[r]|-{\object@{|}}^-{f^*}
& X\ar @/^6ex/[rrrr]|-{\object@{|}}^-C
\ar @/_/[dr]|-{\object@{|}}^-C
& \rrtwocell<\omit>{<-3>\Delta}
&&&
X\ar[r]|-{\object@{|}}^-{f_*} & Y \\
&& X\ar @/_/[dr]|-{\object@{|}}_-{f_*}
\ar @/^3ex/[rr]|-{\object@{|}}^-{1_X}
\rrtwocell<\omit>{\check{\eta}} &&
X\ar @/_/[ur]|-{\object@{|}}^-C && \\
&&&
Y\ar @/_/[ur]|-{\object@{|}}_-{f^*} &&&} \\
\xymatrix @R=.12in
{& X\ar @/_2ex/[r]|-{\object@{|}}_-{1_X}
\ar @/^2ex/[r]|-{\object@{|}}^-C 
\rtwocell<\omit>{\epsilon} 
\rtwocell<\omit>{<5.3>\check{\varepsilon}} &
X\ar[dr]|-{\object@{|}}^-{f_*} & \\
Y\ar @/_3ex/[rrr]|-{\object@{|}}_-{1_Y} 
\ar[ur]|-{\object@{|}}^-{f^*} &&& Y,}
\end{gather*}
where $\check{\varepsilon}$ and $\check{\eta}$
are the counit and unit of the adjunction
$f_*\dashv f^*$ in $\ca{V}$-$\B{Mat}$ and $\Delta$
and $\epsilon$ the comonoid structure maps of $C$. In terms
of pasting oparation, the new comultiplication
and counit can be written as 
\begin{gather*}
\Delta'=f_*\big((C\check{\eta}C)\cdot\Delta\big)f^*, \\
\epsilon'=\check{\varepsilon}\cdot(f_*\epsilon f^*).
\end{gather*}
The coassociativity
and counit axioms follow immediately from the axioms of
the comonoid$\SelectTips{eu}{10}\xymatrix @C=.2in
{C:X\ar[r]|-{\object@{|}} &X}$and the
the triangular identities for 
$\check{\varepsilon}$ and $\check{\eta}$.
\end{proof}
Once again, it can be deduced that the diagrams 
(\ref{cofunctaxioms}) a $\ca{V}$-cofunctor
$F:(C,X)\to(D,Y)$
has to satisfy coincide with the ones
for a comonoid arrow 
between $f_*Cf^*$ and $D$.
The following characterization is now established.
\begin{lem}\label{charactVCocat}
Objects in $\ca{V}$-$\B{Cocat}$ are pairs
$$(C,X)\in\Comon(\ca{V}\textrm{-}\B{Mat}(X,X))\times\B{Set}$$ 
and morphisms are pairs $(\psi,f):(C,X)\to(D,Y)$ where 
\begin{displaymath}
\begin{cases}
\psi:f_*Cf^*\to D &\textrm{in }\Comon(\ca{V}\textrm{-}\B{Mat}(Y,Y))\\
f:X\to Y & \textrm{in }\B{Set}.
\end{cases}
\end{displaymath}
\end{lem}
Notice how, out of the two equivalent
formulations for $\ca{V}$-graph morphisms of 
Definition \ref{charactVGrph}, $\ca{V}$-functors 
are expressed via pairs $(\phi,f)$ and 
$\ca{V}$-cofunctors are expressed via 
pairs $(\psi,f)$, where the 2-cells $\phi:G\Rightarrow f^*Hf_*$
and $\psi:f_*Gf^*\Rightarrow H$ 
are mates in $\ca{V}$-$\Mat$.

The category $\ca{V}$-$\B{Cocat}$ obtains a monoidal 
structure when $\ca{V}$
is symmetric mo\-noi\-dal. For two $\ca{V}$-cocategories 
$\ca{C}_X$ and $\ca{D}_Y$,
$\ca{C}\otimes\ca{D}$ is their tensor product 
as $\ca{V}$-graphs, \emph{i.e.}
has as set of objects the cartesian product $X\times Y$ and 
consists of the family of objects in $\ca{V}$
\begin{displaymath}
 (\ca{C}\otimes\ca{D})\big((z,w),(x,y)\big)=\ca{C}(z,x)\otimes\ca{D}(w,y).
\end{displaymath}
The cocomposition law is given by the composite
\begin{displaymath}
 \xymatrix @C=.7in @R=.2in
{\ca{C}(z,x)\otimes D(w,y)\ar@{-->}[r]
\ar@/_8ex/[ddr]_-{\Delta^C_{z,x}\otimes\Delta^D_{w,y}} &
\sum\limits_{(x',y')}{\ca{C}(z,x')\otimes
\ca{D}(w,y')\otimes\ca{C}(x',x)\otimes\ca{D}(y',y)} \\
& \sum\limits_{(x',y')}{\ca{C}(z,x')\otimes
\ca{C}(x',x)\otimes\ca{D}(w,y')\otimes\ca{D}(y',y)}\ar[u]_-s \\
& \sum\limits_{x'}{\ca{C}(z,x')\otimes\ca{C}(x',x)}\otimes
\sum\limits_{y'}{\ca{D}(w,y')\otimes\ca{D}(y',y)}
\ar[u]_-{\cong}}
\end{displaymath}
and the coidentity element is
\begin{displaymath}
\ca{C}(x,x)\otimes\ca{D}(y,y)\xrightarrow{\;\epsilon^C_{x,x}\otimes\epsilon^D_{y,y}\;}
I\otimes I\cong I.
\end{displaymath}
The unit for this tensor product is the unit $\ca{V}$-graph
$\ca{I}$ with obvious cocomposition and coidentities.
Similarly we can define the tensor product
of two $\ca{V}$-cofunctors between $\ca{V}$-cocategories,
and also symmetry is inherited,
hence $(\ca{V}\text{-}\B{Cocat},\otimes,\ca{I})$ is
a symmetric monoidal category. 

Dually to Proposition \ref{freeVcatfunctor},
we now construct the `cofree $\ca{V}$-cocategory'
functor using the cofree comonoid construction.
As discussed in Section 
\ref{Categoriesofmonoidsandcomonoids},
the existence of the cofree comonoid
usually requires more assumptions on $\ca{V}$
than the free monoid, and the following is no exception.
\begin{prop}\label{cofreeVcocatfunctor}
Suppose $\ca{V}$ is a locally presentable monoidal category,
such that $\otimes$ preserves colimits in both variables. 
Then, the evident 
forgetful functor
\begin{displaymath}
\tilde{U}:\ca{V}\textrm{-}\B{Cocat}\longrightarrow\ca{V}\textrm{-}\B{Grph}
\end{displaymath}
has a right adjoint $\tilde{R}$, which maps a $\ca{V}$-graph
$(G,Y)$ to the cofree comonoid $(RG,Y)$ on 
$G\in\ca{V}$-$\B{Mat}(Y,Y)$.
\end{prop}
\begin{proof}
The forgetful functor $\tilde{U}$
maps any $\ca{V}$-cocategory $(C,X)$ to 
the `underlying' $\ca{V}$-graph $(UC,X)$,
where $U$ is the forgetful functor from
the category of comonoids
of the monoidal category 
($\ca{V}$-$\B{Mat}(Y,Y),\circ,1_Y)$. 
By Corollary \ref{cofreecomonVMat}, $U$
has a right adjoint
\begin{displaymath}
R:\ca{V}\textrm{-}\Mat(Y,Y)\longrightarrow
\Comon(\ca{V}\textrm{-}\Mat(Y,Y))
\end{displaymath}
namely the cofree comonoid functor.
By Lemma \ref{charactVCocat}, the pair 
$(RG,Y)$ where $RG$ is the cofree comonoid
on an endoarrow$\SelectTips{eu}{10}
\xymatrix @C=.2in{G:Y\ar[r]|-{\object@{|}} & Y}$is 
in fact a $\ca{V}$-cocategory 
with set of objects $Y$. 
We claim that the mapping
\begin{equation}\label{deftildeR}
\tilde{R}:\xymatrix @R=.02in
{\ca{V}\textrm{-}\B{Grph}\ar[r] & 
\ca{V}\textrm{-}\B{Cocat}\\
(G,Y)\ar@{|->}[r] &
(RG,Y)}
\end{equation}
gives rise to a right adjoint of the forgetful $\tilde{U}$.
It is enough to show that for 
$\varepsilon$ the counit of the cofree
comonoid adjunction $U\dashv R$,
the $\ca{V}$-graph arrow
$\tilde{\varepsilon}=(\varepsilon,\mathrm{id}_Y):\tilde{U}\tilde{R}(G,Y)
\to(G,Y)$
is universal. This means that for 
any $\ca{V}$-cocategory $\ca{C}_X$
and any $\ca{V}$-graph morphism $F$ from its underlying
$\ca{V}$-graph $\tilde{U}(C,X)$ to $(G,Y)$, there exists 
a unique $\ca{V}$-cofunctor $H:(C,X)\to(RG,Y)$
such that the diagram
\begin{equation}\label{thisuniversality}
\xymatrix @R=.35in
{\tilde{U}(RG,Y)\ar[rr]^-{\tilde{\varepsilon}} && (G,Y)\\
& \tilde{U}(C,X)\ar @{.>}[ul]^-{\tilde{U}H}
\ar[ur]_-F &}
\end{equation} 
commutes.

The $\ca{V}$-graph arrow $F$ can be seen as a pair
$(\psi,f)$ where $f:X\to Y$ is the function
on objects and $\psi:f_*Cf^*\to G$ is an arrow
in $\ca{V}$-$\B{Mat}(Y,Y)$.
However, by Lemma \ref{C*comonoid}
the composite $f_*Cf^*$ is
an object of $\Comon(\ca{V}$-$\B{Mat}(Y,Y))$, since $C$
is a comonoid itself. Due to $RG$
being the cofree comonoid on $G$, this $\psi$ extends
uniquely to a comonoid arrow $\chi:f_*Cf^*\to RG$
such that the diagram
\begin{displaymath}
\xymatrix @R=.35in
{RG \ar[rr]^-{\varepsilon} && G\\
& f_*Cf^* \ar[ur]_-{\psi}
\ar @{.>}[ul]^-{U\chi} &}
\end{displaymath}
commutes in $\ca{V}$-$\B{Mat}(Y,Y)$.
Then, by Lemma \ref{charactVCocat} 
this 2-cell $\chi$ in $\Comon(\ca{V}$-$\B{Mat}(Y,Y))$
along with the function $f:X\to Y$
determines a $\ca{V}$-cofunctor 
$H:(C,X)\to(RG,Y)$, which satisfies 
the commutativity of (\ref{thisuniversality}).
Therefore $\tilde{R}$ extends to a functor 
with mapping on objects as in (\ref{deftildeR}),
which establishes
the `cofree $\ca{V}$-cocategory' adjunction
$\tilde{U}\dashv\tilde{R}:\ca{V}\text{-}\B{Grph}
\to\ca{V}\text{-}\B{Cocat}$.
\end{proof}
At this point, properties of $\ca{V}$-$\B{Cocat}$
cease to be straightforward dualizations of the ones
of $\ca{V}$-$\B{Cat}$. As an example,
in order to deduce results such as comonadicity 
of $\ca{V}$-$\B{Cocat}$ over $\ca{V}$-$\B{Grph}$, 
we will later show that $\ca{V}$-$\B{Cocat}$
is locally presentable via a different method,
under the conditions for the existence of the 
cofree $\ca{V}$-cocategory functor $\tilde{R}$.
 
We close this section by the construction of colimits in 
$\ca{V}$-$\B{Cocat}$. In fact, this follows from
the construction of colimits in $\ca{V}$-$\B{Grph}$
in Proposition \ref{Vgrphcocomplete}, with 
an induced extra structure on the colimiting
cocone which amounts to
a colimit of $\ca{V}$-cocategories.
\begin{prop}\label{VCocatcocomplete}
Suppose that $\ca{V}$ is a locally presentable
monoidal category, such that $\otimes$ preserves colimits
in both variables. The category $\ca{V}$-$\B{Cocat}$ has
all small colimits.
\end{prop}
\begin{proof}
Consider a diagram in $\ca{V}$-$\B{Cocat}$ given by
\begin{displaymath}
 D:
\xymatrix @R=.05in @C=.6in
{\ca{J}\ar[r] & \ca{V}\textrm{-}\B{Cocat} \\
j\ar@{|.>}[r]
\ar[dd]_-\theta & 
(C_j,X_j)\ar[dd]^-{(\psi_\theta,f_\theta)} \\
\hole \\
k\ar@{|.>}[r] & (C_k,X_k)}
\end{displaymath}
for a small category $\ca{J}$.
By Lemma \ref{charactVCocat},
$f_\theta:X_j\to X_k$ is a function
and $\psi_\theta$ is an arrow
$(f_\theta)_*C_j(f_\theta)^*\to C_k$
in $\Comon(\ca{V}$-$\Mat(X_k,X_k))$, i.e.
a 2-cell in $\ca{V}$-$\Mat$
\begin{displaymath}
 \xymatrix
{X_j\ar[r]|-{\object@{|}}
^-{C_j} & X_j \ar[d]|-{\object@{|}}
^-{(f_\theta)_*}\\
X_k\ar[u]|-{\object@{|}}
^-{(f_\theta)^*}\ar[r]|-{\object@{|}}_-{C_k}
\rtwocell<\omit>{<-4>\;\psi_\theta}
& X_k}
\end{displaymath}
satisfying the usual comonoid morphism properties.
We can first construct the colimit of 
the underlying $\ca{V}$-graphs of 
this diagram as in Proposition 
\ref{Vgrphcocomplete}. We then obtain a colimiting cocone
\begin{equation}\label{colimgraph}
 \big((C_j,X_j)\xrightarrow{\;(\lambda_j,\tau_j\;)}
(C,X)\,|\,j\in\ca{J}\big)
\end{equation}
in $\ca{V}$-$\B{Grph}$, where 
$(\tau_j:X_j\to X\,|\,j\in\ca{J})$
is the colimit of the sets of objects
of the $\ca{V}$-cocategories in $\B{Set}$,
and $(\lambda_j:(\tau_j)_*C_j(\tau_j)^*\to C\,|\,j\in\ca{J})$ 
is the colimiting cocone
of the diagram $K$ as in (\ref{defKdiagram})
in the cocomplete
$\ca{V}$-$\Mat(X,X)$.

Notice that $K:\ca{J}\to\ca{V}\textrm{-}\Mat(X,X)$
in fact lands inside $\Comon(\ca{V}$-$\Mat(X,X))$:
Lemma \ref{C*comonoid} ensures that $\ca{V}$-matrices
of the form $f_*Cf^*$ for any comonoid $C$
inherit a comonoid structure, and  also
the composite arrows (\ref{Konarrows})
where the middle 2-cell is now the comonoid arrow 
$\psi_\theta$ ensure that $K\theta$ are comonoid morphisms.
Since by Corollary \ref{cofreecomonVMat}
the category of comonoids
is comonadic over $\ca{V}\textrm{-}\Mat(X,X)$,
the respective forgetful functor creates all
colimits, therefore$\SelectTips{eu}{10}\xymatrix@C=.2in
{C:X\ar[r]|-{\object@{|}} & 
X}$obtains a unique
comonoid structure. Moreover, the legs of the cocone
\begin{displaymath}
\xymatrix
{X_j\ar[r]|-{\object@{|}}
^-{C_j} & X_j \ar[d]|-{\object@{|}}
^-{(\tau_j)_*}\\
X\ar[u]|-{\object@{|}}
^-{(\tau_j)^*}\ar[r]|-{\object@{|}}_-{C}
\rtwocell<\omit>{<-4>\lambda_j}
& X}
\end{displaymath}
are comonoid arrows, hence together with the 
functions $\tau_j$ they form $\ca{V}$-cofunctors.
Thus the colimit (\ref{colimgraph})
lifts in $\ca{V}$-$\B{Cocat}$.
\end{proof}

\section{Enrichment of $\ca{V}$-categories in $\ca{V}$-cocategories}
\label{enrichmentofVcatsinVcocats}

We now wish to extend the results
presented in Section \ref{Universalmeasuringcomonoid},
where the existence of the 
universal measuring comonoid and 
the induced enrichment of monoids in comonoids
were established.
Similarly to the previous development,
we aim to identify an \emph{action}
of the symmetric monoidal closed
category $\ca{V}$-$\B{Cocat}$ on the ordinary
category $\ca{V}$-$\B{Cat}$ (or better its opposite), with a
parametrized adjoint which will turn out to
be the `enriched-hom' functor of a 
($\ca{V}$-$\B{Cocat}$)-enriched category
with underlying category $\ca{V}$-$\B{Cat}$. The relevant
theory which underlies this process is contained
in Section \ref{actions}.

Suppose that $\ca{V}$ is a cocomplete symmetric monoidal closed 
category with products.
Recall that there 
exists a lax functor of bicategories
\begin{displaymath}
 \Hom:(\ca{V}\textrm{-}\Mat)^{\textrm{co}}
\times\ca{V}\textrm{-}\Mat\longrightarrow
\ca{V}\textrm{-}\Mat
\end{displaymath}
defined as in (\ref{defimportHom}).
Then the functor between the hom-categories (of endoarrows)
$\Hom_{(X,Y),(X,Y)}$ induces the internal hom 
$^g\Hom:\ca{V}\textrm{-}\B{Grph}^\op\times\ca{V}\textrm{-}\B{Grph}\to
\B{Grph}$ of $\ca{V}$-graphs
as described in Proposition
\ref{VGrphclosed}, via 
\begin{displaymath}
\Hom((G,X),(H,Y))(k,s):=
\prod_{x,x'\in X}{[G(x',x),H(kx',sx)]}
\end{displaymath}
for all $k,s\in Y^X$.
Moreover, by Lemma \ref{lemmonlaxfun},
every lax functor between bicategories
induces a functor between monoids of hom-categories
of endoarrows. For the lax functor $\Hom$,
we obtain
\begin{equation}\label{MonHom_}
\scriptstyle{\Mon(\Hom_{(X,Y),(X,Y)})}:\;
\scriptstyle{\Comon(\ca{V}\textrm{-}\Mat(X,X))^\op
\;\times\;\Mon(\ca{V}\textrm{-}\Mat(Y,Y))}\;\to\;
\scriptstyle{\Mon(\ca{V}\textrm{-}\Mat(Y^X,Y^X))}
\end{equation}
which is just the restriction of $\Hom_{(X,Y),(X,Y)}$
on the category
\begin{align*}
\scriptstyle{\Mon\big((\ca{V}\textrm{-}\Mat^\mathrm{co}\times\ca{V}\textrm{-}\Mat)
((X,Y),(X,Y))\big)}
&\scriptstyle{\cong\;\Mon\big(\ca{V}\textrm{-}\Mat(X,X)^\op
\times\ca{V}\textrm{-}\Mat(Y,Y)\big)} \\
&\scriptstyle{\cong\;\Mon\big(\ca{V}\textrm{-}\Mat(X,X)^\op\big)\times
\Mon\big(\ca{V}\textrm{-}\Mat(Y,Y)\big)} \\
&\scriptstyle{\cong\;\Comon\big(\ca{V}\textrm{-}\Mat(X,X)\big)^\op\times
\Mon\big(\ca{V}\textrm{-}\Mat(Y,Y)\big)}.
\end{align*}
Since a $\ca{V}$-cocategory $\ca{C}_X=(C,X)$
has the structure of a comonoid in the monoidal
$(\ca{V}$-$\Mat(X,X),\circ,1_X)$
and a $\ca{V}$-category $\ca{B}_Y=(B,Y)$
has the structure of a monoid in 
$(\ca{V}$-$\Mat(Y,Y),\circ,1_Y)$,
we deduce that $\Mon(\Hom_{(X,Y),(X,Y)})$
is in fact the object mapping of a functor
\begin{equation}\label{defK}
 K:\ca{V}\textrm{-}\B{Cocat}^\op\times\ca{V}\textrm{-}\B{Cat}
\longrightarrow\ca{V}\textrm{-}\B{Cat}
\end{equation}
which is the restriction of the functor $^g\Hom$
on the product of $\ca{V}$-cocategories and $\ca{V}$-categories.
This concretely means 
that whenever we have a $\ca{V}$-cocategory $\ca{C}_X$
and a $\ca{V}$-category $\ca{B}_Y$, the $\ca{V}$-graph
$K(\ca{C}_X,\ca{B}_Y)\equiv{\Hom(\ca{C},\ca{B})_{Y^X}}$ 
obtains the structure
of a $\ca{V}$-category. 

Explicitly, for each triple of functions $k,s,t\in Y^X$, the
composition law $M:K(\ca{C},\ca{B})(k,s)\otimes K(\ca{C},\ca{B})(s,t)
\to K(\ca{C},\ca{B})(k,t)$ for $\ca{K}(\ca{C},\ca{B})$ is
an arrow
\begin{displaymath}
\prod_{a,a}{[\ca{C}(a',a),\ca{B}(ka',sa)]}\otimes
\prod_{b,b'}{[\ca{C}(b',b),\ca{B}(sb',tb)]}\to
\prod_{c,c'}{[\ca{C}(c',c),\ca{B}(kc',tc)].}
\end{displaymath}
This is defined via its adjunct under the usual tensor-hom adjunction
\begin{displaymath}
\xymatrix @C=.5in
{\scriptstyle{\prod\limits_{a,a'}{[\ca{C}(a',a),\ca{B}(ka',sa)]}\otimes
\prod\limits_{b,b'}{[\ca{C}(b',b),\ca{B}(sb',tb)]}\otimes
\ca{C}(c',c)}\ar[d]_-{\scriptscriptstyle{1\otimes\Delta_{c',c}}}\ar@{-->}[r] & 
\scriptstyle{\ca{B}(kc',tc)}\\
\scriptstyle{\prod\limits_{a,a'}{[\ca{C}(a',a),\ca{B}(ka',sa)]}\otimes
\prod\limits_{b,b'}{[\ca{C}(b',b),\ca{B}(sb',tb)]}\otimes
\sum\limits_{c''}{\ca{C}(c',c'')\otimes\ca{C}(c'',c)}\ar[d]_-{\scriptscriptstyle{s}}} &\\
\scriptstyle{\sum\limits_{c''}{\prod\limits_{a,a'}{[\ca{C}(a',a),\ca{B}(ka',sa)]}\otimes
\ca{C}(c',c'')\otimes
\prod\limits_{b,b'}{[\ca{C}(b',b),B(sb',tb)]}\otimes
\ca{C}(c'',c)}} \ar[d]_-{\scriptscriptstyle{\pi_{c',c''}\otimes1\otimes\pi_{c'',c}\otimes1}} & \\
\scriptstyle{\sum\limits_{c''}{[\ca{C}(c',c''),\ca{B}(kc',sc'')]\otimes
\ca{C}(c',c'')\otimes[\ca{C}(c'',c),\ca{B}(sc'',tc)]\otimes
\ca{C}(c'',c)}}
\ar[r]_-{\scriptscriptstyle{\mathrm{ev}\otimes\mathrm{ev}}} &
\scriptstyle{\sum\limits_{c''}{\ca{B}(kc',sc'')\otimes\ca{B}(sc'',tc)}}
\ar[uuu]_-{\scriptscriptstyle{M_{kc',tc}}}}
\end{displaymath}
for fixed $c,c'$.
The identities for each object $s\in Y^X$ are arrows
\begin{equation}\label{identities}
\eta_k:I\to K(\ca{C},\ca{B})(k,k)=\prod_{a,a'\in X}{[\ca{C}(a',a),\ca{B}(sa',sa)]}
\end{equation}
which correspond uniquely for fixed $a=a'\in X$ 
to the composite
\begin{displaymath}
\xymatrix @R=.3in
{I\otimes\ca{C}(a,a)\ar @{-->}[rrr]
\ar @/_/[dr]_-{1\otimes\epsilon_{a,a}} &&& \ca{B}(sa,sa). \\
& I\otimes I\ar[r]_-{r_I} & I\ar @/_/[ur]_-{\eta_{sa,sa}} &}
\end{displaymath}
At the diagrams above, 
$\Delta$ and $\epsilon$ are the cocomposition and coidentites of $\ca{C}$
and $M,$ $\eta$ the composition and identities of $\ca{B}$.
For $a\neq a'$, 
the arrow (\ref{identities}) corresponds to 
\begin{displaymath}
 I\otimes\ca{C}(a',a)\xrightarrow{1\otimes\epsilon_{a',a}}0\xrightarrow{!}
\ca{B}(sa',sa).
\end{displaymath}
Moreover, it can be checked that
for a $\ca{V}$-cofunctor
$F_f:\ca{C}'_{X'}\to\ca{C}_X$ and a $\ca{V}$-functor
$G_g:\ca{B}_Y\to\ca{B}'_{Y'}$, the
$\ca{V}$-graph arrow
\begin{displaymath}
^g\Hom(F,G)_{g^f}:{^g\Hom}(\ca{C},\ca{B})_{Y^X}
\to{^g\Hom}(\ca{C}',\ca{B}')_{Y'^{X'}}
\end{displaymath}
as defined in (\ref{defHom(F,D)}) is in fact
a $\ca{V}$-functor between the $\ca{V}$-categories,
\emph{i.e.} respects the compositions and 
identities described above.
Therefore we deduce that the functor $K$ 
is well defined.
\begin{prop}\label{Kaction}
Suppose that $\ca{V}$ is a cocomplete symmetric monoidal closed
category with products. The functor $K$ (\ref{defK}) is an action, and so is 
its opposite functor
\begin{displaymath}
 K^\op:
\xymatrix @R=.02in 
{\ca{V}\text{-}\B{Cocat}\times
\ca{V}\text{-}\B{Cat}^\op\ar[r] &
\ca{V}\text{-}\B{Cat}^\op\quad \\
\qquad(\;\ca{C}_X\;,\;\ca{B}_Y\;)\qquad\ar@{|->}[r] & 
\Hom^\op(\ca{C},\ca{B})_{Y^X}.}
\end{displaymath}
\end{prop}
\begin{proof}
By Lemma \ref{inthomaction}, the internal hom functor in 
any symmetric monoidal closed category $\ca{V}$ 
constitutes an action of $\ca{V}^\op$ on $\ca{V}$. Thus
for the symmetric monoidal closed category
of $\ca{V}$-graphs, the functors 
\begin{displaymath}
 ^g\Hom:\ca{V}\textrm{-}\B{Grph}^\op\times
\ca{V}\textrm{-}\B{Grph}\to\ca{V}\textrm{-}\B{Grph}
\end{displaymath}
as well as ${^g\Hom}^\op$ are actions.
As stressed earlier, $K$ is the restriction of $^g\Hom$
on $\ca{V}$-$\B{Cocat}^\op\times\ca{V}$-$\B{Cat}$,
hence there exists isomorphisms 
\begin{align*}
 \Hom(\ca{C}\otimes\ca{D},\ca{A})
&\xrightarrow{\;\sim\;}\Hom(\ca{C},\Hom(\ca{D},\ca{A})) \\
\Hom(\ca{I},\ca{D})&\xrightarrow{\;\sim\;}\ca{D}
\end{align*}
for any $\ca{V}$-cocategories $\ca{C}_X$, $\ca{D}_Y$ and 
$\ca{V}$-category $\ca{A}_Z$, initially in $\ca{V}$-$\B{Grph}$.
Notice that $\otimes$
and $\ca{I}$ of the monoidal $\ca{V}$-$\B{Cocat}$ are inherited from 
$\ca{V}$-$\B{Grph}$, and $\Hom$ is the object function
of both $^g\Hom$ and $K$.

Since $\tilde{S}:\ca{V}\text{-}\B{Cat}\to\ca{V}\text{-}\B{Grph}$
is conservative, these isomorphisms are reflected into 
$\ca{V}$-$\B{Cat}$, and the coherence diagrams still commute.
Therefore $K$ is an action, and in particular its opposite functor
$K^\op$ is an action of the symmetric monoidal category
$\ca{V}$-$\B{Cocat}$ on the category $\ca{V}$-$\B{Cat}^\op$.
\end{proof}
What is left to show is that this action $K^\op$ 
has a parametrized adjoint, which will induce 
the enrichment of the category on
which the monoidal category acts. In order to prove the 
existence of the adjoint in question, we need 
some preliminary results which further clarify
the structure of $\ca{V}$-$\B{Cocat}$.

First of all, we can apply the techniques from
Propositions \ref{moncomonadm} and
\ref{comodlocpresent} regarding the 
expression of the categories $\Comon(\ca{V})$ and
$\Comod_\ca{V}(C)$ as an equifier,
so that we obtain the following result.
\begin{prop}\label{VCocatlocpresent}
 Suppose that $\ca{V}$ is a locally presentable
monoidal category, such that $(-\otimes-)$ preserves
colimits on both sides. Then, the category $\ca{V}$-$\B{Cocat}$
is a locally presentable category.
\end{prop}
\begin{proof}
 Define an endofunctor on the category of $\ca{V}$-graphs by
\begin{displaymath}
 F:\xymatrix @R=.07in @C=.5in
{\ca{V}\textrm{-}\B{Grph}\ar[r] & \ca{V}\textrm{-}\B{Grph}\quad \\
(G,X)\ar@{|.>}[r]\ar[dd]_-{(\psi,f)} & (G\circ G,X)\times(1_X,X)
\ar[dd]^-{F(\psi,f)} \\
\hole \\
(H,Y)\ar@{|.>}[r] & (H\circ H,Y)\times(1_Y,Y).}
\end{displaymath}
The mapping on arrows,
for a 2-cell $\psi:f_*Gf^*\Rightarrow H$, is explicitly
\begin{displaymath}
\xymatrix @C=.5in @R=.7in
{X\rrtwocell<\omit>{<6>\psi}
\ar[rr]|-{\object@{|}}^-G && X
\ar[d]|-{\object@{|}}_-{f_*}
\ar[r]|-{\object@{|}}^-{1_X}
\rtwocell<\omit>{<4>\check{\eta}}
& X\rtwocell<\omit>{<6>\psi}
\ar[r]|-{\object@{|}}^-G & X\ar[d]|-{\object@{|}}^-{f_*} \\
Y\ar[u]|-{\object@{|}}^-{f^*}\ar[rr]|-{\object@{|}}_-H &&
Y\ar@/_/[ur]|-{\object@{|}}_-{f^*} \ar[rr]|-{\object@{|}}_-H && Y}
\xymatrix @R=.1in{\hole \\ \times}
\xymatrix @C=.5in @R=.65in
{X\ar[rr]|-{\object@{|}}^-{1_X}
\ar[drr]|-{\object@{|}}^-{f_*}
 & \rtwocell<\omit>{<4>\cong} & X\ar[d]|-{\object@{|}}^-{f_*} \\
Y\ar[u]|-{\object@{|}}^-{f^*} \ar[rr]|-{\object@{|}}_-{1_Y}
\rtwocell<\omit>{<-4>{\check{\varepsilon}}} &&  Y}
\end{displaymath}
where the left unitor $\lambda$ of the bicategory $\ca{V}$-$\Mat$ is suppressed.

The category of coalgebras $\Coalg F$
for this endofunctor
has as objects $\ca{V}$-graphs $(C,X)$
equipped with a
morphism $\alpha:C\to C\circ C\times 1_X$, \emph{i.e.}
two $\ca{V}$-graph arrows 
\begin{displaymath}
 \alpha_1:(C,X)\to (C\circ C,X)\quad\textrm{and}\quad\alpha_2:(C,X)\to
(1_X,X).
\end{displaymath}
A morphism $(C,\alpha)\to(D,\beta)$
is a $\ca{V}$-graph morphism $(\psi,f):(C,X)\to(D,Y)$
which is compatible with $\alpha$ and $\beta$,
\emph{i.e.} satisfy the equalities
\begin{displaymath}
 \xymatrix @R=.6in
{X\ar@/^8ex/[rrrr]|-{\object@{|}}^-C
\rrtwocell<\omit>{<5>\psi}
\ar[rr]|-{\object@{|}}^-C &
\rrtwocell<\omit>{<-4>\;\alpha_1} & X
\rtwocell<\omit>{<3>\check{\eta}}
\ar[d]|-{\object@{|}}_-{f_*}
\ar[r]|-{\object@{|}}^-{1_X}
& X\rtwocell<\omit>{<5>\psi}
\ar[r]|-{\object@{|}}^-C & X\ar[d]|-{\object@{|}}^-{f_*} \\
Y\ar[u]|-{\object@{|}}^-{f^*}\ar[rr]|-{\object@{|}}_-D &&
Y\ar@/_/[ur]|-{\object@{|}}_-{f^*} \ar[rr]|-{\object@{|}}_-D && Y}
\xymatrix @R=.1in{\hole \\ = \\ \hole}
\xymatrix @R=.2in
{X \ar[rr]|-{\object@{|}}^-C
\rrtwocell<\omit>{<5>\psi}&& X
\ar[dd]|-{\object@{|}}^-{f_*}\\
& \\
Y \ar[uu]|-{\object@{|}}^-{f^*}
\rrtwocell<\omit>{<3>\beta_1}
\ar @/_/[dr]|-{\object@{|}}_-D
\ar[rr]|-{\object@{|}}^-D && Y\\
& Y \ar @/_/[ur]|-{\object@{|}}_-D &}
\end{displaymath}
\begin{displaymath}
\xymatrix @C=.6in @R=.6in
{X\rrtwocell<\omit>{<-2>\;\alpha_2}
\ar @/^5ex/[rr]|-{\object@{|}}^-C
\ar[rr]|-{\object@{|}}_-{1_X}
\ar @{.>}[drr]|-{\object@{|}}_-{f_*}
&& X\ar[d]|-{\object@{|}}^-{f_*}
_{\cong\phantom{ab}} \\
Y\ar[u]|-{\object@{|}}^-{f^*}
\rtwocell<\omit>{<-3.5>\check{\varepsilon}}
\ar[rr]|-{\object@{|}}_-{1_Y}
 && Y}
\xymatrix{ = \\ \hole}
\xymatrix @C=1.1in @R=.6in
{X\ar[r]|-{\object@{|}}
^-C  & X \ar[d]|-{\object@{|}}
^-{f_*} \\
Y\ar[u]|-{\object@{|}}
^-{f^*}
\ar @/_5ex/[r]|-{\object@{|}}
_-{1_Y} \rtwocell<\omit>{<2>\;\beta_2}
\ar[r]|-{\object@{|}}^-D 
\rtwocell<\omit>{<-5>\psi} & Y.}
\end{displaymath}
Notice that 
the category $\Coalg F$ 
contains $\ca{V}$-$\B{Cocat}$ as a 
full subcategory: the morphisms are precisely
the same, by comparing the above diagrams with 
(\ref{cofunctaxioms}) where $\hat{\phi}$ is a mate
of $\psi$, and objects are $\ca{V}$-graphs equipped with cocomposition
and coidentities arrows that don't necessarily satisfy 
coassociativity and counit axioms. 

Since $\ca{V}$-$\B{Cocat}$ is a cocomplete
category by Proposition \ref{VCocatcocomplete},
we claim that it is furthermore accessible,
thus a locally presentable category. It is enough to express
$\ca{V}$-$\B{Cocat}$ as an equifier of a family of 
pairs of natural transformations between
accessible functors, \emph{i.e.} functors
between accessible categories that 
preserve filtered colimits.

First of all, we have to show that the endofunctor
$F$ preserves all filtered colimits.
Take a colimiting cocone
\begin{displaymath}
 \big((G_j,X_j)\xrightarrow{\;(\lambda_j,\tau_j)\;}
(G,X)\,|\,j\in\ca{J}\big)
\end{displaymath}
in $\ca{V}$-$\B{Grph}$ for a diagram like 
(\ref{diagraminVgraph}) for a small filtered category
$\ca{J}$, constructed as in
Proposition \ref{Vgrphcocomplete}, \emph{i.e.}
$(\tau_j:X_j\to X)$ is a colimiting cocone in $\B{Set}$
and $(\lambda_j:(\tau_j)_*C_j(\tau_j)^*\to C)$ is a
colimiting cocone in $\ca{V}$-$\Mat(X,X)$. We require its image
under $F$
\begin{equation}\label{imageunderF}
 F(\lambda_j,\tau_j):
(G_j\circ G_j,X_j)\times(1_{X_j},X_j)\to
(G\circ G,X)\times(1_X,X)
\end{equation}
to be a colimiting cocone in $\ca{V}$-$\B{Grph}$.

For the first part of the diagram, 
we can immediately deduce that 
\begin{displaymath}
(\tau_j)_*\circ G_j\circ(\tau_j)^*\circ(\tau_j)_*\circ G_j\circ(\tau_j)^*
\xrightarrow{\;\lambda_j*\lambda_j\;}G\circ G
\end{displaymath}
is a colimit in $(\ca{V}$-$\Mat(X,X),\circ,1_X)$, as the 
tensor product (horizontal composite) of two colimiting cocones. We claim that 
pre-composing this with the unit 
\begin{displaymath}
1*\check{\eta}*1:(\tau_j)_*\circ G_j\circ1_{X_j}\circ G_j\circ(\tau_j)^*
\to(\tau_j)_*\circ G_j\circ(\tau_j)^*\circ(\tau_j)_*\circ G_j\circ(\tau_j)^*
\end{displaymath}
still gives a colimiting cocone. Indeed, if we take components in $\ca{V}$
of the respective 2-cells in $\ca{V}$-$\Mat$, this comes down to showing that 
the inclusion
\begin{displaymath}
 \sum^{\scriptscriptstyle{\stackrel{\tau_ju=x'}{\tau_jw=x}}}_{z\in X_j}
{G_j(u,z)\otimes G_j(z,w)}\hookrightarrow
\sum^{\scriptscriptstyle{\stackrel{\tau_ju=x'}{\tau_jw=x}}}_{\tau_ja=\tau_jb}
{G_j(u,a)\otimes G_j(b,w)}
\end{displaymath}
for any two fixed $x,x'\in X$, where $u,w,a,b\in X_j$, does not 
alter the colimit. One way of showing this is by considering
the following discrete opfibrations over the filtered shape 
$\ca{J}$:
\begin{align*}
\ca{L}&=\{(j,a,b)\,|\,j\in\ca{J},a,b\in X_j,\tau_ja=\tau_jb\} \\
\ca{M}&=\{(j,z)\,|\,j\in\ca{J},z\in X_j\}
\end{align*}
where for example the arrows $(j,a,b)\to(j',a',b')$ in $\ca{L}$ 
are determined by arrows $\theta:j\to j'$ 
in $\ca{J}$ such that $a'=f_\theta(a)$ and $b'=f_\theta(b)$
(the function $f_\theta:X_j\to X_{j'}$ 
is the image of the diagram (\ref{diagraminVgraph}) in $\B{Set}$).
We can now define diagrams of shape $\ca{L}$ and 
$\ca{M}$ in $\ca{V}$
\begin{displaymath}
 L:\xymatrix@R=.02in{\ca{L}\ar[r] & \ca{V}\qquad\qquad \\
(j,a,b)\ar@{|->}[r] & G_j(u,a)\otimes G_j(b,w)}\qquad
M:\xymatrix@R=.02in{\ca{M}\ar[r] & \ca{V}\qquad\qquad \\
(j,z)\ar@{|->}[r] & G_j(u,z)\otimes G_j(z,w)}
\end{displaymath}
and appropriately on morphisms. The colimits for these diagrams in 
$\ca{V}$, taking into account that the fibres 
are discrete categories, are
\begin{align*}
 \colim L & \cong\colim_j\sum_{\tau_ja=\tau_jb}
{G_j(u,a)\otimes G_j(b,w)} \\
\colim M & \cong\colim_j\sum_{z\in X_j}
{G_j(u,z)\otimes G_j(z,w)}.
\end{align*}
Finally, notice that there exists a functor 
$T:\ca{M}\to\ca{L}$ mapping each $(j,z)$ to $(j,z,z)$
and making the triangle
\begin{displaymath}
 \xymatrix
{\ca{M}\ar[rr]^-T\ar[dr]_-M && \ca{L}\ar[dl]^-L \\
&\ca{V}&}
\end{displaymath}
commute. Due to the construction of filtered colimits in $\B{Set}$,
it is not hard to show that the slice category $\big((j,z,w)\downarrow T\big)$
is non-empty and connected. Hence $T$ is a final 
functor and we can restrict the diagram on $\ca{L}$
to $\ca{M}$ without changing the colimit, as claimed.

For the second part of the diagram, it is enough to show that
\begin{displaymath}
\xymatrix @C=.5in @R=.25in
{\rrtwocell<\omit>{<5>\check{\varepsilon}}  &
X_j\ar[dr]|-{\object@{|}}^-{(\tau_j)_*} & \\
X\ar@/_2ex/[rr]|-{\object@{|}}_-{1_X} 
\ar[ur]|-{\object@{|}}^-{(\tau_j)^*} && Y}
\end{displaymath}
is a colimiting cocone in $\ca{V}$-$\Mat(X,X)$, for the diagram
mapping each $j$ to 
\begin{displaymath}
\xymatrix{X\ar[r]|-{\object@{|}}^-{(\tau_j)^*} & 
X_j\ar[r]|-{\object@{|}}^-{1_{X_j}} & 
X_j\ar[r]|-{\object@{|}}^-{(\tau_j)_*} & X} 
\end{displaymath}
as in (\ref{defK}). This can be established
by first verifying that $\check{\varepsilon}$ is a cocone, 
and then that it has the required universal property. 

We have thus shown that the cocone (\ref{imageunderF}) is indeed
colimiting, hence $F$ is a finitary functor as required.
This part of the proof is due to Ignacio Lopez Franco.

Since $\ca{V}$-$\B{Grph}$ is locally presentable
and the endofunctor $F$ preserves filtered colimits,
$\Coalg F$ is a locally presentable category 
by the basic facts for endofunctor coalgebra categories
in Section \ref{Categoriesofmonoidsandcomonoids}. Also
the forgetful functor 
$\overline{V}:\Coalg F\to\ca{V}$-$\B{Grph}$
creates all colimits. Now consider the following pairs of 
natural transformations
between functors from $\Coalg F$ to $\ca{V}$-$\B{Grph}$:
\begin{displaymath}
\phi^1,\psi^1:\overline{V}\Rightarrow FF\overline{V},\quad
\phi^2,\psi^2:\overline{V}\Rightarrow (-\circ 1_X)\overline{V},\quad
\phi^3,\psi^3:\overline{V}\Rightarrow \overline{V}(-\circ 1_X)
\end{displaymath}
given by the components
\begin{align*}
\phi^1_{(C,X)}:
\xymatrix @R=.1in
{X\ar @/_/[dr]|-{\object@{|}}_-C
\ar @/^3ex/[drr]|-{\object@{|}}^-C
\ar @/^4ex/[rrr]|-{\object@{|}}^-C
\drrtwocell<\omit>{<+.3>\;\alpha_1}
&\drrtwocell<\omit>{<-1.3>\;\alpha_1} && X, \\
& X\ar[r]|-{\object@{|}}_-C &
X\ar @/_/[ur]|-{\object@{|}}_-C &}&\quad
\psi^1_{(C,X)}:
\xymatrix @R=.1in
{X\ar @/_/[dr]|-{\object@{|}}_-C
\ar @/^4ex/[rrr]|-{\object@{|}}^-C
&&& X \\
\urrtwocell<\omit>{<-1.3>\;\alpha_1} & 
X\ar[r]|-{\object@{|}}_-C 
\ar @/^3ex/[urr]|-{\object@{|}}^-C 
\urrtwocell<\omit>{<+.3>\;\alpha_1} &
X\ar @/_/[ur]|-{\object@{|}}_-C &} \\
\phi^2_{(C,X)}:
\xymatrix @C=.6in @R=.05in
{X\drrtwocell<\omit>{<-2.3>\;\alpha_1}
\drtwocell<\omit>{<-0.4>\;\alpha_2}
\ar @/_2ex/[dr]|-{\object@{|}}_-{1_X}
\ar @/^3ex/[dr]|-{\object@{|}}^-C 
\ar @/^4ex/[rr]|-{\object@{|}}^-C
&& X, \\
& X \ar @/_/[ur]|-{\object@{|}}_-C &}&\quad
\psi^2_{(C,X)}:
\xymatrix @R=.05in @C=.6in
{X \ar@/_/[dr]|-{\object@{|}}
_-{1_X} \ar @/^3ex/[rr]|-{\object@{|}}^-C
\rrtwocell<\omit>{'\cong} && X \\
& X\ar@/_/[ur]|-{\object@{|}}_-{C}  &} \\
\phi^3_{(C,X)}:
\xymatrix @R=.05in @C=.6in
{X \ar @/_/[dr]|-{\object@{|}}_-C
\ar @/^4ex/[rr]|-{\object@{|}}^-C
&& X, \\
\urrtwocell<\omit>{<-2.3>\;\alpha_1}
& X \ar @/_2ex/[ur]|-{\object@{|}}_-{1_X} 
\ar @/^3ex/[ur]|-{\object@{|}}^-C 
\urtwocell<\omit>{<-0.4>\;\alpha_2} &}&\quad
\psi^3_{(C,X)}:
\xymatrix @R=.05in @C=.6in
{X \ar@/_/[dr]|-{\object@{|}}
_-{C} \ar @/^3ex/[rr]|-{\object@{|}}^-C
\rrtwocell<\omit>{'\cong} && X. \\
 & X\ar@/_/[ur]|-{\object@{|}}_-{1_X} &}
\end{align*}
It is now clear that the full
subcategory of $\Coalg F$ spanned by those objects
$(C,X)$ which satisfy $\phi^i_{(C,X)}=\psi^i_{(C,X)}$
is precisely the category of $\ca{V}$-cocategories,
\begin{displaymath}
 \B{Eq}((\phi^i,\psi^i)_{i=1,2,3})=\ca{V}\textrm{-}\B{Cocat}
\end{displaymath}
as in Definition \ref{cocategory}. Since all categories
and functors involved are accessible,
$\ca{V}$-$\B{Cocat}$ is accessible too.
\end{proof}
The fact that $\ca{V}$-$\B{Cocat}$ is a locally 
presentable category is very useful for the proof
of existence of various adjoints, as seen below.
\begin{prop}\label{VCocatcomonadic}
 Suppose $\ca{V}$ is a locally presentable monoidal category
 such that $\otimes$ preserves colimits in both entries.
 The forgetful functor $\tilde{U}:\ca{V}$-$\B{Cocat}\to
\ca{V}$-$\B{Grph}$ is comonadic.
\end{prop}
\begin{proof}
 By Proposition \ref{cofreeVcocatfunctor}
the forgetful $\tilde{U}$ has a right adjoint, namely
the cofree $\ca{V}$-cocategory functor $\tilde{R}$.
By adjusting the arguments of Proposition
\ref{moncomonadm},
consider the following commutative triangle
\begin{displaymath}
\xymatrix @C=.7in @R=.5in
{\ca{V}\textrm{-}\B{Cocat}\ar[dr]_-{\tilde{U}}
\ar@{^(->}[r]^-{\iota} & \Coalg G 
\ar[d]^-{\overline{V}} \\
& \ca{V}\textrm{-}\B{Grph}}
\end{displaymath}
where the top functor is the inclusion of
the full subcategory in the functor coalgebra category 
as described above, and the respective forgetful functors
discard the structures maps $\alpha$
of the coalgebras.
We already know that 
$\Coalg F$ is comonadic over $\ca{V}$-$\B{Grph}$,
hence $\overline{V}$ creates equalizers of 
split pairs,
so it is enough to show that the inclusion $\iota$
also creates equalizers of
split pairs, since we already have $\tilde{U}\dashv\tilde{R}$.
Both $\ca{V}$-$\B{Cocat}$
and $\ca{V}$-$\B{Grph}$ are locally presentable categories
so in particular complete,
and it is easy to see that $\iota$ preserves and reflects,
thus creates, all limits. Hence $\tilde{U}$ satisfy the conditions of 
Precise Monadicity Theorem and the result follows.
\end{proof}
\begin{prop}\label{VCocatclosed}
 Suppose that $\ca{V}$ is a locally 
presentable symmetric monoidal closed category. Then the category
of $\ca{V}$-cocategories is symmetric monoidal closed as well.
\end{prop}
\begin{proof}
The symmetric monoidal structure of $\ca{V}$-$\B{Cocat}$
was described in the previous section
and is given by a functor of two variables
\begin{displaymath}
 -\otimes-:\ca{V}\text{-}\B{Cocat}^\op\times\ca{V}\text{-}\B{Cocat}
\to\ca{V}\text{-}\B{Cocat}.
\end{displaymath}
The functor $(-\otimes\ca{D}_Y)$ for a 
fixed $\ca{V}$-cocategory $\ca{D}_Y$
evidently has a right adjoint:
the following commutative diagram
\begin{displaymath}
 \xymatrix @C=.7in @R=.5in
{\ca{V}\textrm{-}\B{Cocat}\ar[r]^-{(-\otimes\ca{D}_Y)}
\ar[d]_-{\tilde{U}} & \ca{V}\textrm{-}\B{Cocat}\ar[d]^-{\tilde{U}} \\
\ca{V}\textrm{-}\B{Grph}\ar[r]_-{(-\otimes\tilde{U}\ca{D}_Y)} &
\ca{V}\textrm{-}\B{Grph}}
\end{displaymath}
shows it is cocontinuous, since 
the comonadic $\tilde{U}$ creates all colimits
and the bottom arrow preserves them by the adjunction
$(-\otimes\ca{G}_Y)\dashv{^g\Hom}(\ca{G}_Y,-)$
for any $\ca{V}$-graph $\ca{G}_Y$ (Proposition \ref{VGrphclosed}).
Also $\ca{V}$-$\B{Cocat}$ is a locally presentable category,
hence cocomplete with a small dense subcategory. 
Thus by Theorem \ref{Kelly} for example, we have an adjunction
\begin{equation}\label{adjunctionHOMg}
\xymatrix @C=.8in
{\ca{V}\textrm{-}\B{Cocat} \ar @<+.8ex>[r]^-
{-\otimes \ca{D}_Y}\ar@{}[r]|-{\bot}
& \ca{V}\textrm{-}\B{Cocat}\ar @<+.8ex>[l]^-
{^g\HOM(\ca{D}_Y,-)}}
\end{equation}
which exhibits the uniquely induced bifunctor
\begin{displaymath}
 ^g\HOM:\ca{V}\textrm{-}\B{Cocat}^\op\times
\ca{V}\textrm{-}\B{Cocat}\longrightarrow\ca{V}\textrm{-}\B{Cocat}
\end{displaymath}
as the internal hom of $\ca{V}$-$\B{Cocat}$.
\end{proof}
At this point, we possess all the necessary tools
in order to show the existence of an 
adjoint of the action $K^\op$ 
as outlined earlier, as well as demonstrate the
enrichment of $\ca{V}$-categories in $\ca{V}$-cocategories.
\begin{prop}\label{Texistence}
 The functor $K^\op:\ca{V}\textrm{-}\B{Cocat}\times
\ca{V}\textrm{-}\B{Cat}^\op\to
\ca{V}\textrm{-}\B{Cat}^\op$ has a parametrized adjoint
\begin{equation}\label{defT}
 T:\ca{V}\textrm{-}\B{Cat}^\op\times\ca{V}\textrm{-}\B{Cat}
\longrightarrow
\ca{V}\textrm{-}\B{Cocat},
\end{equation}
given by adjunctions $K(-,\ca{B}_Y)^\op\dashv T(-,\ca{B}_Y)$
for every $\ca{V}$-category $\ca{B}_Y$.
\end{prop}
\begin{proof}
By Proposition \ref{VCocatlocpresent}, the domain
$\ca{V}$-$\B{Cocat}$ of $K(-,\ca{B})^\op$
is locally presentable, hence cocomplete
with a small dense subcategory, namely the 
presentable objects. Now consider
the following diagram
\begin{displaymath}
\xymatrix @C=1in @R=.6in
 {\ca{V}\textrm{-}\B{Cocat}\ar[r]^-{K(-,\ca{B}_Y)^\op}
\ar[d]_-{\tilde{U}} & \ca{V}\textrm{-}\B{Cat}^\op \ar[d]^-{\tilde{S}} \\
\ca{V}\textrm{-}\B{Grph}\ar[r]_-{^g\Hom(-,\tilde{S}\ca{B}_Y)^\op} &
\ca{V}\textrm{-}\B{Grph}^\op}
\end{displaymath}
which commutes by definition of $K$,
and the left and right legs create all colimits by
Propositions \ref{VCatmonadic} and \ref{VCocatcomonadic}.
The bottom arrow preserves all colimits by 
$^g\Hom(-,\ca{G}_Y)^\op\dashv{^g\Hom(-,\ca{G}_Y)}$
for any internal hom functor in a monoidal closed category,
thus the functor $K(-,\ca{B})^\op$ is cocontinuous.
By Kelly's adjoint functor theorem \ref{Kelly},
there are adjunctions
\begin{displaymath}
 \xymatrix @C=.7in
{\ca{V}\textrm{-}\B{Cocat}\ar@<+.8ex>[r]^-{K(-,\ca{B}_Y)^\op}
\ar@{}[r]|-{\bot} & \ca{V}\textrm{-}\B{Cat}^\op\ar@<.8ex>[l]^-{T(-,\ca{B}_Y)}}
\end{displaymath}
for all $\ca{V}$-categories $\ca{B}_Y$.
This suffices to uniquely make $T$ into a functor
of two variables (\ref{defT}), which is by definition the parametrized
adjoint of $K^\op$.
\end{proof}
The functor $T$, which is a generalization of the universal
measuring comonoid functor $P$ (\ref{defP}) in
the `many-object' context,
is called \emph{generalized Sweedler hom}. Morever,
it can also be deduced that the functor
$K(\ca{C}_X,-)^\op$ has a right adjoint
for any $\ca{V}$-cocategory $\ca{C}_X$,
or equivalently its opposite functor has a left adjoint.
The following diagram
\begin{displaymath}
 \xymatrix @C=.9in @R=.6in
{\ca{V}\textrm{-}\B{Cat}\ar[r]^-{K(\ca{C}_X,-)} 
\ar[d]_-{\tilde{S}} & 
\ca{V}\textrm{-}\B{Cat}\ar[d]^-{\tilde{S}} \\
\ca{V}\textrm{-}\B{Grph}\ar[r]_-{^g\Hom(\tilde{U}\ca{C}_X,-)} &
\ca{V}\textrm{-}\B{Grph}}
\end{displaymath}
commutes, 
where $\tilde{S}$ is the monadic forgetful functor and the
locally presentable category $\ca{V}$-$\B{Cat}$
has all coequalizers. Thus
by Dubuc's Adjoint Triangle Theorem in \cite{AdjointTriangles},
the existence of a left
adjoint $(\ca{C}_X\otimes-)\dashv{^g\Hom(\ca{C}_X,-)}$
for any (underlying) $\ca{V}$-graph $\ca{C}_X$
in the symmetric monoidal closed $\ca{V}$-$\B{Grph}$
implies the existence of a left adjoint 
$(\ca{C}_X\triangleright-)$ of the top functor.
The induced functor of two variables
\begin{displaymath}
 \triangleright:\ca{V}\textrm{-}\B{Cocat}\times
\ca{V}\textrm{-}\B{Cat}\longrightarrow\ca{V}\textrm{-}\B{Cat}
\end{displaymath}
is called the \emph{generalized Sweedler product},
since it is an extension of the respective functor
(\ref{deftriangle}).

The conditions of Corollaries 
\ref{importcor1} and \ref{importcor2} are now satisfied,
for the symmetric monoidal category closed $\ca{V}$-$\B{Cocat}$
which acts on the opposite of the category $\ca{V}$-$\B{Cat}$
via the action $K^\op$.
\begin{thm}\label{VCatenrichedinVCocat}
Suppose $\ca{V}$ is a symmetric monoidal closed category 
which is locally presentable, and
$T$ is the generalized Sweedler hom functor.
\begin{enumerate}
\item The opposite category of $\ca{V}$-categories
$\ca{V}$-$\B{Cat}^\op$ is enriched in the category
of $\ca{V}$-cocategories $\ca{V}$-$\B{Cocat}$,
with hom-objects 
\begin{displaymath}
 \ca{V}\textrm{-}\B{Cat}^\op(\ca{A}_X,\ca{B}_Y)=T(\ca{B}_Y,\ca{A}_X)
\end{displaymath}
where the ($\ca{V}$-$\B{Cocat}$)-enriched category
with underlying category $\ca{V}$-$\B{Cat}^\op$ is denoted
by the same name.
\item The category $\ca{V}$-$\B{Cat}$
is a tensored and cotensored ($\ca{V}$-$\B{Cocat}$)-enriched category,
with hom-objects
\begin{displaymath}
 \ca{V}\textrm{-}\B{Cat}(\ca{A}_X,\ca{B}_Y)=T(\ca{A}_X,\ca{B}_Y),
\end{displaymath}
cotensor product $K(\ca{C},\ca{B})_{Y^Z}$ and tensor product
$\ca{C}_Z\triangleright\ca{A}_X$, for any 
$\ca{V}$-co\-ca\-te\-go\-ry $\ca{C}_Z$ and 
any $\ca{V}$-categories $\ca{A}_X,\ca{B}_Y$.
\end{enumerate}
\end{thm}

\section{Graphs, categories and cocategories as (op)fibrations}
\label{fibrationalview}

This section presents a different approach
to establishing the enrichment of $\ca{V}$-categories
in $\ca{V}$-cocategories. In the section above, the result
follows from the existence of an adjoint $T$ which
constitues the enriched hom-functor, as a straightforward
application of an adjoint functor theorem (Proposition
\ref{Texistence}). This is possible basically due to 
local presentability of $\ca{V}$-$\B{Cocat}$.
However, the categories $\ca{V}$-$\B{Grph}$, $\ca{V}$-$\B{Cat}$
and $\ca{V}$-$\B{Cocat}$ also have a structure which places them in 
a fibrational context,
allowing the application of the theory of 
fibrations of Chapter \ref{fibrations}. In particular,
we will show how we can alternatively obtain this adjoint $T$
as an application of Theorem \ref{totaladjointthm}
regarding adjunctions between fibrations.

First of all, we are going to exhibit in detail the 
fibrational structure of the categories involved,
a well-known fact at least for $\ca{V}$-categories over sets.
We initially assume that $\ca{V}$ is a cocomplete
monoidal category, such that the tensor product
preserves colimits on both sides.

\begin{prop}\label{VGrphbifibr}
The category $\ca{V}$-$\B{Grph}$ 
of small $\ca{V}$-graphs is a 
bifibration over $\B{Set}$.
\end{prop}
\begin{proof}
Due to the correspondence
between fibrations and pseudofunctors
studied in Theorem \ref{maintheoremfibr},
it is enough to define certain 
indexed categories, \emph{i.e.} pseudofunctors 
$\ps{M}:\B{Set}^\mathrm{op}\to\B{Cat}$
and $\ps{F}:\B{Set}\to\B{Cat}$ which
give rise to a fibration and opfibration with
total category isomorphic to $\ca{V}$-$\B{Grph}$,
via the Grothendieck construction.

Define the pseudofunctor $\ps{M}$ as follows:
\begin{equation}\label{defM}
\ps{M}:
\xymatrix @R=.02in
{\B{Set}^\op\ar[r] & \B{Cat} \\
X\ar @{|.>}[r]
\ar [dd]_-f & \ca{V}\textrm{-}\Mat(X,X) \\
\hole \\
Y\ar @{|.>}[r] &
\ca{V}\textrm{-}\Mat(Y,Y),\ar[uu]_-{\ps{M}f}}
\end{equation}
where the functor $\ps{M}f$ is given by the mapping
\begin{displaymath}
\SelectTips{eu}{10}
\xymatrix{(Y\ar[r]|-{\object@{|}}^-H & Y)
\ar @{|.>}[r] & 
(X \ar[r]|-{\object@{|}}^-{f_*} & 
Y \ar[r]|-{\object@{|}}^-H &
Y\ar[r]|-{\object@{|}}^-{f^*} & X)}
\end{displaymath}
on objects and 
\begin{displaymath}
\SelectTips{eu}{10}
\xymatrix{(Y \ar @/^2ex/[r]|-{\object@{|}}^-H
\ar@/_2ex/[r]|-{\object@{|}}_-{H'}
\rtwocell<\omit>{\sigma} & Y)
\ar @{|.>}[r] & (X\ar[r]|-{\object@{|}}
^-{f_*} & Y \ar @/^2ex/[r]|-{\object@{|}}^-H
\ar@/_2ex/[r]|-{\object@{|}}_-{H'}
\rtwocell<\omit>{\sigma} & Y
\ar[r]|-{\object@{|}}^-{f^*} & X)}
\end{displaymath}
on arrows. In other words,
$\ps{M}f=(f^*\circ-\circ f_*)$ is the functor 
`pre-composition with $f_*$ and
post-composition with $f^*$', where the induced 
$\ca{V}$-matrices $f_*$ and 
$f^*$ are defined as in (\ref{f*}).
In terms of components, the family $\{H(y',y)\}_{y,y'\in Y}$
of objects in $\ca{V}$ which defines the $\ca{V}$-matrix $H$,
is mapped to the family
\begin{displaymath}
\{\big((\ps{M}f)H\big)(x',x)\}_{x,x'\in X}=
\{I\otimes H(fx',fx)\otimes I\}_{fx,fx'\in Y}
\end{displaymath}
and the family
$\{\sigma_{y',y}:H(y',y)\to H'(y',y)\}_{y',y}$
of arrows in $\ca{V}$ which define the 2-cell $\sigma$,
is mapped to the family
\begin{displaymath}
\big((\ps{M}f)\sigma\big)_{x',x}:
I\otimes H(fx',fx)\otimes I
\xrightarrow{\;1\otimes\sigma_{fx',fx}\otimes1\;} 
I\otimes H'(fx',fx)\otimes I
\end{displaymath}
for all $x',x\in X$.

In order to show that the above data determine
a pseudofunctor $\ps{M}$, we need the existence of 
certain natural isomorphisms
satisfying coherence conditions as in Definition \ref{laxfunctor}.
For every triple of 
sets $X,Y,Z$, there is a natural isomorphism
$\delta$ with components
\begin{displaymath}
\xymatrix @R=.04in
{& \ca{V}\text{-}\Mat(Y,Y)\ar@/^/[dr]^-{\ps{M}f} & \\
\ca{V}\text{-}\Mat(Z,Z)\ar@/^/[ur]^-{\ps{M}g}
\ar@/_3ex/[rr]_-{\ps{M}(g\circ f)}
\rrtwocell<\omit>{\quad\delta^{g,f}} && 
\ca{V}\text{-}\Mat(X,X)}
\end{displaymath}
for any $f:X\to Y$ and $g:Y\to Z$, satisfying the commutativity
of (\ref{laxcond1}). Explicitely, each $\delta^{g,f}$
has components, for each $\ca{V}$-matrix$\SelectTips{eu}{10}
\xymatrix@C=.2in{J:Z\ar[r]|-{\object@{|}} & Z,}$ the invertible arrows
\begin{displaymath}
\delta^{g,f}_J:(\ps{M}f\circ\ps{M}g)J
\xrightarrow{\;\sim\;}\ps{M}(g\circ f)J
\end{displaymath}
in $\ca{V}$-$\Mat(X,X)$
which are the composite 2-cells
\begin{equation}\label{deltaforM}
\xymatrix @R=.4in @C=.5in
{X\ar[r]|-{\object@{|}}^-{f_*}
\ar @/_2ex/[drr]|-{\object@{|}}_-{(gf)_*} &
Y\ar[r]|-{\object@{|}}^-{g_*} &
Z\ar[r]|-{\object@{|}}^-J
\ar @{=}[d] & 
Z\ar[r]|-{\object@{|}}^-{g^*}
\ar @{=}[d] &
Y\ar[r]|-{\object@{|}}^-{f^*} &
X \\
& \rtwocell<\omit>{<-4>\quad\zeta^{g,f}} & 
Z\ar[r]|-{\object@{|}}_-J
\rtwocell<\omit>{<-4>\;\;1_J} &
Z\rtwocell<\omit>{<-4>\quad\xi^{g,f}}
\ar @/_2ex/[urr]|-{\object@{|}}_-{(gf)^*} &&}
\end{equation}
where the isomorphisms
$\zeta$ and $\xi$ are
defined in Lemma \ref{isosofstars}.
This 2-isomorphism 
\begin{displaymath}
 \delta^{g,f}=\xi^{g,f}*1_J*\zeta^{g,f}
\end{displaymath}
is given by the family
of invertible arrows
\begin{displaymath}
(\delta^{g,f}_J)_{x',x}:
I\otimes I\otimes J(gfx',gfx)\otimes I\otimes I
\xrightarrow{\;r_I\otimes1\otimes r_I\;}I\otimes J(gfx',gfx)\otimes I 
\end{displaymath}
in $\ca{V}$, and the coherence axiom is satisfied by
the properties of $\xi$ and $\zeta$ (see
Lemma \ref{corisos}). 
Moreover, for any set $X$ there is a natural isomorphism
$\gamma$ with components the natural transformations
\begin{displaymath}
\xymatrix @C=.5in{\ca{V}\text{-}\B{Mat}(X,X)
\rrtwocell<5>^{\B{1}_{\ca{V}\text{-}\B{Mat}(X,X)}}
_{\ps{M}(\mathrm{id}_X)}
{\quad\gamma^X} && \ca{V}\textrm{-}\B{Mat}
(X,X)}
\end{displaymath}
where $\mathrm{id}_X$ is the identity function
on any set $X$ and $\B{1}$ is the identity functor.
Explicitly, $\gamma^X$ has as components 
invertible arrows in $\ca{V}$-$\Mat(X,X)$ 
\begin{displaymath}
\gamma^X_G:\B{1}_{\ca{V}\text{-}\B{Mat}(X,X)}G\xrightarrow{\;\sim\;}
\ps{M}(\mathrm{id}_X)G
\end{displaymath}
for any $\ca{V}$-matrix$\SelectTips{eu}{10}
\xymatrix@C=.2in{G:X\ar[r]|-{\object@{|}} & X,}$which 
are the composite 2-cells
\begin{equation}\label{gammaforM}
\gamma^X_G:
\xymatrix @R=.25in @C=.5in
{X\ar @/^5ex/[rrr]|-{\object@{|}}^-G
\ar @/_/[dr]|-{\object@{|}}_-{(\mathrm{id}_X)_*}
\rrtwocell<\omit>{\quad\rho_G^{-1}}
&&& X. \\
& X \ar[r]|-{\object@{|}}_-G
\ar @/^3ex/[urr]|-{\object@{|}}^-G
\rrtwocell<\omit>{<-3>{\quad{\lambda_G^{-1}}}}
& X \ar @/_/[ur]|-{\object@{|}}_-{(\mathrm{id}_X)^*} &}
\end{equation}
By recalling that $(\mathrm{id}_X)^*=
(\mathrm{id}_X)_*=1_X$ by (\ref{idX=1X}),
this isomorphism
\begin{displaymath}
 \gamma^X_G=(\lambda_G^{-1} 1_X)\cdot(\rho_G^{-1})
\end{displaymath}
consists of the family of invertible arrows
\begin{displaymath}
(\gamma^X_G)_{x',x}:G(x',x)\xrightarrow{l^{-1}}I\otimes G(x',x)
\xrightarrow{1\otimes r^{-1}}I\otimes G(x',x)\otimes I
\end{displaymath}
in $\ca{V}$. It can be verified that the axioms
(\ref{laxcond2}) are satisfied, therefore $\ps{M}$ is 
a pseudofunctor.

The Grothendieck category $\Gr{G}\ps{M}$
for this pseudofunctor
has as objects pairs $(G,X)$, where $X$ is
a set and $G$ is an object in the category 
$\ps{M}X=\ca{V}$-$\B{Mat}(X,X)$, and as arrows
$(\phi,f):(G,X)\to(H,Y)$ pairs
\begin{displaymath}
\begin{cases} 
G\xrightarrow{\phi}(\ps{M}f)H &\text{in }\ps{M}X\\
X\xrightarrow{f}Y &\text{in }\B{Set}
\end{cases}
\quad =\quad
\begin{cases} 
G\xrightarrow{\phi}f^*\circ H\circ f_* 
&\text{in }\ca{V}\textrm{-}\B{Mat}(X,X)\\
X\xrightarrow{f}Y &\text{in }\B{Set}.
\end{cases}
\end{displaymath}
By Definition \ref{charactVGrph}, this category
is isomorphic to $\ca{V}$-$\B{Grph}$, in 
the sense that there is a one-to-one correspondence
between the objects, which can actually be identified,
and the hom-sets.
Thus $\ps{M}$ gives rise to a fibration
$P_\ps{M}:\Gr{G}\ps{M}\to\B{Set}$ which is isomorphic to
the forgetful functor
$Q:\ca{V}\text{-}\B{Grph}\to\B{Set}$, \emph{i.e.}
\begin{displaymath}
 \xymatrix @C=.4in @R=.5in
{\Gr{G}\ps{M}\ar[rr]^-{\cong}
\ar[dr]_-{P_{\ps{M}}} && 
\ca{V}\text{-}\B{Grph}\ar[dl]^-{Q} \\
& \B{Set} &}
\end{displaymath}
commutes by definition of the functors involved,
hence $Q$ is a fibration.

Now, define a covariant indexed category
$\ps{F}$ as follows: 
\begin{equation}\label{defF}
\ps{F}:
\xymatrix @R=.02in
{\B{Set}\ar[r] & \B{Cat} \\
X\ar @{|.>}[r]
\ar [dd]_-f & \ca{V}\textrm{-}\Mat(X,X)\ar[dd]^-{\ps{F}f} \\
\hole \\
Y\ar @{|.>}[r] &
\ca{V}\textrm{-}\Mat(Y,Y)}
\end{equation}
where the mapping on objects is the same
as for the pseudofunctor $\ps{M}$ above,
and $\ps{F}f$ is the mapping 
\begin{displaymath}
\SelectTips{eu}{10}
\xymatrix{(X \ar @/^2ex/[r]|-{\object@{|}}^-G
\ar@/_2ex/[r]|-{\object@{|}}_-{G'}
\rtwocell<\omit>{\tau} & X)
\ar @{|.>}[r] & (Y\ar[r]|-{\object@{|}}
^-{f^*} & X \ar @/^2ex/[r]|-{\object@{|}}^-G
\ar@/_2ex/[r]|-{\object@{|}}_-{G'}
\rtwocell<\omit>{\tau} & X
\ar[r]|-{\object@{|}}^-{f_*} & Y)}
\end{displaymath}
on objects and on arrows, \emph{i.e.}
$\ps{F}f=(f_*\circ-\circ f^*$).
In terms of components, the family $\{G(x',x)\}_{x,x'\in X}$
of objects in $\ca{V}$ which define the $\ca{V}$-matrix
$G$, is mapped to the family
\begin{displaymath}
\{\ps{F}f(G)(y',y)\}_{y,y'\in Y}=
\{\sum_{\scriptscriptstyle{\stackrel{fx'=y'}{fx=y}}}I\otimes
G(x',x)\otimes I\}_{y,y'\in Y}
\end{displaymath}
and the family $\{\tau_{x,x'}:G(x',x)\to G'(x',x)\}_{x,x'}$ of arrows
in $\ca{V}$ which defines the 2-cell $\tau$,
is mapped to the family of arrows
\begin{displaymath}
\ps{F}f(\tau)_{y',y}:
\sum I\otimes G(x',x)\otimes I
\xrightarrow{\;\sum 1\otimes\sigma_{x',x}\otimes1\;}
\sum I\otimes G'(x',x)\otimes I,
\end{displaymath}
where the sums are over $x,x'$ such that $fx'=y'$, $fx=fy$,
based on the computations of Section 
\ref{bicatVMat}. Again, there exist natural isomorphisms
$\delta$, $\gamma$ with components
\begin{gather*}
\delta^{g,f}:\ps{F}g\circ\ps{F}f\Rightarrow\ps{F}(g\circ f):
\ca{V}\textrm{-}\B{Mat}(X,X)\to\ca{V}\textrm{-}\B{Mat}(Z,Z) \\
\gamma^X:
\B{1}_{\ca{V}\textrm{-}\Mat(X,X)}\Rightarrow
\ps{F}(\mathrm{id}_X):
\ca{V}\textrm{-}\Mat(X,X)\to
\ca{V}\textrm{-}\Mat(X,X)
\end{gather*}
which satisfy the properties
(\ref{laxcond1}) and (\ref{laxcond2}) 
from the definition of a pseudofunctor.
In fact, they are essentially the same as in the 
case of $\ps{M}$, \emph{i.e.} $\delta$ now has components
the invertible composite 2-cells
\begin{equation}\label{deltaforF}
\xymatrix @R=.002in{\hole \\
\delta^{f,g}_G:} 
\xymatrix @R=.01in @C=.5in
{& Y \ar @/^/[dr]|-{\object@{|}}^-{f^*}
&&& Y \ar @/^/[dr]|-{\object@{|}}^-{g_*} & \\
Z\ar @/^/[ur]|-{\object@{|}}^-{g^*}
\ar @/_3ex/[rr]|-{\object@{|}}_-{(gf)^*}
\rrtwocell<\omit>{\quad\xi^{g,f}} &&
X\ar @/^2ex/[r]|-{\object@{|}}^-G
\ar @/_2ex/[r]|-{\object@{|}}_-G 
\rtwocell<\omit>{\;1_G} & 
X\ar @/^/[ur]|-{\object@{|}}^-{f_*}
\rrtwocell<\omit>{\quad\zeta^{g,f}}
\ar @/_3ex/[rr]|-{\object@{|}}_-{(gf)_*} && Z,}
\end{equation}
which are formed like (\ref{deltaforM})
but composing with $\zeta$ and $\xi$ in the reverse order,
and $\gamma$ is the same as in (\ref{gammaforM}).
Therefore $\ps{F}$ is a pseudofunctor,
and by Theorem \ref{maintheoremopfibr} it gives 
rise to an opfibration
\begin{displaymath}
 U_{\ps{F}}:\Gr{G}\ps{F}\to\B{Set}.
\end{displaymath}
The Grothendieck category in this case coincides with
the isomorphic characterization of $\ca{V}$-$\B{Grph}$
in Definition \ref{charactVGrph}, with the
`second version' form of arrows.
Hence $U_{\ps{F}}$ is again isomorphic to
the forgetful
$Q:\ca{V}\textrm{-}\B{Grph}\to\B{Set}$,
endowing it with the structure of an opfibration. 
Thus $\ca{V}$-$\B{Grph}$ is 
a bifibration over $\B{Set}$.
\end{proof}
Notice that we could immediately deduce that the fibration
$Q$ is a bifibration
via Remark \ref{rmkforadjointintexingbifr}. The reindexing functor
$\ps{M}f=f_*\circ\text{-}\circ f^*$ does have a left adjoint
$f^*\circ\text{-}\circ f_*$, by the natural bijection between 2-cells
of the form (\ref{mates2cells}). 
We explicitly constructed the opfibration
above in order to employ it later.

An immediate consequence of viewing
the category of $\ca{V}$-graphs as a bifibration 
is that we can discuss fibred and opfibred
limits and colimits (see Section \ref{fibredadjunctions}).
\begin{cor}\label{Qfibredcomplete}
The bifibration $Q:\ca{V}$-$\B{Grph}\to\B{Set}$ has all fibred
limits and all opfibred colimits, 
when $\ca{V}$ is complete and cocomplete respectively.
\end{cor}
\begin{proof}
By Corollary \ref{AhasPpreserves} and its dual, 
an (op)fibration with (co)complete
base category has all (op)fibred (co)limits if and only if the 
total category has all (co)limits and the fibration strictly preserves them.
In this case, the base of the bifibration is the complete
and cocomplete category of sets, 
and since the total category $\ca{V}$-$\B{Grph}$ 
is (co)complete when $\ca{V}$ is
and the forgetful functor $Q$
preserves limits and colimits `on the nose'
by construction, the result follows.
\end{proof}
Moreover, by Proposition \ref{reindexcont} 
we can now deduce that the reindexing functors
$(f^*\circ\text{-}\circ f_*)$ and $(f_*\circ\text{-}\circ f^*)$
preserve all limits and colimits between the complete and 
cocomplete fibres $\ca{V}$-$\Mat(X,X)$. 
The latter was evident by Proposition \ref{propVMat}.
 
The construction of the two pseudofunctors
$\ps{M}$ and $\ps{F}$ which exhibit $\ca{V}$-$\B{Grph}$
as a bifibred category over $\B{Set}$ clarify the way
in which the categories $\ca{V}$-$\B{Cat}$ and $\ca{V}$-$\B{Cocat}$
are themselves fibred and opfibred respectively over $\B{Set}$.
\begin{prop}\label{VCatfibred}
The category $\ca{V}$-$\B{Cat}$ of small
$\ca{V}$-categories is a fibration over $\B{Set}$.
\end{prop}
\begin{proof}
Similarly to the above proof,
it will suffice to construct an indexed category
$\ps{L}:\B{Set}^\mathrm{op}\to\B{Cat}$ such that
the category $\ca{V}$-$\B{Cat}$ is isomorphic
to the Grothendieck category of the fibration 
$P_\ps{L}$.

Define the pseudofunctor $\ps{L}$ as follows: a 
set $X$ is mapped to the category 
\begin{displaymath}
\ps{L}X=\Mon(\ca{V}\textrm{-}\Mat(X,X))
\end{displaymath}
of monoids of the monoidal category 
of endoarrows $(\ca{V}\text{-}\Mat(X,X),\circ,1_X)$, 
and a function between sets
$f:X\to Y$ is mapped contravariantly
to the functor
\begin{displaymath}
\ps{L}f:
\xymatrix@R=.02in
{\Mon(\ca{V}\textrm{-}\Mat(Y,Y))\ar[r] &
\Mon(\ca{V}\textrm{-}\Mat(X,X)) \\
(B,\mu,\eta)\ar@{|.>}[r]\ar[dd]_-{\sigma} & 
(f^*Bf_*,\mu',\eta')\ar[dd]^-{f^*\sigma f_*} \\
\hole \\
(E,\mu,\eta)\ar@{|.>}[r] & (f^*Ef_*,\mu',\eta').}
\end{displaymath}
As described in detail in Lemma \ref{B*monoid},
the induced monoid $f^*Bf_*$ has multiplication
$\mu'=f^*[\mu\cdot(B\check{\epsilon}B)]f_*$
and unit $\eta'=(f^*\eta f_*)\cdot\check{\eta}$,
where $\check{\epsilon}$ and $\check{\eta}$ are the counit
and unit of the adjunction $f_*\dashv f^*$,
and also $(f^*\sigma f_*)$ can easily be checked to commute
with the appropriate monoid structure maps. 
Evidently, this functor
$\ps{L}f$ is just $\ps{M}f=(f^*\circ-\circ f_*)$ 
defined in (\ref{defM}), restricted between the respective
categories of monoids.

Again, we need to identify natural transformations
$\gamma$ and $\delta$ satisfying certain coherence
axioms, for $\ps{L}$ to be a pseudofunctor according
to Definition \ref{laxfunctor}.
In this case, these will have components
natural isomorphisms
\begin{gather*}
\delta^{g,f}:\ps{L}f\circ\ps{L}g\Rightarrow\ps{L}(g\circ f):
\Mon(\ca{V}\textrm{-}\B{Mat}(Z,Z))\to\Mon(\ca{V}\textrm{-}\B{Mat}(X,X)) \\
\gamma^X:
\B{1}_{\Mon(\ca{V}\textrm{-}\Mat(X,X))}\Rightarrow
\ps{L}(\mathrm{id}_X):
\Mon(\ca{V}\textrm{-}\Mat(X,X))\to
\Mon(\ca{V}\textrm{-}\Mat(X,X))
\end{gather*}
for $X\xrightarrow{f}Y\xrightarrow{g}Z$ in $\B{Set}$,
where $\mathrm{id}_X$ is the identity function.
We can define $\delta^{g,f}$ and 
$\gamma^X$ to be natural transformations given exactly 
as the ones for the pseudofunctor $\ps{M}$,
as in (\ref{deltaforM}) and (\ref{gammaforM}).
The domains and codomains of these composite
2-cells are by default monoids
in the appropriate endoarrow hom-categories of $\ca{V}$-matrices,
and it can be verified via computations that
the invertible arrows
$\delta^{g,f}_J$ and $\gamma^X_A$ for 
monoids$\SelectTips{eu}{10}\xymatrix @C=.2in
{J:Z\ar[r]|-{\object@{|}} & Z}$and$\SelectTips{eu}{10}\xymatrix @C=.2in
{A:X\ar[r]|-{\object@{|}} & X}$commute
with the respective multiplications and units
of the monoids involved. Moreover,
the diagrams (\ref{laxcond1}, \ref{laxcond2}) commute
because they do for all $\ca{V}$-matrices, by
pseudofunctoriality of $\ps{M}$.
Therefore $\ps{L}$ is indeed a pseudofunctor.

If we construct the Grothendieck category
for $\ps{L}:\B{Set}^\op\to\B{Cat}$,
with objects pairs $(A,X)$ where $A\in\Mon(\ca{V}\text{-}
\Mat(X,X))$ for a set $X$, and arrows
$(A,X)\to (B,Y)$ pairs
\begin{displaymath}
\begin{cases} 
A\xrightarrow{\phi}f^*Bf_* 
&\text{in }\Mon(\ca{V}\textrm{-}\Mat(X,X))\\
X\xrightarrow{f}Y &\text{in }\B{Set},
\end{cases}
\end{displaymath}
it is evident by Lemma \ref{charactVCat}
that $\Gr{G}\ps{L}\cong\ca{V}\text{-}\B{Cat}$.
Moreover, both forgetful functors to $\B{Set}$
have the same effect on objects and on arrows,
namely separating the set-part of the data.
Hence
\begin{displaymath}
 P:\ca{V}\text{-}\B{Cat}\longrightarrow\B{Set}
\end{displaymath}
is a fibration, isomorphic to $P_\ps{L}$
arising via the Grothendieck construction.
\end{proof}
\begin{cor}\label{Pfibredcomplete}
 The fibration $P:\ca{V}$-$\B{Cat}\to\B{Set}$ has
all fibred limits when $\ca{V}$ is complete.
\end{cor}
\begin{proof}
Since the fibration $P$ has as base category 
the complete category $\B{Set}$, in order
for $P$ to have all fibred limits it suffices for 
$\ca{V}$-$\B{Cat}$ to be complete
and for the forgetful $P$ to preserve all limits strictly,
again by Corollary \ref{AhasPpreserves}.
Corollary \ref{VCatcomplete} ensures that $\ca{V}$-$\B{Cat}$
has all limits and since a limit of $\ca{V}$-graphs
has as underlying set precisely 
the limit of the sets, the result follows.
\end{proof}
Finally, in order to establish
that $\ca{V}$-$\B{Cocat}$ is opfibred over $\B{Set}$,
we are going to use the pseudofunctor
$\ps{F}$ defined as in ($\ref{defF}$).
\begin{prop}\label{VCocatopfibred}
The category $\ca{V}$-$\B{Cocat}$ of small
$\ca{V}$-cocategories is an 
opfibration over $\B{Set}$.
\end{prop}
\begin{proof}
We will once more 
construct a covariant indexed category
$\ps{K}:\B{Set}\to\B{Cat}$,
for which the Grothendieck
construction gives a category
isomorphic to $\ca{V}$-$\B{Cocat}$
along with the forgetful functor to sets, mapping
every $\ca{V}$-cocategory to its set of objects.

Define $\ps{K}$ as follows:
a set $X$ is mapped to
the category of comonoids in the monoidal category 
($\ca{V}$-$\B{Mat}(X,X),\circ,1_X$),
and a function $f:X\to Y$ is mapped to the functor
\begin{displaymath}
\ps{K}f:\Comon(\ca{V}\textrm{-}\B{Mat}(X,X))\to
\Comon(\ca{V}\textrm{-}\B{Mat}(Y,Y))
\end{displaymath}
which precomposes with $f^*$ and postcomposes
with $f_*$ both $\ca{V}$-matrices and 2-cells.
Explicitly, the functor $\ps{K}f$
is defined on objects by
\begin{displaymath}
\xymatrix{(C,\Delta,\epsilon)\ar @{|.>}[r] &
(f_*Cf^*,\Delta',\epsilon')}
\end{displaymath}
where $\Delta'=f_*[(C\check{\eta}C)\cdot\Delta]f^*$
and $\epsilon'=\check{\epsilon}\cdot(f_*\epsilon f^*)$
as described in detail in Lemma \ref{C*comonoid},
and on arrows
\begin{displaymath}
\xymatrix
{(C\ar @{=>}[r]^-\tau & D) \ar @{|.>}[r] & (f_*Cf^*
\ar @{=>}[r]^-{f_*\tau f^*} & f_*Df^*)}
\end{displaymath}
where $f_*\tau f^*$ can easily be verified to 
commute with the respective
counits and comultiplications.
Again, notice that $\ps{K}f$
is in fact the restriction of $\ps{F}f$ (\ref{defF})
to the categories of comonoids. 
The above mappings define a pseudofunctor $\ps{K}$,
since the two natural transformations 
$\gamma$ and $\delta$ in this case,
with components
natural isomorphisms
\begin{displaymath}
 \xymatrix @C=.4in
{\Comon(\ca{V}\textrm{-}\B{Mat}(X,X))\ar @/^3ex/[rr]^-{\ps{K}g\circ\ps{K}f}
\ar @/_3ex/[rr]_-{\ps{K}(g\circ f)} \rrtwocell<\omit>{\quad\delta^{g,f}} &&
\Comon(\ca{V}\textrm{-}\B{Mat}(Z,Z))}
\end{displaymath}
\begin{displaymath}
\xymatrix @C=.4in
{\Comon(\ca{V}\textrm{-}\B{Mat}(X,X))
\ar @/^3ex/[rr]^-{\B{1}_{\Comon(\ca{V}\textrm{-}\B{Mat}(X,X))}}
\ar @/_3ex/[rr]_-{\ps{K}(\;\mathrm{id}_X)}
\rrtwocell<\omit>{\;\gamma^X} &&
\Comon(\ca{V}\textrm{-}\B{Mat}(X,X))}
\end{displaymath}
consist of the invertible composite 2-cells
as in (\ref{deltaforF}) and (\ref{gammaforM})
for the pseudofunctor $\ps{F}$.
Their domains and codomains are 
by construction comonoids in the appropriate
categories of $\ca{V}$-matrices, and they satisfy
the properties of comonoid morphisms.
Hence $\delta$ and
$\gamma$ are well-defined, and the diagrams
(\ref{laxcond1}, \ref{laxcond2}) commute by 
pseudofunctoriality of $\ps{F}$.

The Grothendieck category $\Gr{G}\ps{K}$
for this pseudofunctor
has as objects pairs $(C,X)$ where
$C\in\Comon(\ca{V}\text{-}\Mat(X,X))$ for a set $X$,
and as arrows $(C,X)\to(D,Y)$ pairs
\begin{displaymath}
 \begin{cases} 
f_*Cf^*\xrightarrow{\psi}D 
&\text{in }\Comon(\ca{V}\textrm{-}\Mat(Y,Y))\\
X\xrightarrow{f}Y &\text{in }\B{Set}.
\end{cases}
\end{displaymath}
By Lemma \ref{charactVCocat}, this is isomorphic 
to the category $\ca{V}$-$\B{Cocat}$.
As a result, the forgetful functor
\begin{displaymath}
 W:\ca{V}\text{-}\B{Cocat}\longrightarrow\B{Set}
\end{displaymath}
is an opfibration, isomorphic to $U_\ps{K}$
arising via the Grothendieck construction
since they have the same effect on objects
and on morphisms.
\end{proof}
\begin{cor}\label{Wopfibredcocomplete}
The opfibration $W:\ca{V}$-$\B{Cocat}\to\B{Set}$ has 
all opfibred colimits, when $\ca{V}$ is locally presentable.
\end{cor}
\begin{proof}
 The base category of this opfibration
is again the cocomplete category of sets,
and also the total category $\ca{V}$-$\B{Cocat}$
has all colimits which by construction are
strictly preserved by the forgetful functor, 
see Proposition
\ref{VCocatcocomplete}. Therefore by the dual of
Corollary \ref{AhasPpreserves}, the opfibration $W$ has
all opfibred colimits.
\end{proof}
\begin{rmk*}
For the definition of the two pseudofunctors 
which give rise to $\ca{V}$-categories 
and $\ca{V}$-cocategories as their Grothendieck categories, 
the functors $\ps{M}f=f^*\circ-\circ f_*$ 
and $\ps{F}f=f_*\circ-\circ f^*$ as in 
(\ref{defM}), (\ref{defF}) were employed. Lemmas
\ref{B*monoid} and \ref{C*comonoid} suggested already
that these two functors may `lift' to the 
respective categories of monoids and comonoids.
This can be further clarified if we observe that 
both these functors have the structure
of a lax/colax monoidal functor 
respectively, between
the monoidal hom-categories of endomorphisms
in $\ca{V}$-$\Mat$.
For example, for a function $f:X\to Y$
and two $\ca{V}$-matrices$\SelectTips{eu}{10}\xymatrix @C=.2in
{B,B':Y\ar[r]|-{\object@{|}} & Y,}$the lax
monoidal structure map 
\begin{displaymath}
\phi_{B,B'}: f^*\circ B\circ f_*\circ f^*\circ B'\circ f_*
\Rightarrow f^*\circ B\circ B'\circ f_*
\end{displaymath}
of $\ps{M}f$ has components
the composite 2-cells
\begin{displaymath}
 \xymatrix @R=.1in 
{&&& X\ar@/^/[dr]|-{\object@{|}}^-{f_*} &&& \\
X\ar[r]|-{\object@{|}}^-{f_*} & Y\ar[r]|-{\object@{|}}^-{B'}
& Y\ar@/^/[ur]|-{\object@{|}}^-{f^*}
\ar[rr]|-{\object@{|}}_-{1_Y} 
\rrtwocell<\omit>{<-2>\check{\varepsilon}}
\ar@/_5ex/[rrr]|-{\object@{|}}_-B 
& \rtwocell<\omit>{<2>\stackrel{\;\rho}{\cong}} &
Y\ar[r]|-{\object@{|}}^-B & Y\ar[r]|-{\object@{|}}^-{f^*} &X.}
\end{displaymath}
Similarly for $\phi_0$,
and also for the functor $\ps{F}f$.
Therefore, these lax/colax monoidal functors induce functors
between the categories of monoids and comonoids
of $\ca{V}$-$\Mat(X,X)$ in a straightforward way, as in (\ref{MonF}). 
\end{rmk*}
The fibre categories for the bifibration, fibration and
opfibration $Q,\;P$ and $W$ respectively are
\begin{align*}
\ca{V}\text{-}\B{Grph}_X&=\ca{V}\text{-}\Mat(X,X) \\
\ca{V}\text{-}\B{Cat}_X&=\Mon(\ca{V}\text{-}\Mat(X,X)) \\
\ca{V}\text{-}\B{Cocat}_X&=\Comon(\ca{V}\text{-}\Mat(X,X)).
\end{align*}
Notice that, even if the total categories 
of $\ca{V}$-categories and 
$\ca{V}$-cocategories have a monoidal structure
as seen in Section \ref{VcatsandVcocats},
their fibres are not monoidal categories,
since the monoidal $(\ca{V}$-$\Mat(X,X),\circ,1_X)$ fails
to be symmetric or braided.

We now turn back to the 
primary question of the existence of a
right adjoint for the functor
\begin{displaymath}
 K(-,\ca{B}_Y)^\op:\ca{V}\text{-}\B{Cocat}\longrightarrow
\ca{V}\text{-}\B{Cat}^\op
\end{displaymath}
coming from $K$ (\ref{defK}), which
in reality is the internal hom
functor $^g\Hom$ of the monoidal closed category
of small $\ca{V}$-graphs
restricted on $\ca{V}$-cocategories and $\ca{V}$-categories,
as explained in detail in Section 
\ref{enrichmentofVcatsinVcocats}.
The plan is to now obtain this adjoint
via the theory of fibrations and in particular 
from Theorem \ref{totaladjointthm}, and then
the enrichment of $\ca{V}$-categories
in $\ca{V}$-cocategories will follow in the exact same way
as in the end of previous section.
\begin{lem}\label{partialKopfibred1cell}
The diagram
\begin{displaymath}
\xymatrix @C=.8in @R=.4in
{\ca{V}\textrm{-}\B{Cocat}
\ar[r]^-{K(-,(B,Y))^\mathrm{op}}
\ar[d]_-{W} & \ca{V}\textrm{-}\B{Cat}^\mathrm{op}
\ar[d]^-{P^\mathrm{op}} \\
\B{Set}\ar[r]_-{{Y^{(-)}}^\op} & \B{Set}^\mathrm{op}}
\end{displaymath}
exhibits $\big(K(-,(B,Y))^\op,{Y^{(-)}}^\op\big)$ as 
an opfibred 1-cell between 
the opfibrations $W$ and $P^\mathrm{op}$.
\end{lem}
\begin{proof}
It is straightforward to verify that the above diagram
commutes, since the set of objects
of the internal hom is by construction the
exponential of the underlying sets of objects
of the $\ca{V}$-cocategory and
the $\ca{V}$-category, and similarly for morphisms
(see Proposition \ref{VGrphclosed}).
It remains to show that $K(-,(B,Y))^\op$
is a cocartesian functor, or equivalently
that the contravariant
$K(-,(B,Y))$ maps cocartesian liftings
to cartesian liftings. 

Using the canonical choice of cocartesian liftings
for any opfibration obtained via the Grothendieck 
construction (see Theorem \ref{maintheoremfibr}),
consider a cocartesian lifting of 
$(C,X)$ along the function $f:X\to Z$ with
respect to the opfibration 
$W:\ca{V}$-$\B{Cocat}\to\B{Set}$:
\begin{displaymath}
\xymatrix @C=.3in
{C\ar[rrr]^-{1_{f_*Cf^*}}
\ar @{.>}[d] &&& f_*Cf^* \ar @{.>}[d] &\textrm{in }\ca{V}\text{-}\B{Cocat} \\
X\ar[rrr]_-{f} &&& Z & \textrm{in }\B{Set}.}
\end{displaymath}
Notice that the pair notation
for objects in the total category
is suppressed, since the respective set
of objects of the $\ca{V}$-cocategories is clear 
from the picture. The image of this arrow
under $K(-,(B,Y))$ gives 
\begin{displaymath}
\xymatrix @C=.3in
{\Hom((f_*Cf^*,Z),(B,Y)) \ar[rrr]^-{[[1_{f_*Cf^*},1_B]]}
\ar @{.>}[d] &&& \Hom((C,X),(B,Y)) \ar @{.>}[d] 
&\textrm{in }\ca{V}\text{-}\B{Cat} \\
Y^Z\ar[rrr]_-{Y^f} &&& Y^X & \textrm{in }\B{Set}}
\end{displaymath}
by definition of the functor $^g\Hom$, and the 2-cell
in $\Mon(\ca{V}$-$\Mat(Y^Z,Y^Z))$
\begin{displaymath}
 \xymatrix @C=1.3in @R=.6in
{Y^Z\ar[r]|-{\object@{|}}^-{\Hom(f_*Cf^*,B)}
\ar[d]|-{\object@{|}}_-{(Y^f)_*}
\rtwocell<\omit>{<7>\qquad\qquad[[1_{f_*Cf^*},1_B]]} & 
Y^Z \\
Y^X\ar[r]|-{\object@{|}}_-{\Hom(C,B)} & 
Y^X\ar[u]|-{\object@{|}}_-{(Y^f)^*}}
\end{displaymath}
as in (\ref{def[[]]}) explicitly consists of arrows 
$[[1_{f_*Cf^*},1_B]]_{k,s}$
\begin{displaymath}
 \prod_{z,z'}{[\sum_{\scriptscriptstyle{\stackrel{fx=z}{fx'=z'}}}I\otimes
C(x',x)\otimes I,B(kz',sz)]}\to 
I\otimes\prod_{x,x'}{[C(x',x),B(kfx',sfx)]}\otimes I
\end{displaymath}
in $\ca{V}$ for all $k,s\in Y^Z$. On the other hand,
the canonical cartesian lifting of
$(\Hom(C,B),Y^X)$ along the function $Y^f$
with respect to the fibration $P:\ca{V}$-$\B{Cat}\to\B{Set}$
is
\begin{displaymath}
\xymatrix @C=.4in
{(Y^f)^*\Hom(C,B)(Y^f)_*
\ar[rrr]^-{1_{(Y^f)^*\Hom(C,B)(Y^f)_*}}
\ar @{.>}[d] &&& \Hom(C,B) \ar @{.>}[d]\\
Y^Z\ar[rrr]_-{Y^f} &&& Y^X.}
\end{displaymath}
By comparing this cartesian arrow with the image
under $K(-,(B,Y))$ above, it is enough to 
show that $[[1_{f^*Cf_*},1_B]]$
is isomorphic to the identity arrow
in the fibre 
$\ca{V}$-$\B{Cat}_{Y^Z}=\Mon(\ca{V}$-$\Mat(Y^Z,Y^Z))$.
We have 
natural isomorphisms
\begin{align*}
 \prod_{z,z'}{[\sum_{\scriptscriptstyle{\stackrel{fx=z}{fx'=z'}}}I\otimes
C(x',x)\otimes I,B(kz',sz)]}&\cong
\prod_{z,z'}{\prod_{\scriptscriptstyle{\stackrel{fx=z}{fx'=z'}}}{[I\otimes
C(x',x)\otimes I,B(kz',sz)]}} \\
 \cong\prod_{x',x}{[I\otimes
C(x',x)\otimes I,B(kfx',sfx)]}&\cong 
I\otimes\prod_{x,x'}{[C(x',x),B(kfx',sfx)]}\otimes I
\end{align*}
since sum commutes with $\otimes$ and $[-,A]$
maps colimits to limits for any monoidal closed category $\ca{V}$.
By applying $r$ and $l$ to 
move the $I$'s appropriately, we 
deduce that the result holds.
\end{proof}
\begin{lem}\label{lemmaspecialadjun}
Suppose that $\ca{V}$ is a locally presentable
symmetric monoidal closed category, and 
$\varepsilon$ is the counit of the exponential
adjunction
\begin{equation}\label{exponentialadjunction}
\xymatrix @C=.5in
{\B{Set} \ar @<+.8ex>[r]^-{{Y^{(-)}}^\op}
\ar@{}[r]|-{\bot} & 
\B{Set}^\op.\ar @<+.8ex>[l]^-{Y^{(-)}}}
\end{equation}
For any $\ca{V}$-category $\ca{B}_Y$ and any set $Z$, 
the composite functor
\begin{displaymath}
 \ca{V}\textrm{-}\B{Cocat}_{Y^Z}
\xrightarrow{\;K(-,\ca{B}_Y)^\op\;}
\ca{V}\textrm{-}\B{Cat}_{Y^{Y^Z}}^\mathrm{op}
\xrightarrow{\quad(\varepsilon_Z)_!\quad}
\ca{V}\textrm{-}\B{Cat}_Z^\mathrm{op}
\end{displaymath}
has a right adjoint $T_0(-,\ca{B}_Y)$.
\end{lem}
\begin{proof} 
We can rewrite the above composite as
\begin{displaymath}
 \xymatrix @C=1.2in @R=.5in
{\Comon(\ca{V}\text{-}\Mat(Y^Z,Y^Z))
\ar@{-->}[dr]
\ar[r]^-{\Mon(\Hom(-,(B,Y)))^\op} &
\Mon(\ca{V}\text{-}\Mat({Y^Y}^Z,{Y^Y}^Z))^\op
\ar[d]^-{\ps{L}\varepsilon_Z} \\
& \Mon(\ca{V}\text{-}\Mat(Z,Z))^\op}
\end{displaymath}
where the top functor was already given by (\ref{MonHom_})
but is now viewed as the induced `functor between the fibres'
from $K(-,\ca{B})$, as in (\ref{inducedfunfibres}). 
By Corollary \ref{cofreecomonVMat},
the category of comonoids 
$\Comon(\ca{V}\text{-}\Mat(Y^Z,Y^Z))$
of the locally presentable monoidal
category $\ca{V}\text{-}\Mat(Y^Z,Y^Z)$
is also locally presentable. As such,
it is in particular cocomplete
and has a small dense subcategory.
Moreover, the following commutative diagram
\begin{displaymath}
\xymatrix @C=1.3in
{\Comon(\ca{V}\textrm{-}\B{Mat}(X,X))
\ar[r]^-{\Mon(\Hom_X(-,(B,Y))^\op} \ar[d]_-U &
\Mon(\ca{V}\textrm{-}\B{Mat}(Y^X,Y^X))^\op \ar[d]^-{S^\op} \\
\ca{V}\textrm{-}\B{Mat}(X,X) 
\ar[r]_-{\Hom_X(-,(B,Y))^\op} &
\ca{V}\textrm{-}\B{Mat}(Y^X,Y^X)^\op}
\end{displaymath}
for a fixed $\ca{V}$-category $(B,Y)$
shows that the top arrow $K_X(-,(B,Y))$ is cocontinuous
for any set $X$. 
This is the case because the functors
$U$ and $S^\op$ are comonadic
by Corollary \ref{cofreecomonVMat}
and the bottom arrow is the cocontinuous
internal hom $^g\Hom(-,B)^\op$
restricted between the cocomplete fibres.
Finally, Proposition \ref{reindexcont}
ensures that all reindexing functors
for the fibration $P$ are continuous, 
since $P:\ca{V}\text{-}\B{Cat}\to\B{Set}$ 
has all fibred limits by
Corollary \ref{Pfibredcomplete}.
So the ones for the opfibration
$P^\op$ are cocontinuous, and in particular so is
$(\varepsilon_Z)_!$.
Thus, by Kelly's theorem \ref{Kelly},
the composite functor 
$(\varepsilon_Z)_!\circ K_{Y^Z}(-,\ca{B}_Y)$
has a right adjoint
\begin{displaymath}
\xymatrix @C=1in
{\ca{V}\textrm{-}\B{Cocat}_{Y^Z}
\ar @<+.8ex>[r]^-{(\varepsilon_Z)_!\circ
K(-,\ca{B}_Y)^\op}\ar@{}[r]|-\bot
& \ca{V}\textrm{-}\B{Cat}_Z^\op.
\ar @<+.8ex>[l]^-{T_0(-,\ca{B}_Y)}}
\end{displaymath}
\end{proof}
At this point, all the assumptions of Lemma
\ref{totaladjointlem} are satisfied,
so we can apply it in this setting to 
obtain the enriched hom-functor $T$, evidently
isomorphic to (\ref{defT}) of the previous section.
\begin{prop}\label{propexistenceS2}
The functor between the total categories
\begin{displaymath}
K^\op:\ca{V}\textrm{-}\B{Cocat}
\times\ca{V}\textrm{-}\B{Cat}^\op\to
\ca{V}\textrm{-}\B{Cat}^\op
\end{displaymath}
has a parametrized
adjoint
\begin{displaymath}
 T:\ca{V}\textrm{-}\B{Cat}^\mathrm{op}\times
\ca{V}\textrm{-}\B{Cat}\longrightarrow
\ca{V}\textrm{-}\B{Cocat}
\end{displaymath}
which makes the following diagram serially commute:
\begin{displaymath}
 \xymatrix @R=.5in @C=.8in 
{ \ca{V}\textrm{-}\B{Cocat} \ar @<+.8ex>[r]^-{K(-,\ca{B}_Y)^\mathrm{op}}
\ar@{}[r]|-\bot \ar[d]_-{W} &
\ca{V}\textrm{-}\B{Cat}^\mathrm{op}\ar @<+.8ex>[l]^-{T(-,\ca{B}_Y)}
\ar[d]^-{P^\mathrm{op}} \\ 
\B{Set} \ar@<+.8ex>[r]^-{{Y^{(-)}}^\mathrm{op}} 
\ar@{}[r]|-\bot &
\B{Set}^\mathrm{op}. \ar @<+.8ex>[l]^-{Y^{(-)}}}
\end{displaymath}
\end{prop}
\begin{proof}
By Lemma \ref{partialKopfibred1cell},
we have an opfibred 1-cell
$(K(-,\ca{B}_Y)^\op,{Y^{(-)}}^\op)$
between the opfibrations
$W:\ca{V}$-$\B{Cocat}\to\B{Set}$ and
$P^\op:\ca{V}$-$\B{Cat}^\op\to\B{Set}^\op$.
Also there is an adjunction ${Y^{(-)}}^\op\dashv Y^{(-)}$
between the base categories,
since the exponential is the internal hom
in the cartesian monoidal closed $\B{Set}$.
Lastly, by Lemma \ref{lemmaspecialadjun} the composite
functor between the fibre categories
\begin{displaymath}
 K_{Y^Z}(-,\ca{B}_Y)\circ(\varepsilon_Z)_!:
\ca{V}\textrm{-}\B{Cocat}_{Y^Z}\longrightarrow
\ca{V}\textrm{-}\B{Cat}_Z^\op
\end{displaymath}
has a right adjoint
$T_Z(-,\ca{B}_Y)$ for any fixed set $Z$. 

Therefore,
by Theorem \ref{totaladjointthm}
the functor $K(-,\ca{B}_Y)^\op$ has a right adjoint 
$T(-,\ca{B}_Y)$ between
the total categories
\begin{equation}\label{adjunctionKSpartial}
 \xymatrix @C=.8in
{\ca{V}\textrm{-}\B{Cocat} \ar @<+.8ex>[r]^-{K(-,\ca{B}_Y)^\op}
\ar@{}[r]|-\bot &
\ca{V}\textrm{-}\B{Cat}^\op,\ar @<+.8ex>[l]^-{T(-,\ca{B}_Y)}}
\end{equation}
with $(K(-,\ca{B}_Y)^\op,{Y^{(-)}}^\op)\dashv 
(T(-,\ca{B}_Y),Y^{(-)})$ in $\B{Cat}^\B{2}$,
\emph{i.e.} ($W,P^\op$) is a map of adjunctions. 
The adjunction (\ref{adjunctionKSpartial})
for any $\ca{V}$-category $\ca{B}_Y$
makes $T$ into a functor of two variables 
such that the natural isomorphism
of the adjunction is natural in all three variables,
\emph{i.e.} $T$ is the parametrized adjoint of $K^\op$.
\end{proof}
Notice that the above proof
of existence of the adjoint $T$
between the total categories
automatically provides us with the underlying
set of objects of the $\ca{V}$-cocategory
$T(\ca{A}_X,\ca{B}_Y)$, namely
$Y^X$. On the contrary, Proposition \ref{Texistence}
did not establish this piece of data in a straightforward
way. We could also explicitly construct
$T$ on arrows, using the formulas provided in 
Section \ref{fibredadjunctions}.

\section{$\ca{V}$-modules and $\ca{V}$-comodules}\label{globalVmodsVcomods}

In these last two sections of the chapter,
the aim is to generalize the existence of
the universal measuring comodule,
which induces an enrichment of the global
category of modules in the 
global category of comodules as seen in Section
\ref{Universalmeasuringcomodule}.
This follows the idea
of the ($\ca{V}$-$\B{Cocat}$)-enrichment of 
$\ca{V}$-$\B{Cat}$ as the many-object
generalization of the enrichment of 
monoids in comonoids in $\ca{V}$ of Section
\ref{Universalmeasuringcomonoid}.

We are going to closely follow the development
of the previous chapter
in defining the global category of $\ca{V}$-enriched
modules and the global category of 
$\ca{V}$-enriched comodules. On that level,
by employing the theory of fibrations and
opfibrations once again, we will determine
the objects that induce the enrichment in question.

In Section \ref{Vbimodulesandmodules}, a brief account
of the bicategory of $\ca{V}$-bimodules
was given, with emphasis on the one-sided
modules of $\ca{V}$-categories. In the
current setting of the bicategory of $\ca{V}$-matrices, 
we can reformulate Definition
\ref{leftAmodule} of a left $\ca{A}$-module
for a $\ca{V}$-category $\ca{A}$
in a way that will clarify
how $\ca{V}$-modules are a special case of 
modules for a monad in a bicategory as in Section \ref{monadsinbicats}.
Motivated by Remark \ref{modswithdom1},
we are here interested in
categories of modules in the bicategory $\ca{V}$-$\Mat$
with fixed domain 
the singleton set $1=\{*\}$, \emph{i.e.} the initial object
in $\B{Set}$. The monads in this bicategory
are of course $\ca{V}$-categories$\SelectTips{eu}{10}\xymatrix@C=.2in
{A:X\ar[r]|-{\object@{|}} & X.}$

For the following definitions, the assumptions on
$\ca{V}$ are initially the ones 
required for the formation of $\ca{V}$-$\Mat$, \emph{i.e.}
existence of sums which are preserved by
the tensor product on both sides. 
\begin{defi}\label{leftcaAmodule}
 The category of \emph{left $\ca{A}$-modules}
for a $\ca{V}$-category $\ca{A}_X$, \emph{i.e}
a monad $(A,X)$, is the category
of left $A$-modules with domain the singleton set 
in the bicategory $\ca{V}$-$\Mat$, \emph{i.e.}
the category of Eilenberg-Moore algebras 
for the (ordinary) monad
`post-composition with $A$' on the hom-category
$\ca{V}\text{-}\Mat(1,X)$
\begin{displaymath}
 \ca{V}\text{-}_\ca{A}\Mod=
\ca{V}\text{-}\Mat(1,X)^{\ca{V}\text{-}\Mat(1,A)}.
\end{displaymath}
Explicitly, the objects are $\ca{V}$-matrices$\SelectTips{eu}{10}
\xymatrix@C=.2in
{\Psi:1\ar[r]|-{\object@{|}} & X}$given by a family 
$\{\Psi(x)\}_{x\in X}$ of objects
in $\ca{V}$, equipped with an action
$\mu:A\circ\Psi\Rightarrow\Psi$ with components
\begin{displaymath}
 \mu_x:\sum_{x'\in X}{A(x,x')\otimes\Psi(x')\to\Psi(x)}
\end{displaymath}
such that the diagrams
\begin{gather*}
\xymatrix @C=.05in
{\sum\limits_{x''}{(\sum\limits_{x'}{A(x,x')\otimes A(x',x'')})\otimes\Psi(x'')}
\ar[rr]^-{\alpha}\ar[d]_-{\sum{M_{x,x''}\otimes1}} &&
\sum\limits_{x'}{A(x,x')\otimes(\sum\limits_{x''}{A(x',x'')\otimes\Psi(x'')})}
\ar[d]^-{\sum{1\otimes\mu_{x'}}}\\ 
\sum\limits_{x''}{A(x,x'')\otimes\Psi(x'')}\ar[dr]_{\mu_x} &&
\sum\limits_{x'}{A(x,x')\otimes\Psi(x')}\ar[dl]^-{\mu_x}\\ 
& \Psi(x), &} \\
\xymatrix
{\sum\limits_{x\in X}{A(x,x)\otimes\Psi(x)}\ar[rr]^-{\mu_x} && \Psi(x) \\
& I\otimes\Psi(x)\ar[ul]^-{\eta_x\otimes1}\ar[ur]_-{\lambda} &}\quad\quad
\end{gather*}
commute. $M$ and $\eta$ are the composition law and identities
for $\ca{A}$, and $\alpha$, $\lambda$ are the associator
and left unitor of the bicategory $\ca{V}$-$\Mat$.
Morphisms between two left $\ca{A}$-modules
$\Psi$ and $\Psi'$ are 2-cells
$\sigma:\Psi\Rightarrow\Psi'$ in $\ca{V}$-$\Mat$
compatible with the actions, \emph{i.e.} 
families of arrows
\begin{displaymath}
 \sigma_x:\Psi(x)\to\Psi'(x)
\end{displaymath}
in $\ca{V}$ for all $x\in X$,
making the diagram
\begin{displaymath}
 \xymatrix @C=.6in @R=.45in
{\sum\limits_{x'}{A(x,x')\otimes\Psi(x')}
\ar[r]^-{\mu^{\Psi}_x}\ar[d]_-{\sum{1\otimes\sigma_{x'}}} &
\Psi(x)\ar[d]^-{\sigma_{x}} \\
\sum\limits_{x'}{A(x,x')\otimes\Psi'(x')}\ar[r]_-{\mu^{\Psi'}_x} &
\Psi'(x)}
\end{displaymath}
commute.
\end{defi}
This is essentially
Definition \ref{leftAmodule},
with a slight variation in the notation
due to the different convention used for
composition of $\ca{V}$-matrices. 
It directly follows from Definition 
\ref{lefttmodules} for $\ca{K}=\ca{V}\text{-}\Mat$,
where the axioms (\ref{axiomslefttmodule}, 
\ref{axiomlefttmodulemorphism})
for the appropriate 2-cells 
\begin{displaymath}
 \xymatrix @R=.02in @C=.5in
 {& X\ar@/^/[dr]|-{\object@{|}}^-A & \\
 1\ar@/^/[ur]|-{\object@{|}}^-{\Psi}
 \ar@/_3ex/[rr]|-{\object@{|}}_-{\Psi}
 \rrtwocell<\omit>{\mu} && X,} \qquad 
 \xymatrix @R=.02in @C=.7in
{\hole \\
1\ar@/^3ex/[r]|-{\object@{|}}^-\Psi
\ar@/_3ex/[r]|-{\object@{|}}_-{\Psi'}
\rtwocell<\omit>{\sigma} & X}
\end{displaymath}
expressing the action and left $t$-modules
morphisms, coincide with the above diagrams
for their components in $\ca{V}$.
Notice also how in
Section \ref{Vbimodulesandmodules},
a left $\ca{A}$-module was denoted 
by$\SelectTips{eu}{10}\xymatrix@C=.2in
{\Psi:\ca{A}\ar[r]|-{\object@{|}} & \ca{I},}$
not to be confused with the actual 
$\ca{V}$-matrix$\SelectTips{eu}{10}\xymatrix@C=.2in
{\Psi:1\ar[r]|-{\object@{|}} & X}$ which encodes
its data, where $\ca{I}$ is the unit category
and 1 is the singleton set.

Similarly, we can define the category of \emph{right
$\ca{B}$-modules} for a $\ca{V}$-category $\ca{B}_Y$, 
\emph{i.e.}
a monad$\SelectTips{eu}{10}\xymatrix@C=.2in
{B:Y\ar[r]|-{\object@{|}} & Y,}$to be the 
category of right $B$-modules with codomain $1$
\begin{displaymath}
 \ca{V}\text{-}\Mod_\ca{B}\equiv\ca{V}\text{-}
\Mat(Y,1)^{\ca{V}\text{-}\Mat(B,1)}
\end{displaymath}
and also the more general category 
of $(\ca{A}_X,\ca{B}_Y)$-bimodules
as the category of algebras for the monad
`pre-composition with $B$ and post-composition
with $A$'
\begin{displaymath}
 \ca{V}\text{-}_\ca{A}\Mod_\ca{B}\equiv
\ca{V}\text{-}\Mat(Y,X)^{\ca{V}\text{-}\Mat(B,A)}
\end{displaymath}
which gives the hom-category
of a bicategory of $\ca{V}$-enriched
bimodules $\ca{V}$-$\B{BMod}$. This way of presenting
of enriched bimodules is also included in \cite{VarThrEnr}.
We note that this bicategorical structure
as well as the one that the enriched bicomodules later
possibly form are not central for the current development.

In a completely dual way, 
we now proceed to the study of the notion
of a $\ca{V}$-enriched comodule for a 
$\ca{V}$-cocategory.
The definitions of the
various cases of comodules for comonads in 
bicategories can again be found in 
Section \ref{monadsinbicats}, and in particular
for $\ca{K}=\ca{V}\text{-}\Mat$, a comonad
is a $\ca{V}$-cocategory$\SelectTips{eu}{10}\xymatrix@C=.2in
{C:X\ar[r]|-{\object@{|}} & X.}$
\begin{defi}\label{leftCcomodule}
The category of \emph{left $\ca{C}$-comodules}
for a $\ca{V}$-cocategory $(C,X)$
is the category of left $C$-comodules
with fixed domain the singleton set in 
the bicategory $\ca{V}$-$\Mat$
$$\ca{V}\text{-}_\ca{C}\Comod=
\ca{V}\text{-}\Mat(1,X)^{\ca{V}\text{-}\Mat(1,C)}.$$
Objects are $\ca{V}$-matrices$\SelectTips{eu}{10}\xymatrix@C=.2in
{\Phi:1\ar[r]|-{\object@{|}} & X}$given by a 
family of objects $\{\Phi(x)\}_{x\in X}$ in $\ca{V}$,
equipped with the coaction $\delta:C\circ\Phi\Rightarrow\Phi$,
a 2-cell in $\ca{V}$-$\Mat$
with components
\begin{displaymath}
 \delta_x:\Phi(x)\to\sum\limits_{x'\in X}{C(x,x')\otimes\Phi(x')}
\end{displaymath}
satisfying the commutativity of the 
following diagrams:
\begin{displaymath}
\xymatrix @C=.05in
{& \Phi(x)\ar[dl]_-{\delta_x}
\ar[dr]^-{\delta_x} &\\
\sum\limits_{x''}{C(x,x'')\otimes\Phi(x'')}
\ar[d]_-{\sum{\Delta_{x,x''}\otimes1}} &&
\sum\limits_{x'}{C(x,x')\otimes\Phi(x')}
\ar[d]^-{\sum{1\otimes\delta_{x'}}} \\
\sum\limits_{x''}(\sum\limits_{x'}{C(x,x')\otimes C(x',x'')})
\otimes\Phi(x'')\ar[rr]^-{\alpha} &&
\sum\limits_{x'}{C(x,x')\otimes(\sum\limits_{x''}{C(x',x'')\otimes
\Phi(x'')}}),}
\end{displaymath}
\begin{displaymath}
\quad\quad\xymatrix
{\Phi(x)\ar[rr]^-{\delta_x}
\ar[dr]_-{\lambda^{-1}} && 
\sum\limits_{x}{C(x,x)\otimes\Phi(x)}
\ar[dl]^-{\epsilon_x\otimes1} \\
& I\otimes\Phi(x). &}
\end{displaymath}
$\Delta$ and $\epsilon$ are the cocomposition 
law and coidentities for $\ca{C}$. 
Morphisms between two left $\ca{C}$-comodules
$\Phi$ and $\Phi'$ are 2-cells
$\tau:\Phi\Rightarrow\Phi'$ in $\ca{V}$-$\Mat$
which are compatible with the coactions, \emph{i.e.}
families of arrows 
\begin{displaymath}
 \tau_x:\Phi(x)\to\Phi'(x)
\end{displaymath}
in $\ca{V}$ for all $x\in X$, which satisfy
the commutativity of
\begin{displaymath}
 \xymatrix @C=.6in
{\Phi(x)\ar[r]^-{\delta^{\Phi}_x}\ar[d]_-{\tau_x} &
\sum\limits_{x'}{C(x,x')\otimes\Phi(x')}
\ar[d]^-{\sum{1\otimes\tau_{x'}}} \\
\Phi'(x)\ar[r]_-{\delta^{\Phi'}_x} &
\sum\limits_{x'}{C(x,x')\otimes\Phi'(x')}.}
\end{displaymath}
\end{defi}
In an analogous way, we can define the 
category of \emph{right $\ca{D}_Y$-comodules}
for a $\ca{V}$-cocategory
to be the category of right $D$-comodules
with codomain 1
\begin{displaymath}
 \ca{V}\text{-}\Comod_\ca{D}=
 \ca{V}\text{-}\Mat(Y,1)^{\ca{V}\text{-}\Mat(D,1)}
\end{displaymath}
and also more generally the category
of \emph{left $\ca{C}_X$/right $\ca{D}_Y$-bicomodules}
as the category of coalgebras for the monad 
`pre-composition with $D$ and post-composition with $C$'
\begin{displaymath}
 \ca{V}\text{-}_\ca{C}\Mod_\ca{D}=
 \ca{V}\text{-}\Mat(Y,X)^{\ca{V}\text{-}\Mat(D,C)}.
\end{displaymath}

By Proposition \ref{propVMat},
the hom-categories $\ca{V}$-$\Mat(X,Y)=\ca{V}^{Y\times X}$ 
of the bicategory $\ca{V}$-$\Mat$
have various useful 
properties, which may be transferred to
the categories defined above. For example, 
$\ca{V}$-$_\ca{A}\Mod$ and $\ca{V}$-$_\ca{C}\Comod$
which are monadic and comonadic by definition,
have all limits/colimits that $\ca{V}$ has,
and those colimits/limits that are preserved by the
monad/comonad. Also, they inherit local presentability,
as explained below.
\begin{prop}\label{propfibresVmodcomod}
Suppose $\ca{V}$ is a cocomplete monoidal category
such that the tensor product preserves colimits in both variables.
\begin{enumerate}
\item The category of left $\ca{A}$-modules 
for a $\ca{V}$-category
$\ca{A}_X$ is cocomplete and locally presentable 
when $\ca{V}$ is.
\item The category of left $\ca{C}$-comodules
for a $\ca{V}$-cocategory $\ca{C}_X$ is cocomplete
and locally presentable when $\ca{V}$ is.
\end{enumerate}
\end{prop}
\begin{proof}

$(1)$ The ordinary monad $\ca{V}$-$\Mat(1,A)$
which post-composes every 
$\ca{V}$-matrix$\SelectTips{eu}{10}\xymatrix@C=.2in
{S:1\ar[r]|-{\object@{|}} &X}$with the 
monad$\SelectTips{eu}{10}\xymatrix@C=.2in
{A:X\ar[r]|-{\object@{|}} &X}$preserves 
colimits, since composition of $\ca{V}$-matrices
commutes with all colimits in general. 

In particular,
$A\circ-$ preserves filtered colimits,
therefore $\ca{V}\text{-}_\ca{A}\Mod$ is finitary monadic
over $\ca{V}$-$\Mat(1,X)$, which is
locally presentable when $\ca{V}$ is.
By Theorem \ref{Monadiccomonadicpresentability},
categories of finitary algebras of 
locally presentable
categories are also locally presentable, hence the 
result follows.

$(2)$ The category $\ca{V}\text{-}_\ca{C}\Comod$ has all
colimits since they are created from those in 
the cocomplete $\ca{V}$-$\Mat(1,X)$.
The endofunctor 
\begin{displaymath}
 F_C:
 \ca{V}\text{-}\Mat(1,X)\xrightarrow{C\circ-} 
 \ca{V}\text{-}\Mat(1,X)
\end{displaymath}
which gives rise to that comonad is again finitary,
so for a locally presentable 
$\ca{V}$, Theorem \ref{Monadiccomonadicpresentability} applies.
\end{proof}
\begin{rmk*}\hfill

$(i)$ We can also express the axioms which
define the objects and the arrows in $\ca{V}$-$_\ca{C}\Comod$
by the diagrams
\begin{displaymath}
\xymatrix
{\Phi\ar@{=>}[r]^-\alpha \ar@{=>}[d]_-\alpha & 
C\circ\Phi\ar@{=>}[d]^-{1\circ\alpha} \\
C\circ\Phi\ar@{=>}[r]_-{\Delta\circ1} &
C\circ C\circ\Phi,}\quad
\xymatrix @C=.2in
{\Phi\ar@{=>}[rr]^-\alpha\ar@{=>}[dr]_-{1_\Phi} &&
C\circ\Phi\ar@{=>}[dl]^-{\epsilon\circ1} \\
& \Phi &}
\end{displaymath}
for a $\ca{V}$-matrix $\Phi$ with domain $1$ 
equipped with $\alpha:\Phi\Rightarrow C\circ\Phi$,
and
\begin{displaymath}
 \xymatrix
 {\Phi\ar@{=>}[r]^-\alpha\ar@{=>}[d]_-k &
 C\circ\Phi\ar@{=>}[d]^-{1\circ k} \\
 \Psi\ar@{=>}[r]_-\beta & C\circ\Psi}
\end{displaymath}
for a 2-cell $k:\Phi\Rightarrow\Psi$. This
could create the impression that $\ca{V}$-$_\ca{C}\Comod$
is an ordinary category of comodules for
a comonoid, here $C\in\Comon(\ca{V}\text{-}\Mat(X,X))$,
in some monoidal category. However, that 
would require everything to take place in the context of 
the fixed monoidal category $(\ca{V}\text{-}\Mat(X,X),\circ,1_X)$,
therefore the comodules category would be
\begin{displaymath}
\Comod_{\ca{V}\text{-}\Mat(X,X)}(C)=
\ca{V}\text{-}\Mat(X,X)^{\ca{V}\text{-}\Mat(X,C)}
\end{displaymath}
by Proposition \ref{modulesmonadic}. 
In our terminology, this is the category of left $C$-comodules with fixed 
domain $X$ in the
bicategory $\ca{V}$-$\Mat$, rather than just
the ones with domain $1=\{*\}$, like $\ca{V}$-$_\ca{C}\Comod$
was defined. The same applies to the categories of modules for
a $\ca{V}$-category $A\in\Mon(\ca{V}\text{-}\Mat(X,X))$.

From this point of view, we could 
formulate all the above
definitions in a more abstract way:
left $\ca{A}_X$-modules could be
$\ca{V}$-matrices$\SelectTips{eu}{10}\xymatrix@=.2in
{\Psi:Y\ar[r]|-{\object@{|}} & X}$with
arbitrary domain set $Y$, given by a family of objects 
$\{\Psi(x,y)\}_{(x,y)\in X\times Y}$ in $\ca{V}$
and a left action from $A$
given by arrows
\begin{displaymath}
\mu_{x.y}:\sum\limits_{x'\in X}A(x,x')\otimes\Psi(x',y)\to\Psi(x,y)
\end{displaymath}
satisfying appropriate axioms. This is also how $\ca{V}$-bimodules
are defined. Nevertheless, for the 
purposes of this thesis we are interested in
$\ca{V}$-modules/comodules given by families indexed
only over the set of objects of the underlying 
$\ca{V}$-category/co\-ca\-te\-go\-ry.

$(ii)$ Notice that establishing local presentability
for particular categories of interest has been of varied difficulty,
depending on their further structure. For example, for the categories
$\Comod_\ca{V}(C)$ (Proposition \ref{comodlocpresent})
and $\ca{V}$-$_\ca{C}\Comod$ the result was straightforward
because they were both evidently finitary comonadic over locally presentable
categories. On the other hand, for $\Comon(\ca{V})$ and $\ca{V}$-$\B{Cocat}$
we first had to verify local presentability
(Propositions \ref{moncomonadm} and \ref{VCocatlocpresent}),
and comonadicity followed afterwards. Notably,
expressing a category as an equifier of a family of natural transformations
of accessible functors between accessible categories has been 
the underlying key technique in all cases.
\end{rmk*}
We now consider global categories
of enriched modules and comodules,
\emph{i.e.} (left) $\ca{V}$-modules and (left)
$\ca{V}$-comodules for which
the $\ca{V}$-category and $\ca{V}$-cocategory
which acts or co-acts is not fixed as above,
but varies. The definitions below are motivated by the 
concepts in Section \ref{globalcats}.
\begin{defi}\label{globalcatVmods}
The \emph{global category of left $\ca{V}$-modules}
$\ca{V}$-$\Mod$ is defined as follows.
Objects are left $\ca{A}$-modules $\Psi$
for an arbitrary $\ca{V}$-category $\ca{A}_X$, denoted
by $\Psi_\ca{A}$, and a
morphism $\kappa_F:\Psi_\ca{A}\to\Xi_\ca{B}$ between
a $\ca{A}_X$-module $\Psi$ and a $\ca{B}_Y$-module $\Xi$ 
consists of a $\ca{V}$-functor 
$F_f:\ca{A}_X\longrightarrow\ca{B}_Y$
and a family of arrows in $\ca{V}$
$\kappa_x:\Psi(x)\longrightarrow\Xi(fx)$
for all objects $x\in X$ of $\ca{A}$,
such that the diagram
\begin{equation}\label{Vmodmaps}
 \xymatrix @C=.6in @R=.6in
 {\sum\limits_{x'}{A(x,x')\otimes\Psi(x')}
 \ar[rr]^-{\mu^{\Psi}_x} \ar[d]_-{\sum{1\otimes \kappa_x}} &&
 \Psi(x)\ar[d]^-{\kappa_x} \\
 \sum\limits_{x'}{A(x,x')\otimes\Xi(fx')} 
 \ar[r]_-{\sum{F_{x,x'}\otimes1}} & 
 \sum\limits_{x'}{B(fx,fx')\otimes\Xi(fx')} 
 \ar[r]_-{\mu^{\Xi}_{fx}} & \Xi(fx)}
\end{equation}
commutes. The arrows $\mu^\Psi$ and $\mu^\Xi$ are the 
left $\ca{A}$ and $\ca{B}$ actions
on $\Psi$ and $\Xi$ respectively.

Dually, the \emph{global category 
of left $\ca{V}$-comodules} $\ca{V}$-$\Comod$
has as objects left $\ca{C}$-comodules
for an arbitrary $\ca{V}$-cocategory $\ca{C}_X$,
denoted by $\Phi_\ca{C}$,
and a morphism $s_G:\Phi_\ca{C}\to\Omega_\ca{D}$
consists of a $\ca{V}$-cofunctor
$G_g:\ca{C}_X\to\ca{D}_Y$
and a family of arrows in $\ca{V}$
$\nu_x:\Phi(x)\to\Omega(gx)$
for all $x\in X$, such that the diagram
\begin{equation}\label{VComodmaps}
 \xymatrix @C=.6in @R=.2in
 {\Phi(x)\ar[r]^-{\delta^\Phi_x} \ar[dd]_-{\nu_x} &
 \sum\limits_{x'}{C(x,x')\otimes\Phi(x')}
 \ar[r]^-{\sum{G_{x,x'}\otimes1}} & 
 \sum\limits_{x'}{D(gx,gx')\otimes\Phi(x')}
 \ar[d]^-{\sum{1\otimes\nu_x}} \\
 && \sum\limits_{x'}{D(gx,gx')\otimes\Omega(gx')}
 \ar@{>->}[d]^-{\iota} \\
 \Omega(gx)\ar[rr]_-{\delta^\Omega_{gx}} &&
 \sum\limits_{y\in Y}{D(gx,y)\otimes\Omega(y)}}
\end{equation}
commutes. The arrows $\delta^\Phi$ and $\delta^\Omega$
are the corresponding coactions, and $\iota$ is the inclusion
into a larger sum.
\end{defi}
Notice the similarities between the diagrams (\ref{Vmodmaps}),
(\ref{VComodmaps})
that morphisms between $\ca{V}$-modules and $\ca{V}$-comodules
over different $\ca{V}$-categories and $\ca{V}$-cocategories
have to satisfy, with the respective diagrams
from Definition \ref{defComod}. This was of course 
expected, since $\ca{V}$-$\Mod$ and $\ca{V}$-$\Comod$
are to be thought of as the many-object generalizations
of the global categories $\Mod$ and $\Comod$.

Both global categories of $\ca{V}$-enriched 
modules and comodules have the structure of a 
(symmetric) monoidal category,
when $\ca{V}$ is symmetric monoidal. For 
a left $\ca{A}_X$-module $\Psi$ and 
a left $\ca{B}_Y$-module $\Xi$, their tensor 
product is a $\ca{V}$-matrix
\begin{equation}\label{tensorofVmods}
 \Psi\otimes\Xi:\SelectTips{eu}{10}\xymatrix
{1\ar[r]|-{\object@{|}} & X\times Y}
\end{equation}
given by the family of objects in $\ca{V}$
\begin{displaymath}
(\Psi\otimes\Xi)(x,y):=\Psi(x)\otimes\Xi(y)
\end{displaymath}
equipped with a left $(\ca{A}\otimes\ca{B})_{X\times Y}$ 
action (since $\ca{V}$-$\B{Cat}$ is monoidal)
a 2-cell $\mu:(\ca{A}\otimes\ca{B})\circ(\Psi\otimes\Xi)
\Rightarrow\Psi\otimes\Xi$, with components arrows in $\ca{V}$
\begin{displaymath}
 \mu_{(x,y)}:\sum\limits_{(x',y')\in X\times Y}
(\ca{A}\otimes\ca{B})((x,y),(x',y'))\otimes(\Psi\otimes\Xi)
(x',y')\to(\Psi\otimes\Xi)(x,y)
\end{displaymath}
which are explicitly the composites
\begin{displaymath}
 \xymatrix @C=.7in
{A(x,x')\otimes B(y,y')\otimes\Psi(x')\otimes\Xi(y') 
\ar[r]^-{1\otimes s\otimes1}
\ar@/_/@{-->}[dr] & 
A(x,x')\otimes\Psi(x')\otimes B(y,y')\otimes\Xi(y')
\ar[d]^-{\mu^\Psi_{x}\otimes\mu^\Xi_{y}} \\
& \Psi(x)\otimes\Xi(y)}
\end{displaymath}
for all $x,x'\in X$ and $y,y'\in Y$. The axioms 
for an $\ca{A}\otimes\ca{B}$-action are satisfied
by the axioms for $\mu^\Psi$ and $\mu^\Xi$.
Dually, if $\Phi$ is a left $\ca{C}_X$-comodule
and $\Omega$ is a left $\ca{D}_Y$-module, 
their tensor product is a $\ca{V}$-matrix 
$\Phi\otimes\Omega$ as
(\ref{tensorofVmods}) given by
$(\Phi\otimes\Omega)(x,y)=\Phi(x)\otimes\Omega(y)$,
with left $(\ca{C}\otimes\ca{D})_{X\times Y}$-action
consisting of the composite arrows
\begin{displaymath}
\xymatrix @C=1.2in
{\Phi(x)\otimes\Omega(y)\ar[r]^-{\delta^\Phi_x\otimes\delta^\Omega_y}
\ar@/_/@{-->}[dr] & 
\sum\limits_{x'\in X}C(x,x')\otimes\Phi(x')\otimes
\sum\limits_{y'\in Y}D(y,y')\otimes\Omega(y')
\ar[d]^-{1\otimes s\otimes1} \\
& \sum\limits_{\scriptscriptstyle{\stackrel{x'\in X}{y'\in Y}}}
C(x,x')\otimes D(y,y')\otimes\Phi(x')\otimes\Omega(y').}
\end{displaymath}
Notice that the right arrow incorporates an isomorphism due to
$\otimes$ preserving sums. It is not hard to check that 
we can extend the definition of a tensor product to
$\ca{V}$-module and comodule morphisms, and also symmetry
from $\ca{V}$ is clearly inherited. The monoidal unit in 
both cases is again the unit $\ca{V}$-matrix$\SelectTips{eu}{10}
\xymatrix@C=.2in{\ca{I}:1\ar[r]|-{\object@{|}} & 1},$ with
trivial $\ca{I}$-action from the unit $\ca{V}$-(co)category.

There are obvious forgetful functors from these
global categories to $\ca{V}$-categories
and $\ca{V}$-cocategories
\begin{gather*}
 N:\ca{V}\text{-}\Mod\to\ca{V}\text{-}\B{Cat} \\
 H:\ca{V}\text{-}\Comod\to\ca{V}\text{-}\B{Cocat}
\end{gather*}
which map any left $\ca{A}$-module $\Psi_A$ and $\ca{C}$-comodule
$\Phi_\ca{C}$ to the $\ca{V}$-category $\ca{A}$ and $\ca{V}$-cocategory
$\ca{C}$ respectively, and the morphisms to the underlying
$\ca{V}$-functor and $\ca{V}$-cofunctor. These functors
will turn out to be a fibration and an opfibration,
allowing us to once again employ Theorem \ref{totaladjointthm}
regarding adjunctions between fibrations, in order to establish an 
enrichment of $\ca{V}$-$\Mod$ in $\ca{V}$-$\Comod$.

Similarly to the $\ca{V}$-categories and $\ca{V}$-cocategories
development,
we will first formulate isomorphic characterizations
of these two categories which will clarify the fibrational
and opfibrational structure later. Lemmas 
\ref{charactVCat} and \ref{charactVCocat}
justify the form of the $\ca{V}$-functors and $\ca{V}$-cofunctors
used below.
\begin{lem}\label{Xi*module}
Suppose that$\SelectTips{eu}{10}\xymatrix@C=.2in
{\Xi:1\ar[r]|-{\object@{|}} &Y}$is a left $\ca{B}$-module
and $F:(A,X)\xrightarrow{(\phi,f)}(B,Y)$ is a $\ca{V}$-functor. 
Then, the composite
$\ca{V}$-matrix
\begin{displaymath}
\SelectTips{eu}{10}\xymatrix @C=.3in
{1\ar[r]|-{\object@{|}}^-{\Xi} &
Y\ar[r]|-{\object@{|}}^-{f^*} &
X}
\end{displaymath}
has the structure of a left $\ca{A}$-module.
Moreover, this mapping gives rise to a functor
\begin{displaymath}
 (f^*\circ-):\ca{V}\text{-}_\ca{B}\Mod\longrightarrow
\ca{V}\text{-}_\ca{A}\Mod.
\end{displaymath}
\end{lem}
\begin{proof}
The induced left $\ca{A}$-action $\mu'$ on 
$f^*\Xi$ is the composite 2-cell
\begin{displaymath}
\xymatrix @C=.6in @R=.5in
{ & Y\ar[r]|-{\object@{|}}^-{f^*} \ar@/_/[dr]|-{\object@{|}}_-B &
X \ar@/^/[dr]|-{\object@{|}}^-A &\\
1 \rrtwocell<\omit>{<-3>\mu}
\ar@/^/[ur]|-{\object@{|}}^-\Xi 
\ar@/_/[rr]|-{\object@{|}}_-\Xi &
\rrtwocell<\omit>{<-5>\hat{\phi}} & 
Y\ar[r]|-{\object@{|}}_-{f^*} & X}
\end{displaymath}
where $\hat{\phi}:f_*A\Rightarrow Bf_*$ 
corresponds bijectively to $\phi:A\Rightarrow f^*Bf_*$
via mates.
In terms of pasting operations, this is the composite
2-cell
\begin{displaymath}
\mu':\xymatrix
{Af^*\Xi\ar@{=>}[r]^-{\hat{\phi}\Xi} & 
f^*B\Xi\ar@{=>}[r]^-{f^*\mu} & 
f^*\Xi.}
\end{displaymath}
The fact that $\mu'$ 
satisfies the axioms for an $A$-action
for a monad$\SelectTips{eu}{10}\xymatrix@C=.2in
{A:X\ar[r]|-{\object@{|}} &X}$follows from the axioms
of the $\ca{V}$-functor $F=(\phi,f)$
and the left $B$-action $\mu$ on $\Xi$.
Also, it is easy to check 
that if $\sigma:\Xi\to\Xi'$ is a left $\ca{B}$-module
morphism, then
\begin{displaymath}
 \xymatrix @C=.5in
{1\ar@/^3ex/[r]|-{\object@{|}}^-\Xi 
\ar@/_3ex/[r]|-{\object@{|}}_-{\Xi'}
\rtwocell<\omit>{\sigma} & 
Y\ar[r]|-{\object@{|}}^-{f^*} & X}
\end{displaymath}
is a left $\ca{A}$-module morphism. 
In terms of components, the family 
$\{\Xi(y)\}_{y\in Y}$
of objects in $\ca{V}$ is mapped to the 
family
\begin{displaymath}
 \{(f^*\circ\Xi)(x)\}_{x\in X}=\{I\otimes\Xi(fx)\}_{x\in X}
\end{displaymath}
and the family $\sigma_y:\Xi(y)\to\Xi'(y)$
of arrows in $\ca{V}$ is mapped to
\begin{displaymath}
 (f^*\sigma)_x:I\otimes\Xi(fx)\xrightarrow{\;1\otimes\sigma_{fx}\;}
I\otimes\Xi'(fx).
\end{displaymath}
Compatibility with composition and identities for 
this functor follow
from properties of vertical and horizontal 
composition of 2-cells.
\end{proof}
Notice that the above lemma, like other results
of this section, does not only hold for 
left modules with fixed domain the singleton set 1, but
for modules with arbitrary domain. Similarly, for
right modules with fixed codomain, if we replace
$(f^*\circ\text{-})$ with $(\text{-}\circ f_*)$ we get
an analogous functor. Dually, we can consider
left $\ca{V}$-comodules.
\begin{lem}\label{Phi*comodule}
 If$\SelectTips{eu}{10}\xymatrix@C=.2in
{\Phi:1\ar[r]|-{\object@{|}} &X}$is a left
$\ca{C}$-comodule and $G:(C,X)\xrightarrow{(\psi,g)}(D,Y)$
is a $\ca{V}$-cofunctor, the composite $\ca{V}$-matrix
\begin{displaymath}
\SelectTips{eu}{10}\xymatrix @C=.3in
{1\ar[r]|-{\object@{|}}^-{\Phi} &
X\ar[r]|-{\object@{|}}^-{g_*} &
Y}
\end{displaymath}
obtains the structure of a left $\ca{D}$-comodule. This 
mapping gives rise to a functor
\begin{displaymath}
 (g_*\circ -):\ca{V}\text{-}_\ca{C}\Comod\longrightarrow
\ca{V}\text{-}_\ca{D}\Comod.
\end{displaymath}
\end{lem}
\begin{proof}
 The induced $\ca{D}$-coaction $\delta'$
on $g_*\Phi$ is the composite 2-cell
\begin{displaymath}
 \xymatrix @C=.6in @R=.5in
{1\ar[rr]|-{\object@{|}}^-{\Phi}
\rrtwocell<\omit>{<4>\delta}
\ar@/_/[dr]|-{\object@{|}}_-{\Phi} &
\rrtwocell<\omit>{<5>\hat{\psi}} &
X\ar[r]|-{\object@{|}}^-{g_*} & Y \\
& X\ar@/^/[ur]|-{\object@{|}}_-C 
\ar[r]|-{\object@{|}}_-{g_*} & 
Y\ar@/_/[ur]|-{\object@{|}}_-D &}
\end{displaymath}
where again $\hat{\psi}$ is the mate of 
$\psi$ `on the one side'. This is the 
pasted composite
\begin{displaymath}
\delta':
\xymatrix
{g_*\Phi\ar@{=>}[r]^-{g_*\delta} & 
g_*C\Phi\ar@{=>}[r]^-{\hat{\psi}\Phi} & 
Dg_*\Phi},
\end{displaymath}
and the $D$-coaction axioms are satisfied
by the axioms for $\delta$ and the $\ca{V}$-cofunctor
$G=(\psi,g)$.
Moreover, if $\tau:\Phi\to\Phi'$ is a left 
$\ca{C}$-comodule morphism, post-composing
it with $g_*$
produces a 2-cell which satisfies the axioms for a left 
$\ca{D}$-comodule. 
In terms of components,
the functor $(g_*\circ-)$
maps the family $\{\Phi(x)\}_{x\in X}$
of objects in $\ca{V}$ to
\begin{displaymath}
\{(g_*\circ\Phi)(y)\}_{y\in Y}=
\{\sum\limits_{y=fx}{I\otimes\Phi(x)\}_{y\in Y}}
\end{displaymath}
and the family $\tau_x:\Phi(x)\to\Phi(x')$
of arrows in $\ca{V}$ to
\begin{displaymath}
 (g_*\tau)_y:\sum\limits_{y=fx}{I\otimes\Phi(x)}
\xrightarrow{\;\sum{1\otimes\tau_x}\;}
\sum\limits_{y=fx}{I\otimes\Phi'(x)}.
\end{displaymath}
This mapping is a functor since it preserves
composition and identities for evident reasons.
\end{proof}
We can now give the following characterizations
of the global
categories of $\ca{V}$-modules and $\ca{V}$-comodules.
\begin{lem}\label{charactVMod}
 The objects of $\ca{V}$-$\Mod$ are pairs
$(\Psi,\ca{A}_X)\in\ca{V}\text{-}_\ca{A}\Mod\times\ca{V}\text{-}\B{Cat}$
and morphisms are (in bijection with) pairs $(\kappa,F_f):(\Psi,\ca{A}_X)\to(\Xi,\ca{B}_Y)$
where
\begin{displaymath}
\begin{cases}
\Psi\xrightarrow{\;\kappa\;}f^*\circ\Xi &\textrm{in }\ca{V}\text{-}_\ca{A}\Mod\\
F:(A,X)\xrightarrow{(\phi,f)}(B,Y) & \textrm{in }\ca{V}\text{-}\B{Cat}.
\end{cases}
\end{displaymath}
\end{lem}
Evidently the objects of this description are 
exactly the same as in Definition 
\ref{globalcatVmods}, whereas the morphisms 
satisfy
\begin{displaymath}
 \xymatrix @C=.5in @R=.2in
{ && X\ar@/^/[dr]|-{\object@{|}}^-A & \\ 
1 \rrtwocell<\omit>{<-3>\kappa}
\ar@/^3ex/[urr]|-{\object@{|}}^-\Psi
\ar@/_/[r]|-{\object@{|}}_-\Xi 
\ar@/_4ex/[drr]|-{\object@{|}}_-\Xi & 
Y \rrtwocell<\omit>{\hat{\phi}}
\ar@/_/[ur]|-{\object@{|}}_-{f^*}
\ar@/_/[dr]|-{\object@{|}}_-B && X \\
\rrtwocell<\omit>{<-2>\mu} && 
Y\ar@/_/[ur]|-{\object@{|}}_-{f^*}}
\quad 
\xymatrix @C=.5in @R=.2in
{\hole \\
=}
\quad
\xymatrix @C=.5in @R=.2in
{& X\ar@/^/[dr]|-{\object@{|}}^-A & \\
1\rrtwocell<\omit>{<-2.5>\mu} 
\ar@/^/[ur]|-{\object@{|}}^-\Psi
\ar@/_/[rr]|-{\object@{|}}^-\Psi
\ar@/_/[dr]|-{\object@{|}}_-\Xi
\rrtwocell<\omit>{<3.5>\kappa} && X. \\
& Y\ar@/_/[ur]|-{\object@{|}}_-{f^*} &}
\end{displaymath}
where the multiplication of $f^*\circ\Xi$ is given 
by Lemma \ref{Xi*module}. If translated in terms of components
$\kappa_x:\Psi(x)\to I\otimes\Xi(fx)$,
the above is equivalent to the commutative diagram
(\ref{Vmodmaps}), again `up to tensoring with $I$ 
in the left'. This implies that there is a 
bijection between
these two forms of the morphisms.
\begin{lem}\label{charactVComod}
 The objects of $\ca{V}$-$\Comod$ are pairs
$(\Phi,\ca{C}_X)\in\ca{V}\text{-}_\ca{C}\Comod\times\ca{V}\text{-}\B{Cocat}$
and morphisms are pairs $(\nu,G_g):(\Phi,\ca{C}_X)\to(\Omega,\ca{D}_Y)$ where
\begin{displaymath}
\begin{cases}
g_*\circ\Phi\xrightarrow{\;\nu\;}\Omega &\textrm{in }\ca{V}\text{-}_\ca{C}\Comod\\
G:(C,X)\xrightarrow{(\psi,g)}(D,Y) & \textrm{in }\ca{V}\text{-}\B{Cocat}.
\end{cases}
\end{displaymath}
\end{lem}
We are now in position to illustrate
the fibrational and opfibrational structure
of the categories of enriched modules
and comodules. Similarly to Section
\ref{fibrationalview}, the idea
is to define appropriate pseudofunctors,
which will then give rise via
the Grothendieck construction to
(op)fibrations isomorphic to the
forgetful functors $N$ and $T$. The 
fibre categories will evidently be the categories
of left modules/comodules for a fixed
$\ca{V}$-category/cocategory.
\begin{prop}\label{VModfibred}
The global category of $\ca{V}$-modules
$\ca{V}$-$\Mod$ is fibred over the 
category of $\ca{V}$-categories $\ca{V}$-$\B{Cat}$.
\end{prop}
\begin{proof}
 Define an indexed category $\ps{H}$ as follows:
\begin{displaymath}
\ps{H}:
\xymatrix @R=.05in @C=.5in
{\ca{V}\text{-}\B{Cat}^\op \ar[r] & \B{Cat} \\
(A,X)\ar @{|.>}[r]
\ar [dd]_-{(\phi,f)} & \ca{V}\textrm{-}_\ca{A}\Mod \\
\hole \\
(B,Y)\ar @{|.>}[r] &
\ca{V}\textrm{-}_\ca{B}\Mod\ar[uu]_-{\ps{H}(\phi,f)}}
\end{displaymath}
where $\ps{H}(\phi,f)=(f^*\circ\text{-})$
as described in Lemma \ref{Xi*module}, \emph{i.e.}
post-composition with the $\ca{V}$-matrix 
$f^*$ induced from the object 
mapping $f$ of the $\ca{V}$-functor.
For any two composable $\ca{V}$-functors 
$F_f:(A,X)\to(B,Y)$
and $G_g:(B,Y)\to(E,Z)$, there is a natural
isomorphism 
\begin{displaymath}
\xymatrix @R=.04in @C=.5in
{& \ca{V}\text{-}_\ca{B}\Mod\ar@/^/[dr]^-{\ps{H}F} & \\
\ca{V}\text{-}_\ca{E}\Mod\ar@/^/[ur]^-{\ps{H}G}
\ar@/_3ex/[rr]_-{\ps{H}(G\circ F)}
\rrtwocell<\omit>{\quad\delta^{F,G}} && 
\ca{V}\text{-}_\ca{A}\Mod}
\end{displaymath}
with components invertible arrows in $_\ca{A}\Mod$
\begin{displaymath}
 \xymatrix @R=.02in
{\hole \\
\delta^{G,F}_{\Psi}:}
\xymatrix @R=.05in @C=.5in
{&& Y\ar@/^/[dr]|-{\object@{|}}^-{f^*} & \\
1\ar[r]|-{\object@{|}}^-\Psi & 
Z\ar@/^/[ur]|-{\object@{|}}^-{g^*} 
\ar@/_4ex/[rr]|-{\object@{|}}_-{(gf)^*} 
\rrtwocell<\omit>{\quad\xi^{g,f}} && X}
\end{displaymath}
where $\xi$ is like in (\ref{zeta}).
These 2-cells consist of families of isomorphisms in 
$\ca{V}$
\begin{displaymath}
(\delta^{G,F}_{\Psi})_x\;:I\otimes I\otimes\Psi(gfx)
\xrightarrow{\;r_I\otimes1\;}I\otimes\Psi(gfx)
\end{displaymath}
which trivially commute with the induced $\ca{A}$-actions
of the modules $f^*g^*\Psi$ and $(gf)^*\Psi$.
Also, for any $\ca{V}$-category $(A,X)$, there is
a natural isomorphism 
\begin{displaymath}
\xymatrix @C=.5in
{\ca{V}\text{-}_\ca{A}\Mod
\rrtwocell<\omit>{\quad\gamma^A}
\ar@/^4ex/[rr]^-{\B{1}_{\ca{V}\text{-}{_\ca{A}\Mod}}}
\ar@/_4ex/[rr]_-{\ps{H}(\B{1}_\ca{A})}
 && \ca{V}\textrm{-}_\ca{A}\Mod}
\end{displaymath}
with components invertible arrows 
\begin{displaymath}
\gamma^A_{\Psi}:
\xymatrix@R=.02in @C=.5in
{1\ar@/^3ex/[rr]|-{\object@{|}}^-\Psi 
\ar@/_/[dr]|-{\object@{|}}_-\Psi 
\rrtwocell<\omit>{\quad\lambda^{-1}} && X \\
& X\ar@/_/[ur]|-{\object@{|}}_-{1_X} &}
\end{displaymath}
where $(\mathrm{id}_X)^*=1_X$
is the underlying function
of the identity functor $\B{1_\ca{A}}$ and
$\lambda$ is the left unitor of 
the bicategory $\ca{V}$-$\Mat$, thus
consist of isomorphisms
\begin{displaymath}
(\gamma^A_{\Psi})_x\;:\Psi(x)\xrightarrow{\;l^{-1}\;}
I\otimes\Psi(x) ,
\end{displaymath}
again trivially being left $\ca{A}$-module
morphisms.
The natural transformations $\delta$ and 
$\gamma$ with components the above isomorphisms
can be verified to satisfy the conditions
\ref{laxcond1} and \ref{laxcond2}, therefore $\ps{H}$
is a well-defined pseudofunctor.

By Theorem \ref{maintheoremfibr}, the Grothendieck
category $\Gr{G}\ps{H}$ has as objects pairs
$(\Psi,\ca{A}_X)$ where $\ca{A}_X$ is in $\ca{V}$-$\B{Cat}$
and $\Psi$ is in $\ca{V}\text{-}_\ca{A}\Mod$, and 
morphisms $(\Psi,\ca{A}_X)\to(\Xi,\ca{B}_Y)$
are pairs
\begin{displaymath}
\begin{cases}
\Psi\to(\ps{H}F)\Xi &\textrm{in }\ps{H}\ca{B}_Y\\
F:(A,X)\to(B,Y) &\textrm{in }\ca{V}\text{-}\B{Cat}
\end{cases}
\end{displaymath}
which, by definition of the functor $\ps{H}F$, 
coincide with the isomorphic formulation of 
left $\ca{V}$-module morphisms as in Lemma
\ref{charactVMod}, hence
$\Gr{G}\ps{H}\cong\ca{V}\text{-}\Mod.$
Moreover, the forgetful functor 
$N:\ca{V}\text{-}\Mod\to\ca{V}\text{-}\B{Cat}$
which keeps the $\ca{V}$-category and $\ca{V}$-functor
part of structure, has essentially the same 
effect as the fibration
\begin{displaymath}
 P_{\ps{H}}:\Gr{G}\ps{H}\longrightarrow\ca{V}\text{-}\B{Cat}
\end{displaymath}
so $N\cong P_{\ps{H}}$ exhibits $N$ as a fibration itself.
\end{proof}
\begin{prop}\label{VComodopfibred}
 The global category of (left) $\ca{V}$-comodules
$\ca{V}$-$\Comod$ is opfibred over 
the category of $\ca{V}$-cocategories $\ca{V}$-$\B{Cocat}$.
\end{prop}
\begin{proof}
 Define a (covariant) indexed category as follows:
\begin{displaymath}
\ps{S}:
\xymatrix @R=.02in @C=.5in
{\ca{V}\text{-}\B{Cocat} \ar[r] & \B{Cat} \\
(C,X)\ar @{|.>}[r]
\ar[dd]_-{(\psi,f)} & 
\ca{V}\textrm{-}_\ca{C}\Comod \ar[dd]^-{\ps{S}(\psi,f)}\\
\hole \\
(D,Y)\ar @{|.>}[r] &
\ca{V}\textrm{-}_\ca{D}\Comod}
\end{displaymath}
where $\ps{S}(\psi,f)=(f_*\circ\text{-})$ as in Lemma 
\ref{Phi*comodule}.
For any two composable $\ca{V}$-cofunctors
$F_f:(C,X)\to(D,Y)$ and $G_g:(D,Y)\to(E,Z)$,
we have a natural isomorphism $\delta^{G,F}:\ps{S}G\circ\ps{S}F\Rightarrow
\ps{S}(G\circ F)$ with components the composite 2-cells
\begin{displaymath}
\xymatrix @R=.02in
{\hole \\
\delta^{G,F}_{\Phi}:}
\xymatrix @R=.02in @C=.5in
{&& Y\ar@/^/[dr]|-{\object@{|}}^-{g_*} & \\
1\ar[r]|-{\object@{|}}^-\Phi & 
X\ar@/^/[ur]|-{\object@{|}}^-{f_*} 
\ar@/_3ex/[rr]|-{\object@{|}}_-{(gf)_*} 
\rrtwocell<\omit>{\quad\zeta^{g,f}} && Z}
\end{displaymath}
in $\ca{V}\text{-}_\ca{E}\Comod$, consisting of the families of arrows
in $\ca{V}$
\begin{displaymath}
 (\delta^{G,F}_\Phi)_z\;:\sum\limits_{\scriptscriptstyle{\stackrel{z=gy}{y=fx}}}
{I\otimes I\otimes\Phi(x)}\xrightarrow{\;\sum{r_I\otimes1}\;}
\sum\limits_{z=gfx}{I\otimes\Phi(x)}
\end{displaymath}
which trivially commute with the respective $\ca{E}$-coactions.
Moreover, for any $\ca{V}$-cocategory $(\ca{C},X)$, we have
a natural isomorphism $\gamma^C:\B{1}_{\ca{V}\text{-}{_\ca{C}\Comod}}
\Rightarrow \ps{S}(\B{1}_\ca{C})$
with components the same invertible arrows $\lambda^{\text{-}1}$
as in the previous proof.
The natural isomorphisms $\delta$ and $\gamma$ 
can be checked to satisfy
the appropriate axioms \ref{laxcond1} and \ref{laxcond2}, so
$\ps{S}$ is a well-defined pseudofunctor. Via Grothendieck construction, 
it gives rise to an opfibration
\begin{displaymath}
 U_{\ps{S}}:\Gr{G}\ps{S}\longrightarrow\ca{V}\text{-}\B{Cocat}
\end{displaymath}
which maps a pair $(\Psi,\ca{C}_X)$ where 
$\Psi\in\ca{V}\text{-}_\ca{C}\Comod$ to its $\ca{V}$-cocategory 
$\ca{C}_X$, and
\begin{displaymath}
 \begin{cases}
(\ps{S}F)\Phi\to\Omega &\textrm{in }\ps{S}\ca{C}_X\\
F:(C,X)\to(D,Y) & \textrm{in }\ca{V}\text{-}\B{Cocat}
\end{cases}
\end{displaymath}
to the $\ca{V}$-functor $F$.
By Lemma \ref{charactVComod} it is now 
evident that $U_\ps{S}\cong H$,
hence the forgetful functor
$H:\ca{V}\text{-}\Comod\to\ca{V}\text{-}\B{Cocat}$
is an opfibration.
\end{proof}
\begin{cor}\label{VComodcomplete}
The opfibration $H$ has all opfibred colimits, hence
$\ca{V}$-$\Comod$ has all colimits and $H$ 
strictly preserves them.
\end{cor}
\begin{proof}
 The fibre categories of the opfibration $H$ are 
the cocomplete categories
$\ca{V}\text{-}_\ca{C}\Comod$ for each $\ca{V}$-cocategory
$\ca{C}_X$, and the reindexing functors $(f_*\circ\text{-})$
for any $\ca{V}$-cofunctor $F_f$ preserve colimits
(as composition of $\ca{V}$-matrices always does). Therefore,
Proposition \ref{reindexcont} ensures that $H$ is opfibred cocomplete,
so by Corollary \ref{AhasPpreserves} and cocompleteness
of $\ca{V}$-$\B{Cocat}$, the result follows.
\end{proof}
\begin{rmk*}
In this section, emphasis was given to the study of 
left-sided $\ca{V}$-modules and $\ca{V}$-comodules, whereas
in Section \ref{globalcats}
where the `one-object case' global categories $\Mod$ and $\Comod$
where defined, the distinction between 
left and right was mostly omitted due to
symmetry in $\ca{V}$. In fact, in a very similar manner 
we could have defined \emph{global categories of right $\ca{V}$-modules} and 
\emph{$\ca{V}$-comodules}. Then by slightly changing the 
reindexing functors (replacing post- with pre-composition, 
and lower with upper stars), we would end up with a 
fibrational characterization as above. 

However, in this case
there does not exist an isomorphism between
right and left enriched modules 
and comodules as before, which would allow us to
regard the different (fibre and total) categories
as essentially the same. Explicitly, for $\Mod_\ca{V}(A)$
with $\ca{V}$ symmetric, a left $A$-module
$(M,\mu)$ for a monoid $A$ always gives rise to a right $A$-action $\mu'$ on $M$ via
\begin{displaymath}
 \xymatrix @C=.8in
{A\otimes M\ar[r]^-{\mu} & M \\
M\otimes A\ar[ur]_-{\mu'}\ar[u]^-{\lcong}_-{s} &}
\end{displaymath}
and all appropriate axioms are satisfied.
On the other hand, the left $\ca{A}$-action for a 
$\ca{V}$-category $\ca{A}_X$ on a $\ca{V}$-module $\Psi$
is given by arrows in $\ca{V}$
\begin{displaymath}
 \ca{A}(x,x')\otimes\Psi(x')\to\Psi(x)
\end{displaymath}
for all $x,x'\in X$, which are 
not in bijective correspondence with arrows which
would define a right $\ca{A}$-action on $\Psi$,
of the form
\begin{displaymath}
 \Psi(x')\otimes\ca{A}(x',x)\to\Psi(x)
\end{displaymath}
for all $x,x'$, even if $\ca{V}$ is symmetric.
This is because the elements of the indexing set of the
family of objects of $\Psi$ in the formula
would agree with the second, rather than the first entry of 
the hom-sets of $\ca{A}$ in the above formula.
\end{rmk*}

\section{Enrichment of $\ca{V}$-modules in $\ca{V}$-comodules}

Similarly to Sections \ref{Universalmeasuringcomodule}
and \ref{enrichmentofVcatsinVcocats}, we are now going 
to work our way through the data which induce
an enrichment of the global category of enriched
modules $\ca{V}$-$\Mod$ in the global category
of enriched comodules $\ca{V}$-$\Comod$.

Suppose that $\ca{V}$ is a symmetric monoidal
closed category, with products and coproducts. 
Recall that the lax functor 
$\Hom:\ca{V}\text{-}\Mat^{\mathrm{co}}\times\ca{V}\text{-}\Mat\to
\ca{V}\text{-}\Mat$ as in (\ref{defimportHom})
provides a functor 
between the hom-categories
\begin{displaymath}
 \Hom_{(Y,W),(X,Z)}:\ca{V}\text{-}\Mat(Z,X)^\op\times
\ca{V}\text{-}\Mat(W,Y)\to\ca{V}\text{-}\Mat(W^Z,Y^X)
\end{displaymath}
which maps a pair of $\ca{V}$-matrices$\SelectTips{eu}{10}\xymatrix @C=.2in
{(S:Z\ar[r]|-{\object@{|}} & X,}\SelectTips{eu}{10}\xymatrix @C=.2in
{T:W\ar[r]|-{\object@{|}} & Y)}$to 
$\Hom(S,T)$ given 
by the family of objects in $\ca{V}$
\begin{displaymath}
 \Hom(S,T)(k,m)=\prod\limits_{\scriptscriptstyle{\stackrel{x\in X}{z\in Z}}}
{[S(x,z),T(mx,kz)]}
\end{displaymath}
for all $k\in W^Z$ and $m\in Y^X$. 
Moreover, in Section \ref{enrichmentofVcatsinVcocats}
we made use of the induced functor
$\Mon(\Hom_{(X,Y),(X,Y)})$ as in (\ref{MonHom_}),
between the categories of comonoids and monoids of the
endoarrow hom-category. This gave rise to the functor
\begin{displaymath}
 K:\ca{V}\text{-}\B{Cocat}^\op\times\ca{V}\text{-}\B{Cat}\to
\ca{V}\text{-}\B{Cat}
\end{displaymath}
between $\ca{V}$-(co)categories,
\emph{i.e.} the $\ca{V}$-matrix $K(\ca{C}_X,\ca{B}_Y)=
\Hom(\ca{C},\ca{B})_{Y^X}$
obtains the structure a $\ca{V}$-category.

Now, by Proposition
\ref{laxfunctorbetweenmodules},
we know that for any
lax functor $\ps{F}$ between bicategories $\ca{K},\ca{L}$
and any monad $t$ in $\ca{K}$, there is an induced
functor $\Mod(\ps{F}_{A,B})$
between the category of left $t$-modules
in $\ca{K}$ and left $\ps{F}t$-modules in $\ca{L}$.
If we apply this in the current setting, the induced
functor is $\Mod(\Hom_{(Y,W),(X,Z)})$\vspace{.05in}
\begin{displaymath}
\scriptstyle{\big(\ca{V}\text{-}
\Mat(Z,X)^{\ca{V}\text{-}\Mat(Z,C)}\big)^\op\times
\ca{V}\text{-}\Mat(W,Y)^{\ca{V}\text{-}\Mat(W,B)}\longrightarrow
\ca{V}\text{-}\Mat(W^Z,Y^X)^{\ca{V}\text{-}\Mat(W^Z,\Hom(C,B))}}\vspace{.05in}
\end{displaymath}
for $(C,X)$ a $\ca{V}$-cocategory and $(B,Y)$ 
a $\ca{V}$-category, for any sets $X,Y,Z,W$. 
This is the case,
because a monad in the domain category
of the lax functor $\Hom$ is a pair $(C,B)$
where $C$ is a monad in 
$\ca{V}\text{-}\Mat^{\mathrm{co}}$,
\emph{i.e.} a comonad in $\ca{V}$-$\Mat$, 
and $B$ is a monad in $\ca{V}$-$\Mat$. Also the domain
of the above induced functor is isomorphic to\vspace{.05in}
\begin{displaymath}
\Big((\ca{V}\text{-}\Mat^{\mathrm{co}}\times\ca{V}\text{-}\Mat)((Z,W),(X,Y))\Big)
^{(\ca{V}\text{-}\Mat^{\mathrm{co}}\times\ca{V}\text{-}\Mat)((Z,W),(C,B))}
\end{displaymath}
since $\ca{V}$-$\Mat^\mathrm{co}(Z,X)=\ca{V}$-$\Mat(Z,X)^\op$
and the category of algebras for the monad (in fact, opposite comonad) 
$\ca{V}$-$\Mat(Z,C)^\op$ on this category is 
precisely the opposite category of coalgebras
\begin{displaymath}
 \big(\ca{V}\text{-}\Mat(Z,X)^{\ca{V}\text{-}\Mat(Z,C)}\big)^\op.
\end{displaymath}
In particular, if we choose $Z\text{=}W\text{=}1$
to be the singleton set,
we obtain the functor\vspace{.05in}
\begin{displaymath}
\scriptstyle{\Mod(\Hom_{(1,1),(X,Y)})}:
\scriptstyle{\big(\ca{V}\text{-}
\Mat(1,X)^{(\text{-}\circ C)}\big)^\op\times
\ca{V}\text{-}\Mat(1,Y)^{(\text{-}\circ B)}}\to
\scriptstyle{\ca{V}\text{-}\Mat(1,Y^X)
^{(\text{-}\circ\Hom(C,B))}}\vspace{.05in}
\end{displaymath}
where the `pre-composition' monads and comonads
are just the endofunctors
$\ca{V}$-$\Mat(1,C)$, $\ca{V}$-$\Mat(1,B)$ and 
$\ca{V}$-$\Mat(1,\Hom(C,B))$ respectively. 
We denote this
functor by
\begin{displaymath}
\bar{K}_{(X,Y)}:
\xymatrix @R=.05in @C=.5in
{\ca{V}\text{-}_\ca{C}\Comod^\op\times
\ca{V}\text{-}_\ca{B}\Mod\ar[r] &\ca{V}\text{-}_{\Hom(\ca{C},\ca{B})}\Mod \\
\qquad(\;\Phi\;,\;\Psi\;)\ar@{|->}[r] & \Hom(\Phi,\Psi)}
\end{displaymath}
using Definitions \ref{leftcaAmodule} and \ref{leftCcomodule}
for the categories involved. This concretely means that
whenever $\Phi$ is a left $\ca{C}_X$-comodule and 
$\Psi$ is a left $\ca{B}_Y$-module, 
the $\ca{V}$-matrix 
\begin{displaymath}
\SelectTips{eu}{10}\xymatrix@C=.5in
{\Hom(\Phi_\ca{C},\Psi_\ca{B}):1\ar[r]|-{\object@{|}} & Y^X} 
\end{displaymath}
obtains
the structure of a left $\Hom(\ca{C},\ca{B})$-module,
where $\SelectTips{eu}{10}\xymatrix@C=.2in
{\Hom(\ca{C},\ca{B}):Y^X\ar[r]|-{\object@{|}} & Y^X}$is a monad
in $\ca{V}$-$\Mat$ as mentioned above. Explicitly,
the left $\Hom(\ca{C},\ca{B})$-action
\begin{displaymath}
 \mu_{s}:\sum_{t\in Y^X}{\Hom(\ca{C},\ca{B})(s,t)\otimes\Hom(\Phi,\Psi)(t)
\to\Hom(\Phi,\Psi)(s)}
\end{displaymath}
for all $s\in Y^X$
is given by a family of arrows in $\ca{V}$
\begin{displaymath}
 \sum\limits_{t\in Y^X}{\prod\limits_{a,a'\in X}{[\ca{C}(a',a),\ca{B}(sa',ta)]}\otimes
\prod\limits_{b\in X}{[\Phi(b),\Psi(tb)]}}\to
\prod\limits_{c\in X}{[\Phi(c),\Psi(sc)]}
\end{displaymath}
which, for fixed $t\in Y^X$ and $c\in X$,
corresponds bijectively under the usual tensor-hom adjunction to
the composite

\begin{displaymath}
\xymatrix @C=1in @R=.3in
{\scriptstyle{\prod_{a,a'}{[\ca{C}(a',a),\ca{B}(sa',ta)]}\otimes
\prod_{b}{[\Phi(b),\Psi(tb)]}\otimes\Phi(c)}
\ar@{-->}[r] \ar[d]_-{\scriptscriptstyle{1\otimes\delta_c}} &
\scriptstyle{\Psi(sc)} \\
\scriptstyle{\prod_{a,a'}{[\ca{C}(a',a),\ca{B}(sa',ta)]}\otimes
\prod_{b}{[\Phi(b),\Psi(tb)]}\otimes
\sum_{c'}{\ca{C}(c,c')\otimes\Phi(c')}}
\ar[d]_-{\scriptscriptstyle{\lcong}} & \\
\scriptstyle{\sum_{c'}\prod_{b}{[\Phi(b),\Psi(tb)]}\otimes
\ca{C}(c,c')\otimes\prod_{b}{[\Phi(b),\Psi(tb)]}\otimes
\Phi(c')}\ar[d]_-{\scriptscriptstyle{\sum\pi_{c,c'}\otimes1\otimes\pi_{c'}\otimes1}} & \\
\scriptstyle{\sum_{c'}[\ca{C}(c,c'),\ca{B}(sc,tc')]\otimes\ca{C}(c,c')\otimes
[\Phi(c'),\Psi(tc')]\otimes\Phi(c')}
\ar[r]_-{\scriptscriptstyle{\sum\mathrm{ev}\otimes\mathrm{ev}}} & 
\scriptstyle{\sum_{c'}\ca{B}(sc,tc')\otimes\Psi(tc')}
\ar[uuu]_-{\scriptscriptstyle{\mu_{sc}}}}
\end{displaymath}
where $\delta$ is the left $\ca{C}$-coaction on $\Phi$ and 
$\mu$ is the left $\ca{A}$-action on $\Phi$.
Notice that for this formula to work, both the $\ca{V}$-module
and the $\ca{V}$-comodule have to be left-sided. 
Also, by Proposition \ref{laxfunctorbetweenmodules} again, 
this induced functor between the categories of modules
is by construction such that the diagram
\begin{equation}\label{ModHomcocontinuous}
 \xymatrix@C=.8in
{\ca{V}\text{-}_\ca{C}\Comod^\op\times
\ca{V}\text{-}_\ca{B}\Mod\ar[r]^-{\bar{K}_{(X,Y)}}\ar[d] & 
\ca{V}\text{-}_{\Hom(\ca{C},\ca{B})}\Mod \ar[d] \\
\ca{V}\text{-}\Mat(1,X)^\op\times\ca{V}\text{-}\Mat(1,Y)
\ar[r]_-{\Hom_{(1,1),(X,Y)}} & \ca{V}\text{-}\Mat(1,Y^X).}
\end{equation}
commutes. The left and right arrows are the respective
monadic forgetful functors from the categories of algebras
to the base categories, for $\ca{C}_X$ a $\ca{V}$-cocategory and 
$\ca{B}_Y$ a $\ca{V}$-category.

As done earlier for the functor $K$ (\ref{defK}),
we can now define a functor between the 
global categories of left $\ca{V}$-modules
and $\ca{V}$-comodules
\begin{equation}\label{defbarK}
 \bar{K}:\ca{V}\text{-}\Comod^\op\times\ca{V}\text{-}\Mod
\longrightarrow\ca{V}\text{-}\Mod
\end{equation}
given by $\bar{K}_{(X,Y)}$ on objects.
For any left $\ca{V}$-module morphism 
$\kappa_F:\Psi_\ca{B}\to\Psi'_\ca{B'}$ and 
left $\ca{V}$-comodule morphism 
$\nu_G:\Phi'_\ca{C'}\to\Phi_\ca{C}$
as in Definition \ref{globalcatVmods}, 
define a morphism
\begin{displaymath}
 \bar{K}(\nu,\kappa):\Hom(\Phi,\Psi)_{\Hom(\ca{C},\ca{B})}\longrightarrow
\Hom(\Phi',\Psi')_{\Hom(\ca{C}',\ca{B}')}
\end{displaymath}
in the global category $\ca{V}$-$\Mod$ as follows: 
it consists of 
the $\ca{V}$-functor 
\begin{displaymath}
 K(G,F)_{f^g}:\Hom(\ca{C},\ca{B})_{Y^X}\to\Hom(\ca{C}',\ca{B}')_{Y'^{X'}}
\end{displaymath}
between the $\ca{V}$-categories which act on the $\ca{V}$-modules,
and the family of arrows
\begin{displaymath}
\scriptstyle{\bar{K}(\nu,\kappa)_s:\Hom(\Phi,\Psi)(s)\longrightarrow
\Hom(\Phi',\Psi')(f^g(s))\equiv\prod\limits_{x}[\Phi(x),\Psi(sx)]\longrightarrow
\prod\limits_{x'}[\Phi'(x'),\Psi'(fsgx')]}
\end{displaymath}
which correspond uniquely, for a fixed $x'\in X$,
to the composite morphism
\begin{displaymath}
 \xymatrix @C=.35in @R=.4in
{\prod_{x\in X}[\Phi(x),\Psi(sx)]\otimes\Phi'(x')
\ar@{-->}[rr] \ar[d]_-{\nu_{x'}}
 && \Psi'(fsgx') \\
\prod_{x\in X}[\Phi(x),\Psi(sx)]\otimes\Phi(gx')
\ar[r]_-{\pi_{gx'}\otimes1} & 
[\Phi(gx'),\Psi(sgx')]\otimes\Phi(gx') \ar[r]_-{\mathrm{ev}} & 
\Psi(sgx')\ar[u]_-{\kappa_{sgx'}}}
\end{displaymath}
under the tensor-hom
adjunction in the monoidal closed $\ca{V}$.
It can be verified via computations that 
these arrows satisfy the commutativity of
(\ref{Vmodmaps}) thus $\bar{K}(\nu,\kappa)$ is a 
well-defined $\ca{V}$-module morphism.

Following the same approach as for 
earlier results, 
we would now like to exhibit this functor $\bar{K}$
as an action, whose adjoint will induce the 
suggested enrichment. Before we
continue in this direction, we introduce
a category whose properties will further clarify
the current setting. In fact, the following structure 
serves very similar
purposes as $\ca{V}$-graphs, which were used
as the `base case' for $\ca{V}$-$\B{Cat}$
and $\ca{V}$-$\B{Cocat}$. If we conceive of
$\ca{V}$-$\B{Grph}$ as
the category of all 
endo-1-cells of the bicategory $\ca{V}$-$\Mat$,
the following is the category of all
1-cells with fixed domain the singleton set $1$.

Consider a category $\ca{C}$ with objects all $\ca{V}$-matrices
of the form$\SelectTips{eu}{10}\xymatrix@C=.2in
{S:1\ar[r]|-{\object@{|}} & X}$for any set $X$, \emph{i.e.}
families of objects $\{S(x)\}_{x\in X}$ in $\ca{V}$, where
a morphism $\nu$ from $S$ with codomain $X$ to $T$ with 
codomain $Y$ 
\begin{displaymath}
 \nu_f:\SelectTips{eu}{10}\xymatrix
{(1\ar[r]|-{\object@{|}}^S & X})\to
\SelectTips{eu}{10}\xymatrix
{(1\ar[r]|-{\object@{|}}^T & Y})
\end{displaymath}
consists of a function $f:X\to Y$ and 
arrows $\nu_x:S(x)\to T(fx)$ in $\ca{V}$
for all $x\in X$. Moreover, this category is
in fact bifibred over $\B{Set}$, with reindexing 
functors those used in Propositions \ref{VModfibred}
and \ref{VComodopfibred}. However, this fact is 
not fundamental at this point
since the (op)fibrations $N$ and $H$ have already been
established, so details are not provided.

Under this section's assumptions on $\ca{V}$,
the category $\ca{C}$ is a symmetric monoidal
category, with the family of objects in $\ca{V}$
\begin{displaymath}
 \{(S\otimes T)(x,y)\}_{(x,y)\in X\times Y}=
\{S(x)\otimes T(x)\}_{\scriptscriptstyle{\stackrel{x\in X}{y\in Y}}}
\end{displaymath}
determining
the tensor product$\SelectTips{eu}{10}\xymatrix@C=.2in
{S\otimes T:1\ar[r]|-{\object@{|}} & X\times Y}$of 
$\ca{V}$-matrices $S$ and $T$ with codomains $X$ and $Y$
accordingly. In a sense, this is 
where the tensor products of $\ca{V}$-$\Mod$ and
$\ca{V}$-$\Comod$ come from. Moreover,
$\ca{C}$ is a monoidal closed category:
for all $\ca{V}$-matrices $S$, $T$ and $R$ 
with codomains $X$, $Y$ and $Z$ respectively, 
there is a bijective correspondence between arrows
\begin{displaymath}
\xymatrix @R=.02in @C=.15in
{\qquad\quad (S\otimes T)_{X\times Y} \ar[rrr] &&& R_Z\phantom{ABCDE} 
&\mathrm{in}\;\ca{C}\\ 
\ar@{-}[rrr] &&& \\  
\qquad\qquad S_X \ar[rrr] &&& \Hom(T,R)_{Z^Y}\phantom{ABCDEFGH} 
& \mathrm{in}\;\ca{C}} 
\end{displaymath}
where $\Hom(T,R)$ is the mapping on objects
of the functor $\Hom_{(1,1),(Y,Z)}$ as in (\ref{Hom_}).
Indeed, any arrow $\kappa:S\otimes T\to M$ in $\ca{C}$,
given by a function $f:X\times Y\to Z$ and arrows
$\kappa_{(x,y)}:S(x)\otimes T(y)\to R(f(x,y))$ in $\ca{V}$
corresponds bijectively, under the tensor-hom
adjunction in $\ca{V}$, to 
\begin{displaymath}
 S(x)\to[T(y),R(f(x,y))]
\end{displaymath}
for all $x\in X$, $y\in Y$. Having in mind
that by cartesian closedness in $\B{Set}$,
$f(x,y)=\bar{f}_x(y)$ for the corresponding
function $\bar{f}:X\to Z^Y$, the above is a family of arrows
\begin{displaymath}
 S(x)\to\prod_{y\in Y}[T(x),R(\bar{f}_xy)]
\end{displaymath}
for all $x\in X$, which together with $\bar{f}$
uniquely determine an arrow 
$S\to\Hom(T,R)$ in $\ca{C}$ as expected.
This is natural in $S$, therefore 
$\Hom_{(1,1),(Y,Z)}(T,-)$ is the object function
of a right adjoint of $-\otimes T$ 
which induces a functor of two variables
\begin{displaymath}
 ^c\Hom(-,-):\ca{C}^\op\times\ca{C}\longrightarrow\ca{C}
\end{displaymath}
namely the internal hom of $\ca{C}$. This is obviously
very similar to the proof of Proposition \ref{VGrphclosed}.
It is also evident that $\ca{V}$ has all small limits, 
and the proof is almost identical with that of completence of 
$\ca{V}$-$\B{Grph}$ in Section \ref{Vgraphs}.

Notice that the global categories $\ca{V}$-$\Mod$ and 
$\ca{V}$-$\Comod$ are (non-full) subcategories of this $\ca{C}$,
like $\ca{V}$-$\B{Cat}$ and $\ca{V}$-$\B{Cocat}$ 
were subcategories of $\ca{V}$-$\B{Grph}$. Their objects 
are objects of $\ca{C}$ with extra structure. In particular,
the functor $\bar{K}$ defined earlier is a restriction
of $^c\Hom(-,-)$ to the appropriate subcategory of $\ca{C}^\op\times\ca{C}$.

We are now going to employ this category $\ca{C}$ in order
to obtain comonadicity of $\ca{V}$-$\Comod$ and monadicity 
of $\ca{V}$-$\Mod$ over appropriate categories, similarly 
to Propositions \ref{Comodcomonadic} and \ref{Modmonadic}
of the previous chapter.
\begin{prop}\label{monadicoverpullbackcat}
The global category of $\ca{V}$-modules is monadic over
the pullback category
$\ca{C}\times_{\scriptscriptstyle{\B{Set}}}\ca{V}\text{-}\B{Cat}$ 
and the global
category of $\ca{V}$-comodules is 
comonadic over the pullback category
$\ca{C}\times_{\scriptscriptstyle{\B{Set}}}\ca{V}\text{-}\B{Cocat}$.
\end{prop}
\begin{proof}
 Consider the functor
\begin{displaymath}
U:\xymatrix @R=.05in @C=.7in
{\ca{V}\text{-}\Mod\ar[r] & 
\ca{C}\times_{\scriptscriptstyle{\B{Set}}}\ca{V}\text{-}\B{Cat} \\
\Psi_\ca{A}\ar@{.>}[r]\ar[dd]_-{\kappa_F} & 
(\Psi_X\;,\;\ca{A}_X)\qquad\ar[dd]^-{(\kappa_f,F_f)} \\
\hole \\
\Xi_\ca{B}\ar@{.>}[r] & (\Xi_Y\;,\;\ca{B}_Y)\qquad}
\end{displaymath}
which `separates' the $\ca{V}$-matrix with domain $1$
from the $\ca{V}$-category which acts on it. This 
is well-defined: the pullback category is formed as in
\begin{displaymath}
\xymatrix @C=.5in @R=.5in
{\ca{C}\times_{\B{Set}}\ca{V}\text{-}\B{Cat}
\pullbackcorner[ul]
\ar[r]\ar[d] & \ca{V}\text{-}\B{Cat} \ar[d] \\
\ca{C}\ar[r] & \B{Set},}
\end{displaymath}
where the right edge is the fibration $P$
which maps any $\ca{V}$-category to its 
set of objects, and the bottom edge
is the bifibration which maps a $\ca{V}$-matrix with 
fixed domain $1$ to its codomain.
We will now construct a left adjoint to $U$, and the 
category of algebras for the induced monad will
turn out to be $\ca{V}$-$\Mod$.
Define
\begin{displaymath}
G:\xymatrix @R=.05in @C=.7in
{\ca{C}\times_{\scriptscriptstyle{\B{Set}}}\ca{V}\text{-}\B{Cat}\ar[r] & 
\ca{V}\text{-}\Mod \\
(S_X,\ca{A}_X)\ar@{.>}[r]\ar[dd]_-{(\nu_f,F_f)} & 
(A\circ S)_\ca{A}\ar[dd] \\
\hole \\
(T_Y,\ca{B}_Y)\ar@{.>}[r] & (B\circ T)_\ca{B}}
\end{displaymath}
where the composite $\ca{V}$-matrix$\SelectTips{eu}{10}
\xymatrix@C=.2in{1\ar[r]|-{\object@{|}}^-S & X
\ar[r]|-{\object@{|}}^-A & X}$obtains
a left $\ca{A}$-action via multiplication of 
the monad $A$, and the image of the morphism between the two
left $\ca{V}$-modules consists of the $\ca{V}$-functor
$F_f:\ca{A}_X\to\ca{B}_Y$ and the family 
$\kappa_x:(A\circ S)(x)\to(B\circ T)(fx)$
of the composite arrows in $\ca{V}$
\begin{displaymath}
\xymatrix @C=.7in
{\sum\limits_{x'\in X} A(x,x')\otimes S(x') 
\ar@{-->}[r]
\ar[d]_-{\sum F_{x,x'}\otimes\nu_{x'}} &
\sum\limits_{y'\in Y} B(fx,y')\otimes T(y'). \\
\sum\limits_{x'\in X} B(fx,fx')\otimes S(fx')
\ar@{^(->}@/_3ex/[ur]_-{\iota} &}
\end{displaymath}
The above is possible only
because $\nu$ and $F$ have the 
same `underlying function' $f$ between the 
`underlying sets' $X$ and $Y$ of the enriched modules
and the enriched categories, since they determine
an arrow in the specific pullback
category. This morphism $\kappa_F$ commutes with 
the $\ca{A}$-action and $\ca{B}$-action on
$(A\circ S)$ and $(B\circ T)$ respectively, since 
$\ca{F}_f$ respects the composition laws of $\ca{A}$ and 
$\ca{B}$ which induce the actions. 
Now, there is a bijective 
correspondence between the hom-sets
\begin{displaymath}
 \ca{V}\text{-}\Mod(G(S,\ca{A}_X),\Xi_\ca{B})\cong
(\ca{C}\times_{\scriptscriptstyle{\B{Set}}}
\ca{V}\text{-}\B{Cat})((S,\ca{A}_X),U(\Xi,\ca{B}_Y))
\end{displaymath}
for any left $\ca{B}_Y$-module $\Xi$, 
$\ca{V}$-matrix$\SelectTips{eu}{10}
\xymatrix@C=.2in{S:1\ar[r]|-{\object@{|}} & X}$and
$\ca{V}$-category $\ca{A}_X$, as follows.

$(i)$ Given a left $\ca{V}$-module morphism
$\kappa_F:(A\circ S)_\ca{A}\to\Xi_\ca{B}$ with $\ca{V}$-functor
$F_f$ and arrows 
$\kappa_x:\sum_{x'} A(x,x')\otimes S(x')\to\Xi(fx)$
in $\ca{V}$, we can form a pair of morphisms
$(\nu_f:S\to\Xi,F_f)$ in the pullback category,
where $\nu$ in $\ca{C}$ is given by the function $f:X\to Y$
and the composite arrows in $\ca{V}$
\begin{displaymath}
 \nu_x:S(x)\cong I\otimes S(x)\xrightarrow{\eta_x\otimes1}
A(x,x)\otimes S(x)\xrightarrow{\iota}\sum_{x'\in X}
A(x,x')\otimes S(x')\xrightarrow{\kappa_x}\Xi(fx)
\end{displaymath}
where $\eta$ is the unit of the monad $(A,X)$.

$(ii)$ Given a pair of morphisms
$(\sigma_f,F_f)$ in the pullback,
where 
\begin{displaymath}
 \sigma:\SelectTips{eu}{10}\xymatrix@C=.2in
{(1\ar[r]|-{\object@{|}}^S & X})\to
\SelectTips{eu}{10}\xymatrix@C=.2in
{(1\ar[r]|-{\object@{|}}^\Xi & Y})
\end{displaymath}
with function $f:X\to Y$ and arrows
$\sigma_x:S(x)\to\Xi(fx)$ in $\ca{V}$ is 
a morphism in $\ca{C}$,
and $F_f:\ca{A}_X\to\ca{B}_Y$ is a $\ca{V}$-functor,
we can form a left $\ca{V}$-module morphism
$(A\circ S)_\ca{A}\to\Xi_\ca{B}$ with the same
$\ca{V}$-functor $F_f$ and family of arrows
\begin{displaymath}
 \sum_{x\in x}A(x,x')\otimes S(x')
\xrightarrow{\sum F_{x,x'}\otimes\sigma_{x'}}
\sum_{x'\in X} B(fx,fx')\otimes\Xi(fx')
\xrightarrow{\mu^{\Xi}_{fx}}\Xi(fx).
\end{displaymath}

These two directions are inverse to each other,
due to properties of the arrows involved,
and also the bijection is natural, thus 
we established an adjunction
\begin{displaymath}
\xymatrix @C=.7in
 {\ca{C}\times_{\scriptscriptstyle{\B{Set}}}
\ca{V}\text{-}\B{Cat}\ar@<+.8ex>[r]^-G
\ar@{}[r]|-{\bot} & 
\ca{V}\text{-}\Mod\ar@<+.8ex>[l]^-U}
\end{displaymath}
which gives rise to a monad $(GU,U\varepsilon_G,\eta)$
on $\ca{C}\times_{\scriptscriptstyle{\B{Set}}}
\ca{V}\text{-}\B{Cat}$. The $GU$-algebras are precisely
left $\ca{V}$-modules, since by definition they
are objects in $\ca{V}\text{-}_\ca{A}\Mod$ for 
each different $\ca{V}$-category $\ca{A}$,
and the diagram that a morphism between
$GU$-algebras has to satisfy coincides
with (\ref{Vmodmaps}). Thus 
\begin{displaymath}
\big(\ca{C}\times_{\scriptscriptstyle{\B{Set}}}
\ca{V}\text{-}\B{Cat}\big)^{GU}\cong\ca{V}\text{-}\Mod.
\end{displaymath}
Dually, we can show that the forgetful functor
\begin{displaymath}
\xymatrix @R=.05in @C=.7in
{\ca{V}\text{-}\Comod\ar[r] & 
\ca{C}\times_{\scriptscriptstyle{\B{Set}}}\ca{V}\text{-}\B{Cocat} \\
\Phi_\ca{C}\ar@{.>}[r]\ar[dd]_-{\nu_G} & 
(\Phi\;,\;\ca{C}_X)\qquad\ar[dd]^-{(\nu,G_g)} \\
\hole \\
\Omega_\ca{D}\ar@{.>}[r] & (\Omega\;,\;\ca{D}_Y)\qquad}
\end{displaymath}
has a right adjoint, such that the induced 
comonad on 
$\ca{C}\times_{\scriptscriptstyle{\B{Set}}}\ca{V}\text{-}\B{Cocat}$
is essentially the same as the global category of $\ca{V}$-comodules,
hence $\ca{V}$-$\Comod$ is comonadic over the pullback category.
\end{proof}
The above proposition leads to a better
understanding of the structure and properties of the 
global categories. 
For example, $\ca{V}$-$\Mod$ inherits completeness from 
the pullback category $\ca{C}\times_{\scriptscriptstyle{\B{Set}}}
\ca{V}\text{-}\B{Cat}$ when $\ca{V}$ is complete,
and the forgetful functor to $\ca{V}$-$\B{Cat}$ strictly preserves
all limits by construction. Hence by Corollary \ref{AhasPpreserves},
the fibration $N$ of Proposition \ref{VModfibred} has all 
fibred limits, and Proposition \ref{reindexcont}
implies that the reindexing functors 
\begin{equation}\label{continuousreindexfunctors}
(F_f)^*=(f^*\circ\text{-}):\ca{V}\text{-}_\ca{B}\Mod\to\ca{V}\text{-}_\ca{A}\Mod
\end{equation}
for a $\ca{V}$-functor $F_f:\ca{A}_X\to\ca{B}_Y$
preserve limits between the complete fibre categories.

As a further application, the functor
$\bar{K}$ (\ref{defbarK}) between the global categories 
turns out to be an action, in essence because
the functors $K$ and $^c\Hom$ are actions. 
\begin{prop}\label{barKaction}
The functor $\bar{K}$ between the global categories
of $\ca{V}$-modules and $\ca{V}$-comodules is an action,
hence its opposite functor
\begin{displaymath}
 \bar{K}^\op:\ca{V}\text{-}\Comod\times\ca{V}\text{-}\Mod^\op
\to\ca{V}\text{-}\Mod^\op
\end{displaymath}
is an action of the symmetric monoidal category 
$\ca{V}\text{-}\Comod$ on the (ordinary) category
$\ca{V}\text{-}\Mod^\op$.
\end{prop}
\begin{proof}
As seen in Section \ref{actions}, we need
natural isomorphisms with components
$\bar{K}(\Phi_\ca{C}\otimes\Omega_\ca{D},\Psi_\ca{A})\xrightarrow{\sim}
\bar{K}(\Phi_\ca{C},\bar{K}(\Omega_\ca{D},\Psi_\ca{A}))$ and 
$\bar{K}(\B{1},\Psi_\ca{A})\xrightarrow{\sim}\Psi_\ca{A}$
for $\ca{V}$-comodules $\Phi_\ca{C}$, $\Omega_\ca{D}$ and 
$\ca{V}$-modules $\Psi_\ca{A}$ in the global
category $\ca{V}$-$\Mod$.
By definition of the functor $\bar{K}$, these 
are in fact of the form
\begin{align*}
 \Hom(\Phi\otimes\Omega,\Psi)
_{{\Hom(\ca{C}\otimes\ca{D},\ca{A})}_{Z^{X\times Y}}}&\cong
\Hom(\Phi,\Hom(\Omega,\Psi))
_{{\Hom(\ca{C},\Hom(\ca{D},\ca{A}))}_{Z^{Y^X}}} \\
\Hom(\B{1},\Psi)_{{\Hom(\ca{I},\ca{A})}_{X^1}}&\cong
\Psi_{\ca{A}_X}
\end{align*}
where $\Hom$ is given
by the product (\ref{Homobjects}) in $\ca{V}$.
Now, the functors 
\begin{align*}
 K:\ca{V}\text{-}\B{Cocat}^\op\times\ca{V}\text{-}\B{Cat}
&\to\ca{V}\text{-}\B{Cat} \\
^c\Hom(-,-):\ca{C}^\op\times\ca{C}&\to\ca{C}
\end{align*}
are actions, the former by Proposition \ref{Kaction} and the latters 
as the internal hom of $\ca{C}$.
Thus we have isomorphisms
\begin{align*}
 \Hom(\ca{C}\otimes\ca{D},\ca{A})&\cong\Hom(\ca{C},\Hom(\ca{D},\ca{A})),\;\;
\Hom(\ca{I},\ca{A})\cong\ca{A}\;\textrm{ in }\;\ca{V}\text{-}\B{Cat} \\
\Hom(\Phi\otimes\Omega,\Psi)&\cong\Hom(\Phi,\Hom(\Omega,\Psi)),\;\;
\Hom(\B{1},\Psi)\cong\Psi\;\textrm{ in }\;\ca{C}
\end{align*}
for the two actions (notice they have the same mapping on objects).
If we place these in pairs, they form natural isomorphisms in the 
pullback category $\ca{C}\times_{\scriptscriptstyle{\B{Set}}}\ca{V}\text{-}\B{Cat}$
for the chosen (co)modules over (co)categories.
Since the forgetful functor from $\ca{V}\text{-}\Mod$ is monadic, it
reflects all isomorphisms so these pairs lift to the required 
invertible arrows in $\ca{V}$-$\Mod$.
Moreover, the diagrams (\ref{actiondiag})
commute because they do for all objects
of $\ca{C}$ and the arrows involved are in $\ca{V}$-$\Mod$.
\end{proof}
We aim to establish
an enrichment of $\ca{V}$-$\Mod$ in $\ca{V}$-$\Comod$
by employing the theory of actions, and in particular
Theorem \ref{actionenrich}.
This process is in line with the ones which
led to the enrichment of $\Mon(\ca{V})$ in $\Comon(\ca{V})$
in Section \ref{Universalmeasuringcomonoid}, of
$\Mod$ in $\Comod$ in Section 
\ref{Universalmeasuringcomodule} and 
of $\ca{V}$-$\B{Cat}$ in $\ca{V}$-$\B{Cocat}$
in Section \ref{enrichmentofVcatsinVcocats}.
Therefore, we need to show the existence of a parametrized adjoint 
of the action bifunctor $\bar{K}$, which will be 
the enriched hom functor of the ($\ca{V}$-$\Comod$)-enriched
category with underlying category $\ca{V}$-$\Mod$.
The theory of fibrations and opfibrations will be again
of central importance, and so we 
begin with some lemmas which are helpful
for the application of the main Theorem 
\ref{totaladjointthm}.
\begin{lem}\label{lemmaopfibred1cell2}
 The diagram
\begin{displaymath}
\xymatrix @C=1in @R=.55in
{\ca{V}\textrm{-}\Comod
\ar[r]^-{\bar{K}(-,\Psi_{\ca{B}})^\op}
\ar[d]_-H & \ca{V}\textrm{-}\Mod^\op
\ar[d]^-{N^\op} \\
\ca{V}\text{-}\B{Cocat}\ar[r]_-{K(-,\ca{B}_Y)^\op} & 
\ca{V}\text{-}\B{Cat}^\op}
\end{displaymath}
exhibits the pair of functors 
$(\bar{K}(-,\Psi_{\ca{B}})^\op,K(-,\ca{B}_Y)^\op)$
as an opfibred 1-cell between the opfibrations
$H$ and $N^\op$.
\end{lem}
\begin{proof}
The fact that this diagram commutes can be easily verified. For
example, we already know that 
\begin{gather*}
K(\ca{C}_X,\ca{B}_Y)=\Hom_{(X,Y),(X,Y)}(\ca{C},\ca{B})_{Y^X}\;\textrm{ and } \\
\bar{K}(\Phi_\ca{C},\Psi_\ca{B})=
\Hom_{(1,1),(X,Y)}(\Phi,\Psi)_{K(\ca{C}_X,\ca{B}_Y)}
\end{gather*}
by definition of the two functors, which clearly implies
that the $\ca{V}$-category action on some
$\bar{K}(\Phi,\Psi)$ is precisely $K(\ca{C},\ca{B})$
for the $\ca{V}$-cocategory and $\ca{V}$-category 
which act on the initial $\ca{V}$-comodule and module.

We now have to show that the functor 
$\bar{K}(-,\Psi_\ca{B})^\op$ is cocartesian, \emph{i.e.}
maps a cocartesian lifting in $\ca{V}$-$\Comod$ to 
a cartesian lifting in $\ca{V}$-$\Mod$, since it is
contravariant. By Proposition \ref{VComodopfibred},
we know that $H$ is isomorphic to the opfibration
which arose via the Grothendieck construction
on the pseudofunctor $\ps{S}$, hence the
canonical cocartesian lifting
$\Cocart(F_f,\Phi_\ca{C}):\Phi_\ca{C}\to(F_!\Phi)_\ca{D}$
is
\begin{equation}\label{thiscocartesianlifting}
\xymatrix @C=.4in
{\Phi\ar[rrr]^-{1_{f_*\Phi}}
\ar @{.>}[d] &&& f_*\Phi 
\ar @{.>}[d] &\textrm{in }\ca{V}\text{-}\Comod \\
\ca{C}_X\ar[rrr]_-{F_f} &&& \ca{D}_Z & 
\textrm{in }\ca{V}\text{-}\B{Cocat}}
\end{equation}
since $\ps{S}(F_f)=(f_*\circ-)$ is the reindexing
functor. Notice that the pair notation
of objects and arrows of the Grothendieck category
is again dropped, because it is clear
from the diagram where each element is 
mapped via the opfibration. 

If we apply the functor 
$\bar{K}(-,\Psi_\ca{B})$, we get the arrow
$\bar{K}((1_{f_*\Phi},F_f),1)$
with domain $\Hom((f_*\Phi)_\ca{D},\Psi_\ca{B})$,
whereas the canonical cartesian lifting 
of $\Hom(\Phi_\ca{C},\Psi_\ca{B})$ along
the $\ca{V}$-functor $K(F,1)$ is 
\begin{displaymath}
\xymatrix @C=.3in @R=.41in
{(Y^f)^*\Hom(\Phi,\Psi)
\ar[rrr]^-{1_{\scriptscriptstyle{(Y^f)^*\Hom(\Phi,\Psi)}}}
\ar @{.>}[d] &&& \Hom(\Phi,\Psi) 
\ar @{.>}[d] & \textrm{in }\ca{V}\text{-}\Mod \\
\Hom(\ca{D},\ca{B})_{Y^Z}\ar[rrr]_-{K(F_f,1)} &&&
\Hom(\ca{C},\ca{B})_{Y^X} & 
\textrm{in }\ca{V}\text{-}\B{Cat}.}
\end{displaymath}
This is the case because by Proposition \ref{VModfibred}, 
the reindexing
functor of the isomorphic fibration coming from the 
pseudofunctor $\ps{H}$ is $\ps{H}(G_g)=(g^*\circ-)$.

For the image of (\ref{thiscocartesianlifting})
under $\bar{K}(-,\Psi)$
to be a cartesian arrow then, we have to show
that the canonical arrow between
the domains of the two arrows in $\ca{V}$-$\Mod$
is a vertical isomorphism.
By definition of the operations involved,
the domain of the canonical
cartesian lifting is a family of objects in $\ca{V}$
\begin{displaymath}
 (Y^f)^*\Hom(\Phi_\ca{C},\Psi_\ca{D})(k)=
I\otimes\prod\limits_{x\in X}[\Phi(x),\Psi(kfx)]
\end{displaymath}
for all $k\in Y^Z$, and the domain of the image of 
the cocartesian arrow in $\ca{V}$-$\Comod$
\begin{gather*}
 \Hom(f_*\Phi,\Psi)(k)=\prod\limits_{z\in Z}[(f_*\Phi)(z),\Psi(kz)]=
\prod_{z}[\sum_{x\in f^{\text{-}1}z}I\otimes\Phi(x),\Psi(kz)] \\
\qquad\qquad\qquad\cong\prod_{\stackrel{z\in Z}{\scriptscriptstyle{z\text{=}fx}}}
[I\otimes\Phi(x),\Psi(kz)]=
\prod_{x}[I\otimes\Phi(x),\Psi(kfx)]
\end{gather*}
since the internal hom
maps colimits to limits on the first variable. 
Thus the isomorphism is
\begin{displaymath}
 \prod_{x}[I\otimes\Phi(x),\Psi(kfx)]\xrightarrow{\;\prod[l,1]\;}
\prod_{x}[\Phi(x),\Psi(kfx)]\xrightarrow{\;l^{\text{-}1}\;}
I\otimes\prod_{x}[\Phi(x),\Psi(kfx)]
\end{displaymath}
for $l$ the left unit constraint of $\ca{V}$,
thus $\bar{K}(-,\Psi)^\op$ is a cocartesian functor.
\end{proof}
\begin{lem}\label{compositehasadjoint}
Suppose $\Psi_\ca{B}$ is a $\ca{V}$-module 
and $\ca{A}_Z$, $\ca{B}_Y$ are $\ca{V}$-categories.
If $\tilde{\varepsilon}$ is the counit of the adjunction
\begin{displaymath}
 \xymatrix @C=.7in
{\ca{V}\textrm{-}\B{Cocat}\ar@<+.8ex>[r]^-{K(-,\ca{B}_Y)^\op}
\ar@{}[r]|-{\bot} & \ca{V}\textrm{-}\B{Cat}^\op\ar@<.8ex>[l]^-{T(-,\ca{B}_Y)}}
\end{displaymath}
which defines the generalized Sweedler hom functor $T$,
the composite functor\vspace{.08in}
\begin{equation}\label{thatcomposite}
 \ca{V}\text{-}\Comod_{\scriptscriptstyle{T(\ca{A},\ca{B})}}
\xrightarrow{\;\bar{K}(-,\Psi)^\op\;}\ca{V}
\text{-}\Mod^\op_{\scriptscriptstyle{K(T(\ca{A},\ca{B}),\ca{B})}}
\xrightarrow{\;(\tilde{\varepsilon}_\ca{A})_!\;}\ca{V}\text{-}\Mod^\op_{\scriptscriptstyle{\ca{A}}}
\end{equation}
has a right adjoint $\bar{T}_0(-,\Psi_\ca{B})$.
\end{lem}
\begin{proof}
By Proposition \ref{Texistence}, the functor $T$ was 
defined as the parametrized adjoint of $K^\op$
and was retrieved by Proposition \ref{propexistenceS2}, 
where it was also shown that the underlying set of objects of 
the $\ca{V}$-cocategory $T(\ca{A}_X,\ca{B}_Y)$ is $Y^X$.
The composite in question consists of functors between fibre
categories, and we can view as 
\begin{displaymath}
 \xymatrix @C=1.5in @R=.5in
{\ca{V}\text{-}_{\scriptscriptstyle{S(\ca{A},\ca{B})}}\Comod
\ar@{-->}[dr] \ar[r]^-{\Mod(\Hom_{(1,1),(Y^Z,Y)})^\op} & 
\ca{V}\text{-}_{\scriptscriptstyle{K(S(\ca{A},\ca{B}),\ca{B})}}\Mod^\op
\ar[d]^-{\ps{H}^\op(\tilde{\varepsilon}_{\varepsilon})} \\
& \ca{V}\text{-}_{\scriptscriptstyle{\ca{A}}}\Mod^\op}
\end{displaymath}
where $\varepsilon:{Y^Y}^Z\to Z$ is the counit of the exponential
adjunction (\ref{exponentialadjunction}).
The functor $\Mod(\Hom)$
is continuous by the commutative diagram 
(\ref{ModHomcocontinuous}) for a fixed variable, and so is
the reindexing 
functor $(\tilde{\varepsilon}_\varepsilon)^*$ as in 
(\ref{continuousreindexfunctors}). 
Therefore
the above composite of the opposite functors is cocontinuous.
Since $\ca{V}\text{-}_{\scriptscriptstyle{T(\ca{A},\ca{B})}}\Comod$
is a locally presentable category by Proposition 
\ref{propfibresVmodcomod}, it is cocomplete
and it has a small dense subcategory. Thus, the cocontinuous
composite (\ref{thatcomposite}) has a right adjoint.
\end{proof}
All conditions of Lemma \ref{totaladjointlem}
are now satisfied, hence the existence of 
a pa\-ra\-me\-tri\-zed adjoint of $\bar{K}$ (more precisely, 
of its opposite functor) can be established 
as follows.
\begin{prop}\label{barKhasadjoint}
The functor $\bar{K}^\op:\ca{V}\text{-}\Comod\times
\ca{V}\text{-}\Mod^\op\to\ca{V}\text{-}\Mod^\op$
has a parametrized adjoint 
\begin{equation}\label{defbarS}
 \bar{T}:\ca{V}\text{-}\Mod^\op\times\ca{V}\text{-}\Mod
 \to\ca{V}\text{-}\Comod
\end{equation}
which makes the following diagram of categories and 
functors serially commute:
\begin{equation}\label{thissquare2}
 \xymatrix @R=.65in @C=1in 
{\ca{V}\text{-}\Comod \ar @<+.8ex>[r]^-{\bar{K}(-,\Psi_\ca{B})^\mathrm{op}}
\ar@{}[r]|-\bot \ar[d]_-H &
\ca{V}\text{-}\Mod^\op\ar @<+.8ex>[l]^-{\bar{T}(-,\Psi_\ca{B})}
\ar[d]^-{N^\op} \\ 
\ca{V}\text{-}\B{Cocat} \ar@<+.8ex>[r]^-{K(-,\ca{B}_Y)^\op} 
\ar@{}[r]|-\bot &
\ca{V}\text{-}\Mod^\op. \ar @<+.8ex>[l]^-{T(-,\ca{B}_Y)}}
\end{equation}
\end{prop}
\begin{proof}
We have an opfibred 1-cell 
$(\bar{K}(-,\Psi_{\ca{B}})^\op,K(-,\ca{B}_Y)^\op)$
between the opfibrations $H$ and $N^\op$
by Lemma \ref{lemmaopfibred1cell2}, and an 
adjunction (\ref{adjunctionKSpartial})
between the base categories of 
the opfibrations.
Also, by Lemma \ref{compositehasadjoint},
we have an adjunction
\begin{displaymath}
\xymatrix @C=1.3in
{\ca{V}\text{-}_{T(\ca{A},\ca{B})}\Comod
\ar @<+.8ex>[r]^-{(\tilde{\varepsilon}_\ca{A})_!\circ
\bar{K}^\op_{\scriptscriptstyle{T(\ca{A},\ca{B})}}(-,\Psi)}
\ar@{}[r]|-\bot
& \ca{V}\text{-}_\ca{A}\Mod^\op
\ar @<+.8ex>[l]^-{\bar{T}_\ca{A}(-,\Psi)}}
\end{displaymath}
for any $\ca{V}$-category $\ca{A}$.
By Theorem \ref{totaladjointthm}, these data suffice
for the existence of a right adjoint
\begin{displaymath}
 \bar{T}(-,\Psi_\ca{B}):\ca{V}\text{-}\Comod\to
 \ca{V}\text{-}\Mod^\op
\end{displaymath}
of the functor $\bar{K}^\op(-,\Psi_\ca{B})$
between the total categories of the 
opfibrations, with $\bar{T}_\ca{A}(-,\Psi)$
its mapping on objects.
By construction of this adjoint,
the opfibrations $H$ and $N^\op$ constitute
a map of adjunctions, thus (\ref{thissquare2}) 
is an adjunction in $\B{Cat}^2$.
Moreover, since we have adjunctions 
$\bar{K}^\op(-,\Psi)\dashv\bar{T}(-,\Psi)$
for all left $\ca{B}_Y$-modules
$\Psi$, there is a unique way to make $\bar{T}$
into a functor of two variables as in (\ref{defbarS}).
This determines a parametrized adjoint of $\bar{K}^\op$
and the proof is complete.
\end{proof}
Notice that by construction of $\bar{T}$, the 
$\ca{V}$-comodule $\bar{T}(\Omega_\ca{A},\Psi_\ca{B})$
is a $\ca{V}$-matrix with codomain the set
$Y^X$, and a left $T(\ca{A}_X,\ca{B}_Y)$-action.
This object evidently generalizes the universal measuring
comodule of Proposition \ref{propmeasuringcomodule}.

Using a similar series of arguments, we can also 
deduce that the global category of enriched
comodules is a monoidal closed category,
under assumptions which allow the category
of $\ca{V}$-cocategories to be monoidal closed.
\begin{prop}\label{VModclosed}
Suppose that $\ca{V}$ is a locally presentable
symmetric monoidal closed category.
The global category of left $\ca{V}$-comodules
$\ca{V}$-$\Comod$ is a monoidal closed category too.
\end{prop}
\begin{proof}
We saw in the previous section how the global
categories of modules and comodules are (symmetric) monoidal
when $\ca{V}$ is. We are now going to use Lemma
\ref{totaladjointlem} once again, in order 
to obtain a right adjoint for the
tensor product endofunctor 
$-\otimes\Phi_\ca{C}$ on $\ca{V}$-$\Comod$.

By Proposition \ref{VCocatclosed}, the 
category $\ca{V}$-$\B{Cocat}$ is also a 
symmetric monoidal closed category when $\ca{V}$ 
is, and its internal hom is denoted by 
$^g\HOM$. Hence there is a square
\begin{equation}\label{squaresquare}
 \xymatrix @C=.6in @R=.45in
{\ca{V}\text{-}\Comod\ar[r]^-{-\otimes\Phi_\ca{C}}
\ar[d]_-H & \ca{V}\text{-}\Comod\ar[d]^-H \\
\ca{V}\text{-}\B{Cocat}\ar[r]_-{-\otimes\ca{C}_X} &
\ca{V}\text{-}\B{Cocat}}
\end{equation}
which commutes by definition of the monoidal
structure
of $\ca{V}$-$\Comod$, and also an adjunction
between the base categories
\begin{displaymath}
 \xymatrix @C=1in 
{\ca{V}\text{-}\B{Cocat}\ar@<+.8ex>[r]^-{(-\otimes\ca{C}_X)}
\ar@{}[r]|-\bot & 
\ca{V}\text{-}\B{Cocat}\ar@<+.8ex>[l]^-{^g\HOM(\ca{C}_X,-)}}
\end{displaymath}
as in (\ref{adjunctionHOMg}).
Moreover, the functor $(-\otimes\Phi_\ca{C})$ is cocartesian: 
it maps a cocartesian lifting to the top arrow
of the triangle
\begin{displaymath}
\xymatrix @C=.15in @R=.3in
{\Omega\ar[rr]^-{\Cocart(F,\Omega)}
\ar@{.}[dd] && F_!\Omega\ar@{.}[dd] &&
\Omega\otimes\Phi\ar[rr]^-{\Cocart(F,\Omega)\otimes1}
\ar[drr]_-{\scriptscriptstyle{\Cocart(F\otimes1,\Omega\otimes\Phi)\;\;}}
\ar@{.}[dd] && F_!\Omega\otimes\Phi 
& \textrm{in }\ca{V}\text{-}\Comod \\
&&& {\mapsto} &&& (F\otimes1)_!(\Omega\otimes\Phi)\ar@{-->}[u]_-{\exists!\gamma}
\ar@{.}[d]\\
\ca{D}_Y\ar[rr]_-{F_f} && \ca{E}_Z && (\ca{D}\otimes\ca{C})_{Y\times X}
\ar[rr]_-{(F\otimes1)_{f\times1}} && (\ca{E}\otimes\ca{C})_{Z\times X} &
\textrm{in }\ca{V}\text{-}\B{Cocat}}
\end{displaymath}
for any left $\ca{D}_Y$-comodule $\Omega$.
By Proposition 
\ref{VComodopfibred},
the reindexing functor $(F_f)_!$ for a $\ca{V}$-cofunctor
with underlying function on objects $f$ is given
by post-composition
with the induced $\ca{V}$-matrix $f^*$, \emph{i.e.}
$(F_f)_!=(f_*\circ\text{-}).$ Now, the two $\ca{V}$-matrices
\begin{displaymath}
 (f_*\circ\Omega)\otimes\Phi,\;(f\times1)_*\circ(\Omega\otimes\Phi):
\SelectTips{eu}{10}\xymatrix
{1\ar[r]|-{\object@{|}} & Z\times X}
\end{displaymath}
in $\ca{V}\text{-}_{\ca{E}\otimes\ca{C}}\Comod$
are isomorphic: they are given by the families of objects in $\ca{V}$
\begin{gather*}
 \big((f^*\circ\Omega)\otimes\Phi\big)(z,x)=
(f^*\circ\Omega)(z)\otimes\Phi(x)=
\Big(\sum_{z=fy}I\otimes\Omega(y)\Big)
\otimes\Phi(x) \\
(f\times1)^*\circ(\Omega\otimes\Phi)(z,x)=
\sum_{z=fy}I\otimes(\Omega\otimes\Phi)(y,x)=
\sum_{z=fy}I\otimes\big(\Omega(y)\otimes\Phi(x)\big) 
\end{gather*}
so the isomorphism between them
is given by the fact
that $\otimes$ commutes with sums. Furthermore
the above triangle commutes,
so the square (\ref{squaresquare}) exhibits
$(-\otimes\Phi_\ca{C},-\otimes\ca{C}_X)$
as an opfibred 1-cell between $H$ and $H$.
Finally, if $\bar{\varepsilon}$ is the counit of the adjunction
(\ref{adjunctionHOMg}) which
defines the internal hom $^g\HOM$ for
$\ca{V}$-cocategories,
the composite functor
between the fibres
\begin{displaymath}
 \ca{V}\text{-}\Comod_{^g\HOM(\ca{C},\ca{D})}
\xrightarrow{\;-\otimes\Phi_\ca{C}\;}
\ca{V}\text{-}\Comod_{^g\HOM(\ca{C},\ca{D})\otimes\ca{C}}
\xrightarrow{\;(\bar{\varepsilon}_\ca{D})_!\;}
\ca{V}\text{-}\Comod_\ca{D}
\end{displaymath}
has a right adjoint, call it $^g\overline{\HOM}_\ca{D}(\Phi_\ca{C},-)$.
This is because the category of left $^g\HOM(\ca{C},\ca{D})$-comodules
is locally presentable by Proposition 
\ref{propfibresVmodcomod}, $(\bar{\varepsilon}_\ca{D})_!$ is cocontinuous
because it is composition of $\ca{V}$-matrices, and
$(-\otimes\Phi_\ca{C})$ is cocontinuous
by the commutative diagram
\begin{displaymath}
\xymatrix @C=.9in
{\ca{V}\text{-}\Comod\ar[r]^-{-\otimes\Phi_\ca{C}}
\ar[d] & \ca{V}\text{-}\Comod\ar[d] \\
\ca{C}\times_{\scriptscriptstyle{\B{Set}}}\ca{V}\text{-}\B{Cocat}
\ar[r]_-{(-\otimes\Phi)\times(-\otimes\ca{C}_X)} &
\ca{C}\times_{\scriptscriptstyle{\B{Set}}}\ca{V}\text{-}\B{Cocat}.}
\end{displaymath}
Therefore we have an adjunction
$(-\otimes\Phi_\ca{C})\dashv {^g\overline{\HOM}(\Phi_\ca{C},-)}$
between the total categories
for all $\ca{V}$-comodules $\Phi_\ca{C}$, exhibiting the 
induced bifunctor
\begin{displaymath}
 ^g\overline{\HOM}:\ca{V}\text{-}\Comod^\op\times\ca{V}\text{-}\Comod
\to\ca{V}\text{-}\Comod
\end{displaymath}
as the internal hom of $\ca{V}$-$\Comod$. Also, 
${^g}\overline{\HOM}(\Phi_\ca{C},\Omega_\ca{D})$ is
a ${^g}\HOM(\ca{C},\ca{D})$-comodule.
\end{proof}
Consequently, we can now apply Corollaries
\ref{importcor1} and \ref{importcor2} for the action
$\bar{K}^\op$ of the symmetric monoidal closed
category $\ca{V}$-$\Comod$ on the ordinary category
$\ca{V}$-$\Mod^\op$ and obtain the pursued enrichment.
\begin{thm}\label{VModenrichedinVComod}
 Suppose that $\ca{V}$ is a locally presentable,
symmetric monoidal closed category. 
\begin{enumerate}
 \item The opposite of the global category
of left $\ca{V}$-modules $\ca{V}$-$\Mod^\op$
is enriched in the global category of 
left $\ca{V}$-comodules $\ca{V}$-$\Comod$,
with hom-objects 
\begin{displaymath}
 \ca{V}\text{-}\Mod^\op(\Psi_\ca{A},\Xi_\ca{B})=
\bar{T}(\Xi,\Psi)_{T(\ca{B},\ca{A})}
\end{displaymath}
where the ($\ca{V}\text{-}\Comod$)-enriched category
is denoted by the same name.
\item The global category of left $\ca{V}$-modules
$\ca{V}$-$\Mod$ is a cotensored ($\ca{V}$-$\Comod$)-enriched
category, with hom-objects
\begin{displaymath}
 \ca{V}\text{-}\Mod(\Psi_\ca{A},\Xi_\ca{B})=
\bar{T}(\Psi,\Xi)_{T(\ca{A},\ca{B})}
\end{displaymath}
and cotensor product $\bar{K}(\Phi,\Xi)_{K(\ca{C},\ca{B})}$
for any $\ca{V}$-modules $\Psi_\ca{A},\;\Xi_\ca{B}$ and 
$\ca{V}$-comodules $\Phi_\ca{C}$.
\end{enumerate}
\end{thm}

\chapter{An Abstract Framework}\label{abstractframework}
This last chapter is 
an attempt to exhibit some underlying motives
of certain techniques used in the previous sections, 
and discuss possible generalizations of 
processes which resulted in the main theorems
of the thesis. The previous 
chapter had as its clear goal to generalize the results
of Chapter \ref{enrichmentofmonsandmods} 
in the next level of
`many-object' (co)monoids and (co)modules, 
namely $\ca{V}$-(co)categories and $\ca{V}$-(co)modules. 
The thorough
investigation of this development reveals 
an intrinsic pattern of how the categories
involved are expected to behave.
 
In the first section, the aim is to 
state and justify a definition of the notion of enriched fibration. 
More precisely, we would like to be able to characterize
a (plain) fibration as being enriched in another,
special kind of fibration, serving similar purposes as the monoidal
base of usual enrichment of categories.
There are two things that would incorporate the success
of such a definition, in the frame of this thesis: firstly, 
the carefully examined cases of monoids/modules,
enriched categories/enriched modules and dual structures 
should constitute examples of it, and secondly 
there should be a theorem which, under certain 
assumptions, would ensure the existence of an 
enriched fibration.

A first formal definition in this conceptual
direction
was given in \cite{GouzouGrunig}, called `a fibration
relative to $\ca{A}$', where $\ca{A}$ was 
fibred over a monoidal category in an appropriate sense. 
As mentioned in the introduction, 
Shulman in \cite{Enrichedindexedcats}
develops a theory of `enriched indexed categories', \emph{i.e.}
categories which are simultaneously indexed over a base category
$\caa{S}$ with finite products, and also enriched in an 
$\caa{S}$-indexed monoidal category. 
The definition of an indexed $\ca{V}$-category was also 
given independently by Bunge in \cite{Bunge}. 
The main issue is that even 
if we herein employ the same notion of a \emph{monoidal 
fibration} (Definition \ref{monoidalfibration}),
Bunge's and Shulman's approach only concerns enrichment in fibrations 
strictly over cartesian monoidal bases, which is not 
the chosen monoidal structure of, say, 
$\Comon(\ca{V})$ and $\ca{V}$-$\B{Cocat}$.
Moreover, the notion of an enriched indexed category refers only
to a fibration enriched in another fibration over the same base, 
approximately depicted as
\begin{displaymath}
\xymatrix @C=.8in @R=.5in
{\ca{A}\ar@{-->}[r]^-{\mathrm{enriched}} 
\ar[dr]_-{\mathrm{fibred}} & 
\ca{V} \ar[d]^-{\mathrm{fibred}} \\
& \caa{S}.} 
\end{displaymath}
In our examples,
this is certainly not the case:
we seek for enrichments between both the total
and the base categories of the two fibrations
involved.

In the second section of this chapter, the aim is to give 
an approximate description of a way in which
the central results of this thesis 
fit into the theory of double categories. The motivation for 
this approach is that in the bicategory $\ca{V}$-$\Mat$,
fundamental for the development of the previous 
chapter, the functions $f$ between the sets
and especially the $\ca{V}$-matrices
$f_*$, $f^*$ induced by them were of importance
for our constructions. This belongs to a variety
of examples of bicategorical structures, 
where in fact two natural kinds of morphisms exist, 
typically some complicated ones (like $\ca{V}$-matrices
between sets in our case) comprising the bicategory, 
and some more elementary ones which are discarded
but in fact important.
Therefore, having everything encompassed in a double category
provides a conceptual advantage. 
Often, there is a lifting property which turns a vertical 
1-cell into a horizontal 1-cell as in 
our situation, and this corresponds to the 
concept of a \emph{fibrant double category}.

Due to the lack of machinery for dealing with double 
categories comparatively to bicategories or 
2-categories, recently there has been some serious activity 
regarding the more systematic study and development 
of the theory of double categories. The exposition in this
chapter is not meant to be a significant step in 
this direction, not being as rigorous
or detailed as such an attempt deserves. Rather it introduces
certain notions which might be of use to 
further research on the topic. Categories of monoids (or monads)
in double categories have been methodically studied in 
\cite{Monadsindoublecats}. In the current treatment, they are 
combined with notions of comonoids, modules
and comodules in double categories in order to exhibit
a framework for the existence and properties 
of specific categories we dealt with in 
earlier chapters.

Various important facts about double categories such as 
detailed definitions for double functors and double natural transformations,
monoidal structure, coherence for pseudo double categories and
numerous examples can be found in the references provided in the 
introduction added to the ones mentioned later.
\nocite{Doubleadjunctionsandfreemonads,Monadsindoublecats}
The explicit definition of a \emph{monoidal bicategory}
can be found in \cite{Carmody}, or in \cite{CoherenceTricats}
as a one-object tricategory.

\section{Enriched fibrations}

Chapters \ref{enrichmentofmonsandmods} and \ref{VCatsVCocats}
were devoted to the establishment of the enrichment
of certain, mostly well-studied categories like $\Mon(\ca{V})$,
$\Mod$, $\ca{V}$-$\B{Cat}$ and $\ca{V}$-$\Mod$, in their dual-flavored
monoidal categories $\Comon(\ca{V})$, $\Comod$, $\ca{V}$-$\B{Cocat}$ 
and $\ca{V}$-$\Comod$. Such enrichments were in fact combined, 
in a very natural way,
with the theory of fibrations and opfibrations. 
The very adjunctions inducing enriched hom-functors 
often employed results regarding fibred functors,
implying a strong relation between the two notions. 
Below we graphically 
summarize the results of the two previous chapters.
The monoidal category $\ca{V}$ is required to be 
a locally presentable,
symmetric monoidal closed category.

The category of monoids is 
enriched in the (symmetric monoidal closed)
category of comonoids in $\ca{V}$, with
enriched hom-functor the Sweedler hom 
$$P:\Mon(\ca{V})^\op\times\Mon(\ca{V})\to\Comon(\ca{V})$$
which is the parametrized adjoint of the opposite 
of the restricted
internal hom
\begin{displaymath}
H:\xymatrix @R=.05in
{\Comon(\ca{V})^\op
\times\Mon(\ca{V})\ar[r]
&\Mon(\ca{V})\\
\qquad\;(\;C\;,\;A\;)\;\ar @{|->}[r] & [C,A]}
\end{displaymath}
by Proposition \ref{measuringcomonoidprop} and 
Theorem \ref{MonVenrichedinComonV}. 
Moreover, 
the global category of modules is enriched in the 
(symmetric monoidal closed) global category of 
comodules in $\ca{V}$, with enriched hom-functor the
universal measuring comodule functor
$$Q:\Mod^\op\times\Mod\to\Comod$$
which is the parametrized adjoint of the opposite
of the further restricted
\begin{displaymath}
\bar{H}:
\xymatrix @R=0.05in
{\Comod^\op\times\Mod\ar[r] & \Mod\qquad \\
\qquad(\;X_C\;,\;M_A\;)\ar @{|->}[r] & 
[X,M]_{[C,A]}}
\end{displaymath}
by Proposition \ref{propmeasuringcomodule} and 
Theorem \ref{ModenrichedinComod}.
The diagram
\begin{equation}\label{hugediag1}
\xymatrix @R=.85in @C=1.2in
{\Mod^\mathrm{op}
\ar @/^/[r]^-{Q(-,N_B)}
\ar@{}[r]|-{\top} \ar[d]_-{G^\op} &
\Comod
\ar @/^/[l]^-{\bar{H}(-,N_B)^\mathrm{op}}
\ar[d]^-V \\ 
\Mon(\ca{V})^\mathrm{op}\ar @/^/[r]^-{P(-,B)} 
\ar@{}[r]|-{\top} &
\Comon(\ca{V}) \ar @/^/[l]^-{H(-,B)^\mathrm{op}}}
\end{equation}
which describes the above situation is in fact
an adjunction in the 2-category $\B{Cat}^{\B{2}}$. 

The category of $\ca{V}$-enriched categories is 
enriched in the (symmetric monoidal closed)
category of $\ca{V}$-enriched cocategories, with enriched hom-functor
the generalized Sweedler hom
\begin{displaymath}
T:\ca{V}\text{-}\B{Cat}^\op\times\ca{V}\text{-}\B{Cat}
\to\ca{V}\text{-}\B{Cocat}
\end{displaymath}
which is the parametrized adjoint of the opposite of
the internal hom as $\ca{V}$-graphs
\begin{displaymath}
 K:\xymatrix@R=.05in
{\ca{V}\text{-}\B{Cocat}^\op\times\ca{V}\text{-}\B{Cat}
\ar[r] & \ca{V}\text{-}\B{Cat} \\
\qquad(\;\ca{C}_X\;,\;\ca{B}_Y\;)\ar@{|->}[r] & 
\Hom(\ca{C},\ca{B})_{Y^X}}
\end{displaymath}
defined by $\Hom(\ca{C},\ca{B})(k,s)=\prod_{x',x}
[\ca{C}(x',x),\ca{B}(kx',sx)]$,
by Proposition \ref{propexistenceS2} and Theorem \ref{VCatenrichedinVCocat}.
Moreover, the global category of $\ca{V}$-enriched modules
is enriched in the (symmetric monoidal closed) global category
of $\ca{V}$-enriched comodules,
with enriched hom-functor
\begin{displaymath}
\bar{T}:\ca{V}\text{-}\Mod^\op\times\ca{V}\text{-}\Mod
\to\ca{V}\text{-}\Comod
\end{displaymath}
which is the parametrized adjoint of the opposite of 
\begin{displaymath}
\bar{K}:\xymatrix @R=.05in
{\ca{V}\text{-}\Comod^\op\times\ca{V}\text{-}\Mod
\ar[r] & \ca{V}\text{-}\Mod\qquad\\
\qquad(\;\Phi_\ca{C}\;,\;\Psi_\ca{B}\;)\ar@{|->}[r] & 
\Hom(\Phi,\Psi)_{\Hom(\ca{C},\ca{B})}}
\end{displaymath}
where $\Hom(\Phi,\Psi)(t)=\prod_{x}
[\Phi(x),\Psi(tx)]$,
by Proposition \ref{barKhasadjoint} and Theorem
\ref{VModenrichedinVComod}.
The diagram
\begin{equation}\label{hugediag2}
\xymatrix @R=.77in @C=1.2in 
{\ca{V}\text{-}\Mod^\op\ar@/^/[r]^-{\bar{T}(-,\Psi_\ca{B})}
\ar[d]_-{N^\op} \ar@{}[r]|-{\top} &
\ca{V}\text{-}\Comod \ar@/^/[l]^-{\bar{K}(-,\Psi_\ca{B})^\mathrm{op}}
\ar[d]^-H \\ 
\ca{V}\textrm{-}\B{Cat}^\mathrm{op}\ar @/^/[r]^-{T(-,\ca{B}_Y)}
\ar[d]_-{P^\op}\ar@{}[r]|-{\top} &
\ca{V}\textrm{-}\B{Cocat}\ar @/^/[l]^-{K(-,\ca{B}_Y)^\mathrm{op}}
\ar[d]^-W \\ 
\B{Set}^\mathrm{op}\ar @/^/[r]^-{Y^{(-)}}\ar@{}[r]|-{\top} &
\B{Set} \ar @/^/[l]^-{{Y^{(-)}}^\mathrm{op}}}
\end{equation}
depicts the above situation.

An appropriate enriched fibration notion would successfully
encapsulate the rich structure of the above situations. 
Intuitively, we are looking for a definition 
which would ensure that the opfibration
$G^\op$ is enriched in the opfibration $V$, 
and that the opfibrations $N^\op$ and $P^\op$ 
are enriched in the opfibrations $H$ and $W$ respectively.

Because of the nature of our examples, 
it is now evident that we are unable to employ the definitions
and theory of \cite{Enrichedindexedcats}. As mentioned earlier, 
the numerous examples
therein restrict to fibrations
(or indexed categories) over monoidal
categories with tensor product the cartesian product. However,
in the diagrams (\ref{hugediag1}) 
and (\ref{hugediag2}) the base categories (except $\B{Set}$)
of the fibrations which we intend to use as base for enrichment
are not viewed as cartesian monoidal categories.
Moreover, and perhaps more importantly, 
the indexed enrichment (over the same base category)
as stated in \cite[Definition 4.1]{Enrichedindexedcats}
is conceived as `fibrewise' enrichments
between the fibres of the total categories, plus 
some preservation of the enriched structure
via the reindexing functors. Apart from the absence of a monoidal
structure on the fibre categories here, like $\Comod_\ca{V}(C)$,
the fact that we require an enrichment
between the (distinct) base categories 
of the fibrations makes a great difference.

Therefore, we are going to explore a new approach to this
problem. The basic idea is to shift 
Theorem \ref{actionenrich}
from the context of categories to the context of fibrations. 
The reason for doing so is that this result provides an 
enrichment of an ordinary category
in a monoidal category when certain conditions are satisfied, which can
be rephrased if we replace categories by fibrations. 
This becomes clearer in the light of the following remarks 
(see also Remark \ref{rmkpseudoaction}(ii)).
\begin{itemize}
\item A monoidal category $(\ca{V},\otimes,I,a,l,r)$
is a \emph{pseudomonoid} in the cartesian monoidal
2-category $(\B{Cat},\times,\B{1})$.
\item An action $*$ of a monoidal category $\ca{V}$ 
on an ordinary category $\ca{A}$ is
a \emph{pseudoaction} of a pseudomonoid on 
an object of $(\B{Cat},\times,\B{1})$.
\item A $\ca{V}$-representation $(\ca{A},*)$, \emph{i.e.}
an ordinary category on which $\ca{V}$ acts, is a
\emph{pseudomodule} for the pseudomonoid $\ca{V}$ 
in $(\B{Cat},\times,\B{1})$.
\end{itemize}
Theorem \ref{actionenrich} and its following comments
in fact give the one direction of the 
equivalence
\begin{displaymath}
 \ca{V}\text{-}\mathrm{Rep}_{\textrm{cl}}\simeq
\ca{V}\text{-}\B{Cat}_{\otimes}
\end{displaymath}
on the level of objects 
between \emph{closed} $\ca{V}$-representations 
(\emph{i.e.} equipped with a pa\-ra\-me\-tri\-zed
adjoint) and tensored $\ca{V}$-categories
for $\ca{V}$ a monoidal closed
category. This equivalence is in fact a special case of 
the more general \cite[Theorem 3.7]{enrthrvar}.
We would now like to produce an adjusted version of this,
moving from $(\B{Cat},\times,\B{1})$ to the monoidal 
2-category $(\B{Fib},\times,1_{\B{1}})$, where
$1_{\B{1}}$ is the identity functor on the 
terminal category. Indeed, the 2-functor
\begin{displaymath}
 \times:\B{Cat}\times\B{Cat}\to\B{Cat}
\end{displaymath}
which is the cartesian 2-monoidal structure 
on $\B{Cat}$, induces 
a monoidal structure on the 2-category $\B{Cat}^\B{2}$
which restricts to the sub-2-category $\B{Fib}$,
since the cartesian product of two fibrations is still 
a fibration.

\nocite{Mccruddencoalgebroids}
Initially, we would like to identify the pseudomonoids
in this monoidal 2-category, which will be the analogue
of monoidal categories.
The concept of a pseudomonoid was formally defined
in \cite{Monoidalbicats&hopfalgebroids}, and 
the more general pseudomonad viewpoint can be found
in \cite{Marmolejopseudomonads,LackPseudomonads}.
As an idea, a \emph{tensor object}
in \cite{BraidedTensorCats} already captures the required
structure. By applying this definition 
in the 2-category of fibrations, 
fibred 1-cells and fibred 2-cells,
a \emph{monoidal fibration} is a fibration 
$T:\ca{V}\to\caa{W}$ with arrows $M:T\times T\to T$,
$\eta:1\to T$ equipped with natural isomorphisms
\begin{displaymath}
 \xymatrix @R=.5in
{T\times T\times T\ar[r]^-{M\times 1} \ar[d]_-{1\times M}
\drtwocell<\omit>{'{\stackrel{a}{\cong}}} & 
T\times T \ar[d]^-M \\
T\times T\ar[r]_-M & T}\qquad
\xymatrix @R=.5in
{1\times T \drtwocell<\omit>{'{\stackrel{l}{\cong}}}
\ar@/_3ex/[dr]\ar[r]^-{\eta\times1}
& T\times T\drtwocell<\omit>{'{\stackrel{r}{\cong}}}
\ar[d]_-M & T\times1\ar[l]_-{1\times\eta}
\ar@/^3ex/[dl]\\ 
& T &}
\end{displaymath}
satisfying certain coherence conditions.
More explicitly, there are fibred 1-cells
$M=(M_\ca{V},M_\caa{W})$, $\eta=(I_\ca{V},I_\caa{W})$ displayed 
by the commutative squares
\begin{equation}\label{multunitmonoidalfibr}
 \xymatrix
{\ca{V}\times\ca{V}\ar[r]^-{M_\ca{V}}
\ar[d]_-{T\times T} & \ca{V}\ar[d]^-T \\
\caa{W}\times\caa{W}\ar[r]_-{M_{\caa{W}}} & 
\caa{W}}\qquad\mathrm{and}\qquad
\xymatrix 
{\B{1}\ar[r]^-{I_\ca{V}} \ar[d]_-{1} & 
\ca{V}\ar[d]^-{T} \\
\B{1}\ar[r]_-{I_{\caa{W}}} & \caa{W}}
\end{equation}
where the functors $M_\ca{V}$ and $I_\ca{V}$ are 
cartesian, and 
invertible fibred 2-cells $a=(a^\ca{V},a^\caa{W})$,
$r=(r^\ca{V},r^\caa{W})$, $l=(l^\ca{V},l^\caa{W})$
displayed as
\begin{displaymath}
 \xymatrix @C=1.2in @R=.6in
{\ca{V}\times\ca{V}\times\ca{V}\rtwocell^{M(M\times1)}
_{M(1\times M)}{\;a^\ca{V}}
\ar[d]_-{T\times T\times T} & \ca{V}\ar[d]^-T \\
\caa{W}\times\caa{W}\times\caa{W}
\rtwocell^{M(M\times1)}
_{M(1\times M)}{\;a^\caa{W}} & \caa{W}}
\end{displaymath}
\begin{displaymath}
  \xymatrix @C=.9in @R=.4in
{\ca{V}\times1\rtwocell^{M(1\times I)}
_{\sim}{\;\;r^\ca{V}}
\ar[d]_-{T\times1} & \ca{V}\ar[d]^-T \\
\caa{W}\times1\rtwocell^{M(1\times I)}
_{\sim}{\;\;r^\caa{W}} & \caa{W}}\qquad
\xymatrix @C=.8in @R=.4in
{1\times\ca{V}\rtwocell^{M(I\times1)}
_{\sim}{\;l^\ca{V}}
\ar[d]_-{1\times T} & \ca{V}\ar[d]^-T \\
1\times\caa{W}\rtwocell^{M(I\times1)}
_{\sim}{\;l^\caa{W}} & \caa{W}.}
\end{displaymath}
Recall that the natural isomorphisms 
 $a^\ca{V},r^\ca{V},l^\ca{V}$ lie 
above $a^\caa{W},r^\caa{W},l^\caa{W}$,
by definitions in Section \ref{fibrationsbasicdefinitions}.
The axioms that these data are required to satisfy 
turn out to give the usual axioms which make 
$(\ca{V},M_\ca{V},I_\ca{V})$ and $(\caa{W},M_\caa{W},I_\caa{W})$ 
into monoidal categories, 
with associativity, left and right unit constraints 
$a,r,l$ respectively. This is due to the fact that
the functors $dom,\;cod:\B{Fib}\to\B{Cat}$ 
are strict monoidal functors. In other 
words, the
equality of pasted diagrams
of 2-cells in $\B{Fib}$ breaks down into
equalities for the two natural transformations
it consists of.

Moreover, the strict 
commutativity of the diagrams (\ref{multunitmonoidalfibr})
imply that $T$
preserves the tensor product and the 
unit object
between $\ca{V}$ and $\caa{W}$ on the nose, \emph{i.e.}
\begin{displaymath}
 TA\otimes_\caa{W} TB=T(A\otimes_\ca{V} B),\quad
I_\caa{W}=T(I_\ca{V})
\end{displaymath}
if we denote $M=\otimes$.
Along with the last conditions
that $T(a^\ca{V})=a^\caa{W}$, 
$T(l^\ca{V})=l^\caa{W}$ and $T(r^\ca{V})=r^\caa{W}$, 
these data define a strict monoidal structure on the 
functor $T$.
Therefore we obtain the following definition, which coincides
with \cite[12.1]{Framedbicats}.
\begin{defi}\label{monoidalfibration}
A \emph{monoidal fibration}
is a fibration $T:\ca{V}\to\caa{W}$ such that
\begin{enumerate}[(i)]
\item $\ca{V}$ and $\caa{W}$ are monoidal categories,
\item $T$ is a strict monoidal functor,
\item the tensor product $\otimes_\ca{V}$ of $\ca{V}$
preserves cartesian arrows.
\end{enumerate}
\end{defi}
In a dual way, we can define a 
\emph{monoidal opfibration} to be an opfibration
which is a strict monoidal functor, where the tensor
product of the total category preserves cocartesian arrows. 
Also, if $\ca{V}$ and $\ca{W}$ are 
symmetric monoidal categories and $T$ is a symmetric strict monoidal
functor, call $T$ a \emph{symmetric monoidal fibration}.

We are now going to describe a pseudoaction of 
a pseudomonoid in $\B{Fib}$,
and what it means for a fibration to be a 
pseudomodule for a monoidal fibration $T$. 
For a general 2-category or 
bicategory, the idea of a \emph{pseudomodule}
can be found in similar contexts in 
\cite{Marmolejopseudomonads,LackPseudomonads}
(called (pseudo)algebra for a 
pseudomonad).
Conceptually, as was the case for modules
for monoids in a monoidal category, it arises as 
a pseudoalgebra for the pseudomonad $(M\otimes-)$
in our monoidal bicategory,
where $M$ is a fixed pseudomonoid.

In our case, a \emph{pseudoaction} of a monoidal
fibration $T:\ca{V}\to\caa{W}$ on an ordinary fibration 
$P:\ca{A}\to\caa{X}$ is a fibred 1-cell
$\mu=(\mu^\ca{A},\mu^\caa{X}):T\times P\to P$ displayed by
the commutative
\begin{equation}\label{pseudoaction}
\xymatrix @C=.6in
{\ca{V}\times\ca{A}\ar[r]^-{\mu^\ca{A}}
\ar[d]_-{T\times P} & \ca{A}\ar[d]^-P \\
\caa{W}\times\caa{X}\ar[r]_-{\mu^{\caa{X}}} & 
\caa{X}}
\end{equation}
where $\mu^\ca{A}$ is a cartesian functor, equipped
with natural isomorphisms 
\begin{displaymath}
 \xymatrix @C=.5in
{T\times T\times P\ar[r]^{M\times 1} \ar[d]_-{1\times\mu}
\drtwocell<\omit>{'{\stackrel{\chi}{\cong}}} & 
T\times P \ar[d]^-\mu \\
T\times P\ar[r]_-\mu & T}\qquad
\xymatrix @C=.5in
{1\times P \drtwocell<\omit>{'{\stackrel{\nu}{\cong}}}
\ar@/_3ex/[dr]_-{\sim}\ar[r]^-{\eta\times1}
& T\times P\ar[d]^-\mu \\ 
& P}
\end{displaymath}
in $\B{Fib}$. These are invertible
fibred natural transformations 
$\chi=(\chi^\ca{A},\chi^\caa{X})$,
$\nu=(\nu^\ca{A},\nu^\caa{X})$ represented by
\begin{displaymath}
\xymatrix @C=.4in @R=.003in
{& \ca{V}\times\ca{A}\ar@/^/[dr]^{\mu} & \\
\ca{V}\times\ca{V}\times\ca{A}\ar@/^/[ur]^-{M\times1}
\ar@/_/[dr]_-{1\times\mu}
\rrtwocell<\omit>{\;\;\chi^\ca{A}}
\ar[ddddd]_-{T\times T\times P} && \ca{A}\ar[ddddd]^-P \\
& \ca{V}\times\ca{A} \ar@/_/[ur]_-{\mu} & \\
\hole \\
\hole \\
& \caa{W}\times\caa{X} \ar@/^/[dr]^{\mu} & \\
\caa{W}\times\caa{W}\times\caa{X}
\ar@/^/[ur]^-{M\times1} \ar@/_/[dr]_-{1\times\mu}
\rrtwocell<\omit>{\;\;\chi^\caa{X}} && \caa{X}\\
& \caa{W}\times\caa{X} \ar@/_/[ur]_-{\mu} &}
\qquad
\xymatrix @C=.4in @R=.003in
{& \ca{V}\times\ca{A}\ar@/^/[dr]^{\mu} & \\
1\times\ca{A}\ar@/^/[ur]^-{I\times1}
\ar@/_3ex/[rr]_-{\sim}
\rrtwocell<\omit>{\;\;\nu^\ca{A}}
\ar[ddddd]_-{1\times P} && \ca{A}\ar[ddddd]^-P \\
\hole \\
\hole \\
\hole \\
& \caa{W}\times\caa{X} \ar@/^/[dr]^{\mu} & \\
1\times\caa{X}\ar@/^/[ur]^-{I\times1} 
\ar@/_3ex/[rr]_-{\sim}
\rrtwocell<\omit>{\;\;\nu^\caa{X}} && \caa{X}}
\end{displaymath}
where $\chi^\ca{A},\nu^\ca{A}$ are above 
$\chi^\caa{X},\nu^\caa{X}$ with respect to the 
appropriate fibrations. These data are subject to 
certain axioms, which in fact again
split up in two sets of commutative diagrams, 
for the components of the two natural isomorphisms that 
the fibred 2-cells $\chi$ and $\nu$ consist of. The resulting
diagrams coincide with the ones for an action 
of a monoidal category (\ref{actiondiag}).
\begin{defi}\label{Trepresentation}
 The fibration $P:\ca{A}\to\caa{X}$ is a \emph{$T$-representation}
(or a \emph{$T$-module})
for a monoidal fibration $T:\ca{V}\to\caa{W}$,
when it is equipped with a $T$-pseudoaction $\mu=(\mu^\ca{A},\mu^\caa{X})$. 
This amounts to two actions 
\begin{gather*}
 \mu^\ca{A}=*:\ca{V}\times\ca{A}\longrightarrow\ca{A} \\
\mu^\caa{X}=\diamond:\caa{W}\times\caa{X}\longrightarrow\caa{X}
\end{gather*}
of the monoidal categories $\ca{V}$, $\caa{W}$ on the 
categories $\ca{A}$ and $\caa{X}$ respectively,
such that $\mu^\ca{A}$ preserves cartesian arrows and 
$P\chi^\ca{A}_{XYA}=\chi^\caa{X}_{(TX)(TY)(PA)}$,
$P\nu^\ca{A}_A=\nu^\caa{X}_{PA}$
for all $X,Y\in\ca{V}$ and $A\in\ca{A}$.
\end{defi}
The last two relations are easy to verify 
in specific examples. In greater detail, 
the commutative diagram (\ref{pseudoaction}) representing
the pseudoaction implies that 
\begin{displaymath}
 P(X*A)=TX\diamond PA
\end{displaymath}
for any $X\in\ca{V}$, $A\in\ca{A}$, hence 
the isomorphisms
$\chi^\ca{A}_{XYA}:X*(Y*A)\cong(X\otimes_\ca{V} Y)*A$
lie above certain isomorphisms in $\caa{X}$
\begin{displaymath}
 P\chi^\ca{A}_{XYA}:TX\diamond(TY\diamond PA)\xrightarrow{\;\sim\;}
(TX\otimes_\caa{W} TY)\diamond PA
\end{displaymath}
in $\caa{W}$, since $T$ is strict monoidal.
Similarly, $\nu^\ca{A}_A:I*A\cong A$ is mapped,
under $P$, to 
\begin{displaymath}
 P\nu^\ca{A}_A:I_\caa{X}\diamond PA\xrightarrow{\;\sim\;}
PA
\end{displaymath}
since $P(I_\ca{V}*A)=T(I_\ca{V})\diamond PA=I_\caa{W}\diamond PA$
by strict monoidality of $T$ again. These isomorphisms
then are required to coincide with the components of 
structure isomorphisms
$\chi^{\caa{X}}$ and $\nu^{\caa{X}}$ of the $\caa{W}$-representation
$\caa{X}$.

The last step in order to get a clear picture
of how a modified correspondence between representations
of a monoidal fibration and enriched fibrations would work, 
is to introduce a notion of a parametrized adjunction 
in $\B{Fib}$.
For that, we first re-formulate the `adjunctions with 
a parameter' Theorem \ref{parametrizedadjunctions} 
in the context of $\B{Cat}^\B{2}$.
Even though the abstract definition of an adjunction
applies to any 2-categorical, or bicategorical, setting
as in Definition \ref{adjunction2cat}, for its appropriate
parametrized version
we need the 0-cells of our 2-category to be category-like
themselves. Intuitively,
in such cases, if we have a 1-cell with domain 
a product of two objects $t:A\times B\to C$, we 
are able to consider a 1-cell
$t_a:B\to C$ by fixing 
an `element' of one of the 0-cells,
$a$ in $A$.
\begin{thm}[Adjunctions with a parameter in $\B{Cat}^\B{2}$]
\label{parameteradjunctionCat2}
 Suppose we have a morphism $(F,G)$ of two variables
in $[\B{2},\B{Cat}]$, given by a commutative square
of categories and functors
\begin{equation}\label{FG}
 \xymatrix @C=.6in @R=.4in
{\ca{A}\times\ca{B}\ar[r]^-F
\ar[d]_-{H\times J} & \ca{C}\ar[d]^-K \\
\caa{X}\times\caa{Y}\ar[r]_-G & \caa{Z}.}
\end{equation}
Assume that, for every $B\in\ca{B}$ and $Y\in\caa{Y}$, 
there exist adjunctions $F(-,B)\dashv R(B,-)$ and
$G(-,Y)\dashv S(Y,-),$ such that 
the `partial' morphism $(F(-,B),G(-,JB))$
has a 
right adjoint $(R(B,-),S(JB,-))$ in $\B{Cat}^\B{2}$.
This is represented by
the diagram
\begin{equation}\label{adjunction[2,Cat]}
 \xymatrix @C=.9in @R=.5in
{\ca{A}\ar[d]_-H\ar@<+.8ex>[r]^-{F(-,B)}
\ar@{}[r]|-{\bot} &
\ca{C}\ar[d]^-K\ar@<+.8ex>[l]^-{R(B,-)} \\
\caa{X}\ar@<+.8ex>[r]^-{G(-,JB)}\ar@{}[r]|-{\bot} & 
\caa{Z}\ar@<+.8ex>[l]^-{S(JB,-)}}
\end{equation}
where both squares of left and right adjoints respectively
commute, and $(H,K)$ is a map of adjunctions.
Then, there is a unique way to define a morphism
of two variables 
\begin{equation}\label{defRS}
 \xymatrix @C=.6in @R=.4in
{\ca{B}^\op\times\ca{C}\ar[r]^-R
\ar[d]_-{J^\op\times K} & \ca{A}\ar[d]^-H \\
\caa{Y}^\op\times\caa{Z}\ar[r]_-S & \caa{X}}
\end{equation}
in $\B{Cat}^2$,
for which the natural isomorphisms
\begin{gather*}
 \ca{C}(F(A,B),C)\cong\ca{A}(A,R(B,C)) \\
\caa{Z}(G(X,Y),Z)\cong\caa{X}(X,S(Y,Z))
\end{gather*}
are natural in all three variables.
\end{thm}
\begin{proof}
The result is straightforward from the theory of
parametrized
adjunctions between categories. 
The fact that $(R(B,-),S(JB,-))$ is
an arrow in $\B{Cat}^\B{2}$ for all $B$'s, ensures that the 
diagram (\ref{defRS}) commutes on the second
variable, and also on 
the first variable on objects, since 
$HR(B,C)=S(JB,KC)$. On arrows, commutativity follows
from the 
unique way of defining $R(h,1)$ and 
$S(Jh,1)$ for any $h:B\to B'$ under these assumptions,
given by (\ref{parameterdefinining}). 
More explicitly, it is enough to consider 
the image of $R(h,1)$ under $H$
and use the fact that the unit and counit of 
$F(-,B)\dashv R(B,-)$ are above the unit and 
counit of $G(-,JB)\dashv S(JB,-)$ with respect to
the fibrations $H$ and $K$. 
\end{proof}
We call $(S,R)$ the \emph{parametrized adjoint} of 
$(F,G)$ in $[\B{2},\B{Cat}]$. If we started with 
a morphism of two variables in $\B{Fib}\subset\B{Cat}^2$, 
\emph{i.e.} a fibred 1-cell $(F,G)$ depicted as 
(\ref{FG}), what would 
change in the above statement is that the diagram
(\ref{adjunction[2,Cat]}) would be required to 
be a general fibred adjunction as in 
Definition \ref{generalfibredadjunction}, \emph{i.e.}
the partial right adjoint $R(B,-)$ to be a
cartesian functor itself. 
However, notice that by Lemma \ref{Winskellemma},
right adjoints always preserve cartesian arrows in $\B{Cat}^2$,
therefore we do not need to request this as an extra condition.
The pair $(S,R)$ is then called the 
\emph{fibred parametrized adjoint}
of $(F,G)$. On the other hand, in the context of 
$\B{OpFib}$, for the concept of an \emph{opfibred parametrized
adjoint} we request both $F$ and $R(B,-)$ to be 
cocartesian.

We are now able to propose a definition of an enriched 
fibration, based on the evidence provided above. The 
theorem that follows justifies this statement, 
in the sense that it completes our initial goal: to generalize
Theorem \ref{actionenrich} from $\B{Cat}$ to $\B{Fib}$, in order 
to establish an enrichment on the level of 0-cells
of these 2-categories.
\begin{defi}[Enriched Fibration]\label{enrichedfibration}
Suppose $T:\ca{V}\to\caa{W}$ is a monoidal
fibration. We say that an (ordinary)
fibration $P:\ca{A}\to\caa{X}$
is \emph{enriched} in $T$ when the following
conditions are satisfied:
\begin{itemize}
 \item the total category $\ca{A}$ is enriched in the total
monoidal $\ca{V}$ and the 
base category $\caa{X}$ is enriched in the base 
monoidal $\caa{W}$,
in such a way that 
\begin{equation}\label{enrichedfibrationhom}
\xymatrix @C=.8in @R=.5in
{\ca{A}^\op\times\ca{A}\ar[r]^-{\ca{A}(-,-)}
\ar[d]_-{P^\op\times P} & \ca{V}\ar[d]^-T \\
\caa{X}^\op\times\caa{X}\ar[r]_-{\caa{X}(-,-)} &
\caa{W}}
\end{equation}
commutes;
\item the composition law and the identities of the 
enrichments are compatible, in the 
sense that
\begin{align}\label{compositionidentitiesabove}
TM^{\ca{A}}_{A,B,C}&=M^{\caa{X}}_{PA,PB,PC} \\
Tj^\ca{A}_A&=j^{\caa{X}}_{PA};\notag
\end{align}
\item the partial functor $\ca{A}(A,-)$ is cartesian.
\end{itemize}
\end{defi}
It does not seem completely natural to ask for cartesianness of the 
enriched hom-functor between the total categories only on the second
variable. However
this condition is the only one with real effect, since the 
functor $\ca{A}(-,A):\ca{A}^\op\to\caa{X}$ 
goes from the total category 
of an opfibration to the total category of a fibration.
We accordingly have the notion of an \emph{enriched opfibration}.

The compatibility of the composition and identities 
of the two enrichments only says that if we take the 
image of the arrows
\begin{align*}
 M^{\ca{A}}_{A,B,C}&:\ca{A}(B,C)\otimes_\ca{V}\ca{A}(A,B)\to\ca{A}(A,C) \\
j^\ca{A}_A&:I_\ca{V}\to\ca{A}(A,A)
\end{align*}
in $\ca{A}$ under the (monoidal) fibration $T$, we obtain the actual
\begin{align*}
 M^{\caa{X}}_{PA,PB,PC}&:\caa{X}(PB,PC)\otimes_\caa{W}\caa{X}(PA,PB)\to\caa{X}(PA,PC) \\
j^{\caa{X}}_{PA}&:I_\caa{W}\to\caa{X}(PA,PA)
\end{align*}
where the domains and codomains already coincide by 
strict monoidality of $T$ and the commutativity of (\ref{enrichedfibrationhom}).

Notice that in the above definition, there exists 
the usual abuse of notation, where the same name is given to the 
enriched categories and their underlying ordinary categories. 
If we wanted to be more rigorous, 
we should denote the categories with the additional
enriched structure differently, for example $\B{A}$
and $\B{X}$. In that case the `enriched hom-functor'
(\ref{enrichedfibrationhom}),
analogous to (\ref{enrichedhomfunctor}) for enrichment 
in $\B{Cat}$, would be written as
\begin{displaymath}
\xymatrix @C=.6in
{\ca{A}^\op\times\ca{A}\ar[r]^-{\B{A}(-,-)}
\ar[d]_-{P^\op\times P} & \ca{V}\ar[d]^-T \\
\caa{X}^\op\times\caa{X}\ar[r]_-{\B{X}(-,-)} &
\caa{W}}
\end{displaymath}
and its partial 1-cell $(\B{A}(A,-),\B{X}(PA,-))$
is required to be a fibred 1-cell.
\begin{rmk}
When an ordinary fibration $P:\ca{A}\to\caa{X}$ is enriched in 
a monoidal fibration $T:\ca{V}\to\caa{W}$, the latter 
has a strict monoidal structure 
hence by Proposition \ref{changeofbase} we can make the $\ca{V}$-category
$\ca{A}$ into a $\caa{W}$-enriched $\tilde{T}\ca{A}$, 
with the same set of objects $\ob\ca{A}$ and hom-objects 
$T\ca{A}(A,B)=\caa{X}(PA,PB)$.

Then, the ordinary functor $P$ 
can be viewed as a $\caa{W}$-enriched functor between the 
$\caa{W}$-categories $\tilde{T}\ca{A}$ and $\caa{X}$: on 
objects it is the function $\ob P:\ob\ca{A}\to\ob\caa{X}$
and on hom-objects it is the identity arrow
$T\ca{A}(A,B)\xrightarrow{\;=\;}\caa{X}(PA,PB)$. The compatibility
with the composition and the identities of the enriched categories,
expressed by the commutativity of the diagrams (\ref{Venrichedfunctordiagrams}),
is ensured by the relations (\ref{compositionidentitiesabove}).
\end{rmk}
After a closer comparison between our Definition \ref{enrichedfibration}
of an enriched fibration, and Shulman's \cite[Definition 4.1]{Enrichedindexedcats}
of an indexed $\ca{V}$-category, we conclude that even if 
there are conceptual similarities, our definition
cannot even restrict in a straightforward way 
to the case of fibrations over the same base:
the monoidal category $\caa{W}$ is not 
in principle enriched over itself, and certainly 
not via an identity functor.
For a more accurate description of the similarities
and differences of the two approaches to the subject,
a detailed exposition of the ideas and theory 
in \cite{Enrichedindexedcats} would be needed, but 
this would go beyond the scope of this thesis.

We now proceed to a result which asserts that to give 
a fibration and an action $(*,\diamond)$
of a monoidal fibration $T$ with a 
fibred parametrized adjoint, is to give a $T$-enriched
fibration.
\begin{thm}\label{thmactionenrichedfibration}
 Suppose that $T:\ca{V}\to\caa{W}$ is a monoidal fibration,
which acts on an (ordinary) fibration $P:\ca{A}\to\caa{X}$
via the fibred 1-cell
\begin{displaymath}
 \xymatrix @C=.6in
{\ca{V}\times\ca{A}\ar[r]^-{*}\ar[d]_-{T\times P} & 
\ca{A}\ar[d]^-P \\
\caa{W}\times\caa{X}\ar[r]_-{\diamond} & \caa{X}.}
\end{displaymath}
If this action has a parametrized adjoint 
$(R,S):P^\op\times P\to T$
in $\B{Fib}$, then we can enrich the fibration 
$P$ in the monoidal fibration $T$.
\end{thm}
\begin{proof}
 By Definition \ref{Trepresentation}, the $T$-action on $P$ consists of 
two actions $*$ and $\diamond$ of the monoidal categories
$\ca{V}$ and $\caa{W}$ on the ordinary categories $\ca{A}$ and $\caa{X}$ 
respectively. Moreover, by Theorem \ref{parameteradjunctionCat2}, 
we have two pairs of adjunctions 
\begin{equation}\label{twoadjunctions}
\xymatrix @C=.5in
{\ca{A}\ar@<+.8ex>[r]^-{-*A}
\ar@{}[r]|-\bot & \ca{V}\ar@<+.8ex>[l]^-{\bar{R}(A,-)}}
\quad\mathrm{and}\quad
\xymatrix @C=.5in
{\caa{X}\ar@<+.8ex>[r]^-{-\diamond X}
\ar@{}[r]|-\bot & \caa{W}\ar@<+.8ex>[l]^-{R(X,-)}}
\end{equation}
 for all $A\in\ca{A}$ and $X\in\caa{X}$. By Theorem 
\ref{actionenrich}, there exists a $\ca{V}$-category with underlying
category $\ca{A}$ and hom-objects $\bar{R}(A,B)$ and also
a $\caa{W}$-category with underlying category $\caa{X}$ 
and hom-objects $R(X,Y)$. 
By the definition of fibred parametrized adjoints, we have that 
$(\bar{R},R)$ is a 1-cell in $\B{Cat}^2$ and moreover
$(\bar{R}(A,-),R(PA,-))$ is a 1-cell in $\B{Fib}$.

Lastly, we need to show that the composition and identity laws
of the enrichments are compatible as in (\ref{compositionidentitiesabove}). 
By computing
the adjuncts of $M^\ca{A}_{A,B,C}$ and $j^\ca{A}_A$ under
$(-*A)\dashv\bar{R}(A,-)$ which are given 
explicitly by the arrows (\ref{compositionunderadjunction})
and (\ref{identityunderadjunction}) and taking their images under $T$, 
it can be seen that they bijectively correspond to the morphisms
$M^\caa{X}_{PA,PB,PC}$ and $j^\caa{X}_{PA}$ under the adjunction
$(-\diamond X)\dashv R(X,-)$. For this, we use 
that $(P,T)$ is a map between the adjunctions (\ref{twoadjunctions}),
$T$ is a strict monoidal functor and that 
the actions $*$ and $\diamond$ are compatible, in the sense of 
the definition of a $T$-representation.
\end{proof}
Clearly, there is a dual version of the above, characterizing
the enrichment of an opfibration in a monoidal opfibration.
In order for our examples to fit in this theory, we 
also need the notion of a fibration enriched in an opfibration
and its dual. 
\begin{defi}
Suppose that $T:\ca{V}\to\caa{W}$ is a symmetric monoidal opfibration.
We say that a fibration $P:\ca{A}\to\caa{X}$ is enriched in $T$
if the opfibration $P^\op:\ca{A}^\op\to\caa{X}^\op$ is 
an enriched $T$-opfibration. 
\end{defi}
We can now apply Theorem \ref{thmactionenrichedfibration}
to obtain an enrichment of the fibration $G:\Mod\to\Mon(\ca{V})$
in the monoidal opfibration $V:\Comod\to\Comon(\ca{V})$. First of all, 
$V$ is a strict monoidal functor (\ref{strictmonoidalVG})
and $\otimes:\Comod\times\Comod\to\Comod$ preserves cocartesian 
arrows on the nose (see proof of Proposition \ref{Comodclosed}), thus 
$V$ is indeed a monoidal opfibration. Then, by Definition 
\ref{Trepresentation} we have an action 
of $V$ on $G^\op$, given by the actions $H^\op$ of $\Comon(\ca{V})$
on $\Mon(\ca{V})^\op$ and $\bar{H}^\op$ of $\Comod$
on $\Mod^\op$ as in (\ref{Hopaction}) and (\ref{barHopaction}).
The compatibility conditions between
these two actions hold and 
$\bar{H}^\op$ strictly preserves cocartesian liftings
(see Section \ref{globalcats}). Finally, 
there is evidence that the universal measuring
comodule functor $Q$ preserves cocartesian liftings
on the first variable, which would make $(Q,P)$ into
an opfibred parametrized adjoint for the action $(\bar{H}^\op,H^\op)$.
We can thus enrich $G^\op$ in $V$.
\begin{prop}
If $Q(-,N_B)$ is cocartesian, the fibration $G:\Mod\to\Mon(\ca{V})$ is 
enriched in the monoidal opfibration $V:\Comod\to\Comon(\ca{V})$.
\end{prop}

Of course, it would as well suffice to 
verify the conditions of Definition \ref{enrichedfibration}
for this particular case, in order to obtain the above result.

At this moment, similar complications arise for the proof 
that the generalized Sweedler hom functor $T(-,\ca{B}_Y)$
and the functor $\bar{T}(-,\Psi_\ca{B})$ between $\ca{V}$-modules 
and $\ca{V}$-comodules preserve cartesian liftings. 
As a result, we also cannot claim the enrichment of the 
fibrations $N$ and $P$ in the 
monoidal opfibrations $H$ and $W$ as
in (\ref{hugediag2}) unless this condition is satisfied
(like the above proposition), even though the remaining conditions
hold. We aim to verify these properties with future work.

\section{Double categorical and bicategorical setting}\label{doublecatssetting}
\nocite{Gurskitricats}

We are now interested in generalizing the above development,
starting with an arbitrary bicategory or even a double category
in place of $\ca{V}$-$\Mat$. 
The fact that Chapter \ref{VCatsVCocats} is centered around
the bicategory of $\ca{V}$-matrices
and Chapter \ref{enrichmentofmonsandmods} 
addresses the one-object bicategory case are indicative
of such an extension.
So the driving question of this 
section is to determine what kind of structure a bicategory
$\ca{K}$ should have, in order to recapture the main results
of the previous two chapters.

There are two functors of bicategories
which are fundamental for our purposes. Firstly,
a homomorphism (pseudofunctor)
\begin{displaymath}
 \otimes:\ca{K}\times\ca{K}\longrightarrow\ca{K}
\end{displaymath}
which will be part of a monoidal structure on our bicategory,
and also a lax functor 
\begin{displaymath}
 H:\ca{K}^{\mathrm{co}}\times\ca{K}\longrightarrow\ca{K}
\end{displaymath}
which under circumstances, will lead to enrichment relations
between total categories of 
certain fibrations and opfibrations. 
The above functors of bicategories
provide (ordinary) functors
\begin{align}\label{functorsbetweenhomcats}
 \otimes_{(A,B),(C,D)}&:\ca{K}(A,C)\times\ca{K}(B,D)\to
\ca{K}(A\otimes B,C\otimes D) \\
H_{(A,B),(C,D)}&:\ca{K}(A,C)^\op\times\ca{K}(B,D)\to
\ca{K}(H(A,B),H(C,D)) \notag
\end{align}
between the hom-categories. Moreover, as seen in
Lemma \ref{lemmonlaxfun}, any lax 
functor of bicategories induces a functor between
the categories of monoids of endoarrow hom-categories
with horizontal composition. Here they
produce
\begin{align*}
 \Mon(\otimes_{(A,B)})&:\Mon\ca{K}(A,A)\times\Mon\ca{K}(B,B)\to
\Mon\ca{K}(A\otimes B,A\otimes B) \\
\Mon(H_{(A,B)})&:\Comon\ca{K}(A,A)^\op\times\Mon\ca{K}(B,B)\to
\Mon\ca{K}(H(A,B),H(A,B)).
\end{align*}
These functors are just restrictions of (\ref{functorsbetweenhomcats})
on the appropriate categories, which in fact turn out to be 
fibres of total categories, crucial for the development.
Since $\otimes$ is a pseudofunctor, \emph{i.e.}
also colax with respect to the horizontal composition, 
there is also an induced functor 
\begin{displaymath}
 \Comon(\otimes_{(A,B)}):\Comon\ca{K}(A,A)\times\Comon\ca{K}(B,B)\to
\Comon\ca{K}(A\otimes B,A\otimes B).
\end{displaymath}
Under certain conditions, these functors `between the fibres' 
induce total functors which give rise to specific structures
of importance.

For $\ca{K}=\ca{V}\text{-}\Mat$ for example, these categories
are $\Mon\ca{K}(A,A)=\ca{V}\text{-}\B{Cat}_A$ and 
$\Comon\ca{K}(A,A)=\ca{V}\text{-}\B{Cocat}_A$
for fixed sets of objects $A$.
The bicategory of $\ca{V}$-matrices is in fact 
a monoidal bicategory with tensor product as in 
(\ref{monoidalVMat}) which 
induces the monoidal structure of the total
categories $\ca{V}$-$\B{Cat}$ and $\ca{V}$-$\B{Cocat}$.
Also, the lax functor $H=\Hom:(\ca{V}\textrm{-}\Mat)^{\textrm{co}}
\times\ca{V}\textrm{-}\Mat\to\ca{V}\textrm{-}\Mat$ 
defined as in (\ref{defimportHom})
gives rise to the functor $K$,
whose adjoint induces the enrichment stated by Theorem 
\ref{VCatenrichedinVCocat}.

Furthermore, by Proposition \ref{laxfunctorbetweenmodules}
the lax functors $\otimes$ and $H$ induce 
\begin{align*}
\ca{K}(A,C)^{\ca{K}(A,t)}
\times\ca{K}(B,D)^{\ca{K}(B,s)}&\to
\ca{K}(A\otimes B,C\otimes D)^{\ca{K}(A\otimes B,t\otimes s)} \\
{\ca{K}(A,C)^{\ca{K}(A,u)}}^\op\times
\ca{K}(B,D)^{\ca{K}(B,s)}&\to
\ca{K}(H(A,B),H(C,D))^{\ca{K}(H(A,B),H(u,s))}
\end{align*}
between the categories of left modules and comodules
with fixed domains,
for monads $t:C\to C$, $s:D\to D$ and comonad
$u:C\to C$ in $\ca{K}$. These can also be written as 
\begin{align*}
\Mod(\otimes_{(A,B),(C,D)})&: {^A_t}\Mod\times {^B_s}\Mod
\longrightarrow
{^{A\otimes B}_{t\otimes s}}\Mod \\
\Mod(H_{(A,B),(C,D)})&: {^A_u}\Comod^\op\times {^B_s}\Mod
\longrightarrow
{^{H(A,B)}_{H(u,s)}}\Mod
\end{align*}
by Definitions \ref{lefttmodules}, \ref{rightucomodules}.
Again, since $\otimes$ is a homomorphism of bicategories,
it also induces
\begin{displaymath}
 \Comod(\otimes_{(A,B),(C,D)}): {^A_u}\Comod\times {^B_v}\Comod
\longrightarrow
{^{A\otimes B}_{u\otimes v}}\Comod
\end{displaymath}
between the categories of comodules.
These functors between the fibres of the global
categories are expected to give
the monoidal structures to modules and comodules, and
the enrichment of modules in comodules respectively.
For the bicategory $\ca{V}$-$\B{Mat}$,
the monoidal structures of $\ca{V}$-$\Mod$ and $\ca{V}$-$\Comod$
as well as Theorem \ref{VModenrichedinVComod} 
are obtained by employing instances of the above functors.

In order to identify suitable assumptions
on the bicategory $\ca{K}$,
we are going to employ the theory of double
categories. This turns out to be 
an appropriate theoretical framework leading 
to enriched fibrations as discussed in last section,
because it provides with a better understanding of 
the nature of the categories appearing in our examples.
We largely follow the approach of \cite{ConstrSymMonBicats}, 
where a method for constructing (symmetric) monoidal
bicategories from (symmetric) monoidal
double categories which satisfy a 
lifting condition is described. This process
allows us to reduce a lengthy and demanding task
of verifying the coherence conditions of monoidal structure 
on a bicategory into a much more concise and 
speedy procedure, essentially involving a pair 
of ordinary monoidal categories.
\begin{defi}
 A \emph{(pseudo) double category} $\caa{D}$
consists of a category of objects $\caa{D}_0$ and
a category of arrows $\caa{D}_1$, with structure
functors 
\begin{displaymath}
 \B{1}:\caa{D}_0\to\caa{D}_1,\quad 
\Gr{s},\Gr{t}:\caa{D}_1\rightrightarrows\caa{D}_0,\quad
\odot:\caa{D}_1{\times_{\caa{D}_0}}\caa{D}_1\to\caa{D}_1
\end{displaymath}
such that
$\Gr{s}(1_A)$=$\Gr{t}(1_A)$=$A,\;\Gr{s}(M\odot N)$=$\Gr{s}(N),\;
\Gr{t}(M\odot N)$=$\Gr{t}(M)$
for all $A\in\ob\caa{D}_0$, $M,N\in\ob\caa{D}_1$,
equipped with natural isomorphisms
\begin{align*}
 \alpha:(M\odot N)\odot P&\xrightarrow{\;\sim\;}
M\odot(N\odot P) \\
\lambda:1_{\Gr{s}(M)}\odot M&\xrightarrow{\;\sim\;}M \\
\rho:M\odot1_{\Gr{t}(M)}&\xrightarrow{\;\sim\;}M
\end{align*}
in $\caa{D}_1$ for all $M,N,E\in\ob\caa{D}_1$, such that 
$\Gr{t}(\alpha),\Gr{s}(\alpha),\Gr{t}(\lambda),\Gr{s}(\lambda),
\Gr{t}(\rho),\Gr{s}(\rho)$ are all
identities, and satisfying the usual coherence conditions
(as for a bicategory).
\end{defi}
The objects of $\caa{D}_0$ are called \emph{0-cells}
and the morphisms of $\caa{D}_0$ are called 
\emph{1-morphisms} or \emph{vertical 1-cells},
denoted as $f:A\to B$. The objects of $\caa{D}_1$ are 
the \emph{(horizontal) 1-cells}, denoted as$\SelectTips{eu}{10}
\xymatrix@C=.2in{M:A\ar[r]
\ar@{}[r]|-{\scriptstyle{\bullet}} & B}$
where $\Gr{s}(M)=A$ is the source and $\Gr{t}(M)=B$ the 
target of $M$. The morphisms of $\caa{D}_1$ are the 
\emph{2-morphisms}, denoted as squares
\begin{displaymath}
\xymatrix
{A\ar[r]^-M\ar@{}[r]|-{\scriptstyle{\bullet}} 
\rtwocell<\omit>{<4>\alpha} \ar[d]_-f & B\ar[d]^-g \\
C\ar[r]_-N\ar@{}[r]|-{\scriptstyle{\bullet}} & D}
\end{displaymath}
or $^f\alpha^g:M\Rightarrow N$, 
where $\Gr{s}(\alpha)=f$ and $\Gr{t}(\alpha)=g$.
The composition of vertical 1-cells  
and the vertical composition of 2-morphisms
are strictly associative since $\caa{D}_0$ and 
$\caa{D}_1$ are categories, whereas 
horizontal composition of horizontal
1-cells and 2-morphisms is associative up to isomorphism
due to the isomorphisms $a_{M,N,P}$. These 
are respectively written as
\begin{displaymath}
 \xymatrix @C=.25in @R=.25in
{A\ar[r]^-M\ar@{}[r]|-{\scriptstyle{\bullet}} \ar[d]_-f
\rtwocell<\omit>{<3>\alpha} & B\ar[d]^-g \\
C\ar[r]^-N\ar@{}[r]|-{\scriptstyle{\bullet}} \ar[d]_-h
\rtwocell<\omit>{<3>\beta} & D\ar[d]^-k \\
E\ar[r]_-P\ar@{}[r]|-{\scriptstyle{\bullet}} & F}
\xymatrix{\hole \\ = \\ \hole}
\xymatrix @C=.25in @R=.25in
{A\ar[r]^-M\ar@{}[r]|-{\scriptstyle{\bullet}} \ar[dd]_-{hf} &
B\ar[dd]^-{kg} \\
\qquad\color{white}{C} \rtwocell<\omit>{\;\beta\alpha} & \\
E\ar[r]_-P\ar@{}[r]|-{\scriptstyle{\bullet}} & F,}\quad
 \xymatrix @C=.08in @R=.08in
{&& \\
A\ar[rr]^-M\ar@{}[rr]|-{\scriptstyle{\bullet}} \ar[dd]_-f
\rrtwocell<\omit>{<4>\alpha} && B\ar[rr]^-N
\ar@{}[rr]|-{\scriptstyle{\bullet}} \ar[dd]_-g \rrtwocell
<\omit>{<4>\beta} && C\ar[dd]^-h \\
&&& \\
D\ar[rr]_-P\ar@{}[rr]|-{\scriptstyle{\bullet}} && E\ar[rr]_-K
\ar@{}[rr]|-{\scriptstyle{\bullet}}  && F \\
&&}
\xymatrix @C=.08in @R=.08in
{ \\
\hole \\
= \\ }
\xymatrix @C=.08in @R=.08in
{&& \\
A\ar[rrr]^-{N\odot M}\ar@{}[rrr]|-{\scriptstyle{\bullet}}
\ar[dd]_-f & \rtwocell<\omit>{<4>{\quad\beta\odot\alpha}} && 
C\ar[dd]^-h \\
&&& \\
D\ar[rrr]_-{K\odot P}\ar@{}[rrr]|-{\scriptstyle{\bullet}} &&& F. \\
&&} 
\end{displaymath}
The vertical identity 1-cell $\mathrm{id}_A:A\to A$ for any object $A$
and the identity 2-morphism $1_M$
for any 1-cell $M$ make the vertical compositions also strictly unital. 
Also, the horizontal unit 1-cell$\SelectTips{eu}{10}\xymatrix@C=.2in
{1_A:A\ar[r]\ar@{}[r]|-{\scriptstyle{\bullet}} & A}$for every object 
$A$ and the horizontal unit 2-morphism $1_f$ 
for any 1-morphism $f:A\to B$ make the horizontal
compositions unital up to isomorphism. The 
identity 2-morphisms are denoted by
\begin{displaymath}
 \xymatrix
{A\ar[r]^-M\ar@{}[r]|-{\scriptstyle{\bullet}} \ar[d]_-{\mathrm{id}_A}
\rtwocell<\omit>{<4>{\;\;1_M}} & B \ar[d]^-{\mathrm{id}_B} \\
A\ar[r]_-M\ar@{}[r]|-{\scriptstyle{\bullet}} & B} \qquad
 \xymatrix
{A\ar[r]^-{1_A}\ar@{}[r]|-{\scriptstyle{\bullet}} \ar[d]_-f
\rtwocell<\omit>{<4>{\;\;1_f}} & A \ar[d]^-f \\
B\ar[r]_-{1_B}\ar@{}[r]|-{\scriptstyle{\bullet}} & B}
\end{displaymath}
and in particular $1_{1_A}=1_{\mathrm{id}_A}$. 
Functoriality of the 
horizontal composition $\odot$ results in the
relation $1_N\odot1_M=1_{N\odot M}$ and 
the interchange law which the two different
compositions obey:
\begin{displaymath}
 (\beta'\beta)\odot(\alpha'\alpha)=
(\beta'\odot\alpha')(\beta\odot\alpha).
\end{displaymath}

The \emph{opposite double category} $\caa{D}^\op$
is the double category with vertical category 
$\caa{D}_0^\op$ and horizontal category $\caa{D}_1^\op$.
There also exist the \emph{horizontally opposite}
double category $\caa{D}^\mathrm{hop}$ and 
\emph{vertically opposite} double category 
$\caa{D}^\mathrm{vop}$, where the horizontal and 
vertical categories respectively are the opposite
ones.

A 2-morphism with identity source and target 1-morphisms, 
like $a,l,r$ above, is called  \emph{globular}.
Evidently, for every double
category $\caa{D}$ there is a corresponding bicategory
denoted by $\ca{H}(\caa{D})$ or just $\ca{D}$, 
called its \emph{horizontal bicategory}. It 
consists of the objects, (horizontal) 1-cells
and globular 2-morphisms. In a sense, this comes
from discarding the vertical structure
of the double category. 

Many well-known bicategories arise as the horizontal bicategories
of specific double categories.
For example, consider the double category
$\ca{V}$-$\MMat$:
the category of objects is
$\ca{V}$-$\MMat_0$=$\B{Set}$, 
and the category of arrows $\ca{V}$-$\MMat_1$
consists of $\ca{V}$-matrices$\SelectTips{eu}{10}
\xymatrix@C=.2in
{S:X\ar[r]|-{\object@{|}} & Y}$as 1-cells, and 2-morphisms
$^f\alpha^g:S\Rightarrow T$
given by families of arrows 
\begin{displaymath}
\alpha_{y,x}:S(y,x)\to T(gy,fx)
\end{displaymath}
in $\ca{V}$ for all $x\in X$ and $y\in Y$. The 
structure functor
$\B{1}$ gives the identity $\ca{V}$-matrix$\SelectTips{eu}{10}
\xymatrix@C=.2in
{1_X:X\ar[r]|-{\object@{|}} & X}$for all sets $X$
and the unit 2-morphism $1_f$ with components
arrows
\begin{displaymath}
 (1_f)_{x',x}:1_X(x',x)\to1_X(x',x)\equiv
\begin{cases}
 I\xrightarrow{1_I}I, &\textrm{ if}\;x=x' \\
0\to 0, &\textrm{ if}\;x\neq x'.
\end{cases}
\end{displaymath}
The source and target functors give the evident
sets and functions, and the functor
\begin{displaymath}
 \odot:\ca{V}\text{-}\MMat_1{\times_{\ca{V}\text{-}\MMat_0}}
\ca{V}\text{-}\MMat_1\to\ca{V}\text{-}\MMat_1
\end{displaymath}
is given by the usual composition of $\ca{V}$-matrices
as in (\ref{horizontalcompositionVmatrices})
on objects, and on 2-morphisms 
$^f(\beta\odot\alpha)^g:T\circ S\Rightarrow
T'\circ S'$ is given by the composite arrows
\begin{displaymath}
\xymatrix @C=1in @R=.25in
{\sum_{y} T(z,y)\otimes S(y,x)
\ar[r]^-{\sum\beta_{z,y}\otimes\alpha_{y,z}}
\ar@/_3ex/@{-->}[dr] &
\sum_{y}T'(hz,gy)\otimes S'(gy,fx)
\ar@{_(->}[d]^-{\iota} \\
& \sum_{y'}T'(hz,y')\otimes S'(y',fx)}
\end{displaymath}
in $\ca{V}$, for all $x\in X$ and $z\in Z$. Notice how
this generalizes the operation 
(\ref{horizontalcompositionVmatricearrows}) 
between $\ca{V}$-matrices
of different domain and codomain.
Compatibility conditions of source and target 
functors with composition can be easily checked, and 
the globular 2-isomorphisms are the ones
described in Section \ref{bicatVMat}. 
Of course, its horizontal bicategory
$\ca{H}(\ca{V}\text{-}\MMat)$ is precisely
the bicategory $\ca{V}$-$\Mat$.
\begin{defi}
For $\caa{D}$ and $\caa{E}$ (pseudo)
double categories, a 
\emph{pseudo double functor} $F:\caa{D}\to\caa{E}$
consists of functors $F_0:\caa{D}_0\to\caa{E}_0$ and 
$F_1:\caa{D}_1\to\caa{E}_1$ between the categories
of objects and arrows, such that 
$\Gr{s}\circ F_1=F_0\circ\Gr{s}$ 
and $\Gr{t}\circ F_1=F_0\circ\Gr{t}$, and natural 
transformations $F_{\odot}$, $F_U$ 
with components
globular isomorphisms $F_1M\odot F_1N\xrightarrow{\sim}F_1(M\odot N)$
and $1_{F_0A}\xrightarrow{\sim}F_1(1_A)$ respectively, 
which satisfy the usual coherence axioms
for a pseudofunctor.
\end{defi}
We also have notions of \emph{lax} and \emph{colax double functors}
between pseudo double categories, where the natural
transformations $F_{\odot}$ and $F_U$ have components 
globular 2-morphisms in one of the two possible
directions respectively. The explicit definitions can be 
found in the appendix of \cite{Limitsindoublecats} or 
\cite{Adjointfordoublecats}. In
particular, naturality of $F_\odot$ in this context means
the following: for any composable 2-morphisms 
$^f\alpha^g:M\Rightarrow M'$ and $^g\beta^h:N\Rightarrow N'$
in $\caa{D},$
the components of $F_\odot$ satisfy 
\begin{equation}\label{naturalityFodot}
 \xymatrix @C=.5in
{F_0A\ar[r]^-{F_1M}\ar@{}[r]|-{\scriptstyle{\bullet}}\rtwocell<\omit>{<5>\;\;F_1\alpha}
\ar[d]_-{F_0f} & 
F_0B\ar[r]^-{F_1N}\ar@{}[r]|-{\scriptstyle{\bullet}}\rtwocell<\omit>{<5>\;\;F_1\beta}
\ar[d]^-{F_0g} & F_0C\ar[d]^-{F_0h} \\
F_0A'\ar[r]_-{F_1M'}\ar@{}[r]|-{\scriptstyle{\bullet}}\rrtwocell<\omit>{<5>\;F_\odot}
\ar@{=}[d] & F_0B'\ar[r]_-{F_1N'}\ar@{}[r]|-{\scriptstyle{\bullet}} &
F_0C'\ar@{=}[d] \\
F_0A'\ar[rr]_-{F_1(N'\odot M')}\ar@{}[rr]|-{\scriptstyle{\bullet}} && F_0C'}
\quad\xymatrix{\hole \\ \mathrm{=}}\quad
\xymatrix @C=.5in
{F_0A\ar[r]^-{F_1M}\ar@{}[r]|-{\scriptstyle{\bullet}}\rrtwocell<\omit>{<5>\;F_\odot}
\ar@{=}[d] & 
F_0B\ar[r]^-{F_1N}\ar@{}[r]|-{\scriptstyle{\bullet}} & F_0C\ar@{=}[d] \\
F_0A\ar[rr]_-{F_1(N\odot M)}\ar@{}[rr]|-{\scriptstyle{\bullet}}
\rrtwocell<\omit>{<5>\qquad\; F_1(\beta\odot\alpha)}
\ar[d]_-{F_0f} && F_0C\ar[d]^-{F_0h} \\
F_0A'\ar[rr]_-{F_1(N'\odot M')}\ar@{}[rr]|-{\scriptstyle{\bullet}} && F_0C'.}
\end{equation}

Whenever we have a pseudo double functor $F:\caa{D}\to\caa{E},$
there is an induced pseudofunctor between
the respective horizontal bicategories
\begin{displaymath}
 \ca{H}F:\ca{H}(\caa{D})\to\ca{H}(\caa{E})
\end{displaymath}
which consists of the following data: 

$\cdot$ for each 0-cell $A\in\caa{D}_0$
in the bicategory $\ca{H}(\caa{D})$, a 0-cell $F_0A\in\caa{E}_0$
in the bicategory $\ca{H}(\caa{E})$;

$\cdot$ for each two 0-cells $A,B\in\caa{D}_0$, a functor
\begin{displaymath}
 \ca{H}F_{A,B}:\ca{H}(\caa{D})(A,B)\to\ca{H}(\caa{E})(F_0A,F_0B)
\end{displaymath}
which maps a horizontal 1-cell$\SelectTips{eu}{10}
\xymatrix@C=.2in{M:A\ar[r]\ar@{}[r]|-{\scriptstyle{\bullet}}
 & B}$to the 1-cell$\SelectTips{eu}{10}\xymatrix@C=.2in
{F_1M:F_0A\ar[r]\ar@{}[r]|-{\scriptstyle{\bullet}} & F_0B}$and
\begin{displaymath}
 \xymatrix
{A\ar[r]^-M\ar@{}[r]|-{\scriptstyle{\bullet}} \ar[d]_-{\mathrm{id_A}}
\rtwocell<\omit>{<4>\alpha} & B\ar[d]^-{\mathrm{id_B}} \\
A\ar[r]_-N\ar@{}[r]|-{\scriptstyle{\bullet}} & B}\;
\xymatrix @R=.1in
{\hole \\ \mapsto }\;
\xymatrix @C=.5in
{F_0A\ar[r]^-{F_1M}\ar@{}[r]|-{\scriptstyle{\bullet}} \ar[d]_-{\mathrm{id_{(F_0A)}}}
\rtwocell<\omit>{<4>\quad F_1\alpha} & F_0B\ar[d]^-{\mathrm{id_{(F_0B)}}} \\
F_0A\ar[r]_-{F_1N}\ar@{}[r]|-{\scriptstyle{\bullet}} & F_0B}
\end{displaymath}
using functoriality of $F_0$ and compatibility of $F_0$ and $F_1$
with sources and targets;

$\cdot$ for every triple of 0-cells $A,B,C$,
a natural isomorphism with components invertible
arrows 
\begin{displaymath}
 \delta^{N,M}:F_1N\odot F_1M\xrightarrow{\;\sim\;}F_1(N\odot M)
\end{displaymath}
for$\SelectTips{eu}{10}\xymatrix@C=.2in
{M:A\ar[r]\ar@{}[r]|-{\scriptstyle{\bullet}} & B}$and$\SelectTips{eu}{10}
\xymatrix@C=.2in
{N:B\ar[r]\ar@{}[r]|-{\scriptstyle{\bullet}} & C,}$ given by
$F_\odot$;

$\cdot$ for every 0-cell $A$,
a natural isomorphism with components invertible
\begin{displaymath}
 \gamma^A:1_{F_0A}\xrightarrow{\;\sim\;}F_1(1_A)
\end{displaymath}
given by $F_U$.

The coherence axioms are satisfied by definition
of the pseudo double functor. Similarly we get 
(co)lax functors between bicategories
from (co)lax double functors.
This is indicative of the way that 
structure may be inherited from a pseudo double category to
its horizontal bicategory. From now on, the adjective `pseudo'
will be dropped whenever it is clearly implied.

The formal definition of a 
monoidal double category can be found in 
\cite{ConstrSymMonBicats} and is omitted here.
Notice that in \cite{Adjointfordoublecats}
for example, the tensor product $\otimes$ as below
is required to be a colax double functor rather than pseudo
double. If we unpack the definition, we get the following simplified 
description.
\begin{defi}\label{monoidaldoublecategory}
 A \emph{monoidal double category} is a double category $\caa{D}$
equipped with (pseudo) double functors
\begin{displaymath}
\otimes:\caa{D}\times\caa{D}\to\caa{D}\quad
\mathrm{and}\quad 
\B{I}:\B{1}\to\caa{D}, 
\end{displaymath}
such that
$(\caa{D}_0,\otimes_0,I)$ and 
$(\caa{D}_1,\otimes_1,1_I)$ are monoidal categories
with$\SelectTips{eu}{10}\xymatrix@C=.2in
{1_I:I\ar[r]\ar@{}[r]|-{\scriptstyle{\bullet}} & I}$for 
$I=\B{I}(*),$ 
the functors $\Gr{s},\Gr{t}$ are strict monoidal and preserve
associativity and unit constraints, and there exist 
globular isomorphisms
\begin{align*}
(M\otimes_1 N)\odot(M'\otimes_1 N')&\cong
(M\odot M')\otimes_1(N\odot N') \\
1_{(A\otimes_0 B)}&\cong
1_A\otimes_1 1_B
\end{align*}
subject to coherence conditions.
\end{defi}
For example, consider the double category 
$\ca{V}$-$\MMat$ where both categories of objects and arrows are
monoidal categories. Indeed, 
$(\B{Set},\times,\{*\})$ is cartesian monoidal
and $\ca{V}\text{-}\MMat_1$ has tensor product
\begin{equation}\label{VMMat1monoidal}
 \otimes:\xymatrix @C=1.2in
{\ca{V}\text{-}\MMat_1\times\ca{V}\text{-}\MMat_1
\ar[r] & \ca{V}\text{-}\MMat_1}\phantom{ABC}
\end{equation}\vspace{-0.2in}
\begin{displaymath}
 \xymatrix @C=.025in
{(X\ar[rrr]|-{\object@{|}}^S\ar[d]_-f
&\rtwocell<\omit>{<4>{\alpha}}&& Y\ar[d]^-g
& , & Z\ar[rrr]^-T|-{\object@{|}}\ar[d]_-h
&\rtwocell<\omit>{<4>\beta}&& W)\ar[d]^-k
\ar@{|.>}[rrrr] &&&& X\times Z\ar[rrr]^-{S\otimes T}|-{\object@{|}}
\ar[d]_-{f\times h} &\rtwocell<\omit>{<4>\quad\alpha\otimes\beta}
&& Y\times W\ar[d]^-{g\times k} \\
(X'\ar[rrr]_-{S'}|-{\object@{|}} &&& Y' & , &
Z'\ar[rrr]_-{T'}|-{\object@{|}} &&& W')
\ar@{|.>}[rrrr] &&&& X'\times Z'\ar[rrr]_-{S'\otimes T'}|-{\object@{|}} 
&&& Y'\times W'} 
\end{displaymath}
given by the families $(S\otimes T)((y,w),(x,z)):=S(y,x)\otimes T(w,z)$ 
of objects in $\ca{V}$ and 
\begin{displaymath}
 (\alpha\otimes\beta)_{(y,w),(x,z)}:=
S(y,x)\otimes T(w,z)\xrightarrow{\alpha_{y,x}\otimes\beta_{w,z}}S'(gy,fx)\otimes T'(kw,hz)
\end{displaymath}
of arrows in $\ca{V}$, and monoidal unit the
$\ca{V}$-matrix
$\SelectTips{eu}{10}\xymatrix@C=.2in
{\ca{I}:\{*\}\ar[r]|-{\object@{|}} & \{*\}}$with
$\ca{I}(*,*)=I_\ca{V}$. The conditions for $\Gr{s}$ and $\Gr{t}$
are satisfied, and the natural isomorphisms come down
to combinations of associativity and unit constraints
of $\ca{V}$ and the fact that the tensor product in $\ca{V}$
commutes with sums.
\begin{prop}\label{doubleVMatmonoidal}
The pseudo double category $\ca{V}$-$\MMat$
is monoidal.
\end{prop}
What is further required to obtain a monoidal structure
on the horizontal bicategory of a monoidal double 
category is a way of turning vertical 1-morphisms into 
horizontal 1-cells. The links between vertical and horizontal
1-cells in a double category have been studied
by various authors, and the terminology used
below can be found in 
\cite{Adjointfordoublecats,Framedbicats,Thespanconstruction}.
\begin{defi}\label{deficompconj}
 Let $\caa{D}$ be a double category and $f:A\to B$
a vertical 1-morphism. A \emph{companion} of $f$
is a horizontal 1-cell$\SelectTips{eu}{10}\xymatrix@C=.2in
{\hat{f}:A\ar[r]\ar@{}[r]|-{\scriptstyle{\bullet}} & B}$together with
2-morphisms
\begin{displaymath}
 \xymatrix
{A\ar[r]^-{\hat{f}}\ar@{}[r]|-{\scriptstyle{\bullet}} 
\rtwocell<\omit>{<4>\;p_1} \ar[d]_-f & B\ar[d]^-{\mathrm{id}_B} \\
B\ar[r]_-{1_B}\ar@{}[r]|-{\scriptstyle{\bullet}} & B} 
\quad\xymatrix@R=.05in{\hole \\
\textrm{and}}\quad
 \xymatrix
{A\ar[r]^-{1_A}\ar@{}[r]|-{\scriptstyle{\bullet}} 
\rtwocell<\omit>{<4>\;p_2} \ar[d]_-{\mathrm{id}_A} & 
A\ar[d]^-{f} \\
A\ar[r]_-{\hat{f}}\ar@{}[r]|-{\scriptstyle{\bullet}} & B}
\end{displaymath}
such that $p_1p_2=1_f$ and $p_1\odot p_2\cong1_{\hat{f}}$.
Dually, a \emph{conjoint} of 
$f$ is a horizontal 1-cell$\SelectTips{eu}{10}\xymatrix@C=.2in
{\check{f}:B\ar[r]\ar@{}[r]|-{\scriptstyle{\bullet}} & A}$together with 2-morphisms
\begin{displaymath}
 \xymatrix
{B\ar[r]^-{\check{f}}\ar@{}[r]|-{\scriptstyle{\bullet}} 
\rtwocell<\omit>{<4>\;q_1} \ar[d]_-{\mathrm{id}_B} & A\ar[d]^-f \\
B\ar[r]_-{1_B}\ar@{}[r]|-{\scriptstyle{\bullet}} & B} \quad\xymatrix@R=.05in{\hole \\
\textrm{and}}\quad
 \xymatrix
{A\ar[r]^-{1_A}\ar@{}[r]|-{\scriptstyle{\bullet}} 
\rtwocell<\omit>{<4>\;q_2} \ar[d]_-{f} & 
A\ar[d]^-{\mathrm{id}_A} \\
B\ar[r]_-{\check{f}}\ar@{}[r]|-{\scriptstyle{\bullet}} & A}
\end{displaymath}
such that $q_1q_2=1_f$ and $q_2\odot q_1\cong 1_{\check{f}}$.
\end{defi}
The ideas which led to the above definitions go
back to \cite{Doublegroupoidsandcrossedmodules},
where a \emph{connection} on a double category
corresponds to a strictly functorial choice
of a companion for each vertical arrow.
Now, a \emph{fibrant double category} (\cite[Definition 3.4]{ConstrSymMonBicats})
is a double category
for which every vertical 1-morphism has a companion
and a conjoint (called \emph{framed bicategory} in \cite{Framedbicats}).
Many important properties for fibrant double categories can be 
obtained just from the definitions.
For example, companions
and conjoints of a specific 1-morphism are essentially unique
(up to unique globular isomorphism), 
and  $\hat{g}\odot\hat{f}$,
$\check{g}\odot\check{f}$ are the companion and the conjoint of $gf$.

The significance of these notions is clear in 
the context of our primary example, the double category $\ca{V}$-$\MMat$.
The companion of a function $f:X\to Y$ is the 
$\ca{V}$-matrix$\SelectTips{eu}{10}\xymatrix@C=.2in
{f_*:X\ar[r]|-{\object@{|}} & Y}$and its conjoint is 
$\ca{V}$-matrix$\SelectTips{eu}{10}\xymatrix@C=.2in
{f^*:Y\ar[r]|-{\object@{|}} & X,}$as defined in 
(\ref{f*}).
Properties of these 
$\ca{V}$-matrices, such as the adjunction $f^*\dashv f_*$
in the horizontal bicategory $\ca{V}$-$\Mat$ or
Lemmas \ref{isosofstars} and \ref{corisos}, are in fact
true in the general setting of any fibrant double category.
\emph{I.e.} for any vertical 1-morphism $f$
in $\caa{D}$, we have an adjunction 
$\hat{f}\dashv\check{f}$ in $\ca{H}(\caa{D})$.

Another important example of a fibrant double category is 
the one with horizontal bicategory $\ca{V}$-$\B{BMod}$ (or 
$\ca{V}$-$\B{Prof}$) of enriched bimodules, as briefly described
in Section \ref{Vbimodulesandmodules}. In particular, 
the companion and conjoint for each $\ca{V}$-functor (which
are the vertical 1-morphisms) are given by the 
`representable' profunctors as in (\ref{Fstars}).

The main Theorem 5.1 in \cite{ConstrSymMonBicats}
asserts that the horizontal bicategory of a 
fibrant monoidal double category inherits a 
monoidal structure. Explicitly, it consists of
the induced pseudofunctor of bicategories 
$\ca{H}(\otimes):\ca{H}(\caa{D})\times\ca{H}(\caa{D})\to
\ca{H}(\caa{D})$
and the monoidal unit $1_I$ of $\caa{D}_1$.
In particular, the double category 
of $\ca{V}$-matrices is a fibrant monoidal
double category, hence the result follows
for its horizontal bicategory 
$\ca{H}(\ca{V}\text{-}\MMat)$.
\begin{prop}\label{bicatVMatmonoidal}
 The bicategory $\ca{V}$-$\Mat$ of $\ca{V}$-matrices is 
a monoidal bicategory.
\end{prop}
The monoidal unit is the unit 
$\ca{V}$-matrix $\ca{I}$ and the induced tensor
product pseudofunctor
$\otimes:\ca{V}\text{-}\Mat\times\ca{V}\text{-}\Mat\to\ca{V}\text{-}\Mat$
maps two sets $X,Y$ to their cartesian product
$X\times Y$, and the functor
\begin{displaymath}
\otimes_{(X,Y),(Z,W)}:\ca{V}\text{-}\Mat(X,Z)\times
\ca{V}\text{-}\Mat(Y,W)\to\ca{V}\text{-}\Mat(X\times Y, Z\times W),
\end{displaymath}
is defined as in (\ref{VMMat1monoidal}), for 2-morphisms
with domain and codomain the identity vertical 1-morphisms.

We are now in position to examine how
constructions and results of the previous chapter 
may fit in the general frame of any fibrant double category.
As up to this point, our presentation aims to 
sketch the main ideas rather than rigorously 
establish a theory.

Suppose $\caa{D}$ is an arbitrary 
fibrant double category, with no monoidal structure to 
begin with.
Define the category $\caa{D}_1^\bullet$ to be
the (non-full) subcategory of $\caa{D}_1$ of
all horizontal endo-1-cells and 2-morphisms
with the same source and target. Explicitly, objects
are all 1-cells of the form$\SelectTips{eu}{10}\xymatrix@C=.2in
{M:A\ar[r]\ar@{}[r]|-{\scriptstyle{\bullet}} & A}$and arrows
are of the form
\begin{displaymath}
\xymatrix
{A\ar[r]^-M\ar@{}[r]|-{\scriptstyle{\bullet}} 
\rtwocell<\omit>{<4>\alpha} \ar[d]_-f & A\ar[d]^-f \\
B\ar[r]_-N\ar@{}[r]|-{\scriptstyle{\bullet}} & B}
\end{displaymath}
denoted by $\alpha_f:M_A\to N_B$.
In \cite{Monadsindoublecats}, this category
coincides with the vertical 1-category of the 
double category $\caa{E}\B{nd}(\caa{D})$
of (horizontal) endomorphisms,
horizontal endomorphism maps, vertical endomorphism
maps and endomorphism squares
in $\caa{D}$.

This definition is motivated by the fact that 
$\ca{V}\text{-}\MMat_1^\bullet=\ca{V}\text{-}\B{Grph}$:
objects are $\ca{V}$-graphs, \emph{i.e.}
endo-$\ca{V}$-matrices$\SelectTips{eu}{10}\xymatrix@C=.2in
{G:X\ar[r]|-{\object@{|}} & X}$given by 
objects $\{G(x',x)\}$ in $\ca{V}$, and 
arrows $\alpha_f:G_X\to H_Y$
are $\ca{V}$-graph morphisms, \emph{i.e.} a function
$f:X\to Y$ and arrows $\alpha_{x',x}:G(x',x)\to H(fx',fx)$
in $\ca{V}$. In the view of \cite[Remark 2.5]{Monadsindoublecats},
this is analogous to the fact that 
the category $\B{Grph}_\ca{E}$ of graphs
and graph morphisms internal to a finitely complete
$\ca{E}$ is identified with the category of endomorphisms
and vertical endomorphism maps in the 
double category $\caa{S}\B{pan}_\ca{E}$.
\begin{prop}\label{D_1^.bifibred}
 Suppose $\caa{D}$ is a fibrant double category. The category
$\caa{D}_1^\bullet$ is bifibred over $\caa{D}_0$.
\end{prop}
\begin{proof}
We can easily adjust a series of previous relevant proofs,
in order to construct pseudofunctors whose 
Grothendieck construction
gives rise to a fibration and an opfibration,
isomorphic to the evident
forgetful functor $\caa{D}_1^\bullet\to\caa{D}_0$ 
mapping $G_X$ to $X$ and $\alpha_f$ to $f$.
Like Proposition \ref{VGrphbifibr},
the respective pseudofunctors are
\begin{equation}\label{pseudofunctorsdoublebifibration}
 \ps{M}:
\xymatrix @R=.02in
{\caa{D}_0^\op\ar[r] & \B{Cat}, \\
A\ar @{|.>}[r]
\ar [dd]_-f & \ca{H}(\caa{D})(A,A) \\
\hole \\
B\ar @{|.>}[r] &
\ca{H}(\caa{D})(B,B)\ar[uu]_-{(\check{f}\odot\text{-}\odot\hat{f})}}
\qquad
\ps{F}:
\xymatrix @R=.02in
{\caa{D}_0\ar[r] & \B{Cat} \\
A\ar @{|.>}[r]
\ar [dd]_-f & \ca{H}(\caa{D})(A,A)
\ar[dd]^-{(\hat{f}\odot\text{-}\odot\check{f})} \\
\hole \\
B\ar @{|.>}[r] &
\ca{H}(\caa{D})(B,B).}
\end{equation}
We can illustrate the isomorphism between,
for example, the Grothendieck category 
$\Gr{G}\ps{M}$ and $\caa{D}_1^\bullet$,
just by employing companions 
and conjoints. The objects are 
the same (horizontal endo-1-cells),
and there is a bijective correspondence between 
the morphisms: given an arrow $\alpha_f$
in $\caa{D}_1^\bullet$,
we obtain a composite 2-cell
\begin{displaymath}
\xymatrix @R=.25in
{A\ar[r]^-M\ar@{}[r]|-{\scriptstyle{\bullet}}
\rtwocell<\omit>{<3.5>\alpha} \ar[d]_-f & A\ar[d]^-f \\
B\ar[r]_-N\ar@{}[r]|-{\scriptstyle{\bullet}} & B}
\quad\xymatrix@R=.1in
{\hole \\ \mapsto }\quad
 \xymatrix @C=.5in @R=.25in
{A\ar[d]_-{\mathrm{id}_A} \ar[r]^-{1_A}\ar@{}[r]|-{\scriptstyle{\bullet}}
\rtwocell<\omit>{<3.5>\;p_2} &
 A\ar[r]^-M\ar@{}[r]|-{\scriptstyle{\bullet}} \ar[d]^-f 
\rtwocell<\omit>{<3.5>\alpha} & 
A\ar[d]_-f\ar[r]^-{1_A}\ar@{}[r]|-{\scriptstyle{\bullet}}
\rtwocell<\omit>{<3.5>\;q_2} & 
A\ar[d]^{\mathrm{id}_A} \\
A\ar[r]_-{\hat{f}}\ar@{}[r]|-{\scriptstyle{\bullet}} & 
B\ar[r]_-N\ar@{}[r]|-{\scriptstyle{\bullet}} & 
B\ar[r]_-{\check{f}}\ar@{}[r]|-{\scriptstyle{\bullet}} & B}
\end{displaymath}
which is a morphism in $\Gr{G}\ps{M}$.
This assignment is an isomorphism,
with inverse mapping $\beta\mapsto
(q_1\odot 1_N\odot p_1)\beta$ for
some $\beta:M\Rightarrow\check{f}\circ N\circ\hat{f}$
in $\ca{H}(\caa{D})(A,A)$.

Similarly $\Gr{G}\ps{F}\cong\caa{D}_1^\bullet$,
but we can also deduce that $\caa{D}_1^\bullet$ is a bifibration
by Remark \ref{rmkforadjointintexingbifr}, since we have an 
adjunction $(\check{f}\odot\text{-}\odot\hat{f})\vdash
(\hat{f}\odot\text{-}\odot\check{f})$ for all $f$.
\end{proof}
Even though the above result was independently established
as a generalization of earlier proofs, the fibration part 
was also included in \cite[Proposition 3.3]{Monadsindoublecats}.
We now proceed to the definitions of structures
in arbitrary double categories, which are 
fundamental for the formalization of our examples.

A \emph{monoid} in a double category $\caa{D}$ is 
an endo-1-cell$\SelectTips{eu}{10}\xymatrix@C=.2in
{M:A\ar[r]\ar@{}[r]|-{\scriptstyle{\bullet}} & A}$, \emph{i.e.}
an object in $\caa{D}_1^\bullet$, equipped with 
globular 2-morphisms
\begin{displaymath}
\xymatrix @C=.4in @R=.4in
{A\ar[r]^-M\ar@{}[r]|-{\scriptstyle{\bullet}}  
\rrtwocell<\omit>{<4.5>m} \ar[d]_-{\mathrm{id_A}} 
& A\ar[r]^-M\ar@{}[r]|-{\scriptstyle{\bullet}} & A\ar[d]^-{\mathrm{id}_A} \\
A\ar[rr]_-M\ar@{}[rr]|-{\scriptstyle{\bullet}}  && A,}\qquad
\xymatrix @C=.4in @R=.4in
{A\ar[r]^-{1_A}\ar@{}[r]|-{\scriptstyle{\bullet}}  
\rtwocell<\omit>{<4.5>\eta} \ar[d]_-{\mathrm{id_A}} 
& A\ar[d]^-{\mathrm{id}_A} \\
A\ar[r]_-M\ar@{}[r]|-{\scriptstyle{\bullet}}  & A}
\end{displaymath}
satisfying the usual associativity and unit laws.
In fact, this is the same as a 
monad in its horizontal bicategory $\ca{H}(\caa{D})$.
A \emph{monoid homomorphism}
consists of an arrow $\alpha_f:M_A\to N_B$ in $\caa{D}_1^\bullet$
which respects multiplication and unit:
\begin{displaymath}
\xymatrix@C=.3in @R=.2in
{A\ar[r]^-M\ar@{}[r]|-{\scriptstyle{\bullet}}\ar[d]_-f
\rtwocell<\omit>{<3>\alpha} & 
A\ar[r]^-M\ar@{}[r]|-{\scriptstyle{\bullet}}\ar[d]^-f
\rtwocell<\omit>{<3>\alpha} &
A\ar[d]^-f \\
B\ar[r]_-N\ar@{}[r]|-{\scriptstyle{\bullet}}\rrtwocell<\omit>{<3>m}
\ar@{=}[d] &
B\ar[r]_-N\ar@{}[r]|-{\scriptstyle{\bullet}}
& B \ar@{=}[d] \\
A\ar[rr]_-N\ar@{}[rr]|-{\scriptstyle{\bullet}}
&& A}
\xymatrix@R=.2in{\hole \\ =}
\xymatrix@C=.3in @R=.2in
{A\ar[r]^-M\ar@{}[r]|-{\scriptstyle{\bullet}}
\rrtwocell<\omit>{<3>m}\ar@{=}[d] & 
A\ar[r]^-M\ar@{}[r]|-{\scriptstyle{\bullet}} 
& A\ar@{=}[d] \\
A\ar[rr]_-M\ar@{}[rr]|-{\scriptstyle{\bullet}}\ar[d]_-f
\rrtwocell<\omit>{<4>\alpha} && A\ar[d]^-f \\
B\ar[rr]_-N\ar@{}[rr]|-{\scriptstyle{\bullet}}
 && B,}
\qquad
\xymatrix@C=.3in @R=.2in
{A\ar[r]^-{1_A}\ar@{}[r]|-{\scriptstyle{\bullet}}\ar@{=}[d]
\rtwocell<\omit>{<3>\eta} & A\ar@{=}[d] \\
A\ar[r]_-M\ar@{}[r]|-{\scriptstyle{\bullet}}\ar[d]_-f
\rtwocell<\omit>{<3.5>\alpha} &
A\ar[d]^-f \\
B\ar[r]_-N\ar@{}[r]|-{\scriptstyle{\bullet}} & B}
\xymatrix@R=.2in{\hole \\ =}
\xymatrix@C=.3in @R=.2in
{A\ar[r]^-{1_A}\ar@{}[r]|-{\scriptstyle{\bullet}}\ar[d]_-f
\rtwocell<\omit>{<3>1_f} & A\ar[d]^-f \\
B\ar[r]_-{1_B}\ar@{}[r]|-{\scriptstyle{\bullet}}\ar@{=}[d]
\rtwocell<\omit>{<3.5>\eta} &
B\ar@{=}[d] \\
B\ar[r]_-N\ar@{}[r]|-{\scriptstyle{\bullet}} & B.}
\end{displaymath}
We obtain a category $\Mon(\caa{D})$, 
which is a non-full subcategory of $\caa{D}_1$. 

These definitions
can be found in \cite{Framedbicats} for fibrant double
categories, and in \cite{Monadsindoublecats}
as \emph{monads} and \emph{vertical monad maps} 
in a double category $\caa{D}$. In the terminology
of the latter, 
$\Mon(\caa{D})$ is in fact the vertical 
category of $\caa{M}\B{nd}(\caa{D})$, the double
category of monads, horizontal and vertical monad maps
and monad squares.
\begin{rmk*}
Considering monads in a double category rather than 
in a bicategory or 2-category presents certain advantages.
For example, $\ca{V}$-$\B{Cat}$
is precisely $\Mon(\ca{V}\text{-}\MMat)$: 
objects are monads$\SelectTips{eu}{10}\xymatrix@C=.2in
{A:X\ar[r]|-{\object@{|}} & X}$in the horizontal
bicategory $\ca{H}(\ca{V}$-$\MMat)$, and morphisms 
are $\ca{V}$-graph morphisms (\emph{i.e.} in 
$\ca{V}\text{-}\MMat_1^\bullet$) which respect the appropriate
structure. 

It was noted in Remark \ref{Vfunct=monadopfunct} that
even if objects of $\ca{V}$-$\B{Mat}$
are monads in the 
bicategory of $\ca{V}$-matrices,
$\ca{V}$-functors do not correspond bijectively 
to monad (op)functors in 
$\ca{V}$-$\Mat$. So, 
in order to fully describe $\ca{V}$-$\B{Cat}$ 
as in Lemma \ref{charactVCat}, 
we had to provide isomorphic characterizations for $\ca{V}$-functors. 
Now things are much 
clearer: we are able to recapture the whole category 
as the category of monoids in a double category,
since a $\ca{V}$-functor properly matches the notion of 
a monoid morphism in $\ca{V}$-$\MMat$.
\end{rmk*}

Dually, we can define a category
$\Comon(\caa{D})$ for any double
category. Objects are \emph{comonoids}
in $\caa{D}$, \emph{i.e.} horizontal 
endo-1-cells$\SelectTips{eu}{10}\xymatrix@C=.2in
{C:A\ar[r]\ar@{}[r]|-{\scriptstyle{\bullet}} & A}$equipped with 
globular 1-morphisms
\begin{displaymath}
\xymatrix @C=.4in @R=.4in
{A\ar[rr]^-C\ar@{}[rr]|-{\scriptstyle{\bullet}} \ar[d]_-{\mathrm{id_A}} 
&& A\ar[d]^-{\mathrm{id}_A} \\
A\ar[r]_-C\ar@{}[r]|-{\scriptstyle{\bullet}}  
\rrtwocell<\omit>{<-4>\Delta}  
& A\ar[r]_-C\ar@{}[r]|-{\scriptstyle{\bullet}} & A,}\qquad
\xymatrix @C=.4in @R=.4in
{A\ar[r]^-C\ar@{}[r]|-{\scriptstyle{\bullet}}\ar[d]_-{\mathrm{id_A}}   
& A\ar[d]^-{\mathrm{id}_A} \\
A\ar[r]_-C\ar@{}[r]|-{\scriptstyle{\bullet}}  
\rtwocell<\omit>{<-4>\epsilon} 
& A}
\end{displaymath}
satisfying the usual coassociativity 
and counit axioms for a comonad in 
the horizontal bicategory $\ca{H}(\caa{D})$.
Morphisms are \emph{comonoid homomorphisms},
\emph{i.e.} $\alpha_f:C_A\to D_B$
in $\caa{D}_1^\bullet$
satisfying dual axioms to the monoid ones.
Notice that $\Mon(\caa{D}^\op)=\Comon(\caa{D})^\op$.

For the double category 
$\caa{D}=\ca{V}\text{-}\MMat$,
the above exactly describe 
the category of $\ca{V}$-cocategories
as in the Definition \ref{cocategory}, thus
$\Comon(\ca{V}\text{-}\MMat)=\ca{V}\text{-}\B{Cocat}$.
This is again conceptually simpler 
and more straightforward than the isomorphic
characterization of $\ca{V}$-$\B{Cocat}$ as in 
Lemma \ref{charactVCocat}.
\begin{prop}\label{MonComonfibred}
Let $\caa{D}$ be a fibrant double category.
The forgetful functors 
\begin{displaymath}
 \Mon(\caa{D})\to\caa{D}_0\;\textrm{ and }\;\Comon(\caa{D})\to\caa{D}_0
\end{displaymath}
which map a
horizontal endo-1-cell to its object and a 2-morphism to 
its vertical 1-morphism, 
are a fibration and an opfibration
respectively.
\end{prop}
\begin{proof}
We can again directly generalize
Propositions \ref{VCatfibred} and \ref{VCocatopfibred} by 
restricting (\ref{pseudofunctorsdoublebifibration})
to the categories $\Mon(\ca{H}(\caa{D})(A,A))$ and 
$\Comon(\ca{H}(\caa{D})(A,A))$ respectively.

Alternatively, we can exhibit 
the cartesian lifting of a monoid$\SelectTips{eu}{10}\xymatrix@C=.2in
{N:B\ar[r]\ar@{}[r]|-{\scriptstyle{\bullet}} & B}$
\begin{displaymath}
 \xymatrix @C=.5in @C=.4in
{\check{f}\odot N\odot\hat{f} \ar[rr]^-{\Cart(f,N)}
\ar @{.>}[d] && 
N\ar @{.>}[d] & \textrm{in }\Mon(\caa{D}) \\
A\ar[rr]^-f && B & \textrm{in }\caa{D}_0}
\end{displaymath}
along a 1-morphism $f$ to be the 2-morphism
\begin{displaymath}
\xymatrix @C=.5in @R=.5in
{A\ar[d]_-f \ar[r]^-{\hat{f}}
\ar@{}[r]|-{\scriptstyle{\bullet}}
\rtwocell<\omit>{<5>\;p_1} &
 B\ar[r]^-N\ar@{}[r]|-{\scriptstyle{\bullet}} \ar@{=}[d] 
\rtwocell<\omit>{<5>\;1_N} & 
B\ar@{=}[d]\ar[r]^-{\check{f}}
\ar@{}[r]|-{\scriptstyle{\bullet}}
\rtwocell<\omit>{<5>\;q_1} & 
A\ar[d]^-f \\
B\ar[r]_-{1_B}\ar@{}[r]|-{\scriptstyle{\bullet}} & 
B\ar[r]_-N\ar@{}[r]|-{\scriptstyle{\bullet}} & 
B\ar[r]_-{1_B}\ar@{}[r]|-{\scriptstyle{\bullet}} & B.}
\end{displaymath}
The universal property can be easily verified
using the properties of companions and conjoints. 
Similarly, we can provide the cocartesian liftings
\begin{equation}\label{cocartliftComonD}
\Cocart(f,C):C\Rightarrow\hat{f}\odot C\odot \check{f}\equiv 
p_2\odot 1_C\odot q_2
\end{equation}
for the forgetful $\Comon(\caa{D})\to\caa{D}_0$.
\end{proof}
In the proof of \cite[Proposition 3.3]{Monadsindoublecats},
the new multiplication and unit 
of $(\check{f}\odot N\odot\hat{f})$
for a monoid $N$ is explicitly stated, 
and an analogous version for the comultiplication 
and counit of $(\hat{f}\odot C\odot\check{f})$
for a comonoid can be written. Essentially, they are 
the same as the ones of 
Lemmas \ref{B*monoid} and \ref{C*comonoid}
for the particular case of $\ca{V}$-categories and
$\ca{V}$-cocategories.

We now proceed with the appropriate concepts of modules
and comodules in double categories, and the 
(op)fibrations they form over $\Mon(\caa{D})$ and 
$\Comon(\caa{D})$. 

A \emph{(left) $M$-module} for a 
monoid$\SelectTips{eu}{10}\xymatrix@C=.2in
{M:A\ar[r]\ar@{}[r]|-{\scriptstyle{\bullet}} & A}$in 
a double category $\caa{D}$
is a horizontal 1-cell$\SelectTips{eu}{10}\xymatrix@C=.2in
{\Psi:Z\ar[r]\ar@{}[r]|-{\scriptstyle{\bullet}} & A}$with  
specified target $A$, equipped with a globular 2-morphism
\begin{displaymath}
\xymatrix @C=.4in @R=.4in
{Z\ar[r]^-\Psi\ar@{}[r]|-{\scriptstyle{\bullet}} 
\rrtwocell<\omit>{<4>\mu} \ar@{=}[d] & 
A\ar[r]^-M\ar@{}[r]|-{\scriptstyle{\bullet}} & A\ar@{=}[d] \\
Z\ar[rr]_-\Psi\ar@{}[rr]|-{\scriptstyle{\bullet}} && A}
\end{displaymath}
called the \emph{action}, which satisfies the usual compatibility
axioms with the multiplication and unit of the monoid $M$.
In fact this coincides with the concept of a left $M$-module
for a monad $M$ in the horizontal 
bicategory $\ca{H}(\caa{D})$.

A \emph{(left) module homomorphism}
between a left $M$-module $\Psi$ and a left $N$-module $\Xi$
consists of a monoid map $\alpha_f$ from $M$ to $N$
along with a 2-morphism
\begin{displaymath}
\xymatrix @C=.4in @R=.4in
{Z\ar[r]^-\Psi\ar@{}[r]|-{\scriptstyle{\bullet}} 
\rtwocell<\omit>{<4>\beta} \ar[d]_-k & A\ar[d]^-f \\
W\ar[r]_-\Xi\ar@{}[r]|-{\scriptstyle{\bullet}} & B}
\end{displaymath}
with specified target $f$, which satisfies the equality
\begin{displaymath}
 \xymatrix
{Z\ar[r]^-\Psi\ar@{}[r]|-{\scriptstyle{\bullet}}\ar@{=}[d]
\rrtwocell<\omit>{<4>\mu} & A\ar[r]^-M
\ar@{}[r]|-{\scriptstyle{\bullet}} & A\ar@{=}[d] \\
Z\ar[rr]_-\Psi\ar@{}[rr]|-{\scriptstyle{\bullet}}\ar[d]_-k
\rrtwocell<\omit>{<4>\beta} &&
A\ar[d]^-f \\
W\ar[rr]_-{\Xi}\ar@{}[rr]|-{\scriptstyle{\bullet}} && B}
\xymatrix{\hole \\ =}
\xymatrix
{Z\ar[r]^-\Psi\ar@{}[r]|-{\scriptstyle{\bullet}}\ar[d]_-k
\rtwocell<\omit>{<4>\beta} & 
A\ar[r]^-M\ar@{}[r]|-{\scriptstyle{\bullet}}\ar[d]^-f
\rtwocell<\omit>{<4>\alpha} & A\ar[d]^-f \\
W\ar[r]_-\Xi\ar@{}[r]|-{\scriptstyle{\bullet}}\ar@{=}[d]
\rrtwocell<\omit>{<4>\mu} & B\ar[r]_-N\ar@{}[r]|-{\scriptstyle{\bullet}} &
B\ar@{=}[d] \\
W\ar[rr]_-\Xi\ar@{}[rr]|-{\scriptstyle{\bullet}} && B.}
\end{displaymath}
Denote the category of (left) modules and 
module homomorphisms as $\Mod(\caa{D})$.

There are certain 
subcategories of importance to us:
we can consider all left modules with fixed source
$Z$ and arrows $^k\beta^f$
with $k=\id_Z$ which form a category ${^Z}\Mod(\caa{D})$;
we can also consider the category 
${_M}\Mod(\caa{D})$ of all left $M$-modules
and module homomorphisms $^k\beta^f$ with $f=\id_A$; 
finally we have the 
category ${^Z_M}\Mod(\caa{D})$ of all $M$-modules
with source $Z$ and globular 2-morphisms. 
As expected, the latter is the category ${^Z_M}\Mod(\ca{H}(\caa{D}))
=\ca{H}(\caa{D})(Z,A)^{\ca{H}(\caa{D})(Z,M)}$ as in Definition
\ref{lefttmodules}.

We can dualize the above definitions to obtain the 
category $\Comod(\caa{D})$ of \emph{(left) comodules}
and \emph{comodule homomorphisms} for any double
category $\caa{D}$. Explicitly, for 
a comonoid$\SelectTips{eu}{10}\xymatrix@C=.2in
{C:A\ar[r]\ar@{}[r]|-{\scriptstyle{\bullet}} & A}$in $\caa{D}$, 
a left $C$-comodule 
is a horizontal 1-cell$\SelectTips{eu}{10}\xymatrix@C=.2in
{\Phi:W\ar[r]\ar@{}[r]|-{\scriptstyle{\bullet}} & A}$with
target $A$, equipped with 
a globular 2-morphism
\begin{displaymath}
\xymatrix @C=.4in @R=.4in
{W\ar@{=}[d]
\ar[rr]^-\Phi\ar@{}[rr]|-{\scriptstyle{\bullet}}
\rrtwocell<\omit>{<4>\delta} && 
A\ar@{=}[d] \\
W\ar[r]_-\Phi\ar@{}[r]|-{\scriptstyle{\bullet}} & 
A\ar[r]_-C\ar@{}[r]|-{\scriptstyle{\bullet}} & A}
\end{displaymath}
called the \emph{coaction}, compatible
with the comultiplication and counit 
of the comonoid $C$. A comodule homomorphism
between a $C$-comodule $\Phi$ and a $D$-comodule
$\Omega$ consists of a comonoid map
$\alpha_f$ between $C$ and $D$ and a 2-morphism
$^k\beta^f:\Phi\Rightarrow\Omega$
which respects the coactions. 
Notice how for both module and comodule maps, the 
target agrees with the source (and target)
of the (co)monoid map, 
\emph{i.e}. $\Gr{t}(\beta)=\Gr{s}(\alpha)$. 

Once again,
we have the subcategories $^W\Comod(\caa{D})$ of 
left comodules with fixed source $W$, $_C\Comod(\caa{D})$
of left $C$-comodules for a fixed comonoid $C$, and 
the category of left $C$-comodules with fixed target $W$
\begin{displaymath}
 ^W_C\Comod(\caa{D}):={^W_C}\Comod(\ca{H}(\caa{D}))=
\ca{H}(\caa{D})(W,A)^{\ca{H}(\caa{D})(W,C)}.
\end{displaymath}

We could appropriately define categories of 
\emph{right modules} and \emph{comodules} in a 
double category $\caa{D}$, as well as \emph{bimodules}
and \emph{bicomodules}. 
In fact, bimodules between monoids 
are the horizontal 1-cells 
for a double category $\caa{M}\B{od}(\caa{D})$ 
studied in \cite{Framedbicats},
in the context of fibrant 
double categories.
According to the notation
followed in this thesis though, $\Mod$ corresponds only
to one-sided modules and $\B{BMod}$ to two-sided.

Motivated by Section \ref{globalVmodsVcomods},
we now focus on ${^Z}\Mod(\caa{D})$ and 
${^W}\Comod(\caa{D})$.
Explicitly, for $\caa{D}$=$\ca{V}\text{-}\MMat$
the categories $^1\Mod(\ca{V}\text{-}\MMat)$ 
and $^1\Comod(\ca{V}\text{-}\MMat)$ are precisely
the global categories $\ca{V}$-$\Mod$ and $\ca{V}$-$\Comod$,
where $1={\{*\}}$ is the singleton. 
Whenever appropriate, we will 
briefly remark what the results for the more general
categories of modules and comodules would look like.
\begin{prop}\label{ModComodfibred}
 Suppose $\caa{D}$ is a fibrant double category. The 
categories $^Z\Mod(\caa{D})$ and $^W\Comod(\caa{D})$ are 
fibred and opfibred respectively over $\Mon(\caa{D})$ and 
$\Comon(\caa{D})$, for any 0-cells $Z$ and $W$.
\end{prop}
\begin{proof}
Analogously to Propositions 
\ref{VModfibred} and \ref{VComodopfibred},
the indexed categories which give rise to the fibration
and opfibration in this case are 
\begin{displaymath}
\ps{H}:
\xymatrix @R=.02in
{\Mon(\caa{D})^\op\ar[r] & \B{Cat}, \\
M\ar @{|.>}[r]
\ar [dd]_-{\alpha_f} & ^Z_M\Mod(\caa{D}) \\
\hole \\
N\ar @{|.>}[r] &
^Z_N\Mod(\caa{D})
\ar[uu]_-{(\check{f}\odot\text{-})}}
\qquad
\ps{S}:
\xymatrix @R=.02in
{\Comon(\caa{D})\ar[r] & \B{Cat} \\
C\ar @{|.>}[r]
\ar [dd]_-{\alpha_f} & ^W_C\Comod(\caa{D})
\ar[dd]^-{(\hat{f}\odot\text{-})} \\
\hole \\
D\ar @{|.>}[r] &
^W_D\Comod(\caa{D}).}
\end{displaymath}
As an illustration, if$\SelectTips{eu}{10}\xymatrix@C=.2in
{\Psi:Z\ar[r]\ar@{}[r]|-{\scriptstyle{\bullet}} & B}$is a left 
$N$-module, then
$\SelectTips{eu}{10}\xymatrix@C=.2in
{\check{f}\odot\Psi:Z\ar[r]\ar@{}[r]|-{\scriptstyle{\bullet}} & B
\ar[r]\ar@{}[r]|-{\scriptstyle{\bullet}} & A}$
obtains the structure of a left $M$-module, via the action
\begin{displaymath}
 \xymatrix @C=.3in @R=.25in
{Z\ar[r]^-\Psi\ar@{}[r]|-{\scriptstyle{\bullet}}\rtwocell<\omit>{<3.7>\;1_\Psi}\ar@{=}[d] & 
B\ar[r]^-{\check{f}}\ar@{}[r]|-{\scriptstyle{\bullet}}\rtwocell<\omit>{<3.7>\;1_{\check{f}}}\ar@{=}[d] &
A\ar[rr]^-M\ar@{}[rr]|-{\scriptstyle{\bullet}}\rrtwocell<\omit>{<3.7>\;\;\lambda^{\text{-}1}}\ar@{=}[d] && 
A\ar@{=}[d] \\
Z\ar[r]_-\Psi\ar@{}[r]|-{\scriptstyle{\bullet}}\rtwocell<\omit>{<3.7>\;1_\Psi}\ar@{=}[d] & 
B\ar[r]_-{\check{f}}\ar@{}[r]|-{\scriptstyle{\bullet}}\rtwocell<\omit>{<3.7>q_1}\ar@{=}[d] & 
A\ar[r]_-M\ar@{}[r]|-{\scriptstyle{\bullet}}\rtwocell<\omit>{<3.7>\alpha}\ar[d]^-f & 
A\ar[r]_-{1_A}\ar@{}[r]|-{\scriptstyle{\bullet}}\rtwocell<\omit>{<3.7>q_2}\ar[d]^-f & 
A\ar@{=}[d] \\
Z\ar[r]_-\Psi\ar@{}[r]|-{\scriptstyle{\bullet}}\rtwocell<\omit>{<3.7>\;1_\Psi}\ar@{=}[d] & 
B\ar[r]_-{1_B}\ar@{}[r]|-{\scriptstyle{\bullet}}\rrtwocell<\omit>{<3.7>\rho}\ar@{=}[d] & 
B\ar[r]_-N\ar@{}[r]|-{\scriptstyle{\bullet}} & 
B\ar[r]_-{\check{f}}\ar@{}[r]|-{\scriptstyle{\bullet}}\rtwocell<\omit>{<3.7>\;1_{\check{f}}}\ar@{=}[d] &
A\ar@{=}[d] \\
Z\ar[r]_-\Psi\ar@{}[r]|-{\scriptstyle{\bullet}}\ar@{=}[d] & 
B\ar[rr]_-N\ar@{}[rr]|-{\scriptstyle{\bullet}}\rtwocell<\omit>{<3.7>\mu} &&
B\ar[r]_-{\check{f}}\ar@{}[r]|-{\scriptstyle{\bullet}}\rtwocell<\omit>{<3.7>\;1_{\check{f}}}\ar@{=}[d] & 
A\ar@{=}[d] \\
Z\ar[rrr]_-\Psi\ar@{}[rrr]|-{\scriptstyle{\bullet}} &&& 
B\ar[r]_-{\check{f}}\ar@{}[r]|-{\scriptstyle{\bullet}} & A}.
\end{displaymath}
This essentially generalizes Lemma \ref{Xi*module}, which 
is clearer if we suppress the 
natural isomorphisms $\lambda$, $\rho$ of the pseudo
double category.
In a dual way, we can determine the induced $D$-coaction
on a composite horizontal 1-cell
$\SelectTips{eu}{10}\xymatrix@C=.2in
{\hat{f}\odot\Phi:Z\ar[r]\ar@{}[r]|-{\scriptstyle{\bullet}} & 
A\ar[r]\ar@{}[r]|-{\scriptstyle{\bullet}} & B}$
for a left $C$-module $\Phi$, adjusting Lemma 
\ref{Phi*comodule}.

Alternatively, we can
deduce that the forgetful functors
$ {^Z}\Mod(\caa{D})\to\Mon(\caa{D})$ and
${^W}\Comod(\caa{D})\to\Comon(\caa{D})$
are a fibration and opfibration
respectively, by exhibiting the (co)cartesian arrows. 
For any left $N$-module $\Psi$ and any monoid homomorphism
$\alpha_f:M\to N$, the required cartesian lifting
$\Cart(\Psi,\alpha_f):\check{f}\odot\Psi\to\Psi$ in 
${^Z}\Mod(\caa{D})$ is the left-module morphism
\begin{displaymath}
 \xymatrix @C=.3in @R=.3in
{Z\ar[r]^-\Psi\ar@{}[r]|-{\scriptstyle{\bullet}}\rtwocell<\omit>{<4>\;1_\Psi}
\ar@{=}[d] & B\ar[r]^-{\check{f}}\ar@{}[r]|-{\scriptstyle{\bullet}}
\rtwocell<\omit>{<4>q_1}\ar@{=}[d] & A\ar[d]^-f \\
Z\ar[r]_-\Psi\ar@{}[r]|-{\scriptstyle{\bullet}}\rrtwocell<\omit>{<4>\lambda}
\ar@{=}[d] & B\ar[r]_-{1_B}\ar@{}[r]|-{\scriptstyle{\bullet}} & B\ar@{=}[d] \\
Z\ar[rr]_-\Psi\ar@{}[rr]|-{\scriptstyle{\bullet}} && B.}
\end{displaymath}
The universal property is easily checked
by the relations between $q_1$ and $q_2$, and similarly
we can write the cocartesian liftings for the second
forgetful functor.
\end{proof}
We could also establish a fibration $\Mod(\caa{D})\to\Mon(\caa{D})$
and an opfibration $\Comod(\caa{D})\to\Comon(\caa{D})$
for the categories of left modules and comodules with arbitrary sources. 
The fibre categories
would then be $_M\Mod(\caa{D})$ and $_C\Comod(\caa{D})$
respectively, and the reindexing functors the same as above.
\begin{rmk*}
Consider the categories $^X\caa{D}_1$ for 
any 0-cell $X$, of horizontal 
1-cells with domain $X$ and 2-morphisms with source
$\id_X$. We can generalize Proposition 
\ref{monadicoverpullbackcat} and deduce that 
${^Z}\Mod(\caa{D})$
is monadic over the pullback category 
$^Z\caa{D}_1\times_{\caa{D}_0}\Mon(\caa{D})$, 
and ${^W}\Comod(\caa{D})$
is comonadic over $^W\caa{D}_1\times_{\caa{D}_0}\Comon(\caa{D})$.
This further clarifies the structure and properties of 
these categories. Similarly for (co)modules of 
arbitrary domain, if we replace $^X\caa{D}_1$ by plain $\caa{D}_1$.
\end{rmk*}
We have so far totally recovered the fibrational view 
of Sections \ref{fibrationalview} and 
\ref{globalVmodsVcomods} in the abstract framework of 
fibrant double categories. 
As remarked earlier, the definitions of $\Mon(\ca{V}\text{-}\MMat)$
and $\Comon(\ca{V}\text{-}\MMat)$
wholly encapsulate the categories $\ca{V}$-$\B{Cat}$
and $\ca{V}$-$\B{Cocat}$,
and the same applies to the 
categories $\ca{V}$-$\Mod$ 
and $\ca{V}$-$\Comod$ which are identified with
$\Mod(\ca{V}\text{-}\MMat)$
and $\Comod(\ca{V}\text{-}\MMat)$.
We now turn to the issue of enrichment 
between those categories.

In order to generalize the main results of 
the previous chapter in the monoidal double 
categorical context, we require
the existence of the following functors
(compare also with the beginning of this section): a pseudo
double functor 
\begin{equation}\label{defotimes}
 \otimes:\caa{D}\times\caa{D}\longrightarrow\caa{D}
\end{equation}
which constitutes the tensor product of 
the double category,
and a lax double functor
\begin{equation}\label{Hdouble}
 H:\caa{D}^\op\times\caa{D}\longrightarrow\caa{D}
\end{equation}
with the property that $H_0$ gives a monoidal
closed structure on $(\caa{D}_0,\otimes_0,I)$ 
and $H_1$
a monoidal closed structure on $(\caa{D}_1,\otimes_1,1_I)$.

We could assume that the extra structure given 
by this lax double functor $H$ 
makes $\caa{D}$
into a \emph{monoidal closed double category}. 
However, this seems to not be
the case, even if there is
an analogy with  
Definition \ref{monoidaldoublecategory}
of a monoidal double category, where
the pseudo double functor 
$\otimes=(\otimes_0,\otimes_1)$
induces monoidal structures 
to the vertical and horizontal categories $\caa{D}_0$ 
and $\caa{D}_1$.
 
In \cite{Adjointfordoublecats}, a 
(weakly) monoidal closed pseudo double
category $\caa{D}$ is a monoidal
double category
such that each pseudo double functor 
$(\text{-}\otimes D):\caa{D}\to\caa{D}$
has a lax right adjoint, call it $\Hom^\caa{D}$.
Notice that in fact, $(\text{-}\otimes D)=
(\text{-}\otimes_0D,\text{-}\otimes_11_D)$.
This falls into the more general case of 
\emph{pseudo/lax adjunction} between
pseudo double categories
as described in \cite[3.2]{Adjointfordoublecats},
whereas double adjunctions are also
studied in \cite{Doubleadjunctionsandfreemonads}
in detail.
Explicitly, it consists of two ordinary adjunctions
\begin{displaymath}
\xymatrix @C=.7in
{\caa{D}_0\ar@<+.8ex>[r]^-{(-\otimes_0D)}
\ar@{}[r]|-{\bot} &
\caa{D}_0\ar@<+.8ex>[l]^-{\Hom^\caa{D}_0(D,-)}}, \qquad
\xymatrix @C=.6in
{\caa{D}_1\ar@<+.8ex>[r]^-{-\otimes_11_D)}
\ar@{}[r]|-{\bot} &
\caa{D}_1\ar@<+.8ex>[l]^-{\Hom^\caa{D}_1(1_D,-)}}
\end{displaymath}
for any 0-cell $D$ in $\caa{D}$,
with units and counits $\eta_{0,1},\varepsilon_{0,1}$
satisfying appropriate triangle identities, 
such that conditions 
expressing compatibility with the horizontal composition 
and identities are satisfied.
It immediately follows that $\caa{D}_0$ is a monoidal closed
category, but this cannot be deduced for $\caa{D}_1$
since $1_D$ is not an arbitrary horizontal 1-cell.

We call a monoidal pseudo double category equipped with 
a functor $H$ as in (\ref{Hdouble}) with such properties
a \emph{locally monoidal closed double category}.
The above arguments justify that a monoidal closed structure
on a double category does not imply a 
locally monoidal closed structure.

For example, consider the monoidal double
category $\caa{V}$-$\MMat$. The tensor product is given by 
$\otimes_0=\times$,
the cartesian monoidal structure in $\B{Set}$, 
and $\otimes_1$ defined as in (\ref{VMMat1monoidal}).
Moreover, if $\ca{V}$ is monoidal
closed and has products,
there is a lax double functor $H=(H_0,H_1)$ defined as follows.
On the vertical category, we have the exponentiation functor
\begin{displaymath}
 H_0:\B{Set}^\op\times\B{Set}\xrightarrow{\;(-)^{(-)}\;}\B{Set}
\end{displaymath}
which is the internal hom in $\B{Set}$. On the horizontal 
category
\begin{displaymath}
 H_1:\xymatrix @C=1in
{\ca{V}\text{-}\MMat_1^\op\times\ca{V}\text{-}\MMat_1
\ar[r] & \ca{V}\text{-}\MMat_1}\phantom{ABC}
\end{displaymath}\vspace{-0.2in}
\begin{displaymath}
 \xymatrix @C=.025in @R=.25in
{(X\ar[rrr]|-{\object@{|}}^S\ar[d]_-f
&\rtwocell<\omit>{<4>{\alpha}}&& Y\ar[d]^-g
& , & Z\ar[rrr]^-T|-{\object@{|}}\ar[d]_-h
&\rtwocell<\omit>{<4>\beta}&& W)\ar[d]^-k
\ar@{|.>}[rrr] &&& Z^X\ar[rrrr]^-{H_1(S,T)}|-{\object@{|}}
\ar[d]_-{h^f} &\rtwocell<\omit>{<4>\qquad H_1(\alpha,\beta)}
&&& W^Y\ar[d]^-{k^g} \\
(X'\ar[rrr]_-{S'}|-{\object@{|}} &&& Y' & , &
Z'\ar[rrr]_-{T'}|-{\object@{|}} &&& W')
\ar@{|.>}[rrr] &&& Z'^{X'}\ar[rrrr]_-{H_1(S',T')}|-{\object@{|}} 
&&&& {W'}^{Y'}}
\end{displaymath}
is defined on objects as 
$H_1(S,T)(m,n)=\prod_{(y,x)}[S(y,x),T(m(y),n(x))]$
for all $m\in W^Y$, $n\in Z^X$, and on arrows as 
\begin{align*}
 H_1(\alpha,\beta):H_1(S,T)(m,n)&\to H_1(S',T')(k^g(m),h^f(n)) \equiv \\
\prod_{\scriptscriptstyle{\stackrel{y\in Y}{x\in X}}} [S(y,x),T(m(y),n(x))]
&\to \prod_{\scriptscriptstyle{\stackrel{y'\in Y'}{x'\in X'}}}[S'(y',x'),
T'(kmg(y'),hnf(x'))]
\end{align*}
which corresponds under the adjunction 
$(\text{-}\otimes X)\dashv [X,-]$ in $\ca{V}$
for fixed $y',x'$ to the composite
\begin{displaymath}
\xymatrix @C=.8in
{\prod\limits_{\scriptscriptstyle{\stackrel{y\in Y}{x\in X}}}
[S(y,x),T(my,nx)]\otimes S'(y',x')\ar @{-->}[r]
\ar[d]_-{1\otimes \alpha_{y',x'}} & T'(kmgy',hnfx') \\
\prod\limits_{\scriptscriptstyle{\stackrel{y\in Y}{x\in X}}}
[S(y,x),T(my,nx)]\otimes S(gy',fx')\ar[d]_-{\pi_{gy',fx'}\otimes 1} & \\
[S(gy',fx'),T(mgy',nfx')]\otimes S(gy',fx')\ar[r]^-{\textrm{ev}}
& T(mgy',nfx').\ar[uu]_-{\beta_{mgy',nfx'}}}
\end{displaymath}
The globular transformations 
\begin{displaymath}
H_1(R,O)\odot H_1(S,T)\xrightarrow{\;\sim\;}H_1(R\odot S, O\odot T),\quad 
1_{H_0(X,Y)}\xrightarrow{\;\sim\;}H_1(1_X,1_Y)
\end{displaymath}
which make $H=(H_0,H_1)$ into a lax double functor
are as in (\ref{Homlaxfunctor1}), (\ref{Homlaxfunctor2}).
The functor $H_1$ constitutes a monoidal closed structure
for $(\ca{V}\text{-}\MMat_1,\otimes_1,1_I)$,
the proof being essential Proposition \ref{VGrphclosed} 
in the more general case of arbitrary horizontal 1-cells
and not only endoarrows like $\ca{V}$-graphs. 

For an arbitrary locally monoidal closed double category $\caa{D}$,
we now aim to investigate possible enrichment relations between 
the (op)fibrations of Propositions 
\ref{MonComonfibred} and \ref{ModComodfibred}. 
The following properties of double
functors resemble to properties of monoidal
functors studied in Chapter \ref{monoidalcategories}.
\begin{prop}\label{MonFdouble}
Any lax double functor $(F_0,F_1):\caa{D}\to\caa{E}$ 
induces an ordinary functor
\begin{displaymath}
 \Mon F:\Mon(\caa{D})\to\Mon(\caa{E})
\end{displaymath}
between the categories of monoids, which is 
$F_1$ restricted to $\Mon\caa{D}$. Dually,
any colax double functor induces
a functor between the categories of comonoids.
\end{prop}
\begin{rmk*}
Since monoids in a double category are 
monads in its horizontal bicategory and 
a lax double functor
induces a lax functor between the 
horizontal bicategories,
the above statement on the level of objects
coincides with 
Remark \ref{laxfunctorspreservemonads}. 
\end{rmk*}
\begin{proof}
A monoid$\SelectTips{eu}{10}\xymatrix@C=.2in
{M:A\ar[r]\ar@{}[r]|-{\scriptstyle{\bullet}} & A}$with $m:M\odot M\to M$
and $\eta:1_M\to M$ is mapped to$\SelectTips{eu}{10}\xymatrix@C=.2in
{F_1M:F_0A\ar[r]\ar@{}[r]|-{\scriptstyle{\bullet}} & F_0A}$with 
multiplication and unit
\begin{displaymath}
 \xymatrix
{F_0A\ar[r]^-{F_1M}\ar@{}[r]|-{\scriptstyle{\bullet}}\ar@{=}[d]\rrtwocell<\omit>{<5>\;F_\odot}
& F_0A\ar[r]^-{F_1M}\ar@{}[r]|-{\scriptstyle{\bullet}} & F_0A\ar@{=}[d] \\
F_0A\ar[rr]_-{F_1(M\odot M)}\ar@{}[rr]|-{\scriptstyle{\bullet}}\ar@{=}[d]\rrtwocell<\omit>{<5>\quad F_1m}
&& F_0A\ar@{=}[d] \\
F_0A\ar[rr]_-{F_1 M}\ar@{}[rr]|-{\scriptstyle{\bullet}} && F_0A}\quad
\xymatrix
{\hole \\ \mathrm{and}}
\quad
\xymatrix
{F_0A\ar[r]^-{F_1(1_A)}\ar@{}[r]|-{\scriptstyle{\bullet}}\ar@{=}[d]\rtwocell<\omit>{<5>\;F_U} &
F_0A\ar@{=}[d] \\
F_0A\ar[r]_-{1_{F_0A}}\ar@{}[r]|-{\scriptstyle{\bullet}}\ar@{=}[d]\rtwocell<\omit>{<5>\quad F_1\eta} &
F_0A\ar@{=}[d] \\
F_0A\ar[r]_-{F_1M}\ar@{}[r]|-{\scriptstyle{\bullet}} & F_0A}
\end{displaymath}
and the axioms follow from the axioms for $F_{\odot}$ and 
$F_U$. A monoid arrow $\alpha_f:M\to N$ is mapped to
\begin{displaymath}
 \xymatrix
{F_0A\ar[r]^-{F_1 M}\ar@{}[r]|-{\scriptstyle{\bullet}}\ar[d]_-{F_0f}\rtwocell<\omit>{<4>\;\; F_1\alpha} &
F_0A\ar[d]^-{F_0f} \\
F_0B\ar[r]_-{F_1 N}\ar@{}[r]|-{\scriptstyle{\bullet}} & F_0B}
\end{displaymath}
which respects multiplications and units by naturality 
of $F_{\odot}$ as in (\ref{naturalityFodot}) and $F_U$.
\end{proof}
\begin{prop}\label{ModFdouble}
 Any lax double functor $F:\caa{D}\to\caa{E}$ induces a functor
\begin{displaymath}
 ^Z\Mod F:{^Z}\Mod(\caa{D})\to{^{F_0Z}}\Mod(\caa{E})
\end{displaymath}
between the categories of modules, which
is a restriction of $F_1$. Dually, any colax
double functor $G$ induces a functor
\begin{displaymath}
  ^W\Comod G:{^W}\Comod(\caa{D})\to{^{G_0W}}\Comod(\caa{E}).
\end{displaymath}
\end{prop}
\begin{proof}
On the level of objects, Proposition 
\ref{laxfunctorbetweenmodules} gives functors
\begin{align*}
 {_M^Z}\Mod F&:{_M^Z}\Mod(\caa{D})\to{_{F_1M}^{F_0Z}}\Mod(\caa{E}) \\
 {_C^W}\Comod G&:{_C^W}\Comod(\caa{D})\to{_{G_1C}^{G_0W}}\Comod(\caa{E})
\end{align*}
since (co)modules for a (co)monoid in a double category
are (co)modules for a (co)monad
in its horizontal bicategory.
The $F_1M$-action on$\SelectTips{eu}{10}\xymatrix@C=.2in
{F_1\Psi:F_0Z\ar[r]\ar@{}[r]|-{\scriptstyle{\bullet}} & F_0A}$
for $(\Psi,\mu)$ a left $M$-module is just
\begin{displaymath}
 \xymatrix@R=.25in
{F_0Z\ar[r]^-{F_1\Psi}\ar@{}[r]|-{\scriptstyle{\bullet}}\rrtwocell<\omit>{<4>\;F_\odot}
\ar@{=}[d] & F_0A\ar[r]^-{F_1 M}\ar@{}[r]|-{\scriptstyle{\bullet}} & 
F_0A\ar@{=}[d] \\
F_0Z\ar[rr]_-{F_1(M\odot\Psi)}\ar@{}[rr]|-{\scriptstyle{\bullet}}\rrtwocell<\omit>{<4>\quad F_1\mu}
\ar@{=}[d] && F_0A\ar@{=}[d] \\
F_0Z\ar[rr]_-{F_1\Psi}\ar@{}[rr]|-{\scriptstyle{\bullet}} && F_0A.}
\end{displaymath}
On arrows, the fact that the image 
$^{\id_{F_0Z}}(F_1\beta)^{F_0f}:F_1\Psi\Rightarrow F_1\Xi$
of a left module morphism $\beta$
commutes with the induced actions on $F_1\Psi$, $F_1\Xi$ 
is easily verified,
by naturality of $F_\odot$ and axioms 
for $\beta$. 
\end{proof}
The functors $^Z\Mod F$ and $^W\Comod G$ are in fact 
special cases of the more general 
$\Mod F:\Mod(\caa{D})\to\Mod(\caa{E})$ and
$\Comod G:\Comod(\caa{D})\to\Comod(\caa{E})$,
between categories of (co)modules of arbitrary source, 
with a (co)action of any (co)monoid. 

Motivated by our original examples, 
we wish to employ functors between categories of 
modules with strictly the same domain. 
The following lemma
shows how under certain assumptions on $\caa{D}$
(but not in general), isomorphic 0-cells in 
$\caa{D}_0$ determine equivalent categories of modules
with such domains.
\begin{lem}\label{fibrantmodmonoidal}
Suppose $\caa{D}$ is a fibrant double category. If 
two objects $Z$ and $W$ are isomorphic 
in $\caa{D}_0$, there is 
an equivalence between the 
categories of (left) modules with fixed domain $Z$ 
and $W$, \emph{i.e.} $^Z\Mod(\caa{D})\simeq{^W\Mod(\caa{D})}$.
\end{lem}
\begin{proof}
Recall that for any isomorphism $f$
in $\caa{D}_0$, the adjunction $\hat{f}\dashv\check{f}$
in $\ca{H}(\caa{D})$ is an adjoint equivalence, and in 
particular the unit and counit $\check{\eta}$, $\check{\varepsilon}$
are isomorphisms (\cite[Lemma 3.21]{ConstrSymMonBicats}).

Denote by $f:Z\xrightarrow{\;\sim\;}W$ 
the vertical isomorphism between the 0-cells. 
The functor $(\text{-}\odot\check{f}):{^Z}\caa{D}_1\to{^W}\caa{D}_1$
between categories of horizontal 1-cells with
fixed domains and 2-morphisms with sources vertical identities,
has an inverse up to isomorphism, namely
the functor $(\text{-}\odot\hat{f})$.
For example, there is a natural isomorphism
\begin{displaymath}
 \xymatrix@R=.25in
{Z\ar[rrr]^-\Psi\ar@{}[rrr]|-{\scriptstyle{\bullet}}\ar@{=}[d] &
\rtwocell<\omit>{<4>\;\;\rho^{\text{-}1}} && A\ar@{=}[d] \\
Z\ar[rr]_-{1_Z}\ar@{}[rr]|-{\scriptstyle{\bullet}}\rrtwocell<\omit>{<4>\check{\eta}}
\ar@{=}[d] && Z\ar[r]_\Psi\ar@{}[r]|-{\scriptstyle{\bullet}}\ar@{=}[d]
\rtwocell<\omit>{<4>\;1_\Psi} & A\ar@{=}[d] \\
Z\ar[r]_-{\hat{f}}\ar@{}[r]|-{\scriptstyle{\bullet}} & 
W\ar[r]_{\check{f}}\ar@{}[r]|-{\scriptstyle{\bullet}} & 
Z\ar[r]_-\Psi\ar@{}[r]|-{\scriptstyle{\bullet}} & A}
\end{displaymath}
between $\Psi$ and $\Psi\odot\check{f}\odot\hat{f}$,
since in this case $\check{\eta}$ is invertible.
This equivalence in fact lifts to the categories
of horizontal 1-cells with the structure of a left $M$-module
for an arbitrary monoid $M$ in $\caa{D}$, \emph{i.e.}
$^{(\text{-})}\Mod(\caa{D})$.
\end{proof}
We can now apply the above results to the double functors
$\otimes$ (\ref{defotimes}) and $H$ (\ref{Hdouble})
for our fibrant locally monoidal closed double category $\caa{D}$.

Firstly, in any monoidal double category, the tensor product of
$(\caa{D}_1,\otimes_1,1_I)$ restricts to the category
$\caa{D}_1^\bullet$ of endo-1-cells,
therefore $(\caa{D}_1^\bullet,\otimes_1,1_I)$ is a monoidal
category itself. Then, by Proposition \ref{MonFdouble}, 
the pseudo double functor $\otimes$ induces (ordinary) functors 
\begin{align*}
 \Mon\otimes:&\Mon(\caa{D})\times\Mon(\caa{D})\to\Mon(\caa{D}) \\
\Comon\otimes:&\Comon(\caa{D})\times\Comon(\caa{D})\to\Comon(\caa{D}),
\end{align*}
given by $\otimes_1$ between the specific 
subcategories of $\caa{D}_1^\bullet$. The unit element is 
still $\SelectTips{eu}{10}\xymatrix@C=.2in
{1_I:I\ar[r]\ar@{}[r]|-{\scriptstyle{\bullet}} & I}$for $I$
the unit of $\caa{D}_0$.
\begin{prop}
 If $\caa{D}$ is a monoidal double category, then the
categories $\caa{D}_1^\bullet$, $\Mon(\caa{D})$ and $\Comon(\caa{D})$ 
inherit a monoidal structure from $\caa{D}_1$.
\end{prop}
For the monoidal double category $\caa{D}=\ca{V}\text{-}\MMat$, 
this directly implies that the 
categories $\ca{V}$-$\B{Grph}$, $\ca{V}$-$\B{Cat}$ and $\ca{V}$-$\B{Cocat}$
obtain a monoidal structure essentially given by (\ref{VMMat1monoidal}),
which of course agrees with the previous chapter.

Furthermore, by Proposition \ref{ModFdouble} the tensor product
also gives rise to functors
\begin{align*}
 {^{(Z,Z')}}\Mod\otimes:& \;{^Z}\Mod(\caa{D})\times {^{Z'}}\Mod(\caa{D})
\to {^{Z\otimes_0Z'}}\Mod(\caa{D}) \\
{^{(W,W')}}\Comod\otimes:& \;{^W}\Comod(\caa{D})\times {^{W'}}\Comod(\caa{D})
\to {^{W\otimes_0W'}}\Comod(\caa{D}).
\end{align*}
For the general categories of (left) modules and 
comodules with arbitrary domain $\Mod(\caa{D})$
and $\Comod(\caa{D})$, these mappings 
turn out to induce monoidal structures 
with unit element $1_I$.
However, since we are here interested in 
categories with fixed domains and in particular 
${^I}\Mod(\caa{D})$ and ${^I}\Comod(\caa{D})$
because of our motivating example,
the following `modified' monoidal structure 
is essential.
\begin{lem}\label{tensorproductModD}
Suppose that $\caa{D}$ is a fibrant monoidal double 
category. The categories ${^I}\Mod(\caa{D})$ and 
${^I}\Comod(\caa{D})$ inherit a `tensor product'
functor from $\caa{D}_1$.
\end{lem}
\begin{proof}
 Since $\caa{D}_0$ is a monoidal category with $\otimes_0$,
there exists a vertical isomorphism $r^0_I\text{=}l^0_I:
I\otimes_0 I\xrightarrow{\;\sim\;}I$.
Hence, by Lemma \ref{fibrantmodmonoidal} we have an 
equivalence
\begin{equation}\label{equivI}
{^{I\otimes_0 I}}\Mod(\caa{D})\simeq{^I}\Mod(\caa{D})
\end{equation}
between the categories of left modules with domain $I\otimes_0 I$
and of those with domain $I$. We can thus 
define a composite functor
\begin{equation}\label{tildeotimes}
\tilde{\otimes}:{^I}\Mod(\caa{D})\times{^I}\Mod(\caa{D})
\xrightarrow{\;{^{(I,I)}}\Mod\otimes\;}
{^{I\otimes_0 I}}\Mod(\caa{D})\xrightarrow{\;\simeq\;}
{^I}\Mod(\caa{D})
\end{equation}
where the first functor is $\otimes_1$
and the second is the equivalence $(\text{-}\odot\check{r_I})$.
It can be checked that this composite is equipped with 
natural coherent 
isomorphisms $(\Psi\tilde{\otimes}\Xi)\tilde{\otimes}\Theta\cong
\Psi\tilde{\otimes}(\Xi\tilde{\otimes}\Theta)$, 
coming from the respective ones for $\otimes_1$.
Similarly, we can work out a tensor product for 
$^I\Comod(\caa{D})$, making use of the equivalence
$^{I\otimes_0I}\Comod(\caa{D})\simeq^I\Comod(\caa{D})$.
\end{proof}
Even though, intuitively, this functor should
give rise to a monoidal structure on ${^I}\Mod(\caa{D})$, 
the natural choice of$\SelectTips{eu}{10}\xymatrix@C=.2in
{1_I:I\ar[r]\ar@{}[r]|-{\scriptstyle{\bullet}} & I}$
does not serve as the monoidal unit for
$\tilde{\otimes}$ as in (\ref{tildeotimes}). 
This is due to the fact that there is not an evident 
isomorphism between
\begin{displaymath}
\xymatrix @C=.25in
 {I\ar[r]^-{\check{r_I}}\ar@{}[r]|-{\scriptstyle{\bullet}} & 
I\otimes_0 I\ar[rr]^-{\Psi\otimes_1 1_I}\ar@{}[rr]|-{\scriptstyle{\bullet}}
&& A\otimes I}\;\mathrm{and}\;
\xymatrix
{I\ar[r]^-{\Psi}\ar@{}[r]|-{\scriptstyle{\bullet}} & A}
\end{displaymath}
unless, for example, $q_2$ for the conjoint
$\check{r_I}$ is invertible.
However, 
when the equivalence (\ref{equivI}) is an isomorphism, 
we can deduce that $({^I}\Mod(\caa{D}),\tilde{\otimes},1_I)$ is
a monoidal category.
This is again motivated by $\caa{D}=\ca{V}\text{-}\MMat$, 
where $I\text{=}\{*\}$ is the singleton set.

Now consider the lax double functor 
$H:\caa{D}^\op\times\caa{D}\to\caa{D}$ 
on a locally monoidal closed double category $\caa{D}$.
First of all, it is easy to see that $H_1$
restricts to the subcategory $\caa{D}_1^\bullet$
of endo-1-cells. 
Also, the natural isomorphism
\begin{displaymath}
 \caa{D}_1(M\otimes_1 N,P)\cong\caa{D}_1(M,H_1(N,P))
\end{displaymath}
which defines the adjunction $(-\otimes_1 N)\dashv H_1(N,-)$
implies that $\caa{D}_1^\bullet$ is also a monoidal
closed category. For example, for $\caa{D}=\ca{V}\text{-}\MMat$
this gives the monoidal closed structure on $\ca{V}$-$\B{Grph}$.
Then, by Proposition \ref{MonFdouble} there is an 
induced ordinary functor
\begin{equation}\label{MonHdouble}
 \Mon H:\Comon(\caa{D})^\op\times\Mon(\caa{D})\to\Mon(\caa{D})
\end{equation}
which is $H_1$ on the category 
$\Mon(\caa{D}^\op\times\caa{D})\cong\Mon(\caa{D}^\op)\times\Mon(\caa{D})$.
It is now easy to verify that 
for any monoid$\SelectTips{eu}{10}\xymatrix@C=.2in
{M:A\ar[r]\ar@{}[r]|-{\scriptstyle{\bullet}} & A,}$
the diagram
\begin{displaymath}
 \xymatrix @C=.7in @R=.5in
{\Comon(\caa{D})^\op \ar[r]^-{H_1(-,M)}
\ar[d] & \Mon(\caa{D})\ar[d] \\
\caa{D}_0^\op\ar[r]_-{H_0(-,A)} & \caa{D}_0.}
\end{displaymath}
commutes. There is also an adjunction
between the base categories
\begin{displaymath}
 \xymatrix @C=.8in
{\caa{D}_0\ar@<+.8ex>[r]^-{H_0^\op(-,A)}
\ar@{}[r]|-{\bot} &
\caa{D}_0^\op\ar@<+.8ex>[l]^-{H_0(-,A)}}
\end{displaymath}
for the monoidal closed category $\caa{D}_0$. 
If $\caa{D}$ is moreover fibrant, 
the legs of the diagram are fibrations by Proposition
\ref{MonComonfibred}. Lastly, if $\caa{D}$ is  
\emph{symmetric} monoidal (for the explicit definition, 
see \cite{ConstrSymMonBicats}), the
internal homs $H_0$ and $H_1$ of the monoidal closed
categories $\caa{D}_0$ and $\caa{D}_1$ are actions of 
the monoidal $\caa{D}_0^\op$, $\caa{D}_1^\op$
on the ordinary $\caa{D}_0$, $\caa{D}_1$
by Lemma \ref{inthomaction}.
Subsequently $H_0^\op$ and
\begin{displaymath}
 H_1^\op:\caa{D}_1\times\caa{D}_1^\op
\to\caa{D}_1^\op
\end{displaymath}
are actions too. Then the opposite $\Mon H^\op$ 
of the induced functor between monoids as in 
(\ref{MonHdouble}) is an action of the monoidal category 
$\Comon(\caa{D})$ on the opposite category 
$\Mon(\caa{D})^\op$, since 
the forgetful $\Mon(\caa{D})\to\caa{D}_1$
reflects isomorphisms.

We can now combine the above
with Theorem \ref{thmactionenrichedfibration} 
of the previous section
to outline how we could obtain 
an enriched opfibration from the above data.
\begin{thm}\label{vaguethm1}
Suppose $\caa{D}$ is a fibrant symmetric locally monoidal closed
double category.

$(1)$ If $H_1^\op:\Comon(\caa{D})\times\Mon(\caa{D})^\op\to\Mon(\caa{D})^\op$
is cocartesian with a pa\-ra\-me\-tri\-zed adjoint $P$, the categories $\Mon(\caa{D})^\op$
and $\Mon(\caa{D})$ are enriched in $\Comon(\caa{D})$.

$(2)$ If furthermore 
\begin{displaymath}
\xymatrix @C=.7in @R=.4in
{\Comon(\caa{D})\ar@<.8ex>[r]^-{H_1^\op(-,M)}\ar[d]\ar@{}[r]|-\bot
& \Mon(\caa{D})^\op\ar[d]\ar@<.8ex>[l]^-{P(-,M)} \\
\caa{D}_0\ar@<.8ex>[r]^-{H_0^\op(-,A)}\ar@{}[r]|-\bot & 
\caa{D}_0^\op\ar@<.8ex>[l]^-{H_0(-,A)}}
\end{displaymath}
is a general opfibred adjunction
for any monoid$\SelectTips{eu}{10}\xymatrix@C=.2in
{M:A\ar[r]\ar@{}[r]|-{\scriptstyle{\bullet}} & A}$ in $\caa{D}$,
then the fibration $\Mon(\caa{D})\to\caa{D}_0$
is enriched in the monoidal opfibration
$\Comon(\caa{D})\to\caa{D}_0$.
\end{thm}
Notice that the forgetful $\Comon(\caa{D})\to\caa{D}_0$ is a monoidal
fibration for any fibrant monoidal double category $\caa{D}$:
by definition, 
the tensor product of two comonoids has source and target
the tensor product $\otimes_0$ of the 0-cells in $\caa{D}_0$,
and it can also be verified
that $\otimes_1$ preserves the cocartesian 
liftings (\ref{cocartliftComonD}) in $\Comon(\caa{D})$.
Moreover, for the existence of such an adjoint $P$ and 
the establishment of a parametrized adjunction in $\B{OpFib}$
we can evidently employ Lemma \ref{totaladjointlem}
and Theorem \ref{totaladjointthm} . 

We now shift to the level
of modules and comodules in a fibrant locally 
monoidal closed double category, 
still focusing on categories
of horizontal 1-cells with 
fixed domain $I$, the monoidal unit of $\caa{D}_0$.
By Proposition \ref{ModFdouble}, the lax double functor
$H$ gives rise to a functor
\begin{displaymath}
 ^{(Z,W)}\Mod H:\;{^Z}\Comod(\caa{D})^\op\times{^W}\Mod(\caa{D})
\to{^{H_0(Z,W)}}\Mod(\caa{D})
\end{displaymath}
which is $H_1$ on $^{(Z,W)}\Mod(\caa{D}^\op\times\caa{D})\cong
{^Z}\Mod(\caa{D}^\op)\times{^W}\Mod(\caa{D})$.
We now obtain a commutative diagram 
\begin{displaymath}
 \xymatrix @C=.3in @R=.4in
{^I\Comod(\caa{D})^\op\ar[rrr]^-{^{(I,I)}\Mod H(-,\Psi)}
\ar[d] &&& ^{H_0(I,I)}\Mod(\caa{D})\ar[r]^-{\simeq}\ar[d]
& ^I\Mod(\caa{D})\ar[dl] \\
\Comon(\caa{D})^\op\ar[rrr]_-{\Mon H(-,M)} &&&
\Mon(\caa{D})}
\end{displaymath}
for any left $M$-module $\Psi$, where by 
Lemma \ref{fibrantmodmonoidal}
the equivalence is the functor 
$(\text{-}\odot\check{g})$, for $g:H_0(I,I)\cong I$
the isomorphism in the monoidal closed category $\caa{D}_0$. 

The following roughly sketches how we can 
establish the enrichment of $^I\Mod(\caa{D})$
in $^I\Comod(\caa{D})$, as in our particular examples.
Notice that the modified tensor product of 
Lemma \ref{tensorproductModD} gives a
monoidal structure on $^I\Comod(\caa{D})$
only when the equivalence between $^I\Comod(\caa{D})$ and 
$^{I\otimes_0 I}\Comod(\caa{D})$ is actually an isomorphism.
\begin{thm}\label{vaguethm2}
Suppose that the assumptions of Theorem \ref{vaguethm1} hold, and 
also $^I\Comod(\caa{D})\cong{^{I\otimes_0 I}}\Comod(\caa{D})$. 
If the functor $\Mod H$ has a parametrized adjoint
$Q$ such that for any left $M$-module $\Psi$,
\begin{displaymath}
\xymatrix @C=.8in @R=.4in
{^I\Comod(\caa{D})\ar@<.8ex>[r]^-{H_1^\op(-,\Psi)\odot\check{g}}
\ar@{}[r]|-\bot\ar[d] &
^I\Mod(\caa{D})^\op \ar[d]\ar@<.8ex>[l]^-{\bar{Q}(-\odot\hat{g},\Psi)} \\
\Comon(\caa{D})\ar@<.8ex>[r]^-{H_1^\op(-,M)}\ar@{}[r]|-\bot
& \Mon(\caa{D})^\op\ar@<.8ex>[l]^-{Q(-,M)}}
\end{displaymath}
is a general opfibred adjunction,
then the fibration $^I\Mod(\caa{D})\to\Mon(\caa{D})$ 
is enriched in the monoidal opfibration 
$^I\Comod(\caa{D})\to\Comon(\caa{D})$.
\end{thm}
We should stress that the above two theorems are just
an attempt to place the most significant results
and concepts of this thesis in a framework
where they may arise in a natural way,
rather than of actual importance on 
their own right as mathematical statements.
What should be quite noticeable about this final section is that 
we are more interested in fitting this recurring duality
and enrichment picture into a general theory via fibrations, 
than determining the more technical specifications required for
the exact enrichments to appear, as was the focus in the 
previous two chapters. This explains why we have not
addressed particular issues, such as
existence of limits and colimits in the categories
involved, monadicity, continuity of the key functors,
cartesianness and fibrewise
limits as well as local presentability,
which were broadly studied previously.

Hence, the significance of this abstraction basically
lies in the clarification of a setting 
for an enriched fibration picture between categories
of a dual flavor, and moreover and perhaps most
importantly, the possibility of further applications
in the context of other double categories/bicategories. 
Regarding this
last aspect we should point out the following,
without proceeding into a more detailed description due
to the conceptual limits of this thesis.
In the context of a bicategory of
\emph{$\ca{V}$-symmetries}, following a similar process
we would possibly be able to establish 
enrichments of categories of $\ca{V}$-\emph{operads} in
$\ca{V}$-\emph{cooperads}, and $\ca{V}$-operad \emph{modules}
in $\ca{V}$-cooperad \emph{comodules}. Evidently, this indicates
the necessity of further work in this area.

\backmatter

\bibliographystyle{alpha}
\bibliography{myreferences}

\begin{thebibliography}{BCSW83}

\bibitem[AJ13]{AnelJoyal}
Matthieu Anel and Andr{\'e} Joyal.
\newblock Sweedler theory of (co)algebras and the bar-cobar constructions.
\newblock arXiv:1309.6952 [math.CT], 2013.

\bibitem[AM10]{Species}
Marcelo Aguiar and Swapneel Mahajan.
\newblock {\em Monoidal functors, species and {H}opf algebras}, volume~29 of
  {\em CRM Monograph Series}.
\newblock American Mathematical Society, Providence, RI, 2010.
\newblock With forewords by Kenneth Brown and Stephen Chase and Andr{\'e}
  Joyal.

\bibitem[AP03]{VarietiesCovarieties}
Ji{\v{r}}{\'{\i}} Ad{\'a}mek and Hans-E. Porst.
\newblock On varieties and covarieties in a category.
\newblock {\em Math. Structures Comput. Sci.}, 13(2):201--232, 2003.
\newblock Coalgebraic methods in computer science (Genova, 2001).

\bibitem[AR94]{LocallyPresentable}
Ji{\v{r}}{\'{\i}} Ad{\'a}mek and Ji{\v{r}}{\'{\i}} Rosick{\'y}.
\newblock {\em Locally presentable and accessible categories}, volume 189 of
  {\em London Mathematical Society Lecture Note Series}.
\newblock Cambridge University Press, Cambridge, 1994.

\bibitem[Bar74]{BarrCoalgebras}
Michael Barr.
\newblock Coalgebras over a commutative ring.
\newblock {\em J. Algebra}, 32(3):600--610, 1974.

\bibitem[Bat91]{MeasuringCoalgebras}
Marjorie Batchelor.
\newblock Measuring coalgebras, quantum group-like objects and noncommutative
  geometry.
\newblock In {\em Differential geometric methods in theoretical physics
  ({R}apallo, 1990)}, volume 375 of {\em Lecture Notes in Phys.}, pages 47--60.
  Springer, Berlin, 1991.

\bibitem[Bat94]{Differenceoperators}
Marjorie Batchelor.
\newblock Difference operators, measuring coalgebras, and quantum group-like
  objects.
\newblock {\em Adv. Math.}, 105(2):190--218, 1994.

\bibitem[Bat00]{Batchelor}
Marjorie Batchelor.
\newblock Measuring comodules---their applications.
\newblock {\em J. Geom. Phys.}, 36(3-4):251--269, 2000.

\bibitem[BCSW83]{VarThrEnr}
Renato Betti, Aurelio Carboni, Ross Street, and Robert Walters.
\newblock Variation through enrichment.
\newblock {\em J. Pure Appl. Algebra}, 29(2):109--127, 1983.

\bibitem[B{\'e}n67]{Benabou}
Jean B{\'e}nabou.
\newblock Introduction to bicategories.
\newblock In {\em Reports of the {M}idwest {C}ategory {S}eminar}, pages 1--77.
  Springer, Berlin, 1967.

\bibitem[B{\'e}n73]{Distributeurs}
J.~B{\'e}nabou.
\newblock {\em Les distributeurs}.
\newblock Rapport (Universit{\'e} catholique de Louvain (1970- ).33
  S{\'e}minaire de math{\'e}matique pure)). Institut de math{\'e}matique pure
  et appliqu{\'e}e, Universit{\'e} catholique de Louvain, 1973.

\bibitem[B{\'e}n85]{BenabouFibered}
Jean B{\'e}nabou.
\newblock Fibered categories and the foundations of naive category theory.
\newblock {\em J. Symbolic Logic}, 50(1):10--37, 1985.

\bibitem[BKP89]{2-dimmonadtheory}
R.~Blackwell, G.~M. Kelly, and A.~J. Power.
\newblock Two-dimensional monad theory.
\newblock {\em J. Pure Appl. Algebra}, 59(1):1--41, 1989.

\bibitem[BL85]{BlockLerouxcofreecoalgebras}
Richard~E. Block and Pierre Leroux.
\newblock Generalized dual coalgebras of algebras, with applications to cofree
  coalgebras.
\newblock {\em J. Pure Appl. Algebra}, 36(1):15--21, 1985.

\bibitem[BN96]{BaezNeuchl}
John~C. Baez and Martin Neuchl.
\newblock Higher-dimensional algebra. {I}. {B}raided monoidal {$2$}-categories.
\newblock {\em Adv. Math.}, 121(2):196--244, 1996.

\bibitem[Bor94a]{Handbook1}
Francis Borceux.
\newblock {\em Handbook of categorical algebra. 1}, volume~50 of {\em
  Encyclopedia of Mathematics and its Applications}.
\newblock Cambridge University Press, Cambridge, 1994.
\newblock Basic category theory.

\bibitem[Bor94b]{Handbook2}
Francis Borceux.
\newblock {\em Handbook of categorical algebra. 2}, volume~51 of {\em
  Encyclopedia of Mathematics and its Applications}.
\newblock Cambridge University Press, Cambridge, 1994.
\newblock Categories and structures.

\bibitem[BS76]{Doublegroupoidsandcrossedmodules}
Ronald Brown and Christopher~B. Spencer.
\newblock Double groupoids and crossed modules.
\newblock {\em Cahiers Topologie G\'eom. Diff\'erentielle}, 17(4):343--362,
  1976.

\bibitem[Bun13]{Bunge}
Marta Bunge.
\newblock Tightly bounded completions.
\newblock {\em Theory Appl. Categ.}, 28:No. 8, 213--240, 2013.

\bibitem[Car95]{Carmody}
Sean Carmody.
\newblock {\em Cobordism categories}.
\newblock PhD thesis, University of Cambridge, 1995.

\bibitem[CGR12]{Multivariableadjunctions}
Eugenia Cheng, Nick Gurski, and Emily Riehl.
\newblock Multivariableadjunctions and mates.
\newblock arXiv:1208.4520 [math.CT], 2012.

\bibitem[DNR01]{HopfAlg}
Sorin D{\u{a}}sc{\u{a}}lescu, Constantin N{\u{a}}st{\u{a}}sescu, and
  {\c{S}}erban Raianu.
\newblock {\em Hopf algebras}, volume 235 of {\em Monographs and Textbooks in
  Pure and Applied Mathematics}.
\newblock Marcel Dekker Inc., New York, 2001.
\newblock An introduction.

\bibitem[DPP10]{Thespanconstruction}
Robert Dawson, Robert Par{\'e}, and Dorette Pronk.
\newblock The span construction.
\newblock {\em Theory Appl. Categ.}, 24:No. 13, 302--377, 2010.

\bibitem[DS97]{Monoidalbicats&hopfalgebroids}
Brian Day and Ross Street.
\newblock Monoidal bicategories and {H}opf algebroids.
\newblock {\em Adv. Math.}, 129(1):99--157, 1997.

\bibitem[Dub68]{AdjointTriangles}
Eduardo Dubuc.
\newblock Adjoint triangles.
\newblock In {\em Reports of the {M}idwest {C}ategory {S}eminar, {II}}, pages
  69--91. Springer, Berlin, 1968.

\bibitem[Ehr63]{Ehresmanndouble}
Charles Ehresmann.
\newblock Cat\'egories structur\'ees.
\newblock {\em Ann. Sci. \'Ecole Norm. Sup. (3)}, 80:349--426, 1963.

\bibitem[EK66]{Closedcategories}
Samuel Eilenberg and G.~Max Kelly.
\newblock Closed categories.
\newblock In {\em Proc. {C}onf. {C}ategorical {A}lgebra ({L}a {J}olla,
  {C}alif., 1965)}, pages 421--562. Springer, New York, 1966.

\bibitem[FGK11]{Monadsindoublecats}
Thomas~M. Fiore, Nicola Gambino, and Joachim Kock.
\newblock Monads in double categories.
\newblock {\em J. Pure Appl. Algebra}, 215(6):1174--1197, 2011.

\bibitem[FGK12]{Doubleadjunctionsandfreemonads}
Thomas~M. Fiore, Nicola Gambino, and Joachim Kock.
\newblock Double adjunctions and free monads.
\newblock {\em Cah. Topol. G\'eom. Diff\'er. Cat\'eg.}, 53(4):242--306, 2012.

\bibitem[Fox93]{Foxcofreecoalgebras}
Thomas~F. Fox.
\newblock The construction of cofree coalgebras.
\newblock {\em J. Pure Appl. Algebra}, 84(2):191--198, 1993.

\bibitem[GG76]{GouzouGrunig}
M.F. Gouzou and R.~Grunig.
\newblock Fibration relatives.
\newblock 1976.

\bibitem[GJ14]{GambinoJoyal}
Nicola Gambino and Andr{\'e} Joyal.
\newblock On operads, bimodules and analytic functors.
\newblock arXiv:1405.7270 [math.CT], 2014.

\bibitem[GP97]{enrthrvar}
R.~Gordon and A.~J. Power.
\newblock Enrichment through variation.
\newblock {\em J. Pure Appl. Algebra}, 120(2):167--185, 1997.

\bibitem[GP99]{Limitsindoublecats}
Marco Grandis and Robert Par{\'e}.
\newblock Limits in double categories.
\newblock {\em Cahiers Topologie G\'eom. Diff\'erentielle Cat\'eg.},
  40(3):162--220, 1999.

\bibitem[GP04]{Adjointfordoublecats}
Marco Grandis and Robert Par{\'e}.
\newblock Adjoint for double categories. {A}ddenda to: ``{L}imits in double
  categories'' [{C}ah. {T}opol. {G}\'eom. {D}iff\'er. {C}at\'eg. {\bf 40}
  (1999), no. 3, 162--220; mr1716779].
\newblock {\em Cah. Topol. G\'eom. Diff\'er. Cat\'eg.}, 45(3):193--240, 2004.

\bibitem[GPS95]{CoherenceTricats}
R.~Gordon, A.~J. Power, and Ross Street.
\newblock Coherence for tricategories.
\newblock {\em Mem. Amer. Math. Soc.}, 117(558):vi+81, 1995.

\bibitem[Gra66]{Grayfibredandcofibred}
John~W. Gray.
\newblock Fibred and cofibred categories.
\newblock In {\em Proc. Conf. Categorical Algebra (La Jolla, Calif., 1965)},
  pages 21--83. Springer, New York, 1966.

\bibitem[Gra74]{GrayFormalCategoryTheory}
John~W. Gray.
\newblock {\em Formal category theory: adjointness for {$2$}-categories}.
\newblock Lecture Notes in Mathematics, Vol. 391. Springer-Verlag, Berlin,
  1974.

\bibitem[Gro61]{Grothendieckcategoriesfibrees}
Alexander Grothendieck.
\newblock Cat{\'e}gories fibr{\'e}es et descente, 1961.
\newblock Seminaire de g{\'e}ometrie alg{\'e}brique de l'Institut des Hautes
  {\'E}tudes Scientifiques (SGA 1), Paris.

\bibitem[GS13]{GarnerShulman}
Richard Garner and Michael Shulman.
\newblock Enriched categories as a free cocompletion.
\newblock arXiv:1301.3191 [math.CT], 2013.

\bibitem[GU71]{GabrielUlmer}
P.~Gabriel and F.~Ulmer.
\newblock {\em Lokal Pr{\"a}sentierbare Kategorien}, volume 221 of {\em Lecture
  Notes in Mathematics}.
\newblock Springer-Verlag, 1971.

\bibitem[Gur07]{Gurski}
Nick Gurski.
\newblock {\em An algebraic theory of tricategories}.
\newblock PhD thesis, University of Chicago, 2007.

\bibitem[Gur13]{Gurskitricats}
Nick Gurski.
\newblock {\em Coherence in three-dimensional category theory}, volume 201 of
  {\em Cambridge Tracts in Mathematics}.
\newblock Cambridge University Press, Cambridge, 2013.

\bibitem[Haz03]{Hazewinkelcofreecoalgebras}
Michiel Hazewinkel.
\newblock Cofree coalgebras and multivariable recursiveness.
\newblock {\em J. Pure Appl. Algebra}, 183(1-3):61--103, 2003.

\bibitem[Her93]{hermidaphd}
Claudio~Alberto Hermida.
\newblock {\em {Fibrations, logical predicates and indeterminates}}.
\newblock PhD thesis, University of Edinburgh, 1993.

\bibitem[Her94]{FibredAdjunctions}
Claudio Hermida.
\newblock On fibred adjunctions and completeness for fibred categories.
\newblock In {\em Recent trends in data type specification ({C}aldes de
  {M}alavella, 1992)}, volume 785 of {\em Lecture Notes in Comput. Sci.}, pages
  235--251. Springer, Berlin, 1994.

\bibitem[Jac99]{Jacobs}
Bart Jacobs.
\newblock {\em Categorical logic and type theory}, volume 141 of {\em Studies
  in Logic and the Foundations of Mathematics}.
\newblock North-Holland Publishing Co., Amsterdam, 1999.

\bibitem[JK02]{AnoteonActions}
G.~Janelidze and G.~M. Kelly.
\newblock A note on actions of a monoidal category.
\newblock {\em Theory Appl. Categ.}, 9:61--91, 2001/02.
\newblock CT2000 Conference (Como).

\bibitem[Joh02a]{Elephant1}
Peter~T. Johnstone.
\newblock {\em Sketches of an elephant: a topos theory compendium. {V}ol. 1},
  volume~43 of {\em Oxford Logic Guides}.
\newblock The Clarendon Press Oxford University Press, New York, 2002.

\bibitem[Joh02b]{Elephant2}
Peter~T. Johnstone.
\newblock {\em Sketches of an elephant: a topos theory compendium. {V}ol. 2},
  volume~44 of {\em Oxford Logic Guides}.
\newblock The Clarendon Press Oxford University Press, Oxford, 2002.

\bibitem[JS93]{BraidedTensorCats}
Andr{\'e} Joyal and Ross Street.
\newblock Braided tensor categories.
\newblock {\em Adv. Math.}, 102(1):20--78, 1993.

\bibitem[Kel64]{OnMacLanesconditions}
G.~M. Kelly.
\newblock On {M}ac{L}ane's conditions for coherence of natural associativities,
  commutativities, etc.
\newblock {\em J. Algebra}, 1:397--402, 1964.

\bibitem[Kel72]{AbstrApproachCoherence}
G.~M. Kelly.
\newblock An abstract approach to coherence.
\newblock In {\em Coherence in categories}, pages 106--147. Lecture Notes in
  Math., Vol. 281. Springer, Berlin, 1972.

\bibitem[Kel74a]{Doctrinal}
G.~M. Kelly.
\newblock Doctrinal adjunction.
\newblock In {\em Category {S}eminar ({P}roc. {S}em., {S}ydney, 1972/1973)},
  pages 257--280. Lecture Notes in Math., Vol. 420. Springer, Berlin, 1974.

\bibitem[Kel74b]{ClubsDoctrines}
G.~M. Kelly.
\newblock On clubs and doctrines.
\newblock In {\em Category {S}eminar ({P}roc. {S}em., {S}ydney, 1972/1973)},
  pages 181--256. Lecture Notes in Math., Vol. 420. Springer, Berlin, 1974.

\bibitem[Kel05]{Kelly}
G.~M. Kelly.
\newblock Basic concepts of enriched category theory.
\newblock {\em Repr. Theory Appl. Categ.}, (10):vi+137, 2005.
\newblock Reprint of the 1982 original [Cambridge Univ. Press, Cambridge;
  MR0651714].

\bibitem[Kel06]{Ainfinityalgebras}
Bernhard Keller.
\newblock {$A$}-infinity algebras, modules and functor categories.
\newblock In {\em Trends in representation theory of algebras and related
  topics}, volume 406 of {\em Contemp. Math.}, pages 67--93. Amer. Math. Soc.,
  Providence, RI, 2006.

\bibitem[KK13]{KockKock}
Anders Kock and Joachim Kock.
\newblock Local fibred right adjoints are polynomial.
\newblock {\em Math. Structures Comput. Sci.}, 23(1):131--141, 2013.

\bibitem[KL01]{KellyLack}
G.~M. Kelly and Stephen Lack.
\newblock V-{C}at is locally presentable or locally bounded if {V} is so.
\newblock {\em Theory Appl. Categ.}, 8:555--575, 2001.

\bibitem[KM07]{Equalizerscocompletecocategories}
Bernhard Keller and Oleksandr Manzyuk.
\newblock Equalizers in the category of cocomplete cocategories.
\newblock {\em J. Homotopy Relat. Struct.}, 2(1):85--97, 2007.

\bibitem[KS74]{Review}
G.~Kelly and Ross Street.
\newblock Review of the elements of 2-categories.
\newblock In Gregory Kelly, editor, {\em Category Seminar}, volume 420 of {\em
  Lecture Notes in Mathematics}, pages 75--103. Springer Berlin / Heidelberg,
  1974.
\newblock 10.1007/BFb0063101.

\bibitem[Lac00]{LackPseudomonads}
Stephen Lack.
\newblock A coherent approach to pseudomonads.
\newblock {\em Adv. Math.}, 152(2):179--202, 2000.

\bibitem[Lac10a]{2-catcompanion}
Stephen Lack.
\newblock A 2-categories companion.
\newblock In {\em Towards higher categories}, volume 152 of {\em IMA Vol. Math.
  Appl.}, pages 105--191. Springer, New York, 2010.

\bibitem[Lac10b]{Icons}
Stephen Lack.
\newblock Icons.
\newblock {\em Appl. Categ. Structures}, 18(3):289--307, 2010.

\bibitem[Lac10c]{FreeMonoids}
Stephen Lack.
\newblock Note on the construction of free monoids.
\newblock {\em Appl. Categ. Structures}, 18(1):17--29, 2010.

\bibitem[Law73]{Lawvereclosedcats}
F.~William Lawvere.
\newblock Metric spaces, generalized logic, and closed categories.
\newblock {\em Rend. Sem. Mat. Fis. Milano}, 43:135--166 (1974), 1973.

\bibitem[LF12]{Ignacionotes}
Ignacio Lopez~Franco.
\newblock 2-dimensional category theory, 2012.
\newblock Part III graduate course lecture notes, University of Cambridge.

\bibitem[Lin69]{Linton}
F.~E.~J. Linton.
\newblock Coequalizers in categories of algebras.
\newblock In {\em Sem. on {T}riples and {C}ategorical {H}omology {T}heory
  ({ETH}, {Z}\"urich, 1966/67)}, pages 75--90. Springer, Berlin, 1969.

\bibitem[LP08]{2nervesbicats}
Stephen Lack and Simona Paoli.
\newblock 2-nerves for bicategories.
\newblock {\em $K$-Theory}, 38(2):153--175, 2008.

\bibitem[Lyu03]{Ainfinitycategories}
Volodymyr Lyubashenko.
\newblock Category of {$A_\infty$}-categories.
\newblock {\em Homology Homotopy Appl.}, 5(1):1--48, 2003.

\bibitem[Mar97]{Marmolejopseudomonads}
F.~Marmolejo.
\newblock Doctrines whose structure forms a fully faithful adjoint string.
\newblock {\em Theory Appl. Categ.}, 3:No.\ 2, 24--44, 1997.

\bibitem[McC00a]{Mccruddencoalgebroids}
Paddy McCrudden.
\newblock Balanced coalgebroids.
\newblock {\em Theory Appl. Categ.}, 7:No. 6, 71--147, 2000.

\bibitem[McC00b]{Mccruddencoalgebroidsreps}
Paddy McCrudden.
\newblock Categories of representations of coalgebroids.
\newblock {\em Adv. Math.}, 154(2):299--332, 2000.

\bibitem[ML63]{natass&comm}
Saunders Mac~Lane.
\newblock Natural associativity and commutativity.
\newblock {\em Rice Univ. Studies}, 49(4):28--46, 1963.

\bibitem[ML98]{MacLane}
Saunders Mac~Lane.
\newblock {\em Categories for the working mathematician}, volume~5 of {\em
  Graduate Texts in Mathematics}.
\newblock Springer-Verlag, New York, second edition, 1998.

\bibitem[MLP85]{MacLane-Pare}
Saunders Mac~Lane and Robert Par{\'e}.
\newblock Coherence for bicategories and indexed categories.
\newblock {\em J. Pure Appl. Algebra}, 37(1):59--80, 1985.

\bibitem[MP89]{MakkaiPare}
Michael Makkai and Robert Par{\'e}.
\newblock {\em Accessible categories: the foundations of categorical model
  theory}, volume 104 of {\em Contemporary Mathematics}.
\newblock American Mathematical Society, Providence, RI, 1989.

\bibitem[Por06]{CoringsComod}
Hans-E. Porst.
\newblock On corings and comodules.
\newblock {\em Arch. Math. (Brno)}, 42(4):419--425, 2006.

\bibitem[Por08a]{AdjAlgCoalg}
Hans-E. Porst.
\newblock Dual adjunctions between algebras and coalgebras.
\newblock {\em Arab. J. Sci. Eng. Sect. C Theme Issues}, 33(2):407--411, 2008.

\bibitem[Por08b]{FundConstrCoalgCorComod}
Hans-E. Porst.
\newblock Fundamental constructions for coalgebras, corings, and comodules.
\newblock {\em Appl. Categ. Structures}, 16(1-2):223--238, 2008.

\bibitem[Por08c]{MonComonBimon}
Hans-E. Porst.
\newblock On categories of monoids, comonoids, and bimonoids.
\newblock {\em Quaest. Math.}, 31(2):127--139, 2008.

\bibitem[Pow89]{Ageneralcoherenceresult}
A.~J. Power.
\newblock A general coherence result.
\newblock {\em J. Pure Appl. Algebra}, 57(2):165--173, 1989.

\bibitem[Pow90]{Pastingtheorem}
A.~J. Power.
\newblock A {$2$}-categorical pasting theorem.
\newblock {\em J. Algebra}, 129(2):439--445, 1990.

\bibitem[Shu08]{Framedbicats}
Michael Shulman.
\newblock Framed bicategories and monoidal fibrations.
\newblock {\em Theory Appl. Categ.}, 20:No. 18, 650--738, 2008.

\bibitem[Shu10]{ConstrSymMonBicats}
Michael Shulman.
\newblock Constructing symmetric monoidal bicategories.
\newblock arXiv:1004.0993 [math.CT], 2010.

\bibitem[Shu13]{Enrichedindexedcats}
Michael Shulman.
\newblock Enriched indexed categories.
\newblock {\em Theory Appl. Categ.}, 28:616--696, 2013.

\bibitem[Str72]{FormalTheoryMonadsI}
Ross Street.
\newblock The formal theory of monads.
\newblock {\em J. Pure Appl. Algebra}, 2(2):149--168, 1972.

\bibitem[Str80]{FibrationsinBicats}
Ross Street.
\newblock Fibrations in bicategories.
\newblock {\em Cahiers Topologie G\'eom. Diff\'erentielle}, 21(2):111--160,
  1980.

\bibitem[Str96]{Categoricalstructures}
Ross Street.
\newblock Categorical structures.
\newblock In {\em Handbook of algebra, {V}ol.\ 1}, volume~1 of {\em Handb.
  Algebr.}, pages 529--577. Elsevier/North-Holland, Amsterdam, 1996.

\bibitem[Str07]{Quantum}
Ross Street.
\newblock {\em Quantum groups}, volume~19 of {\em Australian Mathematical
  Society Lecture Series}.
\newblock Cambridge University Press, Cambridge, 2007.
\newblock A path to current algebra.

\bibitem[Swe69]{Sweedler}
Moss~E. Sweedler.
\newblock {\em Hopf algebras}.
\newblock Mathematics Lecture Note Series. W. A. Benjamin, Inc., New York,
  1969.

\bibitem[Vas12]{mine}
Christina Vasilakopoulou.
\newblock Enrichment of categories of algebras and modules.
\newblock pre-print arXiv:1205.6450 [math.CT], 2012.

\bibitem[Web04]{WeberGenericMorphisms}
Mark Weber.
\newblock Generic morphisms, parametric representations and weakly cartesian
  monads.
\newblock {\em Theory Appl. Categ.}, 13:191--234, 2004.

\bibitem[Win90]{Winskel}
Glynn Winskel.
\newblock A compositional proof system on a category of labelled transition
  systems.
\newblock {\em Inform. and Comput.}, 87(1-2):2--57, 1990.

\bibitem[Wis75]{LinearReps}
Manfred~B. Wischnewsky.
\newblock On linear representations of affine groups. {I}.
\newblock {\em Pacific J. Math.}, 61(2):551--572, 1975.

\bibitem[Wol74]{Wolff}
Harvey Wolff.
\newblock {$V$}-cat and {$V$}-graph.
\newblock {\em J. Pure Appl. Algebra}, 4:123--135, 1974.

\end{thebibliography}

\end{document}